\documentclass{acta}

\usepackage{harvard_acta}
\usepackage{amssymb}
\usepackage{amsmath}
\usepackage{verbatim}
\usepackage{pgf,tikz}
\usetikzlibrary{arrows}

\title[]{{Schwarz methods by domain truncation}}
\author[{M. J. Gander and H. Zhang}]{%
  Martin J. Gander\\
  {\it Department of Mathematics,}\\
  {\it University of Geneva, CP64, 1211 Geneva 4, Switzerland}\\
  {\tt martin.gander@unige.ch}\\
  \and
  Hui Zhang\\
  {\it Department of Applied Mathematics and }\\
  {\it Laboratory for Intelligent Computing \& Financial Technology,}\\
  {\it Xi'an Jiaotong-Liverpool University, Suzhou 215123, China}\\
  {\tt hui.zhang@xjtlu.edu.cn} }

\pagerange{\pageref{firstpage}--\pageref{lastpage}}
\doi{XXXXXX}

\newcommand{\eg}{{\it e.g.}}
\newcommand{\ie}{{\it i.e.}}
\newcommand{\etc}{{\it  etc.}}
\newcommand{\cn}{\citeasnoun}
\newcommand{\ignore}[1]{}

\newtheorem{theorem}{Theorem}[section]
\newtheorem{lemma}[theorem]{Lemma}
\newtheorem{remark}[theorem]{Remark}
\counterwithin{figure}{section}
\begin{document}

\label{firstpage}
\maketitle
\vspace{2em}

\begin{abstract}
Schwarz methods use a decomposition of the computational domain into
subdomains and need to put boundary conditions on the subdomain
boundaries. In domain truncation one restricts the unbounded domain to
a bounded computational domain and also needs to put boundary
conditions on the computational domain boundaries. In both fields there
are vast bodies of literature and research is very active and
ongoing. It turns out to be fruitful to think of the domain
decomposition in Schwarz methods as truncation of the domain onto
subdomains. Seminal precursors of this fundamental idea are
\cn{hagstrom1988numerical}, \cn{Despres90} and \cn{lions1990schwarz}. The first
truly optimal Schwarz method that converges in a finite number of
steps was proposed in \cn{Nataf93} and used precisely transparent
boundary conditions as transmission conditions between
subdomains. Approximating these transparent boundary conditions for
fast convergence of Schwarz methods led to the development of
optimized Schwarz methods -- a name that has become common for Schwarz
methods based on domain truncation. Compared to classical Schwarz
methods which use simple Dirichlet transmission conditions and have
been successfully used in a wide range of applications, optimized
Schwarz methods are much less well understood, mainly due to their
more sophisticated transmission conditions.

\vskip1em
A key application of Schwarz methods with such sophisticated
transmission conditions turned out to be time-harmonic wave
propagation problems, because classical Schwarz methods simply do not
work then. The last decade has brought many new Schwarz methods based
on domain truncation. A review from an algorithmic perspective
\cite{gander2019class} showed the equivalence of many of these new
methods to optimized Schwarz methods. The analysis of optimized
Schwarz methods is however lagging behind their algorithmic
development. The general abstract Schwarz framework can not be used
for the analysis of these methods, and thus there are many open
theoretical questions about their convergence. Like for practical
multigrid methods, Fourier analysis has been instrumental for
understanding the convergence of optimized Schwarz methods and to tune
their transmission conditions. Similar to Local Fourier Mode Analysis
in multigrid, the unbounded two subdomain case is used as a model for
Fourier analysis of optimized Schwarz methods due to its
simplicity. Many aspects of the actual situation, \eg, boundary
conditions of the original problem and the number of subdomains, were
thus neglected in the unbounded two subdomain analysis. While this
gave important insight, new phenomena beyond the unbounded two
subdomain models were discovered.

This present situation is the motivation for our survey: to give a
comprehensive review and precise exploration of convergence behaviors
of optimized Schwarz methods based on Fourier analysis taking into
account the original boundary conditions, many subdomain
decompositions and layered media. We consider as our model problem the
operator $-\Delta + \eta$ in the diffusive case $\eta>0$ (screened
Laplace equation) or the oscillatory case $\eta<0$ (Helmholtz
equation), in order to show the fundamental difference of the behavior
of Schwarz solvers for these problems. The transmission conditions we
study include the lowest order absorbing conditions (Robin), and also
more advanced perfectly matched layers (PML), both developed first for
domain truncation. Our intensive work over the last two years on this
review led to several new results presented here for the first time:
in the bounded two subdomain analysis for the Helmholtz equation, we
see a strong influence of the original boundary conditions imposed on
the global problem on the convergence factor of the Schwarz methods,
and the asymptotic convergence factors with small overlap can differ
from the unbounded two subdomain analysis. In the many subdomain
analysis, we find the scaling with the number of subdomains, \eg, when
the subdomain size is fixed, robust convergence of the double sweep
Schwarz method for the free space wave problem either with fixed
overlap and zeroth order Taylor conditions, or with a logarithmically
growing PML, and we find Schwarz methods with PML work like smoothers
that converge faster for higher Fourier frequencies; in particular,
for the free space wave problem plane waves (in the error) passing
through interfaces at a right angle converge slower. In addition to
the main part on analysis, we also give an expository historical
introduction to Schwarz methods at the beginning, and a brief
interpretation of the recently proposed optimal Schwarz methods for
decompositions with cross points from the viewpoint of transmission
conditions, before we conclude with a summary of open research
problems. In Appendix A, we provide a Matlab program for a block LU
form of an optimal Schwarz method with cross points, and in Appendix
B, we give the Maple program for the two subdomain Fourier analysis.  
  
\end{abstract}

\tableofcontents
\vspace{3mm}

\section{Introduction}

Schwarz domain decomposition methods are the oldest domain
decomposition methods. They were invented by Hermann Amandus Schwarz
(see Figure \ref{SchwarzOriginalFig} on the left) in 1869 with a
publication in the following year \cite{Schwarz:1870:UGA}.
\begin{figure}
  \centering
  \tabcolsep2em
  \begin{tabular}{cc}
  \includegraphics[width=0.3\textwidth]{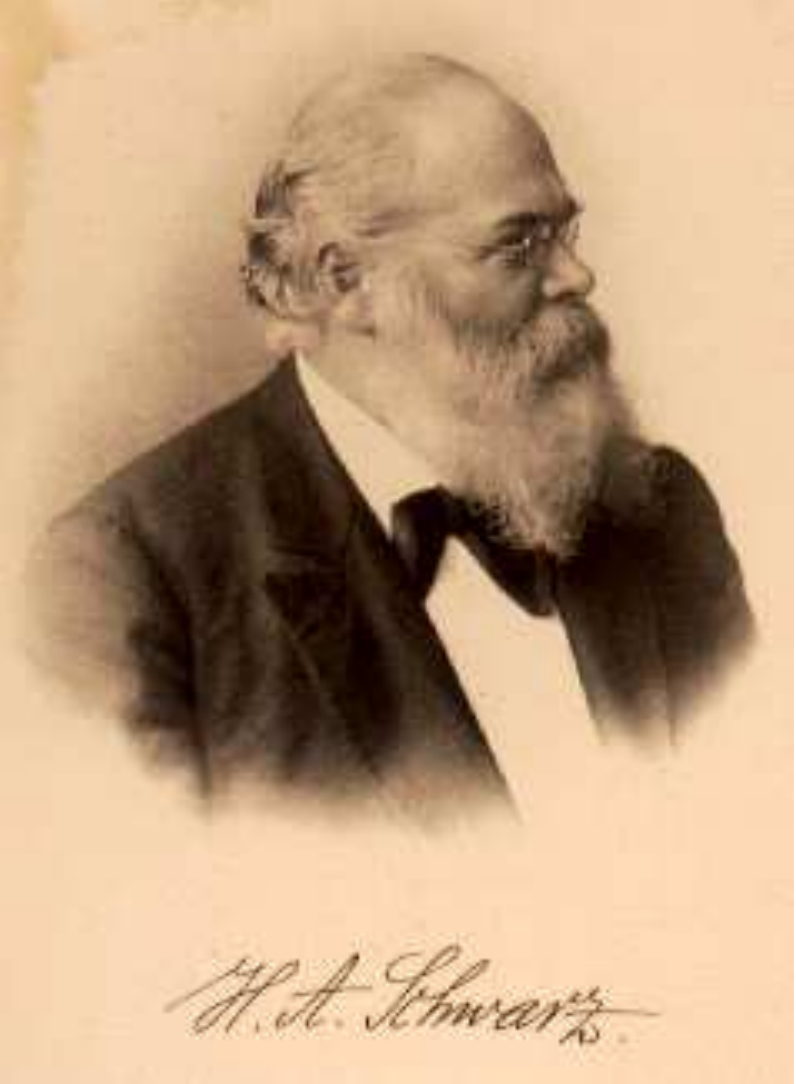}&
  \parbox[b]{0.4\textwidth}{\includegraphics[width=0.4\textwidth]{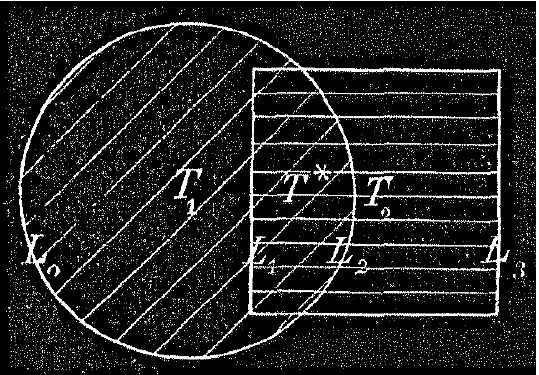}\vspace{1em}\\}
  \end{tabular}
  \caption{Left: Herman Amandus Schwarz (25.1.1843-30.11.1921). Right: original
    figure for the decomposition in the alternating Schwarz method.}
  \label{SchwarzOriginalFig}
\end{figure}
Schwarz invented the now called {\em alternating Schwarz method} in
order to close a gap in the proof of the Riemann mapping theorem at
the end of Riemann's PhD thesis \cite{riemann1851grundlagen}, an
English translation of which is available
\cite{riemann1851foundations}. In his proof, Riemann had assumed that
the Dirichlet problem
\begin{equation}\label{eq:DirichletProblem}
\Delta u=0\ \mbox{in $\Omega$},\quad \hbox{$u=g$ on $\partial \Omega$},
\end{equation}
always had a solution, just by taking the function $u(x,y)$ which
solves the minimization problem\footnote{Indeed, taking a variational
  derivative of ${\cal J}(u)$, we find for an arbitrary variation
  $v$ that vanishes on $\partial \Omega$ that $\frac{\D}{\D \epsilon}{\cal
    J}(u+\epsilon v)=2\iint _{\Omega}(\frac{\partial (u+\epsilon
    v)}{\partial x})\frac{\partial v}{\partial x}+(\frac{\partial
    (u+\epsilon v)}{\partial y})\frac{\partial v}{\partial y}\D x\D y$.
  Therefore at $\epsilon=0$, using integration by parts and that the
  variation $v$ vanishes on $\partial \Omega$, we get $\iint
  _{\Omega}\Delta u v\D x\D y=0$. Since this must hold for all variations
  $v$, we must have $\Delta u=0$, \ie\ equation
  \R{eq:DirichletProblem} holds at a stationary point.}
\begin{equation}\label{eq:DirichletIntegral}
  {\cal J}(u):=\iint _{\Omega} \bigl(\frac{\partial u}{\partial x}\bigr)^2+ \bigl(\frac{\partial u}{\partial y}\bigr)^2\D x\D y\longrightarrow \min,\quad
  \hbox{$u=g$ on $\partial \Omega$},
\end{equation}
and satisfies the boundary condition $u=g$ on $\partial \Omega$.  When asked why this minimization
problem had always a solution, Riemann replied that he had learned this in his analysis course
taught by Dirichlet, the functional ${\cal J}(u)$ being bounded from below, thus coining the term
Dirichlet principle.  \cn{weierstrass} then presented a counterexample to this way of arguing: the
functional to minimize, $\int_{-1}^{+1} (x\cdot u')^2\D x\longrightarrow \min$, is also bounded from
below, but when one tries to find a function that minimizes this integral with boundary conditions
$u(-1)=a$, $u(1)=b$, $a\ne b$, $u'$ should be small when $x$ is away from zero, best is taking
$u'=0$ so there is no contribution to the integral, and when $x=0$, $u'$ can be large, since it is
multiplied with $0$ and thus also does not contribute. The minimizing ``solution'' is thus piecewise
constant and discontinuous at $x=0$, and the gap in Riemann's proof remained. However, the Dirichlet
problem \R{eq:DirichletProblem} had a well known solution on rectangular and circular domains, using
Fourier techniques from 1822.  H.A. Schwarz invented his alternating method by choosing as domain
$\Omega$ the union of a disk and an overlapping rectangle (see Figure \ref{SchwarzOriginalFig}
(right)). If we call the overlapping subdomains $\Omega_1$ and $\Omega_2$ ($T_1$ and $T_2$ in the
original drawing of Schwarz), and the interfaces $\Gamma_1$ and $\Gamma_2$ ($L_2$ and $L_1$ in the
original drawing of Schwarz), the alternating Schwarz method computes alternatingly on the disk and
on the rectangle the Dirichlet problem and carries the newly computed solution from the interface
curves $\Gamma_1$ and $\Gamma_2$ over to the other subdomain as new boundary condition for the next
solve,
\begin{equation}\label{AlternatingSchwarzMethod}
  \begin{array}{rcllrcll}
    \Delta u_1^n & = & 0 & \mbox{in $\Omega_1$},&  
    \Delta u_2^n & = & 0 & \mbox{in $\Omega_2$},\\  
    u_1^n & = & g & \mbox{on $\partial \Omega \cap \overline{\Omega}_1$},\quad &
    u_2^n & = & g & \mbox{on $\partial \Omega \cap
    \overline{\Omega}_2$},\\
    u_1^n &=&  u_2^{n-1}\ & \mbox{on $\Gamma_1$}, &
    u_2^n &=& u_1^n\ & \mbox{on  $\Gamma_2$}.
   \end{array}
\end{equation}
Schwarz proved convergence of this alternating method using the
maximum principle, see \cn{gander2014origins} for more information
on the origins of the alternating Schwarz method. All this happened
during the fascinating time of the development of variational calculus
and functional analysis, which led eventually to the finite element
method, see the historical review \cn{gander2012euler}.

It took many decades before the alternating Schwarz method became a computational
tool. \cn{miller1965numerical} was the first to point out its usefulness for computations, but it
was with the seminal work of Lions \cite{Lions:1988:SAM,lions1989schwarzII,lions1990schwarz}, and
the additive Schwarz method of \cn{dryja1987additive}, that Schwarz methods became powerful and
mainstream parallel computational tools for solving discretized partial differential equations.
\cn{Lions:1988:SAM} also proposed a parallel variant of the method, by simply not using the newest
available value along $\Gamma_2$ but the previous one in \R{AlternatingSchwarzMethod},
\begin{equation}
    u_2^n = u_1^{n-1}\  \mbox{on  $\Gamma_2$},
\end{equation}
so that the subdomain solves can be performed in parallel. This is analogous to the Jacobi
stationary iterative method, compared to the Gauss-Seidel stationary iterative method in linear
algebra. Note that the additive Schwarz method is quite different from the parallel Schwarz method:
it is a preconditioner for the conjugate gradient method, and does not converge when used as a
stationary iteration without relaxation, in contrast to the parallel Schwarz method, see
\cn[Section 3.2]{gander2008schwarz} for more details.

\cn{lions1990schwarz} also introduced a non-overlapping variant of the Schwarz method\footnote{He
  considered therefore decompositions where $\Gamma_1=\Gamma_2$, but the method is equally
  interesting with overlap as well, since it converges faster than the classical Schwarz method, as
  we will see.} using Robin transmission conditions in \R{AlternatingSchwarzMethod},
$$
  \arraycolsep0.2em
    \begin{array}{rcllrcll}
    (\partial_{n_1}\!+\!p_1)u_1^n &=& (\partial_{n_1}\!+\!p_1) u_2^{n-1}\ & \mbox{on $\Gamma_1$}, &
    (\partial_{n_2}\!+\!p_2)u_2^n &=& (\partial_{n_2}\!+\!p_2)u_1^n\ & \mbox{on  $\Gamma_2$},
   \end{array}
$$
where $\partial_{n_j}$, $j=1,2$, denotes the unit outward normal derivative for $\Omega_j$ along the
interfaces $\Gamma_j$, and, following Lions, the $p_j$ can be constants, or functions along the
interface, or even operators. In contrast to the classical Schwarz method, where it was only
important to obtain convergence for closing the gap in the proof of the Riemann mapping theorem, and
convergence speed was not an issue, a good choice of the Robin parameters $p_j$ can greatly improve
the convergence of Schwarz methods. \cn{hagstrom1988numerical} worked on Schwarz methods for
non-linear problems, and advocated to use Robin transmission conditions involving non-local
operators. The work of \cn{tang1992generalized} on generalized Schwarz splittings also points into
this direction, for a more detailed review, see \cn{gander2008schwarz}.

What is the underlying idea of changing the transmission conditions
from Dirichlet to Robin conditions, even involving non-local
operators? In Schwarz methods for parallel computing, one decomposes a
domain into many subdomains, see Figure \ref{1Dand2DDecompositionFig}
for two typical examples.
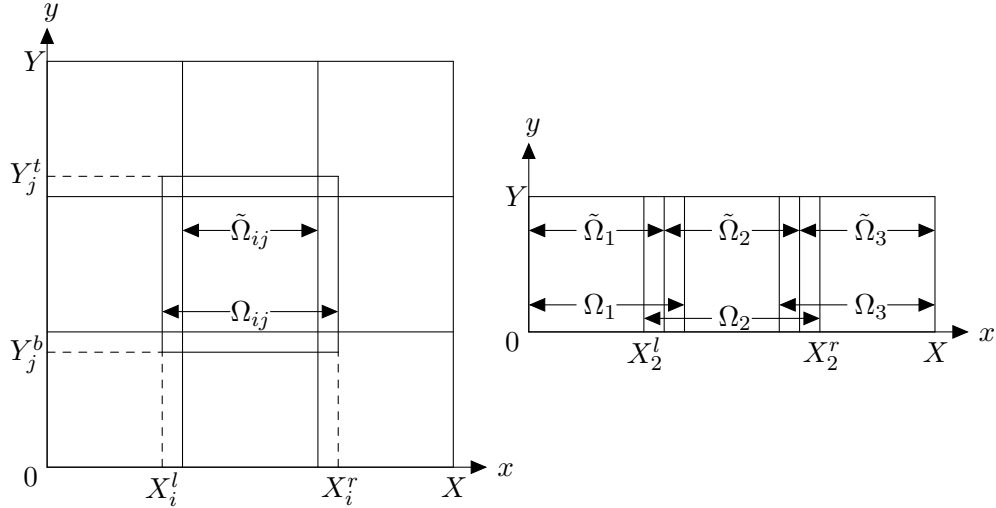
\begin{figure}
    \centering
\mbox{\begin{tikzpicture}[line cap=round,line join=round,>=triangle 45,x=1.0cm,y=1.0cm,scale=0.9]
\clip(-.6,-0.6) rectangle (7,7);
\draw [->] (0,0) -- (6.5,0);
\draw [->] (0,0) -- (0,6.5);
\draw (-0.2,7) node[anchor=north west] {$y$};
\draw (6.5,0.2) node[anchor=north west] {$x$};
\draw (-0.5,0.15) node[anchor=north west] {$0$};
\draw (-0.5,6.3) node[anchor=north west] {$Y$};
\draw (0,6)-- (6,6);
\draw (6,6)-- (6,0);
\draw (6,0)-- (0,0);
\draw (0,0)-- (0,6);
\draw (2,6)-- (2,0);
\draw (4,6)-- (4,0);
\draw (0,4)-- (6,4);
\draw (0,2)-- (6,2);
\draw (1.7,1.7)-- (1.7,4.3);
\draw (1.7,1.7)-- (4.3,1.7);
\draw (4.3,1.7)-- (4.3,4.3);
\draw (4.3,4.3)-- (1.7,4.3);
\draw [->] (2.7,2.3) -- (1.7,2.3);
\draw [->] (3.3,2.3) -- (4.3,2.3);
\draw (2.56,2.6) node[anchor=north west] {$\Omega_{ij}$};
\draw [->] (3.3,3.5) -- (4,3.5);
\draw [->] (2.7,3.5) -- (2,3.5);
\draw (2.56,3.9) node[anchor=north west] {$\tilde{\Omega}_{ij}$};
\draw (1.3,0.07) node[anchor=north west] {$X_i^l$};
\draw [dashed](1.7,0)-- (1.7,1.7);
\draw [dashed](4.3,0)-- (4.3,1.7);
\draw (3.9,0) node[anchor=north west] {$X_i^r$};
\draw (5.67,0) node[anchor=north west] {$X$};
\draw [dashed](0,1.7)-- (1.7,1.7);
\draw [dashed](0,4.3)-- (1.7,4.3);
\draw (-0.7,2.1) node[anchor=north west] {$Y_j^b$};
\draw (-0.7,4.7) node[anchor=north west] {$Y_j^t$};
\end{tikzpicture}\hspace{-1em}
\begin{tikzpicture}[line cap=round,line join=round,>=triangle 45,x=1.0cm,y=1.0cm,scale=0.9]
\clip(-.4,-2.6) rectangle (7,3.3);
\draw [->] (0,0) -- (6.5,0);
\draw [->] (0,0) -- (0,2.8);
\draw (-0.2,3.3) node[anchor=north west] {$y$};
\draw (6.5,0.2) node[anchor=north west] {$x$};
\draw (-0.5,0.15) node[anchor=north west] {$0$};
\draw (-0.5,2.3) node[anchor=north west] {$Y$};
\draw (0,2)-- (6,2);
\draw (6,2)-- (6,0);
\draw (6,0)-- (0,0);
\draw (0,0)-- (0,2);
\draw (2,2)-- (2,0);
\draw (4,2)-- (4,0);
\draw (1.7,2)-- (1.7,0);
\draw (2.3,2)-- (2.3,0);
\draw (3.7,2)-- (3.7,0);
\draw (4.3,2)-- (4.3,0);
\draw [->] (1.3,0.4) -- (2.3,0.4);
\draw [->] (0.7,0.4) -- (0,0.4);
\draw [->] (2.7,0.2) -- (1.7,0.2);
\draw [->] (3.3,0.2) -- (4.3,0.2);
\draw [->] (5.3,0.4) -- (6,0.4);
\draw [->] (4.7,0.4) -- (3.7,0.4);
\draw (0.66,0.7) node[anchor=north west] {$\Omega_1$};
\draw (2.65,0.53) node[anchor=north west] {$\Omega_2$};
\draw (4.66,0.7) node[anchor=north west] {$\Omega_3$};
\draw [->] (1.3,1.5) -- (2,1.5);
\draw [->] (0.7,1.5) -- (0,1.5);
\draw [->] (3.3,1.5) -- (4,1.5);
\draw [->] (2.7,1.5) -- (2,1.5);
\draw [->] (5.3,1.5) -- (6,1.5);
\draw [->] (4.7,1.5) -- (4,1.5);
\draw (0.66,1.9) node[anchor=north west] {$\tilde{\Omega}_1$};
\draw (2.65,1.9) node[anchor=north west] {$\tilde{\Omega}_2$};
\draw (4.66,1.9) node[anchor=north west] {$\tilde{\Omega}_3$};
\draw (1.3,0.07) node[anchor=north west] {$X_2^l$};
\draw (3.9,0) node[anchor=north west] {$X_2^r$};
\draw (5.67,0) node[anchor=north west] {$X$};
\end{tikzpicture}
}
\caption{Two typical domain decompositions: a two dimensional
  decomposition with cross points (left), and a one dimensional or
  sequential domain decomposition (right).}
  \label{1Dand2DDecompositionFig}
\end{figure}
To obtain an overlapping decomposition, it is convenient to first
define a decomposition into non-overlapping subdomains
$\tilde{\Omega}_{ij}$, and then to enlarge these subdomains by a thin
layer to get overlapping subdomains $\Omega_{ij}$, as indicated in
Figure \ref{1Dand2DDecompositionFig}. One then wants to compute only
subdomain solutions on the $\Omega_{ij}$ to approximate the global, so
called mono-domain solution on the entire domain $\Omega$. Let us
imagine for the moment that the source term of the partial
differential equation to be solved has only support in one subdomain
somewhere in the middle of the global domain, like for example in
subdomain $\Omega_{ij}$ in Figure \ref{1Dand2DDecompositionFig} on the
left, and assume that the global domain is infinitely large. Then it
would be best to put on the boundary of the subdomain containing the
source transparent boundary conditions, since then by solving the
subdomain problem we obtain by definition\footnote{Transparent
  boundary conditions are exactly defined by truncating the global
  domain such that the solution on the truncated domain coincides with
  the solution on the global domain.} the restriction of the global,
mono-domain solution to that subdomain. Robin boundary conditions are
approximations of the transparent boundary conditions, and one can
thus expect that with Robin boundary conditions subdomain solvers
compute better approximations to the overall mono-domain solution than
with Dirichlet boundary conditions.

To illustrate this, we solve the Poisson problem $\Delta u=f$ on the
unit square with zero Dirichlet boundary conditions and the right hand
side source function
$f(x,y):=100e^{-100((x-\frac{1}{2})^2+(y-\frac{1}{2})^2)}$. We show in
Figure \ref{PoissonExampleFig} on the left this source term, and on
the right the corresponding solution of the Poisson problem, using a
centered finite difference scheme with mesh size $h=1/40$.
\begin{figure}
  \centering
  \mbox{\includegraphics[width=0.4\textwidth]{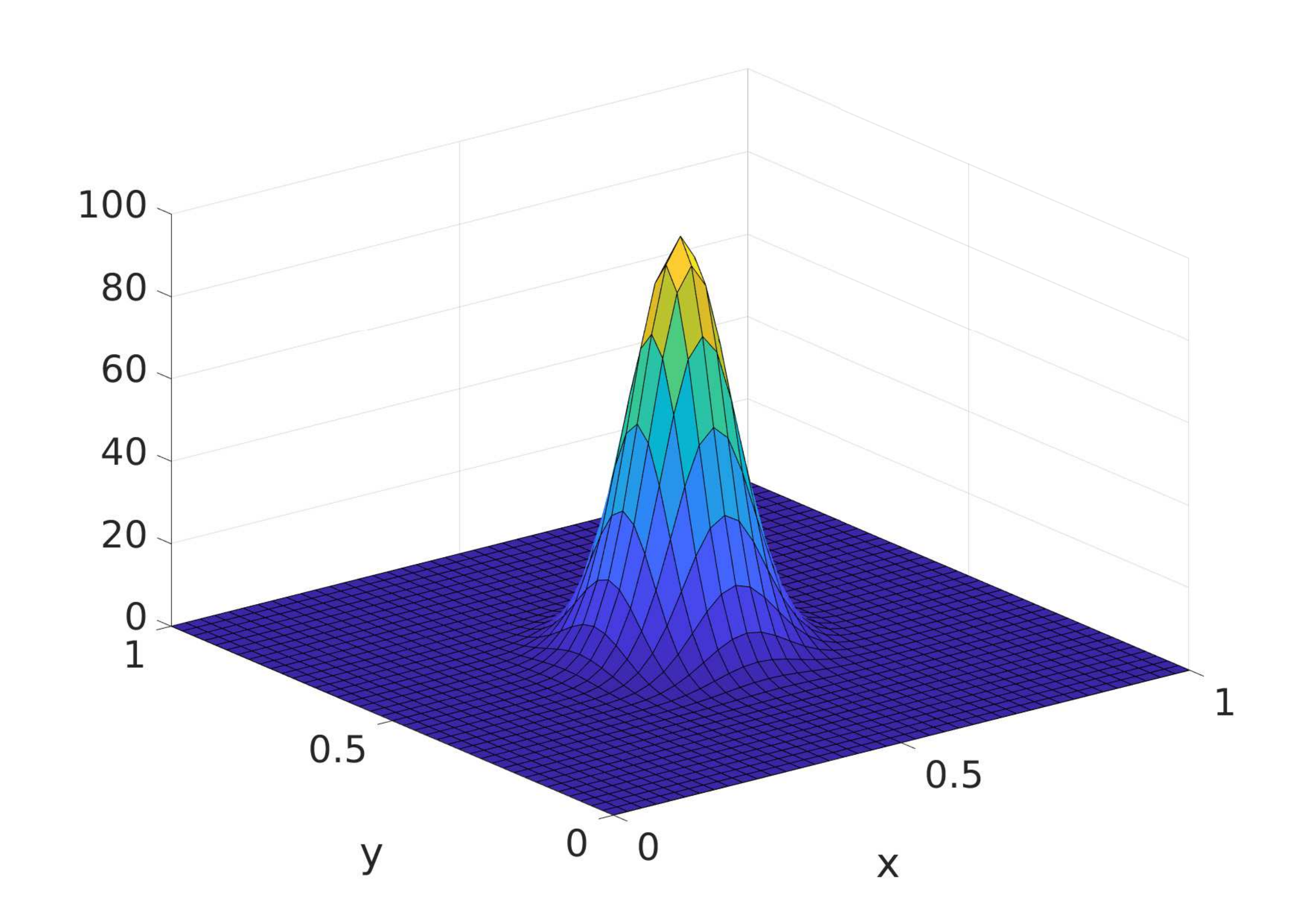}\quad
  \includegraphics[width=0.4\textwidth]{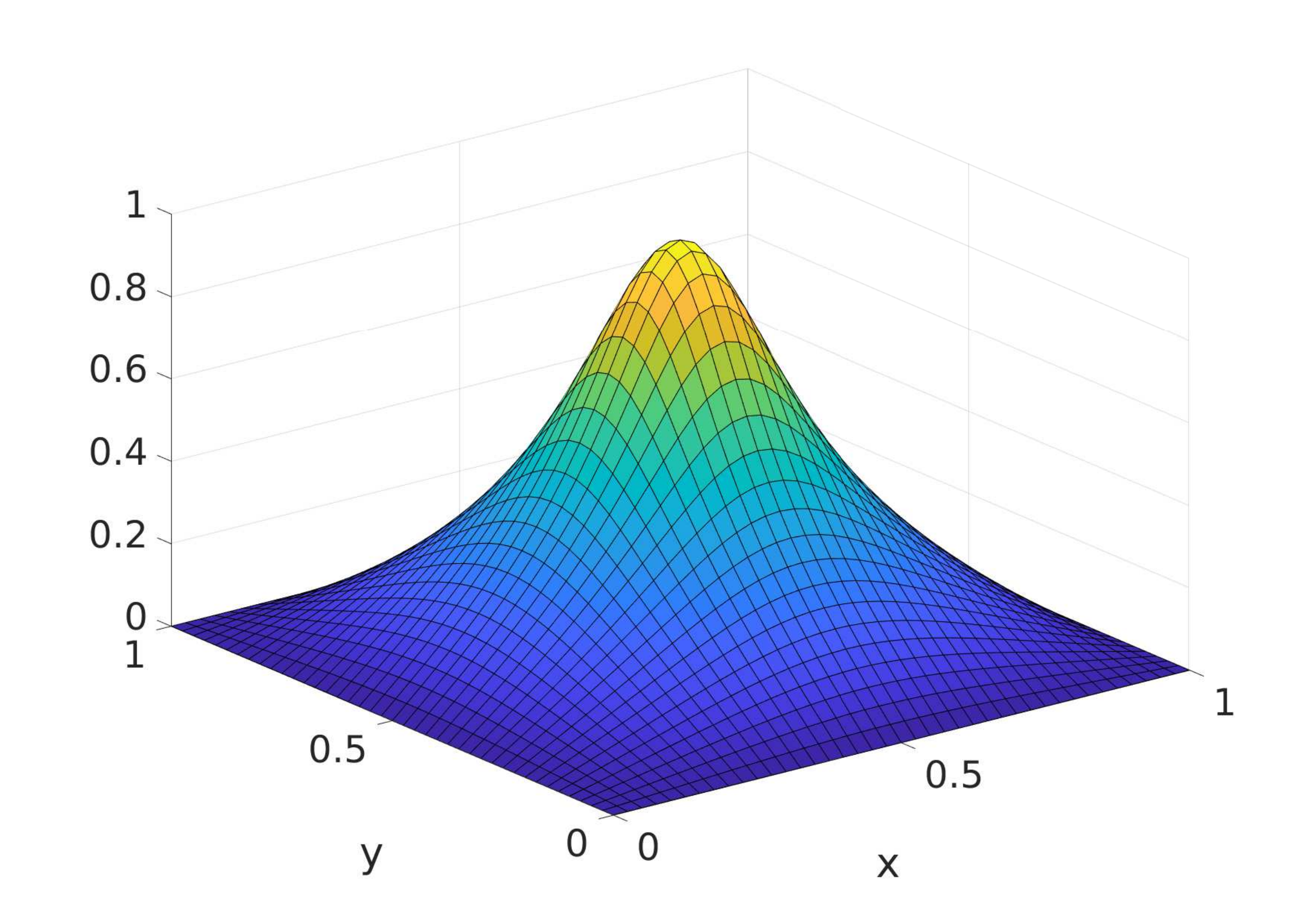}}
  \caption{Gaussian source term (left) and solution on of the
    corresponding Poisson problem (right).}
  \label{PoissonExampleFig}
\end{figure}
If we decompose the unit square domain into $3 \times 3$ subdomains as
indicated in Figure \ref{1Dand2DDecompositionFig} (left), and solve
the subdomain problem in the center with Dirichlet boundary conditions
$u=0$, we obtain the approximation shown in Figure
\ref{PoissonRASORASFig} (top left).
\begin{figure}
  \centering
  \mbox{\includegraphics[width=0.33\textwidth,trim=50 20 50 50,clip]{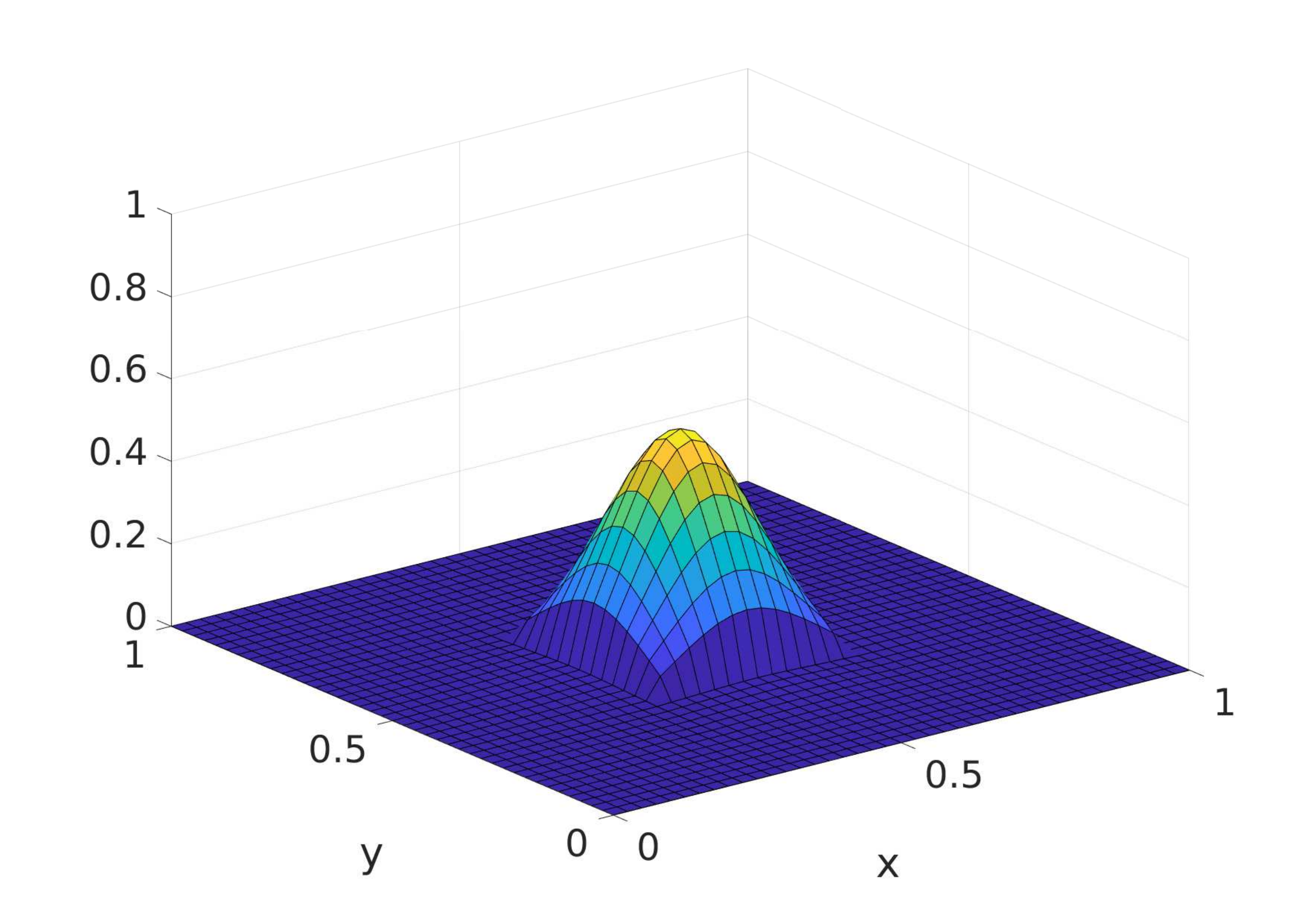}
  \includegraphics[width=0.33\textwidth,trim=50 20 50 50,clip]{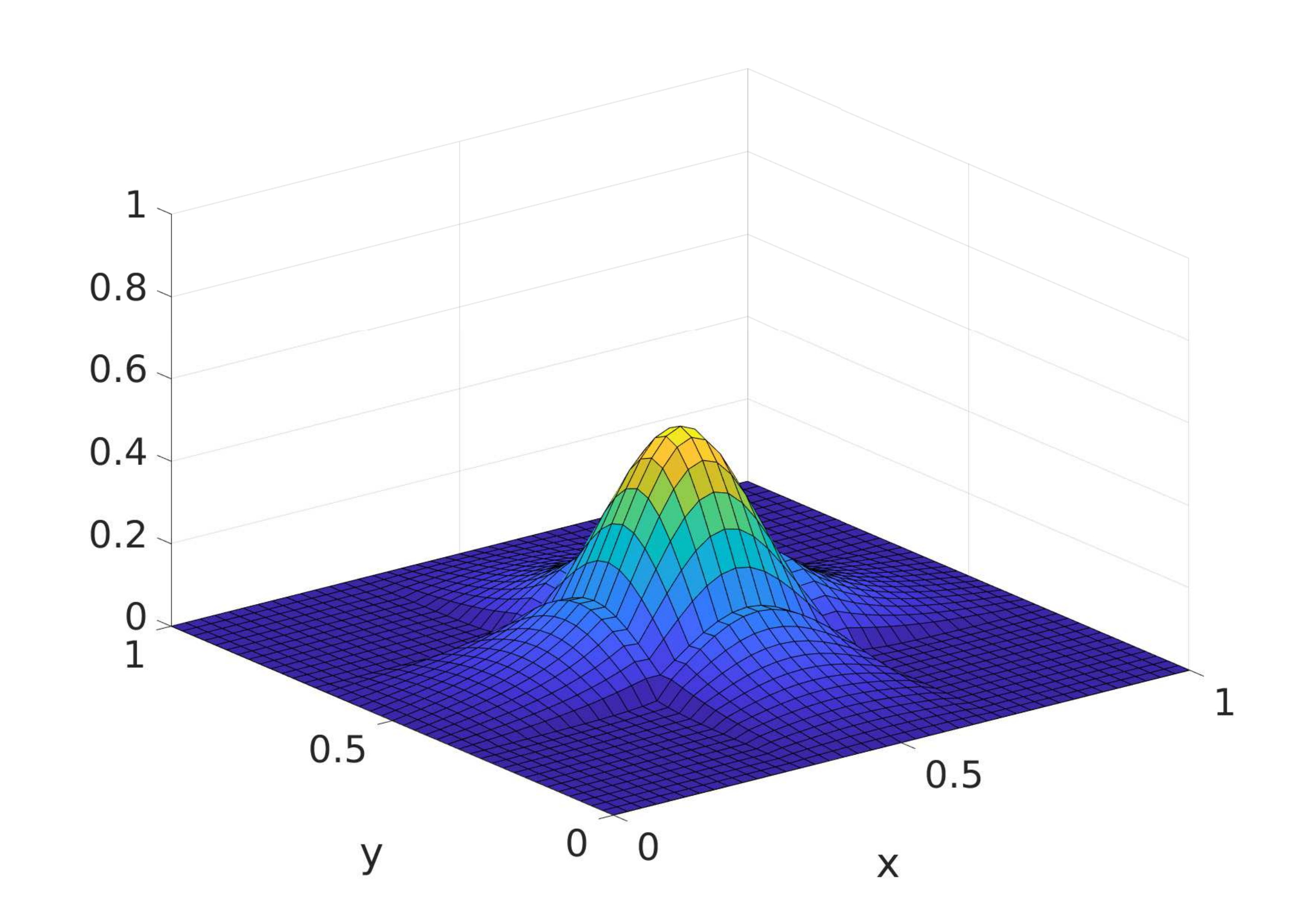}
  \includegraphics[width=0.33\textwidth,trim=50 20 50 50,clip]{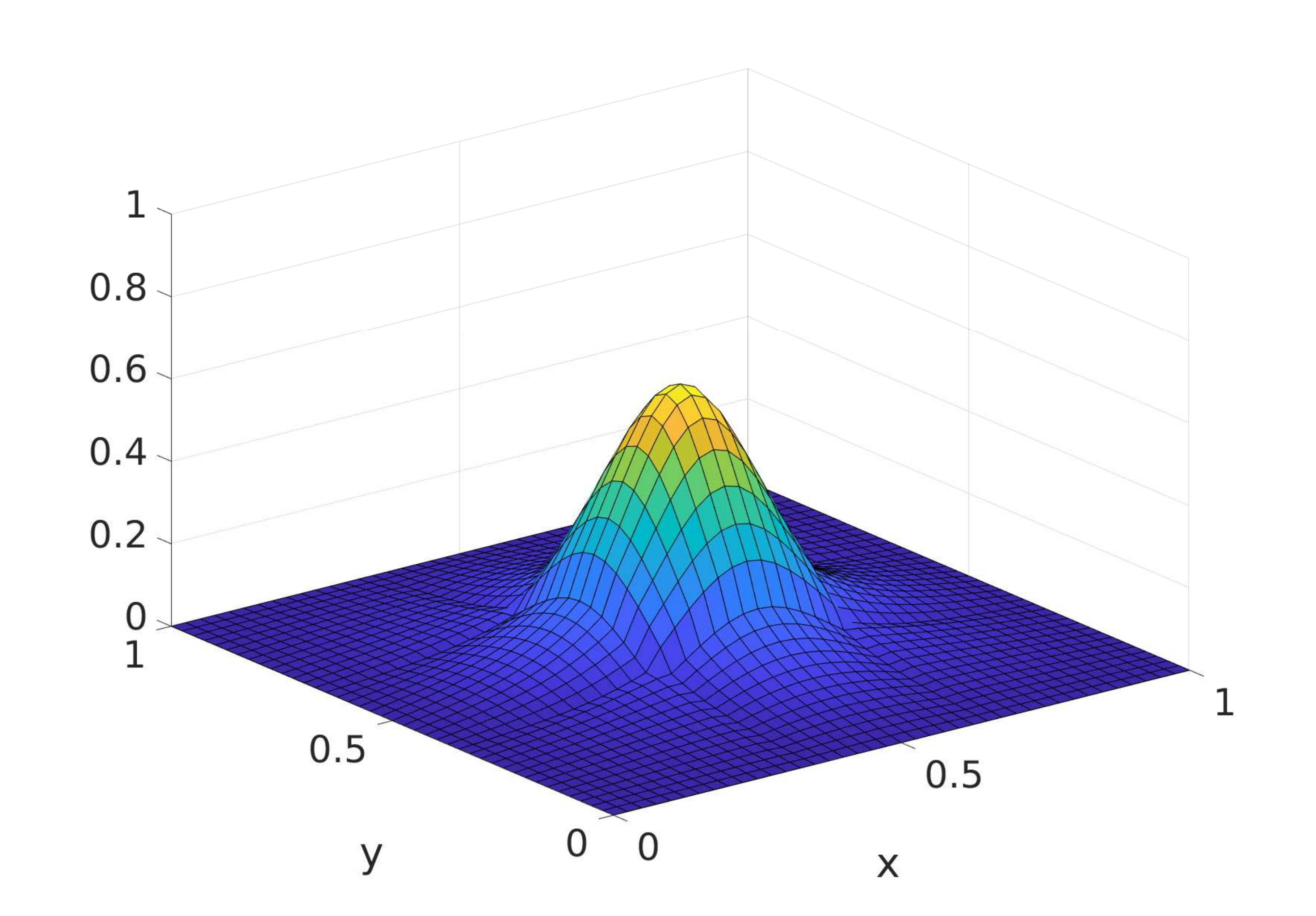}}
  \mbox{\includegraphics[width=0.33\textwidth,trim=50 20 50 50,clip]{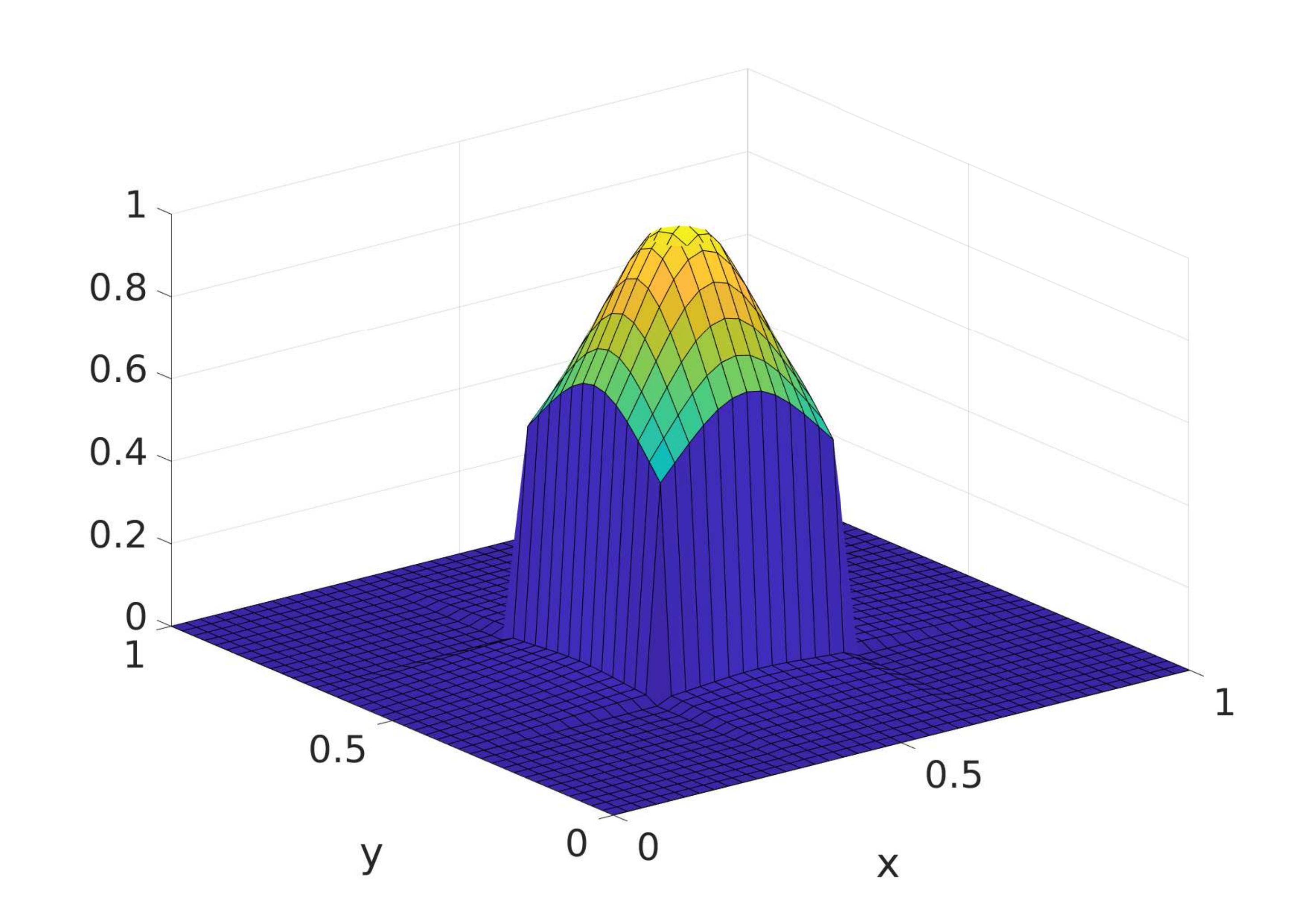}
    \includegraphics[width=0.33\textwidth,trim=50 20 50 50,clip]{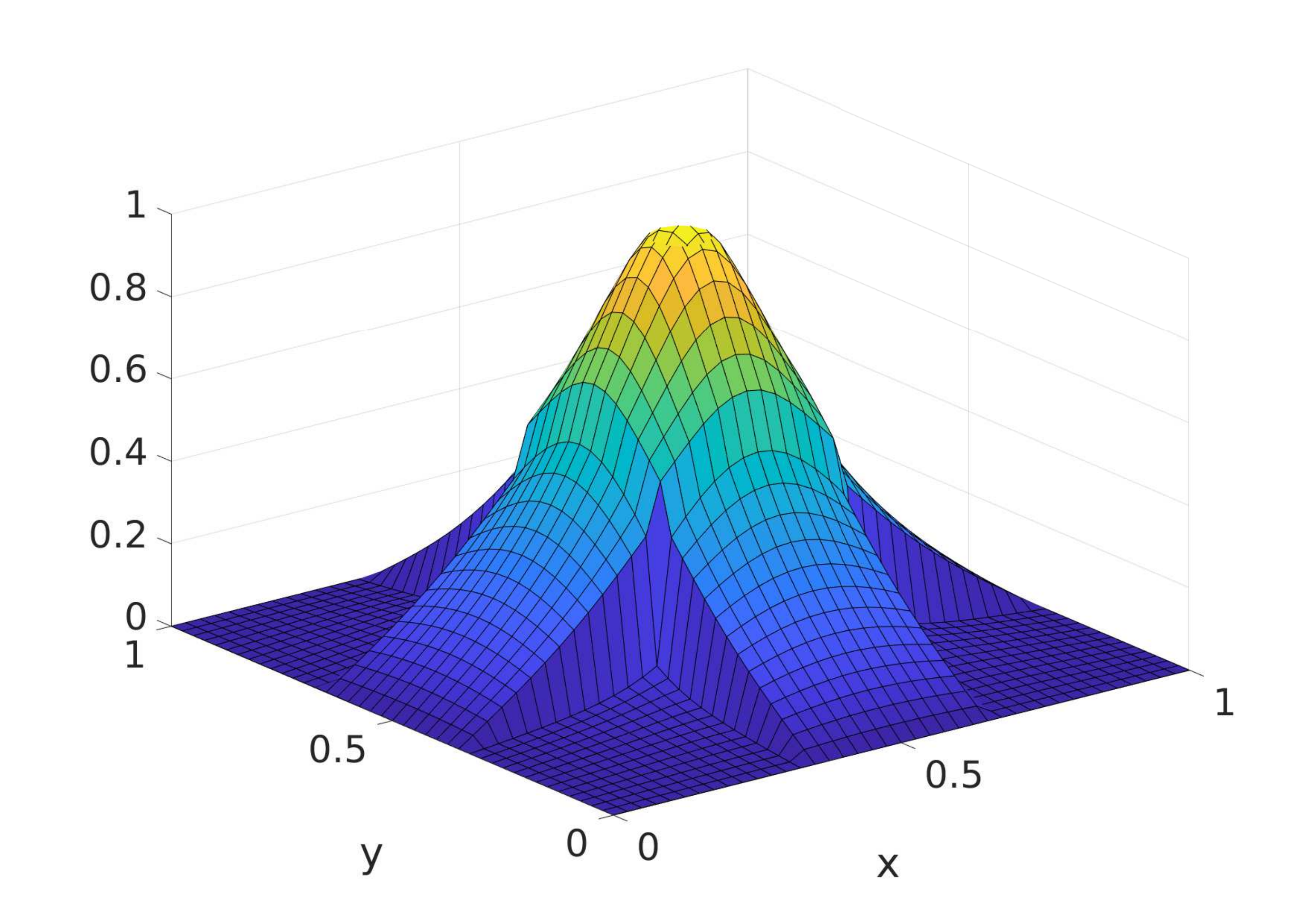}
    \includegraphics[width=0.33\textwidth,trim=50 20 50 50,clip]{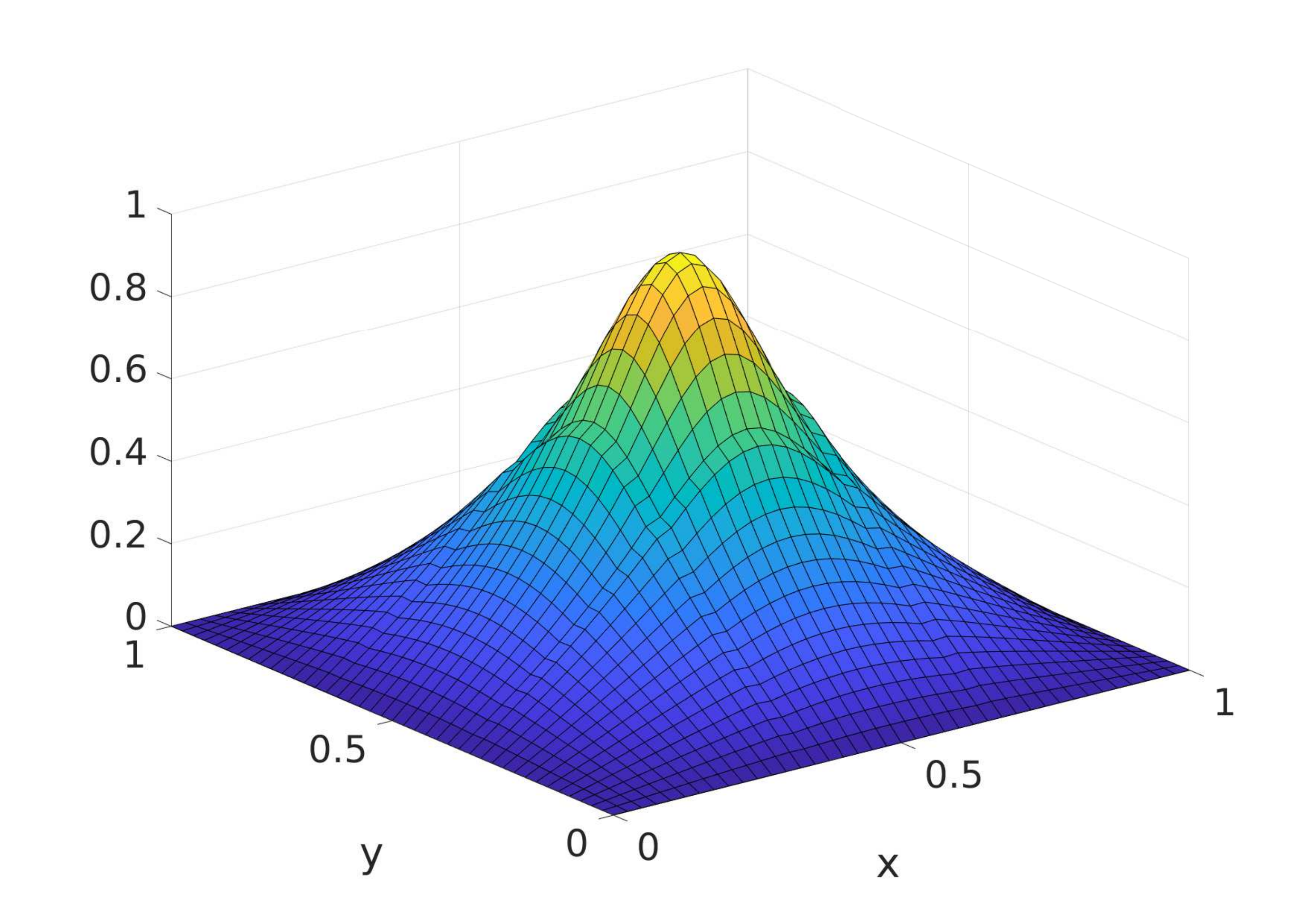}}
  \caption{First few iterations of classical (top) and optimized
    parallel Schwarz (bottom) for the Poisson problem with Gaussian
    source term from Figure \ref{PoissonExampleFig}.}
  \label{PoissonRASORASFig}
\end{figure}
If we perform the solve with Robin conditions $(\partial_n+p)u=0$ with
$p=4$, we get the approximation shown in Figure
\ref{PoissonRASORASFig} (bottom left). We clearly see that the result
when truncating with Dirichlet conditions is much further away from
the desired solution shown in Figure \ref{PoissonExampleFig} (right)
than when truncating with Robin conditions. This is however precisely
the first iteration of a classical Schwarz method with Dirichlet
transmission conditions, compared to an optimized Schwarz method with
Robin transmission conditions, when starting the iteration with a zero
initial guess. We show in Figure \ref{PoissonRASORASFig} also the next
two iterations of the classical (top) and optimized Schwarz method
(bottom) with algebraic overlap of two mesh layers\footnote{This
  corresponds for the classical Schwarz method to a physical overlap
  of $3h$ ($h$ the mesh size), see \cn[Figure 3.1]{gander2008schwarz}
  for the relation between algebraic and physical overlap, and the
  discussion in \cn[Section 4.1]{St-Cyr07} for the algebraic
  overlap when Robin conditions are used, and algebraic overlap of
  two mesh layers corresponds to physical overlap $h$ only.}, and we see the
great convergence enhancement due to the Robin transmission
conditions, which are much better approximations of the transparent
boundary conditions than the Dirichlet transmission conditions. This
illustrates well that Schwarz methods (and domain decomposition
methods in general) are methods which approximate solutions by domain
truncation, and we see that transmission conditions that are better at
truncating domains lead to better convergence.

This became a major new viewpoint on domain decomposition methods over
the past two decades. Naturally Dirichlet (and Neumann) conditions
then appear as not very good candidates to truncate domains and be
used as transmission conditions between subdomains: it is of interest
for rapid convergence to use absorbing boundary conditions which
approximate transparent conditions, like Robin conditions or higher
order Ventcell conditions, and also perfectly matched layers (PMLs), or
integral operators in the transmission conditions between subdomains.

Optimized Schwarz methods were pioneered in the early nineties by
Nataf et al. \cite{Nataf93,NRS94}; see in particular also the early
contributions of \cn{japhet1998optimized}, \cn{Chevalier}, \cn{EZ98},
and \cn{gander2000optimized} where the name optimized Schwarz methods
was coined. Optimized Schwarz methods use Robin or higher order
transmission conditions or PMLs at the interfaces between subdomains,
and all are approximations of transparent boundary conditions, see
\cn{GanderOSM} and references therein for an
introduction\footnote{\label{footnoteoptimalSchwarz} The term optimal
  Schwarz method for Schwarz methods with transparent boundary
  conditions appeared already in \cite{gander1999optimal} for time
  dependent problems, and this use of optimal means really faster is
  not possible, in contrast to the other common use of optimal meaning
  just scalable in the domain decomposition literature.}. This
conceptual change for domain decomposition methods is fundamental and
was discovered independently for solving hard wave propagation
problems by domain decomposition and iteration, since for such
problems classical domain decomposition methods are not effective, see
the seminal work by \cn{Despres90}, \cn{Despres}, and \cn{EG} for a review why it is
hard to solve such problems by iteration. Also rational approximations
have been proposed in the transmission conditions for such problems,
see for example \cn{Boubendir}, \cn{KimZhang}, \cn{Kimsweep}. This new
idea of domain truncation led, independently of the work on optimized
Schwarz methods, to the invention of the sweeping preconditioner
\cite{EY1,EY2,Poulson,Tsuji1,Tsuji2,LiuYingRecur}, the source transfer
domain decomposition \cite{Chen13a,Chen13b,xiang2019double}, the
single layer potential method \cite{Stolk,StolkImproved}, and the
method of polarized traces \cite{Zepeda,ZD,ZepedaNested}. All these
methods are very much related, and can be understood in the context of
optimized Schwarz methods, for a review and formal proofs of
equivalence, see \cn{gander2019class}, and also the references therein
for the many followup papers by the various groups. A key ingredient
of these independently developed methods is the use of perfectly
matched layers as absorbing boundary conditions, a technique which had
only rarely been used in the domain decomposition community before,
for exceptions see \cn{Toselli}, \cn{Schadle}, \cn{SZBKS}.

While optimized Schwarz methods were developed for general
decompositions of the domain into subdomains, including cross points,
as shown on the left in Figure \ref{1Dand2DDecompositionFig} and in
the corresponding example in Figure \ref{PoissonRASORASFig},
the sweeping preconditioner, source transfer domain decomposition, the
single layer potential method and the method of polarized traces were
all formulated for one dimensional or sequential domain
decompositions, as shown in Figure \ref{1Dand2DDecompositionFig} on
the right, without cross points. This is because these methods were
developed with the physical intuition of wave propagation in one
direction, and not with domain decomposition in mind. In these
methods, the authors had in mind the transparent boundary conditions
at the interfaces, which contain the Dirichlet to Neumann (DtN), or more
generally the Steklov Poincaré operator replacing $p$ in the Robin
transmission condition, and with this choice, the method becomes a
direct solver. We illustrate this in Figure
\ref{PoissonOptimalSchwarzFig} for the strip decomposition in Figure
\ref{1Dand2DDecompositionFig} (right) with three subdomains and our
Poisson model problem with Gaussian source from Figure
\ref{PoissonExampleFig}.
\begin{figure}
  \centering
  \mbox{\includegraphics[width=0.33\textwidth,trim=50 30 50 50,clip]{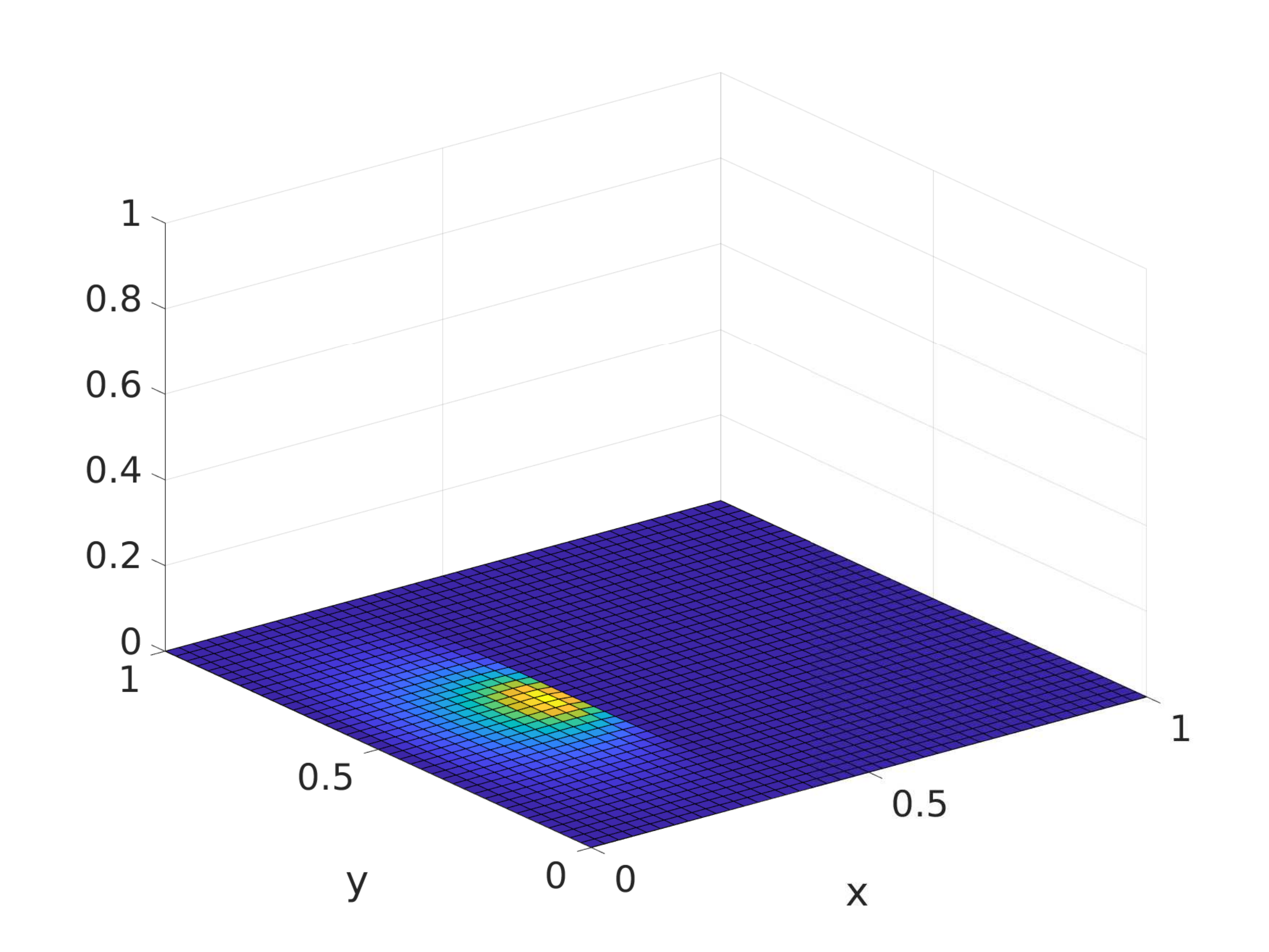}
    \includegraphics[width=0.33\textwidth,trim=50 30 50 50,clip]{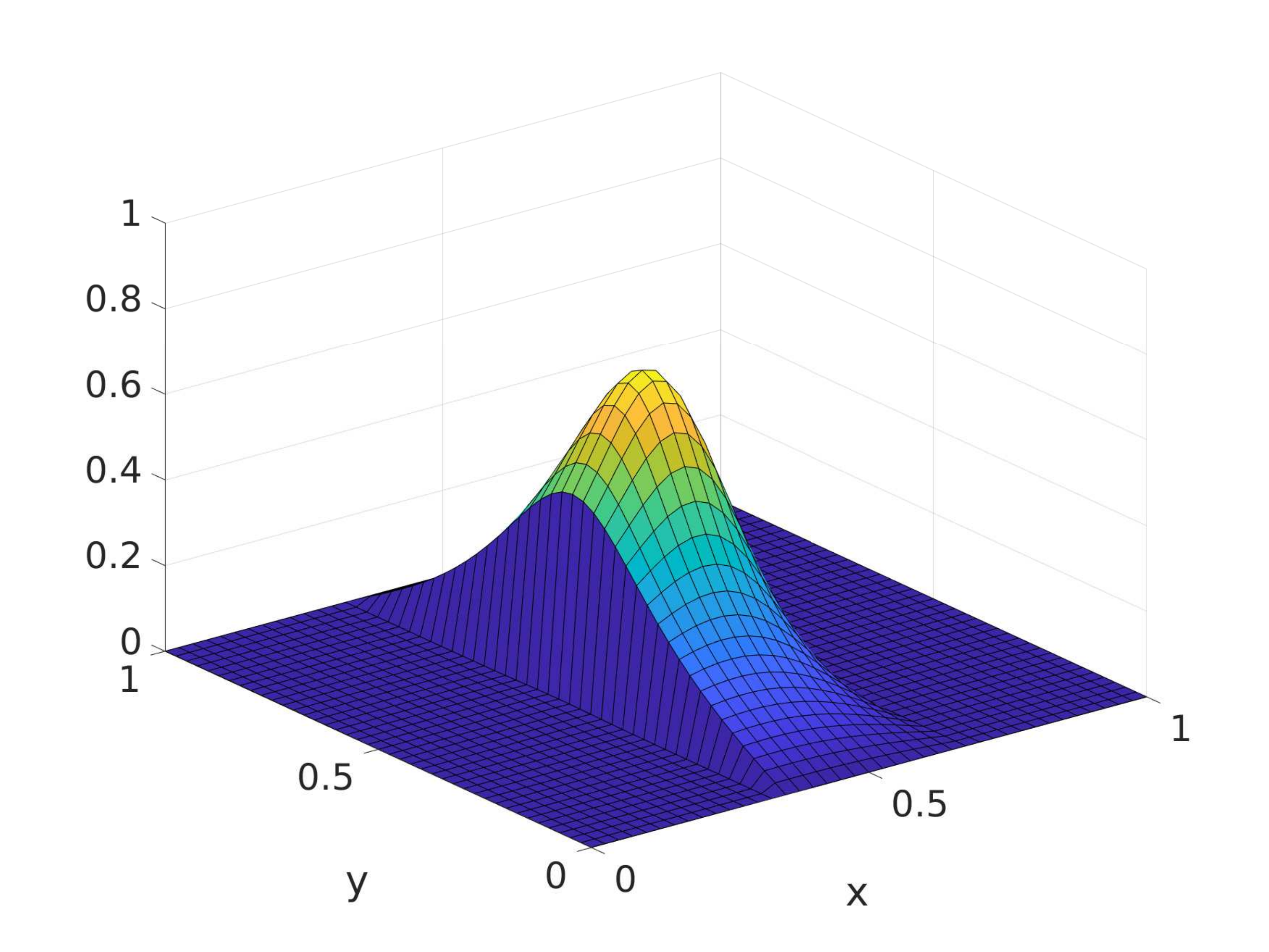}
  \includegraphics[width=0.33\textwidth,trim=50 30 50 50,clip]{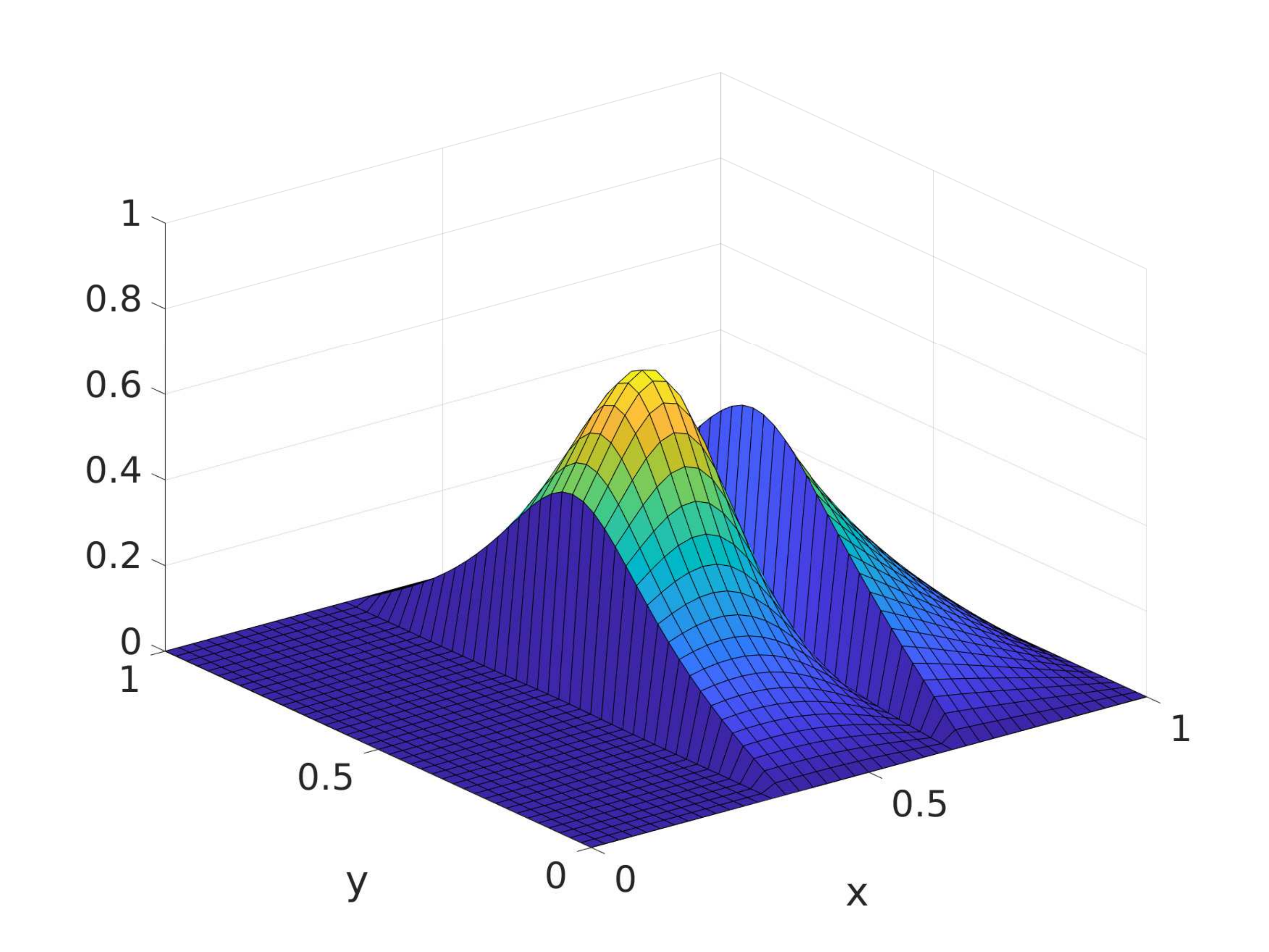}}
  \mbox{\includegraphics[width=0.33\textwidth,trim=50 30 50 50,clip]{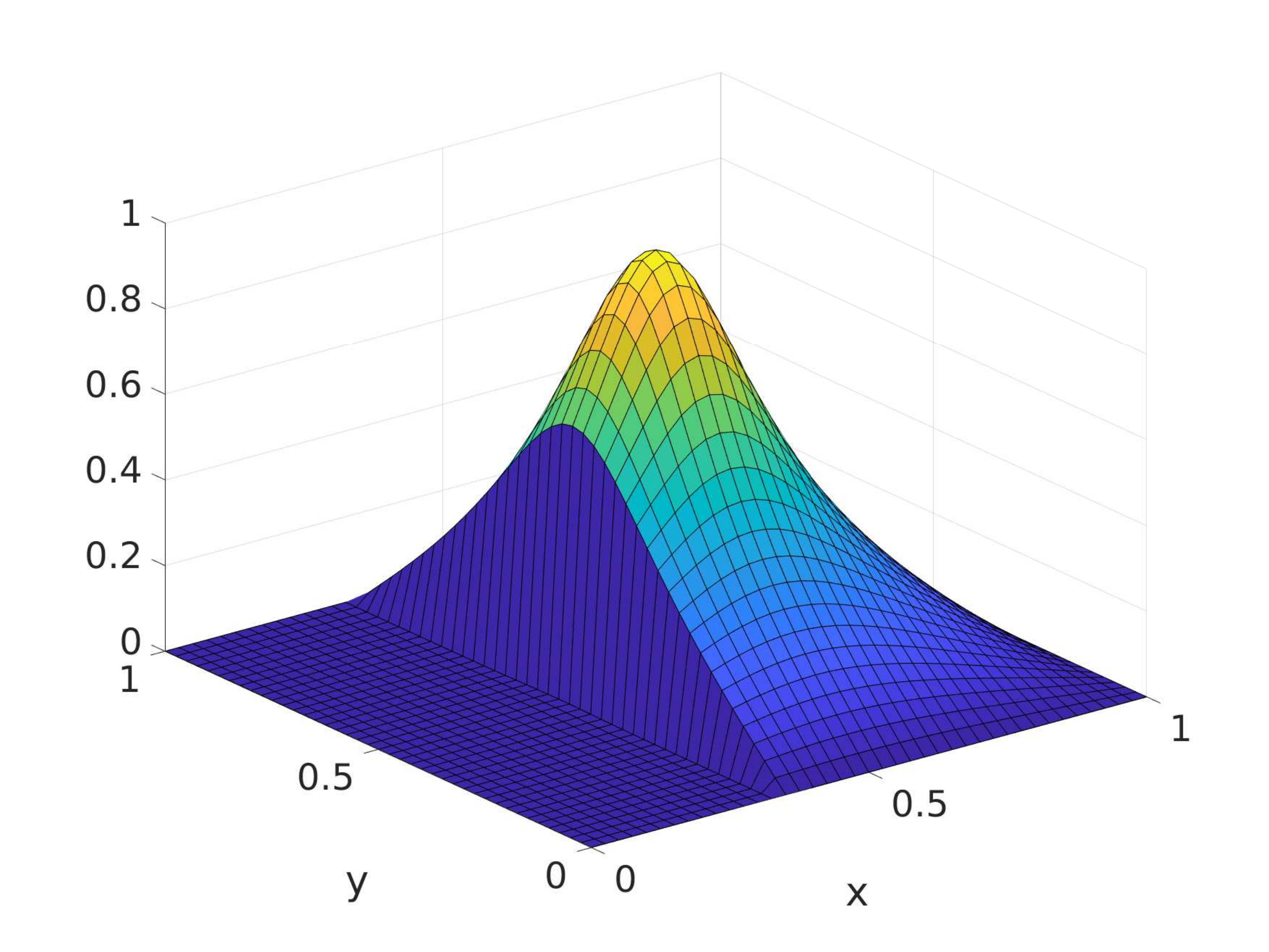}
  \includegraphics[width=0.33\textwidth,trim=50 30 50 50,clip]{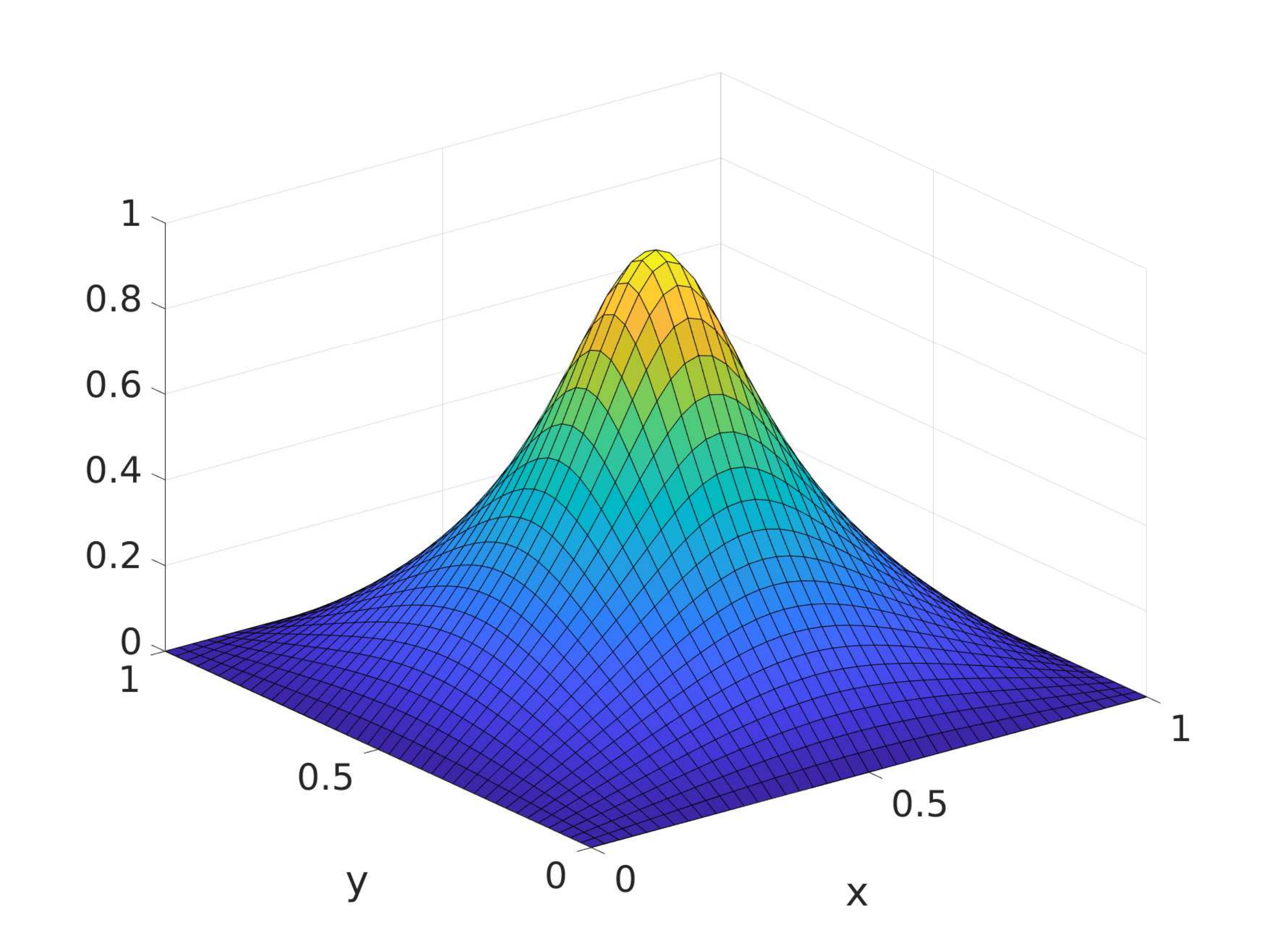}
  \includegraphics[width=0.33\textwidth,trim=50 30 50 50,clip]{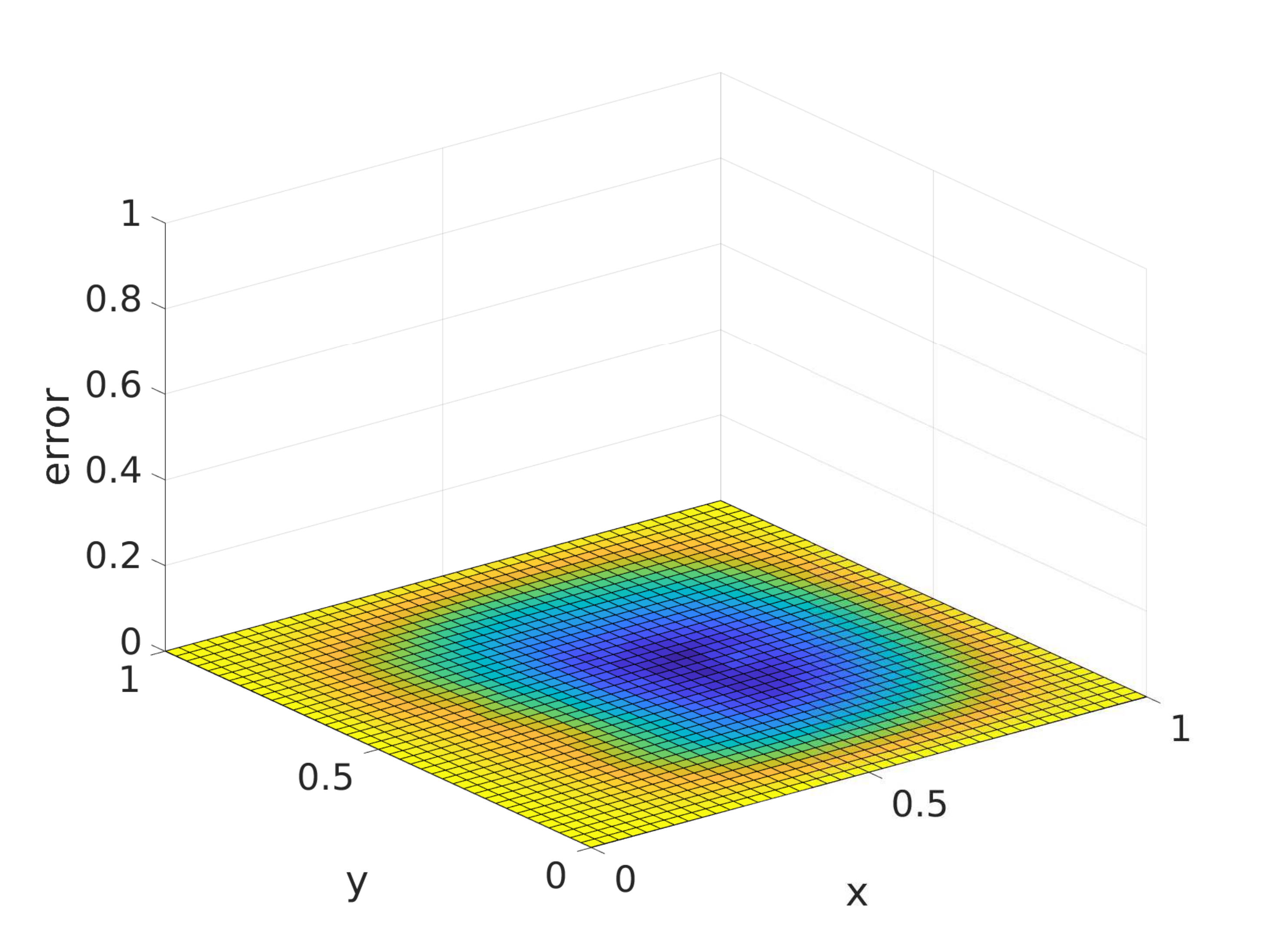}}
  \caption{Optimal alternating Schwarz using DtN
    transmission conditions sweeping with the iterates from left to
    right and then back to left (double sweep) for a sequential
    decomposition into three subdomains, solving the Poisson problem
    with Gaussian source term from Figure \ref{PoissonExampleFig}. The
    last figure shows the error after one double sweep, on the order of
    machine precision.}
  \label{PoissonOptimalSchwarzFig}
\end{figure}
We see that the method converges in one double sweep, \ie\ it is a
direct solver. The iteration matrix (or operator at the continuous
level) is nilpotent, and one can interpret this method as an exact
block LU factorization, where in the forward sweep the lower block
triangular matrix $L$ is solved, and in the backward sweep the
upper block triangular matrix $U$ is solved, and the blocks correspond
to the subdomains. This interpretation had already led earlier to the
Analytic Incomplete LU (AILU) preconditioners, see
\cn{GanderAILU00}, \cn{GanderAILU05} and references therein. Note that
other domain decomposition methods can also be nilpotent for
sequential domain decompositions: for Neumann-Neumann and FETI this is
however only possible for two-subdomain decompositions, and for
Dirichlet-Neumann up to three subdomains, and in some specific
cases more than three subdomains, see
\cn{chaouqui2017nilpotent}. The only domain decomposition method
which can in general become nilpotent for sequential domain
decompositions is the optimal Schwarz method.

If we use approximations of the optimal Schwarz method, for example an
optimized one with Robin transmission conditions, the method is not
nilpotent any more, as shown in Figure \ref{OSMSweepFig},
\begin{figure}
  \centering
  \mbox{\includegraphics[width=0.33\textwidth,trim=50 30 50 50,clip]{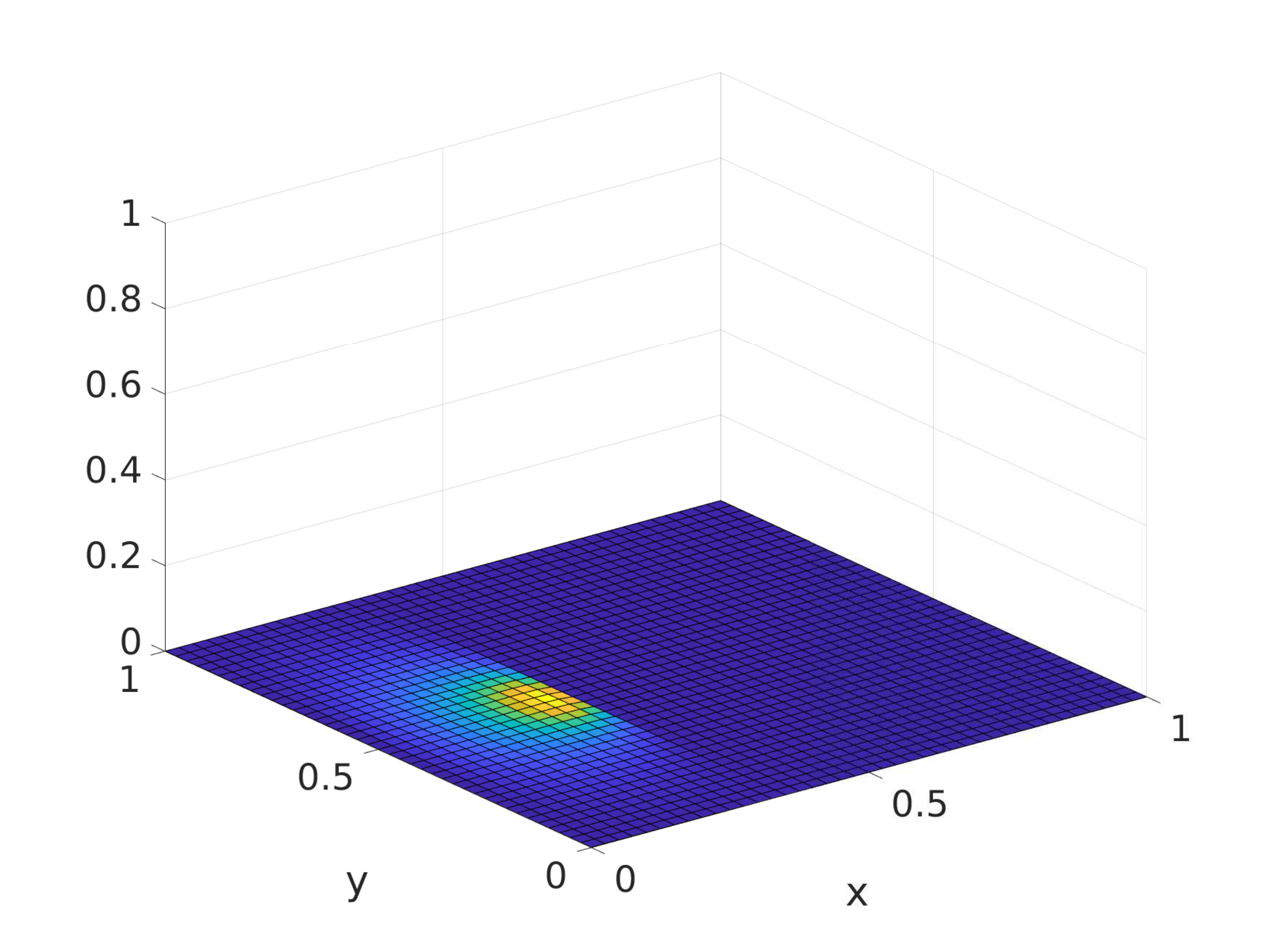}
    \includegraphics[width=0.33\textwidth,trim=50 30 50 50,clip]{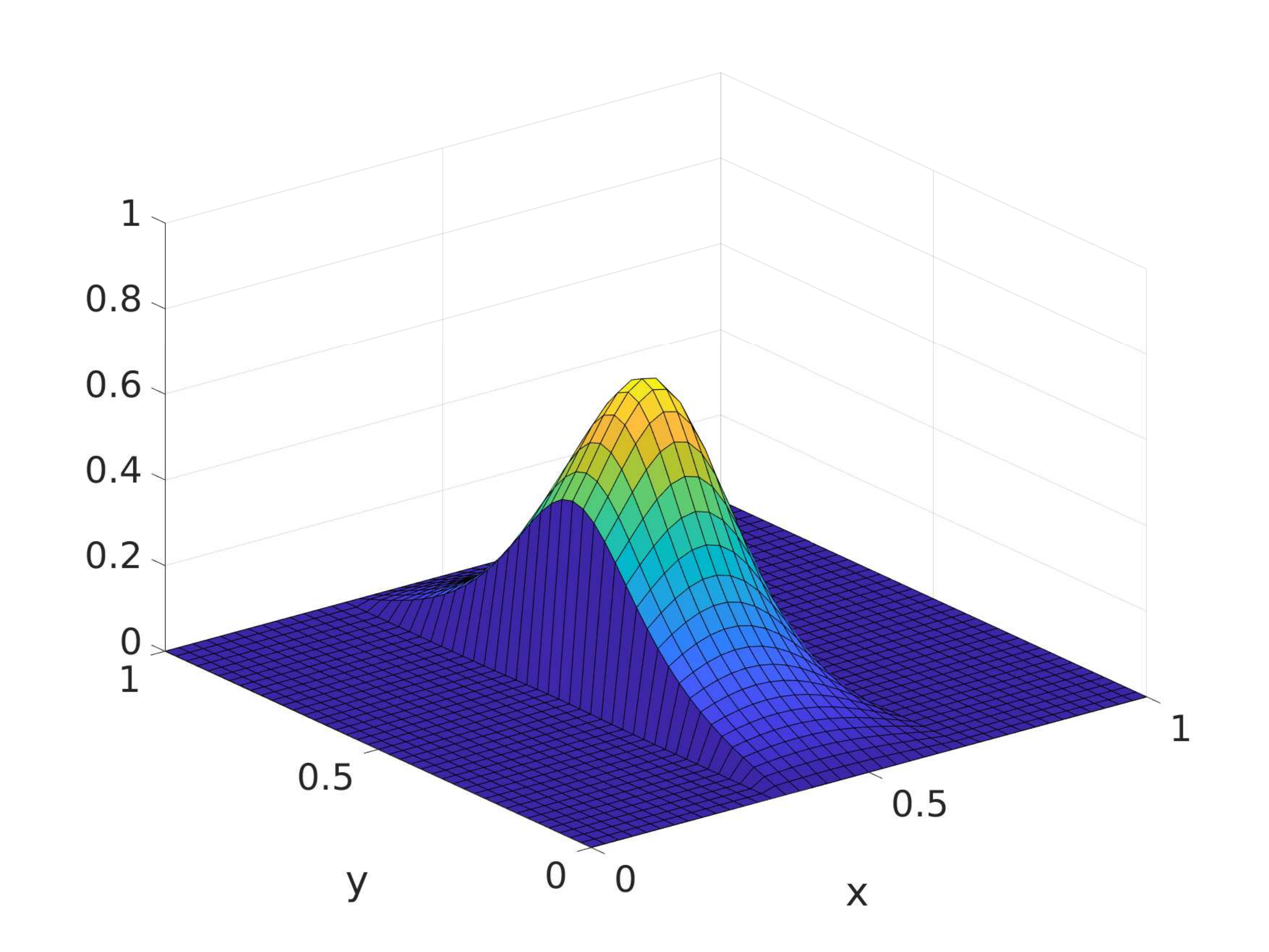}
  \includegraphics[width=0.33\textwidth,trim=50 30 50 50,clip]{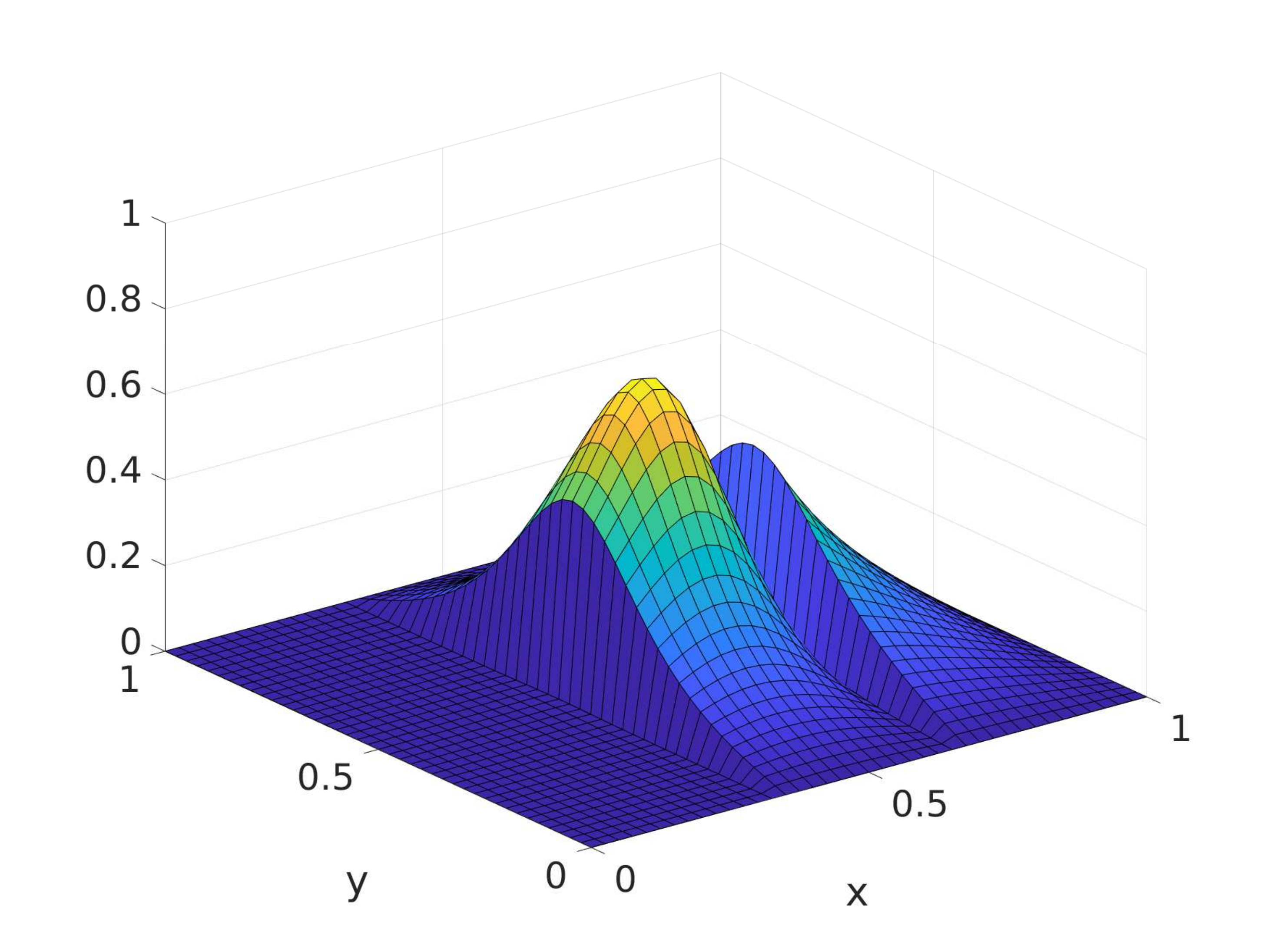}}
  \mbox{\includegraphics[width=0.33\textwidth,trim=50 30 50 50,clip]{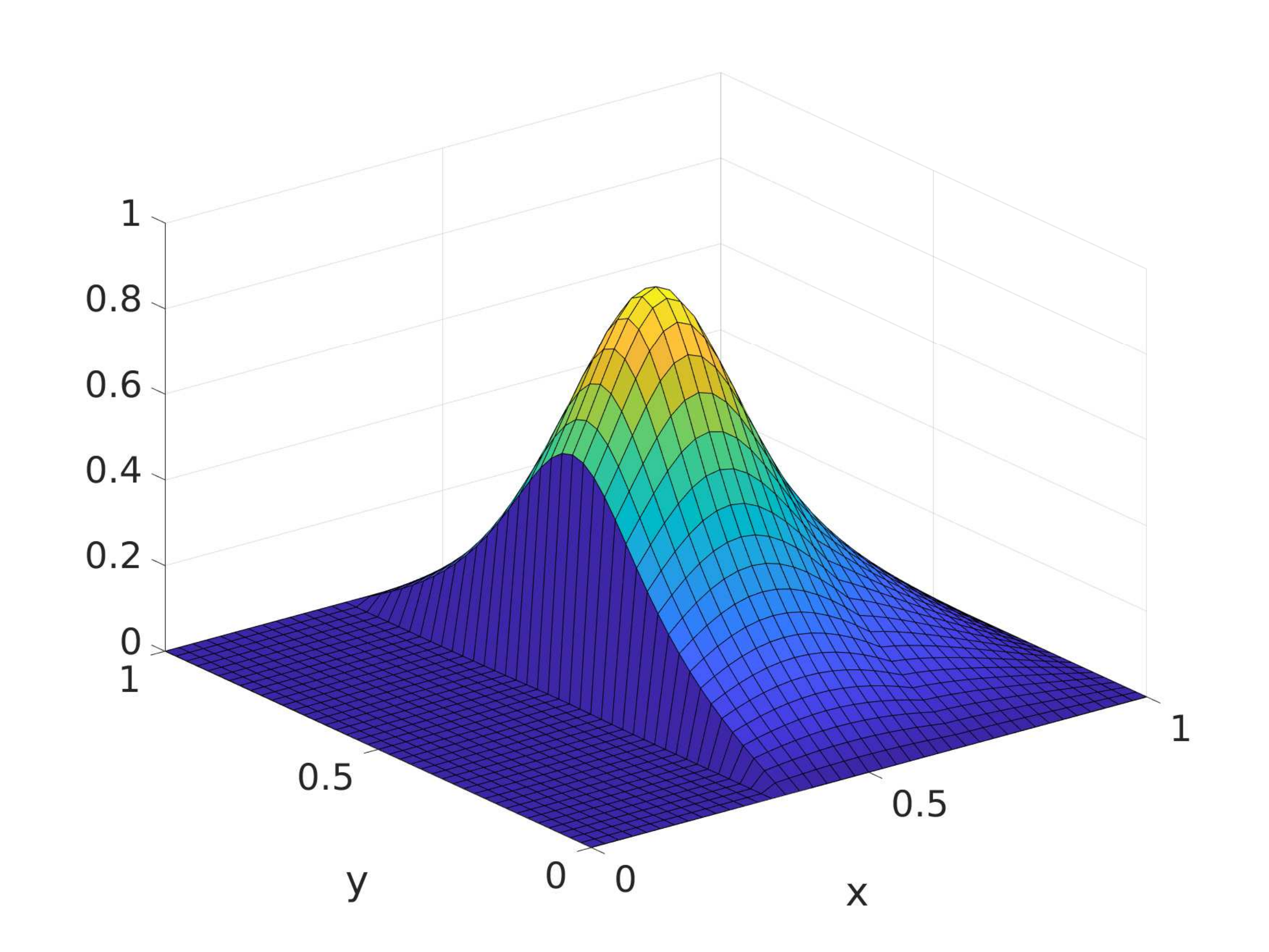}
  \includegraphics[width=0.33\textwidth,trim=50 30 50 50,clip]{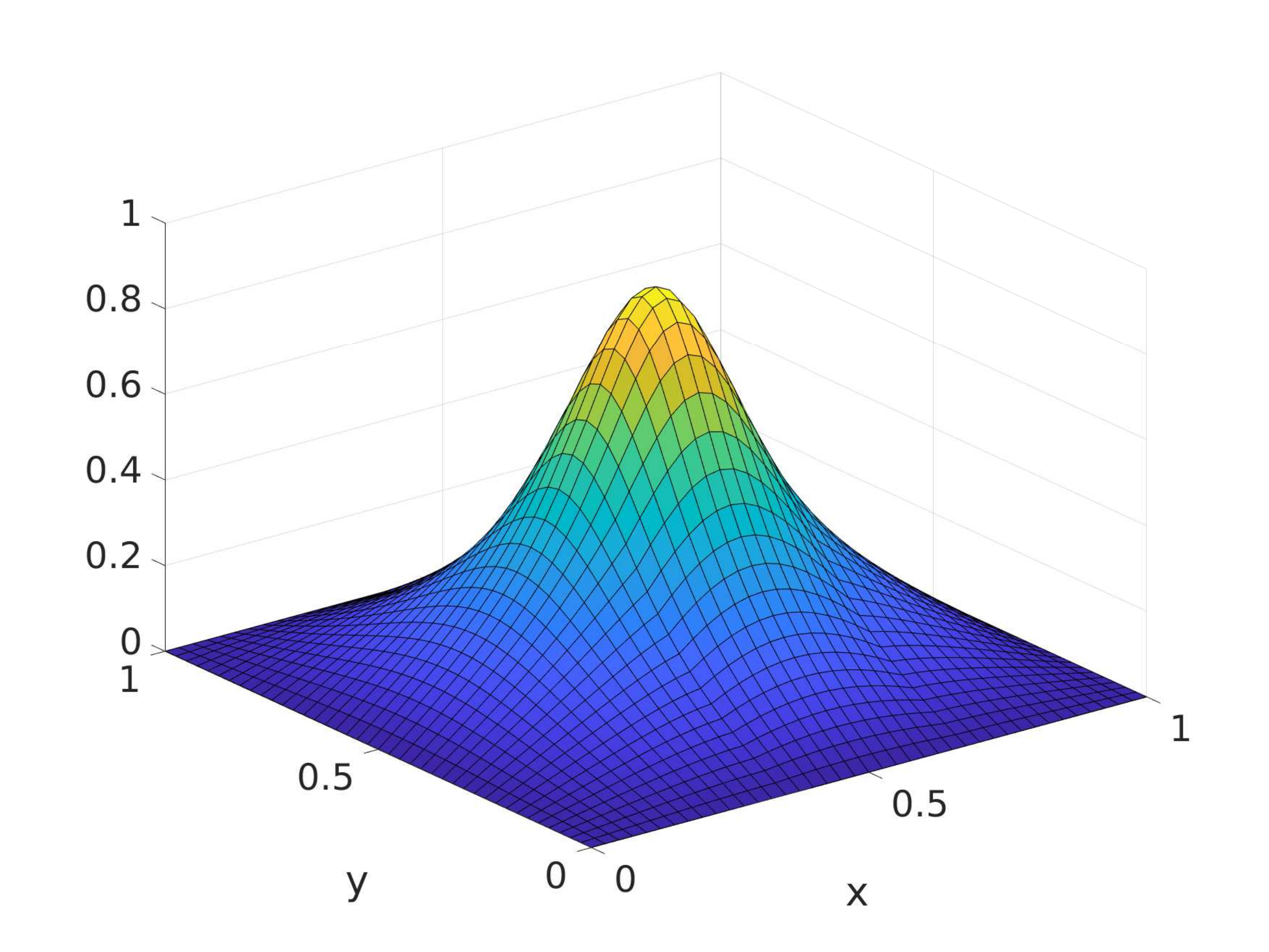}
  \includegraphics[width=0.33\textwidth,trim=50 30 50 50,clip]{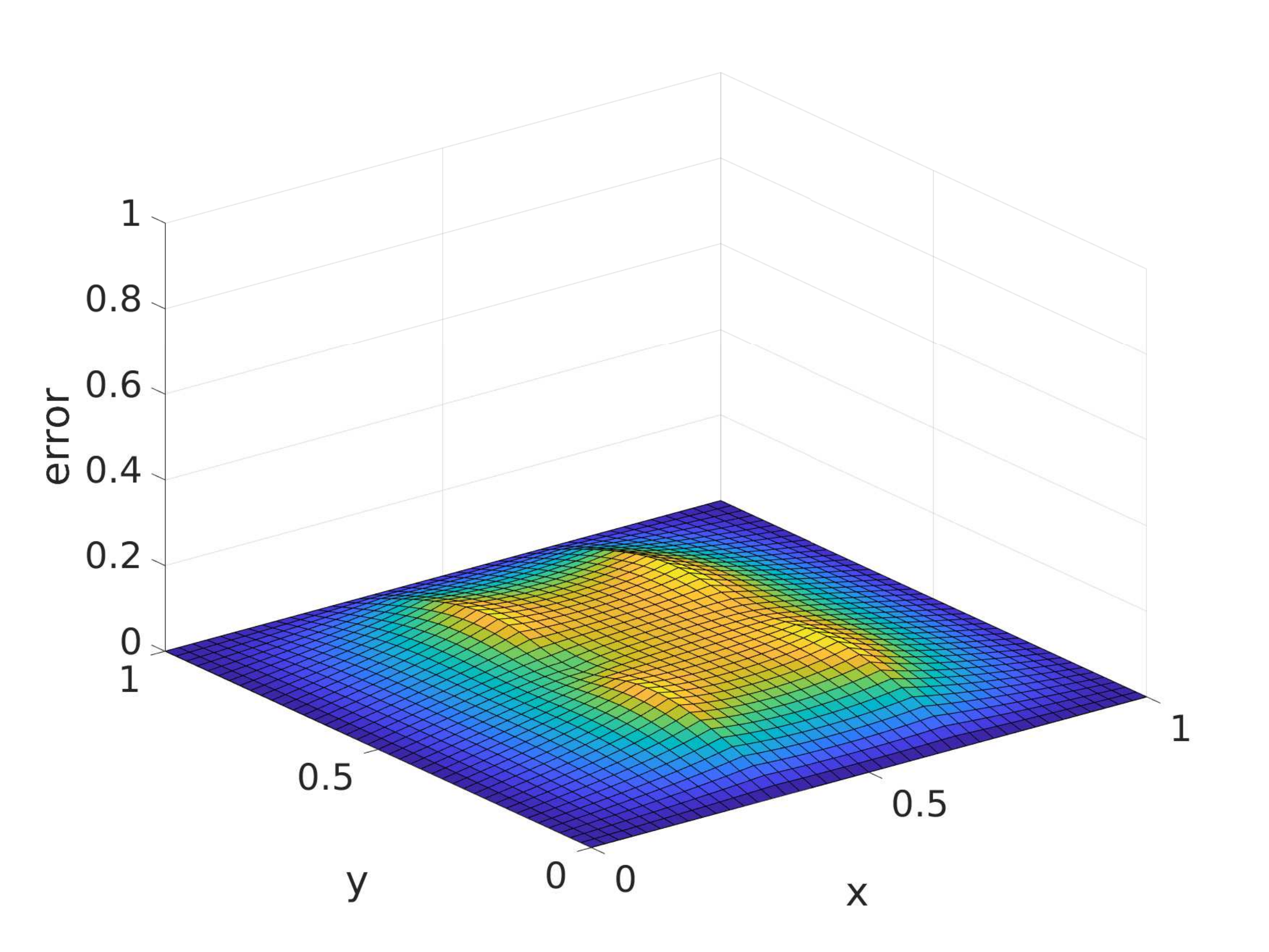}}
  \caption{Optimized alternating Schwarz using Robin transmission
    conditions sweeping like optimal alternating Schwarz in Figure
    \ref{PoissonOptimalSchwarzFig} (the last figure shows the error
    after one double sweep).}
  \label{OSMSweepFig}
\end{figure}
but convergence is still very fast, compared to the classical Schwarz
method with Dirichlet transmission conditions shown in Figure
\ref{SMSweepFig}.
\begin{figure}
  \centering
  \centering
  \mbox{\includegraphics[width=0.33\textwidth,trim=50 30 50 50,clip]{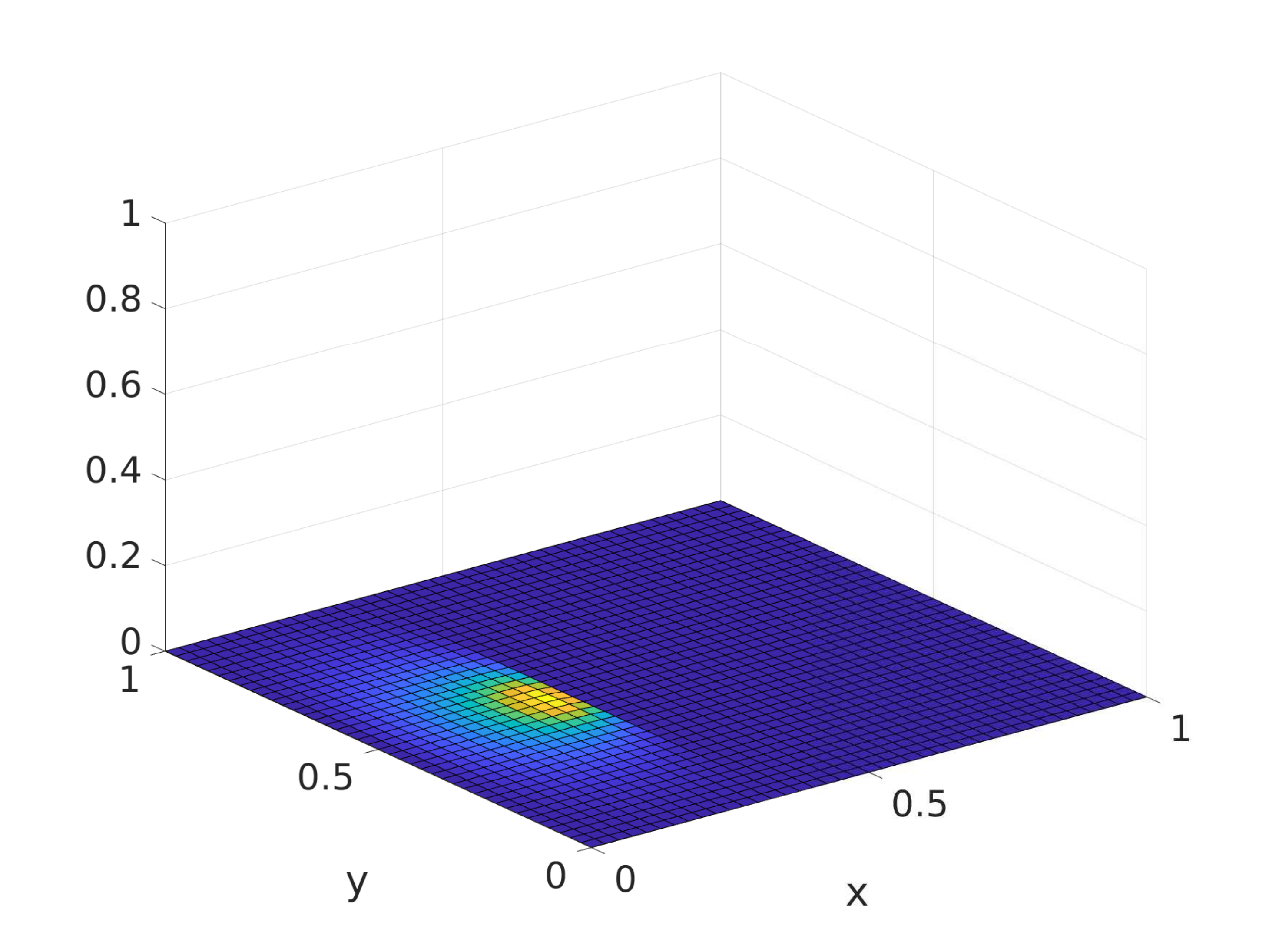}
    \includegraphics[width=0.33\textwidth,trim=50 30 50 50,clip]{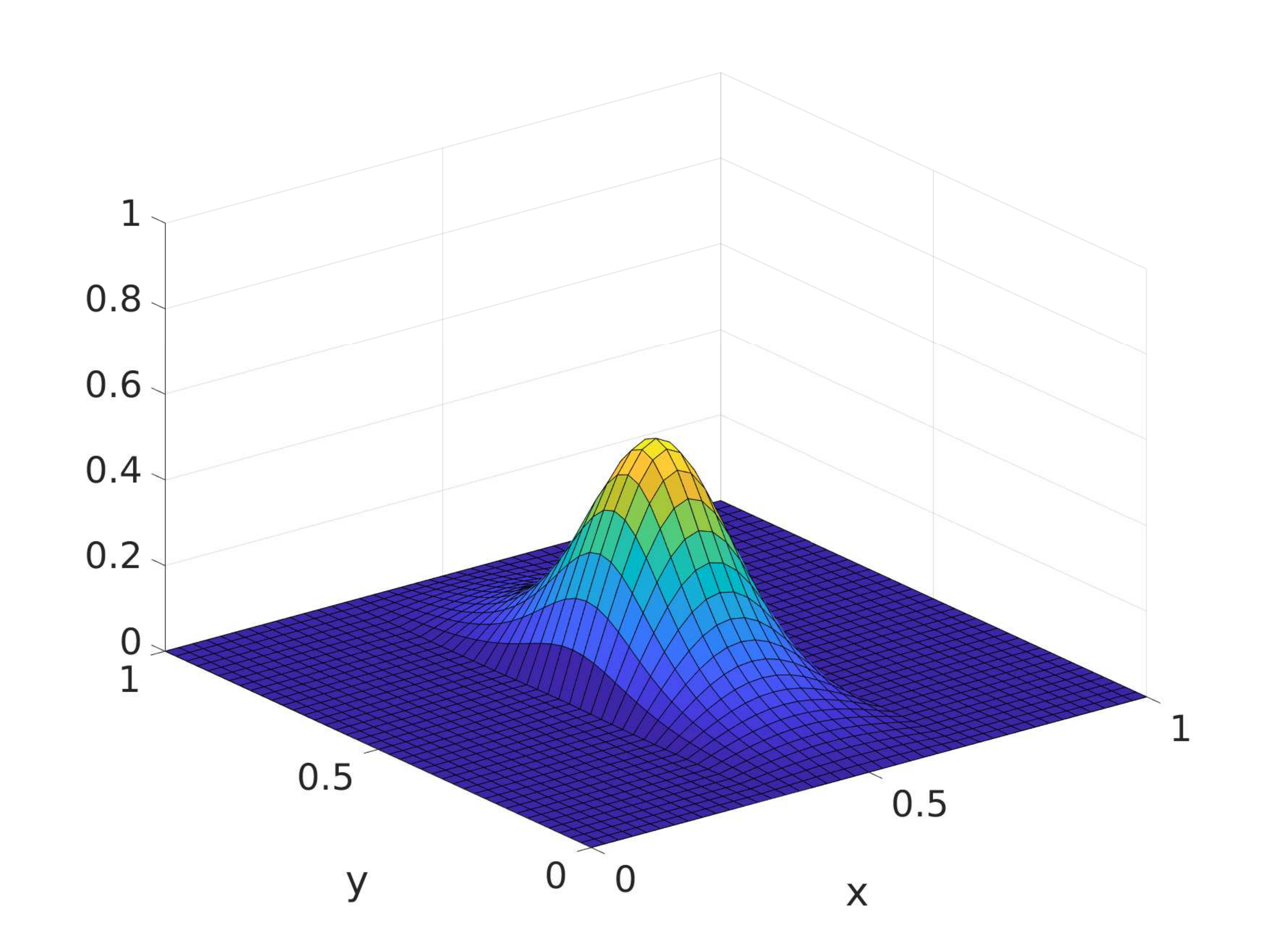}
  \includegraphics[width=0.33\textwidth,trim=50 30 50 50,clip]{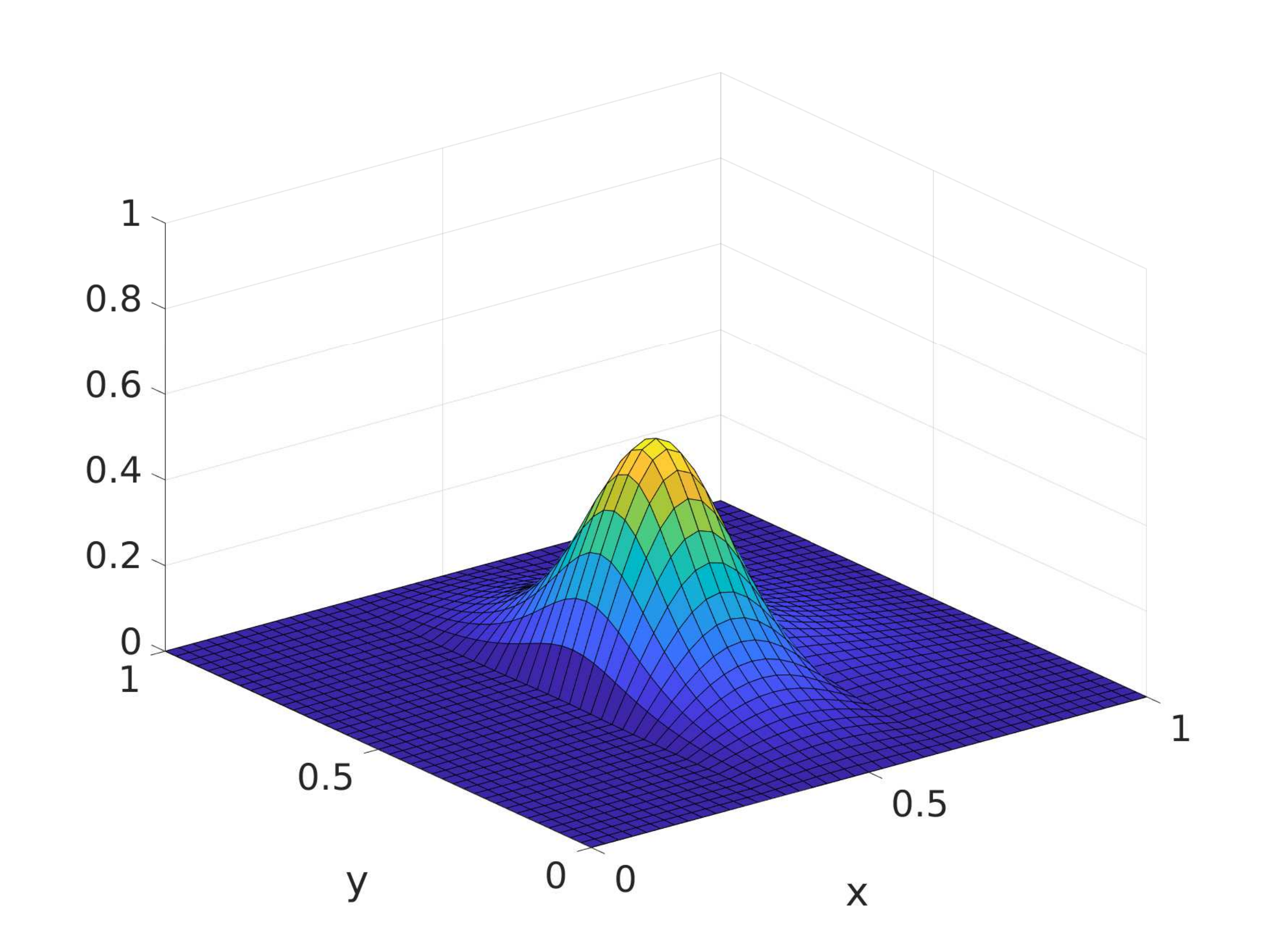}}
  \mbox{\includegraphics[width=0.33\textwidth,trim=50 30 50 50,clip]{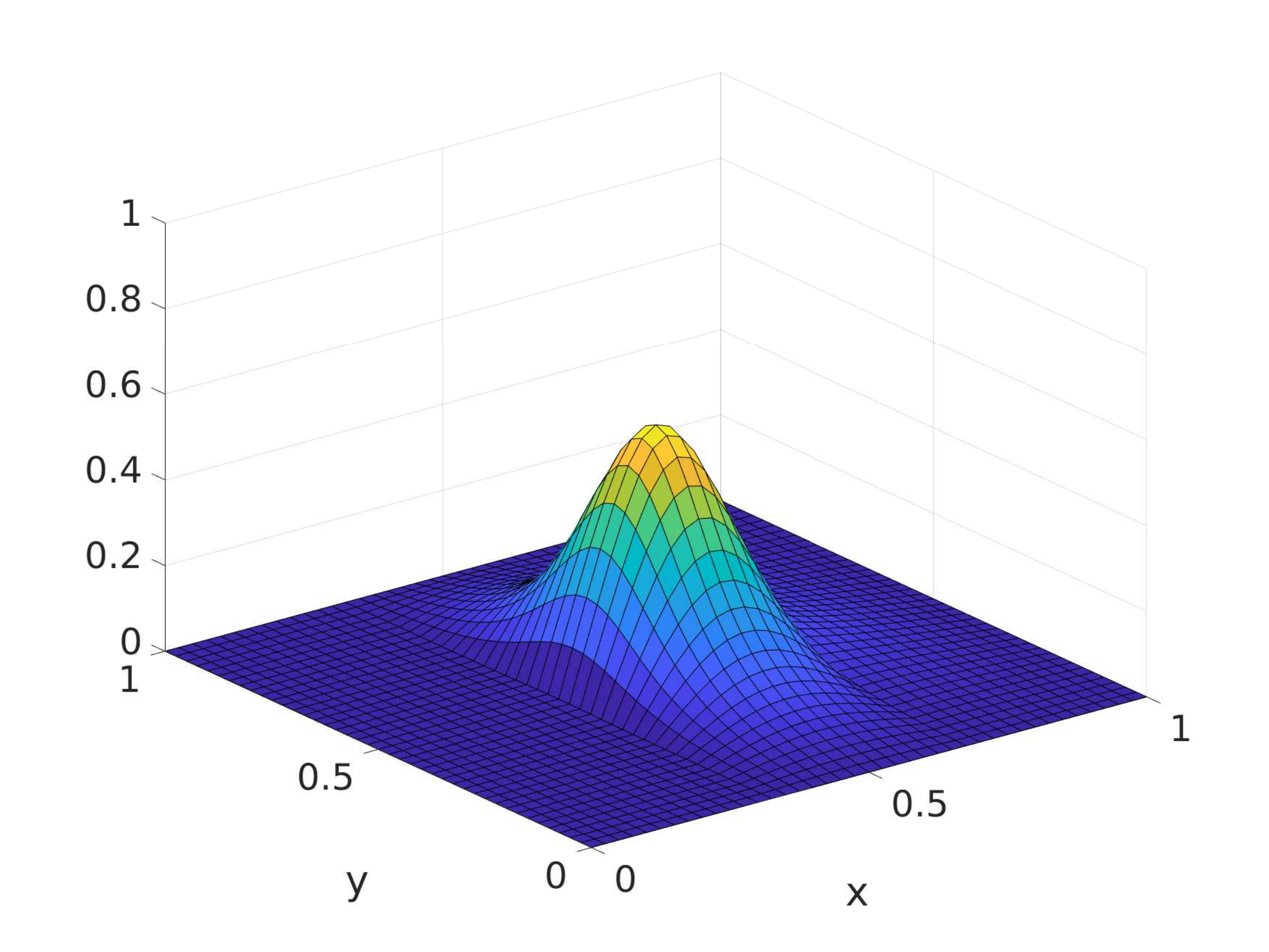}
  \includegraphics[width=0.33\textwidth,trim=50 30 50 50,clip]{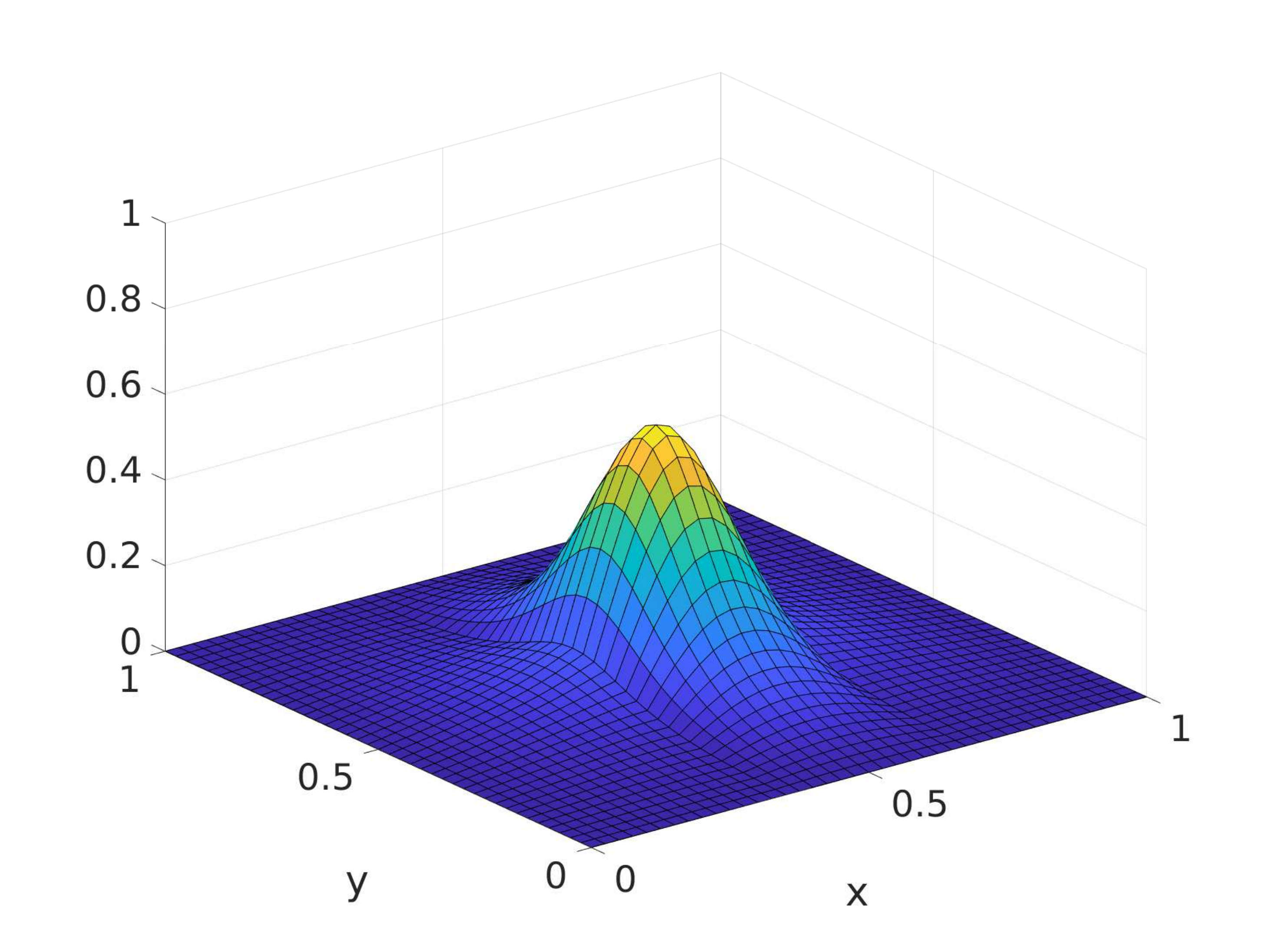}
  \includegraphics[width=0.33\textwidth,trim=50 30 50 50,clip]{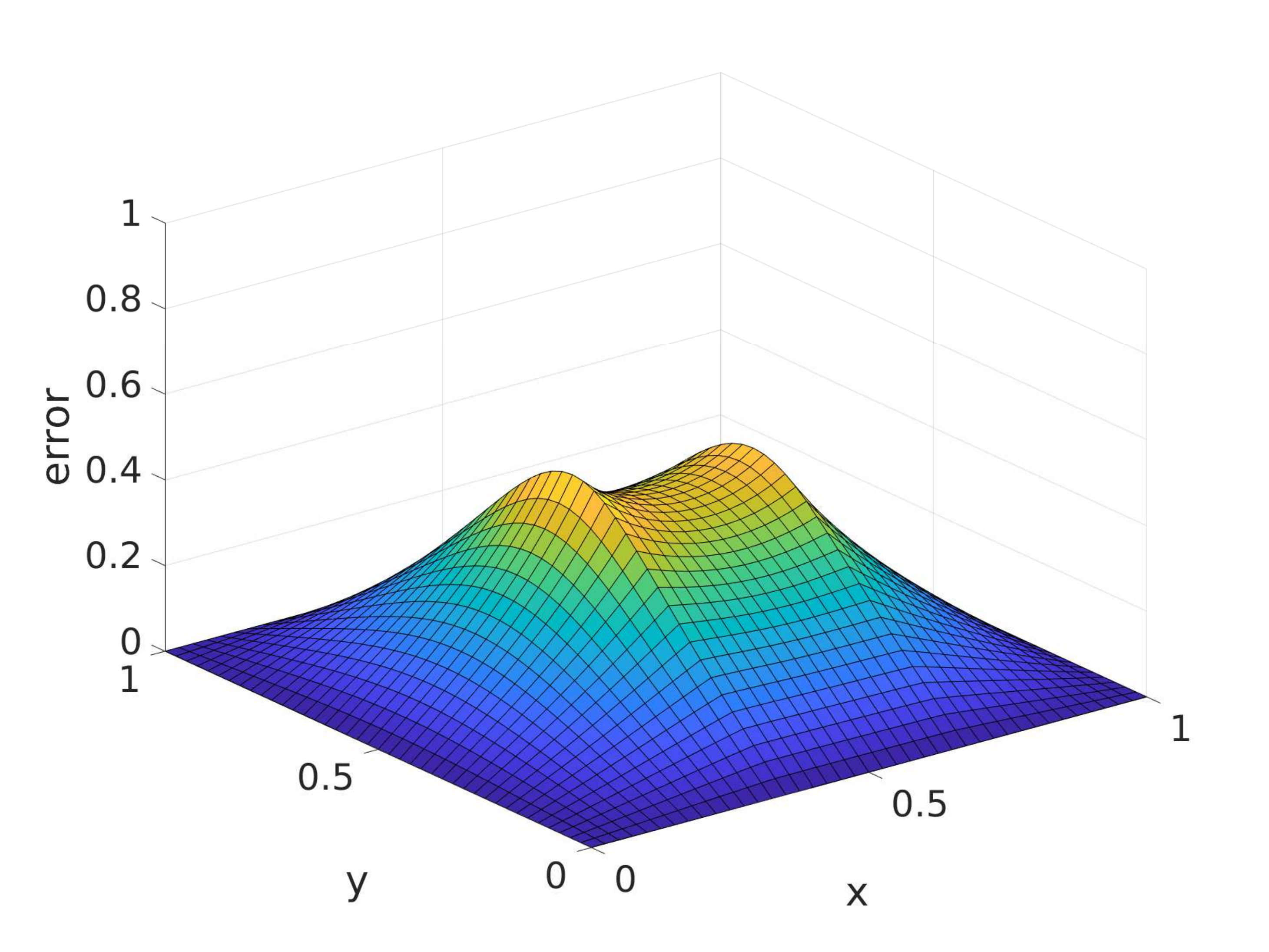}}
  \caption{Classical alternating Schwarz (using Dirichlet transmission
    conditions) sweeping like optimized alternating Schwarz in
    Figure \ref{OSMSweepFig} (the last figure shows the error after one
    double sweep).}
  \label{SMSweepFig}
\end{figure}

Using PML as transmission conditions in optimized Schwarz methods to
approximate the optimal DtN transmission conditions is currently a
very active field of research in the case of cross points.
\cn{leng2019additive} proposed a parallel Schwarz method for
checkerboard domain decompositions including cross points based on the
source transfer formalism which uses PML transmission conditions, and
the method still converges in a finite number of steps, see also
\cn{lengdiag} for a diagonal sweeping variant, and the earlier work
\cite{leng2015overlapping,leng2015fast}. In \cn{taus2020sweeps},
L-sweeps are proposed for the method of polarized traces
interpretation of the optimized Schwarz methods, which traverse a
square domain decomposed into little squares by subdomain solves
organized in the form of L's, from one corner going over the
domain. In these methods, storing the PML layers and using them in the
exchange of information is an essential ingredient, and it is not
clear at this stage if such a formulation using DtN transmission
conditions is possible.

It is however possible to use the block LU decomposition also in the
case of general decompositions including cross points, to obtain a
sweeping domain decomposition method which converges in one double
sweep. We show this in Figure \ref{OptimalSchwarzLUSweep}
\begin{figure}
  \centering
  \mbox{\includegraphics[width=0.33\textwidth,trim=50 20 50 50,clip]{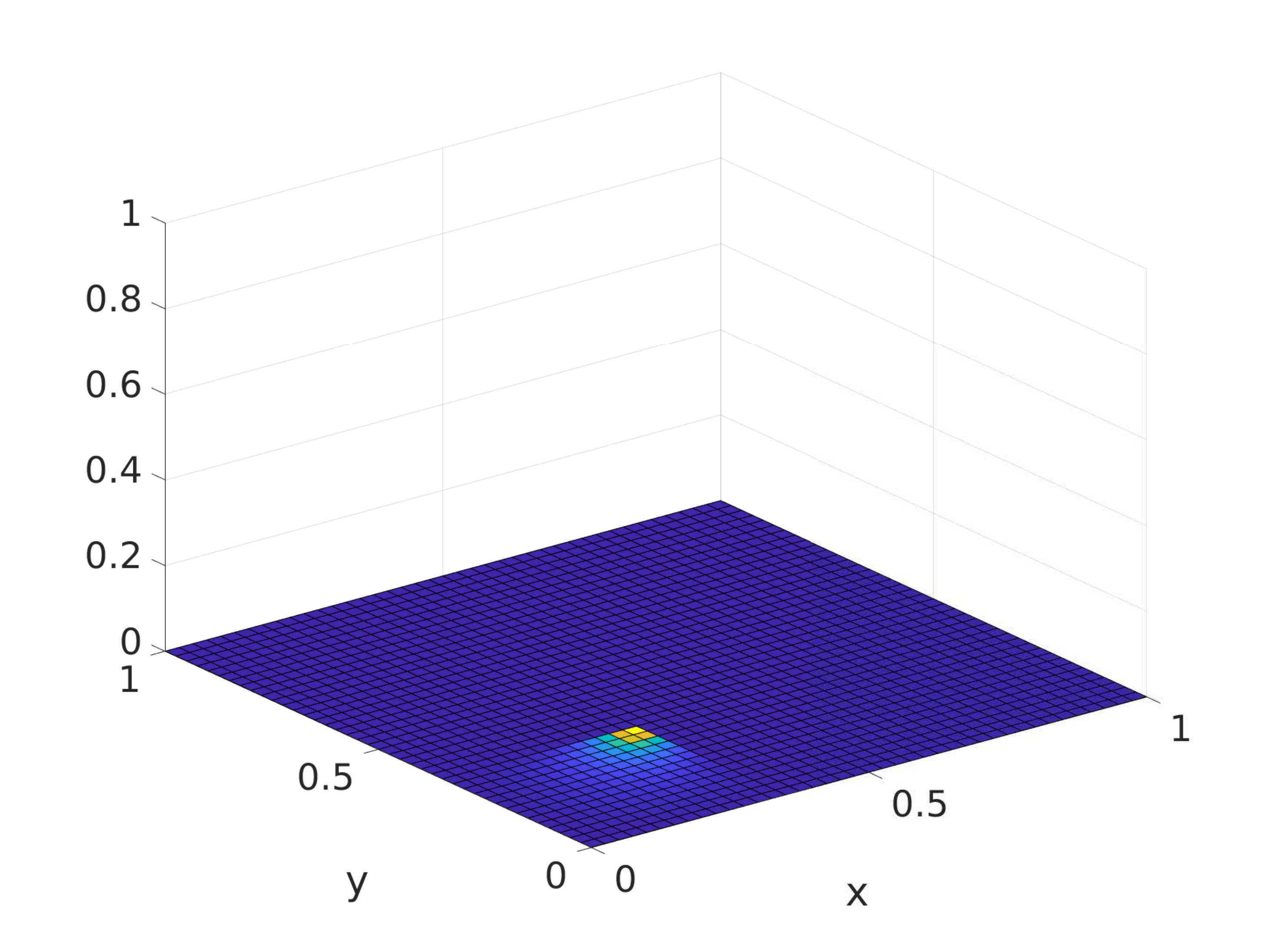}
\includegraphics[width=0.33\textwidth,trim=50 20 50 50,clip]{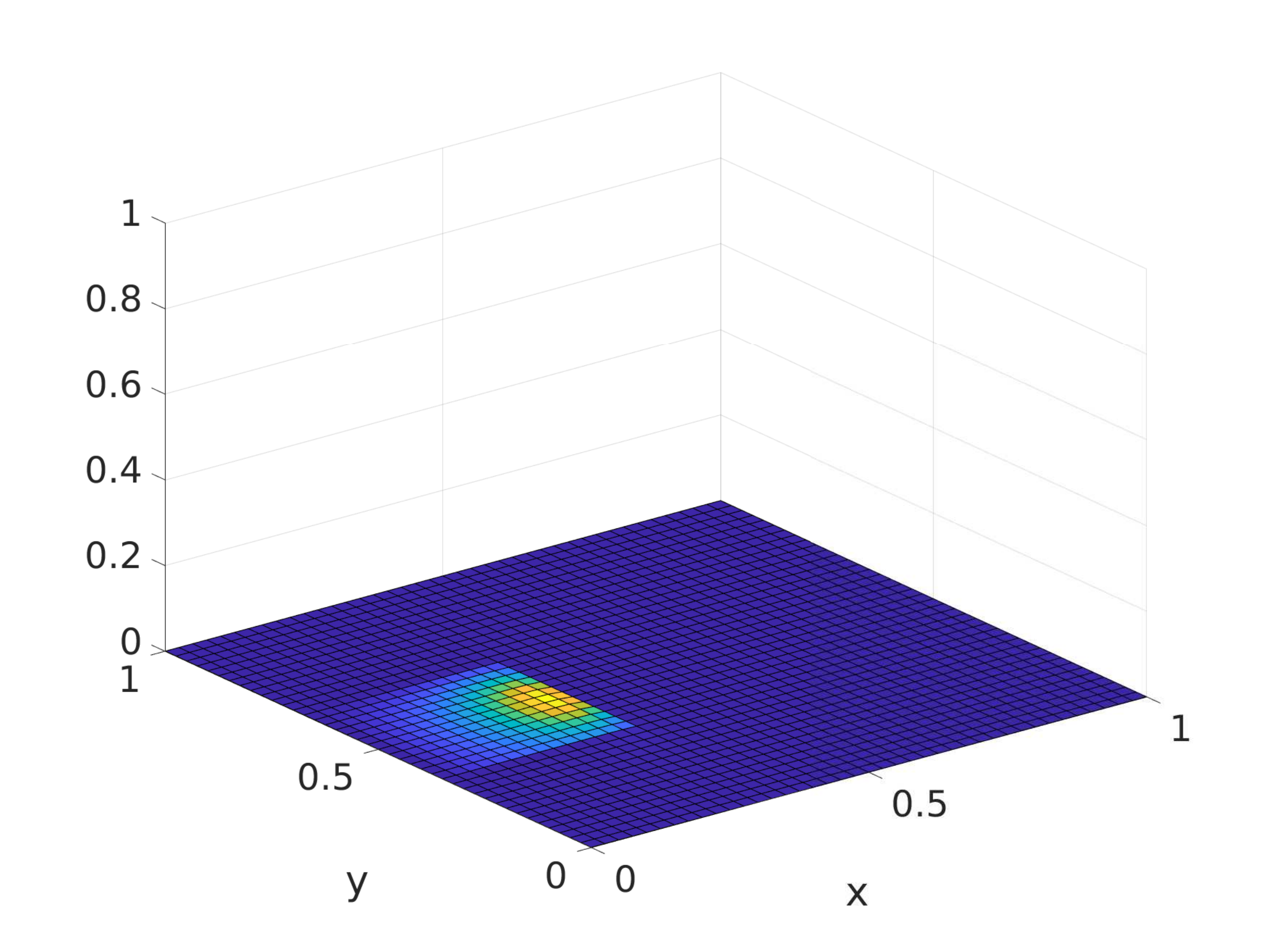}
\includegraphics[width=0.33\textwidth,trim=50 20 50 50,clip]{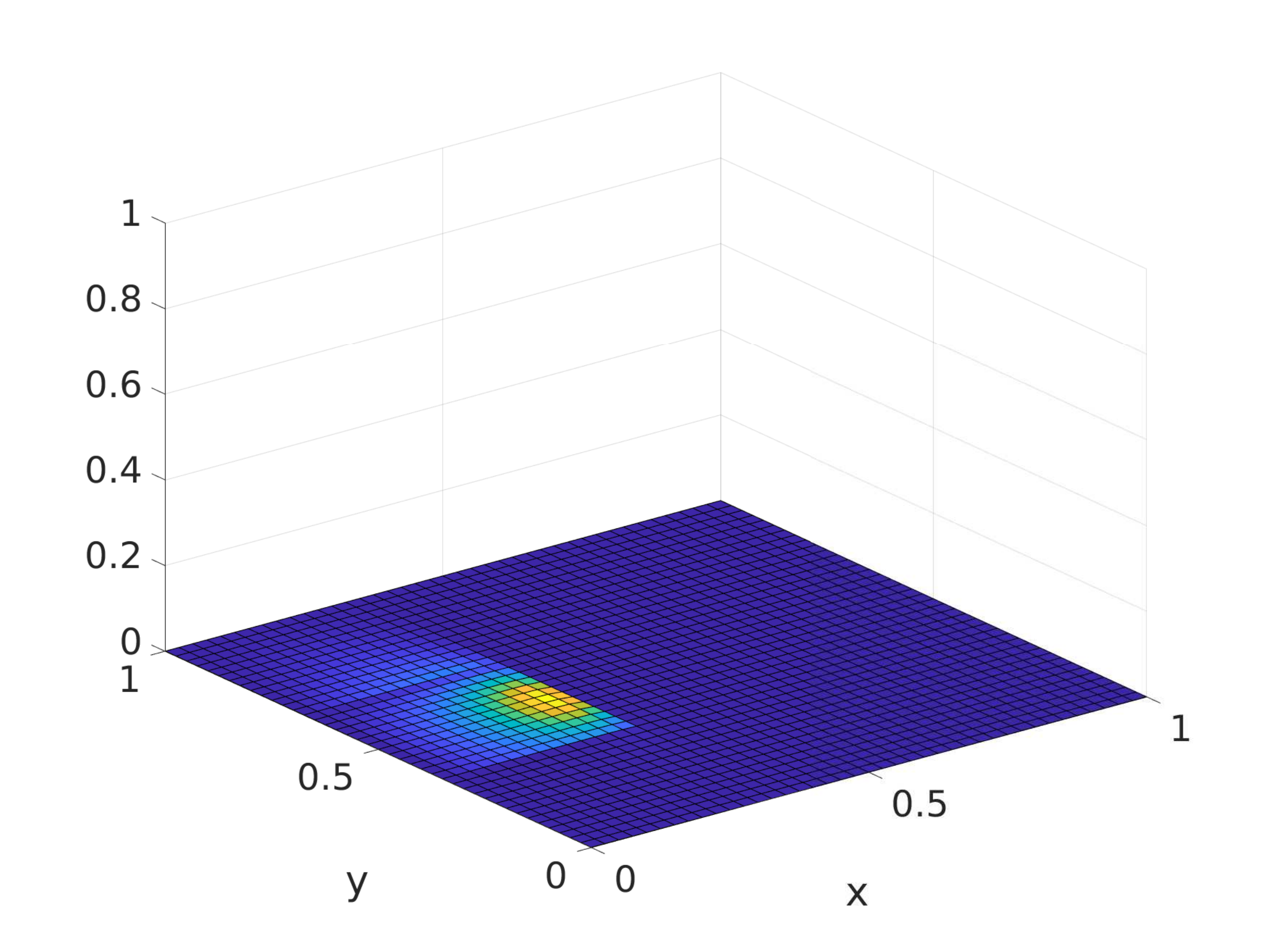}}
  \mbox{\includegraphics[width=0.33\textwidth,trim=50 20 50 50,clip]{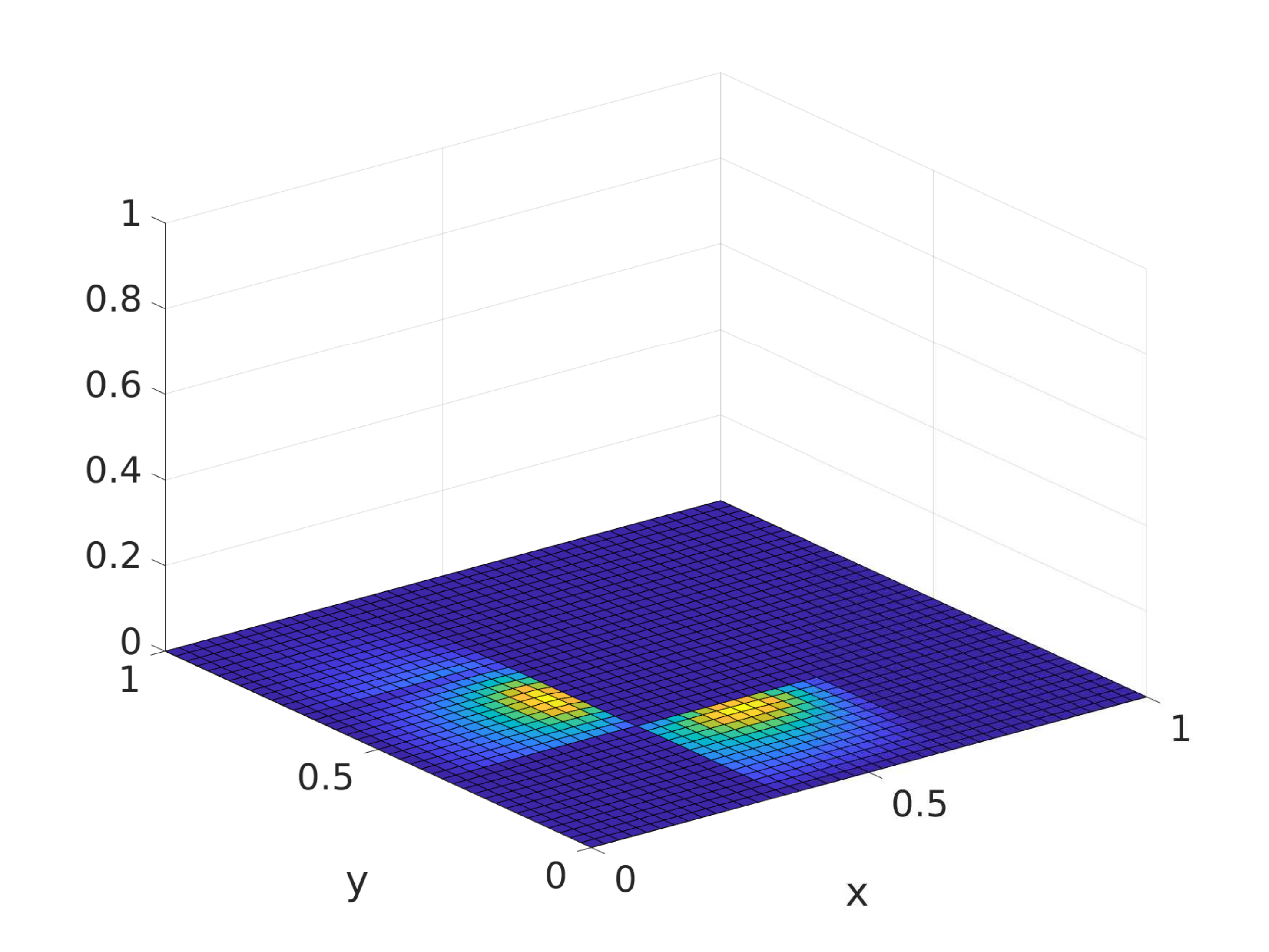}
\includegraphics[width=0.33\textwidth,trim=50 20 50 50,clip]{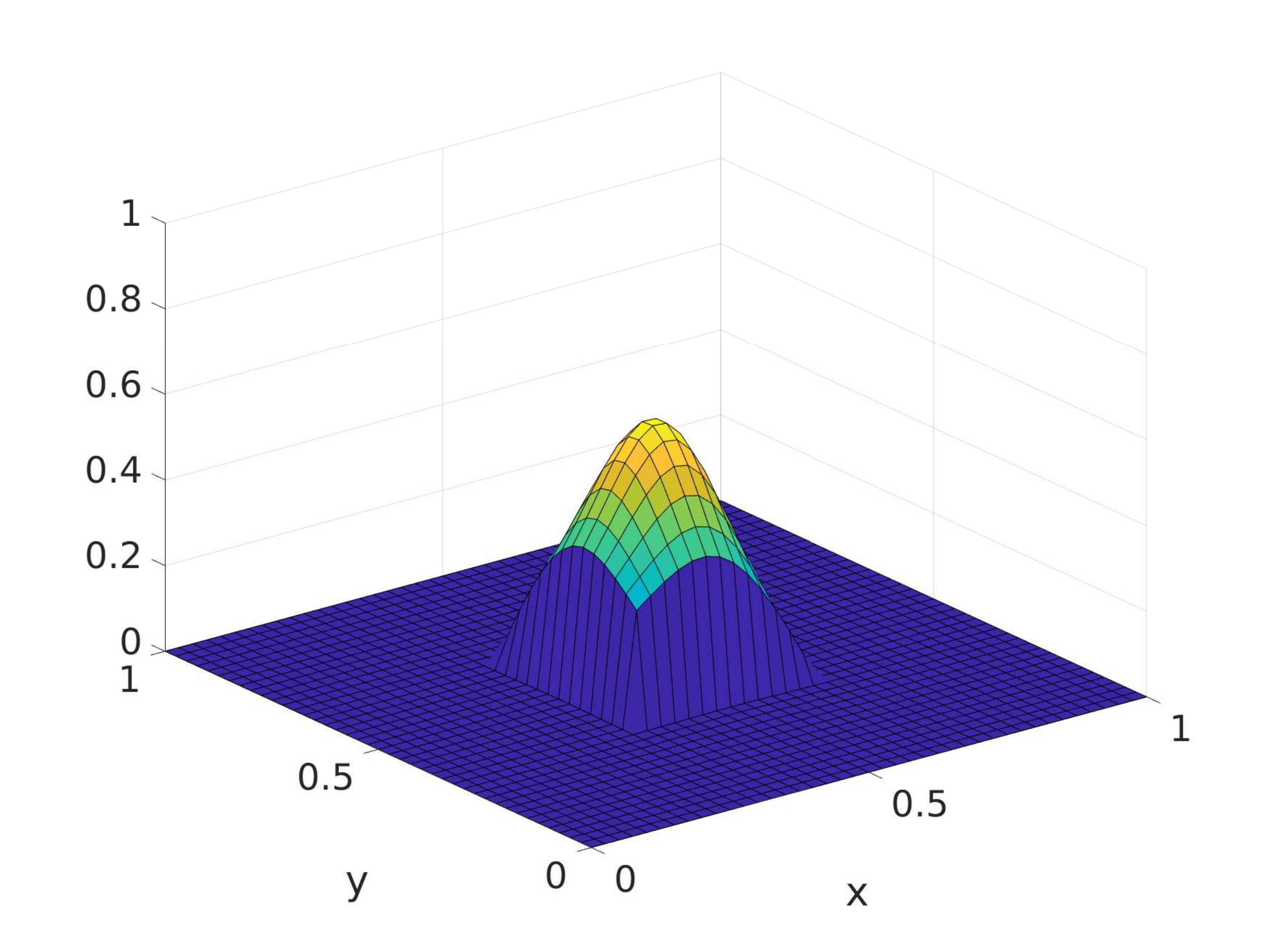}
\includegraphics[width=0.33\textwidth,trim=50 20 50 50,clip]{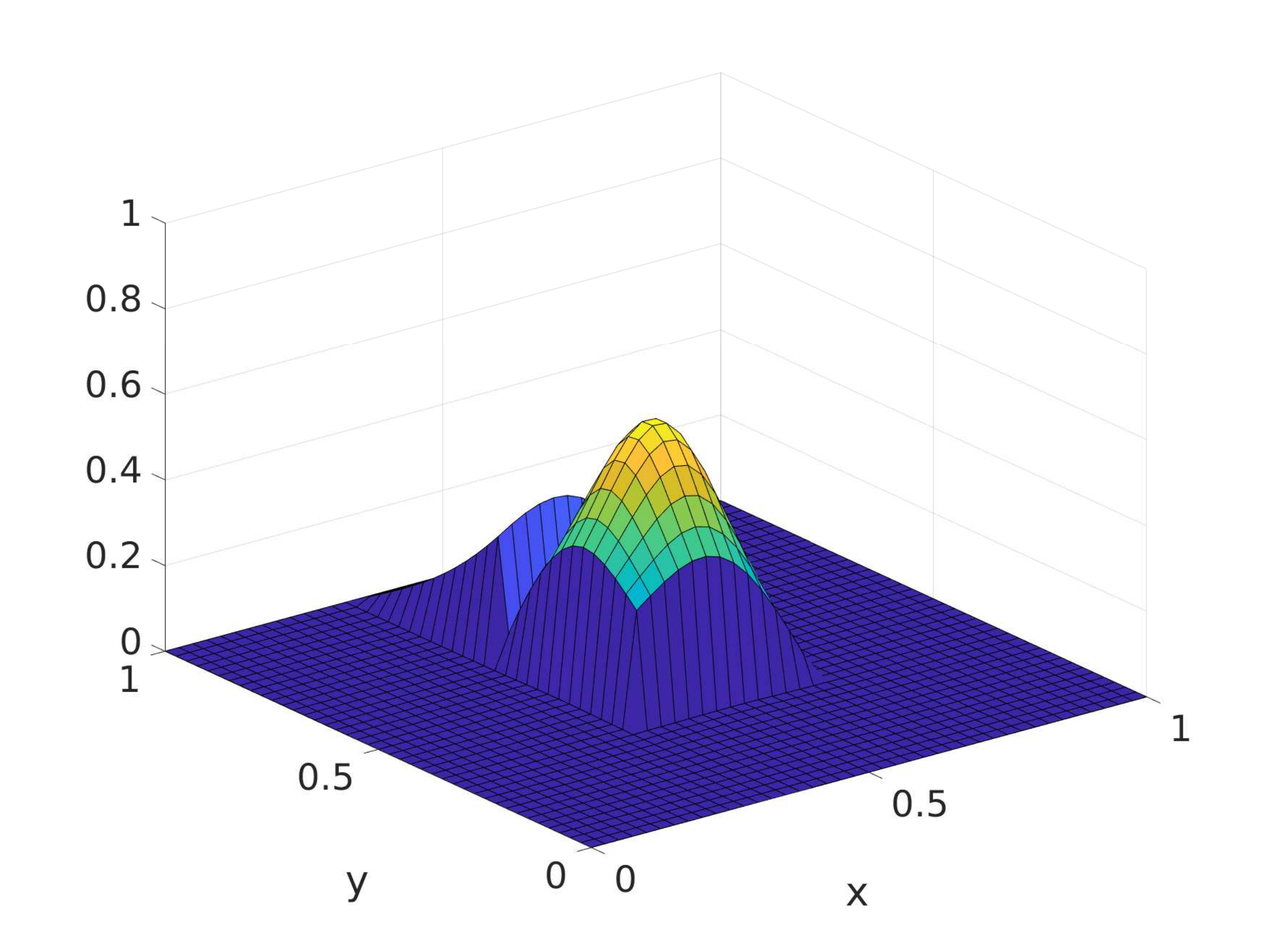}}
  \mbox{\includegraphics[width=0.33\textwidth,trim=50 20 50 50,clip]{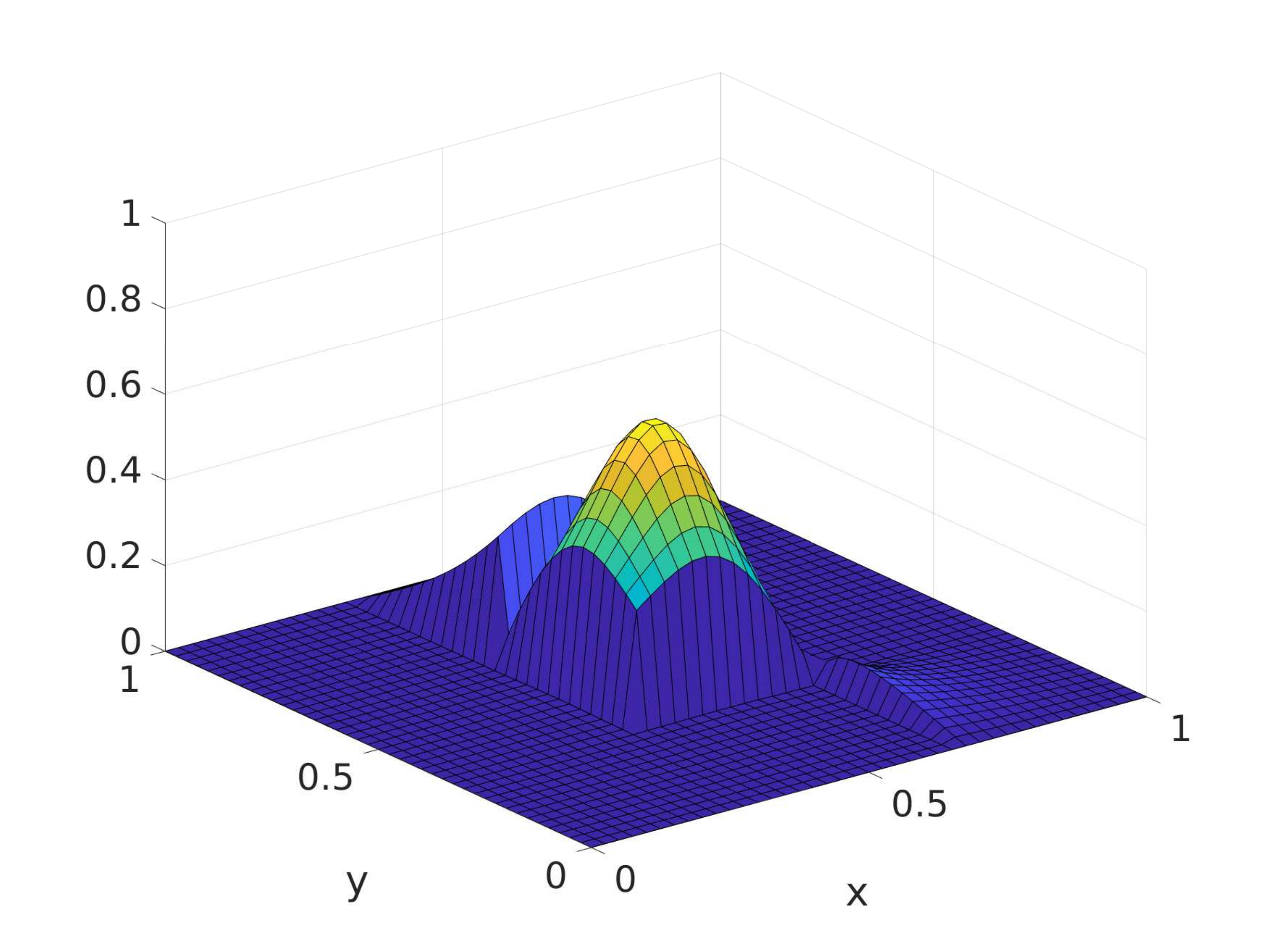}
\includegraphics[width=0.33\textwidth,trim=50 20 50 50,clip]{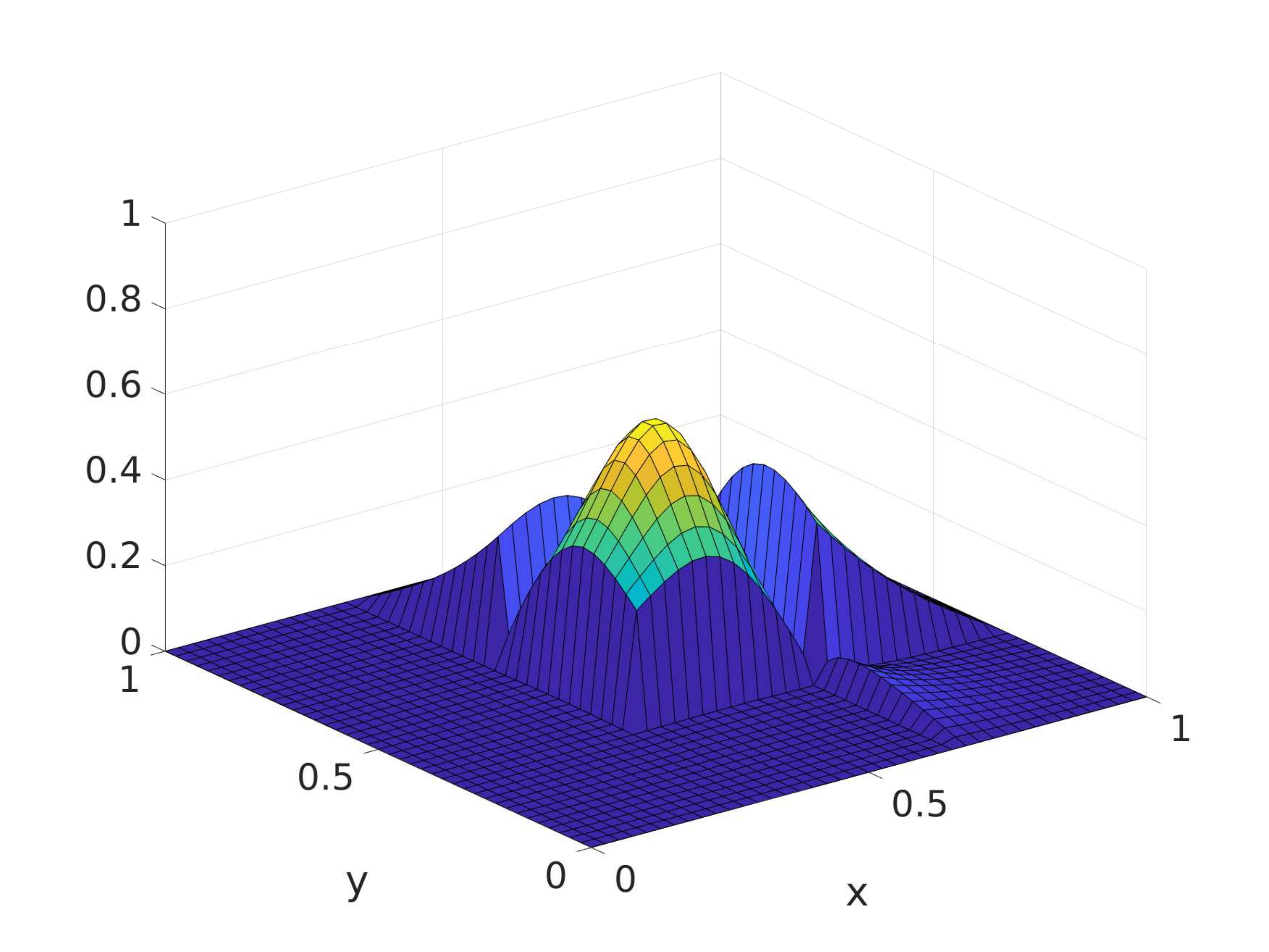}
\includegraphics[width=0.33\textwidth,trim=50 20 50 50,clip]{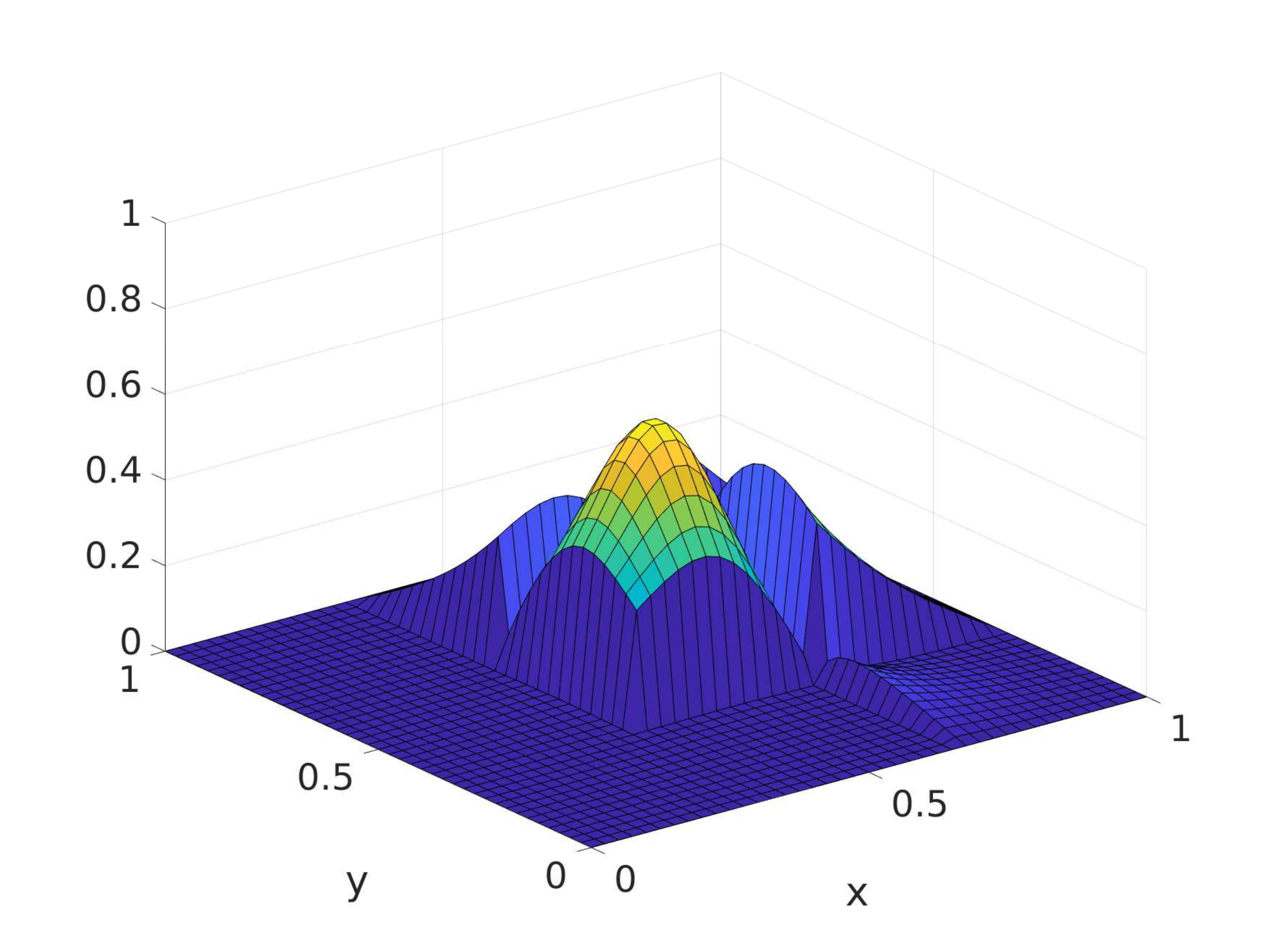}}
  \mbox{\includegraphics[width=0.33\textwidth,trim=50 20 50 50,clip]{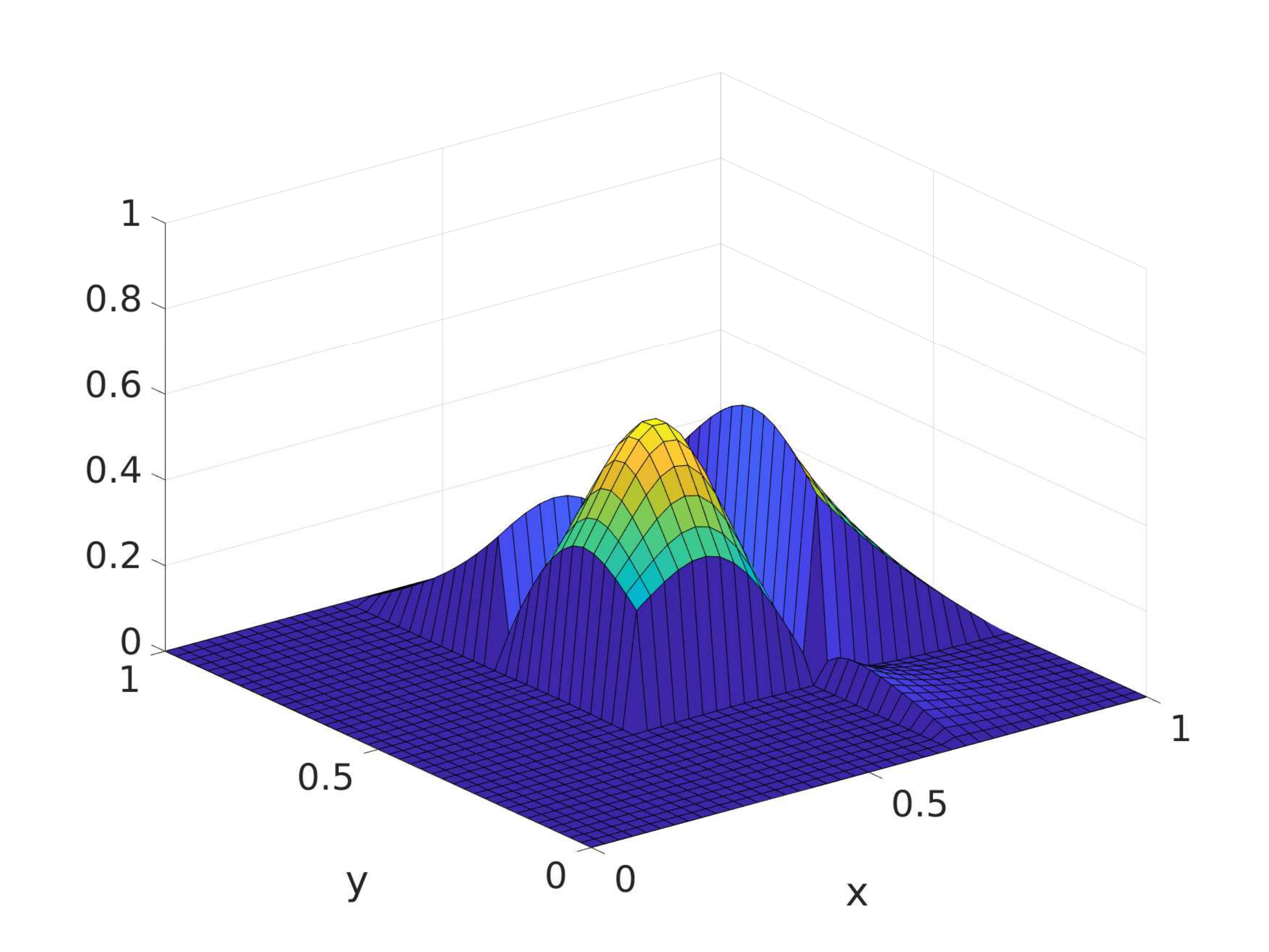}
\includegraphics[width=0.33\textwidth,trim=50 20 50 50,clip]{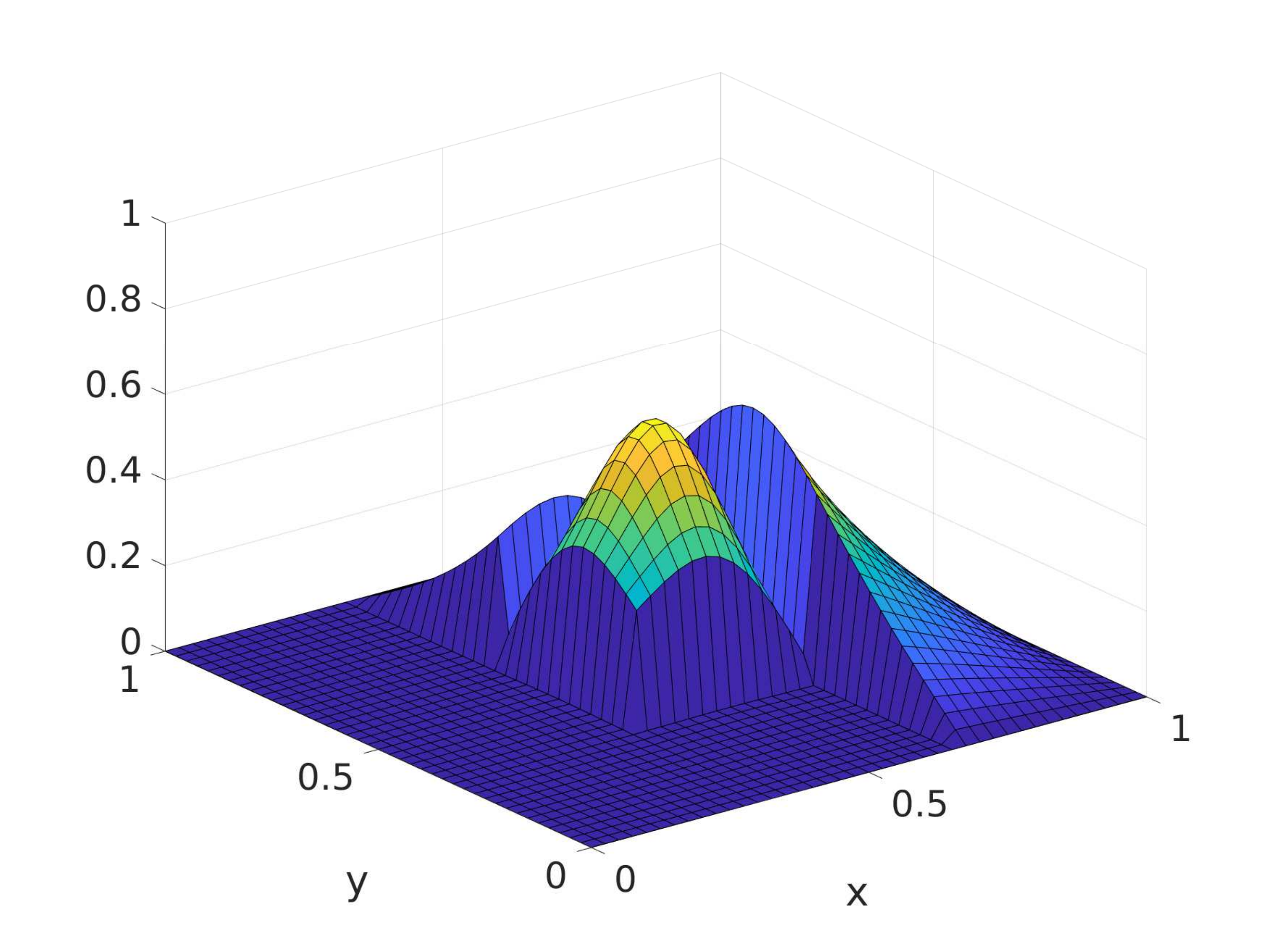}
\includegraphics[width=0.33\textwidth,trim=50 20 50 50,clip]{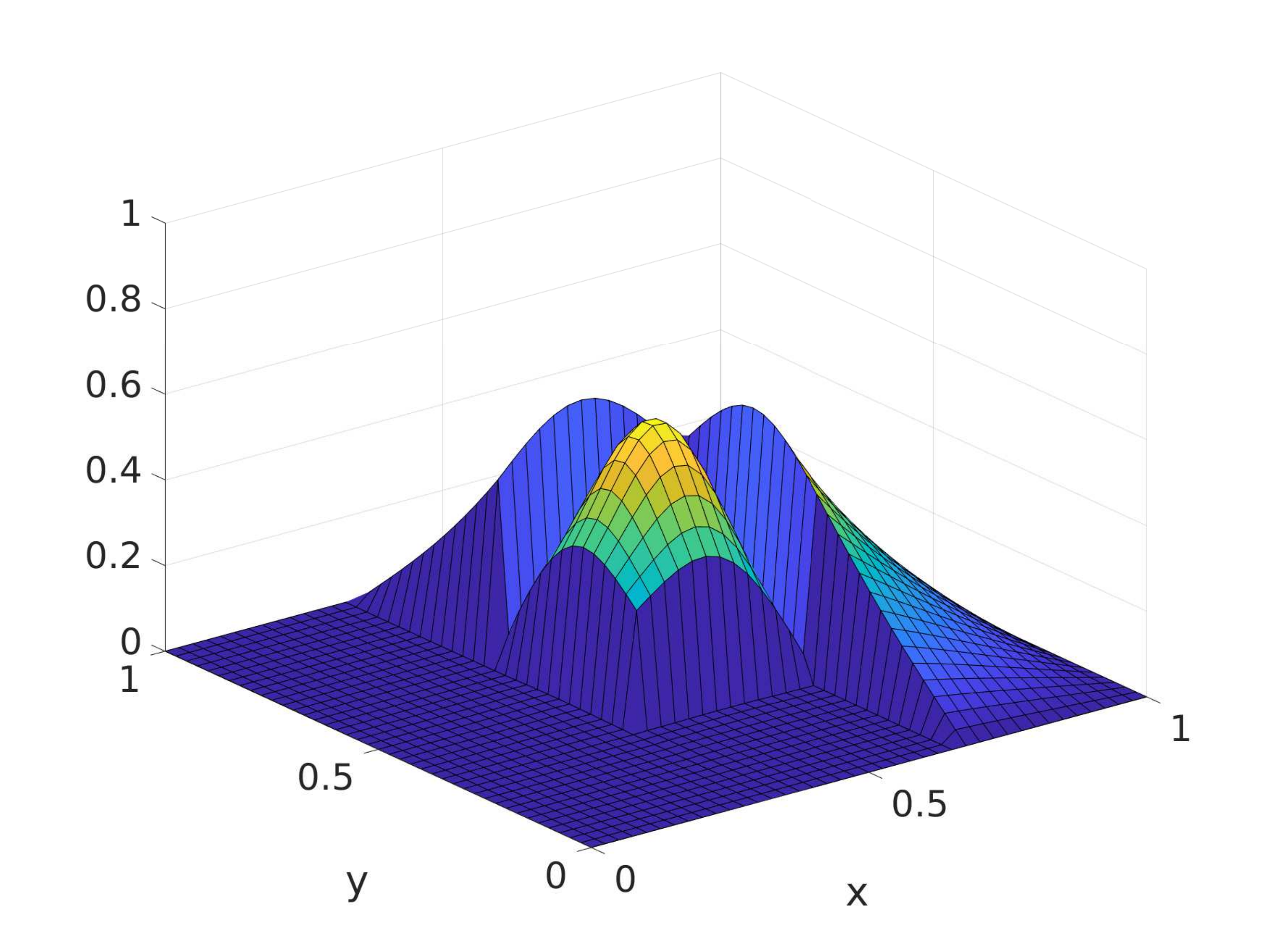}}
  \mbox{\includegraphics[width=0.33\textwidth,trim=50 20 50 50,clip]{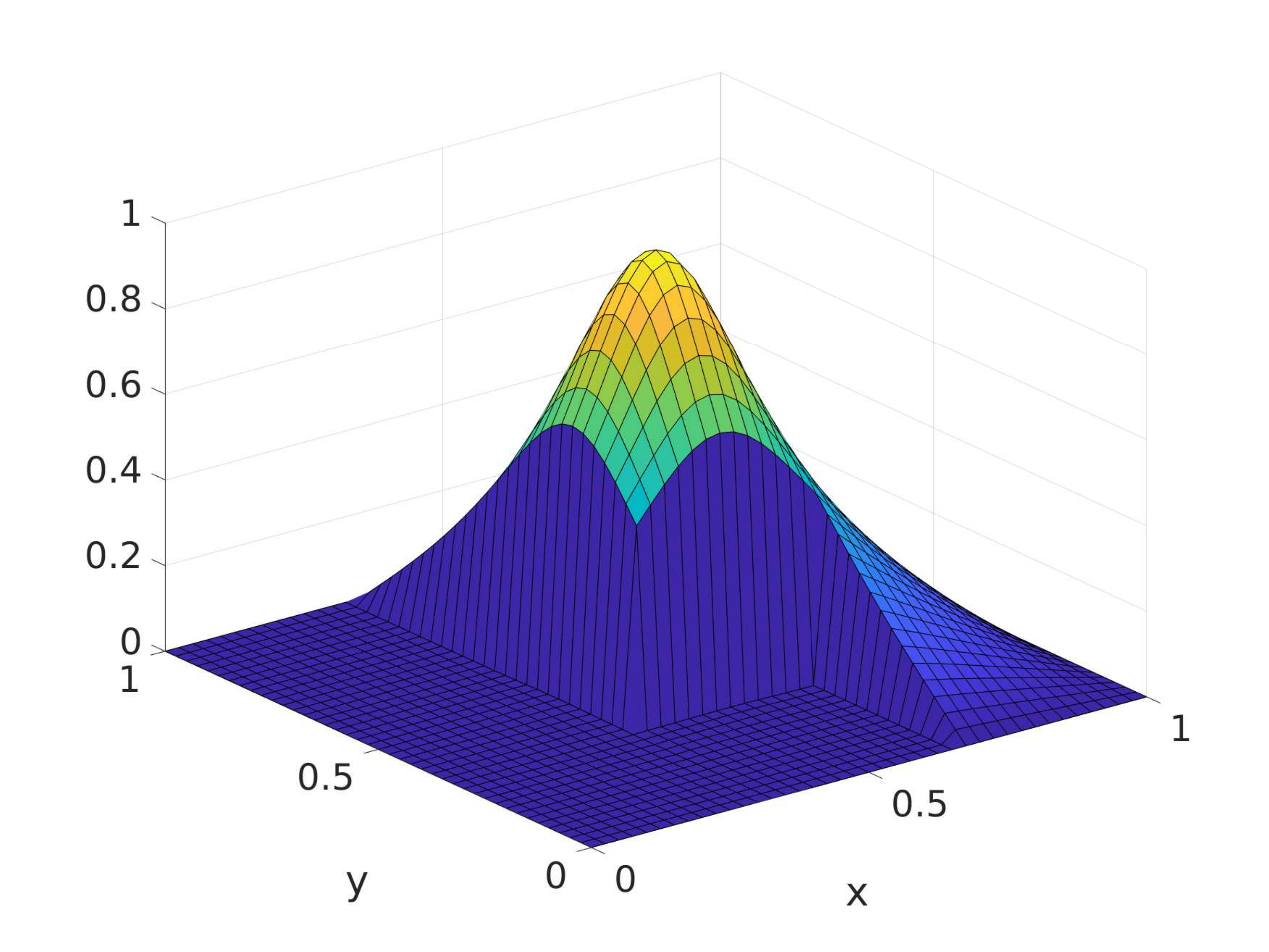}
\includegraphics[width=0.33\textwidth,trim=50 20 50 50,clip]{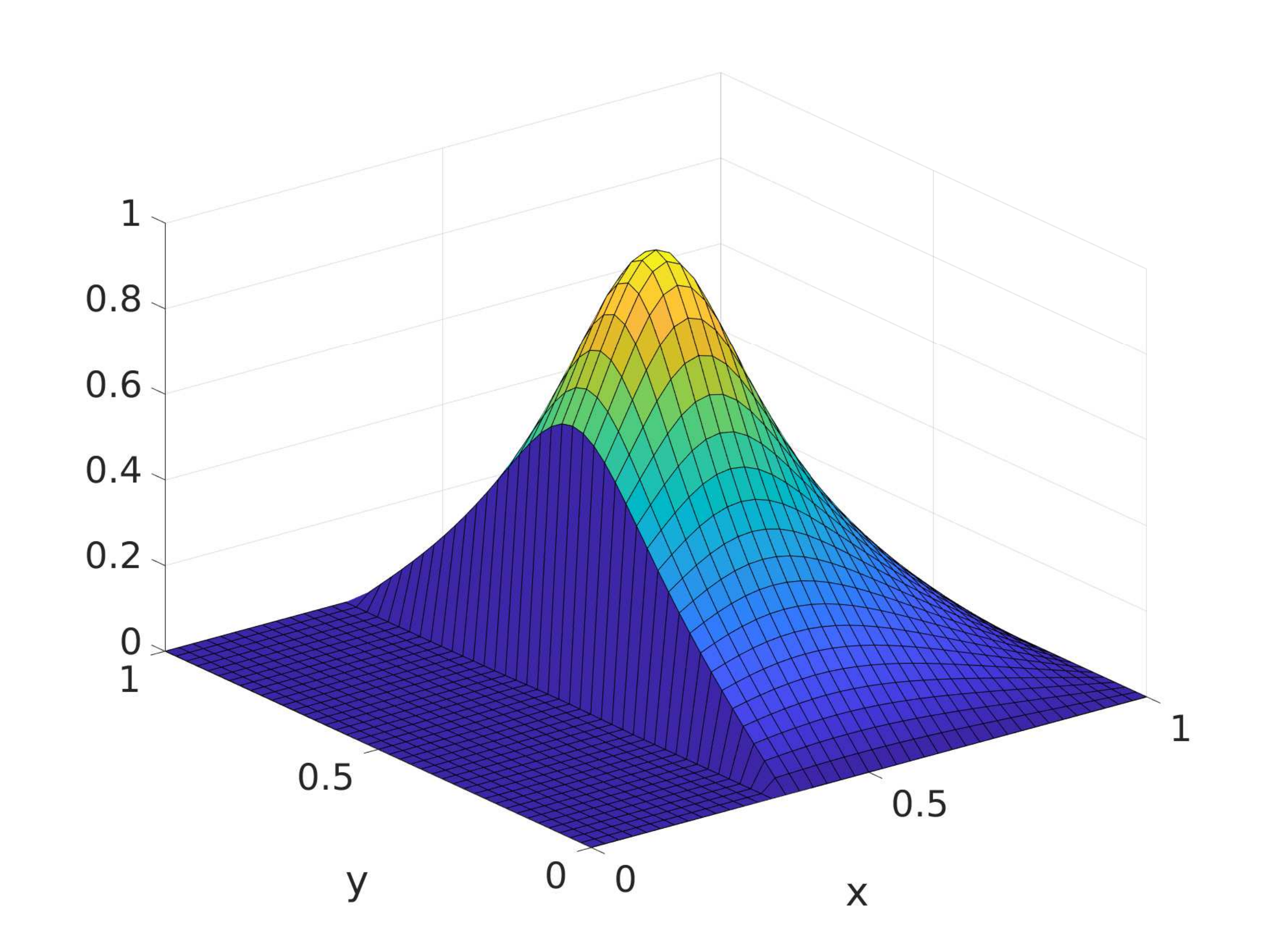}
\includegraphics[width=0.33\textwidth,trim=50 20 50 50,clip]{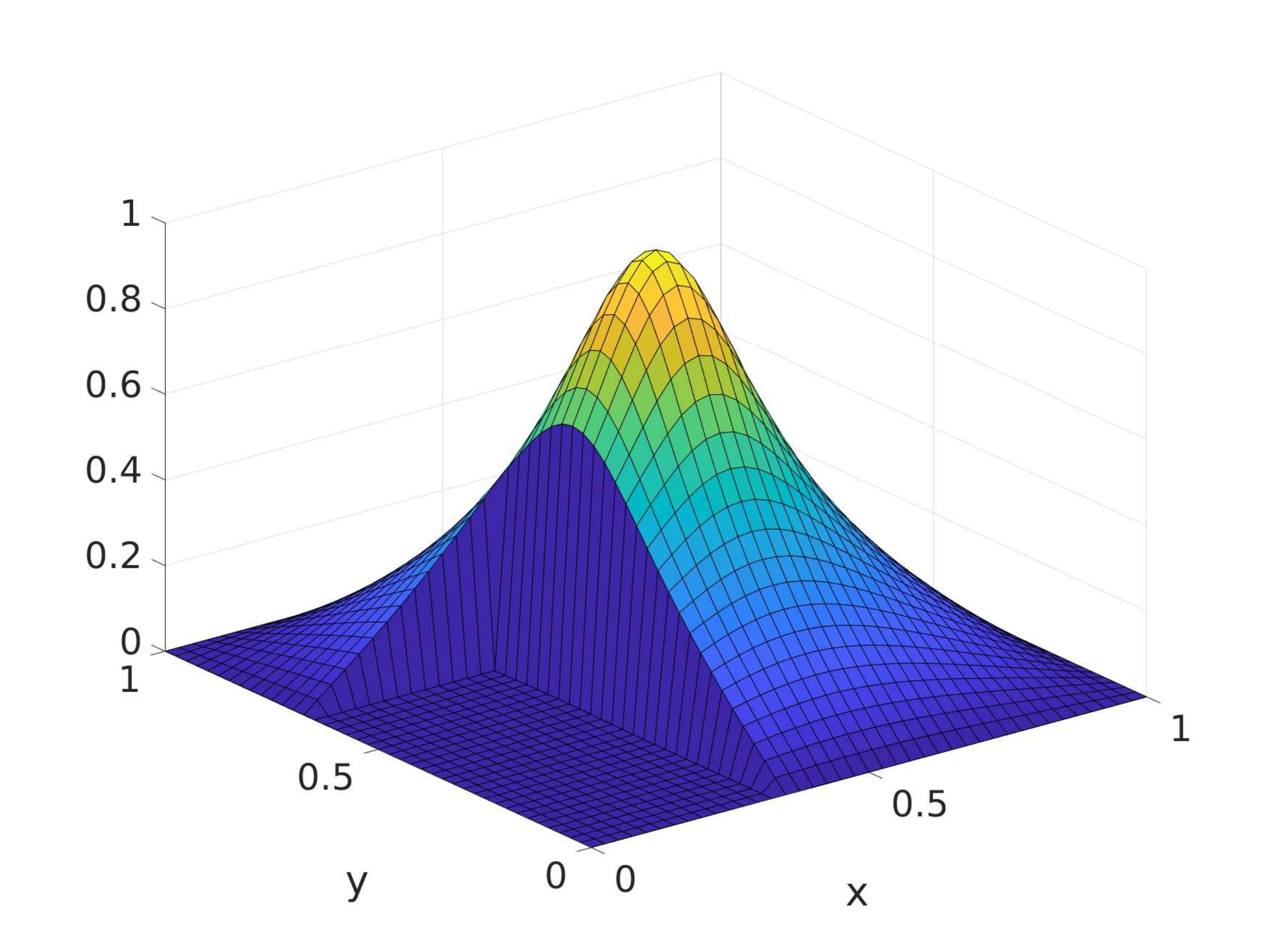}}
  \mbox{\includegraphics[width=0.33\textwidth,trim=50 20 50 50,clip]{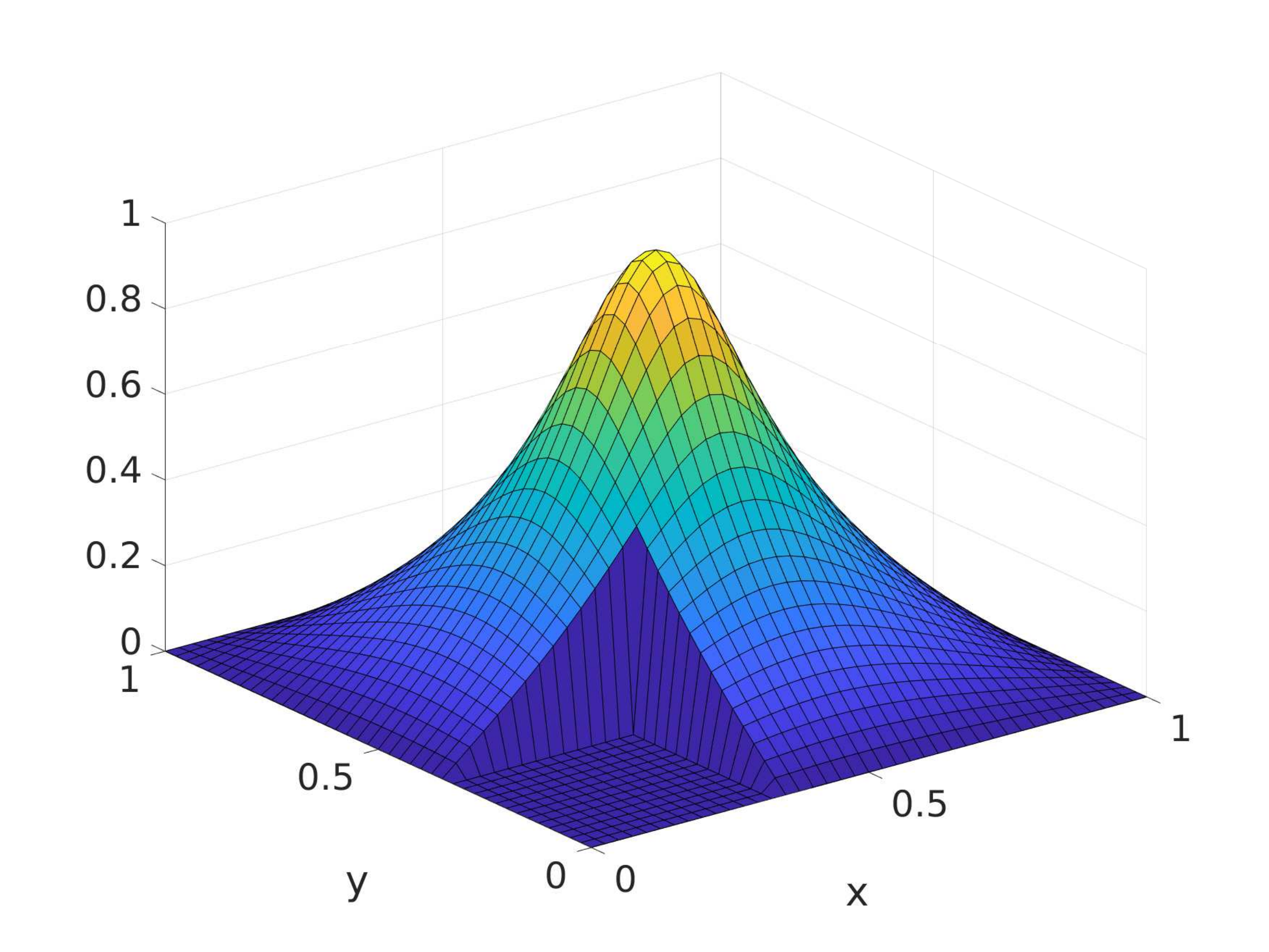}
\includegraphics[width=0.33\textwidth,trim=50 20 50 50,clip]{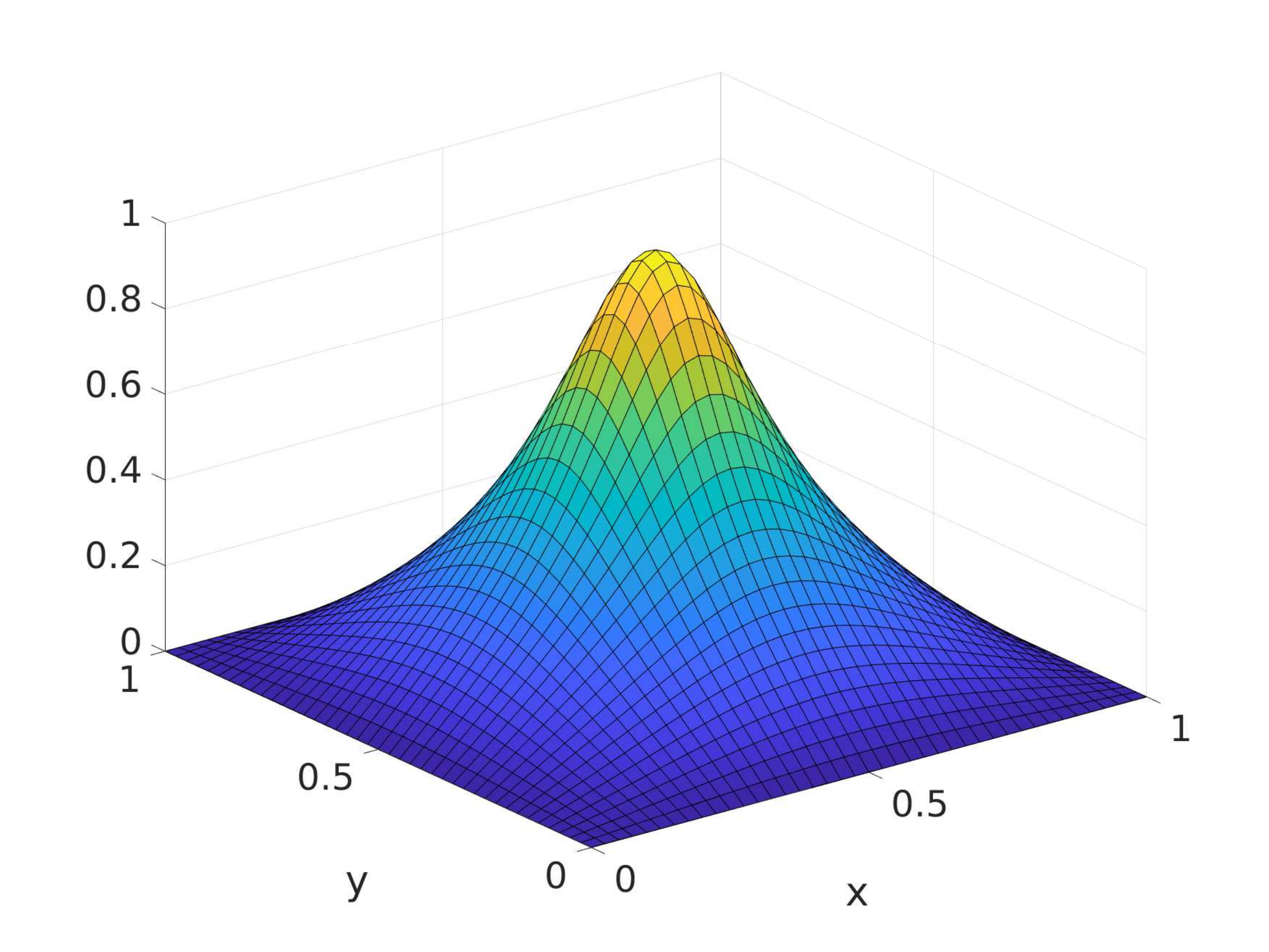}}\hfill
  \caption{Forward and backward sweep for an optimal Schwarz method
    obtained by a block LU decomposition for the model problem and
    $3\times 3$ subdomains. We observe convergence after one double
    sweep.}
  \label{OptimalSchwarzLUSweep}
\end{figure}
for our model problem and the same $3\times3$ domain decomposition
used in Figure \ref{PoissonRASORASFig}. We have chosen here simply a
lexicographic ordering of the subdomains, and we see that the method
converges in one double sweep. We show in Figure \ref{OSMLUfactors}
\begin{figure}
  \centering
  \mbox{\includegraphics[width=0.49\textwidth]{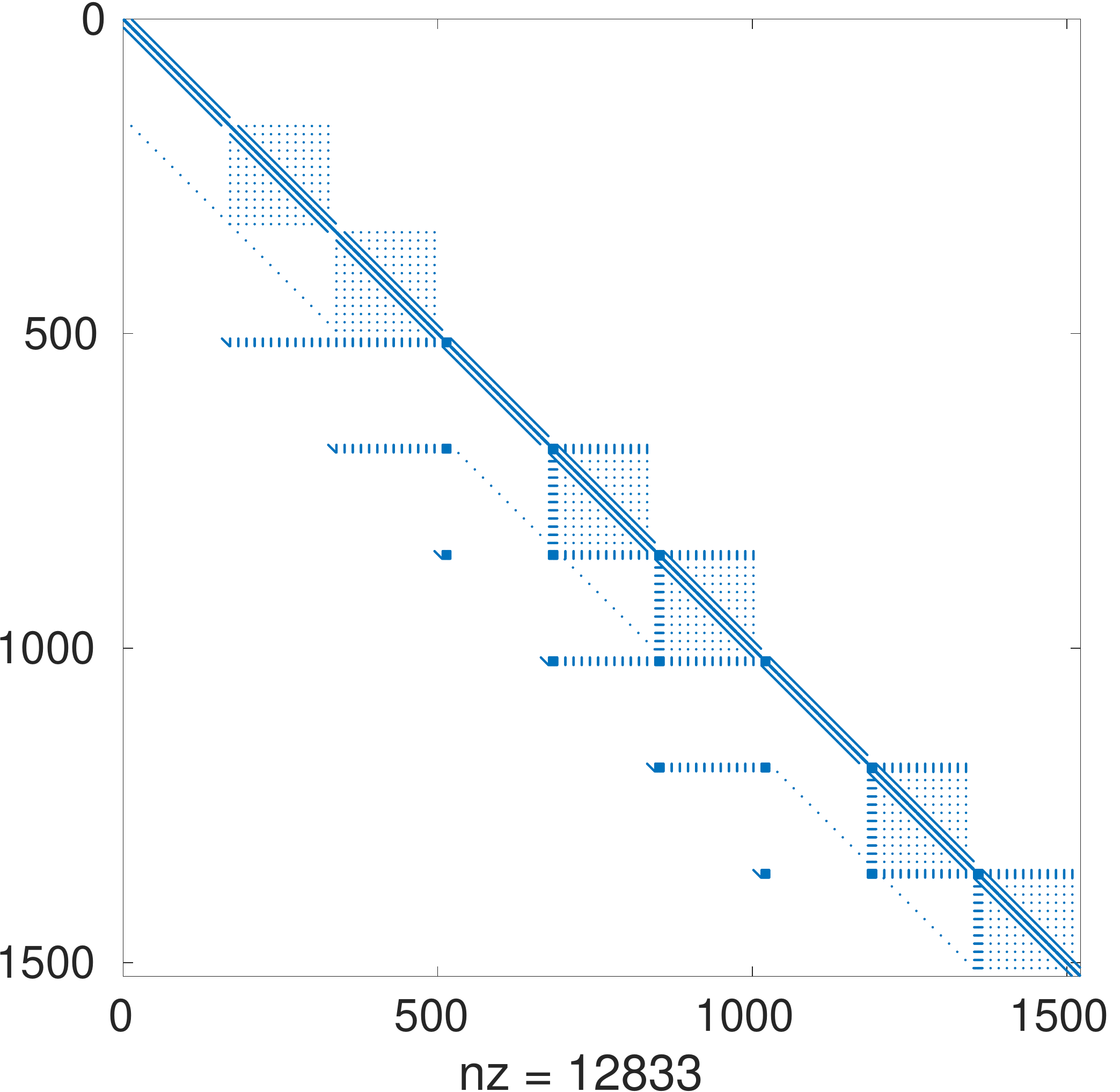}
\includegraphics[width=0.49\textwidth]{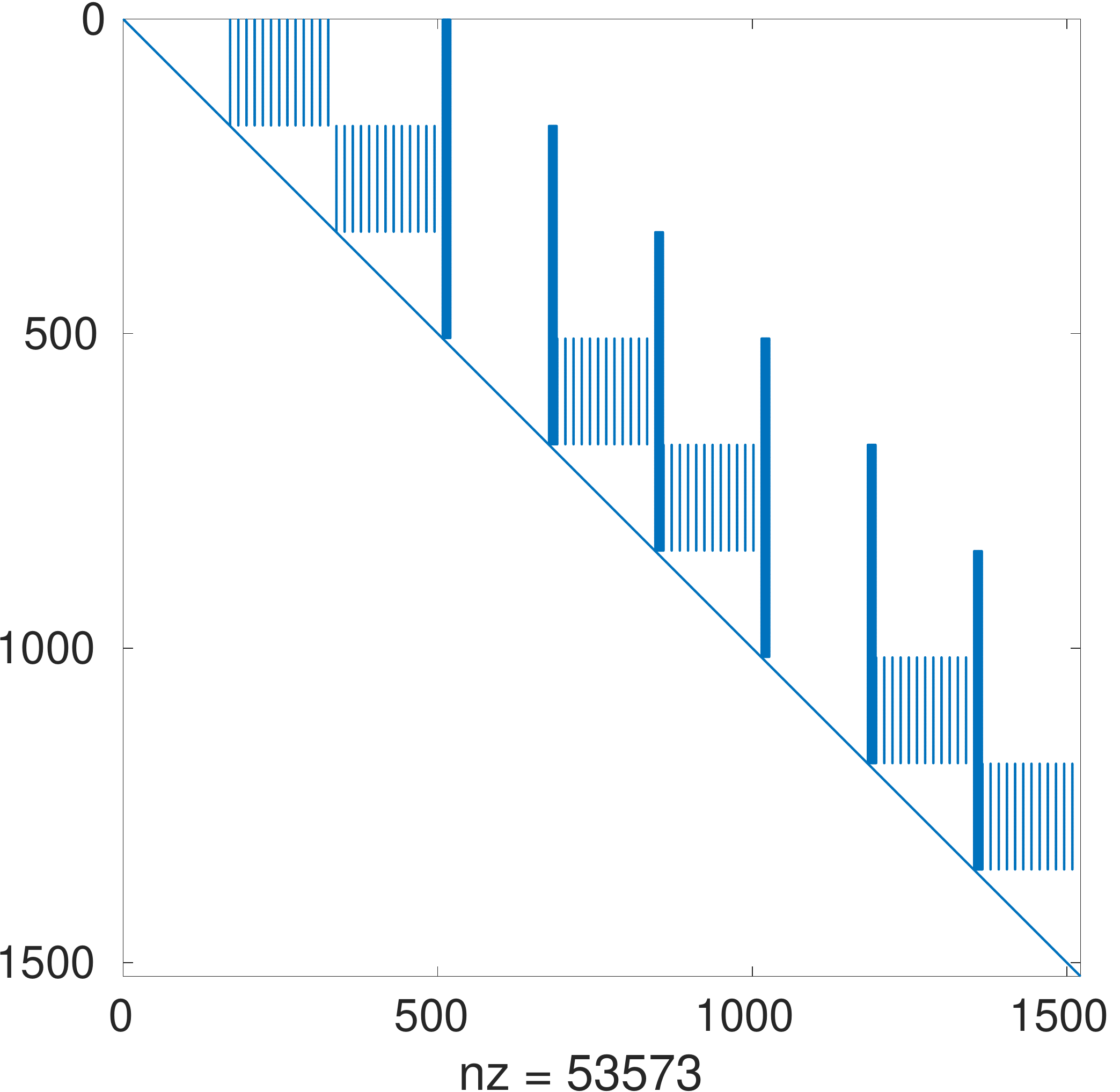}}
  \caption{Sparsity of the block $L$ and $U$ factors of the optimal
    Schwarz method for the $3\times 3$ subdomain decomposition from
    Figure \ref{OptimalSchwarzLUSweep}.}
  \label{OSMLUfactors}
\end{figure}
the sparsity structure of the corresponding LU factors, which indicate the structure of the
transmission conditions generated by the block LU decomposition in this case, a subject that needs
further investigation; if the domain decomposition is a strip decomposition without cross points,
\ie\ like in Figure \ref{1Dand2DDecompositionFig} on the right, it is known that the block LU
decomposition generates transparent transmission conditions on the left and Dirichlet conditions on
the right of the subdomains, like in the design of the source transfer domain decomposition method. There
is uniqueness in the block LU decomposition only once one chooses the diagonal of one of the
factors, \eg~the identity matrices on the diagonal of $U$ as we did here. For the strip
decomposition our block LU factorization gives therefore a different algorithm from the optimal
Schwarz method shown in Figure \ref{PoissonOptimalSchwarzFig} which used transparent boundary
conditions involving the DtN operator on both sides of the subdomains. We give a simple Matlab
implementation for these block LU factorizations in Appendix A for the interested
reader to experiment with.

Note that we could have used any other ordering of the subdomains for
the sweeping before performing the block LU factorization: we show for
example in Figure \ref{OptimalSchwarzLUSweepL}
\begin{figure}
  \centering
  \mbox{\includegraphics[width=0.33\textwidth,trim=50 20 50 50,clip]{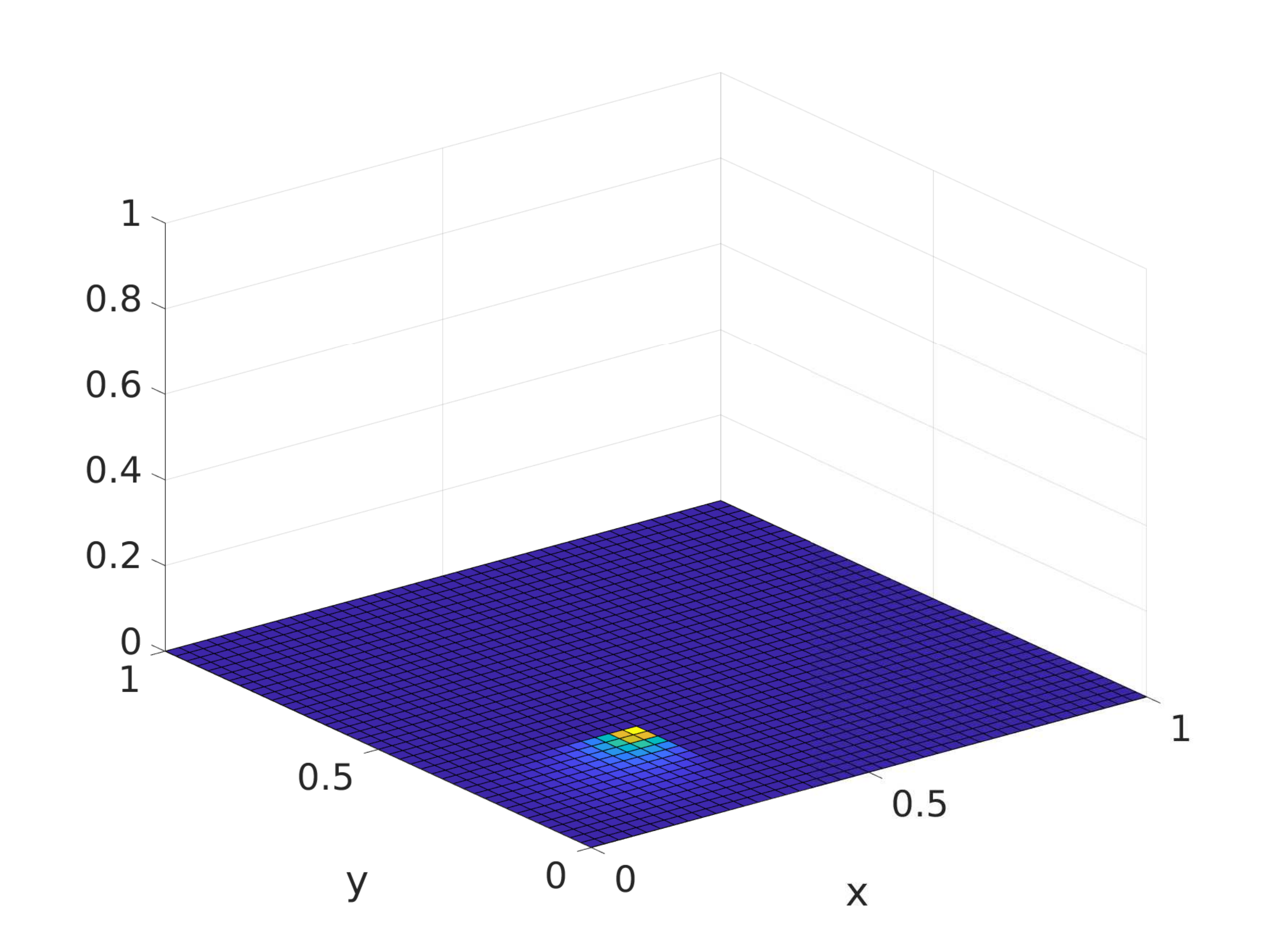}
\includegraphics[width=0.33\textwidth,trim=50 20 50 50,clip]{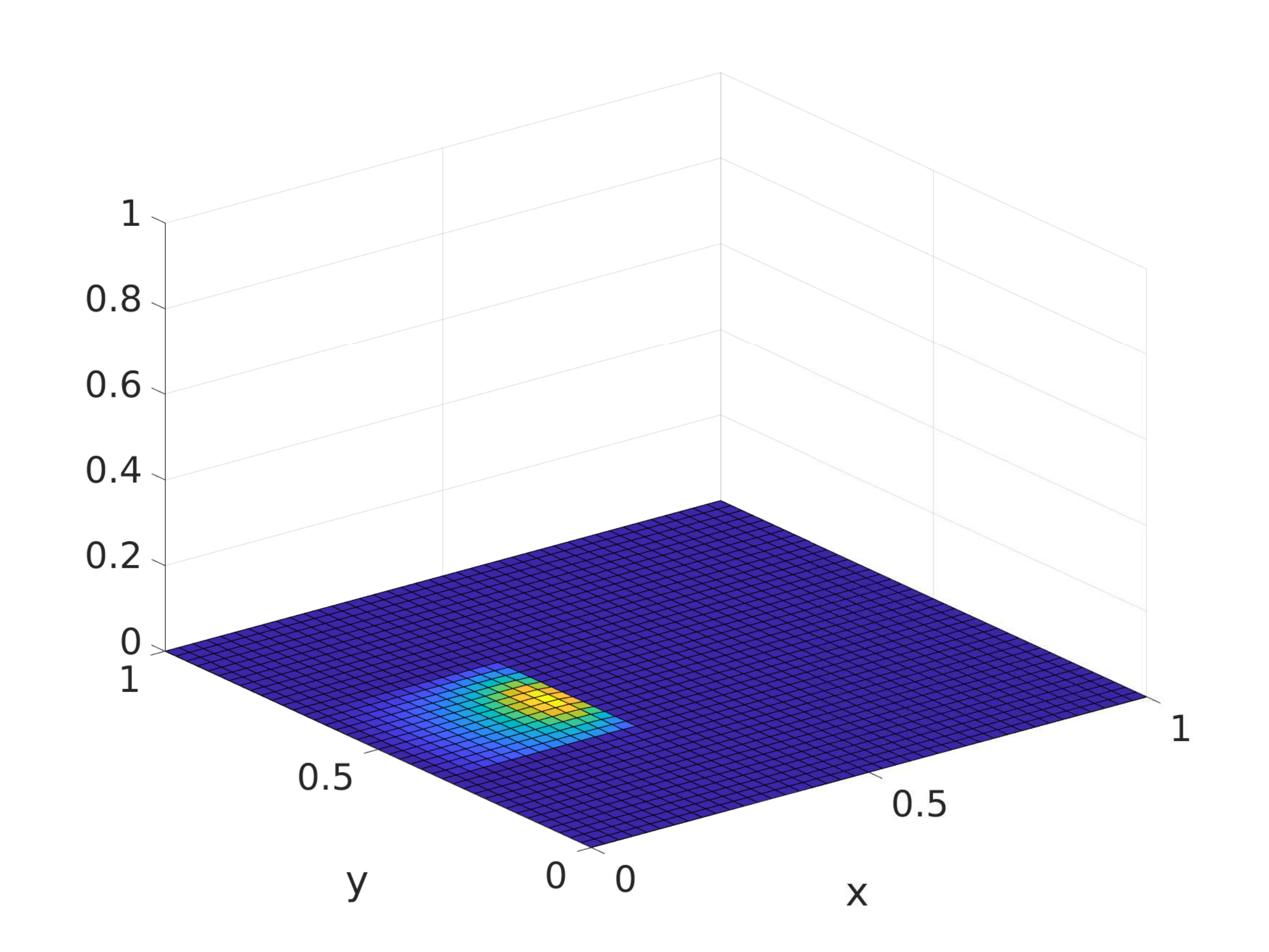}
\includegraphics[width=0.33\textwidth,trim=50 20 50 50,clip]{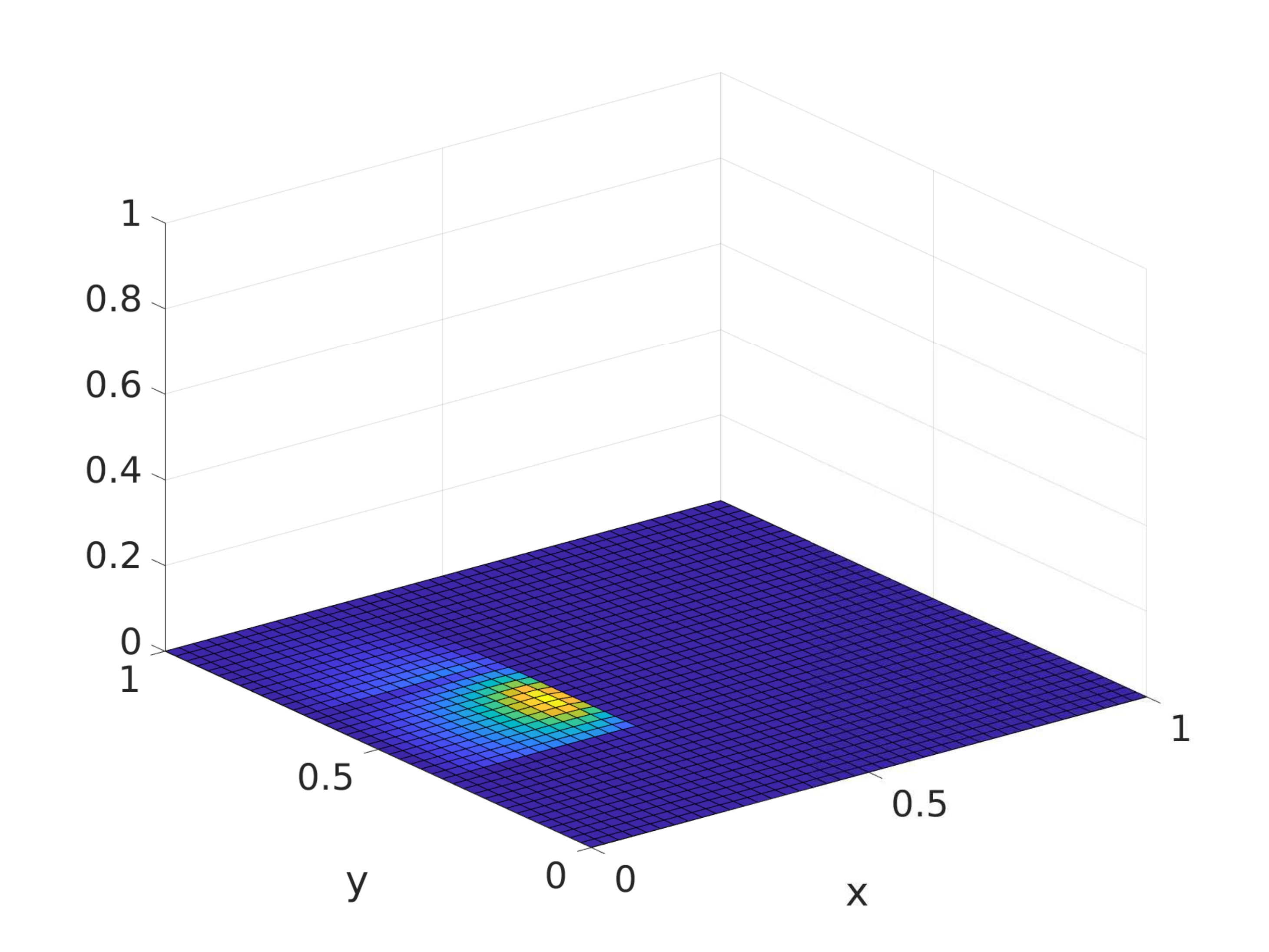}}
  \mbox{\includegraphics[width=0.33\textwidth,trim=50 20 50 50,clip]{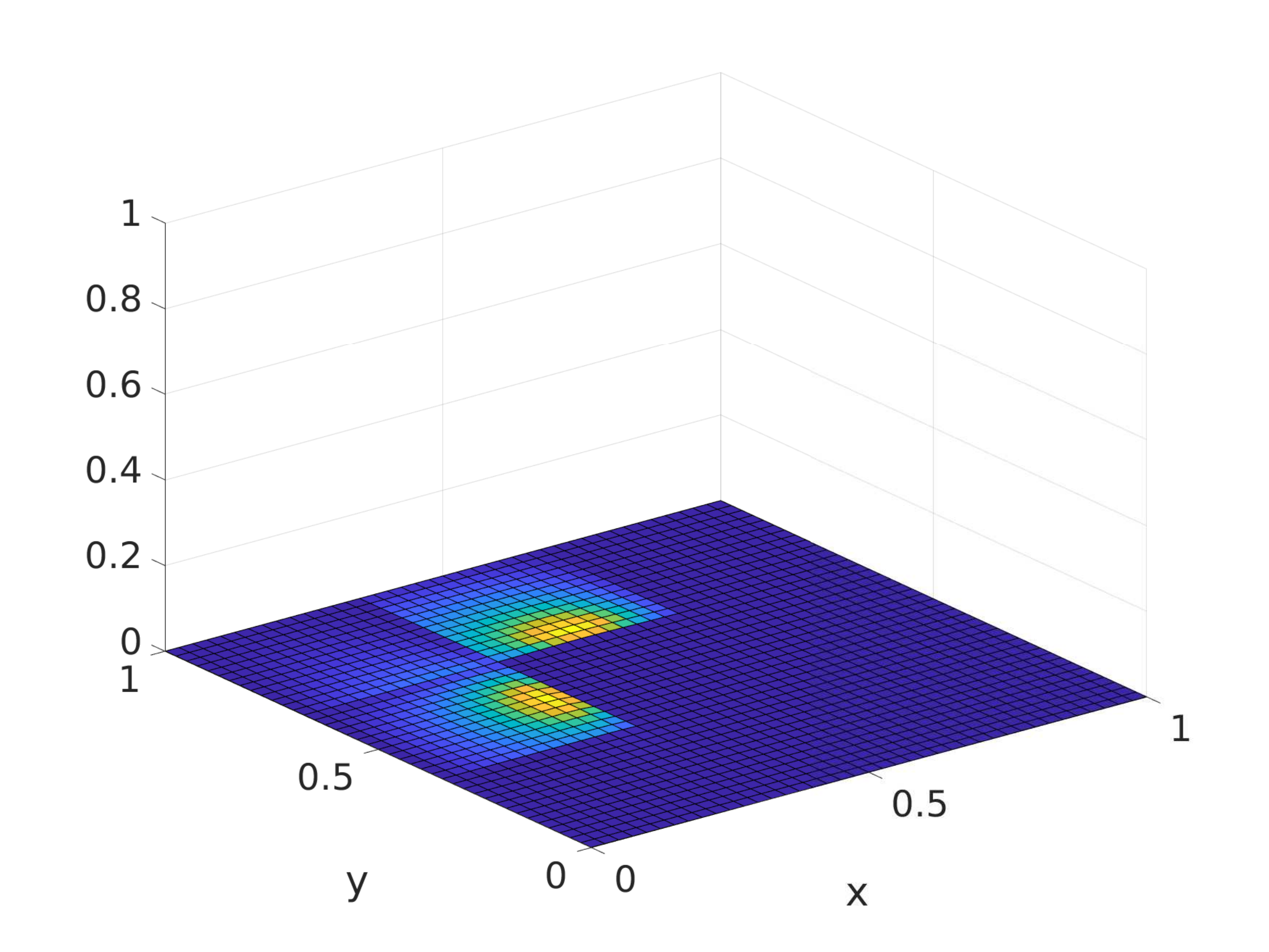}
\includegraphics[width=0.33\textwidth,trim=50 20 50 50,clip]{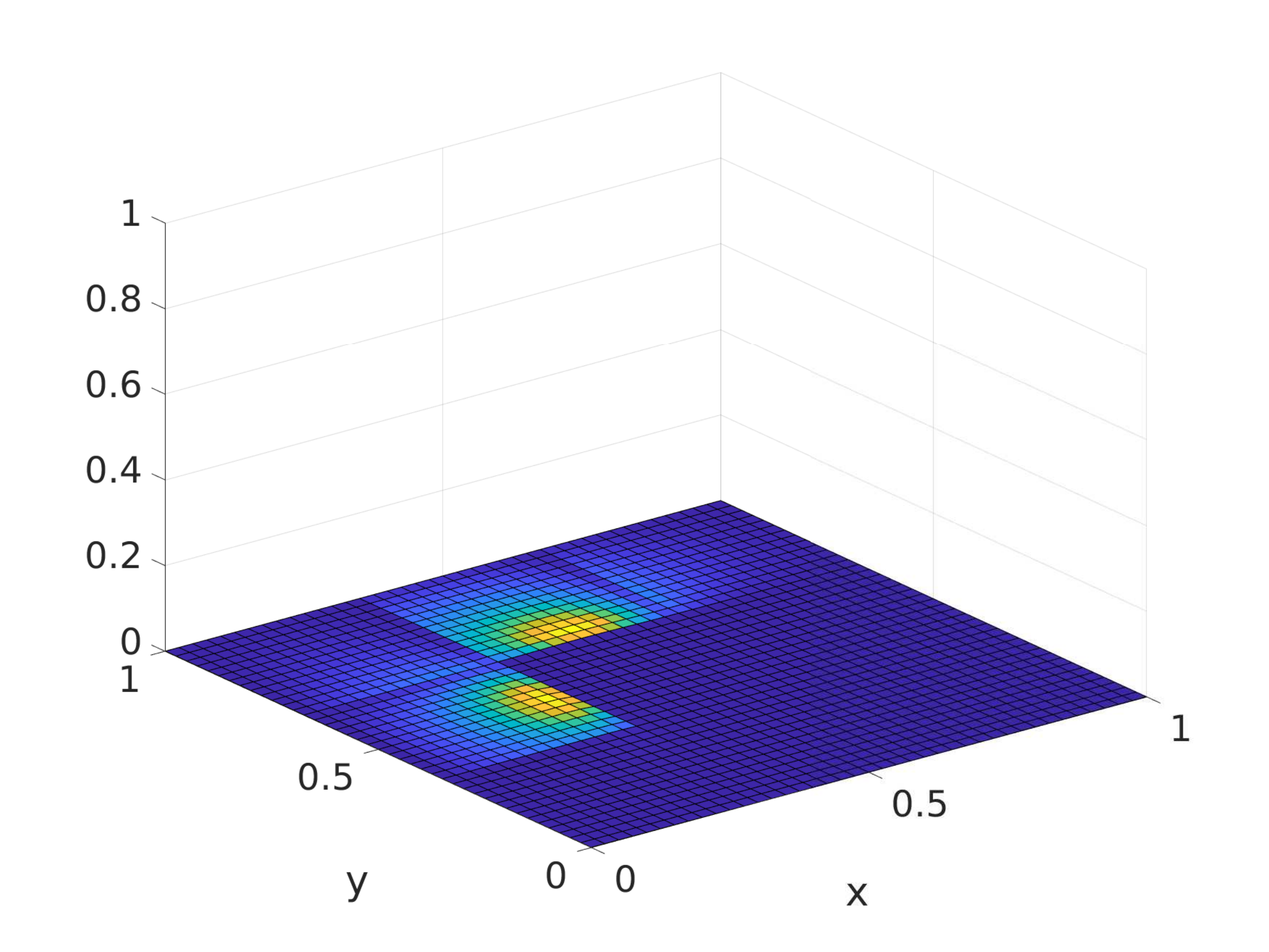}
\includegraphics[width=0.33\textwidth,trim=50 20 50 50,clip]{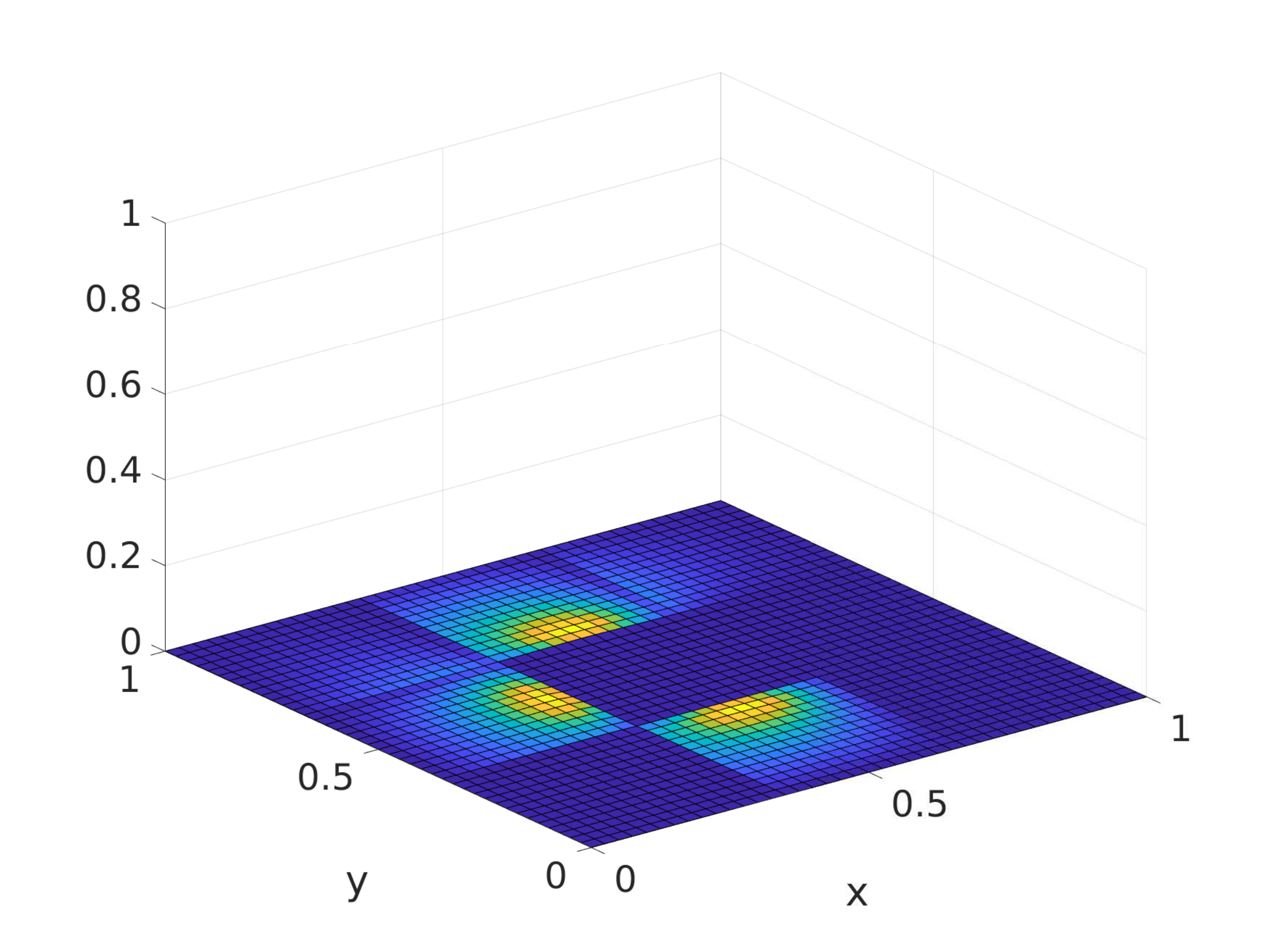}}
  \mbox{\includegraphics[width=0.33\textwidth,trim=50 20 50 50,clip]{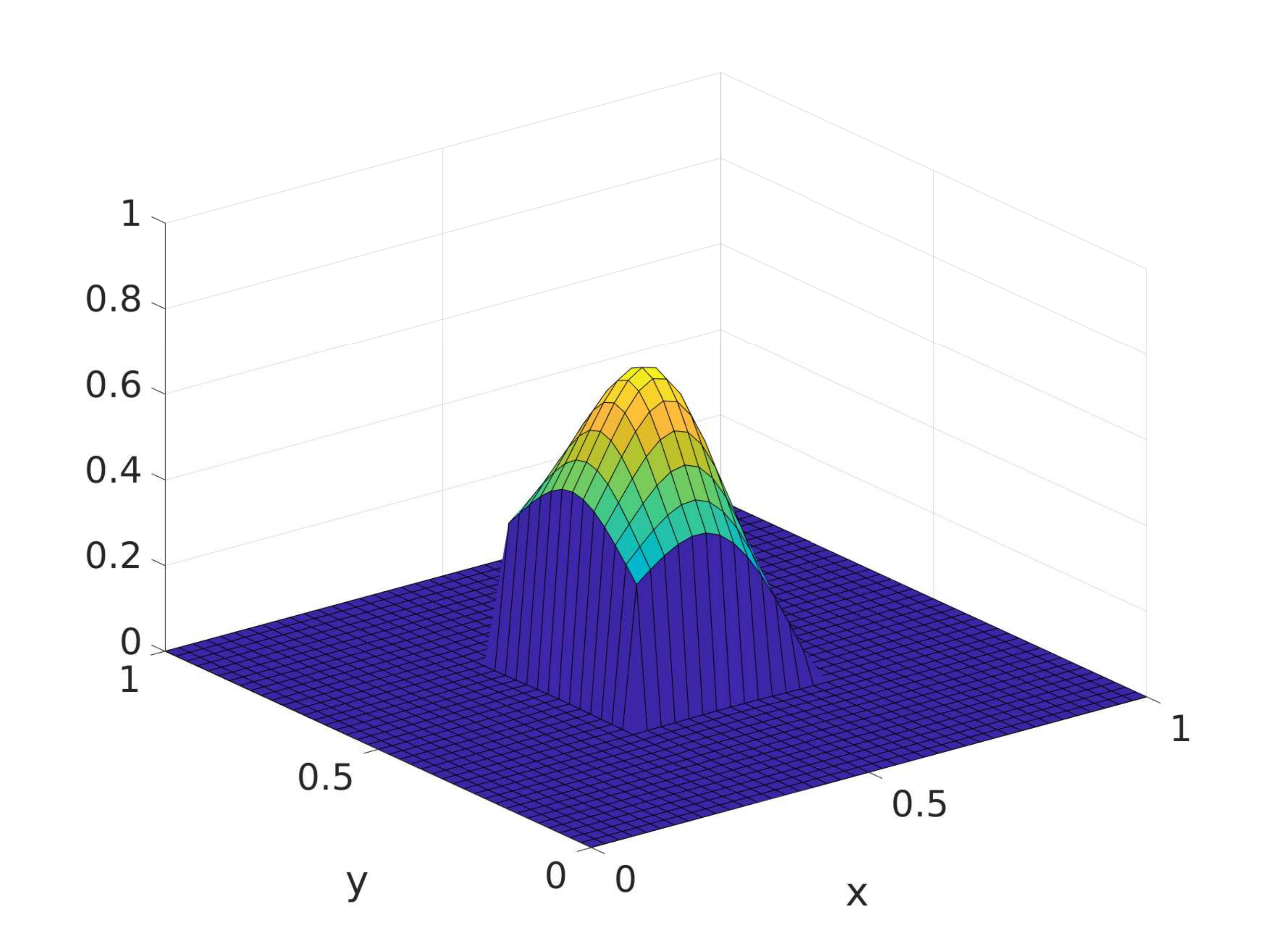}
\includegraphics[width=0.33\textwidth,trim=50 20 50 50,clip]{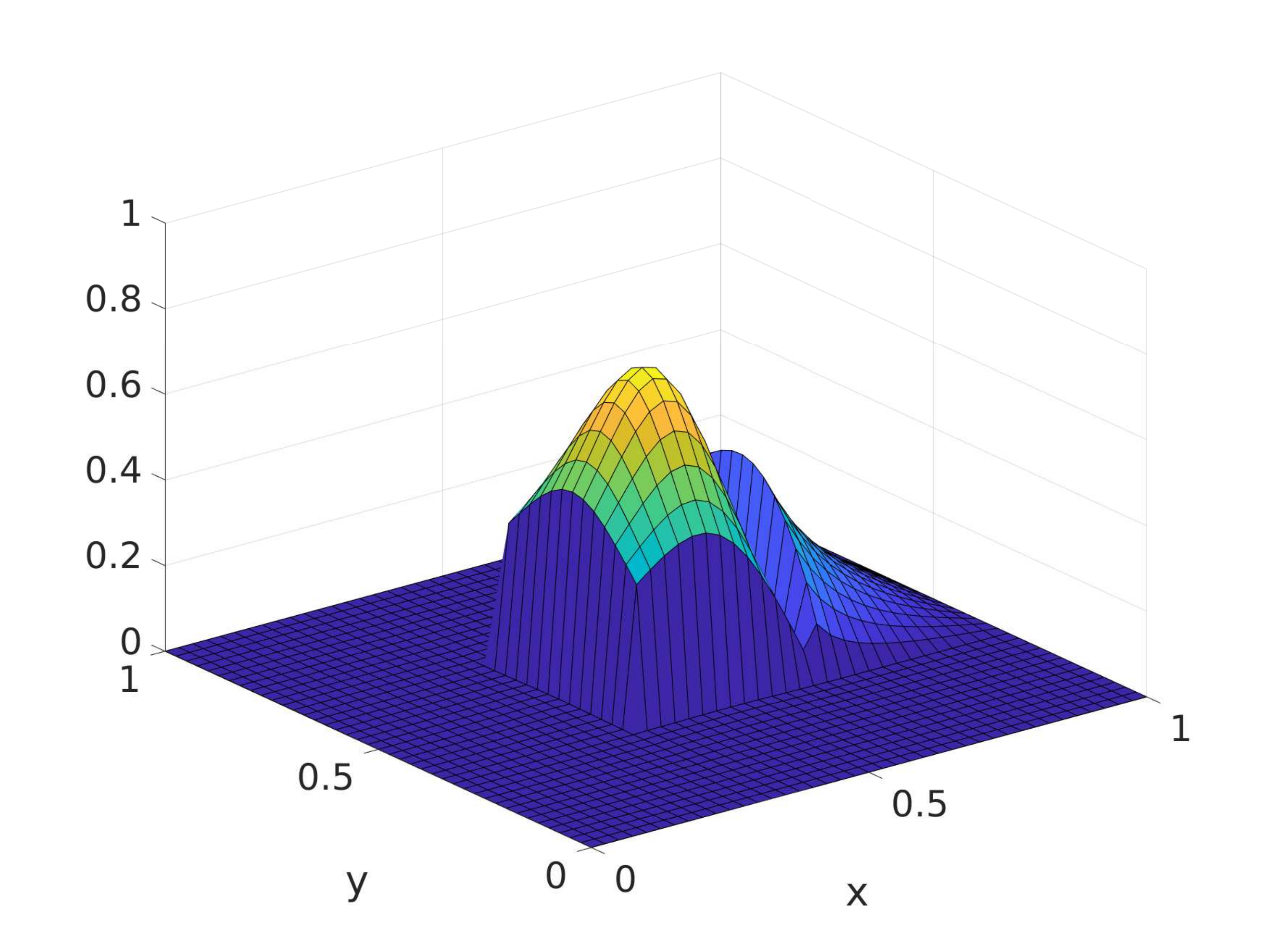}
\includegraphics[width=0.33\textwidth,trim=50 20 50 50,clip]{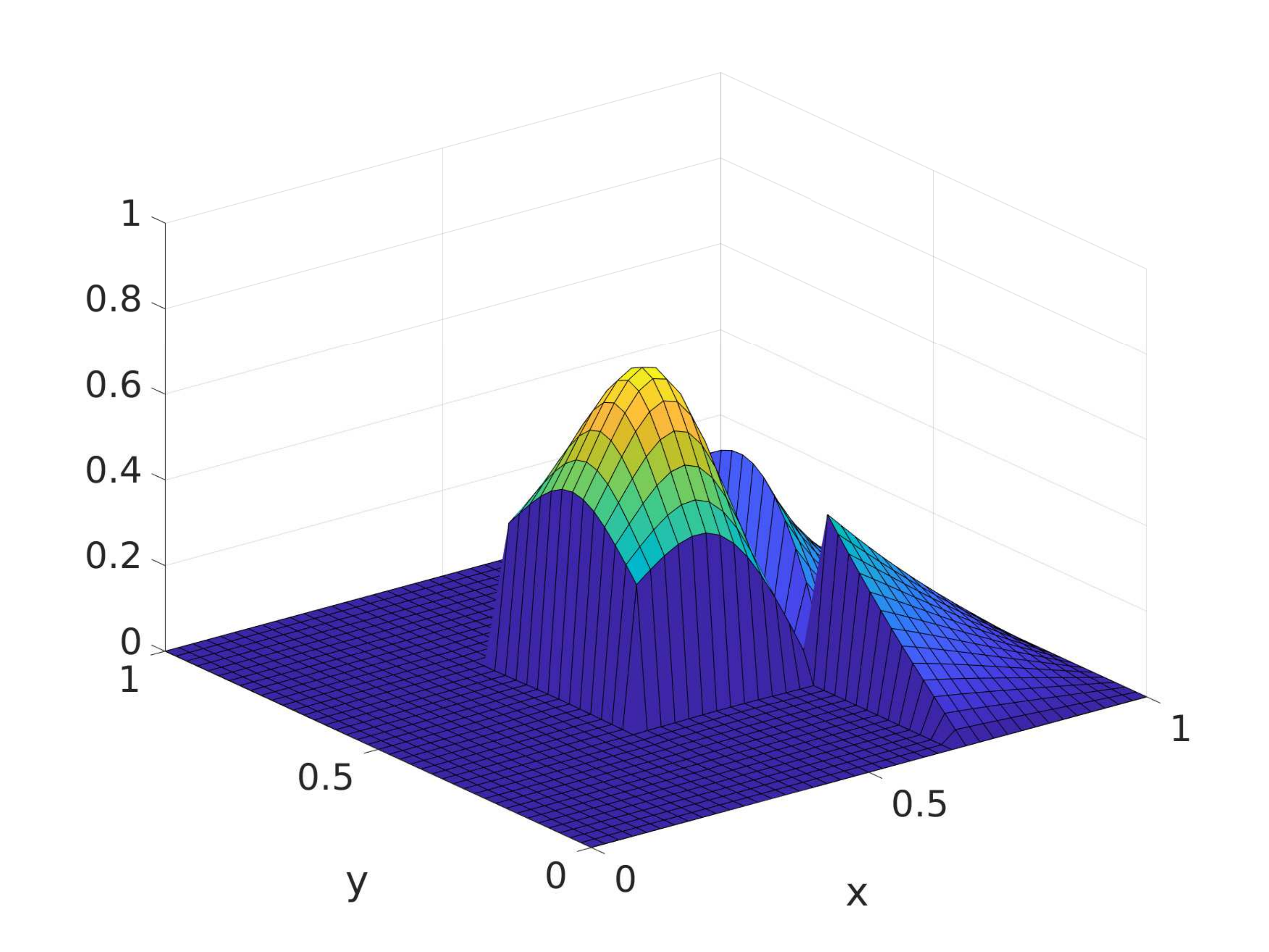}}
  \mbox{\includegraphics[width=0.33\textwidth,trim=50 20 50 50,clip]{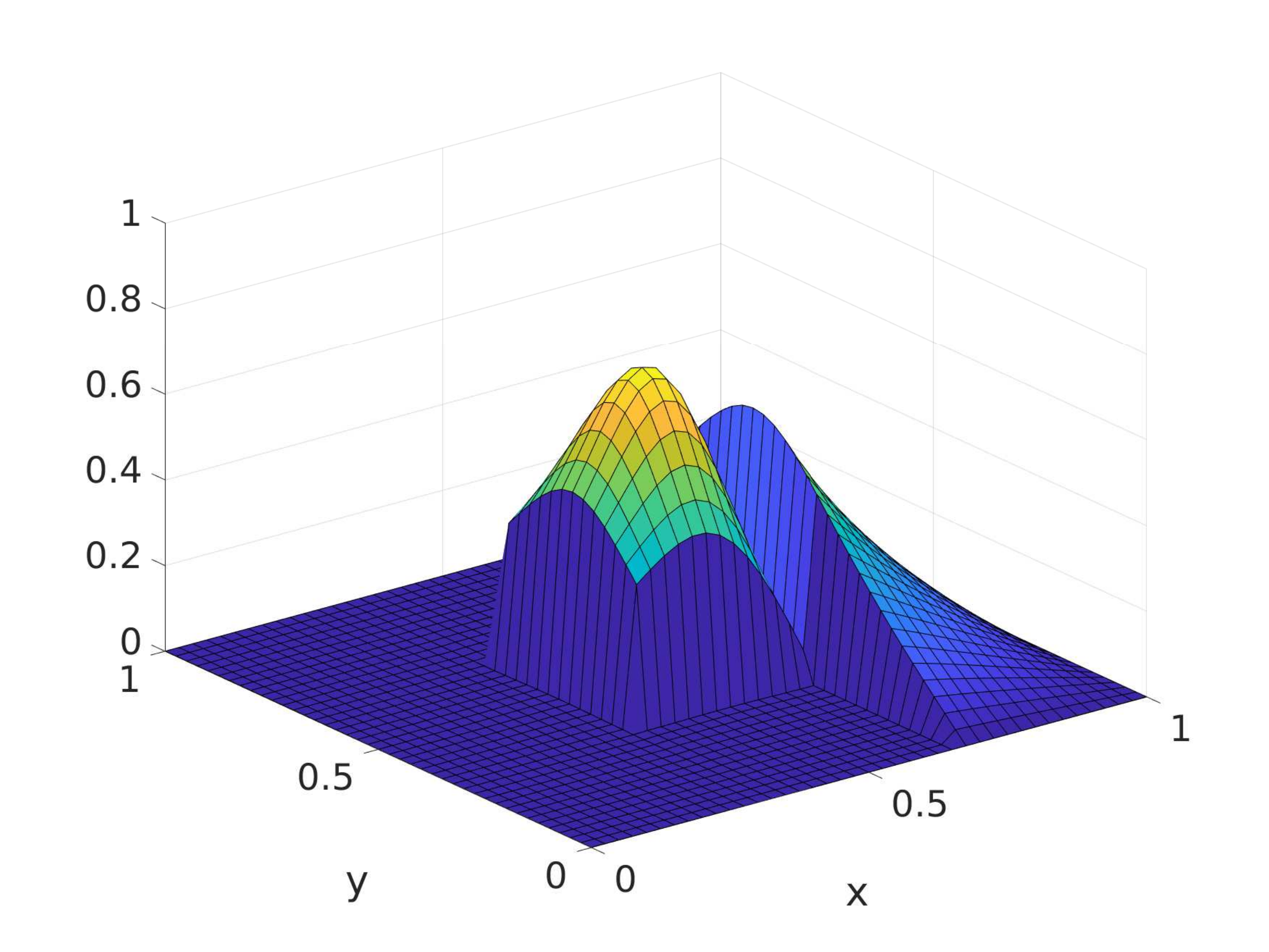}
\includegraphics[width=0.33\textwidth,trim=50 20 50 50,clip]{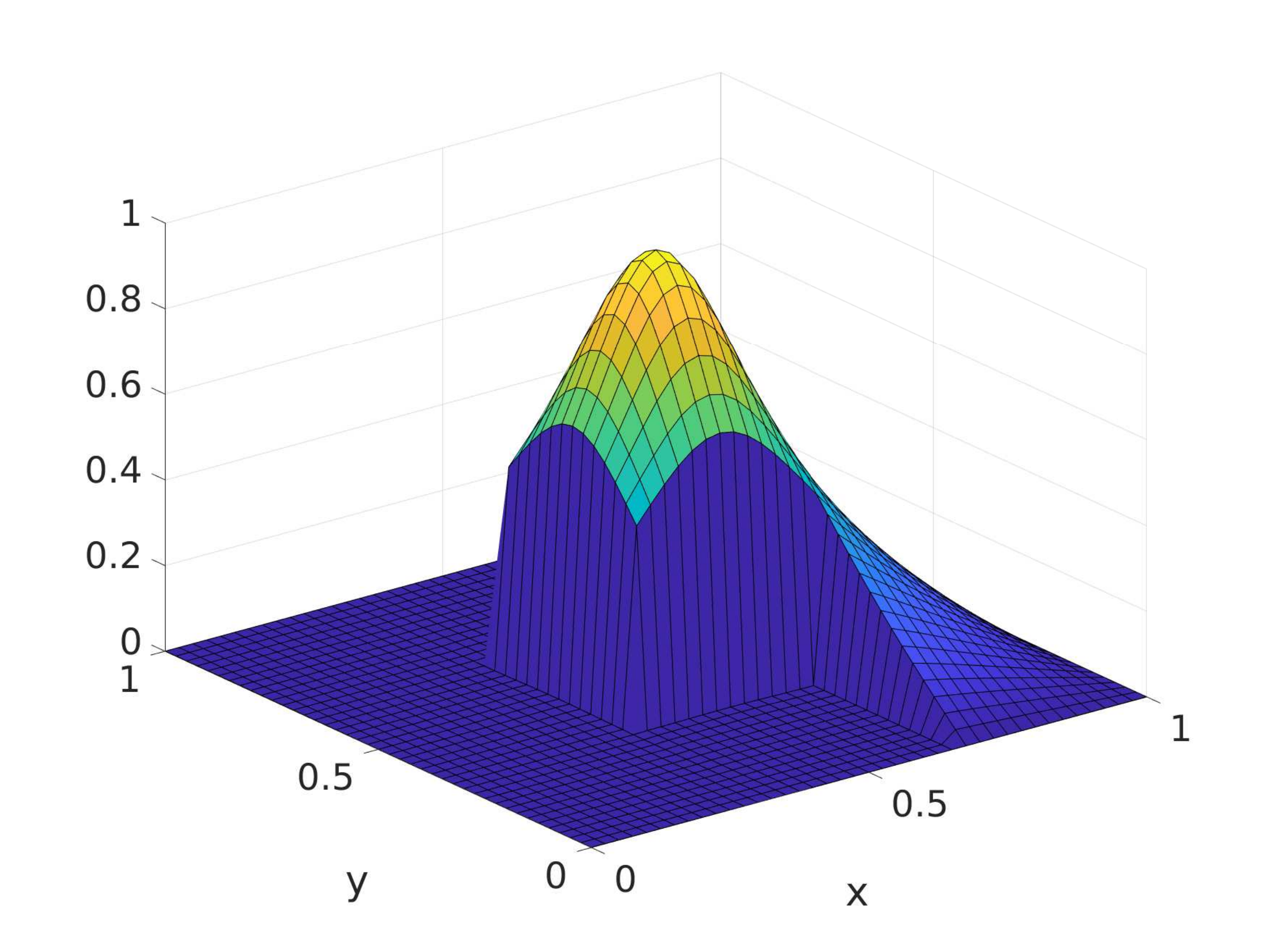}
\includegraphics[width=0.33\textwidth,trim=50 20 50 50,clip]{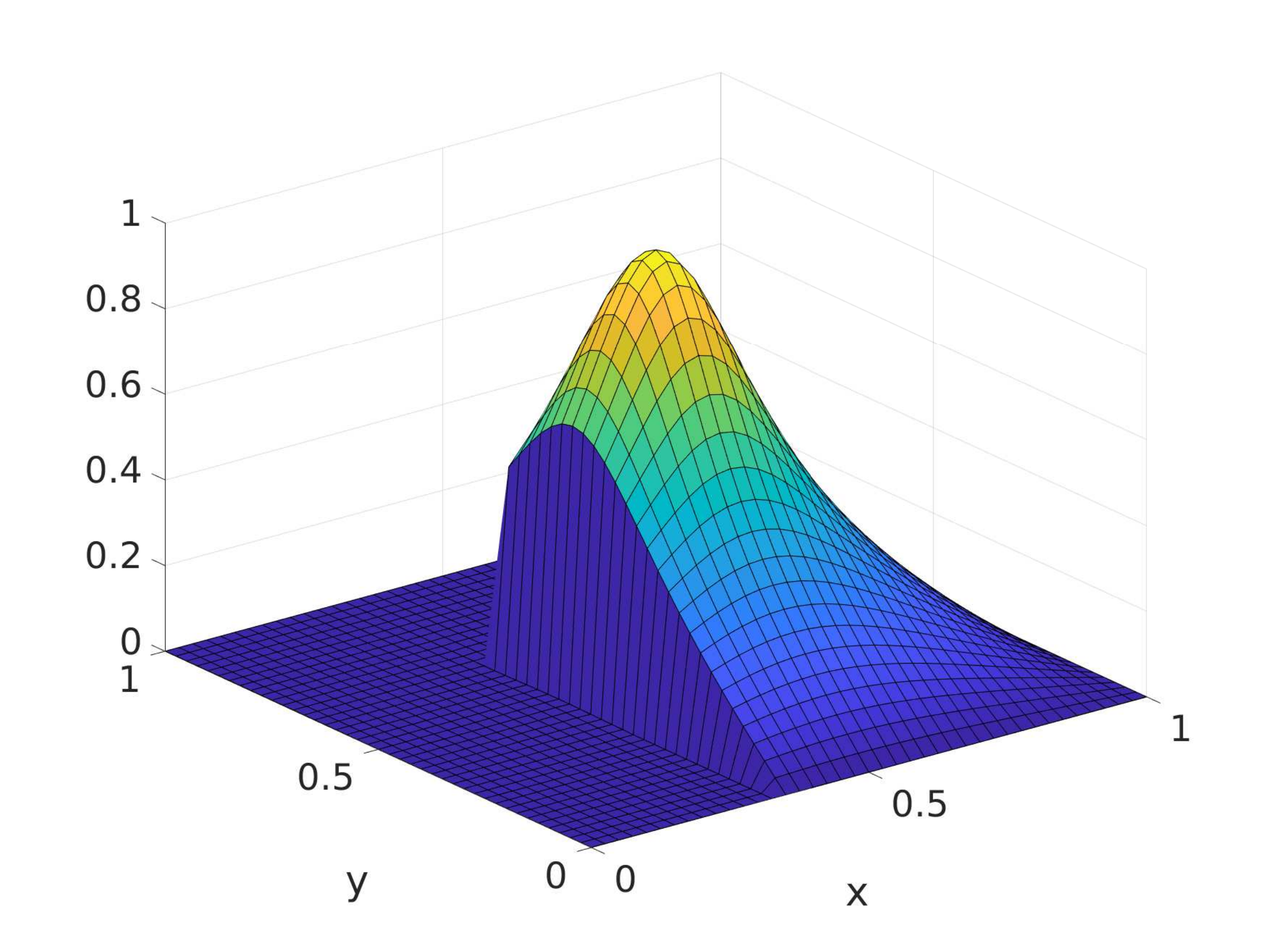}}
  \mbox{\includegraphics[width=0.33\textwidth,trim=50 20 50 50,clip]{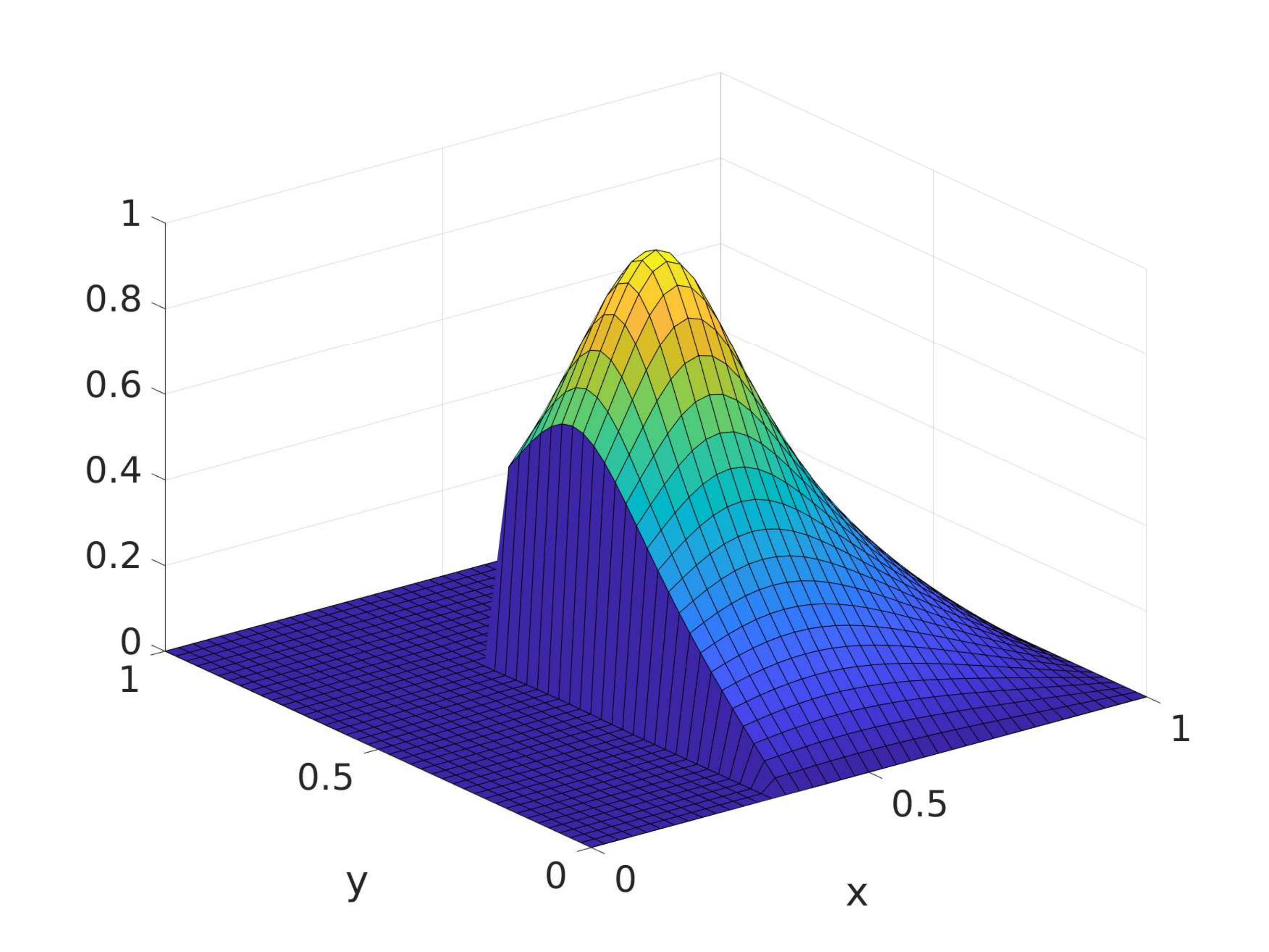}
\includegraphics[width=0.33\textwidth,trim=50 20 50 50,clip]{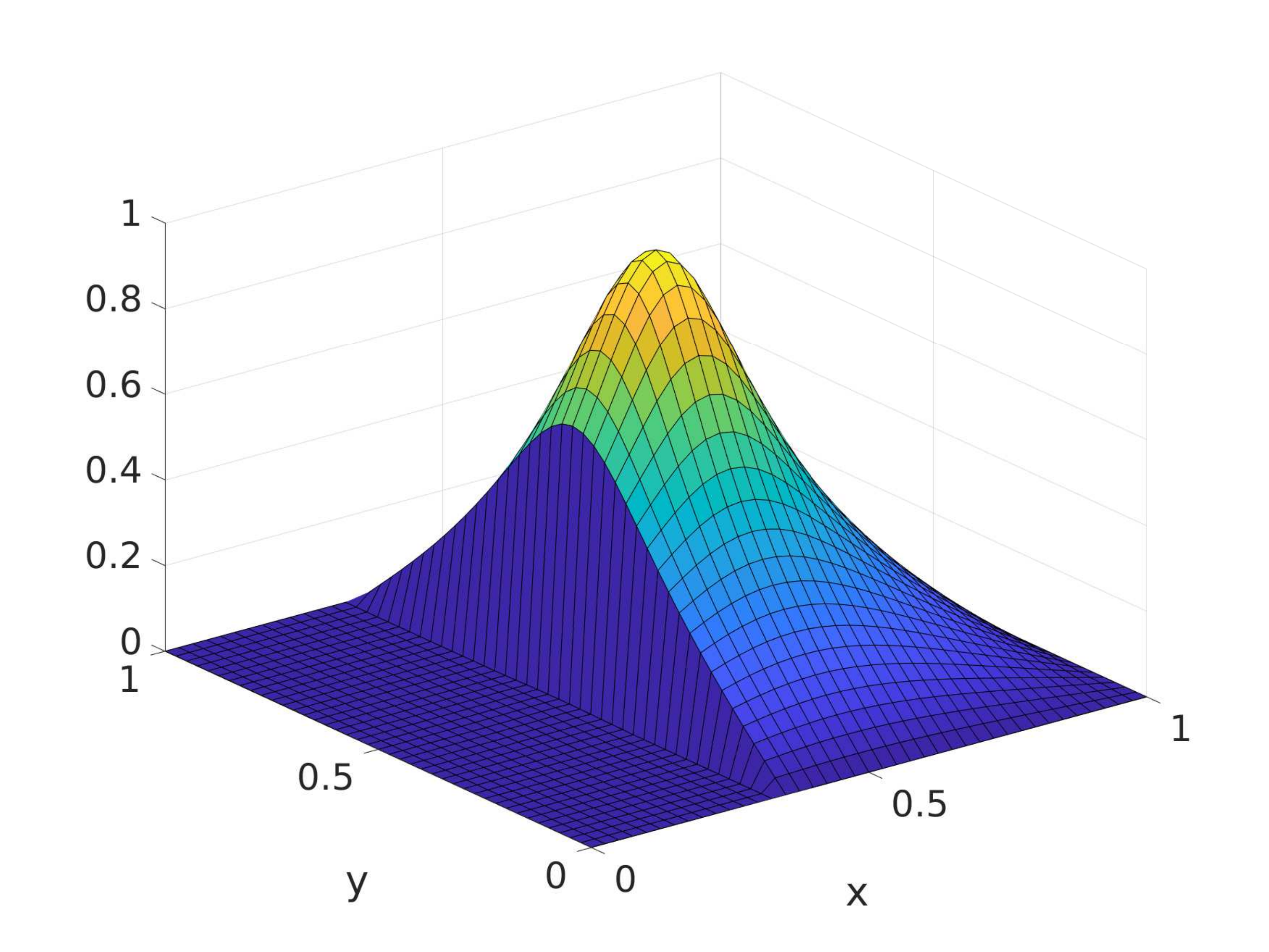}
\includegraphics[width=0.33\textwidth,trim=50 20 50 50,clip]{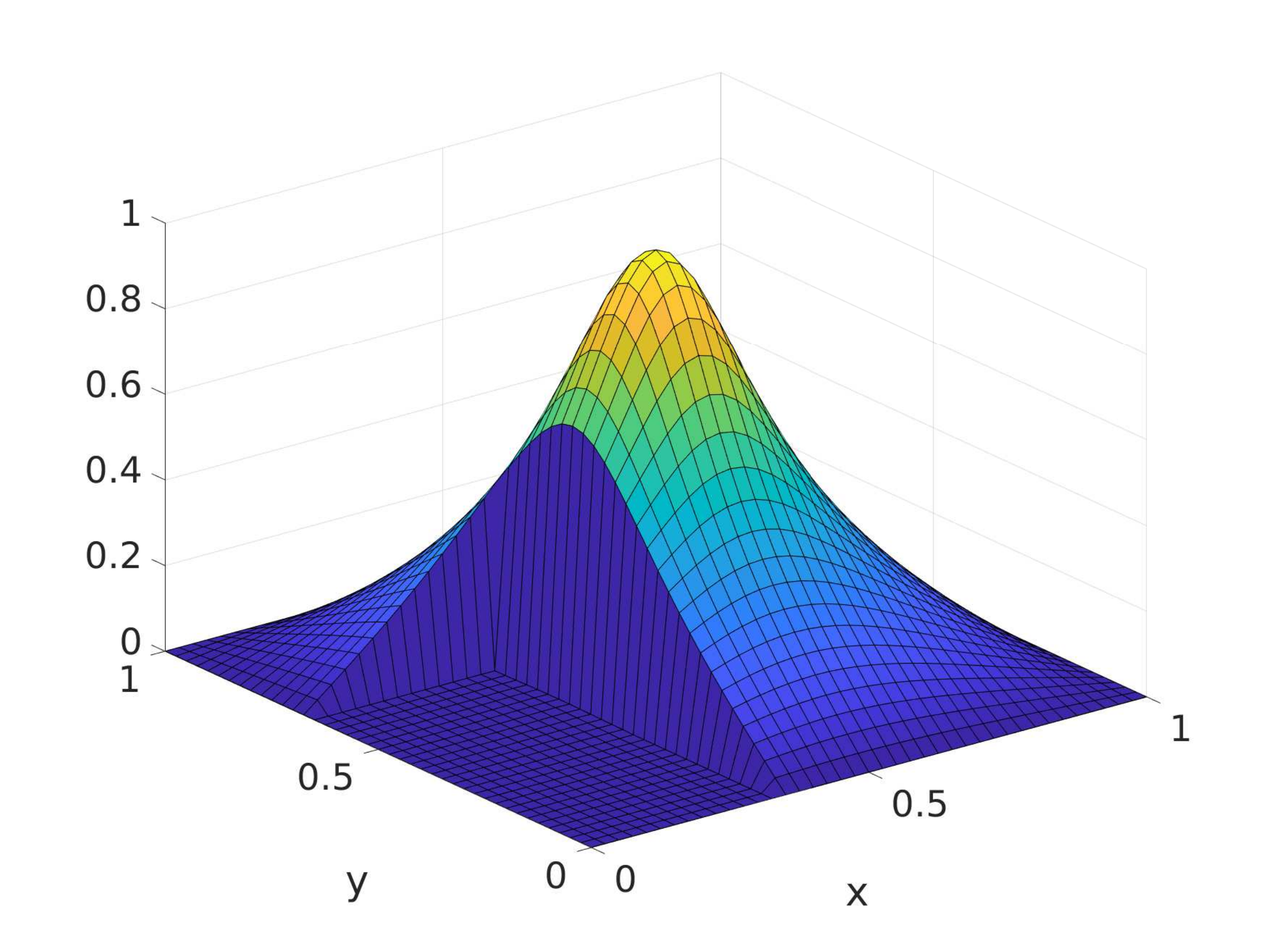}}
  \mbox{\includegraphics[width=0.33\textwidth,trim=50 20 50 50,clip]{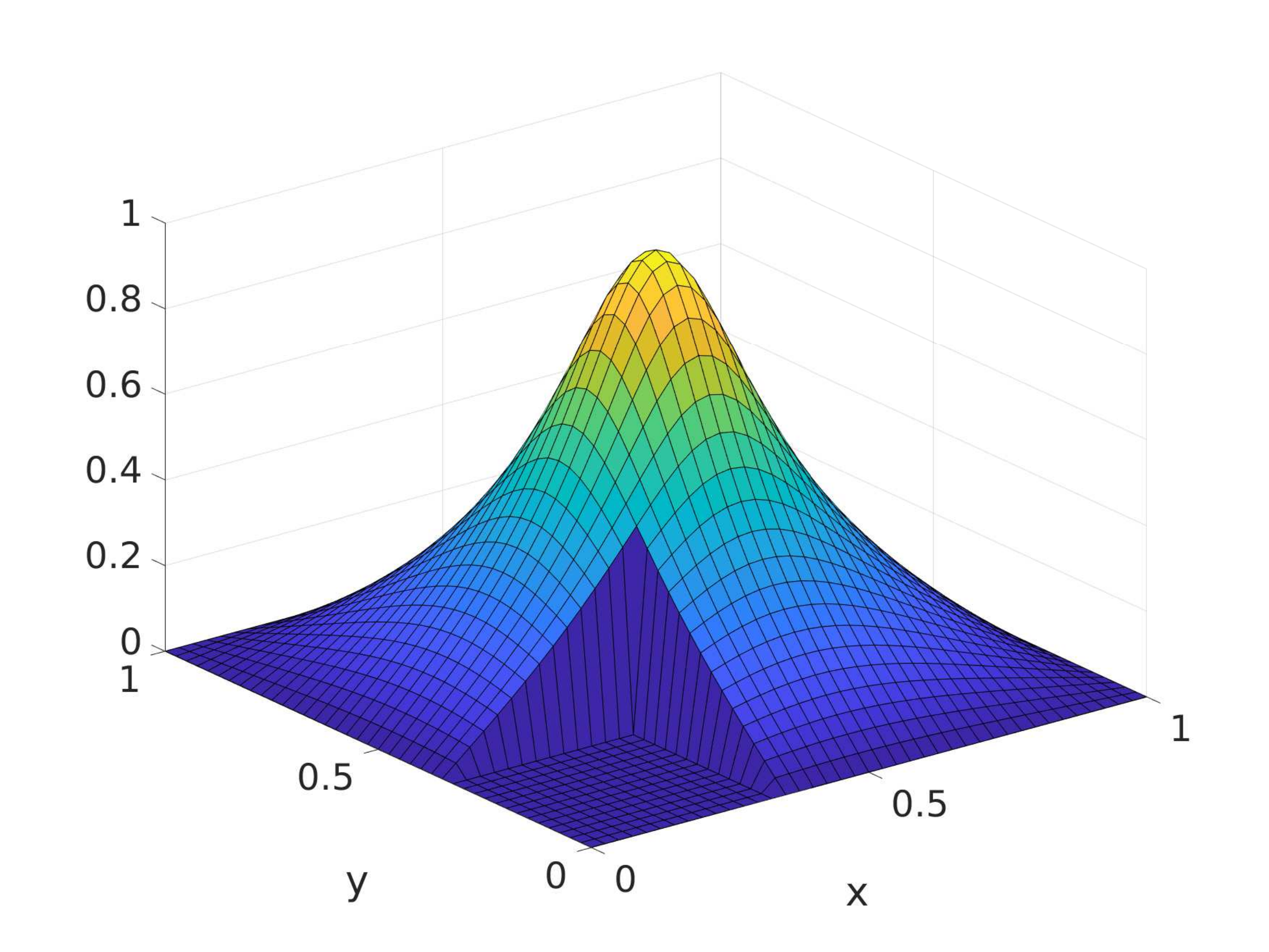}
\includegraphics[width=0.33\textwidth,trim=50 20 50 50,clip]{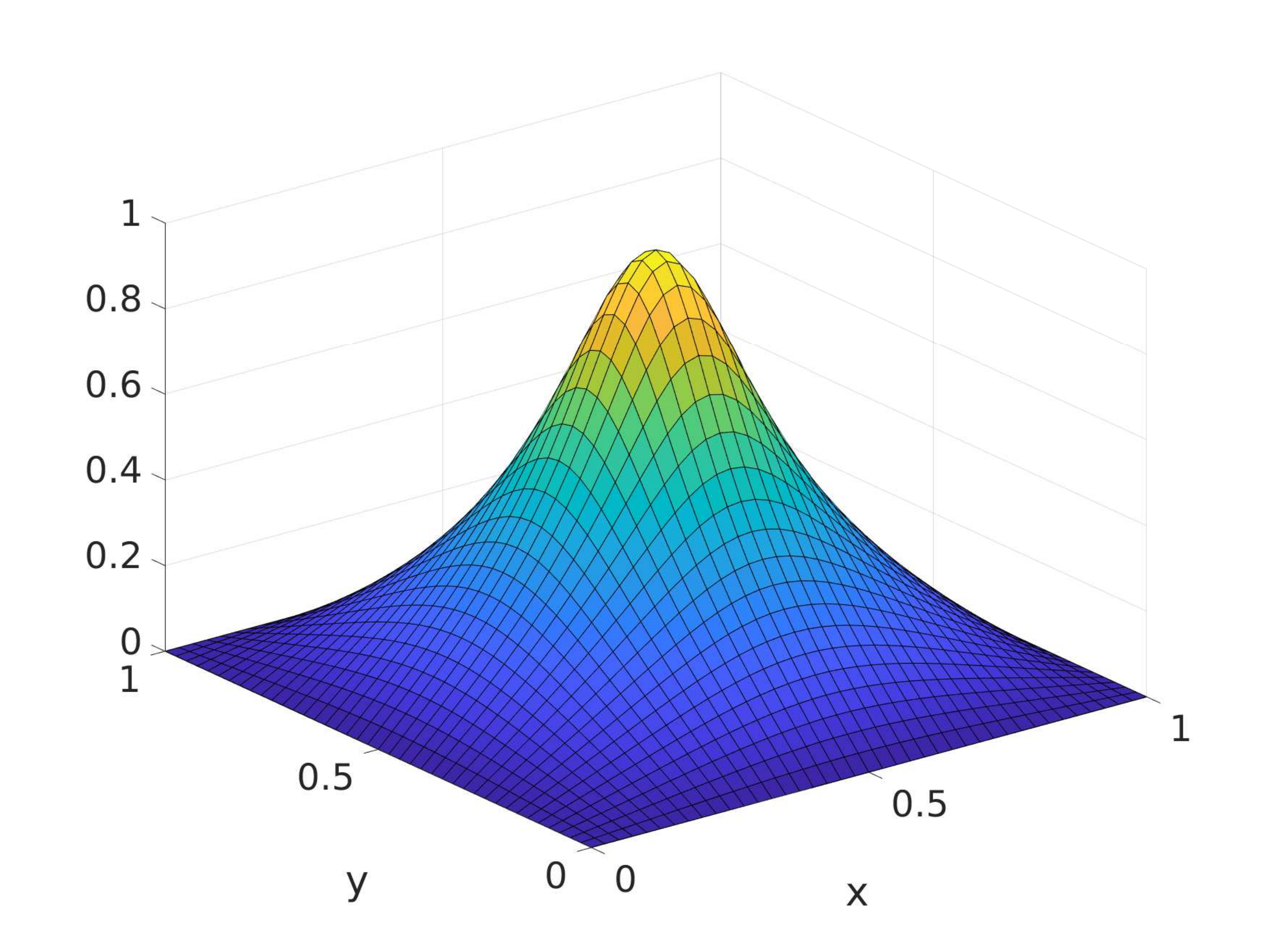}}\hfill
  \caption{Forward and backward sweep for an optimal Schwarz method
    obtained by a block LU decomposition for the model problem and
    $3\times 3$ subdomains using an L-sweep. We observe again
    convergence after one double sweep.}
  \label{OptimalSchwarzLUSweepL}
\end{figure}
the ordering for L-sweeps. We see that the algorithm also converges in
one double sweep, by construction. The corresponding block LU factors
are shown in Figure \ref{OSMLUfactorsL}.
\begin{figure}
  \centering
  \mbox{\includegraphics[width=0.49\textwidth]{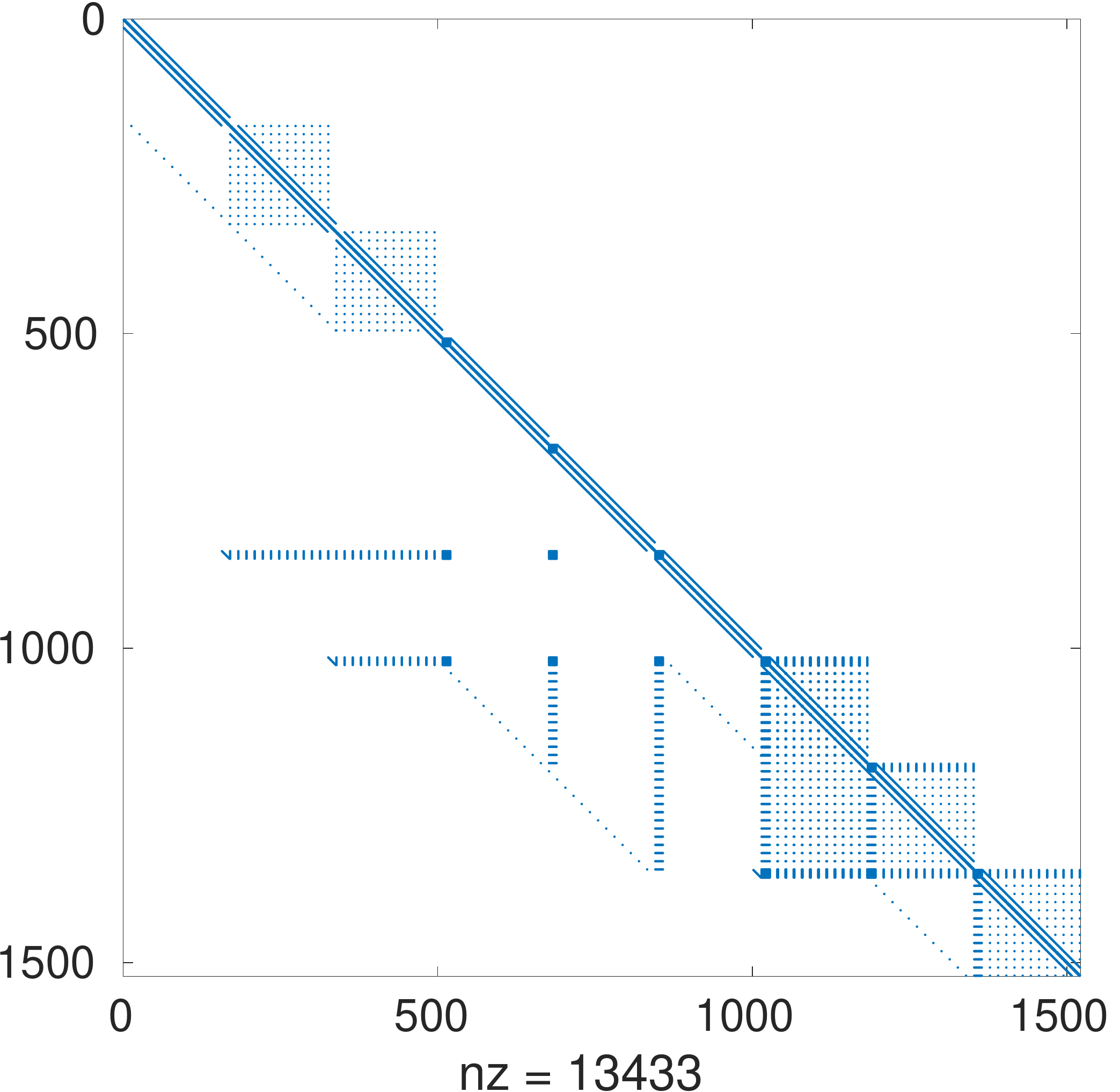}
\includegraphics[width=0.49\textwidth]{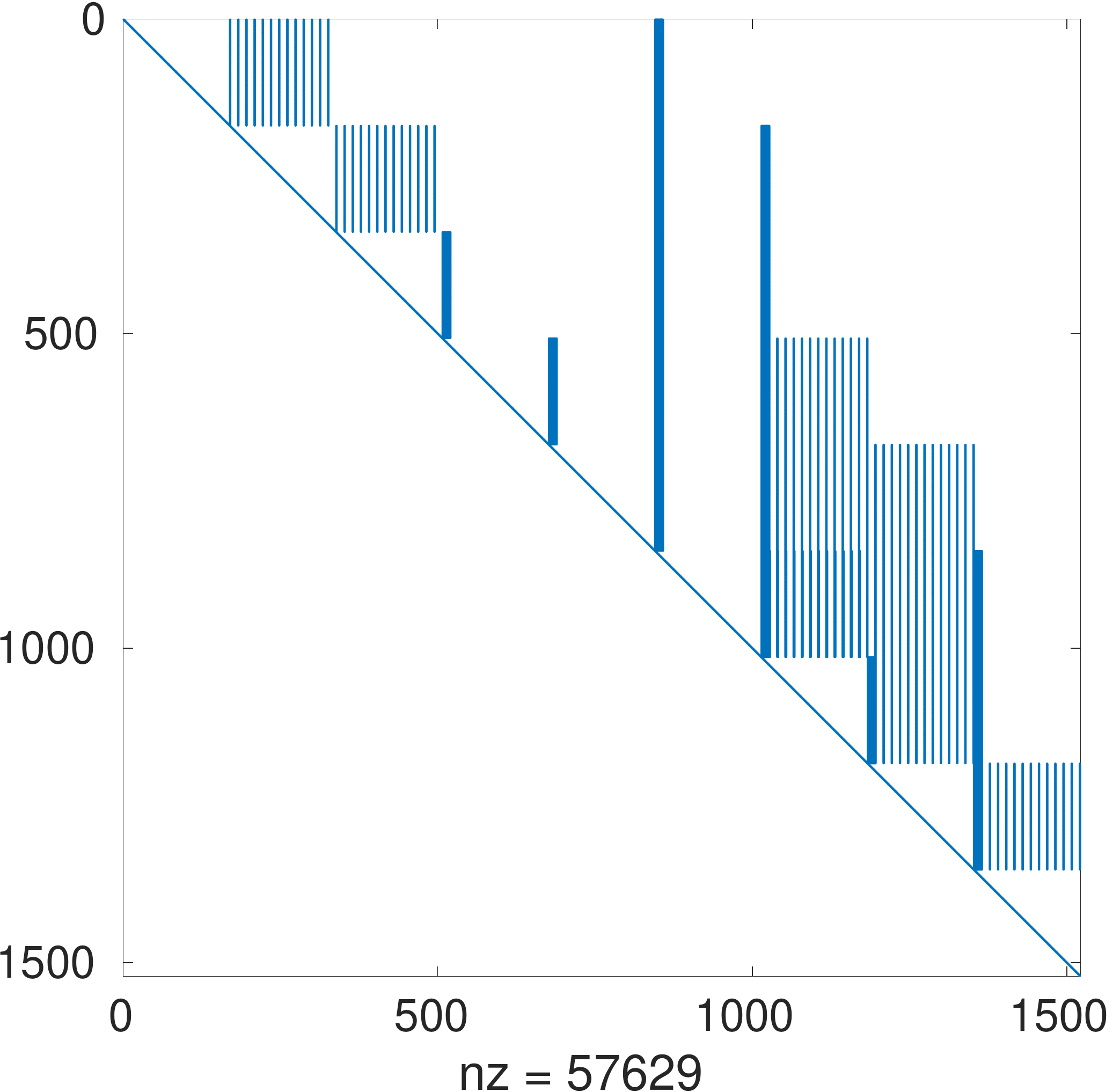}}
  \caption{Sparsity of the block $L$ and $U$ factors of the optimal
    Schwarz method for the $3\times 3$ subdomain decomposition from
    Figure \ref{OptimalSchwarzLUSweep} when using the L-sweep
    ordering.}
  \label{OSMLUfactorsL}
\end{figure}

We show in Figure \ref{OptimalSchwarzLUSweepD} the further popular
diagonal ordering for sweeping.
\begin{figure}
  \centering
  \mbox{\includegraphics[width=0.33\textwidth,trim=50 20 50 50,clip]{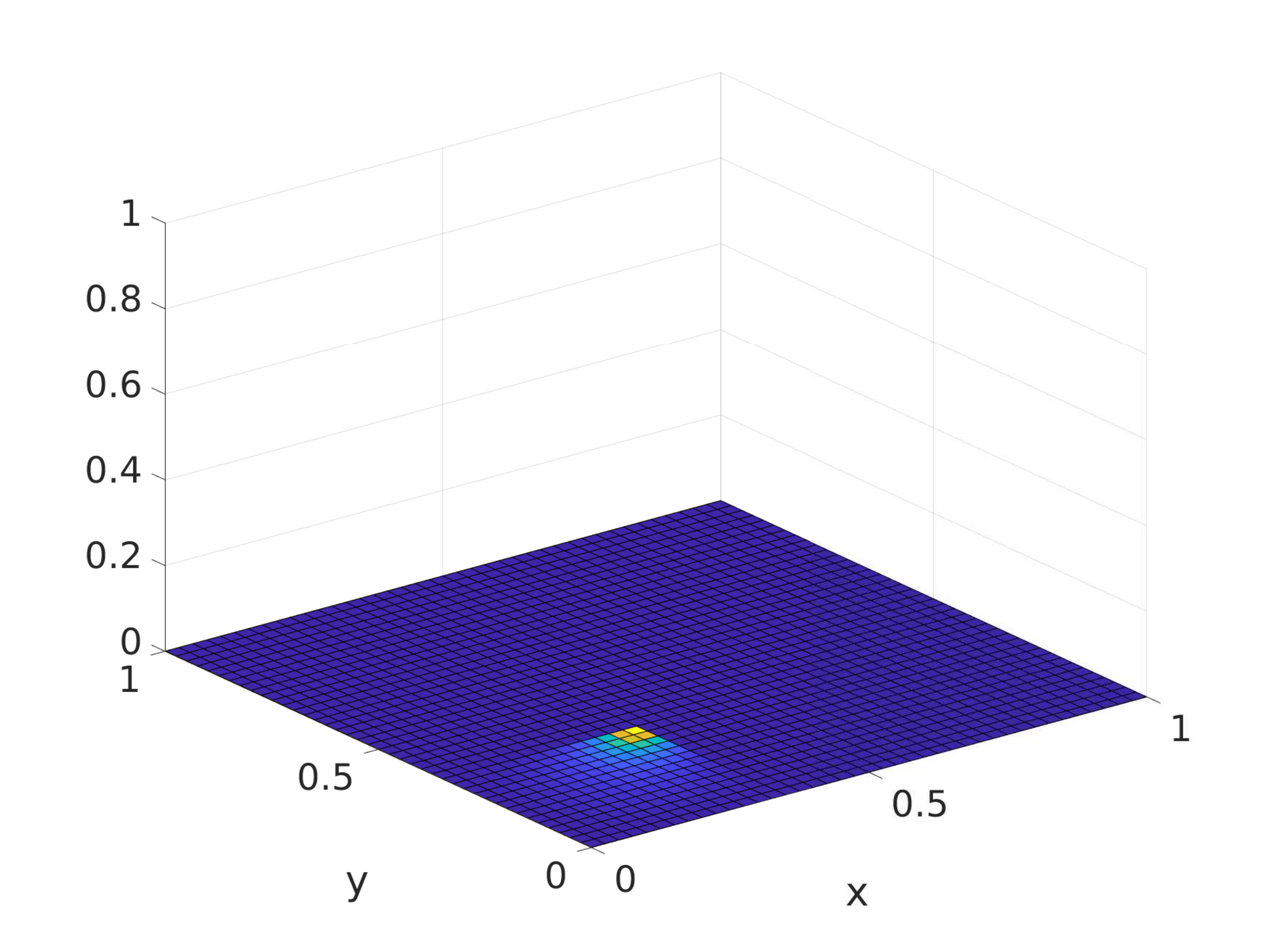}
\includegraphics[width=0.33\textwidth,trim=50 20 50 50,clip]{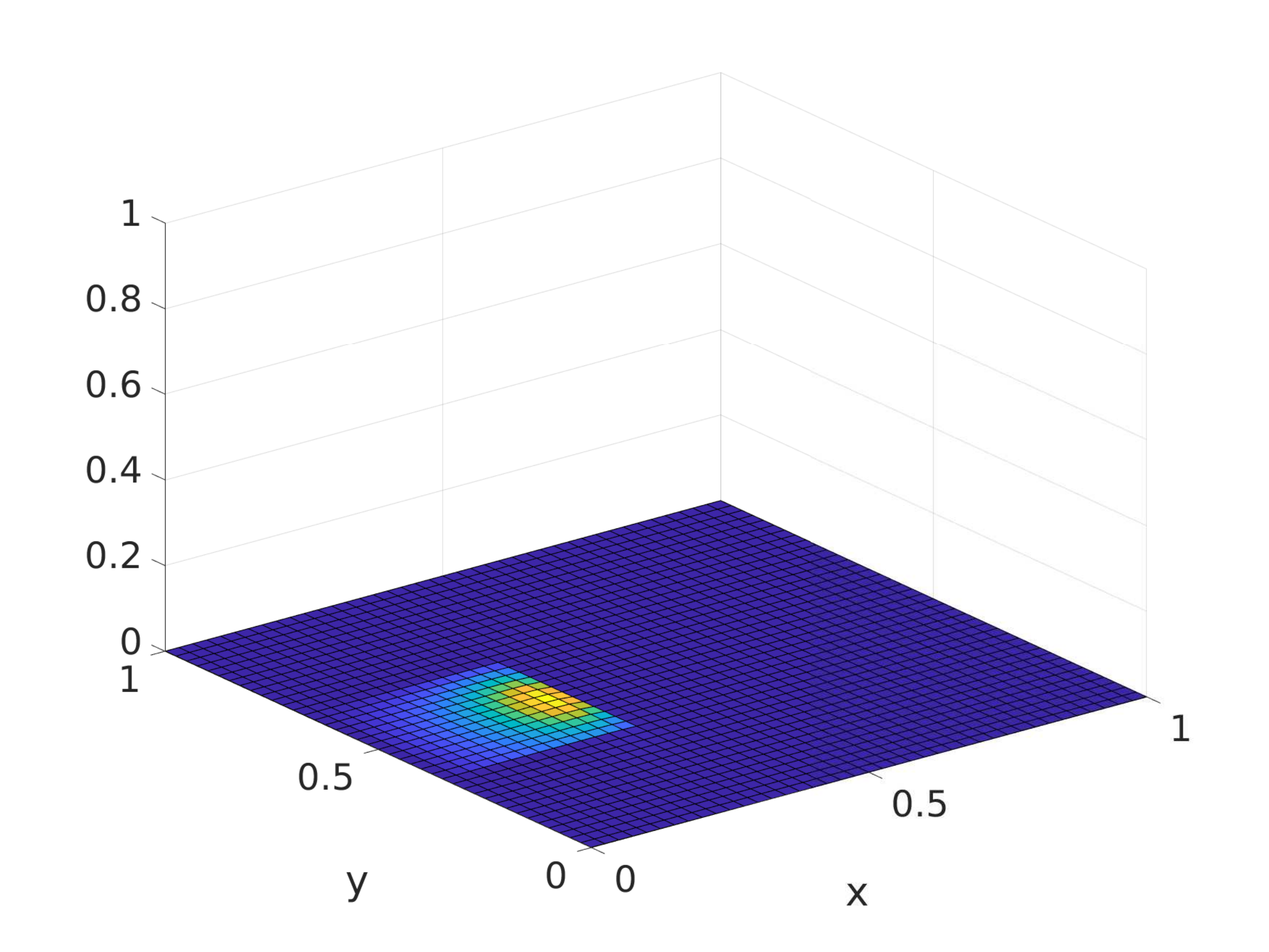}
\includegraphics[width=0.33\textwidth,trim=50 20 50 50,clip]{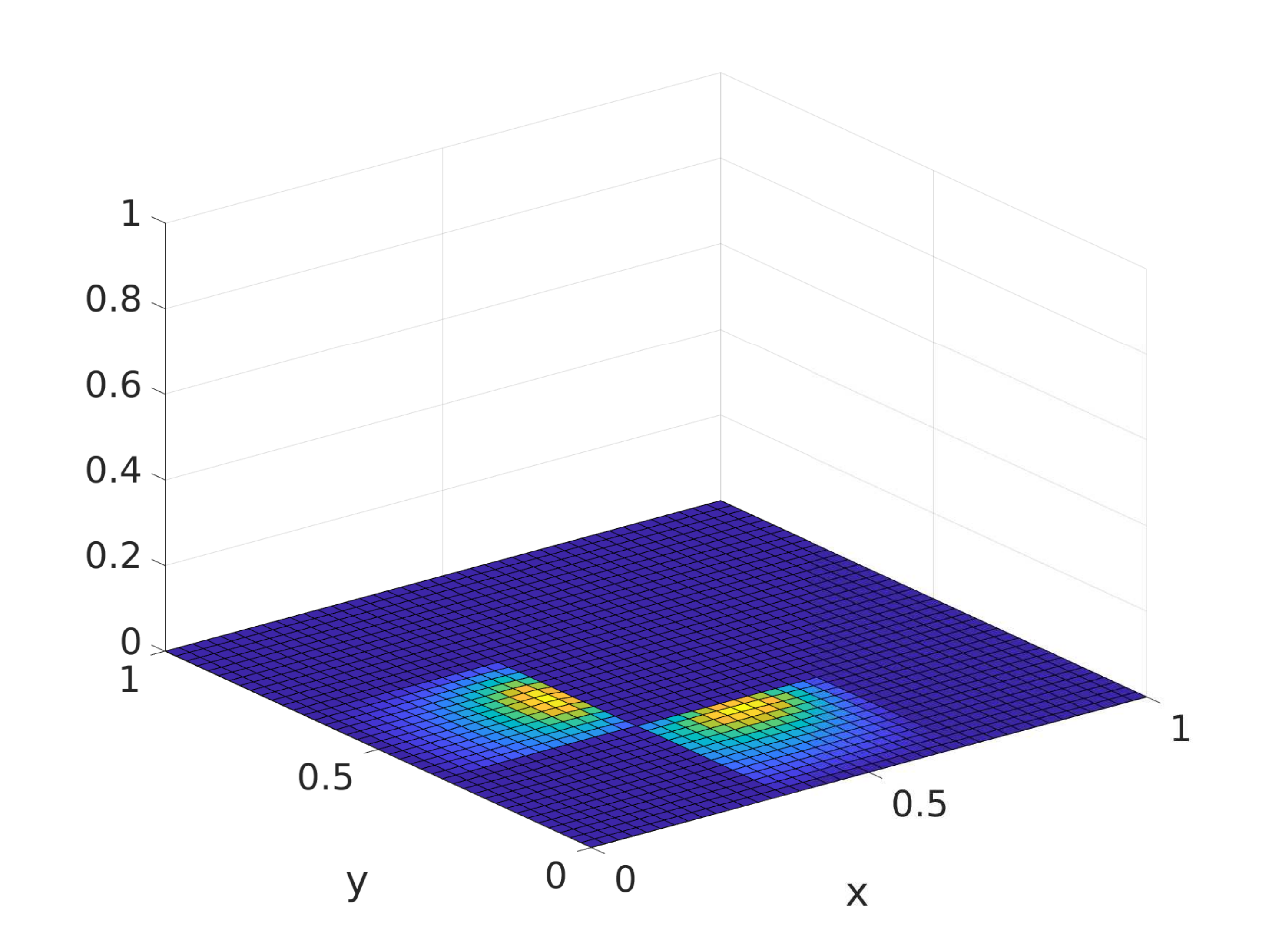}}
  \mbox{\includegraphics[width=0.33\textwidth,trim=50 20 50 50,clip]{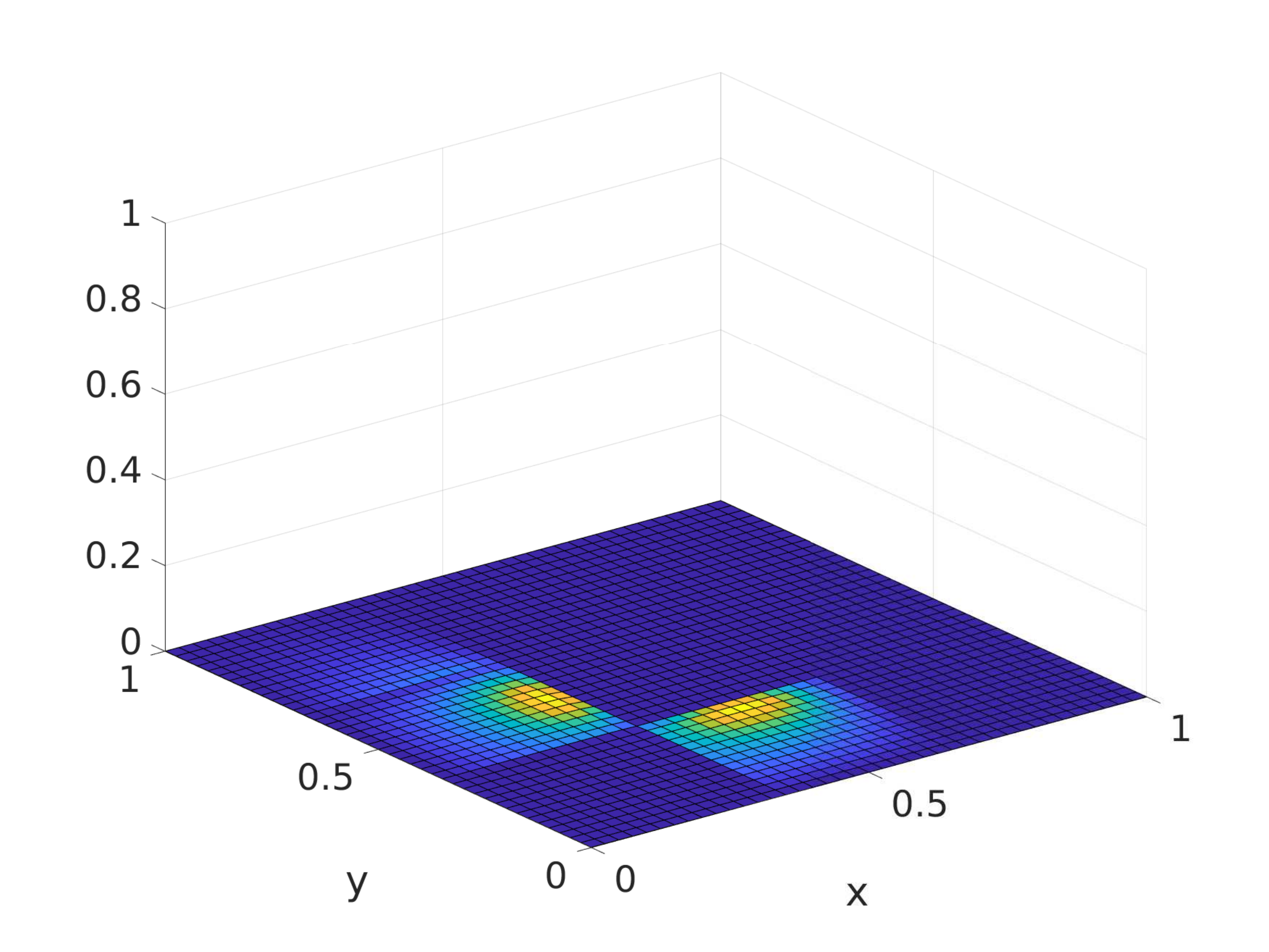}
\includegraphics[width=0.33\textwidth,trim=50 20 50 50,clip]{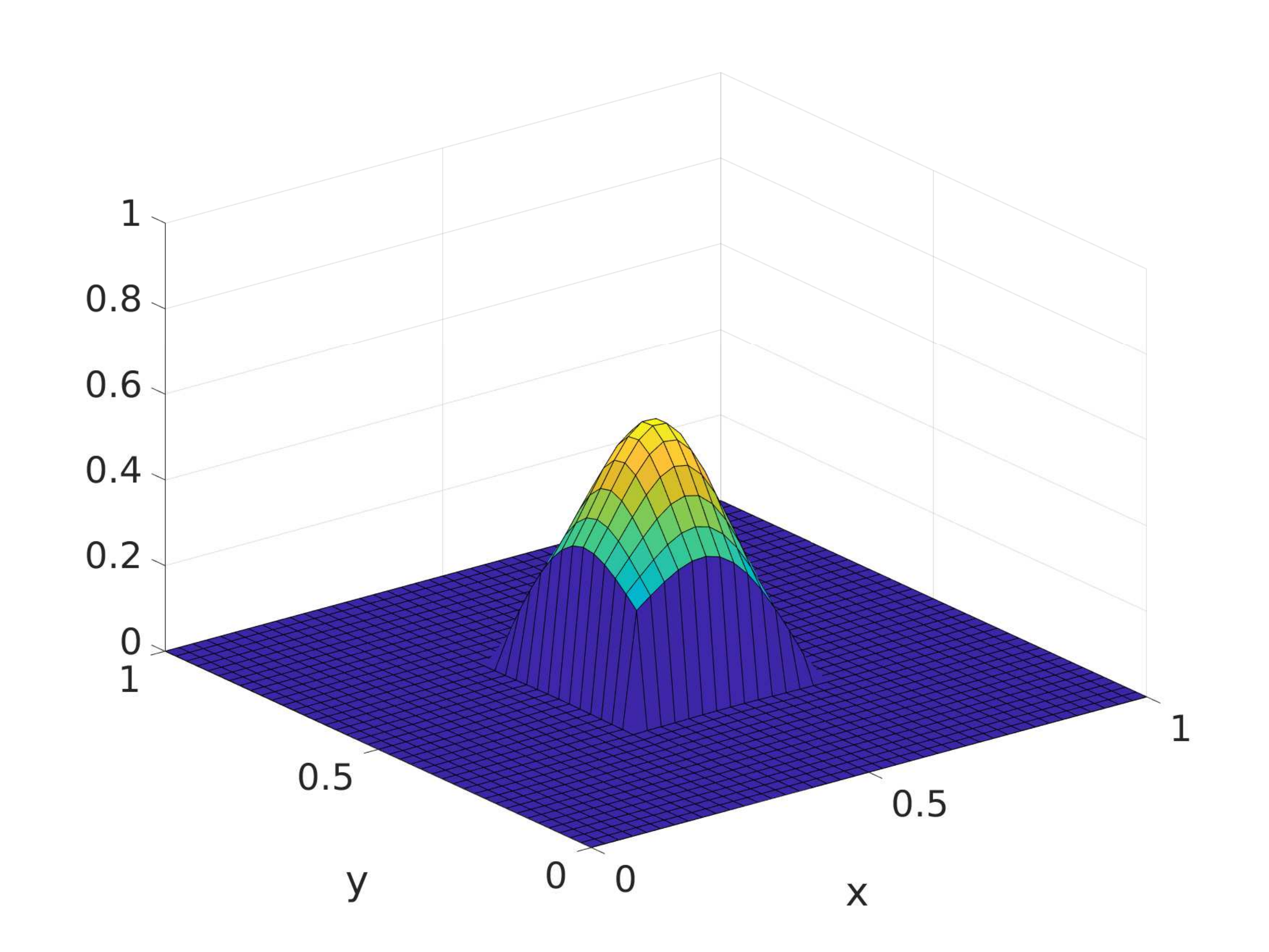}
\includegraphics[width=0.33\textwidth,trim=50 20 50 50,clip]{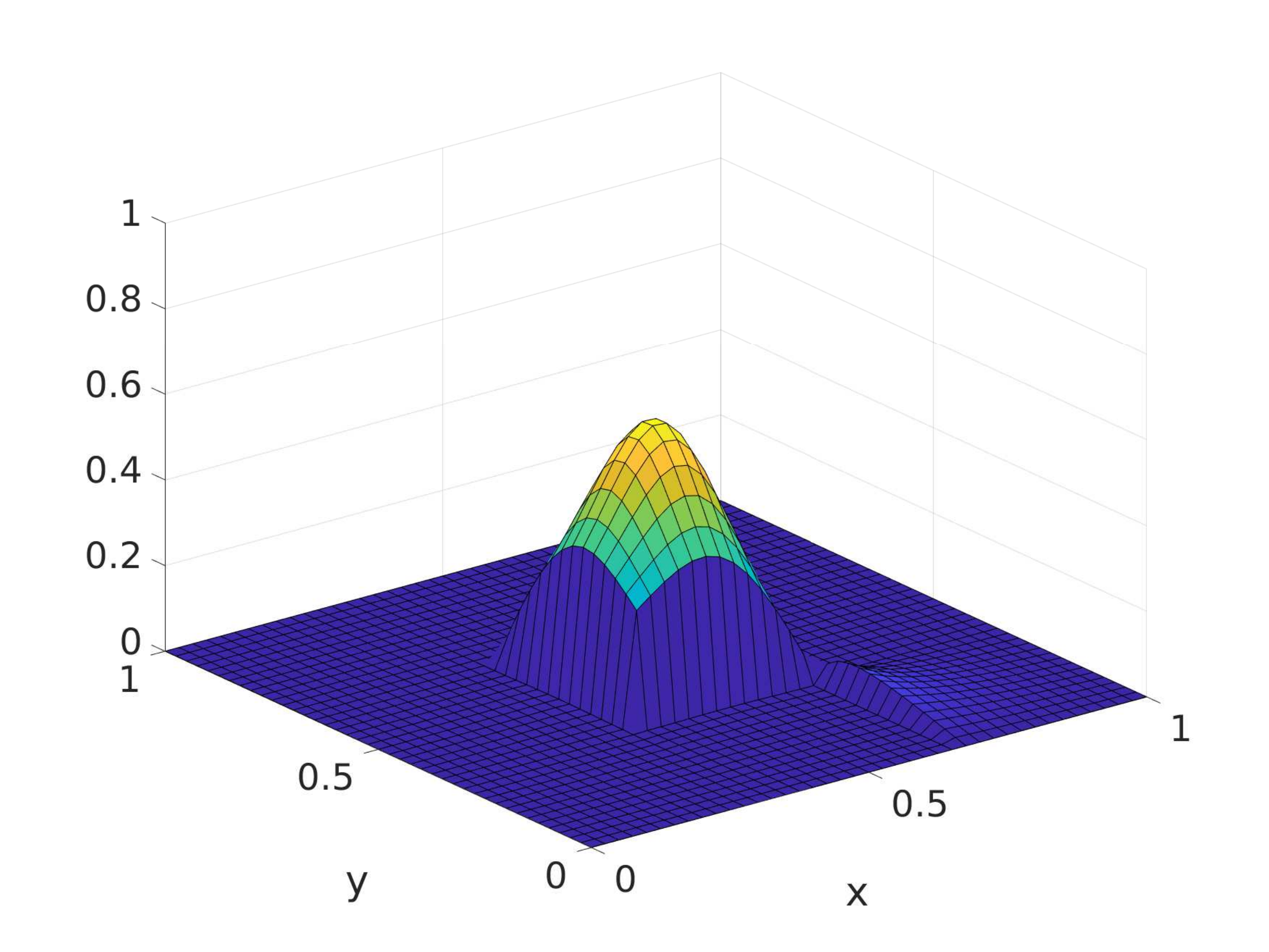}}
  \mbox{\includegraphics[width=0.33\textwidth,trim=50 20 50 50,clip]{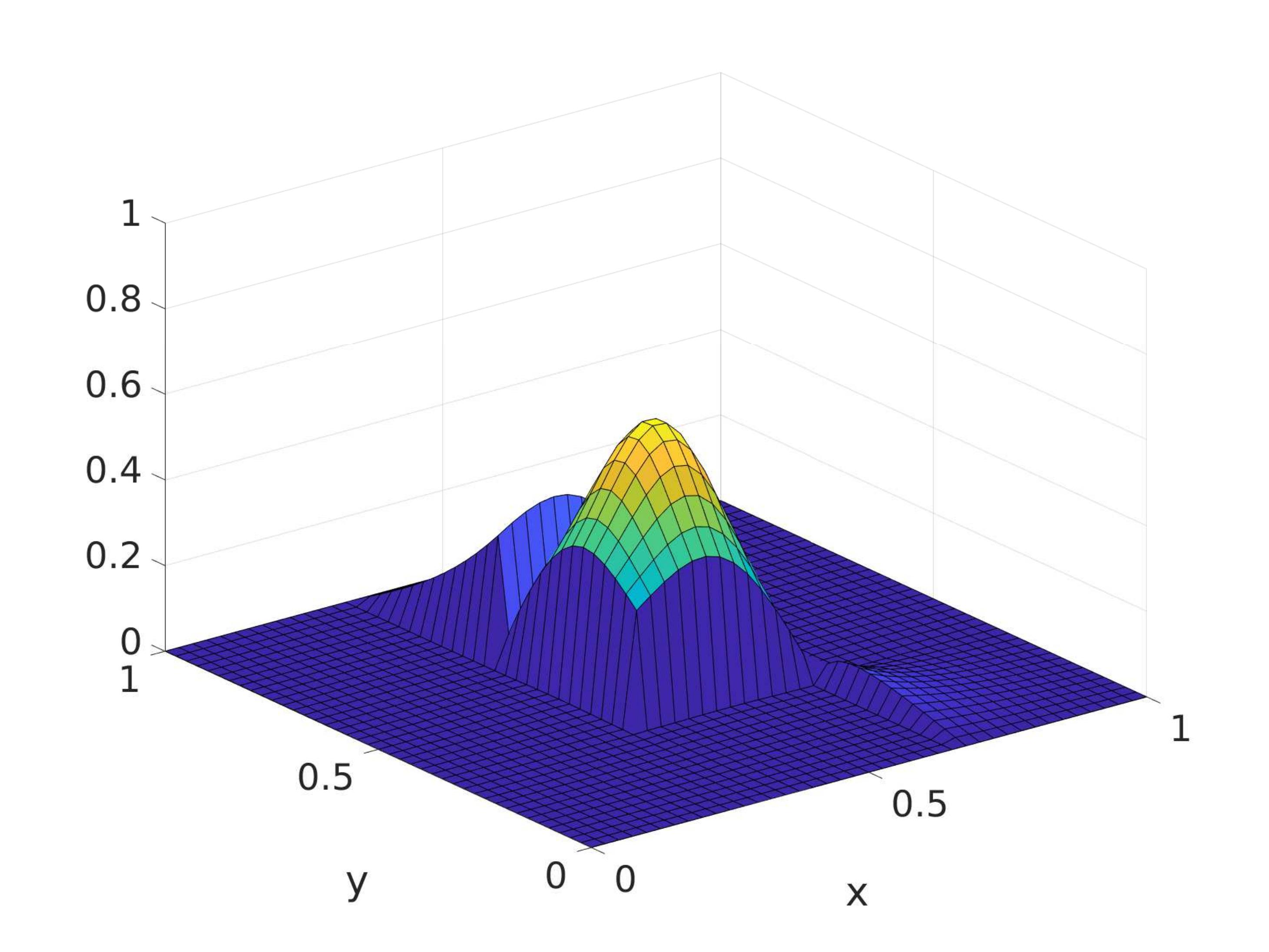}
\includegraphics[width=0.33\textwidth,trim=50 20 50 50,clip]{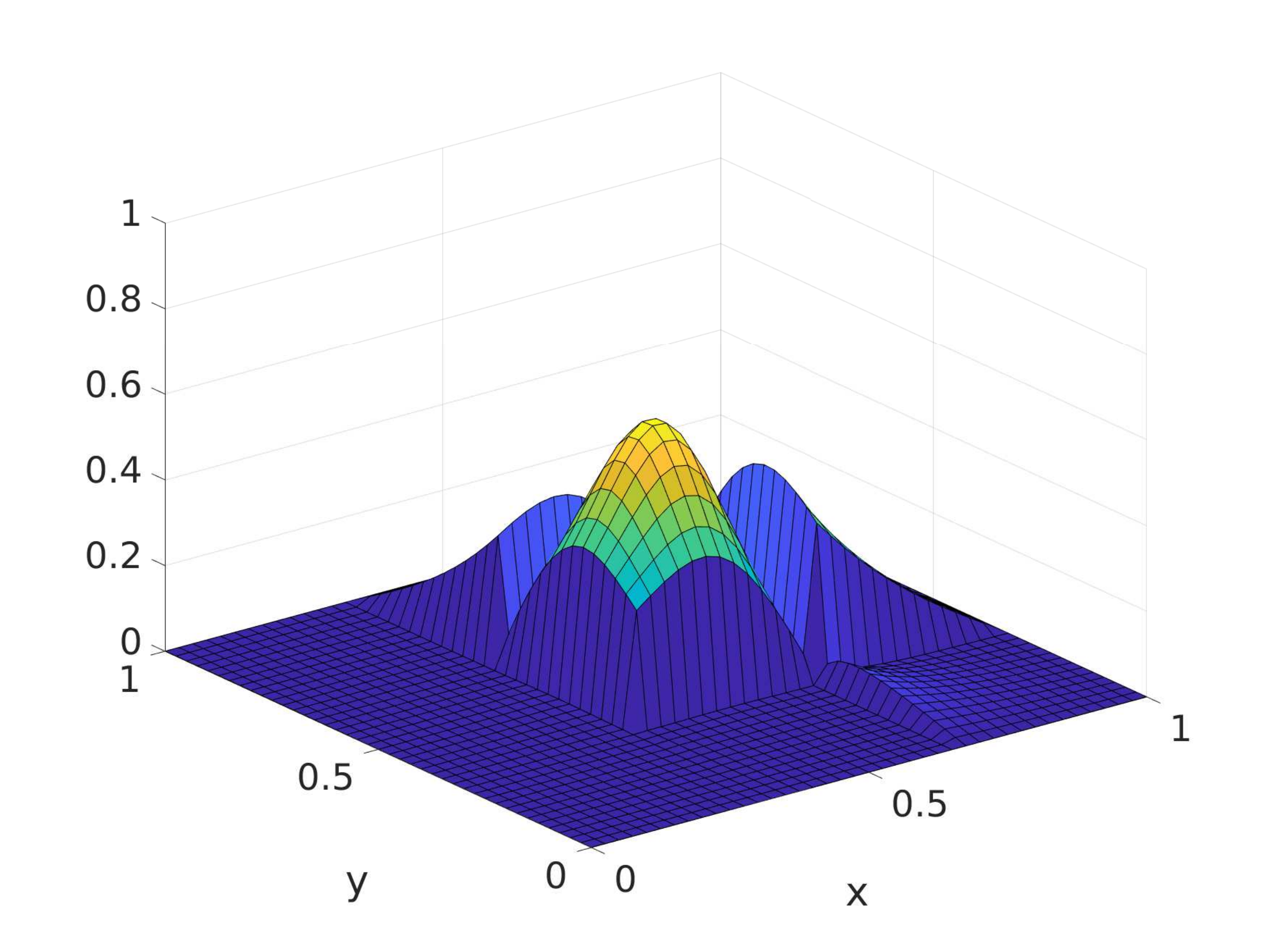}
\includegraphics[width=0.33\textwidth,trim=50 20 50 50,clip]{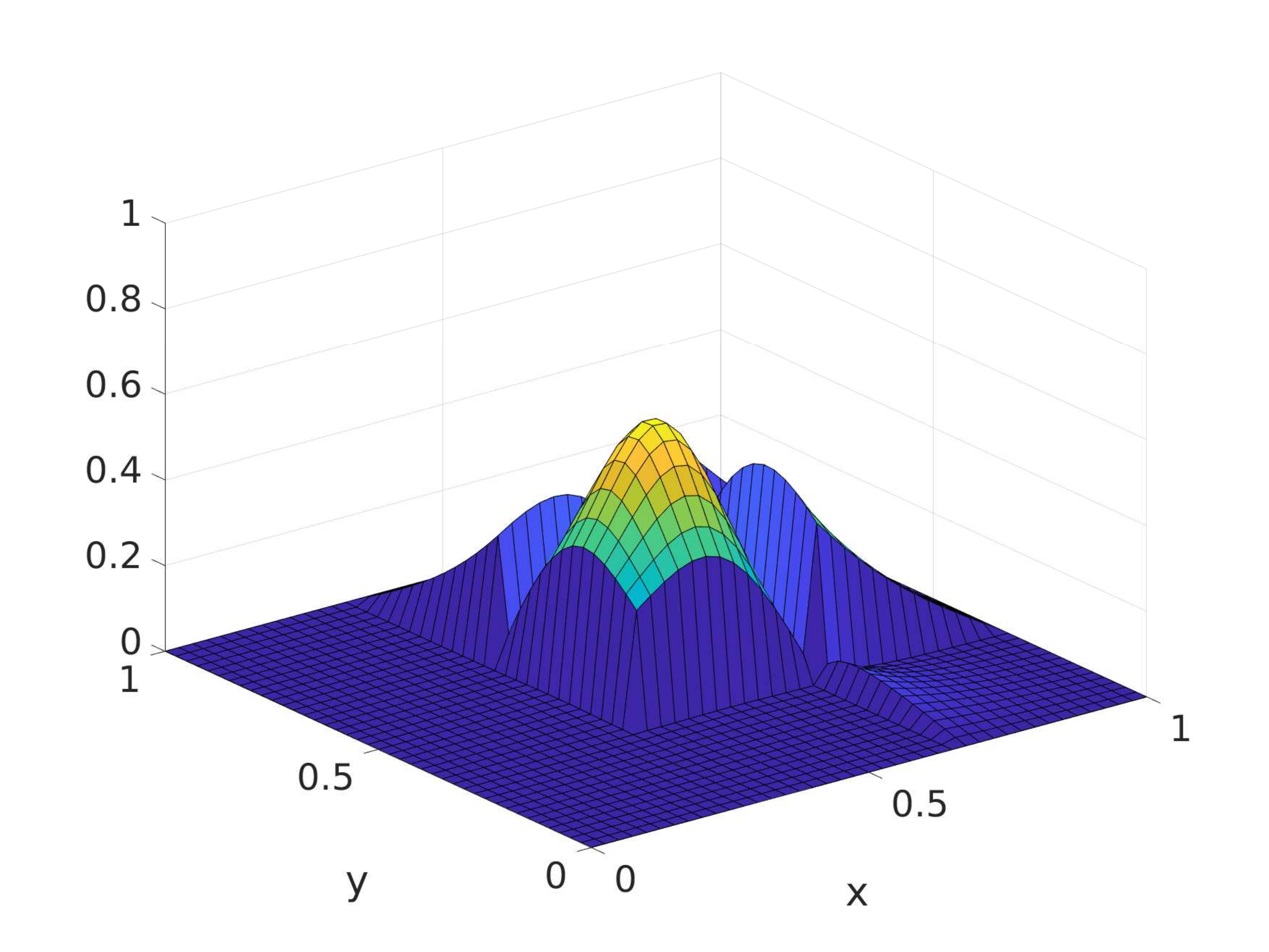}}
  \mbox{\includegraphics[width=0.33\textwidth,trim=50 20 50 50,clip]{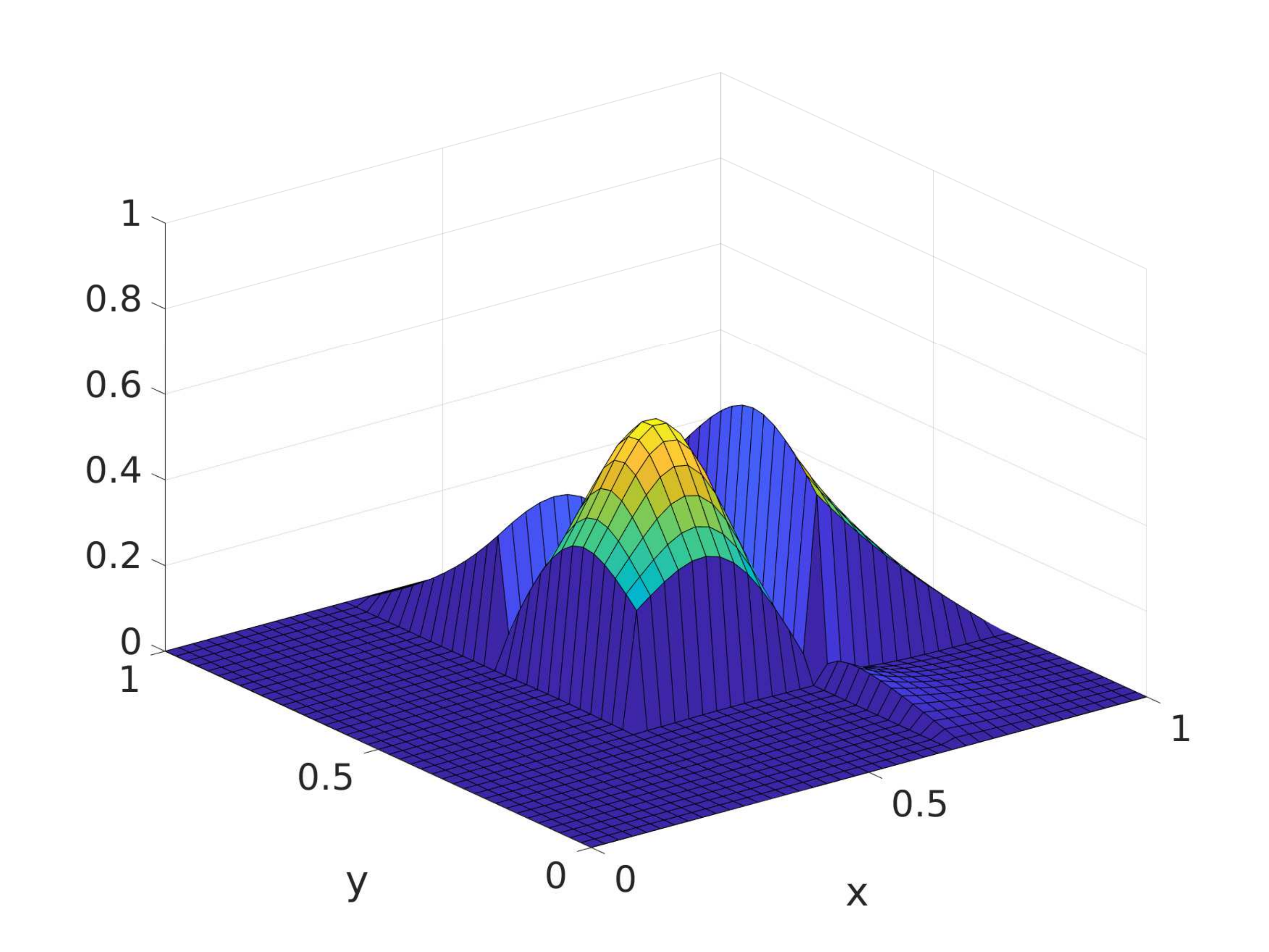}
\includegraphics[width=0.33\textwidth,trim=50 20 50 50,clip]{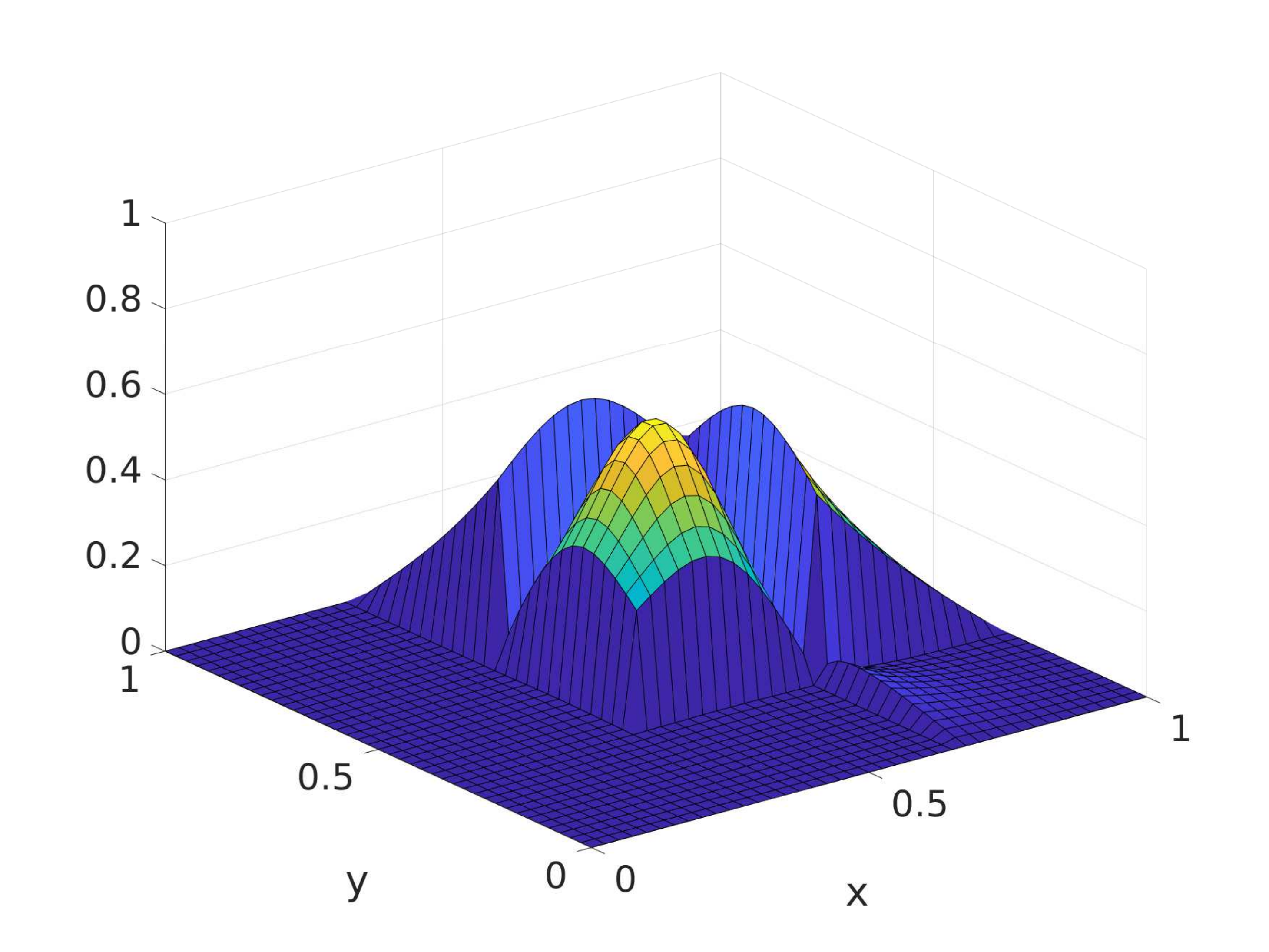}
\includegraphics[width=0.33\textwidth,trim=50 20 50 50,clip]{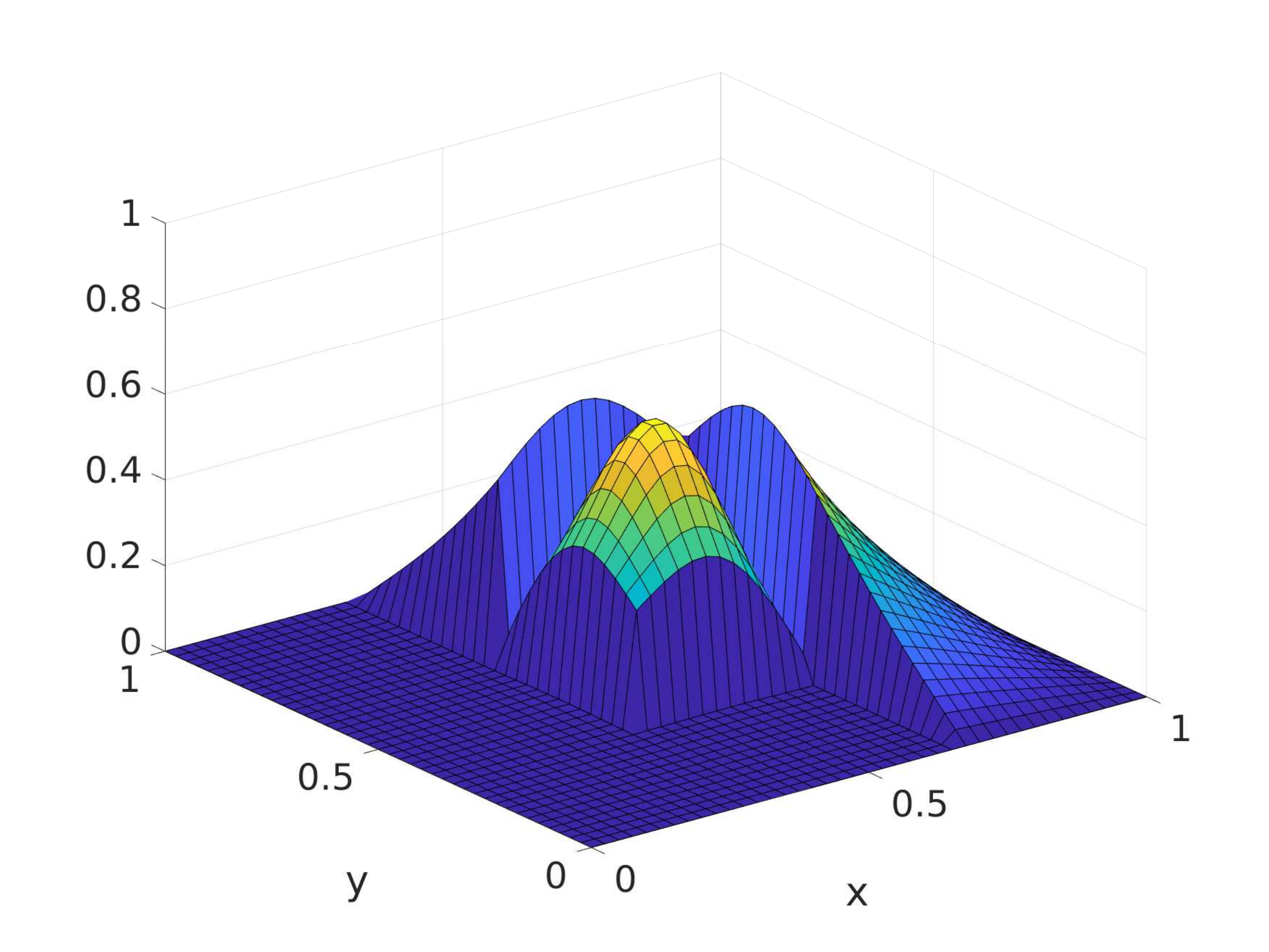}}
  \mbox{\includegraphics[width=0.33\textwidth,trim=50 20 50 50,clip]{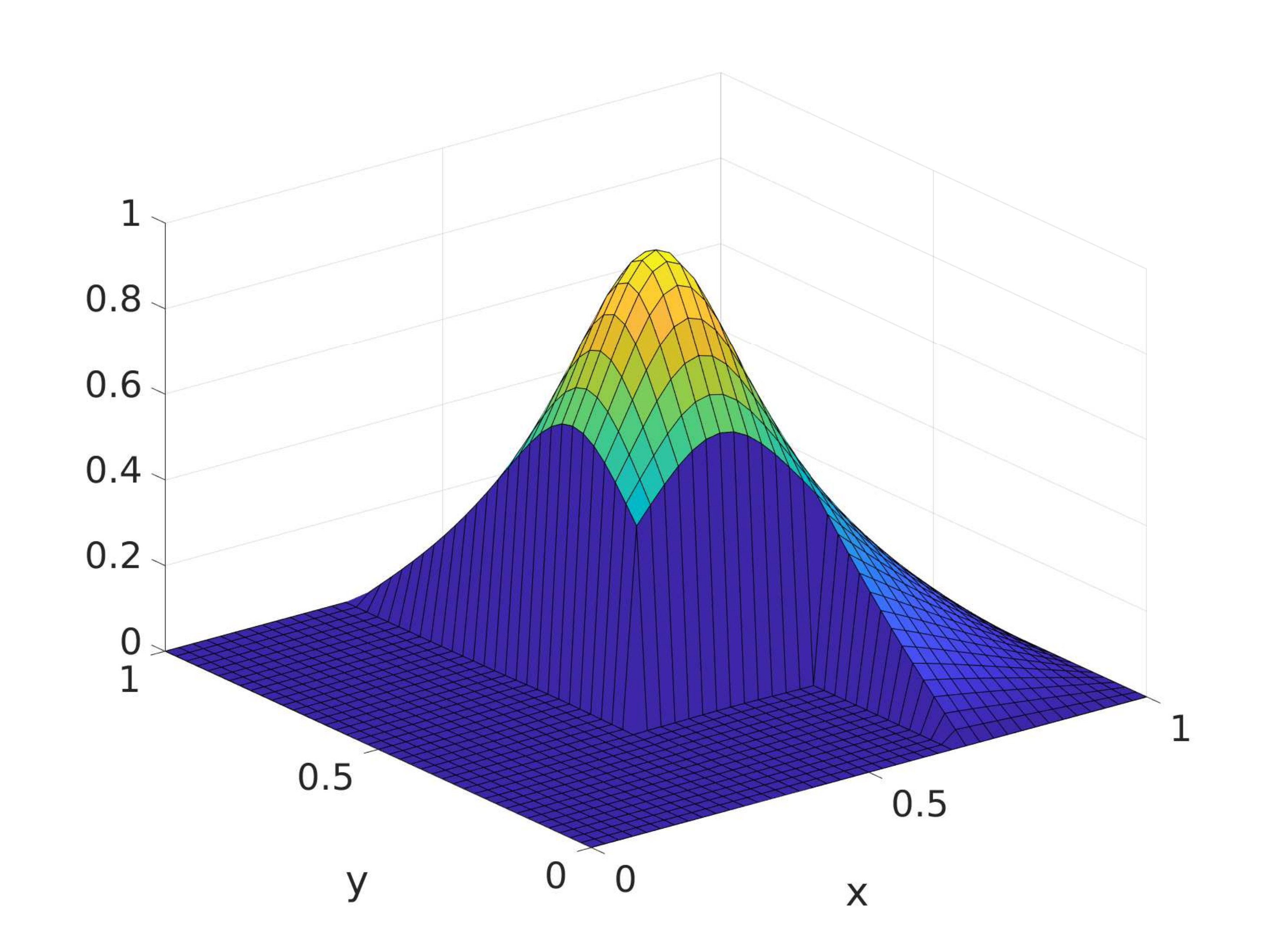}
\includegraphics[width=0.33\textwidth,trim=50 20 50 50,clip]{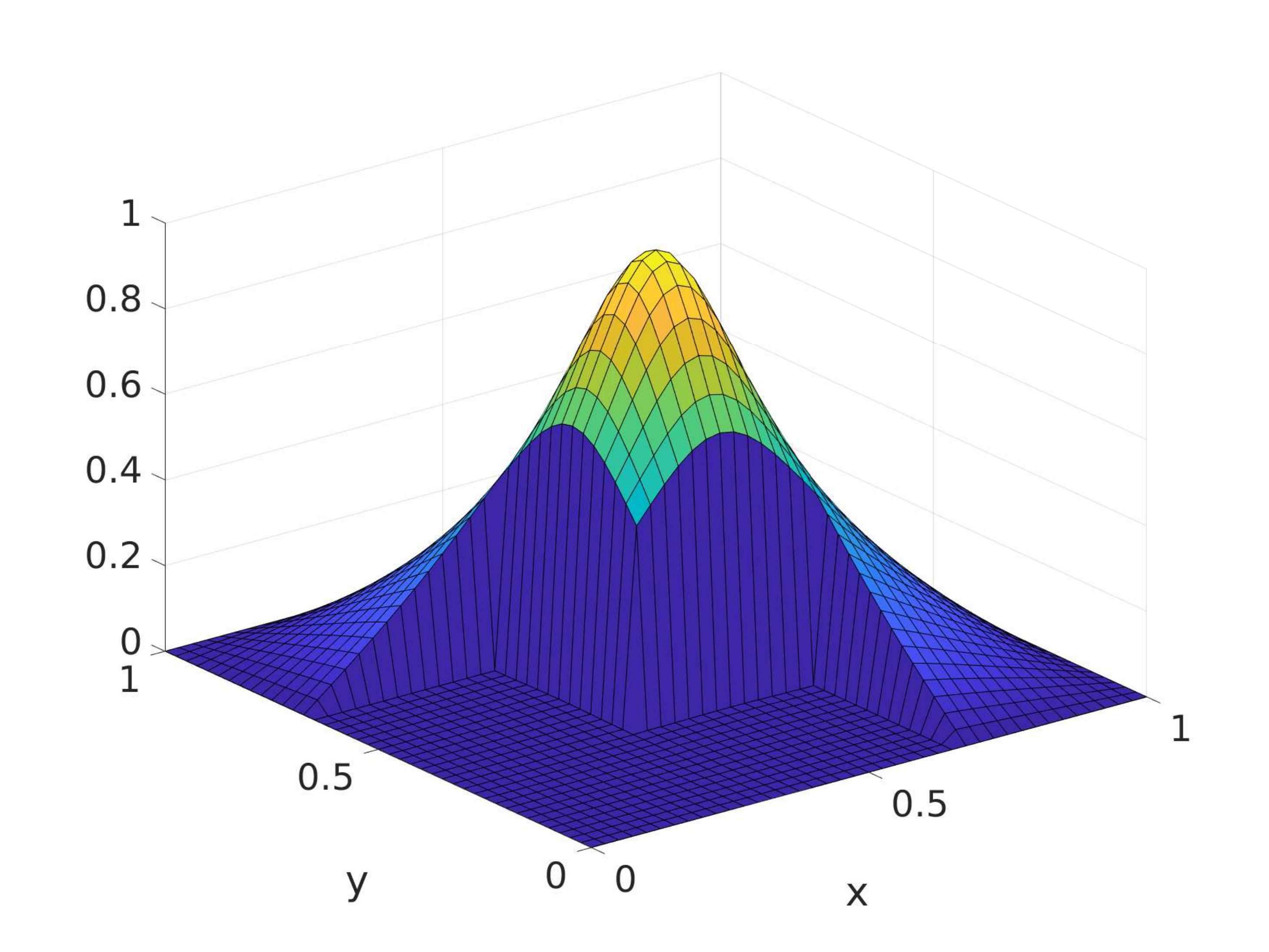}
\includegraphics[width=0.33\textwidth,trim=50 20 50 50,clip]{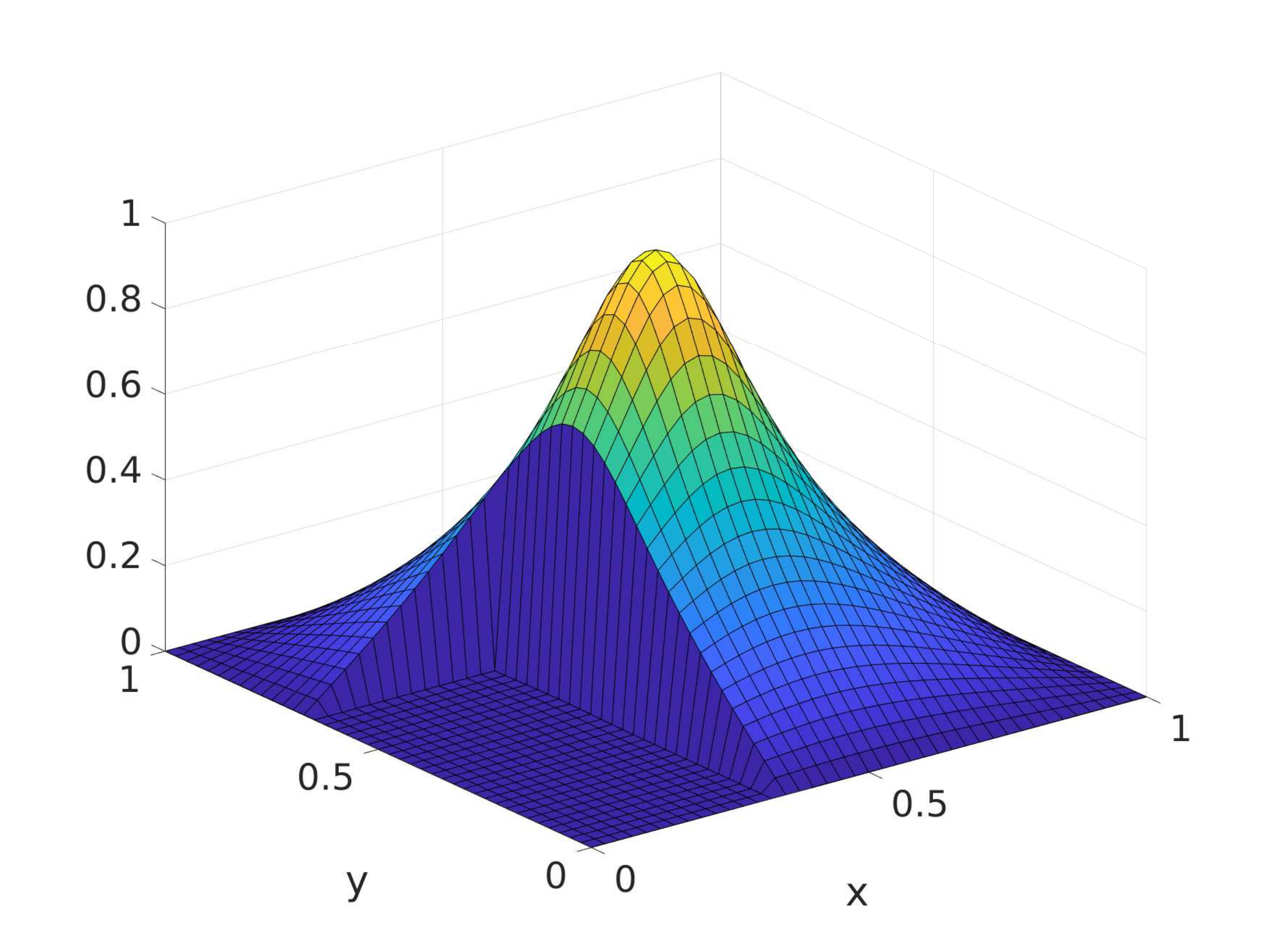}}
  \mbox{\includegraphics[width=0.33\textwidth,trim=50 20 50 50,clip]{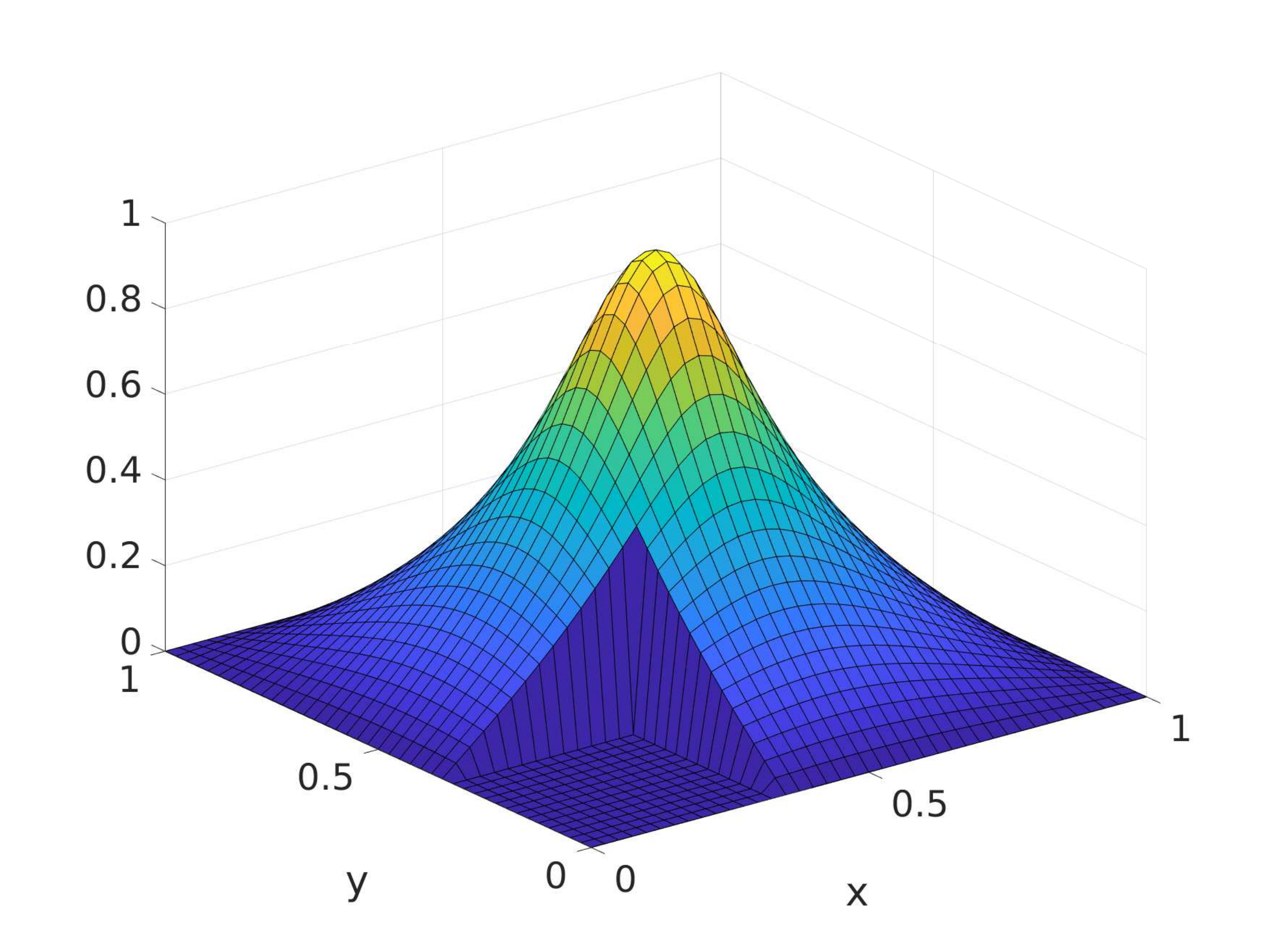}
\includegraphics[width=0.33\textwidth,trim=50 20 50 50,clip]{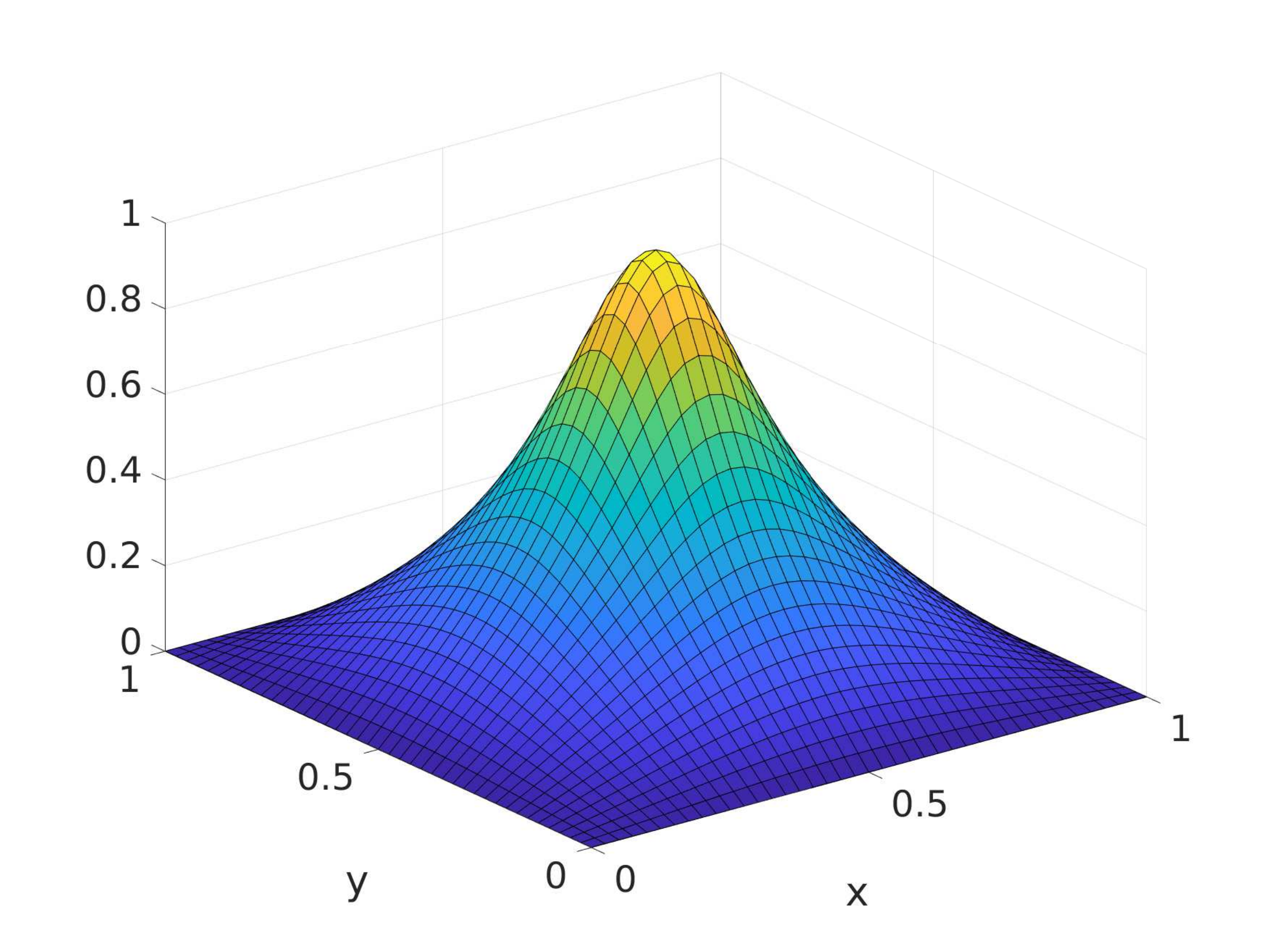}}\hfill
  \caption{Forward and backward sweep for an optimal Schwarz method
    obtained by a block LU decomposition for the model problem and
    $3\times 3$ subdomains using a D-sweep. We observe again
    convergence after one double sweep.}
  \label{OptimalSchwarzLUSweepD}
\end{figure}
Again we see convergence in one double sweep. The corresponding
block LU factors are shown in Figure \ref{OSMLUfactorsD}.
\begin{figure}
  \centering
  \mbox{\includegraphics[width=0.49\textwidth]{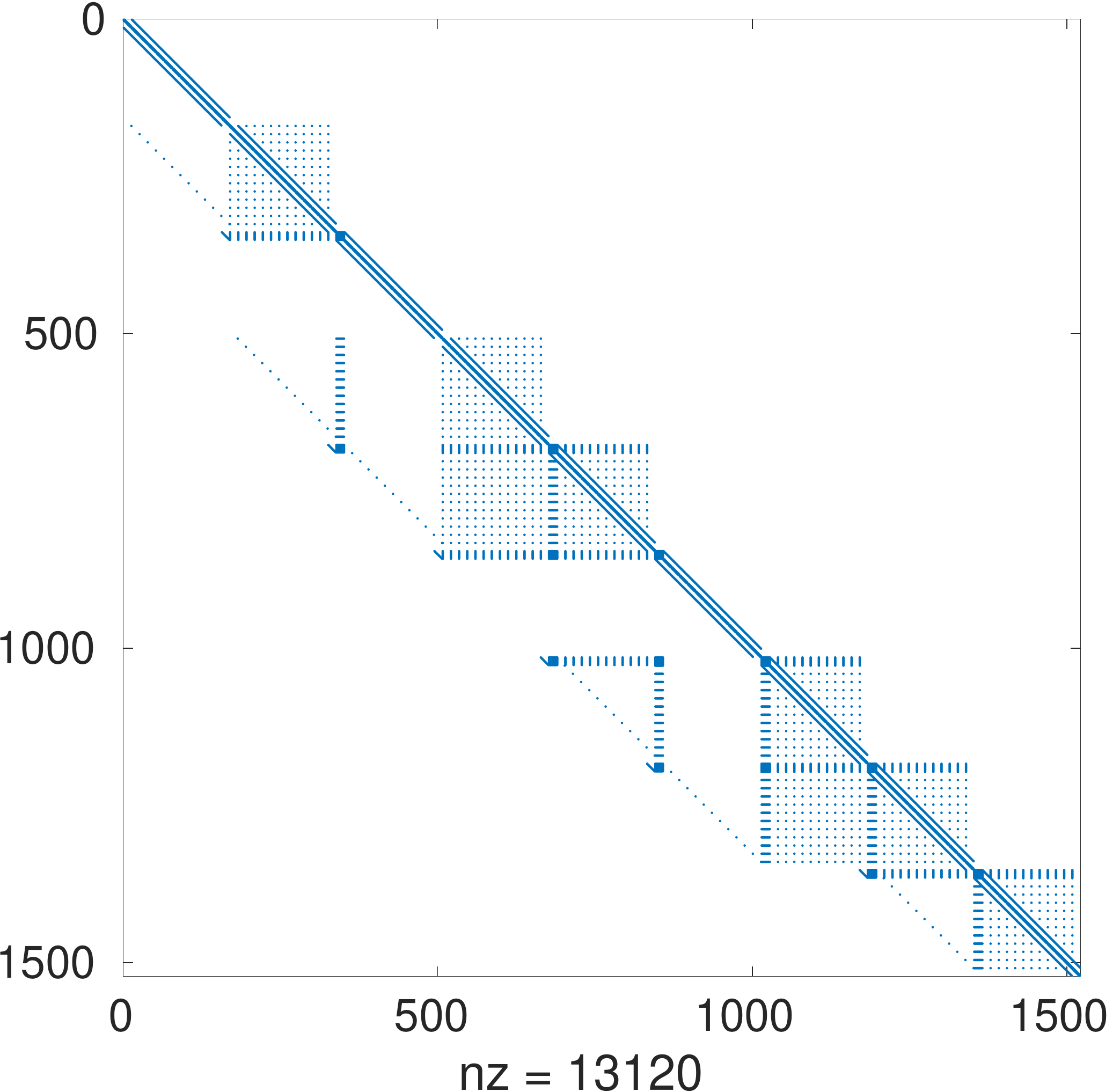}
\includegraphics[width=0.49\textwidth]{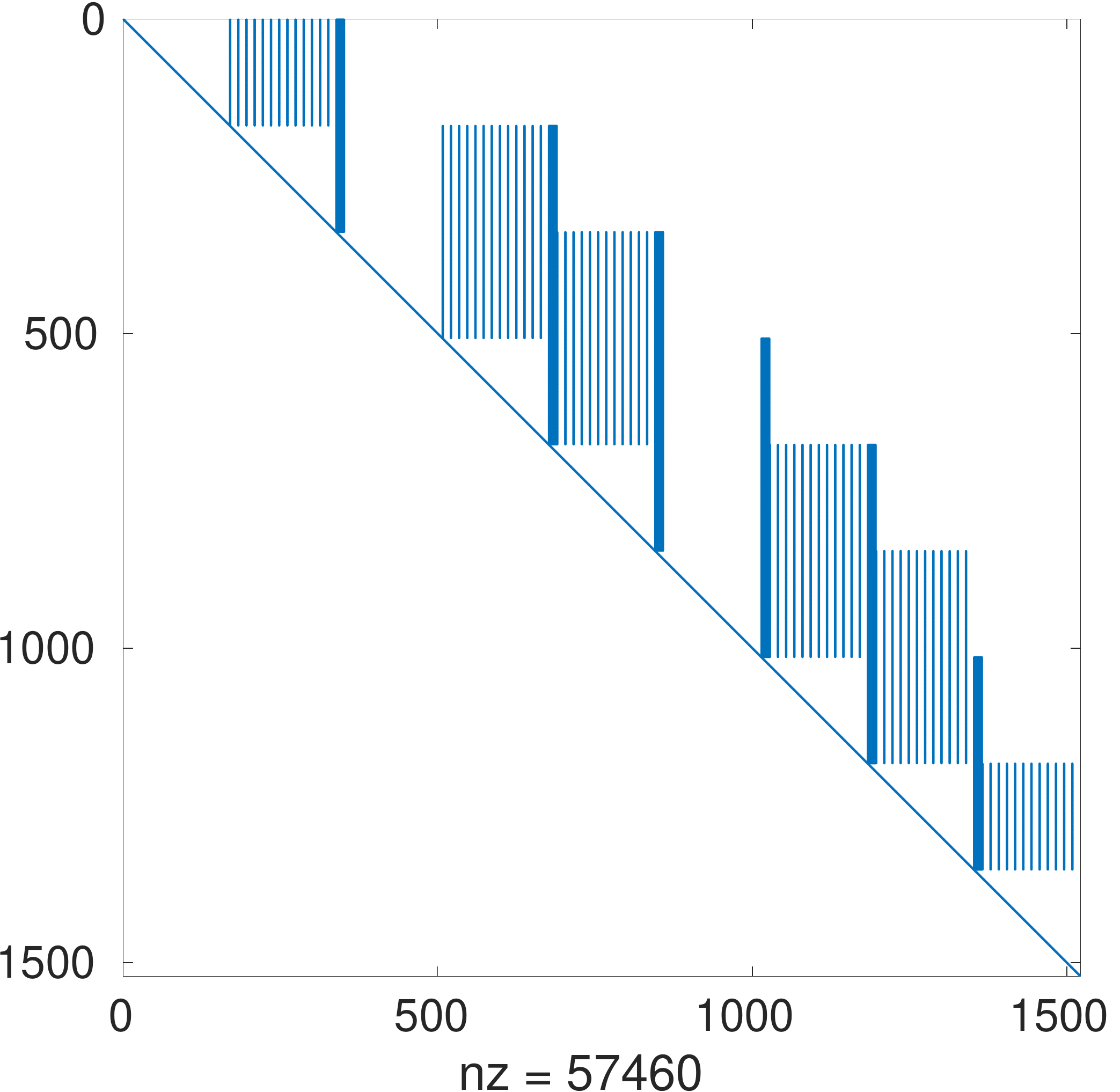}}
  \caption{Sparsity of the block $L$ and $U$ factors of the optimal
    Schwarz method for the $3\times 3$ subdomain decomposition from
    Figure \ref{OptimalSchwarzLUSweep} when using the D-sweep
    ordering.}
  \label{OSMLUfactorsD}
\end{figure}

It is currently not known how to obtain a general parallel nilpotent
Schwarz method for such general decompositions including cross
points. It was however discovered in \cn{GK1} that it is possible to
define transmission conditions at the algebraic level, including a
global communication component, such that the associated optimal
parallel Schwarz method converges in two (!) iterations, independently
of the number of subdomains, the decomposition and the partial
differential equation that is solved. Such a global
communication component is also present in the recent work
\cite{Claeys:2019:ADD,Claeys:2021:RTO} for time harmonic wave
propagation problems, which is based on earlier work of
\cn{claeys2019new}, where the multi-trace formulation was
interpreted as an optimized Schwarz method, including cross points.

We will use in what follows the two specific domain decompositions
shown in Figure \ref{1Dand2DDecompositionFig}, namely one dimensional,
sequential or strip decompositions, and two dimensional decompositions
including cross points. For sequential domain decompositions, we can
use Fourier analysis techniques to accurately study the influence of
the transmission conditions used on the convergence of the domain
decomposition iteration, whereas for two dimensional domain
decompositions, such results are not yet available.

\section{Two subdomain analysis}\label{2SubSec}

It is very instructive to understand Schwarz methods by domain truncation starting first with a
simple two subdomain decomposition, since then many detailed convergence properties of the Schwarz
methods can be obtained by direct, analytical computations. We consider therefore the strip
decomposition shown in Figure \ref{1Dand2DDecompositionFig} on the right but with only two
subdomains, $\Omega_1:=(0,X_1^r)\times(0,Y)$ and $\Omega_2:=(X_2^l,1)\times(0,Y)$. Throughout the
paper, by $y=O(x)$ we mean $C_1|x|\le |y|\le C_2|x|$ for some constant $C_1, C_2>0$ independent of
$x$, and we write the asymptotic $y\sim x$ if $\lim y/x =1$ and $x$ consists of the leading terms.

\subsection{Laplace type problems}

We study first the screened Laplace equation (sometimes also called
the Helmholtz equation with the ``good sign''),
\begin{equation}\label{ScreenedLaplaceEquation}
  (\eta-\Delta)u=f \quad \mbox{in $\Omega:=(0,1)\times(0,Y)$},
\end{equation}
with $\eta\ge 0$. As boundary conditions, we impose on the left and right
a Robin boundary condition,
\begin{equation}\label{bclr}
  {\cal B}^l(u):=\partial_nu+p^lu=g^l,\quad 
  {\cal B}^r(u):=\partial_nu+p^ru=g^r,
\end{equation}
where $\partial_n$ denotes the unit outer normal derivative
(\ie\ $-\partial_x$ on the left and $\partial_x$ on the right). On
top and bottom we impose either a Dirichlet or a Neumann condition,
\begin{equation}\label{bctb}
  \arraycolsep0.2em
  \begin{array}{rcll}
    {\cal B}^b(u)&=&g^b &\mbox{with either ${\cal B}^b(u):=u$ or
      ${\cal B}^b(u):=\partial_nu$}, \\
    {\cal B}^t(u)&=&g^t &\mbox{with either ${\cal B}^t(u):=u$ or
      ${\cal B}^t(u):=\partial_nu$}.
  \end{array}
\end{equation}
An alternating Schwarz method then starts with an initial guess
$u_2^0$ in subdomain $\Omega_2$ and performs for iteration index
$n=1,2,\ldots$ alternatingly solutions on the subdomains,
\begin{equation}\label{SMDT}
  \arraycolsep0.3em
  \begin{array}{rcllrcll}
  (\eta-\Delta)u_1^n&=&f  &\mbox{in $\Omega_1$},&
  (\eta-\Delta)u_2^n&=&f  &\mbox{in $\Omega_2$},\\
    {\cal B}_1^r(u_1^n)&=& {\cal B}_1^r(u_2^{n-1}) &\mbox{at $x=X_1^r$}, &
    {\cal B}_2^l(u_2^n)&=& {\cal B}_2^l(u_1^n) &\mbox{at $x=X_2^l$},\\
    {\cal B}^l(u_1^n)&=&g^l &\mbox{at $x=0$},  &
    {\cal B}^r(u_2^n)&=&g^r &\mbox{at $x=1$}, \\
    {\cal B}^b(u_1^n)&=&g^b &\mbox{at $y=0$},  &
    {\cal B}^b(u_2^n)&=&g^b &\mbox{at $y=0$}, \\
    {\cal B}^t(u_1^n)&=&g^t &\mbox{at $y=Y$},  &
    {\cal B}^t(u_2^n)&=&g^t &\mbox{at $y=Y$}.
  \end{array}
\end{equation}
Here the key ingredient are the transmission conditions ${\cal B}_1^r$
and ${\cal B}_2^l$, which in the classical Schwarz method are
Dirichlet, see \R{AlternatingSchwarzMethod}, but in Schwarz
methods based on domain truncation, \eg~optimized Schwarz methods,
they are of the form
\begin{equation}\label{TC}
  {\cal B}_1^r(u):=\partial_nu+{\cal S}_1^r(u),\quad 
  {\cal B}_2^l(u):=\partial_nu+{\cal S}_2^l(u).
\end{equation}
Here ${\cal S}_1^r$ and ${\cal S}_2^l$ can either be constants, ${\cal
    S}_1^r=p_1$, ${\cal S}_2^l=p_2$, $p_j\in \mathbb{R}$, which leads
  to Robin transmission conditions, or more general tangential
  operators acting along the interface, \eg\ ${\cal
    S}_1^r=p_1-q_1\partial_{yy}$ with $p_1,q_1\in\mathbb{R}$, which
  would be Ventcell transmission conditions, or one can also consider
  more general transmission conditions involving rational functions or
  PMLs, see \cn{gander2019class} and references therein.

In order to study the convergence of the Schwarz method \R{SMDT},
we introduce the error $e_j^n:=u-u_j^n$, $j=1,2$, which by linearity satisfies
the same iteration equations as the original algorithm, but with zero
data, \ie~
\begin{equation}\label{SMDTerror}
  \arraycolsep0.3em
  \begin{array}{rcllrcll}
  (\eta-\Delta)e_1^n&=&0  &\mbox{in $\Omega_1$},&
  (\eta-\Delta)e_2^n&=&0  &\mbox{in $\Omega_2$},\\
    {\cal B}_1^r(e_1^n)&=& {\cal B}_1^r(e_2^{n-1}) &\mbox{at $x=X_1^r$}, &
    {\cal B}_2^l(e_2^n)&=& {\cal B}_2^l(e_1^n) &\mbox{at $x=X_2^l$},\\
    {\cal B}^l(e_1^n)&=&0 &\mbox{at $x=0$},  &
    {\cal B}^r(e_2^n)&=&0 &\mbox{at $x=1$}, \\
    {\cal B}^b(e_1^n)&=&0 &\mbox{at $y=0$},  &
    {\cal B}^b(e_2^n)&=&0 &\mbox{at $y=0$}, \\
    {\cal B}^t(e_1^n)&=&0 &\mbox{at $y=Y$},  &
    {\cal B}^t(e_2^n)&=&0 &\mbox{at $y=Y$},
  \end{array}
\end{equation}
as one can easily verify by directly evaluating the expressions on the
left of the equal signs, \eg
$$
(\eta-\Delta)e_1^n=(\eta-\Delta)(u-u_1^n)=f-f=0.
$$
To obtain detailed information on the functioning of such Schwarz
methods, it is best to expand the errors in a Fourier series in the
$y$ direction\footnote{Note the different symbols $\hat{e}_j^n$ for
  the cosine, and $\hat{\epsilon}_j^n$ for the sine coefficients.},
\begin{equation}\label{CosSineExpansion}
  e_j^n=\sum_{\tilde{k}=0}^\infty\hat{e}_j^n(\tilde{k})\cos(\frac{\tilde{k}\pi}{Y}y)+
  \sum_{\tilde{k}=1}^\infty\hat{\epsilon}_j^n(\tilde{k})\sin(\frac{\tilde{k}\pi}{Y}y).
\end{equation}
Now if at the bottom and top we have Dirichlet conditions, all the
error coefficients of the cosine are zero, and if we have Neumann
conditions, all the error coefficients of the sine are zero. In either
case, inserting the Fourier series into the error equations
\R{SMDTerror}, we obtain by orthogonality of the sine and cosine
functions for each cosine error Fourier mode $\hat{e}_j^n$, $j=1,2$
(and analogously for the sine error Fourier mode $\hat{\epsilon}_j^n$)
the Schwarz iteration
\begin{equation}\label{SMDTerrorF}
  %\arraycolsep0em
  %\begin{array}{rcllrcll}
  \thickmuskip=0.6mu
  \medmuskip=0.2mu
  \begin{aligned}
  (\eta-\partial_{xx}+k^2)\hat{e}_1^n&=0& &\mbox{in $(0,X_1^r)$},&
  (\eta-\partial_{xx}+k^2)\hat{e}_2^n&=0& &\mbox{in $(X_2^l,1)$},\\
    \beta_1^r(\hat{e}_1^n)&=\beta_1^r(\hat{e}_2^{n-1})& &\mbox{at $x=X_1^r$}, &
    \beta_2^l(\hat{e}_2^n)&=\beta_2^l(\hat{e}_1^n)& &\mbox{at $x=X_2^l$},\\
    \beta^l(\hat{e}_1^n)&=0& &\mbox{at $x=0$},  &
    \beta^r(\hat{e}_2^n)&=0& &\mbox{at $x=1$},
  \end{aligned}
  %\end{array}
\end{equation}
where we defined the frequency variable $k:=\frac{\tilde{k}\pi}{Y}$,
and $\beta_1^r$, $\beta_2^l$, $\beta^l$, $\beta^r$ denote the Fourier
transforms of the boundary operators, also called their symbols,
\eg~$\beta_1^r=\partial_n+\sigma_1^r$ with $\sigma_1^r$ the symbol of
the tangential operator chosen in \R{TC}. If this operator was for
example ${\cal S}_1^r=p_1-q_1\partial_{yy}$, $p_1$, $q_1$ constants, then its
symbol would be $\sigma_1^r=p_1+q_1k^2$, from the two derivatives acting
on the cosine, which also leads to the sign change from minus to plus.

Solving the ordinary differential equations in the error iteration
\R{SMDTerrorF}, we obtain using the outer Robin boundary conditions
for each Fourier mode $k$ the solutions\footnote{These computations can easily be performed in Maple, see Appendix B.} (with $\underline{x}:=1-x$)
\begin{equation}\label{ScreenedLaplaceSols}
  \begin{aligned}
    \hat{e}_1^n(x,k)&=A_1^n(k)\left({\sqrt{k^2+\eta}\cosh(\sqrt{k^2+\eta}x)
      + p^l\sinh(\sqrt{k^2 +\eta}x)}\right),\\
    \hat{e}_2^n(x,k)&=A_2^n(k)\left({\sqrt{k^2+\eta}\cosh(\sqrt{k^2+\eta}\underline{x}) + p^r\sinh(\sqrt{k^2 +\eta}\underline{x})}\right).
  \end{aligned}  
\end{equation}
To determine the two remaining constants $A_1^n(k)$ and $A_2^n(k)$ we
insert the solutions into the transmission conditions in
\R{SMDTerrorF}, which leads to
\begin{equation}
  A_1^n=\rho_1(k,\eta,p^l,p^r,\sigma_1^r)A_2^{n-1},\quad
    A_2^n=\rho_2(k,\eta,p^l,p^r,\sigma_2^l)A_1^{n},
\end{equation}
with the two quantities
\begin{equation}\label{rho12}
  \arraycolsep0.1em
  \begin{array}{rcl}
    \rho_1 &=& \frac{(k^2 - \sigma_1^rp^r + \eta)
      \sinh(\sqrt{k^2 + \eta}(X_1^r - 1))
        - \cosh(\sqrt{k^2 + \eta}(X_1^r - 1))
        \sqrt{k^2 + \eta}(p^r - \sigma_1^r)}
           {(k^2 + \sigma_1^rp^l + \eta)\sinh(\sqrt{k^2 + \eta}X_1^r)
             + \cosh(\sqrt{k^2 + \eta}X_1^r)\sqrt{k^2 + \eta}(p^l + \sigma_1^r)},\\
    \rho_2&=&\frac{(k^2 - \sigma_2^lp^l + \eta)
      \sinh(\sqrt{k^2 + \eta}X_2^l)
      + \cosh(\sqrt{k^2 + \eta}X_2^l)
      \sqrt{k^2 + \eta}(p^l - \sigma_2^l)}
        {(k^2 + \sigma_2^lp^r + \eta)\sinh(\sqrt{k^2 + \eta}(X_2^l - 1))
          - \cosh(\sqrt{k^2 + \eta}(X_2^l - 1))\sqrt{k^2 + \eta}(p^r + \sigma_2^l)}.
  \end{array}  
\end{equation}
Their product represents the convergence factor of the Schwarz method,
\begin{equation}\label{RhoOSM}
  \rho(k,\eta,p^l,p^r,\sigma_1^r,\sigma_2^l):=\rho_1(k,\eta,p^l,p^r,\sigma_1^r)\rho_2(k,\eta,p^l,p^r,\sigma_2^l),
\end{equation}
\ie\ it determines by which coefficient the corresponding error
Fourier mode $k$ is multiplied over each alternating Schwarz
iteration.

Since the expression for the convergence factor $\rho$ looks rather
complicated, it is instructive to look at the special case first where
the domain, and thus the subdomains, are unbounded on the left and
right. This can be obtained from the result above by introducing for
the outer boundary conditions a transparent one, \ie\ choosing for the
Robin parameters the symbol of the DtN
operator\footnote{\label{footnoteDtN} The symbol of the DtN operator
  for the transparent boundary condition can be obtained by solving
  \eg~$(\eta-\partial_{xx}+k^2)\hat{e}_1=0$ on the outer domain
  $(-\infty,0)$ with Dirichlet data $\hat{e}_1(0)=\hat{g}$, which
  gives $\hat{e}_1=\hat{g}e^{\sqrt{k^2+\eta}x}$, because solutions
  must remain bounded as $x\to-\infty$.  Since
  $\partial_x\hat{e}_1=\sqrt{k^2+\eta}\hat{e}_1$, the outer solution
  $\hat{e}_1$ satisfies for any $x\in(-\infty,0]$ the identity
$-\partial_x\hat{e}_1+\sqrt{k^2+\eta}\hat{e}_1=0$. One could therefore
also solve this outer problem on any bounded domain, \eg~$(a,0)$
imposing at $x=a$ the transparent boundary condition
$\partial_n\hat{e}_1+\sqrt{k^2+\eta}\hat{e}_1=0$, since the outward
normal $\partial_n=-\partial_x$, and get the same solution.  Note that
$\sqrt{k^2+\eta}$ is the symbol of the DtN map $\hat{g}\mapsto
\partial_n \hat{e}_1$ with $\partial_n=\partial_x$, \ie\ the operator
that takes the Dirichlet data $\hat{g}$, solves the problem on the
domain, and then computes the Neumann data for the outward normal
derivative.}  $p^r:=\sqrt{k^2+\eta}$, $p^l:=\sqrt{k^2+\eta}$. This
  leads after simplification to the convergence factor
\begin{equation}\label{RhoOSMUnbounded}
  \rho(k,\eta,\sigma_1^r,\sigma_2^l)=
  \frac{\sqrt{k^2+\eta}-\sigma_1^r}{\sqrt{k^2+\eta}+\sigma_1^r}
  \frac{\sqrt{k^2+\eta}-\sigma_2^l}{\sqrt{k^2+\eta}+\sigma_2^l}
  e^{-2(X_1^r-X_2^l)\sqrt{k^2+\eta}}.
\end{equation}
This convergence factor shows us a very important property of these
Schwarz methods: if we choose
$\sigma_1^r=\sigma_2^l:=\sqrt{k^2+\eta}$, the tangential symbol of the
transparent boundary condition, the convergence factor vanishes
identically for all Fourier frequencies $k$, \ie\ after two
consecutive subdomain solves, one on the left and one on the right (or
vice versa), the error in each Fourier mode is zero, \ie\ we have the
exact solution on the subdomain after two alternating subdomain
solves! This is called an optimal Schwarz method, see footnote
\ref{footnoteoptimalSchwarz}. If we thus used a double sweep, first
solving on the left subdomain and then on the right, followed by
another solve on the left (or a double sweep in the other direction),
we would have the solution on both subdomains, which shows for two
subdomains the double sweep result illustrated in Figure
\ref{PoissonOptimalSchwarzFig} for three subdomains. This result holds
in general for many subdomains in the strip decomposition case; it was
first proved in \cn[Result 3.1]{Nataf93} for an advection diffusion
problem. If we use the parallel Schwarz method for two subdomains,
then we need two iterations in order to have each subdomain solved one
after the other, so the optimal parallel Schwarz method for two
subdomains will converge in two iterations, a result that generalizes
to convergence in $J$ iterations for $J$ subdomains in the strip
decomposition case, first proved in \cn[Proposition 2.4]{NRS94}.

In the bounded domain case, we can still obtain these same results by
chosing in the transmission conditions the tangential symbols of the
transparent boundary conditions for the bounded subdomains, which is
equivalent to choosing $\sigma_1^r$ and $\sigma_2^l$ such that
$\rho_1$ and $\rho_2$ in \R{rho12} vanish, which leads to
\begin{equation}\label{DtNbdd}
  \begin{array}{rcl}
     \sigma_1^r&=&\frac{\sqrt{k^2+\eta}\left(\tanh(\sqrt{k^2+\eta}(X_1^r-1))\sqrt{k^2+\eta} - p^r\right)}{\tanh(\sqrt{k^2+\eta}(X_1^r - 1))p^r-\sqrt{k^2+\eta}},\\
     \sigma_2^l&=&\frac{\sqrt{k^2+\eta}\left(\tanh(\sqrt{k^2+\eta}X_2^l)\sqrt{k^2+\eta} + p^l\right)}{\tanh(\sqrt{k^2+\eta}X_2^l)p^l+\sqrt{k^2+\eta}}.
  \end{array}
\end{equation}
As in the unbounded domain case, these values correspond to the
symbols of the associated DtN operators on the bounded domain, as one
can verify by a direct computation using the solutions
\R{ScreenedLaplaceSols} on their respective domains, as done for
the unbounded domain in footnote \ref{footnoteDtN}. We see that in the
bounded domain case, the symbols of the DtN maps in \R{DtNbdd},
and hence the DtN maps also, depend on the outer boundary conditions,
since both the domain parameters ($X_1^r$ and $X_2^l$) and the Robin
parameters in the outer boundary conditions ($p^l$ and $p^r$) appear
in them. However, for large frequencies $k$, we have
$$
\sigma_1^r\sim\sqrt{k^2+\eta},\quad  \sigma_2^l\sim\sqrt{k^2+\eta},
$$
since $\tanh z\to\pm1$ as $z\to\pm\infty$, and thus only low frequencies see the difference
from the bounded to the unbounded case in the screened Laplace problem.
  
Choosing Robin transmission conditions, ${\cal
  S}_1^r:=p_1^r=\sigma_1^r$ and ${\cal S}_2^l:=p_2^l=\sigma_2^l$ with
$p_1^r,p_2^l\in\mathbb{R}$ in \R{TC}, and taking the limit in the
convergence factor \R{RhoOSMUnbounded} as the Robin parameters
$p_1^r$ and $p_2^l$ go to infinity, we find
\begin{equation}\label{RhoAltSUnbounded}
  \rho(k,\eta,)=e^{-2(X_1^r-X_2^l)\sqrt{k^2+\eta}},
\end{equation}
which is the convergence factor of the classical alternating Schwarz
method on the unbounded domain, since the transmission conditions
\R{TC} become Dirichlet transmission conditions in this limit of the
Robin transmission conditions.  We see now explicitly the overlap
$L:=X_1^r-X_2^l$ appearing in the exponential function. The classical
Schwarz method for the screened Laplace problem therefore converges
for all Fourier modes $k$, provided there is overlap, and $\eta>0$. If
$\eta=0$, \ie\ we consider the Laplace problem, then the Fourier mode
$k=0$ does not contract on the unbounded domain. Recall however that
$k=0$ is only present in the cosine expansion for Neumann boundary
conditions on top and bottom, it is the constant mode, and the Schwarz
method on the unbounded domain does indeed not contract for this
mode. For Dirichlet boundary conditions on top and bottom, \ie\ 
considering the sine series in \eqref{CosSineExpansion}, the smallest
Fourier frequency is $k=\frac{\pi}{Y}>0$ so the Schwarz method is
contracting, even with $\eta=0$. Note also that the contraction is
faster for large Fourier frequencies $k$ than for small ones due to
the exponential function, and hence the Schwarz method is a smoother
for the screened Laplace equation. A comparison of the classical
Schwarz convergence factor \R{RhoAltSUnbounded} with the convergence
factor for the optimized Schwarz method with Robin transmission
conditions in \R{RhoOSMUnbounded} shows that the latter contains the
former, but in addition also the two fractions in front which are also
smaller than one for suitable choices of the Robin parameters. The
optimized Schwarz method therefore always converges faster than the
classical Schwarz method, and furthermore can also converge without
overlap, which was the original reason for Lions to propose Robin
transmission conditions.

To see how the classical Schwarz method contracts on a bounded domain,
we show in Figure \ref{SchwarzScreenedLaplaceRhos} (left) the
different cases for the outer boundary conditions ($p^l=p^r=5$ in the
Robin case), by plotting \R{RhoOSM} for when the parameters in the
Robin transmission conditions go to infinity, and the model parameter
$\eta=1$, for a small overlap, $L=X_1^r-X_2^l=0.51-0.49=0.02$.
\begin{figure}
  \centering
  \includegraphics[width=0.48\textwidth]{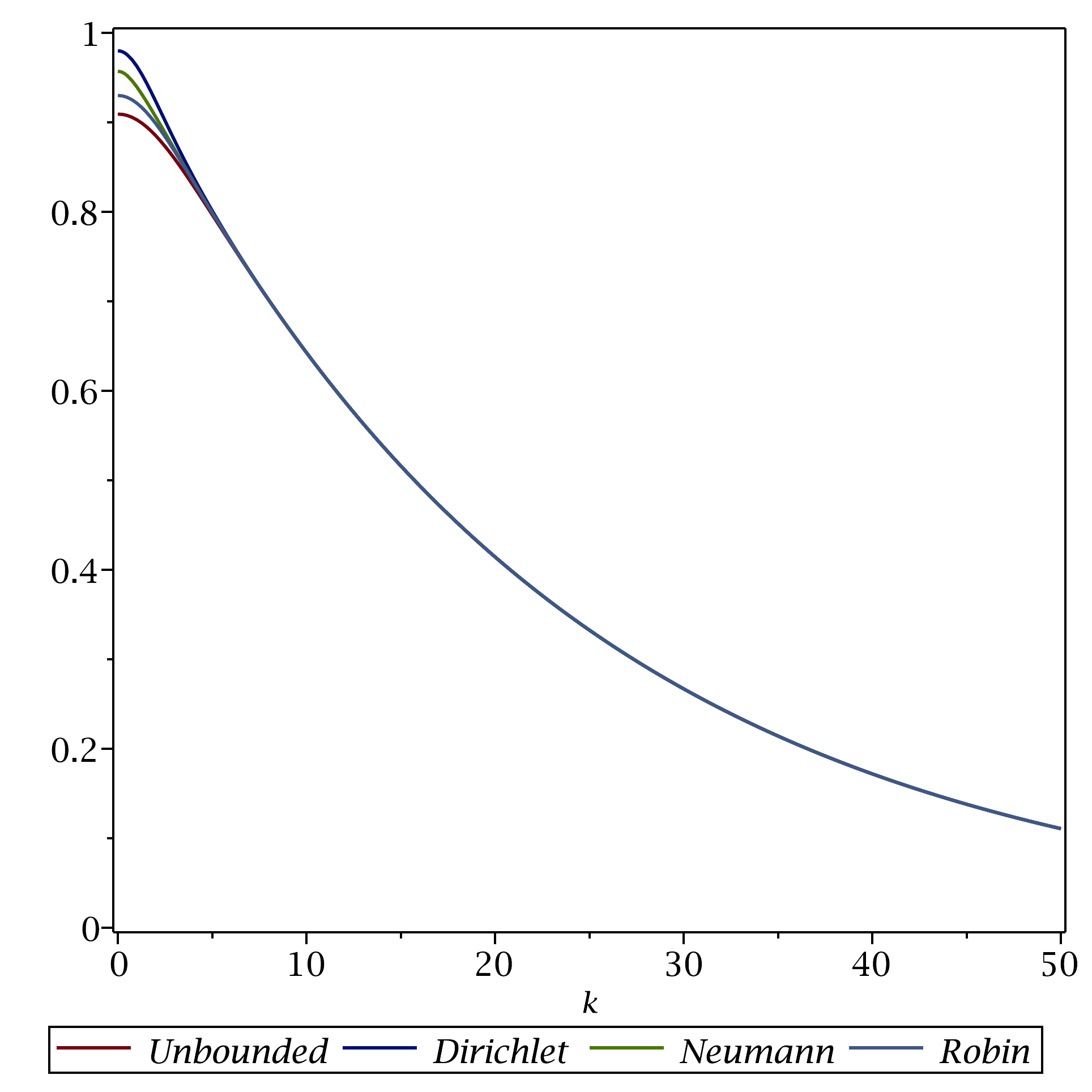}
  \includegraphics[width=0.48\textwidth]{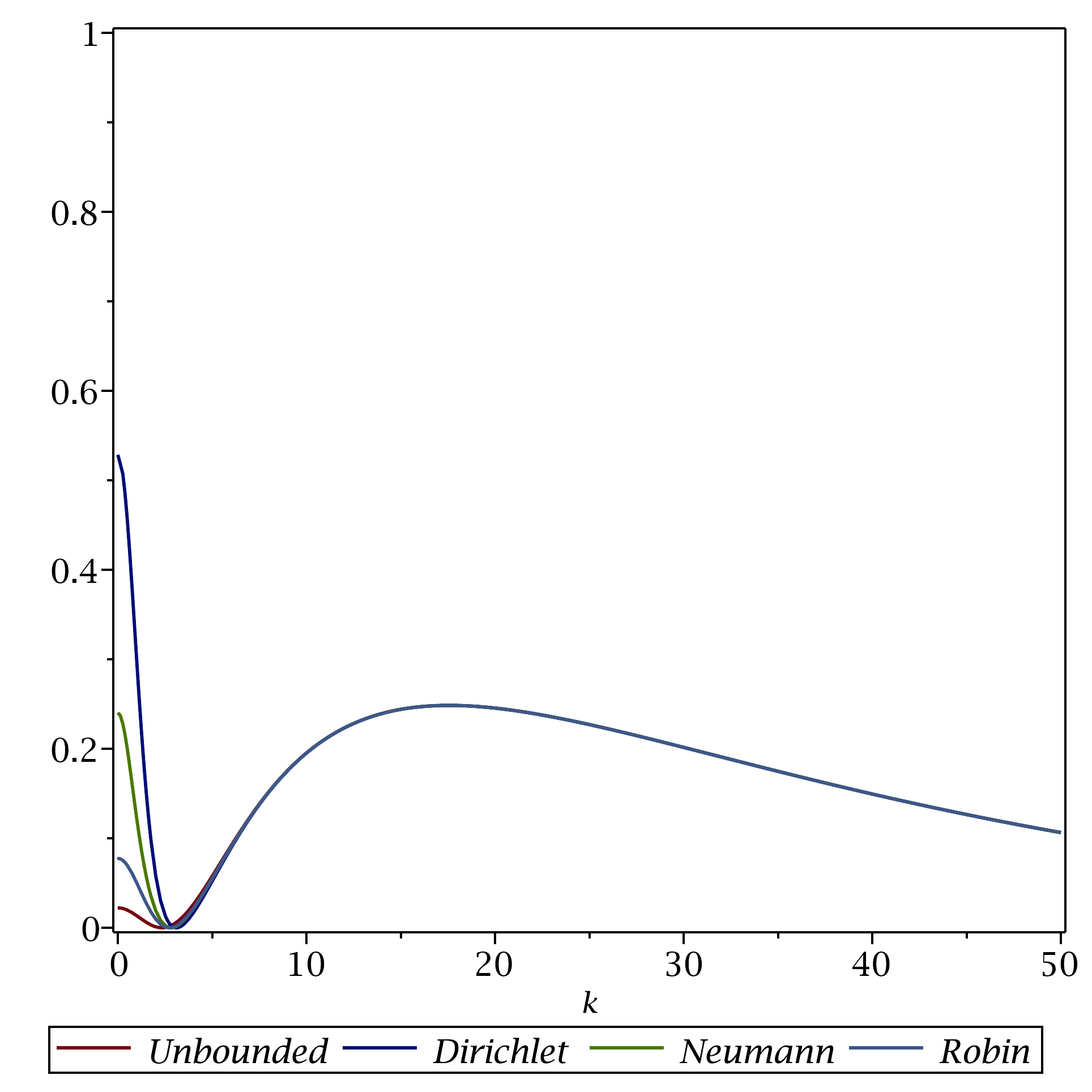}
  \caption{Left: classical Schwarz convergence factor for
    $\eta-\Delta$ and different outer boundary conditions on the left
    and right. Right: corresponding Schwarz convergence
    factor with Robin transmission conditions.}
  \label{SchwarzScreenedLaplaceRhos}
\end{figure}
This shows that the different outer boundary conditions only influence the convergence of the lowest
frequencies ($k$ small) in the error, for larger frequencies there is no influence. This is due to
the diffusive nature of the screened Laplace equation: high freqency components are damped rapidly
by the (screened) Laplace operator, and thus they don't see the different outer boundary conditions.
On the right in Figure \ref{SchwarzScreenedLaplaceRhos} we show the corresponding convergence
factors with Robin transmission conditions, chosen as $p_1^r=p_2^l=3$. This indicates that there is
an optimal choice which leads to the fastest possible convergence, which in our example is achieved
for the Neumann outer boundary condition with the choice $p_1^r=p_2^l\approx 3$, since the maximum
of the convergence factor is minimized by equioscillation, \ie\ the convergence factor at $k=0$
equals the convergence factor at the interior maximum at around $k=18$. For different outer boundary
conditions, there would be a better value of $p_1^r=p_2^l$ that makes \eg~ the blue curve
for Dirichlet outer boundary conditions equioscillating. This optimization process led to the name
optimized Schwarz methods.

Let us compute the optimized parameter values: for the simplified
situation where we choose the same Robin parameter in the transmission
condition, $p_1^r=p_2^l=p$, we need to solve the min-max problem
\begin{equation}
  \min_p\max_{k\in[k_{\min},\infty)}|\rho(k,\eta,p^l,p^r,p,p)|
\end{equation}
with the convergence factor $\rho$ from \R{RhoOSM}. The solution
is given by equioscillation, as indicated in Figure
\ref{SchwarzScreenedLaplaceRhos}, \ie\ we have to solve
\begin{equation}\label{equioscillationEq}
  \rho(k_{\min},\eta,p^l,p^r,p^*,p^*)=\rho(\bar{k},\eta,p^l,p^r,p^*,p^*),
\end{equation}
for the optimal Robin parameter $p^*$, where $\bar{k}$ denotes the
location of the interior maximum visible in Figure
\ref{SchwarzScreenedLaplaceRhos} (right). We observe numerically that
$\bar{k}\sim C_k\frac{1}{L^{2/3}}$, and $p^*\sim
C_p\frac{1}{L^{1/3}}$, see also \cn{GanderOSM} for a proof of this,
and \cn{bennequin2009homographic} for a more comprehensive analysis of such
min-max problems. Inserting this Ansatz into the system of
equations formed by \R{equioscillationEq} and the derivative
condition for the local maximum
\begin{equation}
  \partial_k\rho(\bar{k},\eta,p^l,p^r,p^*,p^*)=0,
\end{equation}
we find by asymptotic expansion\footnote{For the expansions involving
  $\bar{k}$, it is much easier to expand the convergence factor of the
  unbounded domain analysis \R{RhoOSMUnbounded}, which gives the
  same result as the expansion of the convergence factor
  \R{RhoOSM} with \R{rho12} from the bounded domain analysis,
  since they behave the same for large $k$, see Figure
  \ref{SchwarzScreenedLaplaceRhos} on the right. This approach was
  discovered in \cn{gander2014optimized} under the name asymptotic
  approximation of the convergence factor, see also
  \cn{gander2017optimized}, \cn{chen2021optimized} where this new
  technique was used.}  for small overlap $L$
\begin{eqnarray}
  \partial_k\rho(\bar{k},\eta,p^l,p^r,p^*,p^*)&=&
  \frac{2(2C_p-C_k^2)}{C_k^2}L+O(L^{4/3}),\label{de}\\
  \rho(k_{\min},\eta,p^l,p^r,p^*,p^*)&=&1-\frac{C}{C_p}L^{1/3}+O(L^{2/3}),\label{lfe}\\
  \rho(\bar{k},\eta,p^l,p^r,p^*,p^*)&=&1-2\frac{C_k^2 + 2C_p}{C_k}L^{1/3}+O(L^{2/3}),\label{hfe}
\end{eqnarray}
where all the information about the geometry\footnote{See also
  \cn{gander2016optimized} where it is shown that variable
  coefficients also influence essentially the low frequency
  behavior.} is in the constant stemming from the value of the
convergence factor at the lowest frequency \R{lfe},
\begin{equation}\label{Cgeom}
  C=\frac{2s((k_{\min}^2 + p_lp_r + \eta)s_s + s c_s(p_l +
    p_r))}{(((s_sp_r+sc_s)c_{sx}-(c_sp_r + s_ss)s_{sx})(s_{sx}p_l + c_{sx}s))}.
\end{equation}
Here we let $s:=\sqrt{k^2+\eta}$, $c_s:=\cosh(s)$, $s_s:=\sinh(s)$,
$c_{sx}:=\cosh(sX_2^l)$, $s_{sx}:=\sinh(sX_2^l)$ to shorten the
formula. Setting the leading order term of the derivative in
\R{de} to zero, and the other two leading terms from \R{lfe}
and \R{hfe} to be equal for equioscillation leads to the system
for the constants $C_k$ and $C_p$,
$$
\frac{2(2C_p-C_k^2)}{C_k^2}=0,\quad \frac{C}{C_p}=2\frac{C_k^2 +
  2C_p}{C_k},
$$
whose solution is
$$
  C_k=\left(\frac{C}{2}\right)^{1/3}, \quad
  C_p=\frac{1}{2}\left(\frac{C}{2}\right)^{2/3}.
$$
The best choice of the Robin parameter is therefore given for small
overlap $L$ by
\begin{equation}\label{pstar}
  p^*=\frac{1}{2}\left(\frac{C}{2}\right)^{2/3}L^{-1/3},
\end{equation}
with the geometry constant $C$ from \R{Cgeom}, which leads to a
convergence factor of the associated optimized Schwarz method
\begin{equation}\label{OptRho}
  \rho^*\sim 1-2\left(\frac{2}{C}\right)^{2/3}L^{1/3}.
\end{equation}

We show in Figure \ref{SchwarzScreenedLaplaceOptRhos} on the left the
contraction factor for this optimized Schwarz method for the different
outer boundary conditions, and also for the unbounded domain case, for
the same parameter choice as in Figure
\ref{SchwarzScreenedLaplaceRhos}.
\begin{figure}
  \centering
  \includegraphics[width=0.48\textwidth]{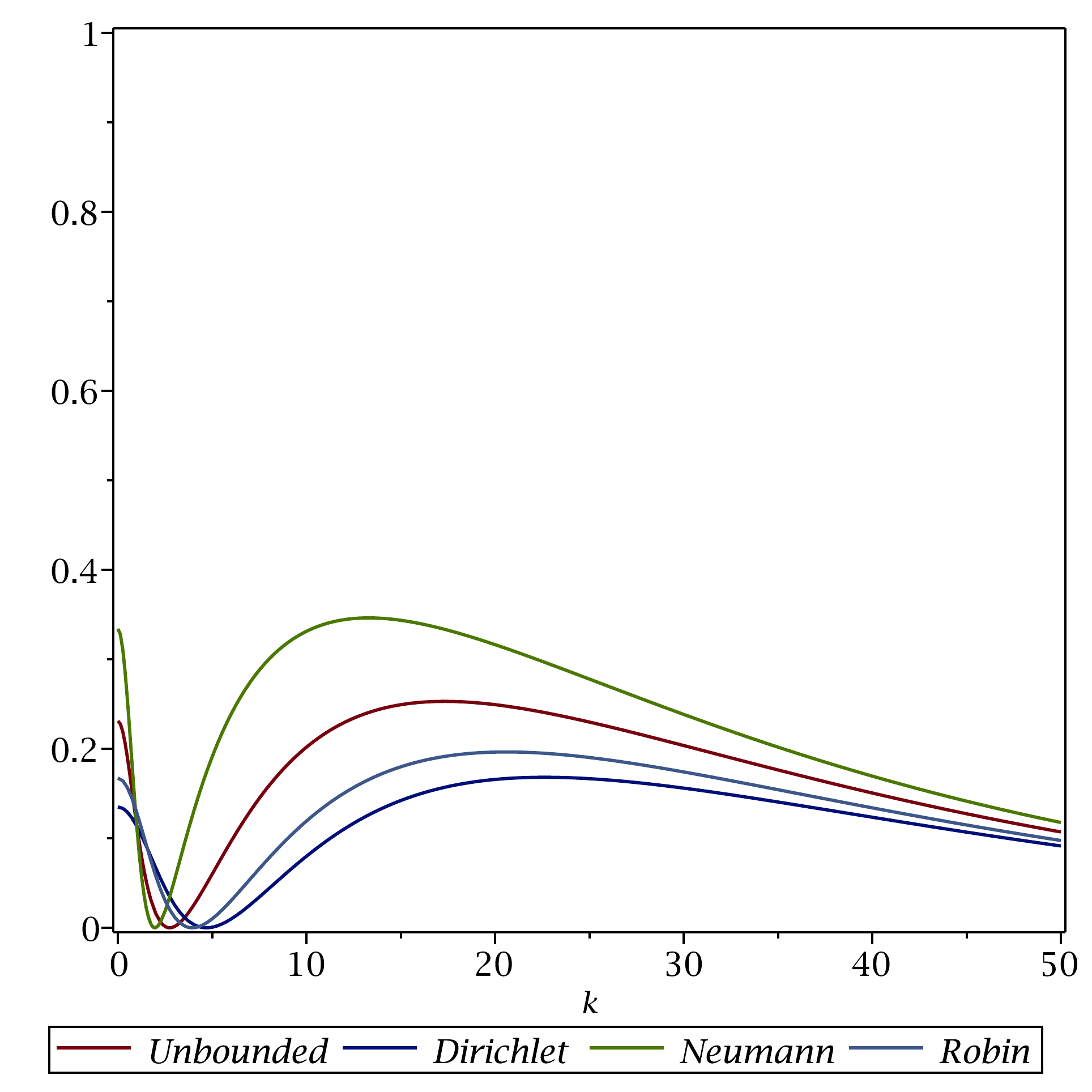}
  \includegraphics[width=0.48\textwidth]{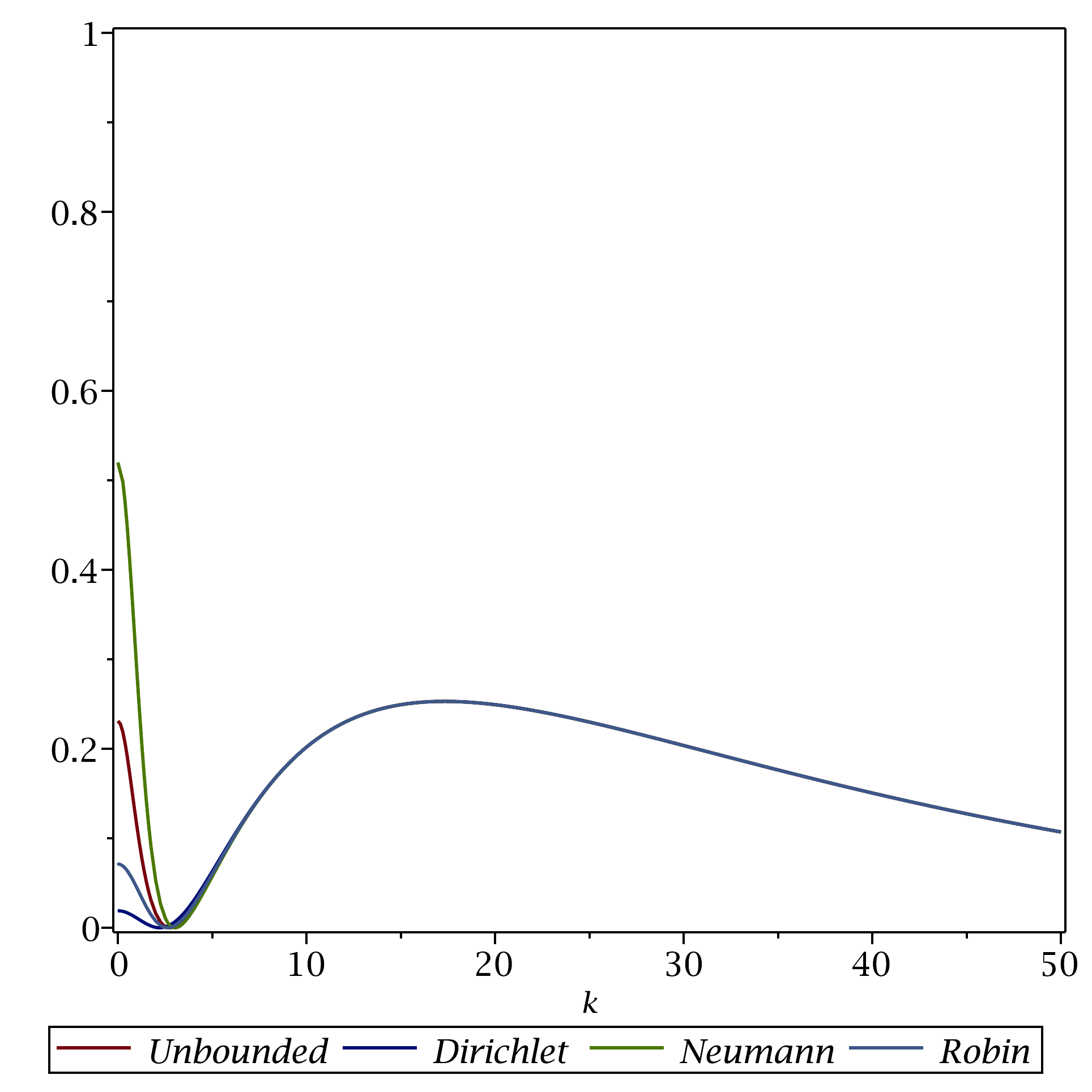}
  \caption{Left: optimized Schwarz convergence factor for
    $\eta-\Delta$ and different outer boundary conditions. Right:
    corresponding convergence factors when using the optimized
    parameter from the simpler unbounded domain analysis.}
  \label{SchwarzScreenedLaplaceOptRhos}
\end{figure}
We see that the contraction factor equioscillates using the asymptotic
formulas already for an overlap $L=0.02$, so the formulas are very
useful in practice. We see also that the unbounded domain analysis
which does not take into account the geometry leads to an optimized
convergence factor in between the other ones. In practice one
can also use the corresponding simpler formula from \cn[equation
  (4.21)]{GanderOSM}, namely
\begin{equation}\label{pstarunbounded}
  p^*=\frac{1}{2}(4(k_{\min}^2+\eta))^{1/3}L^{-1/3}
\end{equation}
in the case of the screened Laplace equation, which only deteriorates
the performance a little, see Figure
\ref{SchwarzScreenedLaplaceOptRhos} on the right.

Another, easy to use choice is to use a low frequency Taylor expansion
about $k=0$ of the optimal symbol of the DtN operator
\R{DtNbdd}, since the classical Schwarz method is not working
well for low frequencies, as we have seen in Figure
\ref{SchwarzScreenedLaplaceRhos} on the left. Expanding
the optimal DtN symbols in \R{DtNbdd}, we obtain 
\begin{equation}\label{pT}
  \begin{array}{rcl}
  p_1^r&=&\frac{\sqrt{\eta}\left(\tanh(\sqrt{\eta}(X_1^r-1))\sqrt{\eta} - p^r\right)}{\tanh(\sqrt{\eta}(X_1^r - 1))p^r-\sqrt{\eta}}+O(k^2),\\
  p_2^l&=&\frac{\sqrt{\eta}\left(\tanh(\sqrt{\eta}X_2^l)\sqrt{\eta} + p^l\right)}{\tanh(\sqrt{\eta}X_2^l)p^l+\sqrt{\eta}}+O(k^2).
  \end{array}
\end{equation}
We see that the Taylor parameters also take into account the
geometry. Using just the zeroth order term leads to the Taylor
convergence factors shown in Figure
\ref{SchwarzScreenedLaplaceTaylorRhos} on the left.
\begin{figure}
  \centering
  \includegraphics[width=0.48\textwidth]{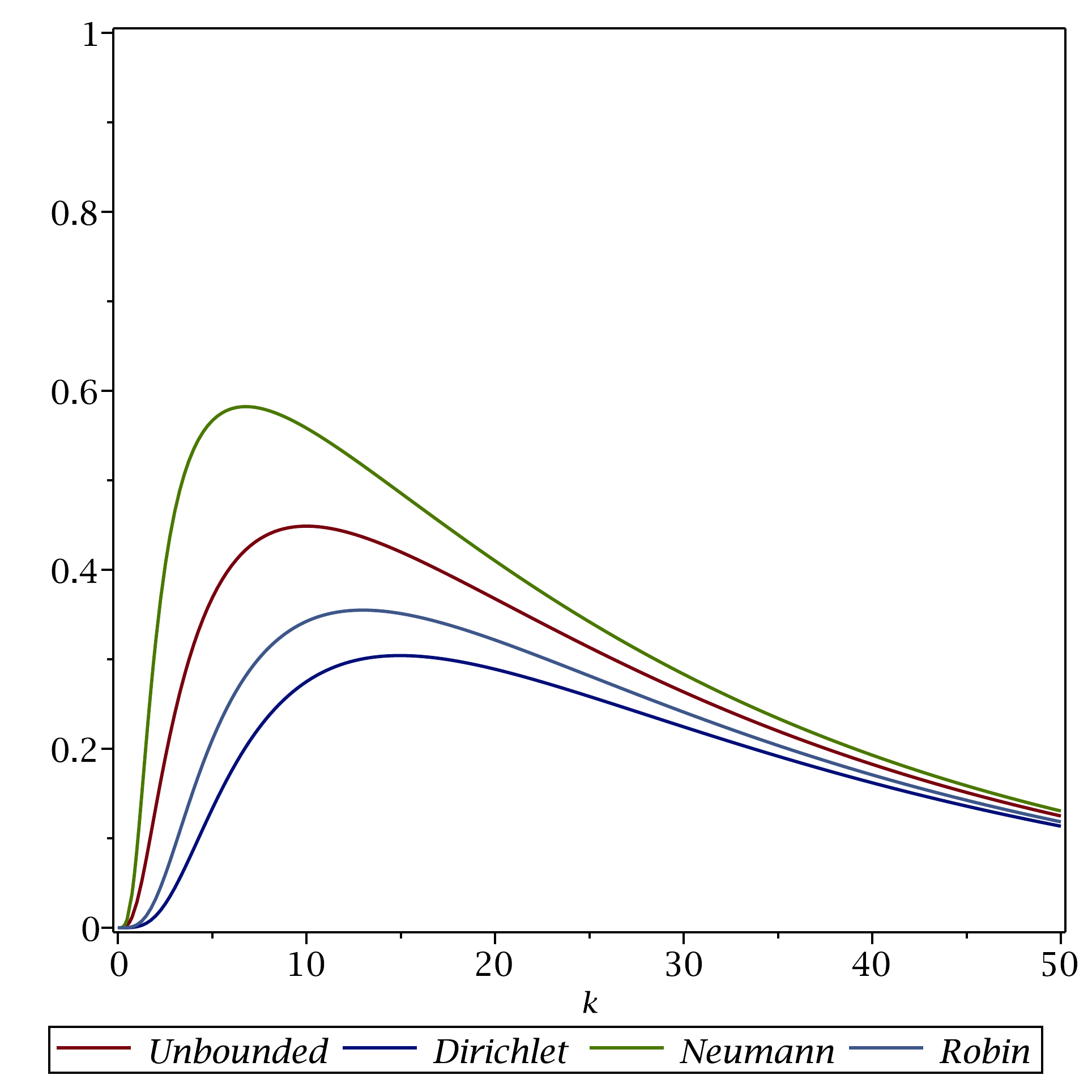}
  \includegraphics[width=0.48\textwidth]{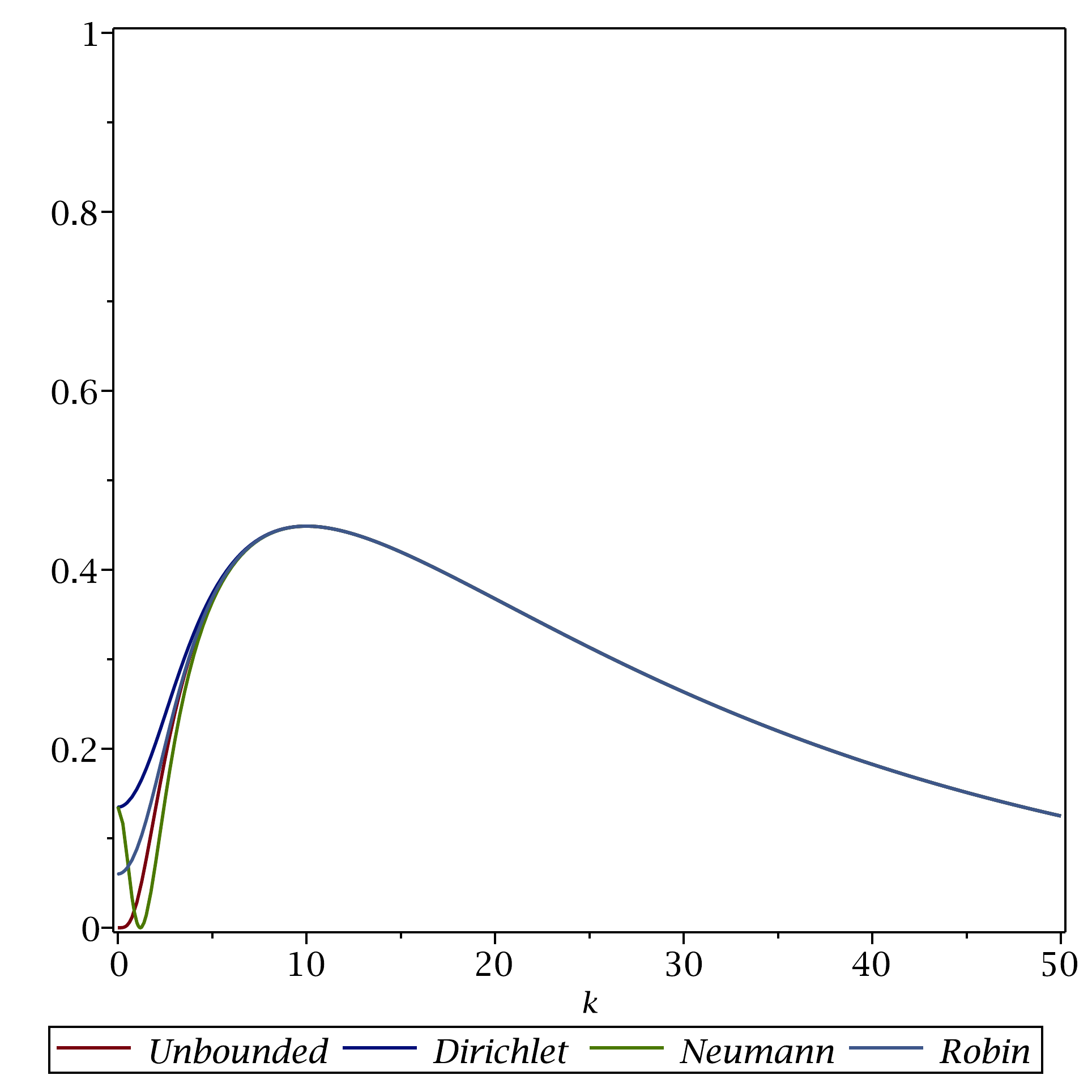}
  \caption{Left: convergence factor for Schwarz with Taylor based
    Robin transmission conditions for $\eta-\Delta$ and different
    outer boundary conditions on the left Right: corresponding Taylor
    Schwarz convergence factor using the zeroth order Taylor
    transmission conditions from the unbounded domain analysis,
    $p_1^r=p_2^l=\sqrt{\eta}$.}
  \label{SchwarzScreenedLaplaceTaylorRhos}
\end{figure}
Their maximum $\bar{k}$ is affected by the outer boundary conditions,
and can be explicitly computed when the overlap $L:=X_2^l-X_1^r$ becomes
small\footnote{We give in Appendix B the Maple commands
  to show how such technical computations can be performed automatically.},
\begin{equation}\label{TaylorRho}
   \bar{k}\sim \frac{\sqrt{p_1^r+p_2^l}}{\sqrt{L}},\quad
   \rho\sim1-4\sqrt{p_1^r+p_2^l}\sqrt{L},
\end{equation}
where $p_1^r$ and $p_2^l$ are from \eqref{pT} without the order term,
and setting $X_1^r:=X_2^l$ in $p_1^r$ due to the expansion for $L$
small.

On the right, we show the results when using the first term in the
Robin transmission conditions from the Taylor expansion of the optimal
symbols from the unbounded domain analysis,
$\sqrt{k^2+\eta}=\sqrt{\eta}+\frac{1}{2\sqrt{\eta}}k^2+O(k^4)$,
\ie\ $p_1^r=p_2^l=\sqrt{\eta}$. We see that this choice works even
better compared to the precise Taylor coefficients adapted to the
outer boundary conditions in the Neumann case, but slightly worse for
the Dirichlet and Robin case. Nevertheless, convergence with the
simple Taylor conditions is not as good as with the optimized ones in
Figure \ref{SchwarzScreenedLaplaceOptRhos}, and this becomes more
pronounced when the overlap $L$ becomes small as shown in
\eqref{TaylorRho} compared to \eqref{OptRho}. Higher order
transmission conditions can also be used and analyzed, see
\cn{GanderOSM}, and give even better performance with weaker
asymptotic dependence on the overlap.

\subsection{Helmholtz problems}

We now investigate what changes if the equation is the Helmholtz equation,
\begin{equation}\label{HelmholtzEquation}
  (\Delta+\omega^2)u=f \quad \mbox{in $\Omega:=(0,1)\times(0,Y)$},
\end{equation}
with $\omega\in\mathbb{R}$. As boundary conditions, we impose on the left and
right again the Robin boundary condition \R{bclr}, and on top and
bottom first a Dirichlet or a Neumann condition as in \R{bctb}.  The
alternating Schwarz method remains as in \R{SMDT}, but with the
differential operator $(\eta-\Delta)$ replaced by the Helmholtz
operator $(\Delta+\omega^2)$. The corresponding error equations for
the Fourier coefficients are
\begin{equation}\label{SMDTerrorFH}
  \thickmuskip=0.6mu
  \medmuskip=0.2mu
  \thinmuskip=0.1mu
  \nulldelimiterspace=-0.1pt
  \scriptspace=0pt
  % \begin{array}{rcllrcll}
  \begin{aligned}
  (\partial_{xx}+\omega^2-k^2)\hat{e}_1^n&=0& &\mbox{in $(0,X_1^r)$},&
  (\partial_{xx}+\omega^2-k^2)\hat{e}_2^n&=0& &\mbox{in $(X_2^l,1)$},\\
    \beta_1^r(\hat{e}_1^n)&=\beta_1^r(\hat{e}_2^{n-1})& &\mbox{at $x=X_1^r$}, &
    \beta_2^l(\hat{e}_2^n)&=\beta_2^l(\hat{e}_1^n)& &\mbox{at $x=X_2^l$},\\
    \beta^l(\hat{e}_1^n)&=0& &\mbox{at $x=0$},  &
    \beta^r(\hat{e}_2^n)&=0& &\mbox{at $x=1$}.
  \end{aligned}
%  \end{array}
\end{equation}
Solving the ordinary differntial equations, we obtain using the outer
Robin boundary conditions for each Fourier mode $k\ne\pm\omega$
the solutions ($\underline{x}:=1-x$)
\begin{equation}\label{HelmholtzSols}
  \thickmuskip=0.9mu
  \begin{aligned}
    \hat{e}_1^n(x,k)&=A_1^n(k)\left({\sqrt{\omega^2-k^2}\cos(\sqrt{\omega^2-k^2}x)
      + p^l\sin(\sqrt{\omega^2-k^2}x)}\right),\\
    \hat{e}_2^n(x,k)&=A_2^n(k)\left({\sqrt{\omega^2-k^2}\cos(\sqrt{\omega^2-k^2}\underline{x}) + p^r\sin(\sqrt{\omega^2-k^2}\underline{x})}\right).
  \end{aligned}
\end{equation}
Comparing with the solutions for the screened Laplace equation in
\R{ScreenedLaplaceSols}, we see that the hyperbolic sines and
cosines are simply replaced by the normal sines and cosines, which
shows that the solutions are oscillatory, instead of decaying. Note
however that the arguments of the sines and cosines are only real if
$k^2<\omega^2$, \ie\ for Fourier modes below the frequency parameter
$\omega$. For larger Fourier frequencies, the argument becomes
imaginary, and the sines and cosines need to be replaced by their
hyperbolic variants and we obtain that solutions behave like for the
screened Laplace problem in \R{ScreenedLaplaceSols}.

The two remaining constants $A_1^n(k)$ and $A_2^n(k)$ are determined
again using the transmission conditions, and we obtain after a short
calculation again the convergence factor of the form \R{RhoOSM},
with
\begin{equation}\label{rho12H}
  \arraycolsep0.1em
  \begin{array}{rcl}
    \rho_1 &=& \frac{(k^2 - \sigma_1^rp^r - \omega^2)
      \sin(\sqrt{\omega^2-k^2}(X_1^r - 1))
        - \cos(\sqrt{\omega^2-k^2}(X_1^r - 1))
        \sqrt{\omega^2-k^2}(p^r - \sigma_1^r)}
           {(k^2 + \sigma_1^rp^l - \omega^2)\sin(\sqrt{\omega^2-k^2}X_1^r)
             + \cos(\sqrt{\omega^2-k^2}X_1^r)\sqrt{\omega^2-k^2}(p^l + \sigma_1^r)},\\
    \rho_2&=& \frac{(k^2 - \sigma_2^lp^l -\omega^2)
      \sin(\sqrt{\omega^2-k^2}X_2^l)
      + \cos(\sqrt{\omega^2-k^2}X_2^l)
      \sqrt{\omega^2-k^2}(p^l - \sigma_2^l)}
        {(k^2 + \sigma_2^lp^r -\omega^2)\sin(\sqrt{\omega^2-k^2}(X_2^l - 1))
          - \cos(\sqrt{\omega^2-k^2}(X_2^l - 1))\sqrt{\omega^2-k^2}(p^r + \sigma_2^l)}.
  \end{array}  
\end{equation}
It is again instructive to look at the special case first where the
domain, and thus the subdomains, are unbounded on the left and right,
which can be obtained from the result above by introducing for the
outer Robin parameters the symbol of the DtN operator,
\ie\ $p^r:=\vartheta\I\sqrt{\omega^2-k^2}$,
$p^l:=\vartheta\I\sqrt{\omega^2-k^2}$, where
$\vartheta=\mathrm{sign}(\omega^2-k^2)$ if $\sqrt{z}$ for
$z\in\mathbb{C}$ uses the branch cut $(-\infty,0)$. These symbols can
be obtained as shown in footnote \ref{footnoteDtN} for the screened
Laplace problem, and leads after simplification to the convergence factor
\begin{equation}\label{RhoOSMUnboundedH}
  \thickmuskip=0.6mu
  \medmuskip=0.3mu
  \rho(k,\omega,\sigma_1^r,\sigma_2^l)=
  \frac{\vartheta\I\sqrt{\omega^2-k^2}-\sigma_2^l}{\vartheta\I\sqrt{\omega^2-k^2}+\sigma_2^l}\,
  \frac{\vartheta\I\sqrt{\omega^2-k^2}-\sigma_1^r}{\vartheta\I\sqrt{\omega^2-k^2}+\sigma_1^r}
  e^{-2(X_1^r-X_2^l)\sqrt{k^2-\omega^2}}.
\end{equation}
As in the case of the screened Laplace equation, we see again that if
we choose in the transmission conditions the symbol of the DtN
operators, $\sigma_1^r=\sigma_2^l=\vartheta\I\sqrt{\omega^2-k^2}$,
the convergence factor vanishes identically for all Fourier
frequencies $k$, and we get an optimal Schwarz method,
see footnote \ref{footnoteoptimalSchwarz}, as for the screened Laplace
problem.

In the bounded domain case, we can still obtain an optimal Schwarz
method choosing in the transmission conditions the tangential symbols of the
transparent boundary conditions for the bounded subdomains, which is
equivalent to choosing $\sigma_1^r$ and $\sigma_2^l$ such that
$\rho_1$ and $\rho_2$ in \R{rho12H} vanish, and leads to
\begin{equation}\label{DtNbddH}
  \begin{array}{rcl}
     \sigma_1^r&=&\frac{\sqrt{\omega^2-k^2}\left(-\tan(\sqrt{\omega^2-k^2}(X_1^r-1))\sqrt{\omega^2-k^2} - p^r\right)}{\tan(\sqrt{\omega^2-k^2}(X_1^r - 1))p^r-\sqrt{\omega^2-k^2}},\\
     \sigma_2^l&=&\frac{\sqrt{\omega^2-k^2}\left(-\tan(\sqrt{\omega^2-k^2}X_2^l)\sqrt{\omega^2-k^2} + p^l\right)}{\tan(\sqrt{\omega^2-k^2}X_2^l)p^l+\sqrt{\omega^2-k^2}}.
  \end{array}
\end{equation}
As in the unbounded domain case, these values correspond to the
symbols of the associated DtN operators, as one can verify by a direct
computation using the solutions \R{HelmholtzSols} on their
respective domains, as done for the unbounded domain case of the
screened Laplace problem in footnote \ref{footnoteDtN}. We see that in
the bounded domain case, the symbols of the DtN maps in
\R{DtNbddH}, and hence the DtN maps also, depend on the outer
boundary conditions, since again both the domain parameters ($X_1^r$
and $X_2^l$) and the Robin parameters in the outer boundary conditions
($p^l$ and $p^r$) appear in them. For large frequencies $k$,
we still have
$$
\sigma_1^r\sim \sqrt{k^2-\omega^2},\quad  \sigma_2^l\sim \sqrt{k^2-\omega^2},
$$
since the tangent function for imaginary argument $\I z$ becomes $\I\tanh z\to\pm\I$ as
  $z\to\pm\infty$, and thus high frequencies, also called evanescent modes in the Helmholtz
solution, still do not see the difference from the bounded to the unbounded case, as in the screened
Laplace problem.
  
Choosing Robin transmission conditions, ${\cal
  S}_1^r:=p_1^r=\sigma_1^r$ and ${\cal S}_2^l:=p_2^l=\sigma_2^l$ with
$p_1^r,p_2^l\in\mathbb{R}$ in \R{TC}, and taking the limit in the
convergence factor \R{RhoOSMUnbounded} as the Robin parameters
$p_1^r$ and $p_2^l$ go to infinity, we find
\begin{equation}\label{RhoAltSUnboundedH}
  \rho(k,\eta,)=e^{-2(X_1^r-X_2^l)\sqrt{k^2-\omega^2}},
\end{equation}
which is the convergence factor of the classical alternating Schwarz
method for the Helmholtz equation on the unbounded domain, and we see
again the overlap $L:=X_1^r-X_2^l$ appearing in the exponential
function. The classical Schwarz method therefore converges for all
Fourier modes $k^2>\omega^2$, provided there is overlap. However for
smaller Fourier frequencies, the method does not contract on the
unbounded domain.

To see how the Schwarz method contracts on a bounded domain, we show
in Figure \ref{SchwarzHelmholtzRhos} (left) the different cases for
the outer boundary conditions ($p^l=p^r=\I\,\omega$ in the Robin case),
by plotting \R{RhoOSM} for when the parameters in the Robin
transmission conditions goes to infinity, and the Helmholtz frequency
parameter $\omega=10$, again with overlap $L=0.02$.
\begin{figure}
  \centering
  \includegraphics[width=0.48\textwidth]{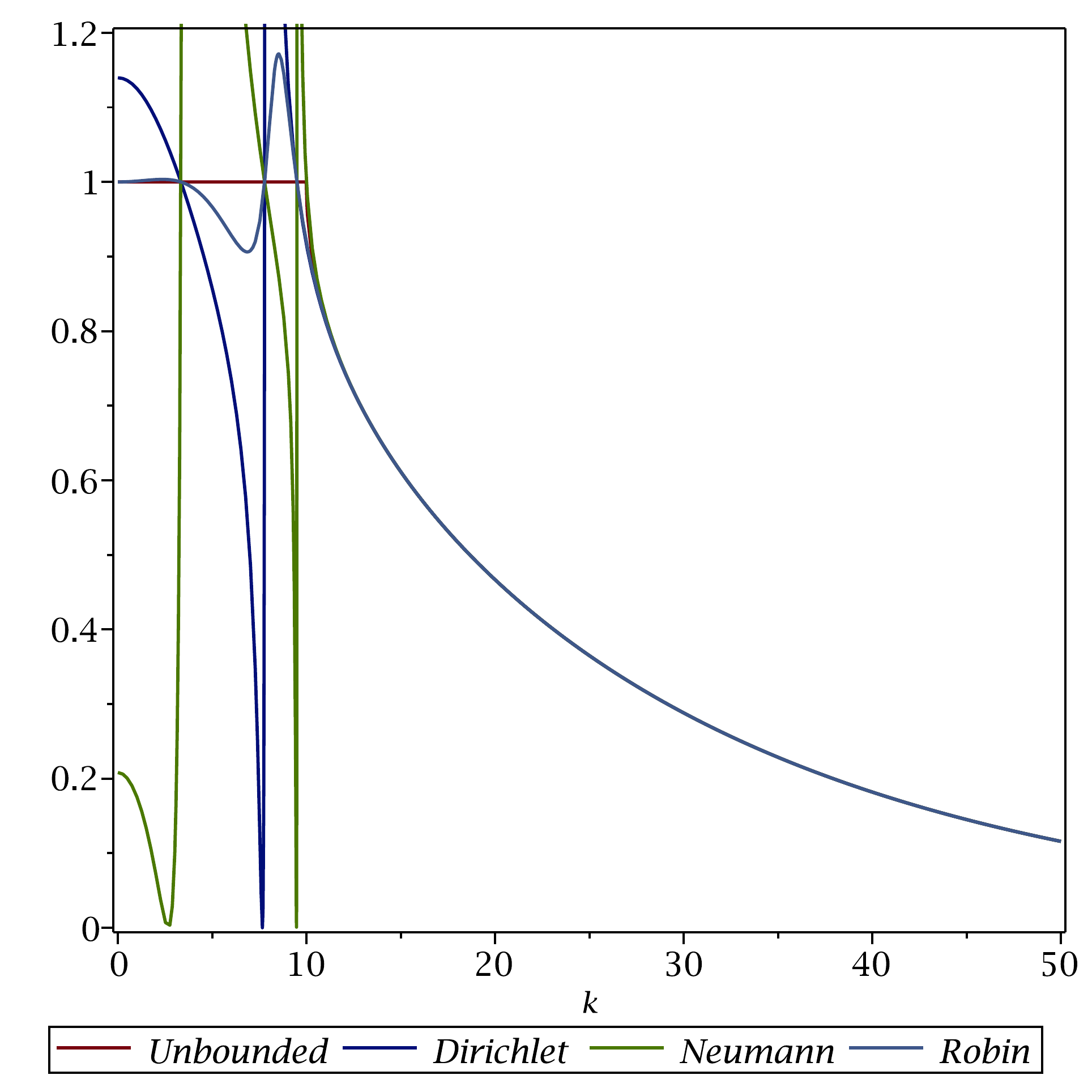}
  \includegraphics[width=0.48\textwidth]{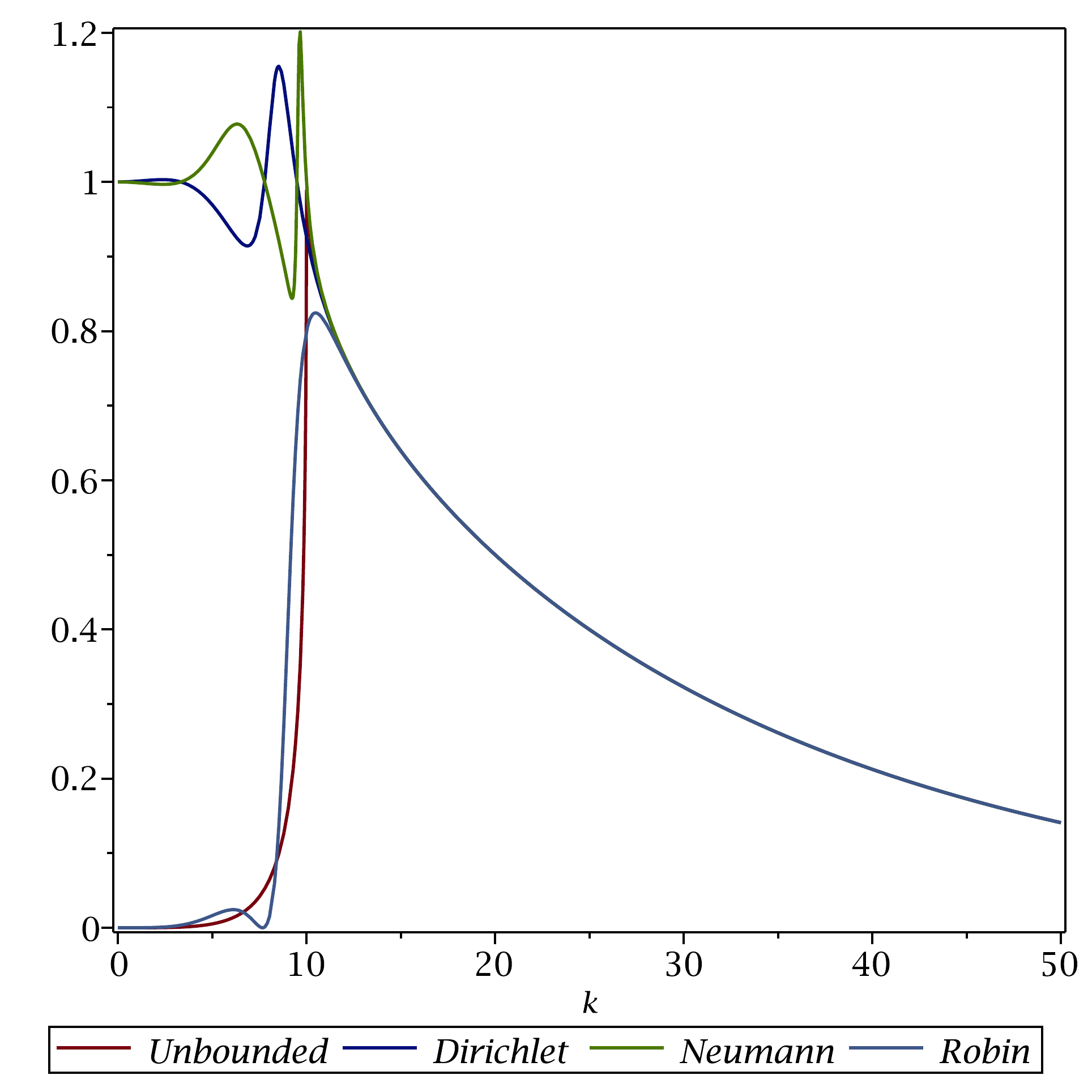}
  \caption{Left: classical Schwarz convergence factor for Helmholtz
    and different outer boundary conditions on the left and
    right. Right: corresponding optimized Schwarz convergence factor
    with Taylor transmission conditions from the unbounded domain
    analysis.}
  \label{SchwarzHelmholtzRhos}
\end{figure}
This shows that the different outer boundary conditions greatly
influence the convergence for Fourier modes $k^2<\omega^2$, the
Schwarz method violently diverges for these modes, except for the
unbounded case where we obtain stagnation, but still no
convergence. For larger frequencies $k^2>\omega^2$, there is no
influence of the outer boundary conditions on the convergence.  On the
right in Figure \ref{SchwarzHelmholtzRhos} we show the corresponding
convergence factors with Robin transmission conditions, chosen as
$p_1^r=p_2^l=\I\,\omega$, which correspond to Taylor transmission
conditions of order zero from the unbounded domain
anaylsis\footnote{It is interesting to note that the same result
    is also obtained for the bounded domain case when the outer Robin
    transmission conditions are chosen as $p^r=p^l=\I\,\omega$.}, by
expansion of the corresponding optimal symbol around $k=0$,
$\I\sqrt{\omega^2-k^2}=\I\,\omega+O(k^2)$. Here we see that again the
different outer boundary conditions greatly influence the convergence
for Fourier modes $k^2<\omega^2$: for unbounded subdomains, and also
subdomains with Robin radiation conditions at the ends, the Schwarz
method converges well for small Fourier frequencies. With Dirichlet
and Neumann outer boundary conditions however, we obtain again
divergence, although a bit less violent than in the classical Schwarz
method. For larger frequencies $k^2>\omega^2$, there is again very
little influence of the outer boundary conditions on the convergence.
We also observe an interesting phenomenon around $k^2=\omega^2$ the so
called resonnance frequency: with Robin radiation conditions, the
Schwarz method also converges in this case, while for unbounded
subdomains we obtain a convergence factor with modulus one there.

It does not help to use the corresponding Taylor transmission
conditions of order zero adapted to each outer boundary conditions, as we show
in Figure \ref{SchwarzHelmholtzRhos2} on the left.
\begin{figure}
  \centering
  \includegraphics[width=0.48\textwidth]{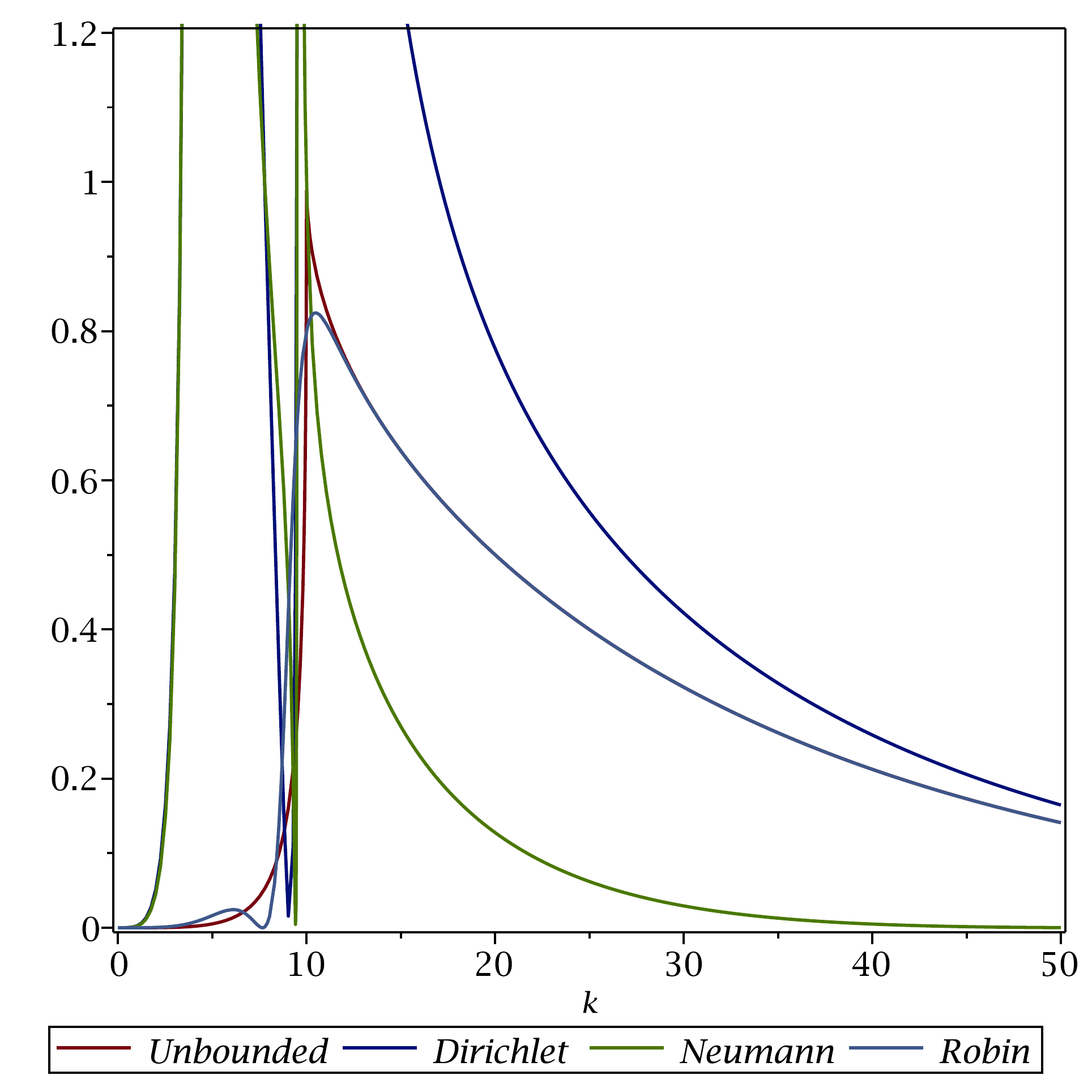}
  \includegraphics[width=0.48\textwidth]{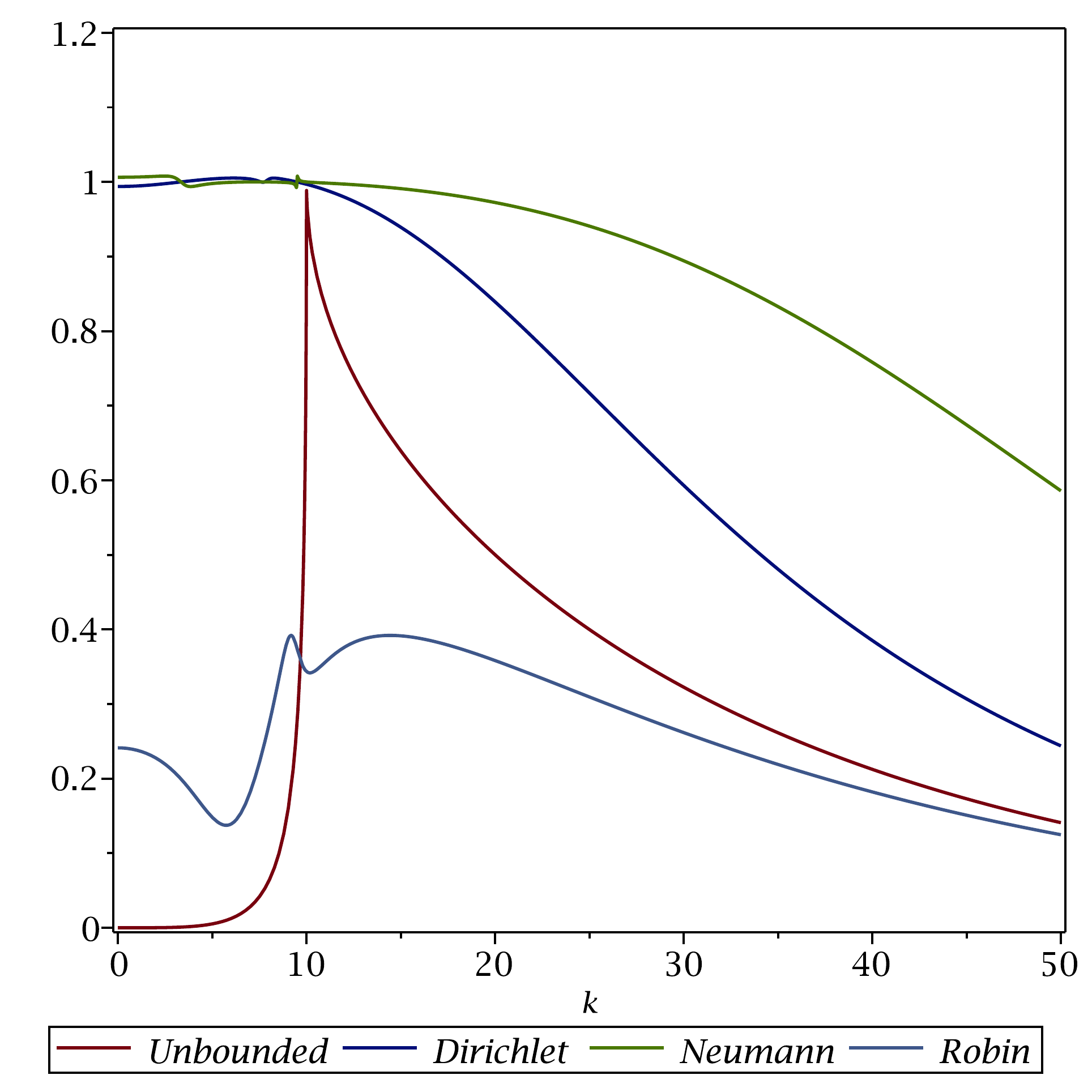}
  \caption{Left: optimized Schwarz convergence factor for
      Helmholtz and different outer boundary conditions on the left
      and right, using the correspondingly adapted Taylor transmission
      conditions. Right: corresponding optimized Schwarz convergence
      factor minimizing the maximum of the convergence factor
      numerically with complex Robin transmission
      conditions.}  \label{SchwarzHelmholtzRhos2}
\end{figure}
Now the low frequencies converge well in all cases, but divergence is
even more violent for other frequencies with Dirichlet and Neumann
outer boundary conditions. We finally show the results of numerically
optimized Robin transmission conditions with
$p_1^r,p_2^l\in\mathbb{C}$ on the right in Figure
\ref{SchwarzHelmholtzRhos2}, minimizing the convergence factor in
modulus\footnote{For the unbounded case we did not optimize since the
  convergence factor equals $1$ there for $k=\omega$. There are
  however optimization techniques in this case as well, see
  \eg\ \cn{GMN02}, \cn{GHM07} and references therein.}. We see that there do not exist complex Robin parameters
that can make the optimized Schwarz method work well for Dirichlet and
Neumann outer boundary conditions on the left and right, low
frequencies simply converge badly or not at all. However with Robin
outer boundary conditions on the left and right, the optimization
leads to a quite fast method now, with a convergence factor bounded by
$0.4$, compared to the Taylor transmission conditions where the
convergence factor was bounded by about $0.8$. This shows that for
Helmholtz problems absorbing boundary conditions on the original
problem are very important for Schwarz methods to converge, and we
thus focus on this case in what follows, \ie\ $p^l=p^r=\I \omega$.

We start by studying the performance with Taylor transmission
  conditions, $p_1^r=p_2^l=\I \omega$, when the overlap is varying,
  and the Helmholtz frequency $\omega$ is increasing. We show in
  Figure \ref{SchwarzHelmholtzLsmall}
\begin{figure}
  \centering
  \includegraphics[width=0.48\textwidth]{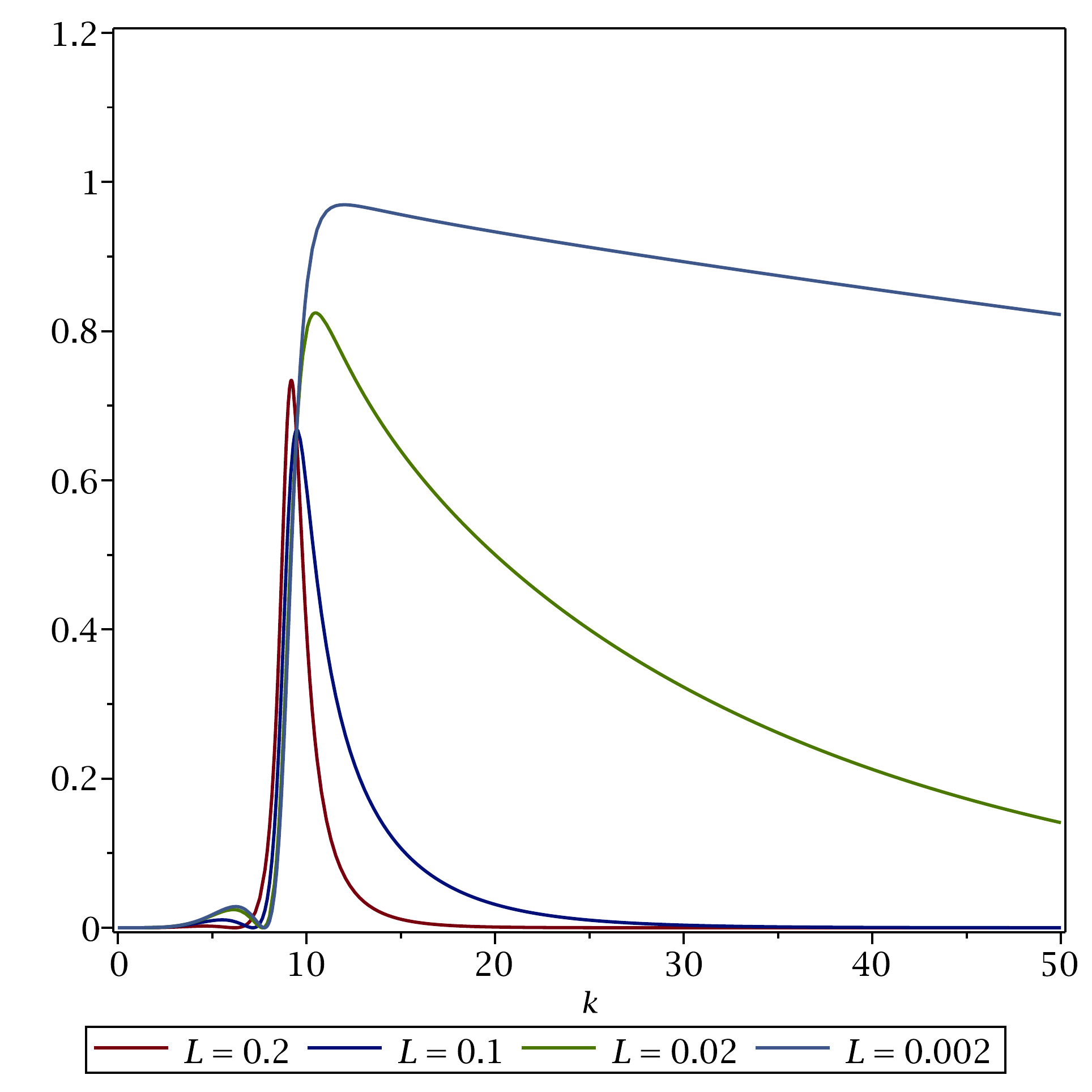}
  \includegraphics[width=0.48\textwidth]{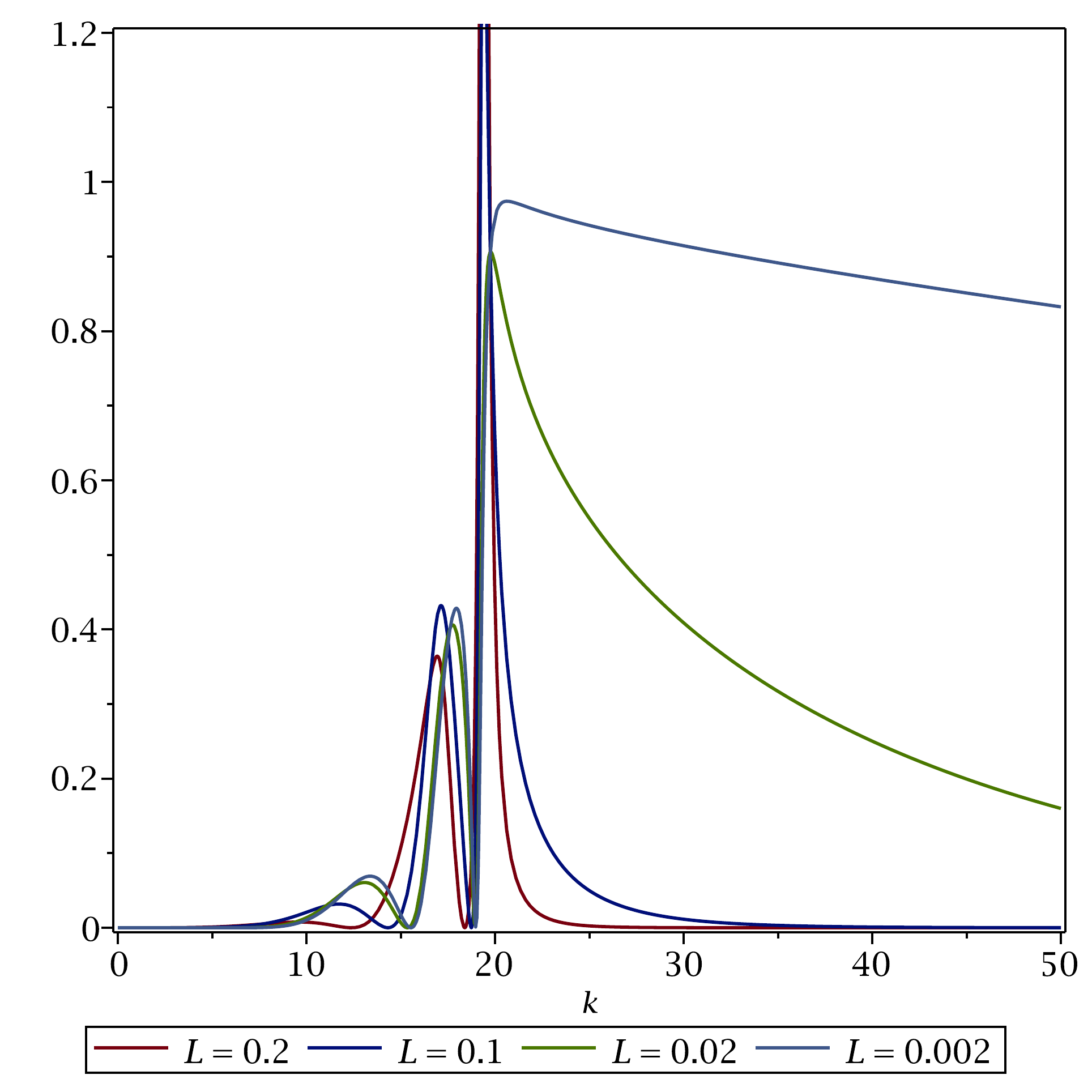}
  \includegraphics[width=0.48\textwidth]{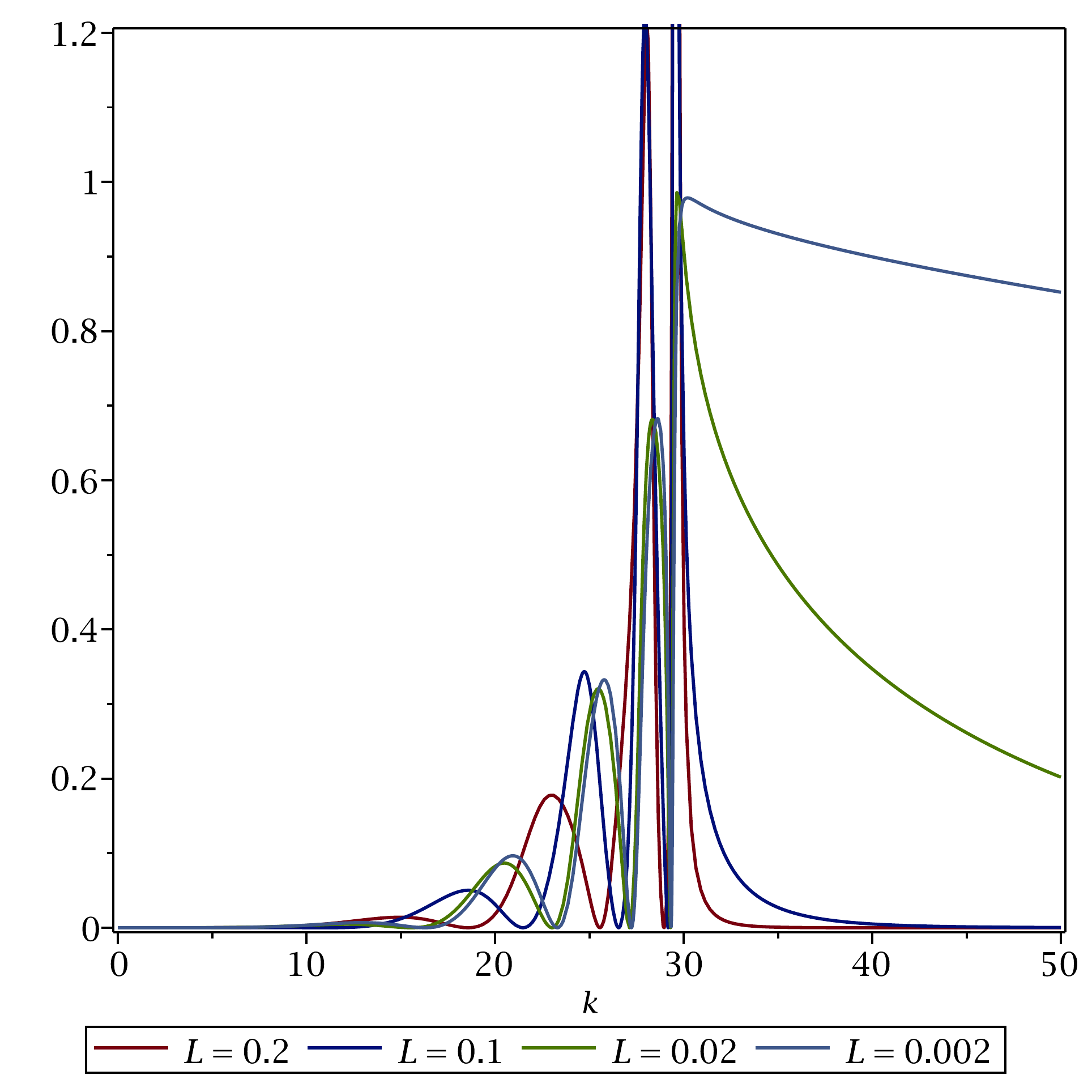}
  \includegraphics[width=0.48\textwidth]{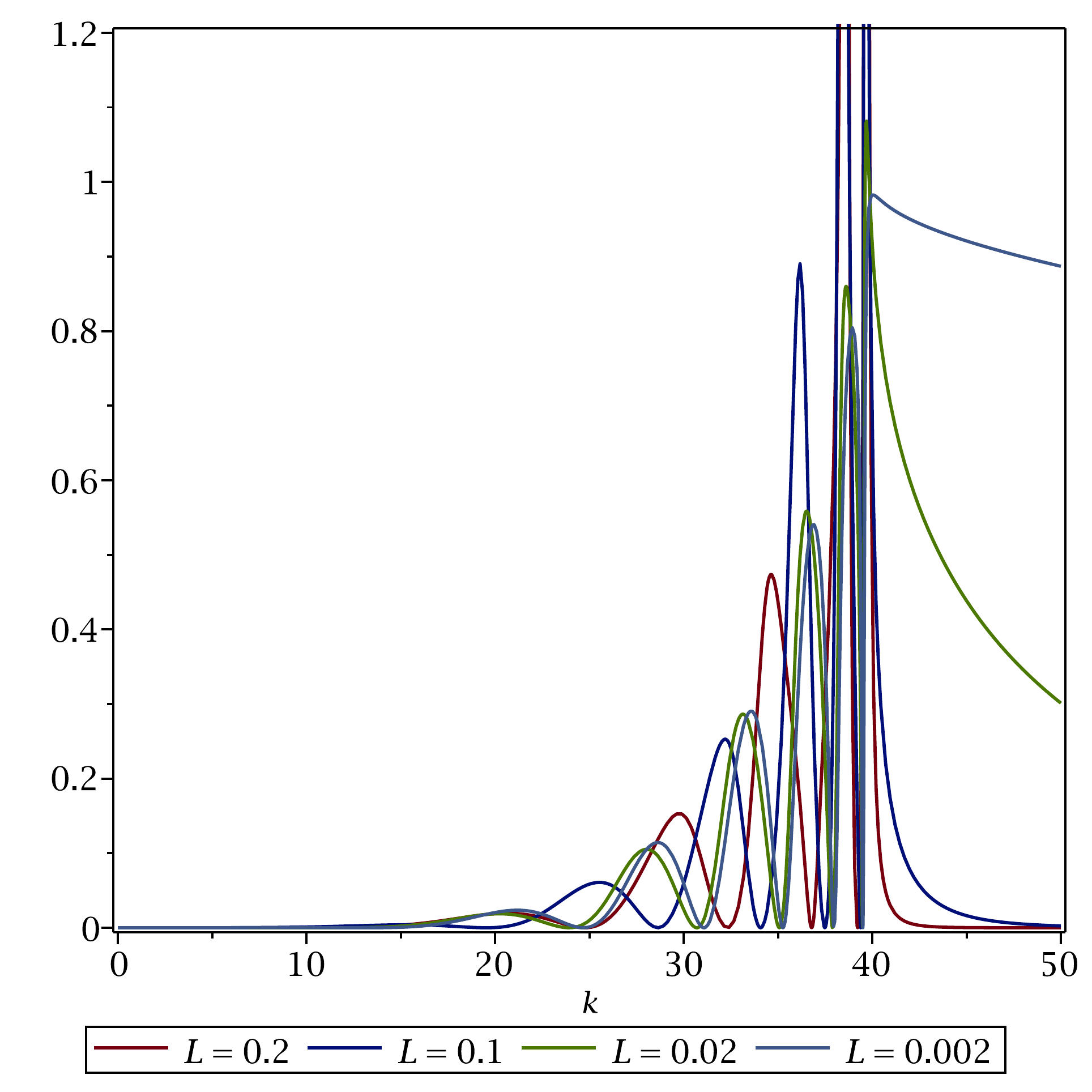}
  \caption{Optimized Schwarz convergence factor for Helmholtz with
      Robin outer boundary conditions and Taylor transmission
      conditions for different overlap sizes: from top left to bottom
      right $\omega=10,20,30,40$.}  \label{SchwarzHelmholtzLsmall}
\end{figure}
the convergence factors for four different overlap sizes
$L=0.2,0.1,0.02,0.002$ and Helmholtz frequencies $\omega=10,20,30,40$.
We see that convergence difficulties increase dramatically with
increasing $\omega$: while for $\omega=10$, the convergence factor is
smaller than one for all overlap sizes $L$, we see already that the
largest overlap $L=0.2$ is not leading to better convergence than the
next smaller one, $L=0.1$. For $\omega=20$ now the two larger overlaps
$L=0.2$ and $L=0.1$ lead to divergence.  For $\omega=30$ the best
performance is obtained for the smallest overlap $L=0.002$, and for
$\omega=40$ the smallest overlap is the only overlap for which the
method still contracts. We thus need small overlap in this waveguide
type setting with Dirichlet or Neumann conditions on top and bottom
for the method to work with Taylor transmission conditions.

Let us investigate this analytically when the overlap $L$ goes to
zero. Computing the location of the maximum at $\bar{k}>\omega$
visible in Figure \ref{SchwarzHelmholtzLsmall} and then evaluating the
value of the convergence factor in modulus at this frequency
$\bar{k}$, we obtain for the convergence factor for small overlap $L$
after a long and technical computation
  \begin{equation}\label{TaylorRhoHelmholtz}
    \max_k|\rho(k,\omega,L)|= 1 - (8-4\ln 2 - 4\ln L)L + O(L^2).
  \end{equation}
  The method therefore converges for small enough overlap also when
  $\omega$ becomes large for this two subdomain decomposition and
  outer Robin conditions on the left and right, $p^l=p^r=\I \omega$
  with Taylor transmission conditions $p_1^r=p_2^l=\I \omega$. However
  convergence is rather slow, compared to the case of the screened
  Laplace equation with convergence factor $1-O(\sqrt{L})$ as shown in
  \eqref{TaylorRho}, even though there is absorption for Helmholtz on
  the left and right. We show in Table \ref{TabTaylorHelmholtz}
  
  \begin{table}
    \centering
    \begin{tabular}{|c|c|c|}
      \hline
      $L$ & $\max_{k}|\rho|$ & $1-\max_{k}|\rho|$ \\\hline
   0.1000000000 &22.7898956625 &-21.7898956625 \\
   0.0100000000 & 1.2309097387 & -0.2309097387 \\
   0.0010000000 & 0.9962177046 &  0.0037822954 \\
   0.0001000000 & 0.9987180209 &  0.0012819791 \\
   0.0000100000 & 0.9998069224 &  0.0001930776 \\
   0.0000010000 & 0.9999749357 &  0.0000250643 \\
   0.0000001000 & 0.9999969499 &  0.0000030501 \\
   0.0000000100 & 0.9999996424 &  0.0000003576 \\
   0.0000000010 & 0.9999999591 &  0.0000000409 \\
   0.0000000001 & 0.9999999954 &  0.0000000046 \\\hline
    \end{tabular}
    \caption{Maximum of the convergence factor for Taylor transmission
      conditions with the same outer boundary conditions on the left
      and right and $\omega=100$, for decreasing overlap size $L$.}
    \label{TabTaylorHelmholtz}    
  \end{table}
  the convergence factor for $\omega=100$ when $L$ goes to zero for a
  symmetric subdomain decomposition, $X_2^l=\frac{1}{2}-L/2$,
  $X_1^r=\frac{1}{2}+L/2$, and one can clearly see the logarithmic term $\ln
  L$ in the last column which displays $1-\max_{k}|\rho|$, from the
  slow increase in the first non-zero digit. One can also see from
  Table \ref{TabTaylorHelmholtz} that for a given Helmholtz frequency
  $\omega$, there is an optimal choice for the overlap, here around
  $L=0.001$ for best performance. 

We next investigate if an optimized choice of the complex
  transmission parameters $p_1^r,p_2^l\in\mathbb{C}$ can further
  improve the convergence behavior with Robin conditions with
  $p^r=p^l=\I\omega$ on the left and right outer boundaries. A
  numerical optimization minimizing the maximum of the convergence
  factor gives the results shown in Figure
  \ref{SchwarzHelmholtzOptLsmall}.
\begin{figure}
  \centering
  \includegraphics[width=0.48\textwidth]{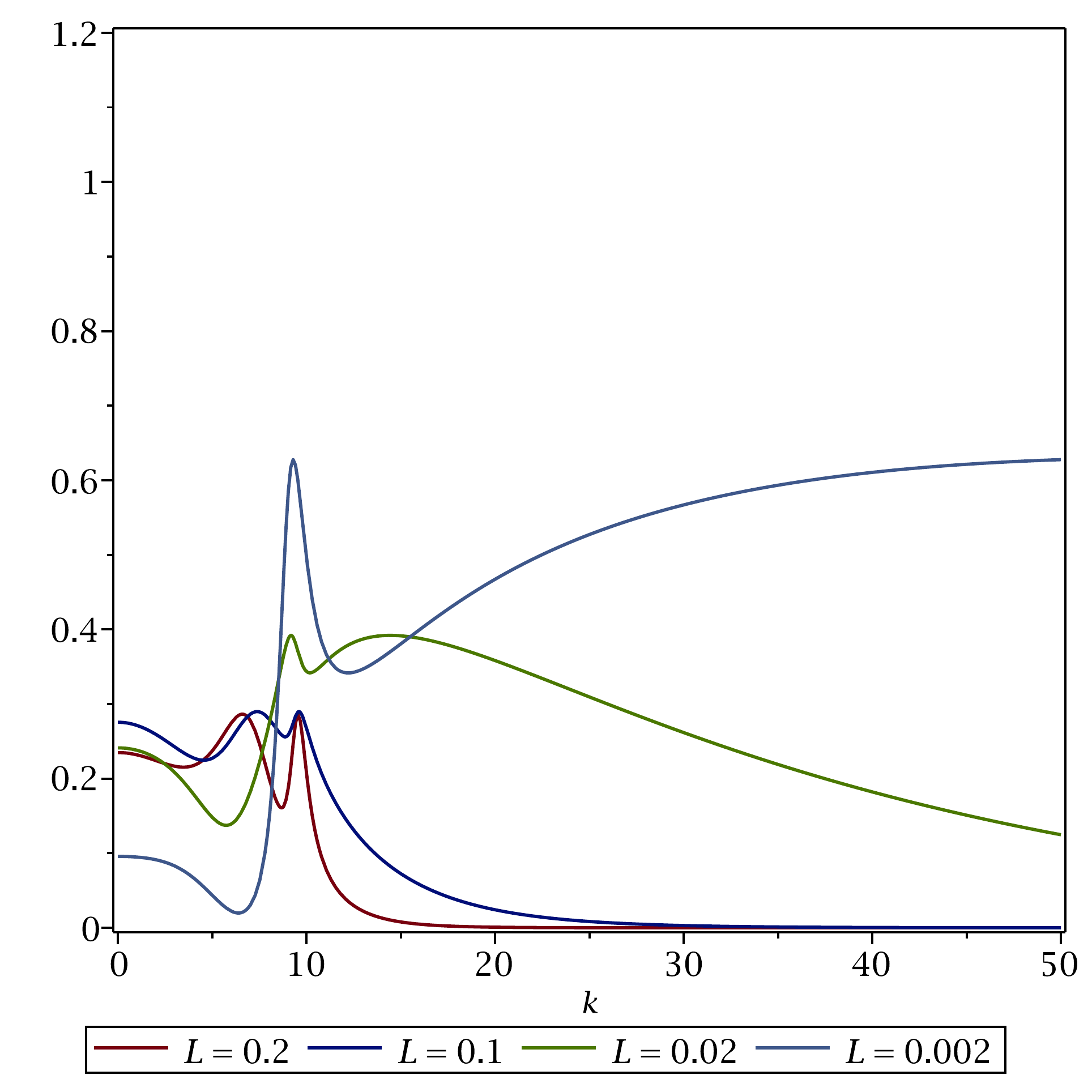}
  \includegraphics[width=0.48\textwidth]{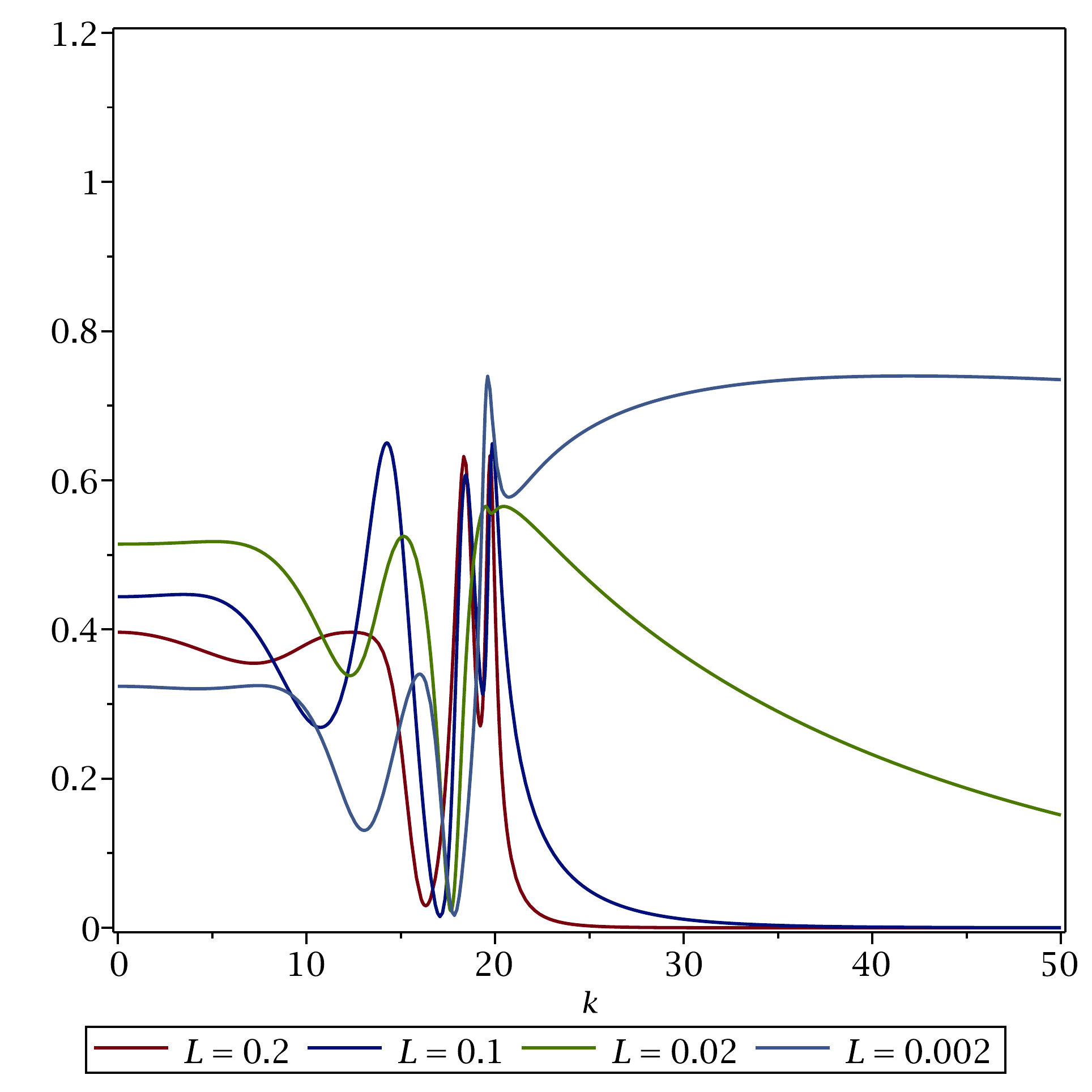}
  \includegraphics[width=0.48\textwidth]{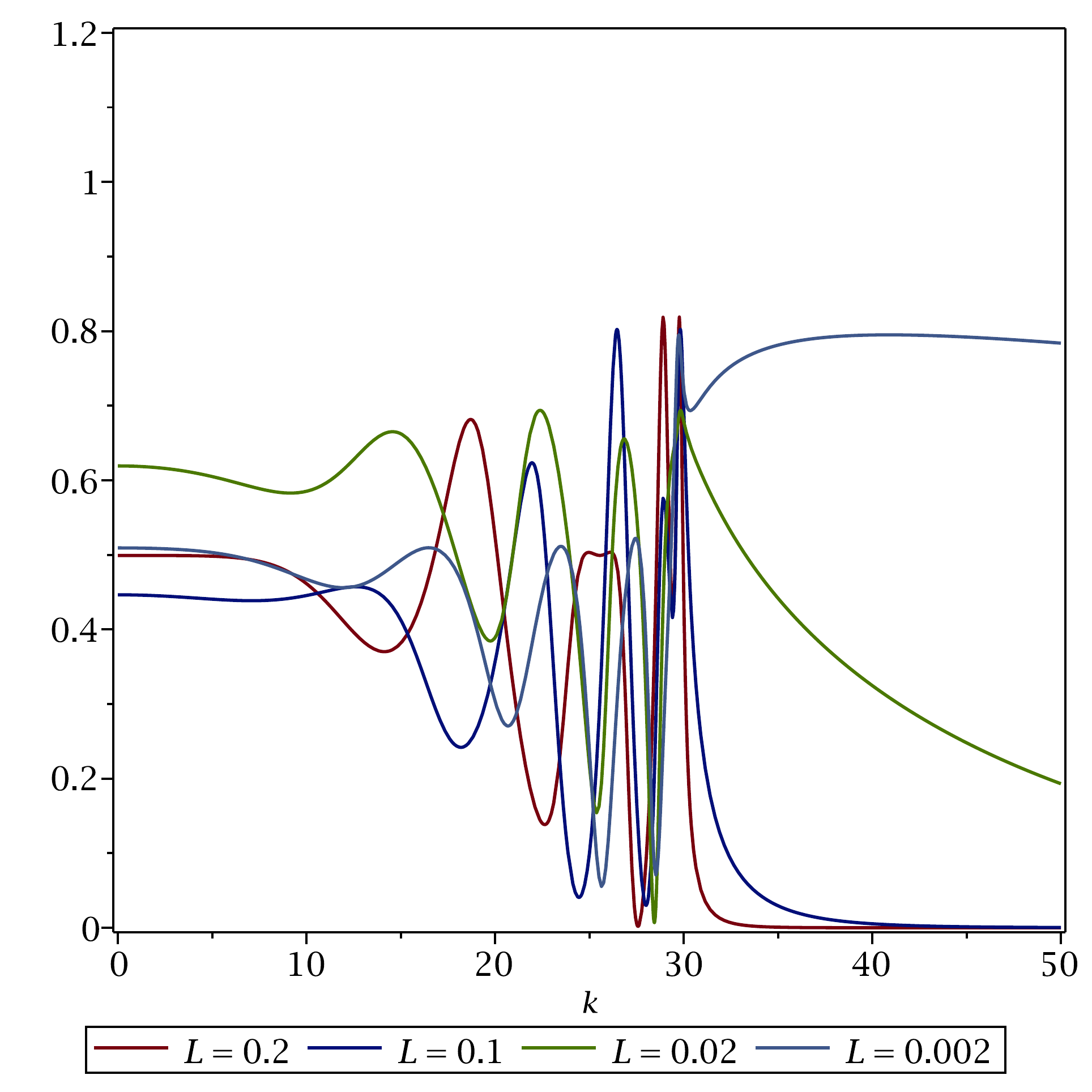}
  \includegraphics[width=0.48\textwidth]{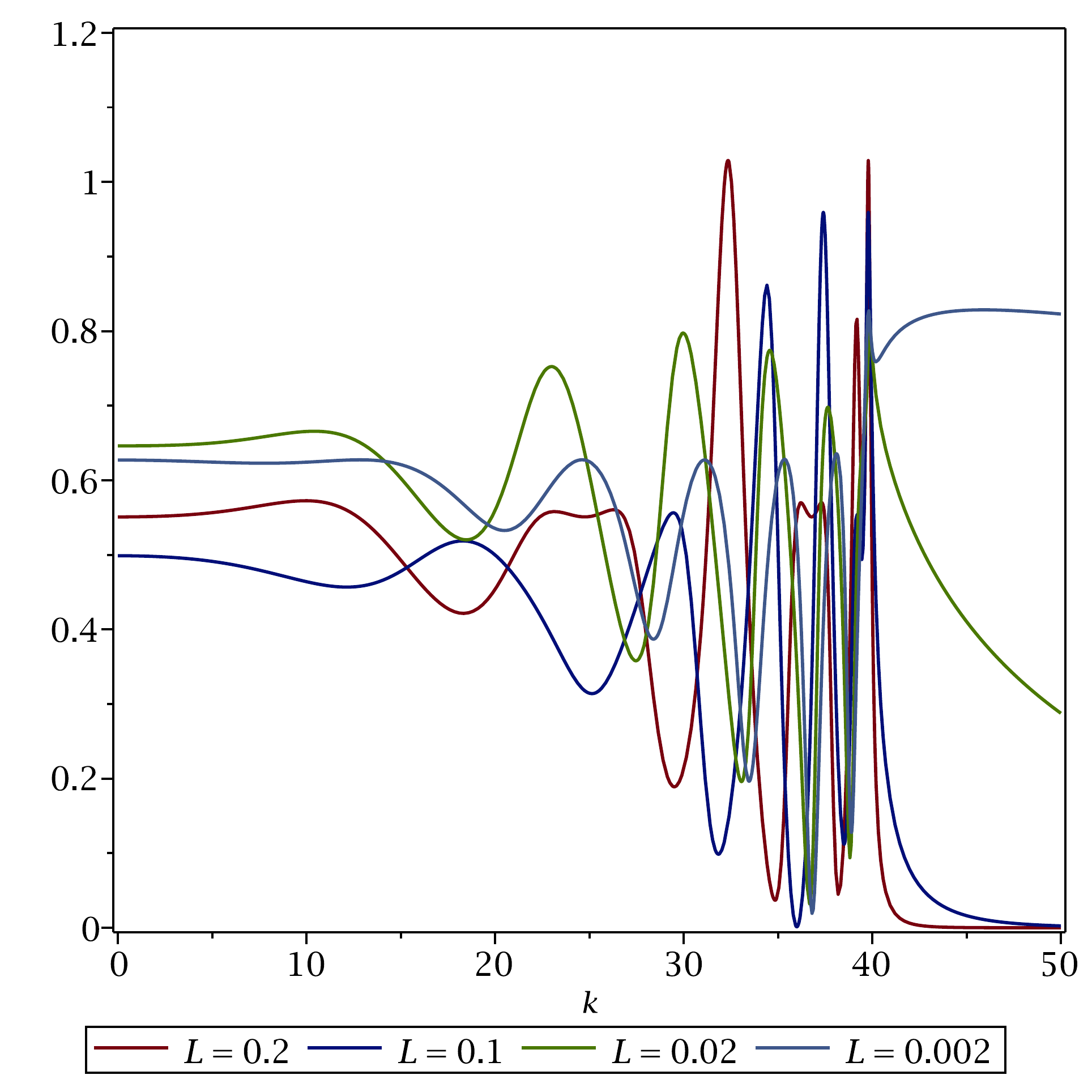}
  \caption{Optimized Schwarz convergence factor for Helmholtz with
      Robin outer boundary conditions and optimized complex Robin
      transmission conditions for different overlap sizes: from top
      left to bottom right
      $\omega=10,20,30,40$.}  \label{SchwarzHelmholtzOptLsmall}
\end{figure}
Comparing with the corresponding Figure \ref{SchwarzHelmholtzLsmall}
for the Taylor transmission conditions, we see that much better
contraction factors can be obtained using the optimized transmission
conditions, and also that the method is convergent for all overlaps
for $\omega=10,20,30$. However, for $\omega=40$, we see again that for
the largest overlap, the optimized contraction factor is above $1$,
and thus as in the Taylor transmission conditions, for large Helmholtz
frequency $\omega$, the overlap will need to be small enough for the
optimized method to converge, and there will be a best choice for the
overlap.

To investigate this further, we show in Table
  \ref{TabOptHelmholtz} the best contraction factor that can be
  obtained for $\omega=100$, as we did in Table
  \ref{TabTaylorHelmholtz} for the Taylor transmission conditions, and
  we also show the value of the optimized complex transmission
  condition parameter.
\begin{table}
    \centering
    \begin{tabular}{|c|c|c|c|}
      \hline
      $L$ & $p_1^r=p_2^l$ & $\max_{k}|\rho|$ & $1-\max_{k}|\rho|$ \\\hline
   0.1000000000 &  -13.7512+\I   17.6068  & 1.366317  & -0.366317 \\
   0.0100000000 &   -0.7183+\I   26.7434  & 0.963833  &  0.036166 \\
   0.0010000000 &    0.7700+\I    7.8147  & 0.882396  &  0.117603 \\
   0.0001000000 &    1.6590+\I   11.5704  & 0.929865  &  0.070134 \\
   0.0000100000 &    3.6059+\I   24.6349  & 0.966615  &  0.033384 \\
   0.0000010000 &    7.7838+\I   52.9876  & 0.984342  &  0.015657 \\
   0.0000001000 &   16.7767+\I  114.0584  & 0.992699  &  0.007300 \\
   0.0000000100 &   36.1477+\I  245.6848  & 0.996604  &  0.003395 \\
   0.0000000010 &   77.8794+\I  529.2904  & 0.998422  &  0.001577 \\
   0.0000000001 &  167.7869+\I 1140.3116  & 0.999267  &  0.000732 \\\hline
    \end{tabular}
    \caption{Best complex parameter choice and maximum of the
      convergence factor for optimized complex Robin transmission
      conditions with outer Robin boundary conditions
      $p^l=p^r=\I \omega$ on the left and right and $\omega=100$, for
      decreasing overlap size $L$.}
    \label{TabOptHelmholtz}    
\end{table}
We see that the optimization leads to a method which degrades much
less, compared to the Taylor transmission conditions, and there is
still a best choice for the overlap that leads to the most
advantageous contraction, around overlap $L=0.001$ as in the Taylor
case. But using the optimized parameters the method will converge
approximately $100$ times faster, since
$0.9987180208^{100}=0.879607\approx 0.882396$. There are however so
far no analytical asymptotic formulas for this optimized choice
available. Nevertheless, plotting the results from Table
\ref{TabOptHelmholtz} in Figure \ref{HelmholtzOO0AsymptPlot}
\begin{figure}
  \centering
  \includegraphics[width=0.48\textwidth]{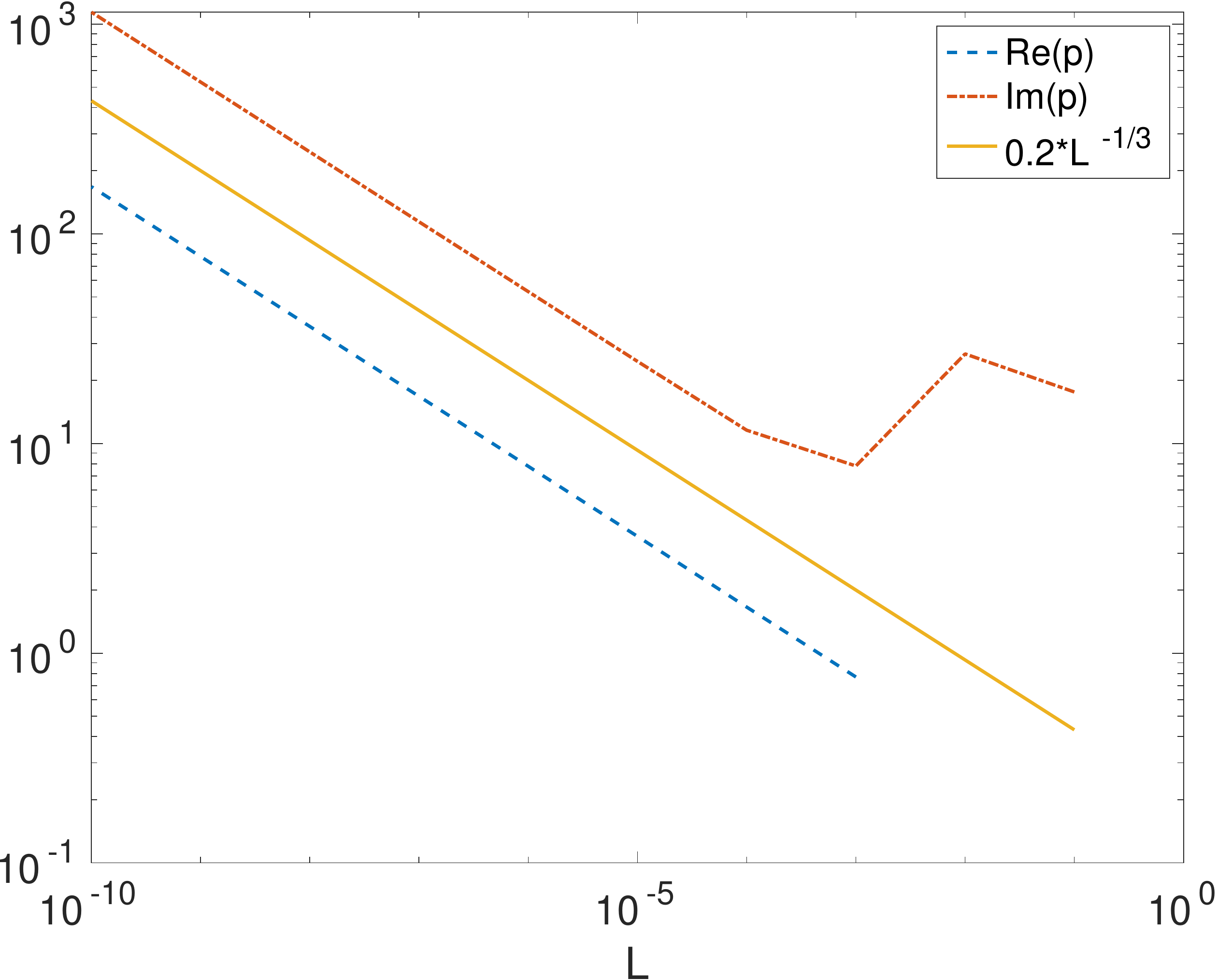}
  \includegraphics[width=0.48\textwidth]{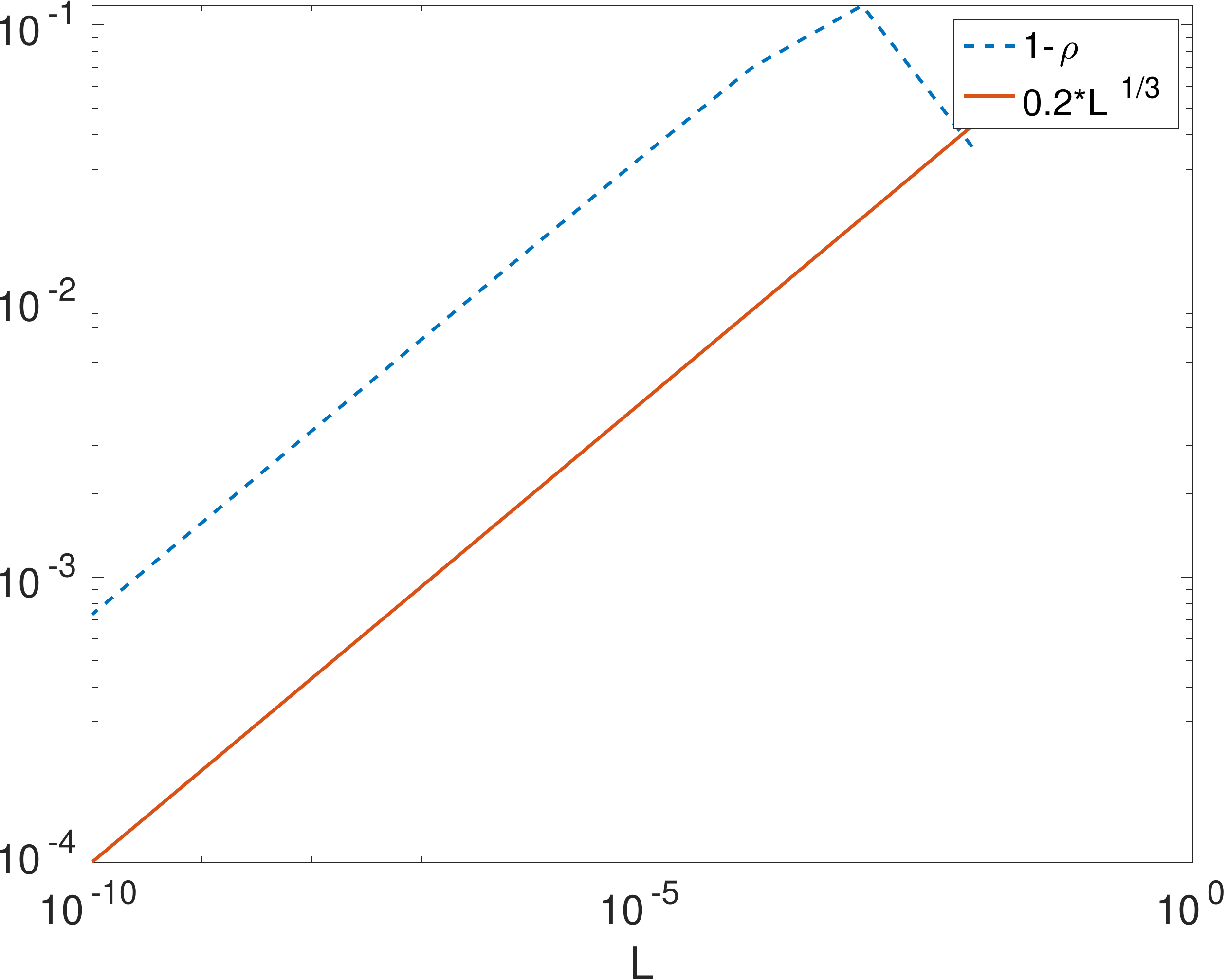}
  \caption{Asymptotic behavior of the optimized complex Robin parameter
    and associated convergence factor distance from $1$.}
  \label{HelmholtzOO0AsymptPlot}
\end{figure}
cleary shows that the optimized choice behaves asymptotically for
small overlap $L$ as
\begin{equation}
  \thickmuskip=0.6mu
  \medmuskip=0.2mu
  \thinmuskip=0.1mu
  \scriptspace=0pt
  \nulldelimiterspace=-0.1pt
  p_1^r\sim(C_1^r+\I \tilde{C}_1^r)L^{-1/3},\quad
  p_2^l\sim(C_2^l+\I \tilde{C}_2^l)L^{-1/3},\quad
  \max_{k}|\rho|= 1-O(L^{1/3}),
\end{equation}
which is a much better result than for the Taylor transmission
conditions that gave $\max_{k}|\rho|= 1-O(L)$ up to the logarithmic
term.
  
Since we have seen how important absorption is for the Helmholtz
  problem, in contrast to the screened Laplace problem, we now study
  the Helmholtz problem also with Robin conditions all around, also on
  top and bottom.
  This needs the solution of the corresponding Sturm-Liuville problem
  (details can be found in Section 3).
  We start again by studying the performance with Taylor transmission
  conditions, $p_1^r=p_2^l=\I \omega$, when the overlap is varying,
  and the Helmholtz frequency $\omega$ is increasing. We show in
  Figure \ref{SchwarzHelmholtzRRLsmall}
\begin{figure}
  \centering
  \includegraphics[width=0.48\textwidth]{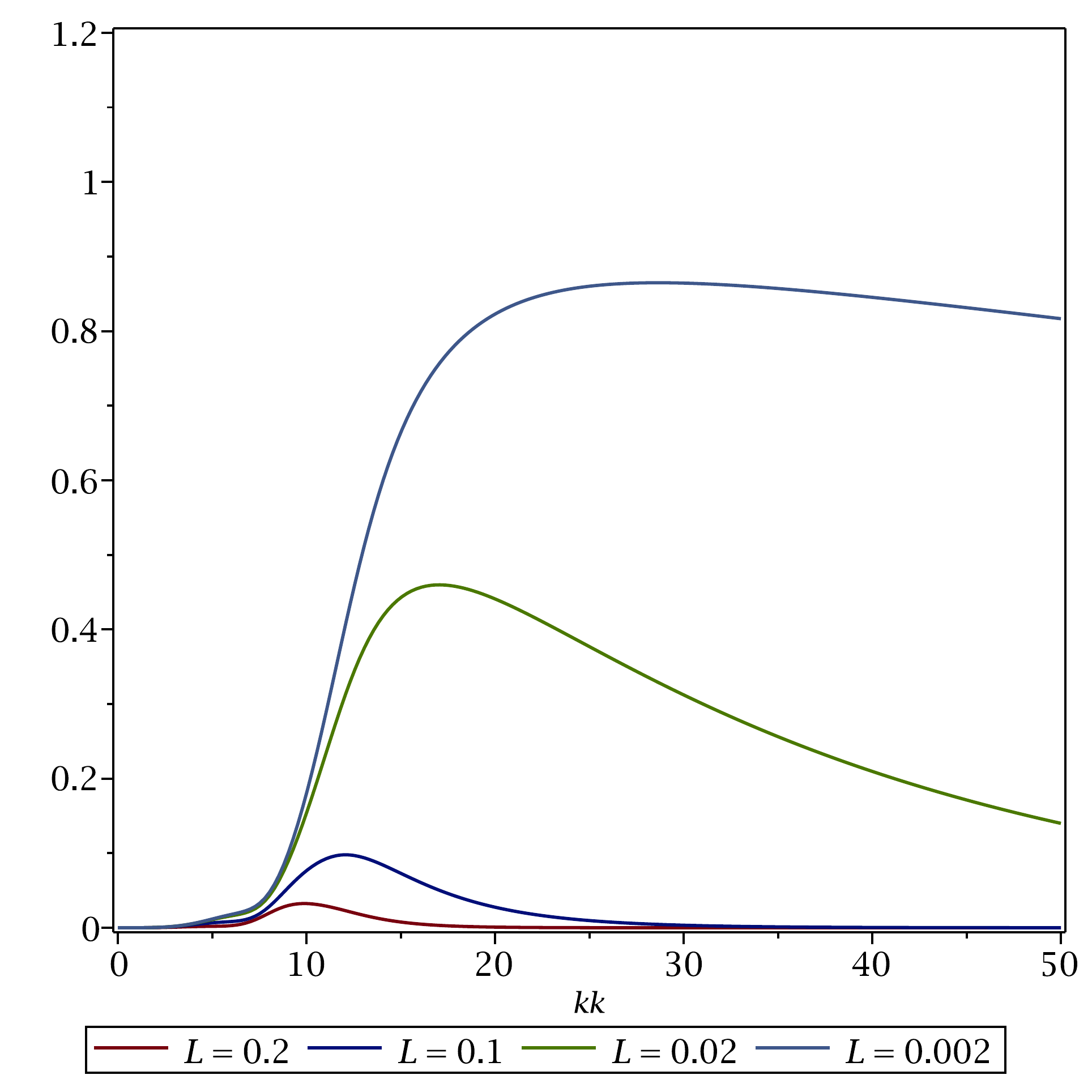}
  \includegraphics[width=0.48\textwidth]{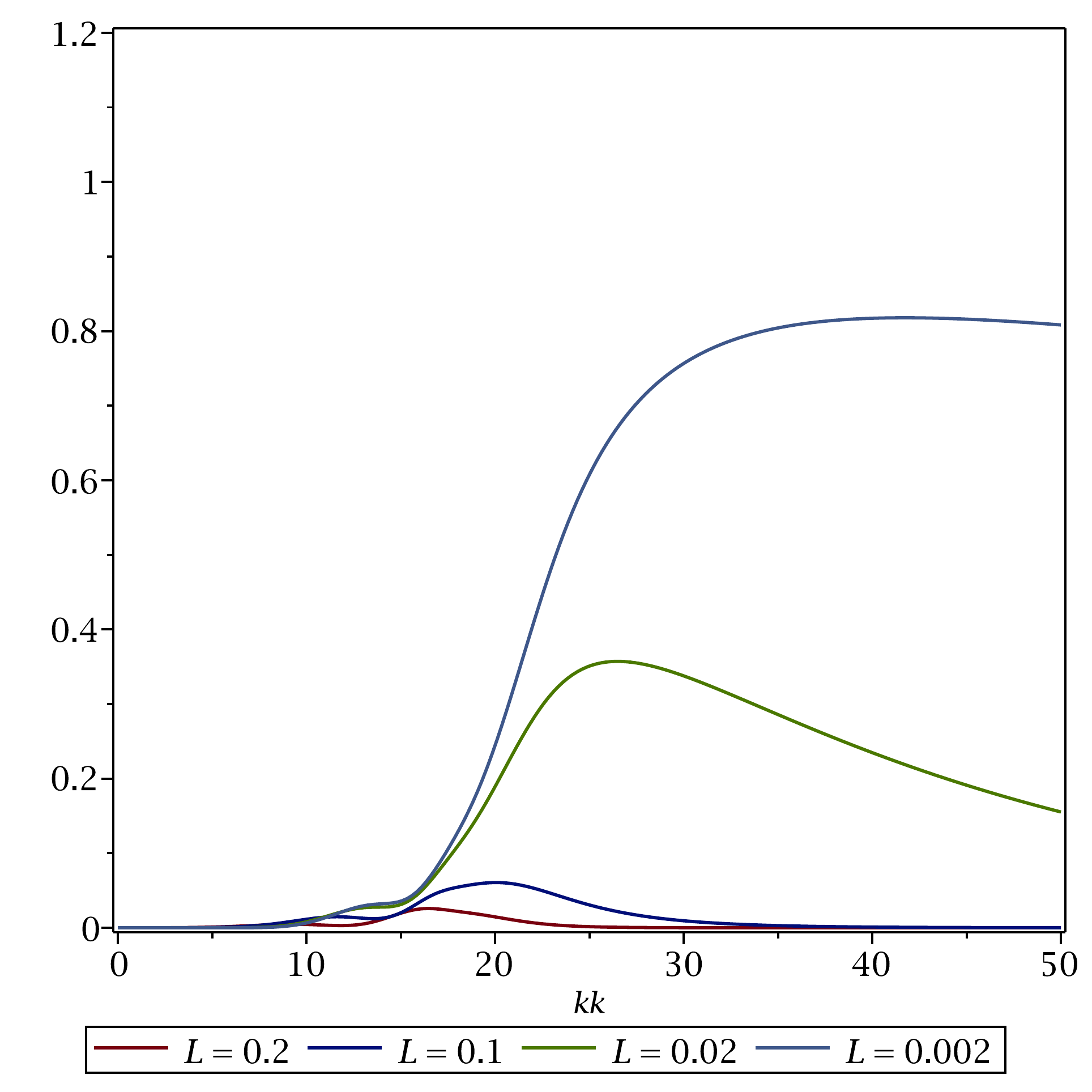}
  \includegraphics[width=0.48\textwidth]{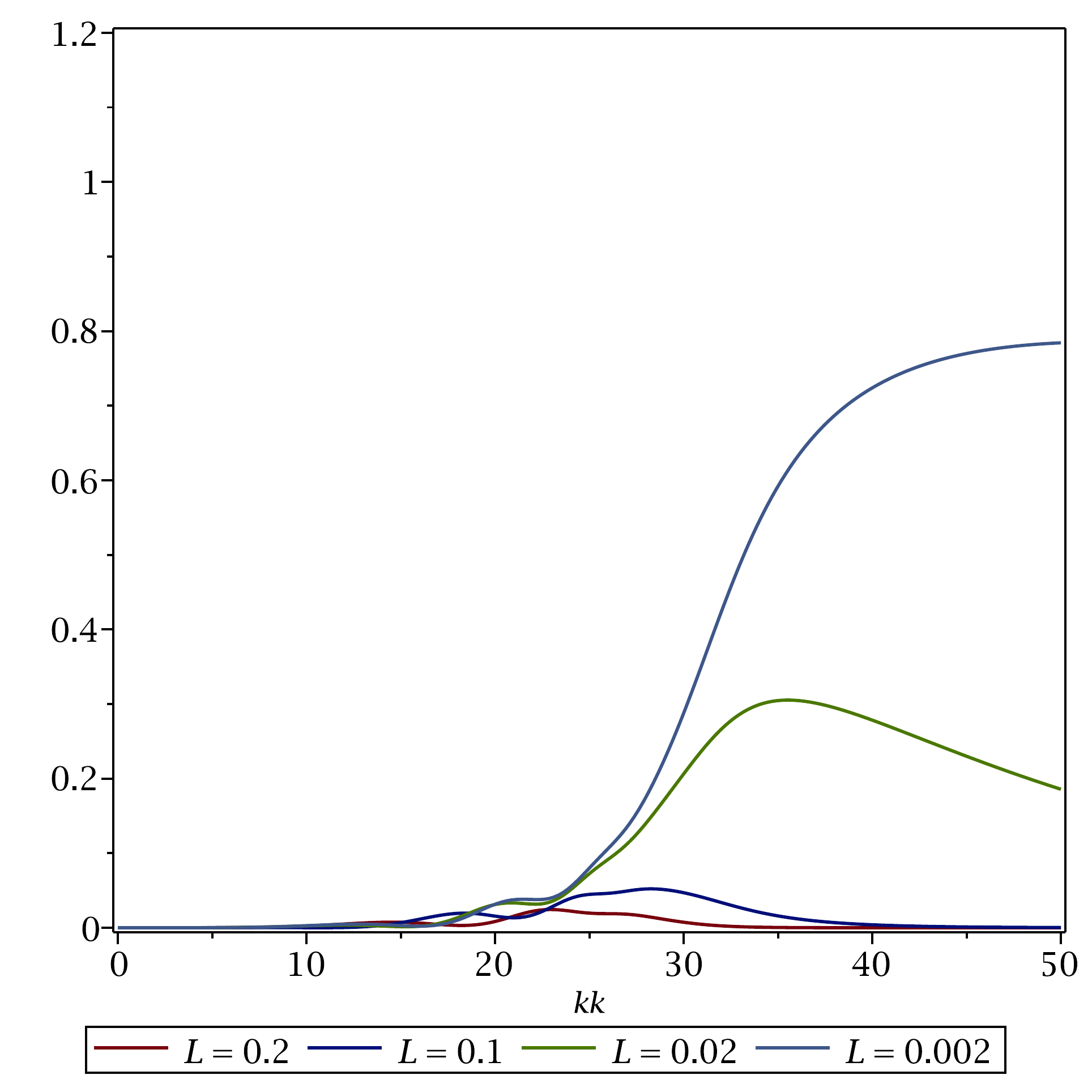}
  \includegraphics[width=0.48\textwidth]{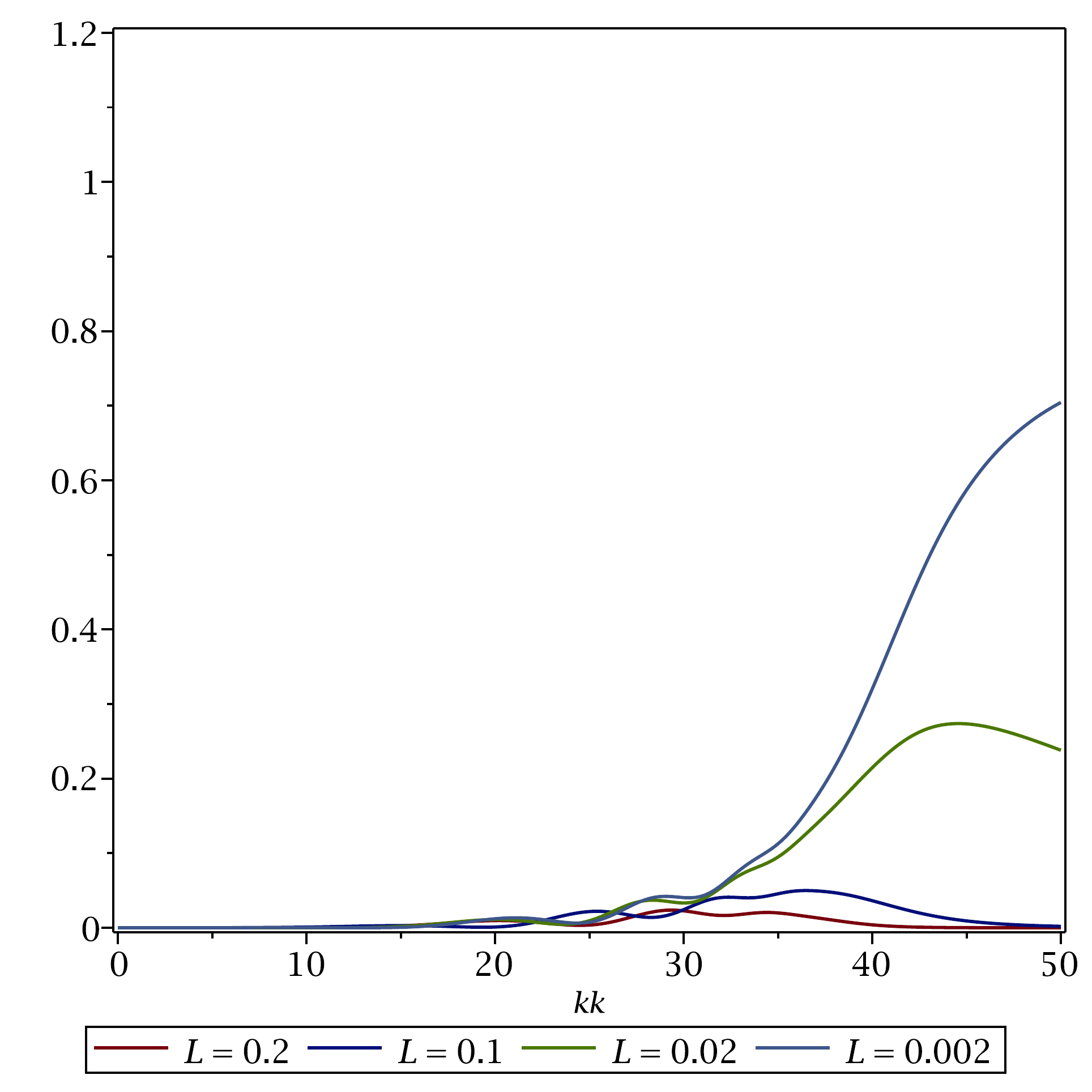}
  \caption{Optimized Schwarz convergence factor for Helmholtz with
      Robin outer boundary conditions all around and Taylor transmission
      conditions for different overlap sizes: from top left to bottom
      right $\omega=10,20,30,40$.}  \label{SchwarzHelmholtzRRLsmall}
\end{figure}
the convergence factors for four different overlap sizes
$L=0.2,0.1,0.02,0.002$ and Helmholtz frequencies $\omega=10,20,30,40$.
We see immediately the tremendous improvement by having now also Robin
conditions on the top and bottom in the Helmholtz equation: the
Schwarz method converges for all overlap sizes $L$, even if the
Helmholtz frequency is increaseing, there is no requirement any more
for the overlap $L$ to be small. We do currently not yet have an
analysis of this behavior, but we show in Table
\ref{TabTaylorRRHelmholtz}
  \begin{table}
    \centering
    \begin{tabular}{|c|c|c|}
      \hline
      $L$ & $\max_{k}|\rho|$ & $1-\max_{k}|\rho|$ \\\hline
   0.1000000000 & 0.0471851340 & 0.9528148659 \\
   0.0100000000 & 0.3360508362 & 0.6639491637 \\
   0.0010000000 & 0.7854583112 & 0.2145416887 \\
   0.0001000000 & 0.9537799195 & 0.0462200804 \\
   0.0000100000 & 0.9913446318 & 0.0086553681 \\
   0.0000010000 & 0.9984392375 & 0.0015607624 \\
   0.0000001000 & 0.9997213603 & 0.0002786396 \\
   0.0000000100 & 0.9999503928 & 0.0000496071 \\
   0.0000000010 & 0.9999911753 & 0.0000088246 \\
   0.0000000001 & 0.9999984305 & 0.0000015694 \\\hline
    \end{tabular}
    \caption{Maximum of the convergence factor for Taylor transmission
      conditions with the same Robin outer boundary conditions all around
      and $\omega=100$, for decreasing overlap size $L$.}
    \label{TabTaylorRRHelmholtz}    
  \end{table}
  the convergence factor for $\omega=100$ when $L$ goes to zero for a
  symmetric subdomain decomposition, $X_2^l=\frac{1}{2}-L/2$,
  $X_1^r=\frac{1}{2}+L/2$, and one can clearly see the much better
  asymptotic performance due to the Robin boundary conditions on top
  and bottom, compared to the waveguide case shown in Table
  \ref{TabTaylorHelmholtz}. We can also numerically observe the
  dependence of the convergence factor in this case to be
  \begin{equation}\label{TaylorRhoRRHelmholtz}
    \max_k|\rho(k,\omega,L)|= 1 - O(L^{3/4})
  \end{equation}
  as we show in Figure \ref{HelmholtzTaylorRRAsymptPlot}
  \begin{figure}
    \centering
    \includegraphics[width=0.48\textwidth]{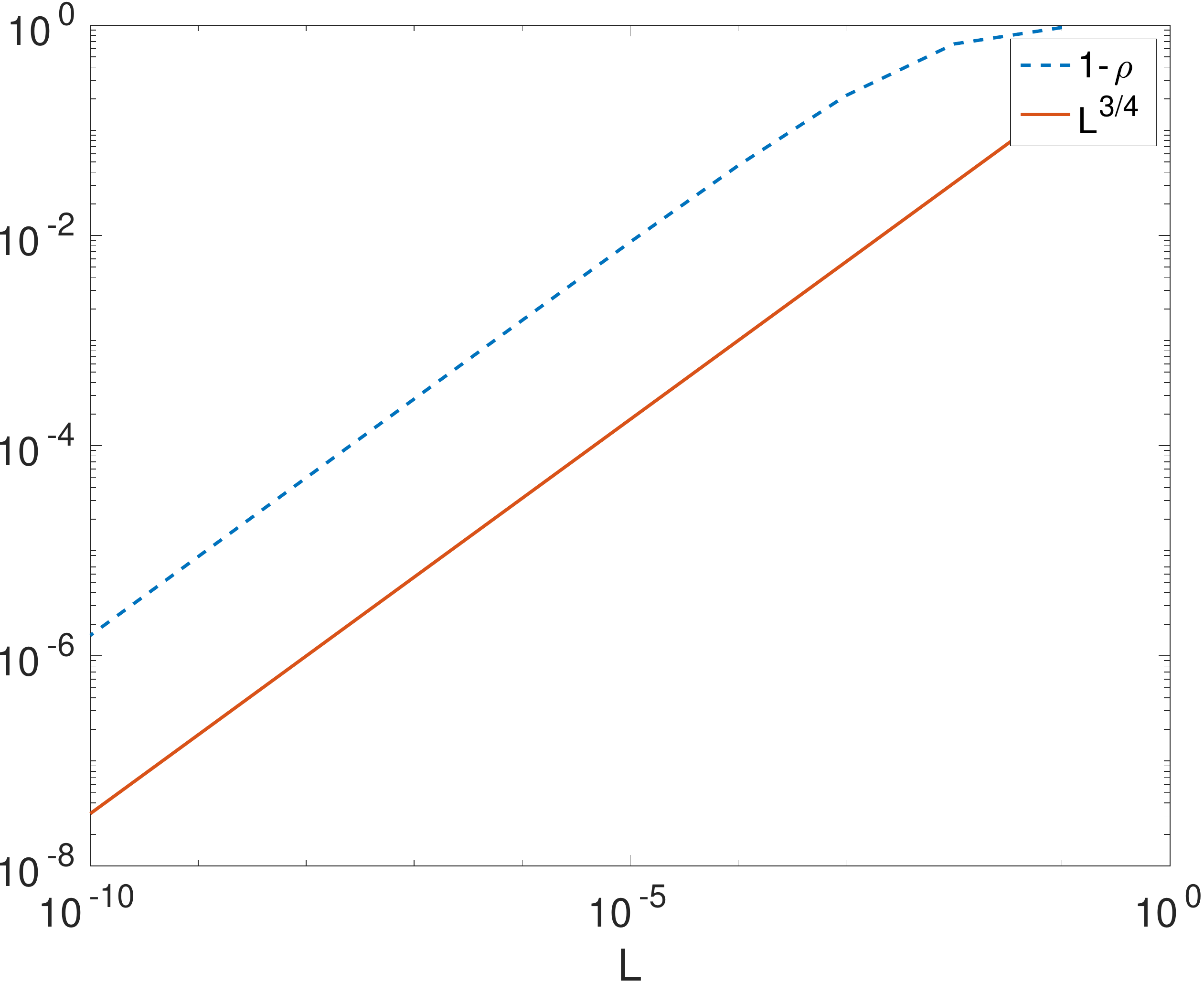}
    \includegraphics[width=0.48\textwidth]{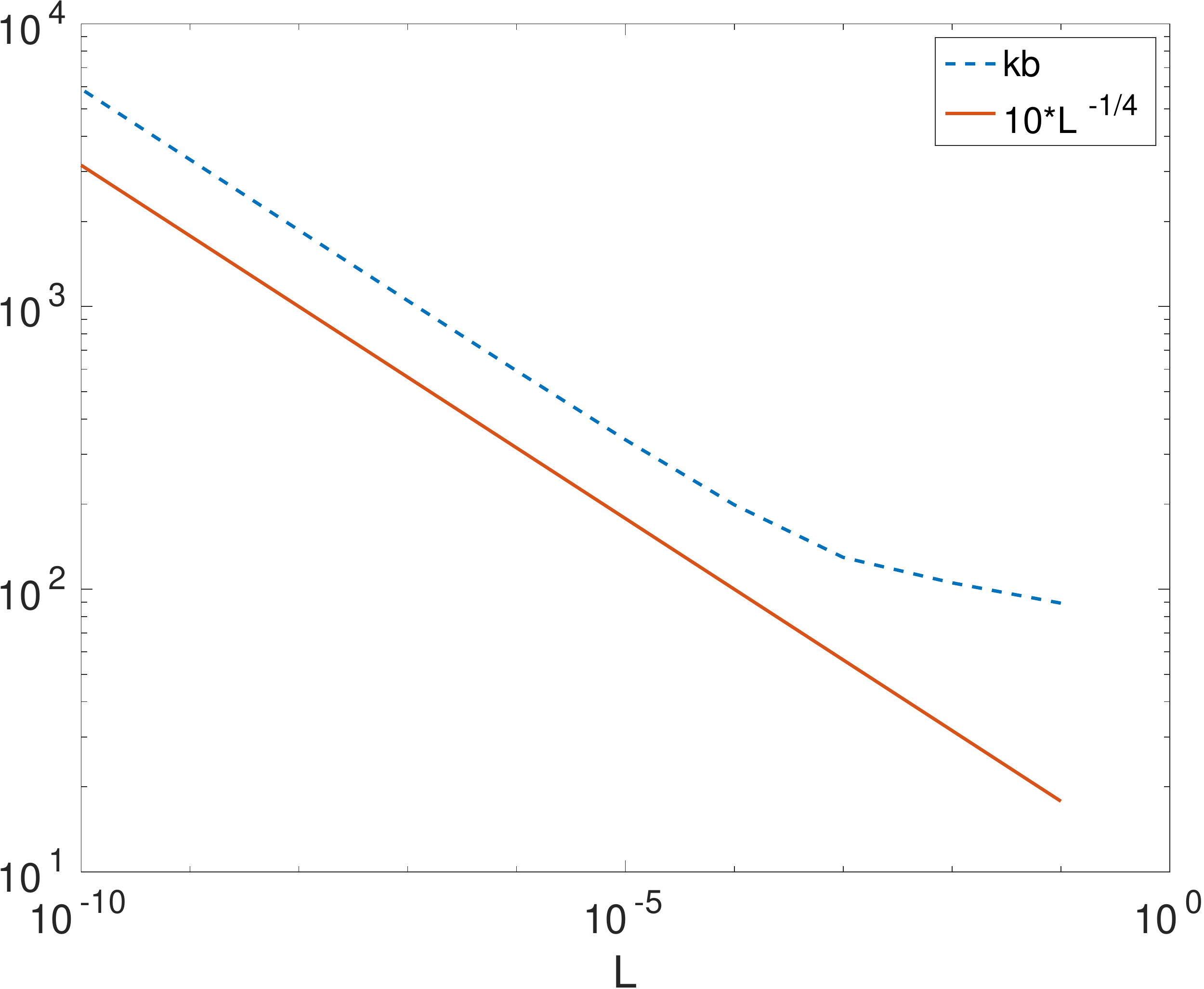}
    \caption{Asymptotic behavior of the convergence factor with Robin
      conditions all around (left) and of the maximum location $\bar{k}$
       (right).}
  \label{HelmholtzTaylorRRAsymptPlot}
  \end{figure}
  on the left. On the right, we show
  how the maximum location $\bar{k}= O(L^{-1/4})$ is increasing now.

We next show that an optimized choice of the complex transmission
  parameters $p_1^r,p_2^l\in\mathbb{C}$ can still further improve the
  convergence behavior with Robin conditions all around the global
  domain. A numerical optimization minimizing the maximum of the
  convergence factor gives the results shown in Figure
  \ref{SchwarzHelmholtzOptRRLsmall}.
\begin{figure}
  \centering
  \includegraphics[width=0.48\textwidth]{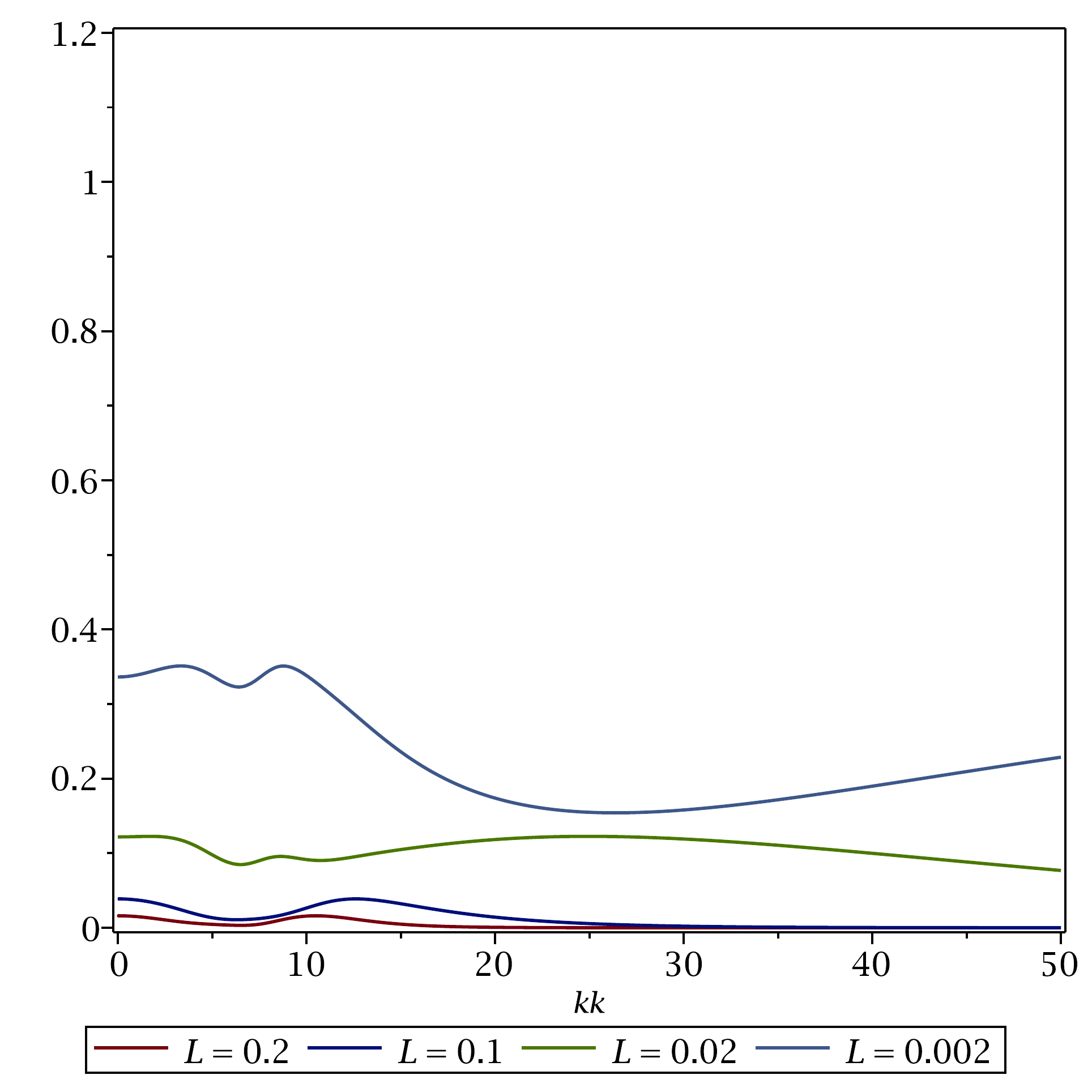}
  \includegraphics[width=0.48\textwidth]{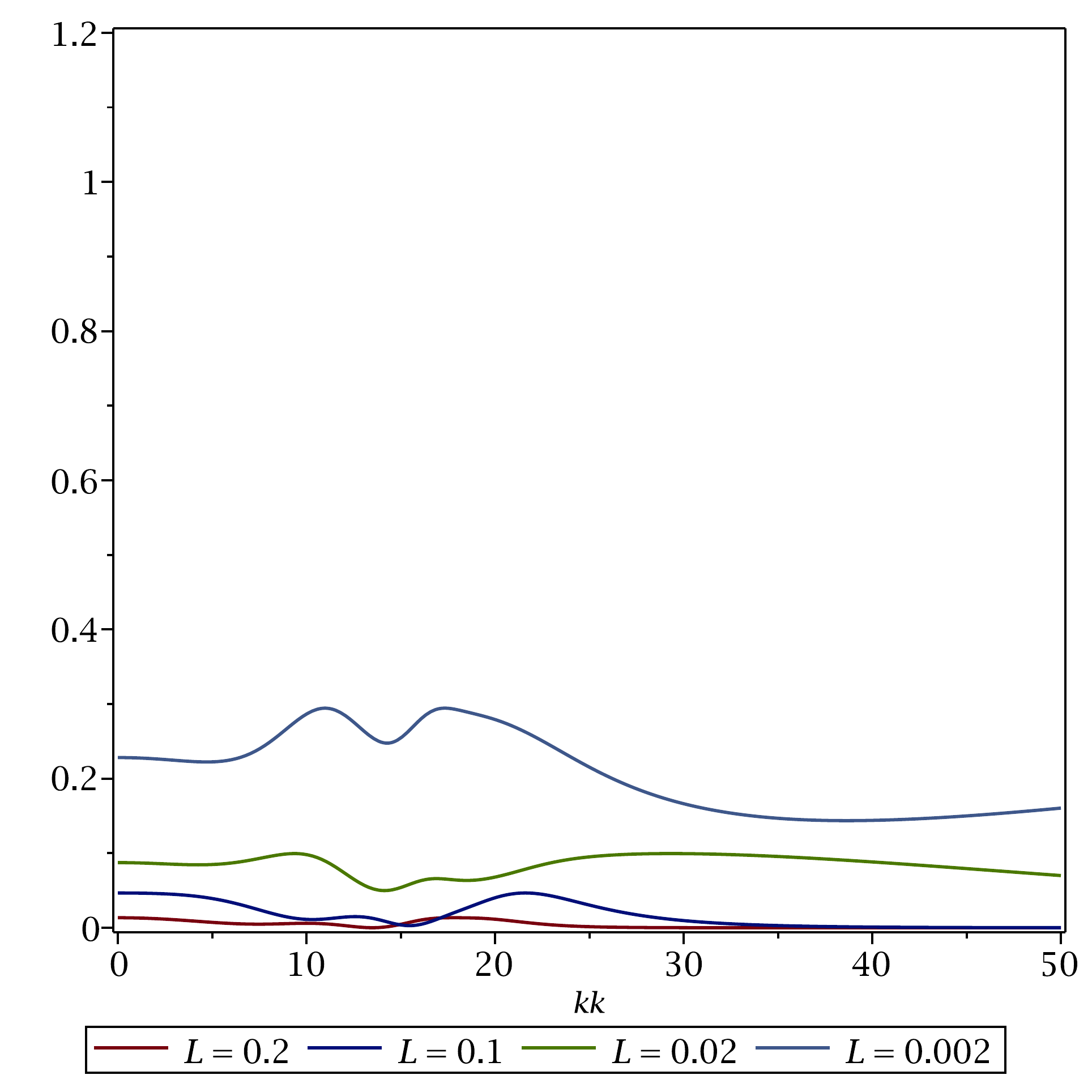}
  \includegraphics[width=0.48\textwidth]{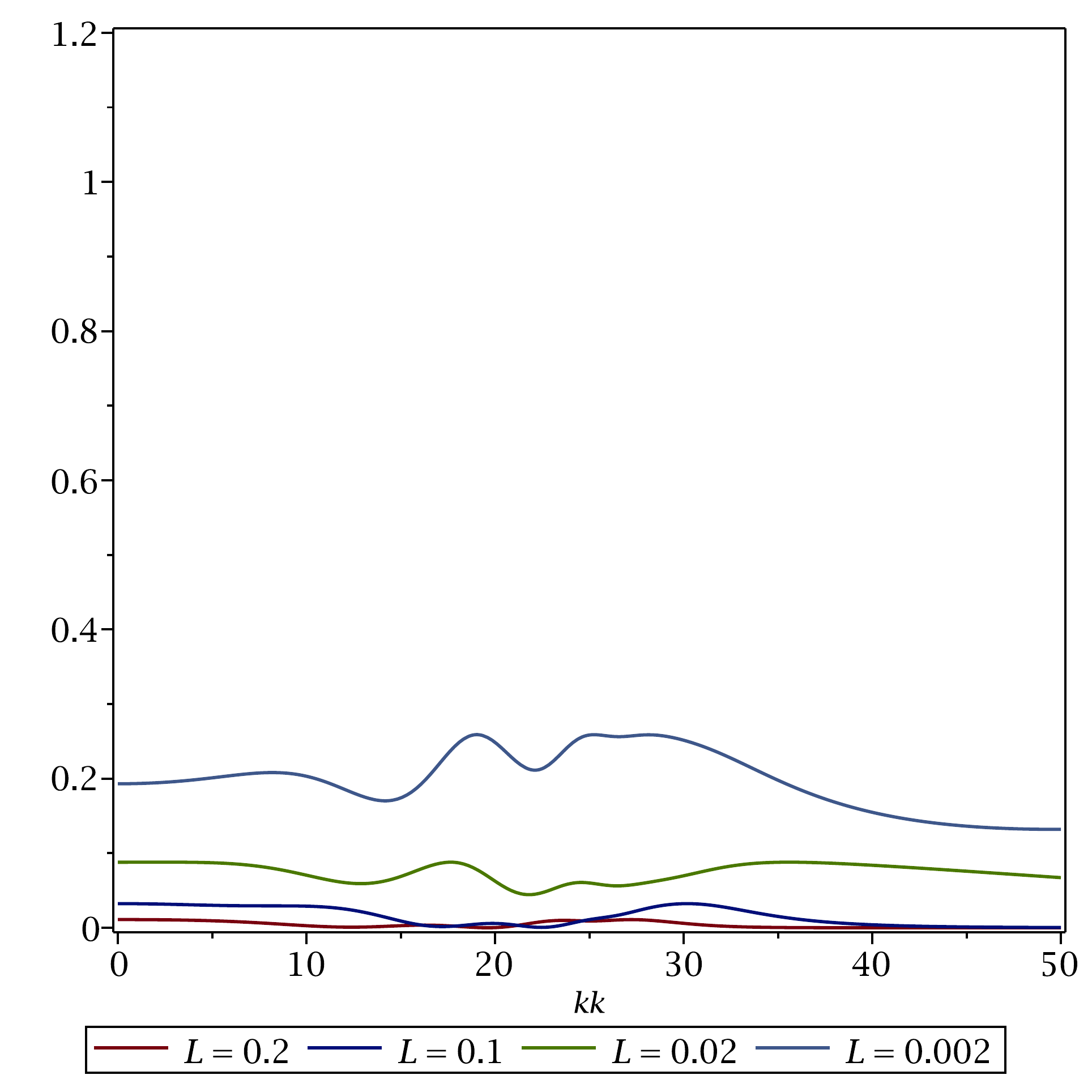}
  \includegraphics[width=0.48\textwidth]{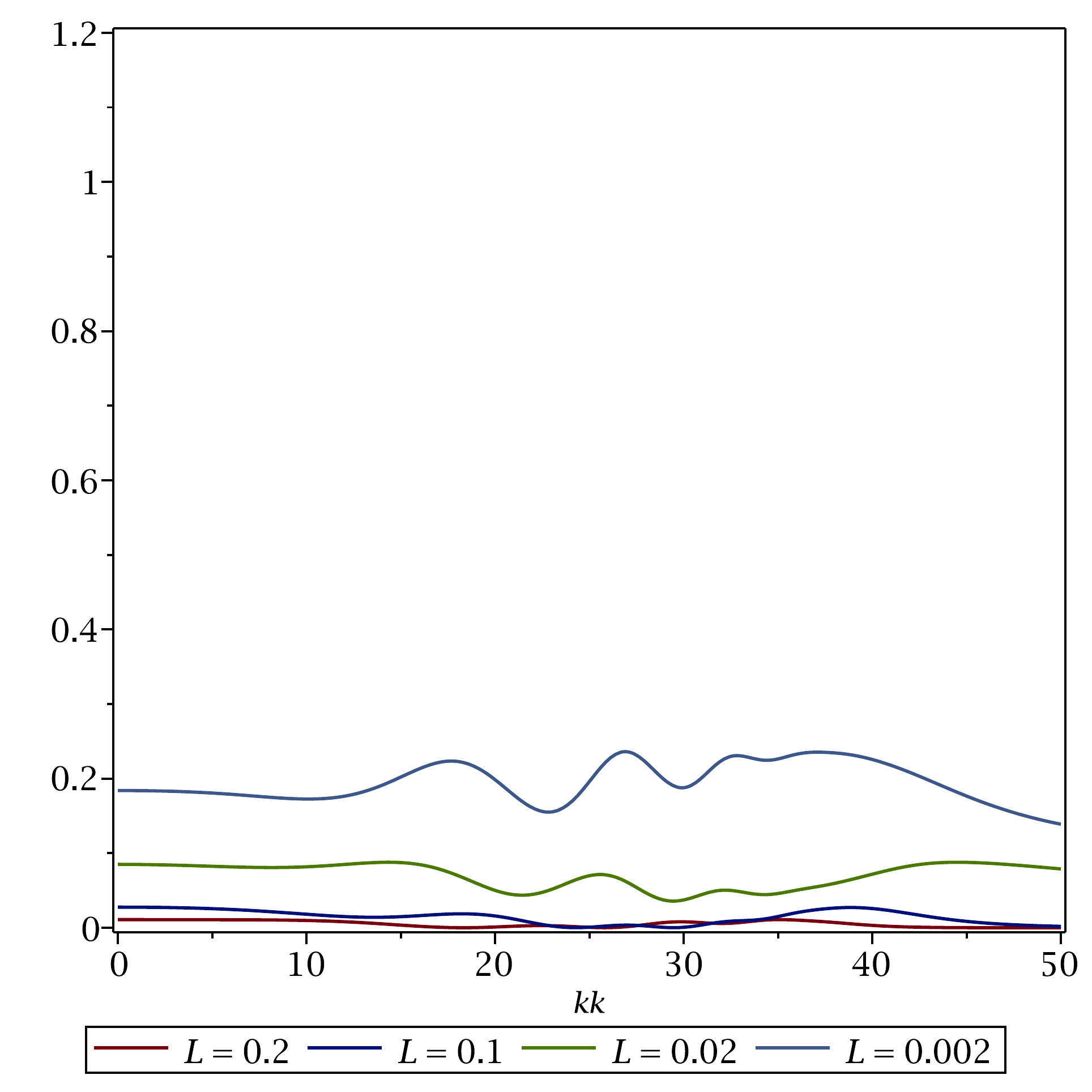}
  \caption{Optimized Schwarz convergence factor for Helmholtz with
      Robin outer boundary conditions all around and optimized complex
      Robin transmission conditions for different overlap sizes: from
      top left to bottom right
      $\omega=10,20,30,40$.}  \label{SchwarzHelmholtzOptRRLsmall}
\end{figure}
Comparing with the corresponding Figure \ref{SchwarzHelmholtzRRLsmall}
for the Taylor transmission conditions, we see that again much better
contraction factors can be obtained using the optimized transmission
conditions. To investigate this further, we show in Table
  \ref{TabOptRRHelmholtz} the best contraction factor that can be
  obtained for $\omega=100$, as we did in Table
  \ref{TabTaylorRRHelmholtz} for the Taylor transmission conditions, and
  we also show the value of the optimized complex transmission
  condition parameter.
\begin{table}
    \centering
    \begin{tabular}{|c|c|c|c|}
      \hline
      $L$ & $p_1^r=p_2^l$ & $\max_{k}|\rho|$ & $1-\max_{k}|\rho|$ \\\hline
   0.1000000000 &      0.0002+\I    82.4166 & 0.057800  & 0.942200  \\
   0.0100000000 &      0.0003+\I    50.8853 & 0.292866  & 0.707134  \\
   0.0010000000 &     47.7606+\I    79.4621 & 0.301674  & 0.698326  \\
   0.0001000000 &    120.2543+\I   142.8230 & 0.538849  & 0.461151  \\
   0.0000100000 &    266.8504+\I   287.5608 & 0.746725  & 0.253275  \\
   0.0000010000 &    578.4744+\I   609.4633 & 0.872807  & 0.127193  \\
   0.0000001000 &   1247.9382+\I  1308.2677 & 0.938763  & 0.061237  \\
   0.0000000100 &   2689.3694+\I  2816.3510 & 0.971090  & 0.028910  \\
   0.0000000010 &   5794.4274+\I  6066.6093 & 0.986475  & 0.013525  \\
   0.0000000001 &  12483.8801+\I 13069.6327 & 0.993699  & 0.006301  \\\hline
    \end{tabular}
    \caption{Best complex parameter choice and maximum of the
      convergence factor for optimized complex Robin transmission
      conditions with outer Robin boundary conditions
      all around and $\omega=100$, for
      decreasing overlap size $L$.}
    \label{TabOptRRHelmholtz}    
\end{table}
We see that the optimization leads to a method which again degrades much
less, compared to the Taylor transmission conditions, and the best choice
is the largest overlap when there are Robin conditions all around the
domain. There are currently no analytical asymptotic formulas for this optimized choice
either, but plotting the results from Table
\ref{TabOptRRHelmholtz} in Figure \ref{HelmholtzOO0RRAsymptPlot}
\begin{figure}
  \centering
  \includegraphics[width=0.48\textwidth]{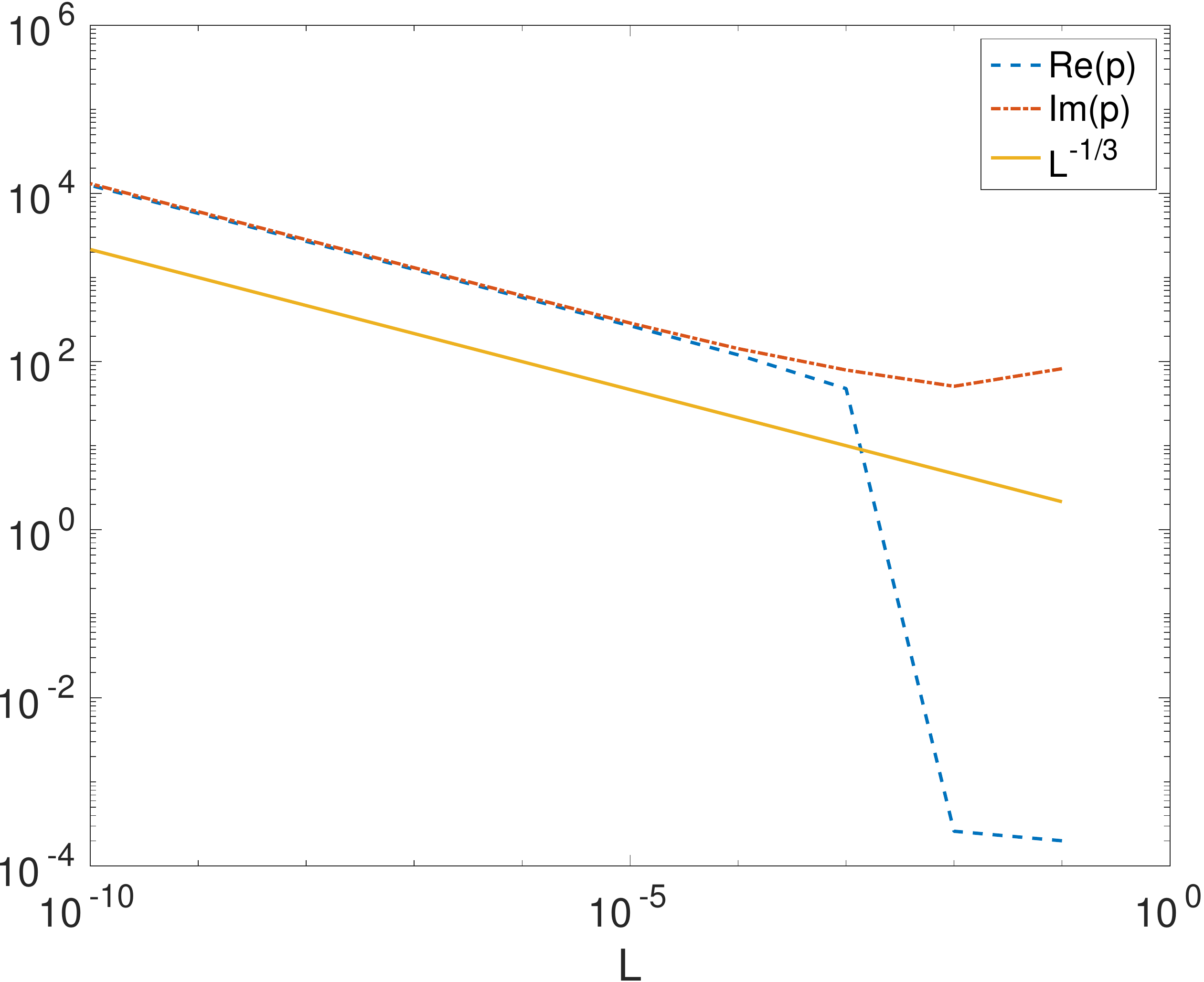}
  \includegraphics[width=0.48\textwidth]{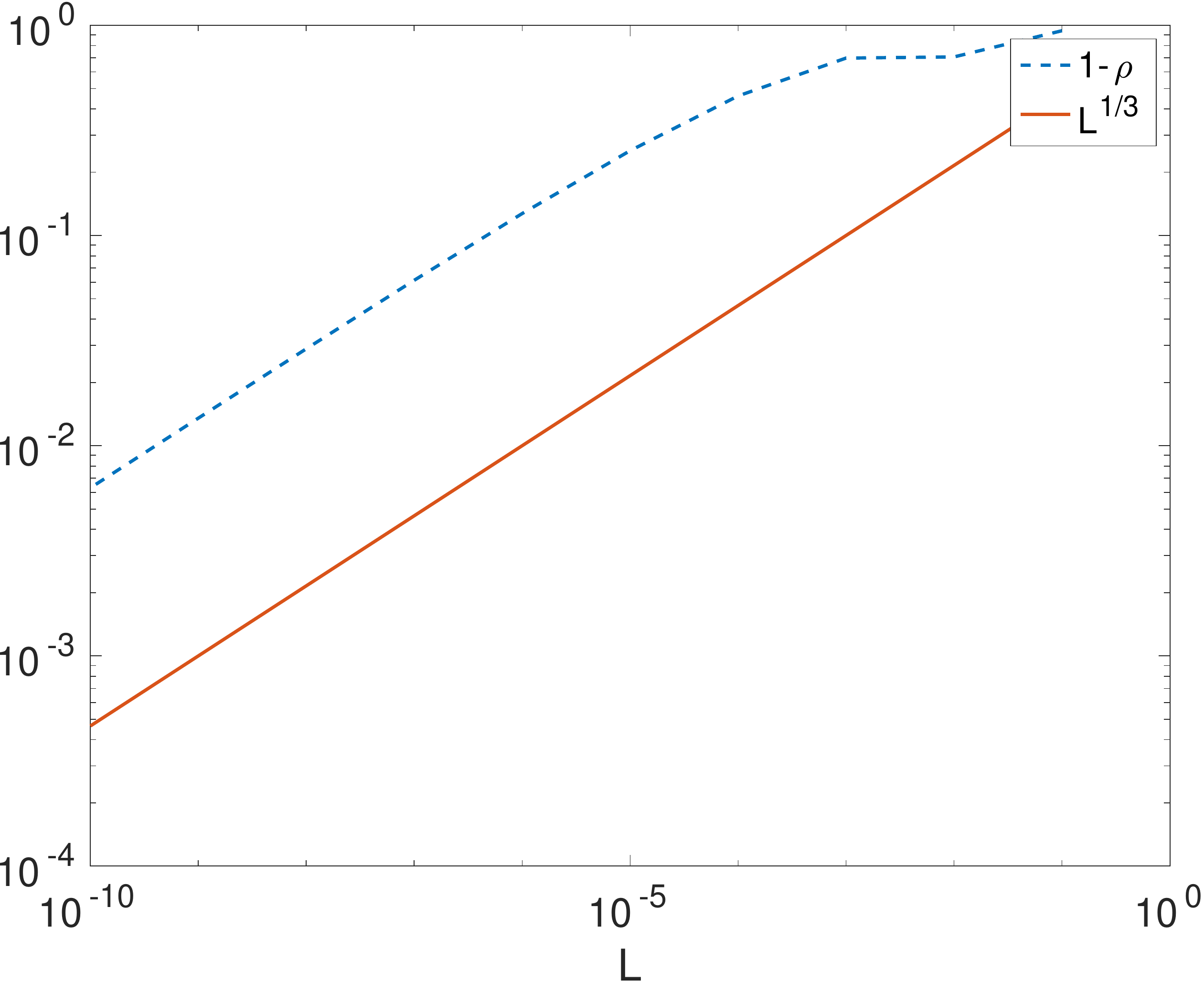}
  \caption{Asymptotic behavior of the optimized complex Robin parameter
    and associated convergence factor distance from $1$ when there are Robin conditions
    all around the domain.}
  \label{HelmholtzOO0RRAsymptPlot}
\end{figure}
cleary shows that the optimized choice behaves asymptotically for
small overlap $L$ again as
\begin{equation}
  \thickmuskip=0.6mu
  \medmuskip=0.2mu
  \thinmuskip=0.1mu
  \scriptspace=0pt
  \nulldelimiterspace=-0.1pt
  p_1^r\sim(C_1^r+\I \tilde{C}_1^r)L^{-1/3},\quad
  p_2^l\sim(C_2^l+\I \tilde{C}_2^l)L^{-1/3},\quad
  \max_{k}|\rho|= 1-O(L^{1/3}),
\end{equation}
like in the waveguide case with Dirichlet or Neumann conditions on top and bottom.
This is still much better than for the Taylor transmission
conditions that gave $\max_{k}|\rho|= 1-O(L^{3/4})$.
  
From the two subdomain analysis in this section, we have already
  learned many things for Schwarz methods by domain truncation, or
  optimized Schwarz methods: for the screened Laplace problem, outer
  boundary conditions are not influencing convergence very much, and
  complete analysis is available for the convergence of these methods,
  also with optimized formulas for Robin transmission
  conditions. Compared to classical Schwarz methods with contraction
  factors $1-O(L)$ when the overlap $L$ becomes small, Taylor
  transmission conditions give a contraction factor $1-O(\sqrt{L})$,
  and optimized Robin transmission conditions give
  $1-O(L^{1/3})$. This is very different for the Helmholtz case, where
  the performance of Schwarz methods is very much dependent on the
  outer boundary conditions imposed on the domain on which the problem
  is considered. Radiation boundary conditions are very important for
  Schwarz methods to function here: in a waveguide setting with Robin
  conditions on the left and right, classical Schwarz methods still do
  not work, but Taylor transmission conditions of Robin type, the
  simplest absorbing boundary conditions, can then lead to convergent
  Schwarz methods, provided the overlap is small enough, with
  convergence factor $1-O(L)$ up to a logarithmic term. Optimized
  Robin transmission conditions lead again to $1-O(L^{1/3})$, as in
  the screened Laplace case. If there are radiation conditions of
  Robin type all around, using these same conditions also as
  transmission conditions leads to contraction factors of the form
  $1-O(L^{3/4})$, but now the methods also work with larger overlap,
  and optimized Robin conditions give $1-O(L^{1/3})$ with numerically
  observed much better constants, asymptotically comparable to the
  screened Laplace case!

%%%%%%%%%%%%%%%%%%%%%%%%%%%%%%%%%%%%%%%%%%%%%%%%%%%%%%%%%%%%%%%%%%%%%%%%%%%%%%%%
%
%
\section{Many subdomain analysis}
%
%
%%%%%%%%%%%%%%%%%%%%%%%%%%%%%%%%%%%%%%%%%%%%%%%%%%%%%%%%%%%%%%%%%%%%%%%%%%%%%%%%

We now investigate the performance of optimized Schwarz methods,
  or equivalently Schwarz methods by domain truncation, when more than
  two subdomains are used. Since we have seen in Section \ref{2SubSec}
  that absorbing boundary conditions (ABCs) play an important role for
  the Helmholtz case, and perfectly matched layers (PMLs) are another
  technique to truncate domains, we now study the convergence of
  Schwarz methods for a sligthly more general problem where the
  operators in the $x$ and $y$ direction are written separately, so
  that PML modifications can be introduced, namely
\begin{equation}\label{eqprb}
  \begin{aligned}
    (\mathcal{L}_x + \mathcal{L}_y)u + \eta u &= f & &\text{ on }\Omega=[0,X]\times[0,Y],\\
    \mathcal{B}^{l}u&=0 & &\text{ on } \{0\}\times [0,Y],\\
    \mathcal{B}^ru&=0 & &\text{ on } \{X\}\times[0,Y],\\
    \mathcal{B}^bu&=0 & &\text{ on } [0,X]\times \{0\},\\
    \mathcal{B}^tu&=0 & &\text{ on } [0,X]\times \{Y\},
  \end{aligned}
\end{equation}
where $\mathcal{L}_{*}$ is a linear partial differential operator
about $*$, $\eta$ is a constant which we will choose to handle both
  the screened Laplace and the Helmholtz case, the unknown $u$ and
the data $f$ are functions on $\Omega$, and $\mathcal{B}^{*}$'s are
linear trace operators.  For Schwarz methods, the domain is decomposed
first into $N$ subdomains $[X_{j}^{l},X_j^{r}]\times[0,Y]$ with
$X_j^{l}=(j-1)H$, $X_j^r=jH+2L$ and $H=(X-2L)/N$, as shown for an
  example in Figure \ref{1Dand2DDecompositionFig} on the right. The
restriction of \R{eqprb} onto the subdomains $\Omega_j$ are
solved with the solutions $u_j$ to satisfy the transmission
conditions
\begin{equation}\label{eqtc}
  \mathcal{B}_j^{l,r}u_j=\mathcal{B}_j^{l,r}u_{j\mp1} \text{ on } \{X_{j}^{l,r}\}\times[0,Y],
\end{equation}
except on the original boundary $\{0,X\}\times[0,Y]$.

\begin{remark}
  We note that if $u_j$ satisfies the restriction of \R{eqprb} onto
  $\Omega_j$, by linearity of \R{eqprb}, the partial differential
  equation for the error $u-u_j$ on $\Omega_j$ is homogeneous (with
  $f=0$) and invariant under the transformation $(x, y,\eta) \mapsto
  (\frac{x}{Y}, \frac{y}{Y}, Y^2\eta)$. We further assume that the
  boundary and transmission conditions have the same invariance. Then,
  we can let $Y=1$ without loss of generality for the convergence
  analysis.
\end{remark}

The generalised Fourier frequencies $\pm\xi$ correspond to the
square-roots of the Sturm-Liouville eigenvalues $\xi^2$,
\begin{equation}\label{eqsturm}
  \begin{aligned}
    \mathcal{L}_y\phi &= \xi^2\phi &  &\text{ on }[0,Y],\\
    \mathcal{B}^b\phi &= 0 & &\text{ at }\{0\},\\
    \mathcal{B}^t\phi &= 0 & &\text{ at }\{Y\}.
  \end{aligned}
\end{equation}
We assume that the eigenfunctions $\{\phi(y,\xi_n)\}$, $\xi\in
\{\xi_0,\xi_1,\xi_2,\ldots\}\subset\mathbb{C}$ allow the solution of
\R{eqprb} to be represented as
$\sum_{n=0}^{\infty}u(x,\xi_n)\phi(y,\xi_n)$\footnote{We still denote
  the Fourier transformed quantities for simplicity by the same
  symbols $u$, $f$, $\mathcal{B}^l$ and $\mathcal{B}^r$ to avoid a
  more complicated notation.}, with $u(x,\xi)$ satisfying the problem
\begin{equation}\label{eqode}
  \begin{aligned}
    \mathcal{L}_xu + (\xi^2+\eta)u &= f\quad & &\mbox{on }[0,X],\\
    \mathcal{B}^lu&=0 \quad & &\mbox{at }x=0,\\
    \mathcal{B}^ru&=0 \quad & &\mbox{at }x=X.
  \end{aligned}  
\end{equation}
Let $g_{j}^l:=\mathcal{B}_j^lu_j$ at $x=X_j^l$ and $g_{j}^r:=\mathcal{B}_j^ru_j$ at $x=X_j^r$. To
rewrite \R{eqtc} in terms of the interface data $g_2^l,\ldots, g_N^l, g_1^r,\ldots, g_{N-1}^r$, \ie\ in substructured form, we
define the interface-to-interface operators (see also Figure~\ref{figi2i})
\[
  \thickmuskip=0.6mu
  \medmuskip=0.2mu
  \begin{aligned}
  a_j:&\left(\ell_j\mbox{\ at }x=X_j^l\right)\rightarrow
  \left(\mathcal{B}_{j+1}^lv_j\mbox{\ at }x=X_{j+1}^l\right)\ 
  \mbox{with $v_j$ solving}\\
  &\mathcal{L}_xv_j+(\xi^2+\eta)v_j=0 \text{ in $(X_j^l,X_j^r)$},\;\,
  \mathcal{B}_j^lv_j=\ell_j\text{ at $x=X_j^l$},\;\,
  \mathcal{B}_j^rv_j=0\text{ at $x=X_j^r$},\\
\end{aligned}
\]
\[
  \thickmuskip=0.6mu
  \medmuskip=0.2mu
  \begin{aligned}
  b_j:&\left(\gamma_j\mbox{\ at }x=X_j^r\right)\rightarrow
  \left(\mathcal{B}_{j+1}^lv_j\mbox{\ at }x=X_{j+1}^l\right)\ 
  \mbox{with $v_j$ solving}\\
  &\mathcal{L}_xv_j+(\xi^2+\eta)v_j=0\text{ in $(X_j^l,X_j^r)$},\;\,
  \mathcal{B}_j^lv_j=0\text{ at $x=X_j^l$},\;\,
  \mathcal{B}_j^rv_j=\gamma_j\text{ at $x=X_j^r$},\\
  \end{aligned}
\]
\[
  \thickmuskip=0.6mu
  \medmuskip=0.2mu
  \begin{aligned}
  c_j:&\left(\gamma_j\mbox{\ at }x=X_j^r\right)\rightarrow
  \left(\mathcal{B}_{j-1}^rv_j\mbox{\ at }x=X_{j-1}^r\right)\ 
  \mbox{with $v_j$ solving}\\
  &\mathcal{L}_xv_j+(\xi^2+\eta)v_j=0\text{ in $(X_j^l,X_j^r)$},\;\,
  \mathcal{B}_j^lv_j=0\text{ at $x=X_j^l$},\;\,
  \mathcal{B}_j^rv_j=\gamma_j\text{ at $x=X_j^r$},\\
\end{aligned}
\]
\[
  \thickmuskip=0.6mu
  \medmuskip=0.2mu
 \begin{aligned}
  d_j:&\left(\ell_j\mbox{\ at }x=X_j^l\right)\rightarrow
  \left(\mathcal{B}_{j-1}^rv_j\mbox{\ at }x=X_{j-1}^r\right)\ 
  \mbox{with $v_j$ solving}\\
  &\mathcal{L}_xv_j+(\xi^2+\eta)v_j=0\text{ in $(X_j^l,X_j^r)$},\;\,
  \mathcal{B}_j^lv_j=\ell_j\text{ at $x=X_j^l$},\;\,
  \mathcal{B}_j^rv_j=0\text{ at $x=X_j^r$},
  \end{aligned}
\]
where $\mathcal{B}_1^l:=\mathcal{B}^l$ and $\mathcal{B}_N^r:=\mathcal{B}^r$. Using these operators,
we can rewrite \R{eqtc} as a linear system, namely
\begin{figure}
  \centering
  \includegraphics[scale=.25]{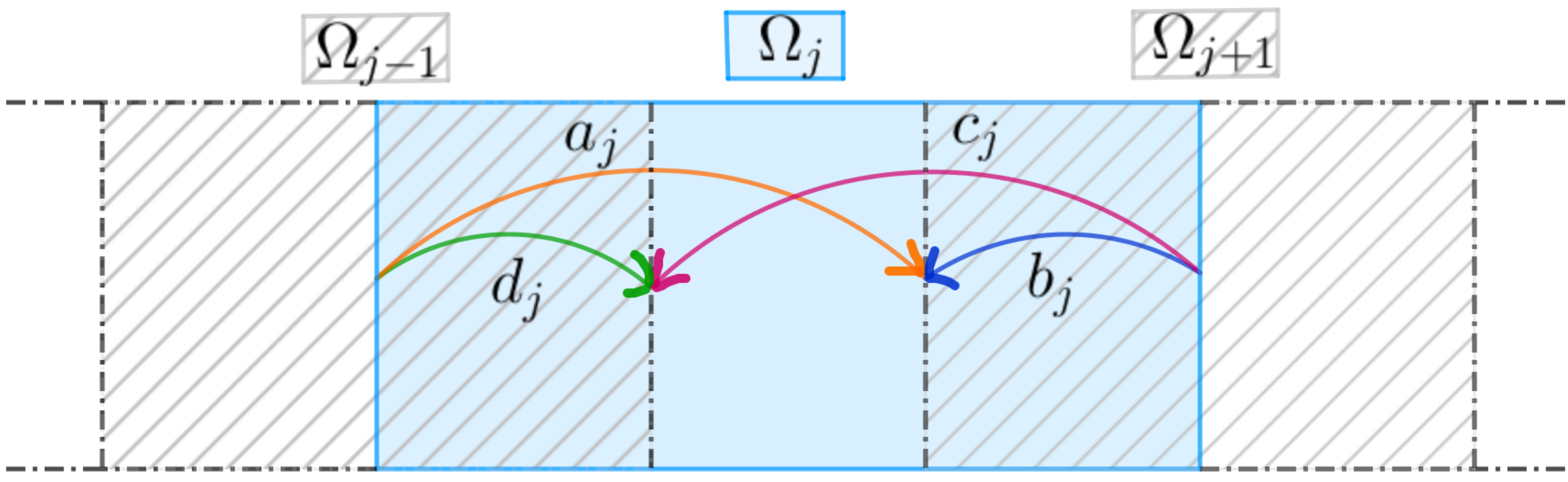}
  \caption{Illustration of the interface-to-interface
    operators.}\label{figi2i}
\end{figure}
\begin{equation}\label{eq:gp}
  {\footnotesize
    \left[
    \begin{array}{cccc|cccc}
      1\vphantom{g_1^l} & & & & -b_1 & & &\\
      -a_2 & 1\vphantom{g_2^l\ddots} & & & & -b_2 & &\\
        & \ddots & \ddots\vphantom{\vdots} & & & & \ddots &\\
        & & -a_{N-1} & 1\vphantom{g_N^l} & & & & -b_{N-1}\\
      \hline
      -d_2 & & & & 1\vphantom{g_2^r} & -c_2 & &\\
        & -d_3 & & &  & 1\vphantom{g_3^r} & \ddots & \\
        & & \ddots & & &  & \ddots & -c_{N-1}\vphantom{\vdots} \\
      & & & -d_{N} & & & & 1\vphantom{g_{N-1}^r}
    \end{array}
  \right]
  \left[
    \begin{array}{c}
      g_2^l\\ g_3^l\vphantom{\ddots} \\ \vdots \\ g_N^l\\\hline
      g_1^r\\ g_2^r\vphantom{\ddots} \\ \vdots \\ g_{N-1}^r
    \end{array}
  \right]=
  \left[
    \begin{array}{c}
      \tau_2^l\\ \tau_3^l\vphantom{\ddots} \\ \vdots \\ \tau_N^l\\\hline
      \tau_1^r\\ \tau_2^r\vphantom{\ddots} \\ \vdots \\ \tau_{N-1}^r
    \end{array}
  \right],
  }
\end{equation}
which we denote by
\begin{equation}\label{eq:bp}
  \begin{bmatrix}
    I-A & -B\\
    -D & I-C    
  \end{bmatrix}
  \begin{bmatrix}
    g^l\\ g^r
  \end{bmatrix}=
  \begin{bmatrix}
    \tau^l\\ \tau^r
  \end{bmatrix},
\end{equation}
where $\tau_j^l:=\mathcal{B}_j^lv_{j-1}$, $\tau_j^r:=\mathcal{B}_j^rv_{j+1}$ with $v_j$ satisfying
\[
  \begin{aligned}
    \mathcal{L}_xv_j+(\xi^2+\eta)v_j&=f \quad & &\mbox{on }[X_j^l, X_j^r],\\
    \mathcal{B}_j^lv_j&=0\quad & &\mbox{at $x=X_j^l$},\\
    \mathcal{B}_j^rv_j&=0\quad & &\mbox{at $x=X_j^r$}.
  \end{aligned}  
\]
The double sweep Schwarz method amounts to a block Gauss-Seidel iteration for \R{eq:bp}: given an
initial guess $g^{r,0}$ of $g^r$, we compute for iteration index $m=0,1,\ldots$
\begin{equation}\label{eq:gs}
  g^{r,m+1}:=(I-C)^{-1}\left[\tau^r+D(I-A)^{-1}(\tau^l+Bg^{r,m})\right].
\end{equation}  
We denote by $\epsilon^{r,m}:=g^{r}-g^{r,m}$ the error, which then by \R{eq:gs}
satisfies a recurrence relation with iteration matrix $T_{\rm{DOSM}}$,
\begin{equation}\label{eq:Tdosm}
  \epsilon^{r,m+1}=T_{\rm{DOSM}}\epsilon^{r,m}:=(I-C)^{-1}D(I-A)^{-1}B\epsilon^{r,m}.
\end{equation}
The Jacobi iteration for \R{eq:gp} gives the parallel Schwarz method with the recurrence relation
\begin{equation}\label{eq:Tposm}
  \epsilon^{m+1}=T_{\rm{POSM}}\epsilon^{m}:=
  \begin{bmatrix}
    A & B\\
    D & C
  \end{bmatrix}
  \begin{bmatrix}
    \epsilon^{l,m}\\ \epsilon^{r,m}
  \end{bmatrix}.
\end{equation}
\begin{remark}
  The block Jacobi iteration for \R{eq:bp} leads to the X (cross) sweeps
  \cite{StolkImproved,Zepeda,ZD} starting from the first and last subdomains simultaneously, namely,
  \begin{equation}\label{eq:Tcosm}
    \epsilon^{m+1}=T_{\rm{COSM}}\epsilon^{m}:=
    \begin{bmatrix}
      O & (I-A)^{-1}B\\
      (I-C)^{-1}D & O
    \end{bmatrix}
    \begin{bmatrix}
      \epsilon^{l,m}\\ \epsilon^{r,m}
    \end{bmatrix}.
  \end{equation}
  It is easy to see that $T_{\rm{COSM}}^2$ has the same spectra as $T_{\rm{DOSM}}$ does.
\end{remark}

\begin{remark}
  For two subdomains $N=2$, \R{eq:gp} reduces to
  \[
    \begin{bmatrix}
      1 & -b_1\\ -d_2 & 1
    \end{bmatrix}
    \begin{bmatrix}
      g_2^l\\ g_1^r
    \end{bmatrix}=
    \begin{bmatrix}
      \tau_2^l\\ \tau_1^r
    \end{bmatrix}.
  \]
  Then, $T_{\rm{DOSM}}=b_1d_2$ is exactly the two-subdomain
  convergence factor \R{RhoOSM}. On the one hand, the influence of the
  outer left/right boundary conditions $\mathcal{B}^{l,r}u=0$ is
  totally contained in $b_1$ and $d_N$. For fast convergence, the
  $b_j$ and $d_j$ are the only entries that can be made arbitrarily
  small by approximating the transparent transmission conditions. On
  the other hand, the influence of $N$ is mainly through $(I-C)^{-1}$,
  $(I-A)^{-1}$ in \R{eq:Tdosm} or the nilpotent $A$, $C$ in
  \R{eq:Tposm}, not directly related to the diagonal $B$, $D$. This
  explains why and how the two-subdomain analysis from Section
    \ref{2SubSec} is useful.
\end{remark}

Let $s:=\sqrt{\xi^2+\eta}$ be the symbol of the square-root operator $\sqrt{\mathcal{L}_y+\eta}$
(also known as the half-plane Dirichlet-to-Neumann operator, see footnote~\ref{footnoteDtN}). Assume
the boundary operator $\mathcal{B}_{*}^{l,r}$ has the symbol
$\mp q_{*}^{l,r}(\xi)\partial_x + p_{*}^{l,r}(\xi)$. Then, one can find the symbols of the operators
$a_j$, $b_j$, $c_j$ and $d_j$ in \R{eq:gp} so that the iteration matrices $T_{\rm{POSM}}$ and
$T_{\rm{DOSM}}$ become symbol matrices; see \eg~\cn{NN97}, \cn{bootland2022analysis},
\cn{GZDD25}. Let
\[
  \mathfrak{A} = \begin{bmatrix} A & O\\ O & O\end{bmatrix},
  \mathfrak{B} = \begin{bmatrix} O & B\\ O & O\end{bmatrix},
  \mathfrak{C} = \begin{bmatrix} O & O\\ O & C\end{bmatrix},
  \mathfrak{D} = \begin{bmatrix} O & O\\ D & O\end{bmatrix}.
\]
There exist the following results in norm.

\begin{theorem}[\cn{NN97}] Let $\|\cdot\|$ be an algebra norm of matrices (so that the submultiplicative property
  $\|AB\|\le\|A\|\|B\|$ holds), with $\|I\|=1$. If for some $\xi$ we have
  \[
    {\varrho}(\xi):=\|\mathfrak{B}(\xi)\| \|\mathfrak{D}(\xi)\|
    \left(\sum_{n=0}^{N-2}\|\mathfrak{A}(\xi)\|^n\right)
    \left(\sum_{n=0}^{N-2}\|\mathfrak{C}(\xi)\|^n\right)<1,
  \]
  then the estimates
  \[
    \begin{aligned}
      \|T_{\rm{POSM}}^n(\xi)\|&\le
      C_0(\xi)\frac{1}{1-{\varrho}(\xi)}{\varrho}(\xi)^{\left[\frac{n}{2N-2}-\frac{3}{2}\right]},
      & &\forall n\ge 2N,\\
      \|T_{\rm{DOSM}}^{n}(\xi)\|&\le C_0(\xi)(1+{\varrho}(\xi)){\varrho}(\xi)^{n-1}, & &\forall
      n\ge 2,
    \end{aligned}
  \]
  hold with
  $C_0(\xi)=\left(1+\frac{{\varrho}(\xi)}{\|\mathfrak{B}(\xi)\|}\right)\left(1+\frac{{\varrho}(\xi)}{\|\mathfrak{D}(\xi)\|}+\frac{{\varrho}(\xi)}{\|\mathfrak{B}(\xi)\|\|\mathfrak{D}(\xi)\|}\right)$.
\end{theorem}
This theorem tells us the optimized Schwarz methods converge as soon as $|b_j|$, $|d_j|$ are
sufficiently small while $|a_j|, |c_j|$ remain bounded, which corresponds to having nearly
transparent conditions on the left and right boundaries of each subdomain. The upper bounds
  suggest that a double sweep iteration $T_{\mathrm{DOSM}}$ is worth $2N-2$ parallel iterations
  $T_{\mathrm{POSM}}^{2N-2}$. But since we do not know how sharp the bounds are, we can still not
answer precisely the scalability questions such as the scaling with $N\to\infty$. For that purpose,
we opt to calculate the eigenvalues of the symbol matrices numerically. In case of $T_{\rm{DOSM}}$
with real-valued symbols, a two-sided estimate is also available \cite{GZDD25}.

Let $\rho=\rho(\xi,\eta,N,H,L,\ldots)$ be the spectral radius of the symbol iteration matrix, where
the dots represent the parameters to be introduced in the transmission operator
$\mathcal{B}^{l,r}_j$.  \emph{The goal} of this section is to find the asymptotic scalings of
$\max_{\xi}\rho$ in terms of $\eta$, $N$, $H$, $L$, \etc\ One scaling is $N\to\infty$ with
$\mathcal{B}^{l,r}=\mathcal{B}_j^{l,r}$ and $\eta$, $H$, $L$ fixed (or generally all the symbol
values $a$, $b$ fixed). Another scaling is to fix $X$ and let $N\to \infty$ with $\eta$ either fixed
or growing. In this scaling, $H=(X-2L)/N\to 0$ and $L=CH^{\nu}$ ($\nu\ge 1$). To distinguish the
effects of $N\to \infty$ and $H\to 0$, it is useful to study a third scaling with $H\to 0$,
$L=CH^{\nu}$ and all the other parameters fixed. In particular, for $N=\infty$, we can see how the
subdomain size (or block size in the matrix language) influences the convergence. Also, the effect
of small overlap can be revealed with $L\to 0$ and $N$, $X$ fixed. 

For the parallel Schwarz iteration with $a_j=a$, $b_j=b$, $c_j=c$, $d_j=d$ independent of $j$ and
$a=c$, a closed formula for $\rho_{\infty}:=\lim_{N\to\infty}\rho$ has been found rigorously
\cite{bootland2022analysis}, which we shall use later. Here, we give a formal derivation when the
domain becomes $(-\infty, \infty)\times(0,Y)$ and $\Omega_j=(X_j^l,X_j^r)\times(0,Y)$ for
$j=0,\pm1,\pm2, ...$. The interface data are collected as the bi-infinite sequence
$g:=(\ldots, g_{j}^l, g_{j-1}^r, g_{j+1}^l, g_j^r, \ldots)$ with the displayed four entries indexed
by $2j-1$, $2j$, $2j+1$ and $2j+2$. We define the splitted Fourier sequences (entries displayed at
the same indices as before)
\begin{align*}
  \phi_{\mu,1}&:=(\ldots,\mathrm{e}^{\mathrm{i}\mu (2j-1)},0,\mathrm{e}^{\mathrm{i}\mu (2j+1)}, 0,\ldots),\\
  \phi_{\mu,2}&:=(\ldots,0,\mathrm{e}^{\mathrm{i}\mu (2j)},0, \mathrm{e}^{\mathrm{i}\mu (2j+2)},\ldots),
\end{align*}
with $\mu\in [-\pi,\pi]$. Then, one can find that $\mathrm{span} \{\phi_{\mu,1}, \phi_{\mu,2}\}$ is
an invariant subspace of $T_{\rm{POSM}}$.  In fact, using the definition \R{eq:Tposm}, we can
calculate
\begin{align*}
  (T_{\rm{POSM}}\phi_{\mu,1})(2j-1)&= a\mathrm{e}^{\mathrm{i}\mu(2j-3)} = a\,\mathrm{e}^{-2\mathrm{i}\mu}
                                     \cdot \phi_{\mu,1}(2j-1),\\
  (T_{\rm{POSM}}\phi_{\mu,1})(2j)&= d\mathrm{e}^{\mathrm{i}\mu(2j-1)} = d\,\mathrm{e}^{-\mathrm{i}\mu}
                                   \cdot \phi_{\mu,2}(2j),
\end{align*}
so we have
$ T_{\rm{POSM}}\phi_{\mu,1}= a\,\mathrm{e}^{-2\mathrm{i}\mu}\cdot \phi_{\mu,1} +
d\,\mathrm{e}^{-\mathrm{i}\mu} \cdot \phi_{\mu,2}.$ Therefore, we find
\begin{equation*}
  T_{\rm{POSM}}(\phi_{\mu,1},\phi_{\mu,2})= (\phi_{\mu,1},\phi_{\mu,2})\left(
    \begin{array}{ll}
      a\,\mathrm{e}^{-2\mathrm{i}\mu} & b\,\mathrm{e}^{\mathrm{i}\mu} \\ 
      d\,\mathrm{e}^{-\mathrm{i}\mu} & c\,\mathrm{e}^{2\mathrm{i}\mu}
    \end{array}
  \right).
\end{equation*}  
When $a=c$, the eigenvalues of the last matrix are
\begin{equation*}
  \lambda_{\mu,\pm}:=a \cos(2\mu) \pm \sqrt{bd-a^2\sin^2(2\mu)},
\end{equation*}
which for $\mu\in[-\pi,\pi]$ generate the continuum part of the limiting spectra proved by
\cn{bootland2022analysis}. But there may be also isolated points to appear in the limiting
  spectra under some circumstances \cite{bootland2022analysis}, which we did not find here. In all
of our numerical observations, we found that no outlier contributed to the spectral radius.

When we graph the function $\rho(\xi)$, in contrast to the two
  subdomain Section \ref{2SubSec}, we will rescale $\xi$ as
  $\sqrt{\xi^2+\eta}/\sqrt{\eta}$ for $\eta>0$ and $\Re \xi/\omega$
  for $\eta=-\omega^2$, $\omega>0$ so that in both cases the
  evanescent modes among $\mathrm{e}^{\pm\sqrt{\xi^2+\eta}x+\I\xi y}$
  always correspond to the rescaled variables greater than
  one. Moreover, in the transmission conditions for domain truncation,
  it is essential to approximate the square-root symbol
  $\sqrt{\xi^2+\eta}$. For example, the Taylor of order zero
  approximation gives $\sqrt{\xi^2+\eta}\approx \sqrt{\eta}$. The
  reflection coefficient without overlap given by
\[
  R=\frac{\sqrt{\eta}-\sqrt{\xi^2+\eta}}{\sqrt{\eta}-\sqrt{\xi^2+\eta}}
  =\frac{1-\frac{\sqrt{\xi^2+\eta}}{\sqrt{\eta}}}{1+\frac{\sqrt{\xi^2+\eta}}{\sqrt{\eta}}}
\]
can naturally be understood as a function of the rescaled variables.

Two different elliptic partial differential equations will be considered. One is the diffusion
equation with $\mathcal{L}_x+\mathcal{L}_y=-\Delta$, $\eta>0$, also known as the screened Laplace
equation, the modified Helmholtz equation, or the Helmholtz equation with the good sign. The
other is the time-harmonic wave equation with $\mathcal{L}_x+\mathcal{L}_y=-\Delta$ in $\Omega$,
$\eta=-\omega^2$, $\omega>0$, also known as the Helmholtz equation. For free space waves, the bottom
and top boundary $[0,X]\times\{0,Y\}$ is from domain trunction of $[0,X]\times(-\infty,\infty)$. In
case PMLs are used for this along the bottom and top boundaries, we consider the problem
\R{eqprb} on the PML-augmented domain $[0,X]\times[-D,Y+D]$ with
$\mathcal{L}_y=-\tilde{s}(y)\partial_y(\tilde{s}(y)\partial_y)$ in the PML,
$\mathcal{B}^{b,t}=\mathcal{C}=\mathcal{I}$ or $\mp \partial_y$ on $[0,X]\times\{-D, Y+D\}$, and the complex
stretching function
\begin{equation}
  \tilde{s}(y)=\left\{
    \begin{array}{ll}
      1 & \text{ on }[0,Y],\\
      (1-\I\tilde{\sigma}(-y))^{-1} & \text{ on }[-D,0],\\
      (1-\I\tilde{\sigma}(y-Y))^{-1} & \text{ on }[Y,Y+D],
    \end{array}
  \right.
  \label{eqpml}
\end{equation}
$\tilde{\sigma}(0)=0$,
$\int_0^D\tilde{\sigma}(y)~\mathrm{d}y=\frac{1}{2}D\gamma$ (a specific
form of $\tilde{\sigma}$ has no influence on the convergence
analysis), and $\gamma>0$ is the PML strength. In this case, we will
consider the generalised Fourier frequency from the Sturm-Liouville
problem \R{eqsturm} defined on $[-D, Y+D]$ and we have
$\xi=j\pi/(Y+(2-\I\gamma)D)$, $j$ a nonnegative integer for
$\mathcal{C}=\mp\partial_y$ or positive integer for
$\mathcal{C}=\mathcal{I}$.  The trace operators in the transmission
conditions \R{eqtc} are chosen from Table~\ref{tabB}, where the PML
operators need more explanation. For the diffusion problem, we will
use the following Dirichlet-to-Neumann operator:
\begin{equation}
  \begin{aligned}
    \mathcal{S}: g\mapsto \partial_xv(0) \text{ with $v$ satisfying}& \\
    (-\hat{s}(x)\partial_x({\hat{s}(x)\partial_x}\cdot)-\partial_{yy}+\eta) v&=0 & &\text{ on }[-D, 0]\times[0,Y], \\
    \mathcal{B}^{b,t}v&=0 & &\text{ on }[-D,0]\times\{0,Y\},\\
    \mathcal{C}v&=0 & &\text{ on }\{-D\}\times[0,Y],\\
    v&=g & &\text{ on }\{0\}\times[0,Y],
  \end{aligned}
  \label{eqsd}
\end{equation}
where $\hat{s}(x)=(1+\tilde{\sigma}(-x))^{-1}$ on $(-D,0)$, $\tilde{\sigma}(0)=0$,
$\int_0^D\tilde{\sigma}(x)\D{x}=\frac{1}{2}D\gamma$, $\gamma>0$, and $\mathcal{C}$ is either
Dirichlet or Neumann trace operator. Let $\hat{\mathcal{S}}$ be the symbol of the above
$\mathcal{S}$ and $\mathbf{n}$ be the outer normal unit vector. We have
\begin{equation}
  \hat{\mathcal{S}} = \left\{
    \begin{array}{ll}
      \displaystyle
      \sqrt{\xi^2+\eta}{{1+E}\over{1-E}},\quad & \text{if }\mathcal{C}=\mathcal{I},\\[10pt]
      \displaystyle
      \sqrt{\xi^2+\eta}{{1-E}\over{1+E}},\quad & \text{if }\mathcal{C}=\partial_{\mathbf{n}},
    \end{array}
  \right.
  \label{eqhats}
\end{equation}
with $E=\exp(-(2+\gamma)D\sqrt{\xi^2+\eta})$.  For the free space wave problem, we will use the
following Dirichlet-to-Neumann operator:
\begin{equation}
  \thickmuskip=0.6mu
  \medmuskip=0.2mu
  \thinmuskip=0.1mu
  \scriptspace=0pt
  \nulldelimiterspace=-0.1pt
  \begin{aligned}
    \mathcal{S}: g\mapsto \partial_xv(0) \text{ with $v$ satisfying}& \\
    \tilde{s}(x)\partial_x({\tilde{s}(x)\partial_x}v)+\tilde{s}(y)\partial_y(\tilde{s}(y)\partial_{y}v)&=\eta v& &\text{ on }[-D, 0]\times[-D,Y+D], \\
    \mathcal{C}v&=0& &\text{ on }[-D,0]\times\{-D,Y+D\},\\
    \mathcal{C}v&=0& &\text{ on }\{-D\}\times[-D,Y+D],\\
    v&=g& &\text{ on }\{0\}\times[-D,Y+D],
  \end{aligned}
  \label{eqsw}
\end{equation}
where the complex stretching function $\tilde{s}$ has been defined in \R{eqpml}. The symbol of the
above $\mathcal{S}$ is \R{eqhats} but with $\displaystyle E=\exp(-(2-\I\gamma)D\sqrt{\xi^2+\eta})$.

\begin{table}
  \caption{Trace operators $\mathcal{B}_j^{l,r}$ in the transmission conditions \R{eqtc}: column 2
    for the diffusion problem, column 3 for the wave problem}
  \begin{tabular}{c|c|c}
    \hline
    \hline
    Problem & $(-\Delta+\eta)u=f$, $\eta>0$ & $(-\Delta-\omega^2)u=f$, $\omega>0$\\
    \hline
    Dirichlet & $\mathcal{I}$ & not used\\
    \hline
    Taylor of order 0 & $\mp\partial_{x}+\sqrt{\eta}$ & $\mp\partial_{x}+\I\omega$ \\
    % \hline
    % Square-root & not used &  $\mp\partial_{x}+\I\omega\sqrt{\alpha-\frac{\beta}{\omega^2}\partial_{yy}}$\\
    \hline
    Continuous PML & $\mp\partial_{x}+\mathcal{S}$ with $\mathcal{S}$ in \R{eqsd}
                                            & $\mp\partial_{x}+\mathcal{S}$ with $\mathcal{S}$ in \R{eqsw} \\
    \hline
  \end{tabular}
  \label{tabB}
\end{table}

\begin{remark}
  The signs in front of the imaginary unit $\I$, \eg, $-$ in \R{eqpml} and $+$ in Table~\ref{tabB},
  reflect that we are using the time-harmonic convention $\mathrm{e}^{\I\omega t}$, and the signs
  would be opposite for the other convention $\mathrm{e}^{-\I\omega t}$ (more popular). Our choice
  is consistent with $\sqrt{\xi^2+\eta}$ using the branch cut $(-\infty, 0)$.
\end{remark}

In the following subsections, we shall explore the various specifications of the Schwarz methods for
the diffusion and free space wave problems in detail based on the numerical calculations of the
eigenvalues of the symbol iteration matrices. For the anxious readers, we first list the main
observations in Table~\ref{tabd} for diffusion and Table~\ref{tabw} for free space waves. More
information, \eg, the dependence on the other parameters, will be revealed later.

\begin{table}
  \caption{Scalings of $\displaystyle\max_{\xi\in [0,\infty)}\rho$ of the Schwarz methods for the
    \emph{diffusion} problem. $N$ number of subdomains; $L$ overlap width; $X$ domain width; $H$
    subdomain width}
  \begin{tabular}{c|c|c}
    \hline
    \hline
    Parallel Schwarz & $N\to\infty$, $H$, $L$ fixed & $N\to\infty$, $X$, $\frac{L}{H}$ fixed \\
    \hline
    Dirichlet & $O(1)\in (0,1)$ & $1-O(N^{-2})$\\
    \hline
    Taylor of order zero & $O(1)\in (0,1)$ & $1-O(N^{-1})$\\
    \hline
    PML fixed & $O(1)\in (0,1)$ & $1-O(N^{-1})$\\
    \hline
  \end{tabular}
  \begin{tabular}{c|c|c}  
    \hline
    \hline
    Double sweep Schwarz &  $N\to\infty$, $H$, $L$ fixed & $N\to\infty$, $X$, $\frac{L}{H}$ fixed \\
    \hline
    Dirichlet & $O(1)\in (0,1)$ & $1-O(N^{-2})$\\
    \hline
    Taylor of order zero & $O(1)\in (0,1)$ & $1-O(N^{-1})$\\
    \hline
    PML fixed & $O(1)\in (0,1)$ & $O(1)\in(0,1)$\\
    \hline
  \end{tabular}
  \label{tabd}  
\end{table}

\begin{table}
  \caption{Scalings of $\max_{\xi}\rho$ of the Schwarz methods for the free space \emph{wave}
    problem. $N$ number of subdomains; $L$ overlap width; $X$ domain width; $H$ subdomain width;
    $\omega$ wavenumber; $D$ PML width}
  \begin{tabular}{c|c|c|c|c}
    \hline
    \hline
    $\begin{aligned}\text{Parallel}\\\text{Schwarz}\end{aligned}$ & $\begin{aligned}&N\to\infty, \omega, \\&H, L\text{ fixed}\end{aligned}$ & $\begin{aligned}&N\to\infty, \omega,\\&X, \Frac{L}{H}\text{ fixed}\end{aligned}$ & $\begin{aligned}&\omega\to\infty, N,\\&X, L\omega\text{ fixed}\end{aligned}$ & $\begin{aligned}&N\to\infty, \Frac{\omega}{N},\\&X, L\omega^{\frac{3}{2}}\text{ fixed}\end{aligned}$\\
    \hline
    Taylor & $1-O(N^{-1})^*$ & $1-O(N^{-\frac{5}{3}})$ & $1-O(\omega^{-\frac{9}{20}})$ & $1-O(N^{-2})$\\
    \hline
    $\begin{aligned}\text{PML}\\L=0\end{aligned}$ & $\begin{aligned}1-O(N^{-1})\\\text{with }D\text{ fixed}\end{aligned}$ & $\begin{aligned}1-O(N^{-1})\\\text{with }D\text{ fixed}\end{aligned}$ & $\begin{aligned}&O(1)\\\text{with }&D\text{ fixed}\end{aligned}$ & $\begin{aligned}1-O(N^{-1})\\\text{with }D\text{ fixed}\end{aligned}$\\
%    \hline
%    Square-root & & & & \\
    \hline
  \end{tabular}
  {\raggedright \footnotesize *~We observed also one exception for which
    $\max_{\xi}\rho=1-O(N^{-3/2})$; see Figure~\ref{figfspt0n}. \par}
  \begin{tabular}{c|c|c|c|c}
    \hline
    \hline
    $\begin{aligned}&\text{Double}\\&\text{sweep}\\&\text{Schwarz}\end{aligned}$ & $\begin{aligned}&N\to\infty, \omega, \\&H, L\text{ fixed}\end{aligned}$ & $\begin{aligned}&N\to\infty, \omega,\\&X, \Frac{L}{H}\text{ fixed}\end{aligned}$ & $\begin{aligned}&\omega\to\infty, N,\\&X, L\omega\text{ fixed}\end{aligned}$ & $\begin{aligned}&N\to\infty, \Frac{\omega}{N},\\&X, L\omega^{\frac{3}{2}}\text{ fixed}\end{aligned}$\\
    \hline
    Taylor & $O(1)$ & diverges & $1-O(\omega^{-\frac{9}{20}})$ & diverges \\
    \hline
    $\begin{aligned}\text{PML}\\L=0\end{aligned}$ & $\begin{aligned}&\to0\text{ with }D\text{ }\\&=O(\log N)\end{aligned}$ & $\begin{aligned}&\to0\text{ with }D\text{ }\\&=O(\log N)\end{aligned}$ & $\begin{aligned}&O(1)\\\text{with }&D\text{ fixed}\end{aligned}$ & $\begin{aligned}\to0\text{ with }D\\=O(\log N)\end{aligned}$\\
    % \hline
    % Square-root & & & & \\
    \hline
  \end{tabular}
  \label{tabw}
\end{table}

%%%%%%%%%%%%%%%%%%%%%%%%%%%%%%%%%%%%%%%%%%%%%%%%%%

\subsection{Parallel Schwarz methods for the diffusion problem}

%%%%%%%%%%%%%%%%%%%%%%%%%%%%%%%%%%%%%%%%%%%%%%%%%%

In this case, the original problem \R{eqprb} has $\mathcal{L}_x+\mathcal{L}_y=-\Delta$, $\eta>
0$. The Neumann boundary condition $\mathcal{B}^{b,t}=\mp\partial_y$ is imposed on top and bottom of
$\Omega$ so that $\xi=0, \pi, 2\pi, 3\pi, ...$ but for simplicity the continuous range
$\xi\in [0, \infty)$ is considered for the diffusion problem.  The boundary condition on left and
right of $\Omega$ is Neumann, Dirichlet or the same as the transmission condition:
$\mathcal{B}^{l,r}=\mp\partial_x$, $\mathcal{I}$ or $\mathcal{B}_j^{l,r}$.

%%%%%%%%%%%%%%%%%%%%%%%%%%%%%%%

\subsubsection{Parallel Schwarz methods with Dirichlet transmission for the diffusion problem}

We begin with the parallel Schwarz method with classical Dirichlet transmission
$\mathcal{B}_j^{l,r}=\mathcal{I}$ for which the overlap width $L>0$ is necessary for
convergence. For a general sequential decomposition, a variational interpretation of the convergence
was given by \cn{Lions:1988:SAM}. An estimate of the convergence rate was derived recently
\cite{ciaramella3}. A general theory for the parallel Schwarz method is far from being as complete
as the theory for the additive Schwarz method \cite{TWbook}. For example, a convergence theory of
the restricted additive Schwarz (RAS) method \cite{cai1999restricted} has been an open problem for
two decades, and RAS is equivalent to the parallel Schwarz method \cite{Efstathiou,St-Cyr07}.

\begin{paragraph}{Convergence with increasing number of fixed size subdomains}
  
  With the number of subdomains $N\to\infty$, the subdomain width $H$ fixed, and the overlap width
  $L$ fixed, the convergence factor $\rho=\rho(\xi)$ is illustrated in the top half of
  Figure~\ref{figpdn}. The plots of $\rho$ on the left show that $\rho_{\infty}$ from the limiting
  spectrum formula is a very accurate approximation of $\rho$ already starting from $N=10$. It is
  also clear that the Schwarz method is a smoother that performs better for larger cross-sectional
  frequency $\xi$.  The plots of the $\log$-scaled $1-\rho$ on the right display a constant slope of
  the initial parts of the curves. The slope is estimated to be $2$ for $\rho_{\infty}$. The
  influence of the original boundary condition is also manisfested in those plots: the asymptotic
  $\rho\to \rho_{\infty}$ as $N\to\infty$ comes later for the Dirichlet problem (more precisely,
  mixed with the Neumann conditions on top and bottom, similarly hereinafter) than for the Neumann
  problem.  In the bottom half of Figure~\ref{figpdn}, the scaling of $1-\max_{\xi}\rho$ (attained
  at $\xi=0$ as seen before) is shown. The conclusion is that $\max_{\xi}\rho=O(1)<1$ independent
  of $N\to\infty$.  But a preasymptotic deterioration with growing $N$ is visible for
  $\mathcal{B}^{l,r}=\mathcal{I}$ and small $\eta$, $H$.

\begin{figure}
  \centering
  \includegraphics[width=.56\textwidth,trim=10 10 0 6,clip]{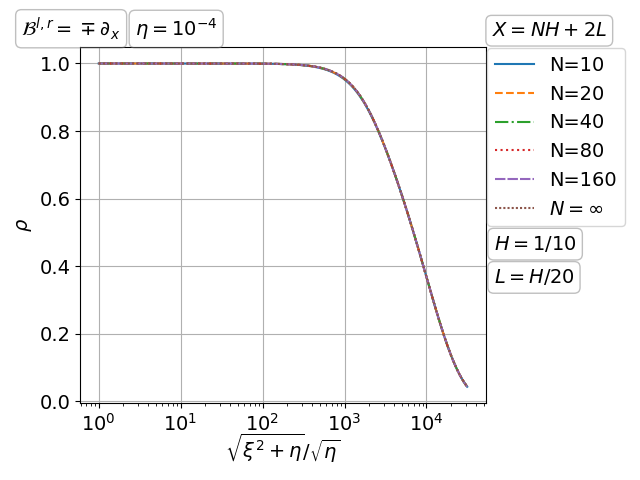}%
  \includegraphics[width=.44\textwidth,height=13em,trim=0 10 0 6,clip]{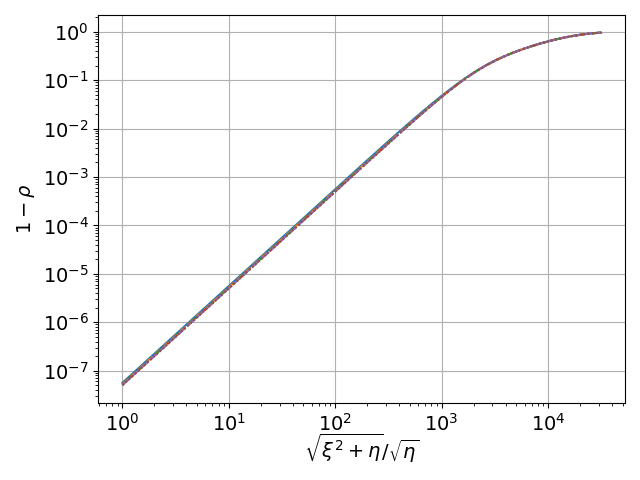}\\
  \includegraphics[width=.56\textwidth,trim=10 10 0 6,clip]{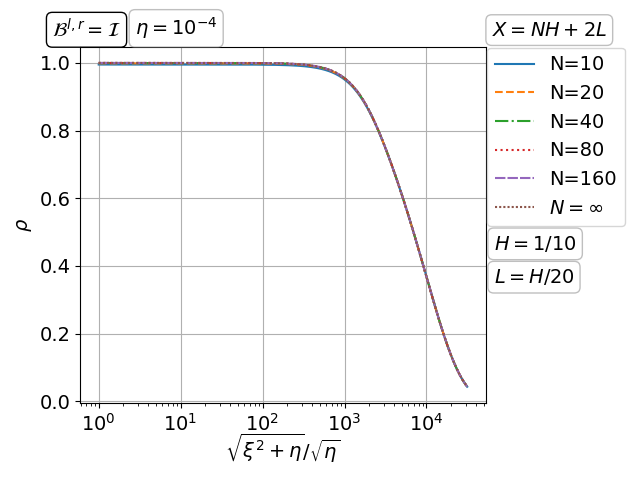}%
  \includegraphics[width=.44\textwidth,height=13em,trim=0 10 0 6,clip]{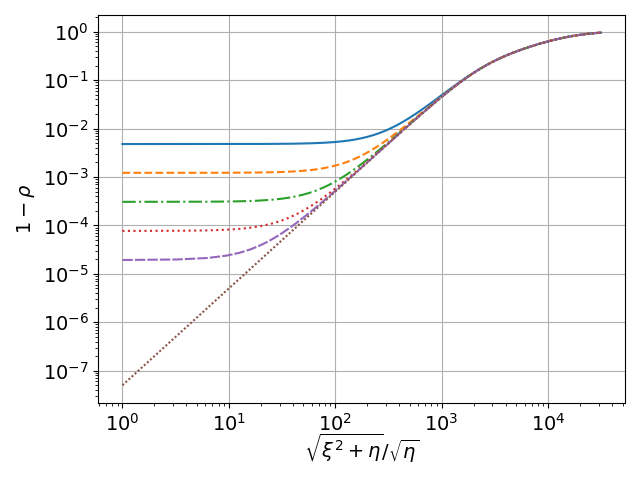}
  \includegraphics[width=.5\textwidth,trim=5 6 0 2,clip]{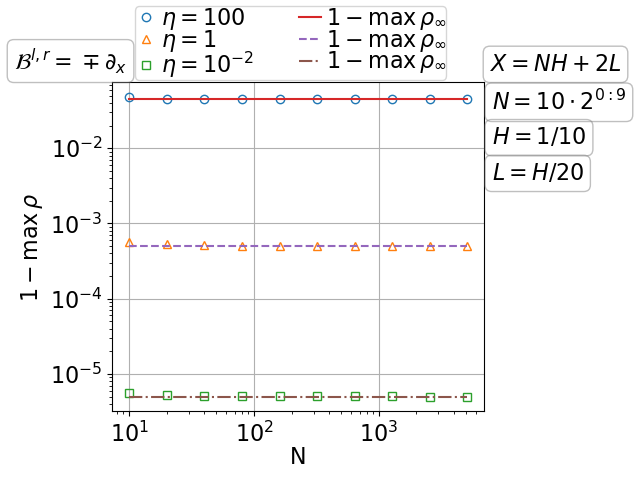}%
  \includegraphics[width=.5\textwidth,trim=5 6 0 2,clip]{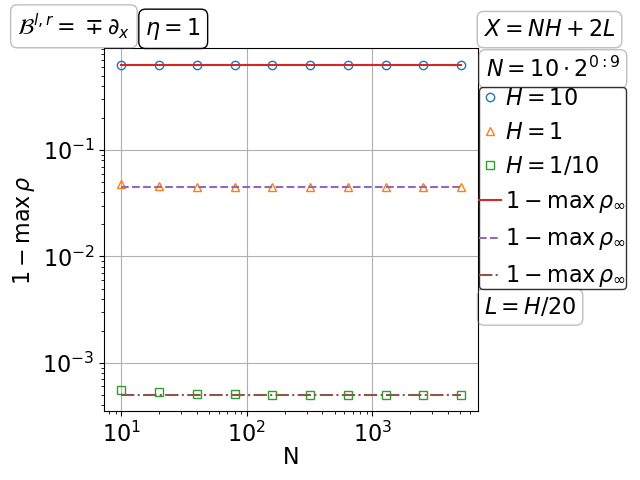}\\
  \includegraphics[width=.5\textwidth,trim=5 6 0 2,clip]{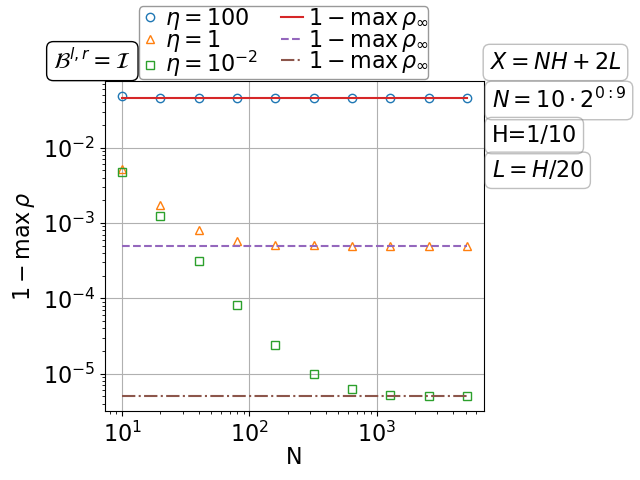}%
  \includegraphics[width=.5\textwidth,trim=5 6 0 2,clip]{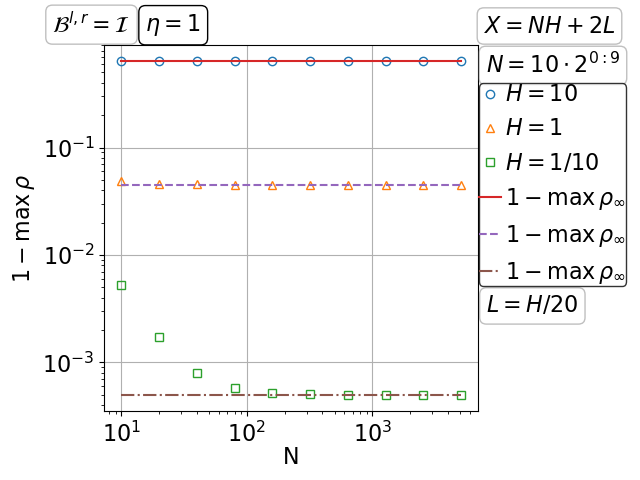}
  \caption{Convergence of the parallel Schwarz method with Dirichlet transmission for
    diffusion with increasing number of {fixed size subdomains}.}
  \label{figpdn}
\end{figure}

\end{paragraph}

\begin{paragraph}{Convergence on a fixed domain with increasing number of subdomains}
  
  With the number of subdomains $N\to\infty$ and the domain width $X$ fixed, the
  results are shown in Figure~\ref{figpd1} in a format similar to
  Figure~\ref{figpdn}. In this scaling, the symbol values $a$, $b$ vary with $N$
  and no limiting spectral radius is known in closed form, since the matrix entries change with the scaling. But the plots of the
  convergence factor $\rho=\rho(\xi)$ in the left column of the top half of
  Figure~\ref{figpd1} indicate the limiting curve is the constant one, and the
  plots in the right column display a constant slope in the $\log$ scale of
  $1-\rho$ for small $\sqrt{\xi^2+\eta}$. The slope is estimated to be $2$ for the
  Neumann problem and $0$ for the Dirichlet problem. The bottom half of
  Figure~\ref{figpd1} shows the scaling $\max_{\xi}\rho = 1-O(N^{-2})$
  for various values of the coefficient $\eta$ and the overlap width $L$.  The hidden
  constant factor in $O(N^{-2})$ can be seen depending linearly on
  $\eta$ and $\frac{L}{H}$ for the Neumann problem, and linearly on
  $\frac{L}{H}$ but is independent of small $\eta$ for the Dirichlet problem.

\begin{figure}
  \centering
  \includegraphics[width=.56\textwidth,trim=10 10 0 6,clip]{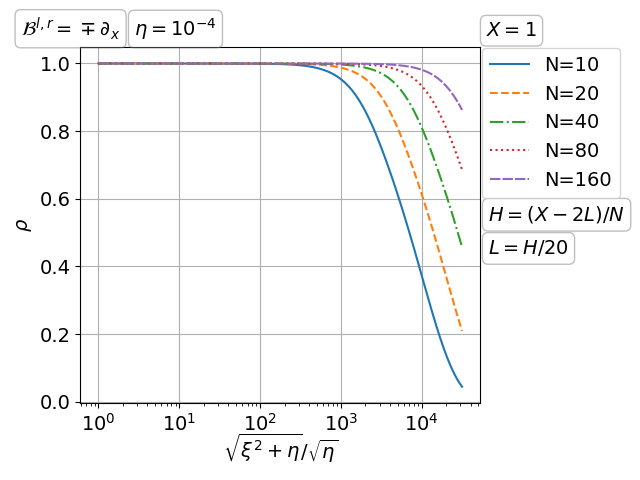}%
  \includegraphics[width=.44\textwidth,height=13em,trim=0 10 0 6,clip]{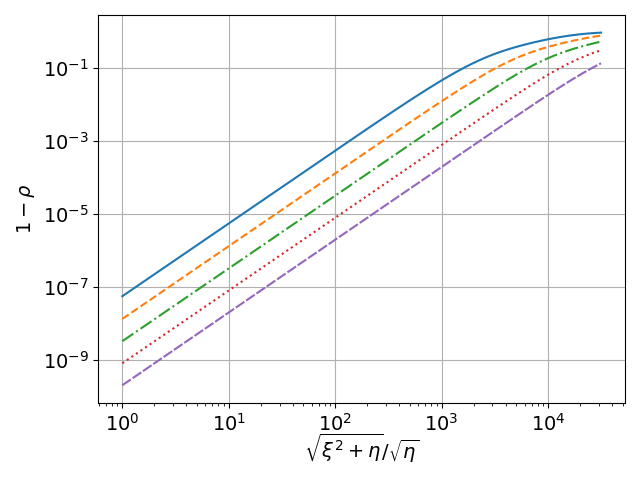}\\
  \includegraphics[width=.56\textwidth,trim=10 10 0 6,clip]{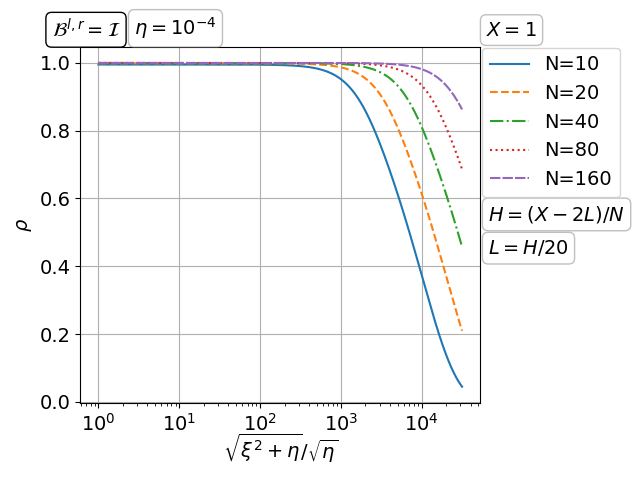}%
  \includegraphics[width=.44\textwidth,height=13em,trim=0 10 0 6,clip]{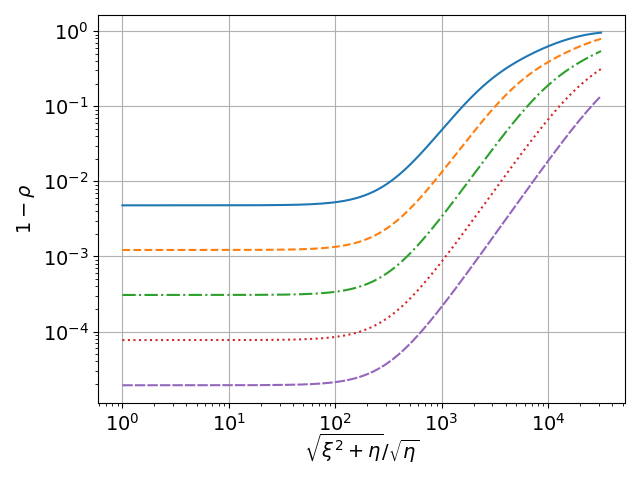}\\
  \includegraphics[width=.5\textwidth,trim=5 6 0 2,clip]{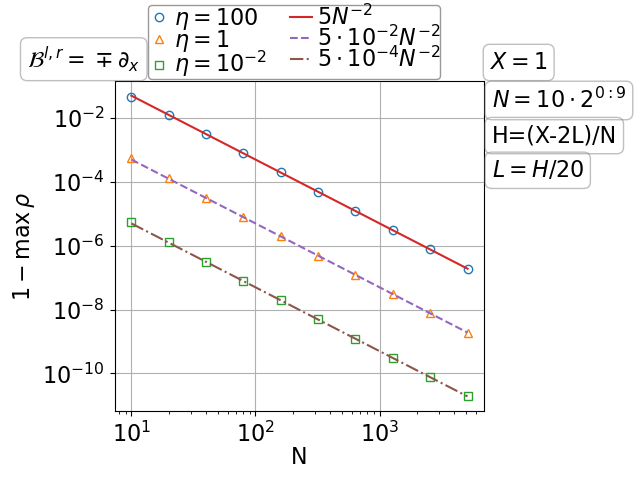}%
  \includegraphics[width=.5\textwidth,trim=5 6 0 2,clip]{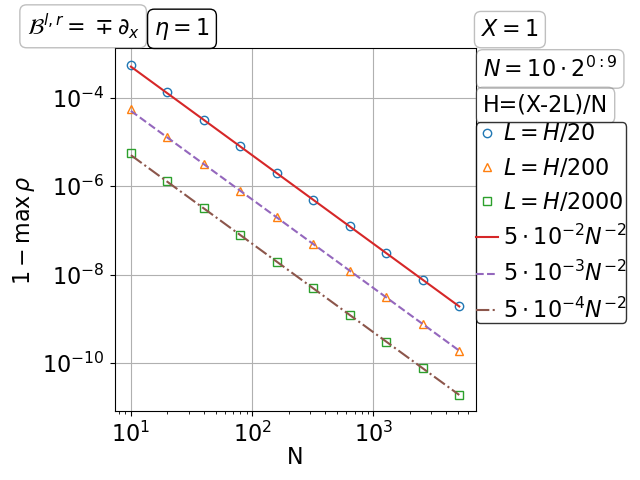}\\
  \includegraphics[width=.5\textwidth,trim=5 6 0 2,clip]{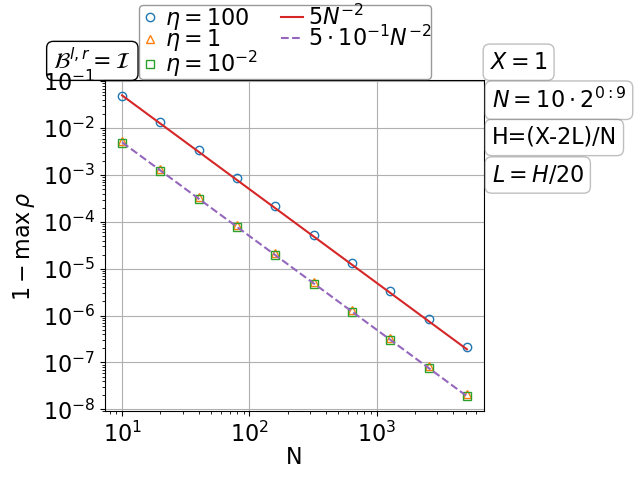}%
  \includegraphics[width=.5\textwidth,trim=5 6 0 2,clip]{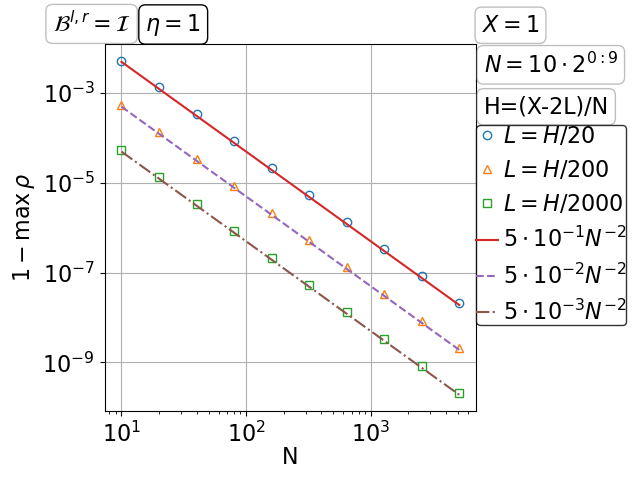}
  \caption{Convergence of the parallel Schwarz method with Dirichlet transmission for
    diffusion on a {fixed domain} with {increasing number of
      subdomains}.}
  \label{figpd1}
\end{figure}

\end{paragraph}

\begin{paragraph}{Convergence with a fixed number of subdomains of shrinking
    width}

  With the number of subdomains $N$ fixed and the subdomain width $H\to 0$, we try to understand the
  dependence of the convergence factor $\rho$ on $H$ separately from the dependence on $N$. As we
  can see from the top half of Figure~\ref{figpdh}, smaller $H$ leads to slower convergence for all
  the Fourier frequencies $\xi$. Note that each graph of $\rho(\xi)$ for $N=10$ is accompanied with
  a graph of the corresponding $\rho_{\infty}$ for infinitely many subdomains $N=\infty$. So by
  ``interpolation'' one can imagine roughly the graphs for other values of $N$ based on the
  knowledge of Figure~\ref{figpdn}. In the right column of Figure~\ref{figpdh}, we see that, for the
  Neumann problem, the process of the subdomain width $H\to 0$ alone incurs the deterioration of
  convergence, while for the Dirichlet problem the process of $H\to 0$ needs to be combined with
  number of subdomains $N\to\infty$ to incur a deterioration. The bottom half of Figure~\ref{figpdh}
  presents the scaling of $1-\max_{\xi}\rho$ with $H\to 0$. On the left for the overlap width
  $L=O(H)$, it shows that $\max_{\xi} \rho = 1-O(H^2)$ for the Neumann problem where $O(H^2)$
  depends linearly on the coefficient $\eta$, and $\max_{\xi}\rho=O(1)$ for the Dirichlet problem
  where $O(1)$ is independent of $\eta$. On the right, different values of $\nu$ for the overlap
  width $L=O({H^\nu})$ are used, which suggests $\max_{\xi}\rho = 1-O(HL)$ for the Neumann problem
  and $\max_{\xi}\rho = 1-O(H^{-1}L)$ for the Dirichlet problem.

\begin{figure}
  \centering
  \includegraphics[width=.56\textwidth,trim=10 10 0 6,clip]{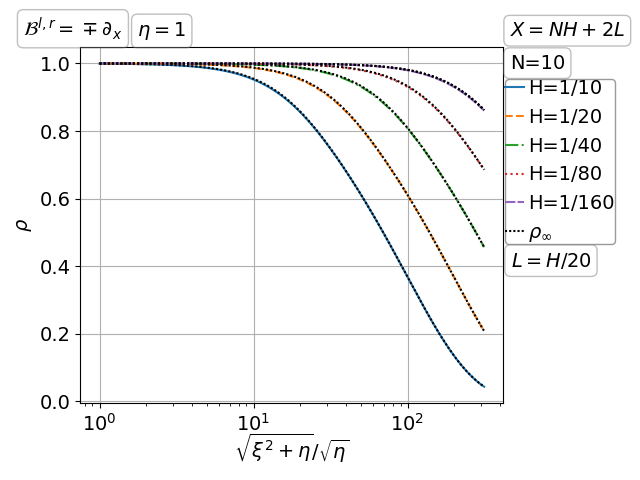}%
  \includegraphics[width=.44\textwidth,height=13em,trim=0 10 0 6,clip]{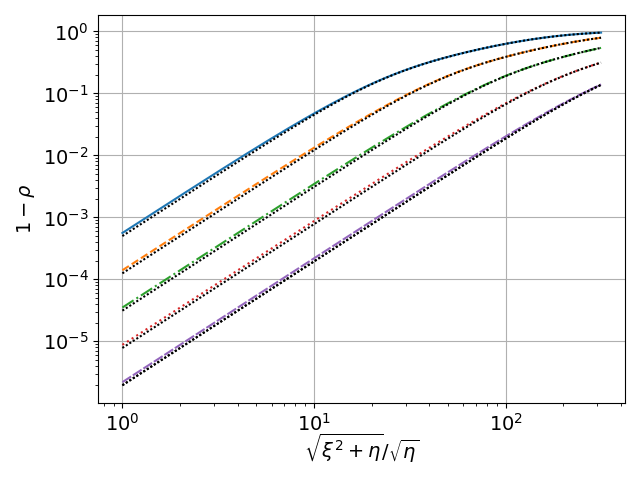}\\
  \includegraphics[width=.56\textwidth,trim=10 10 0 6,clip]{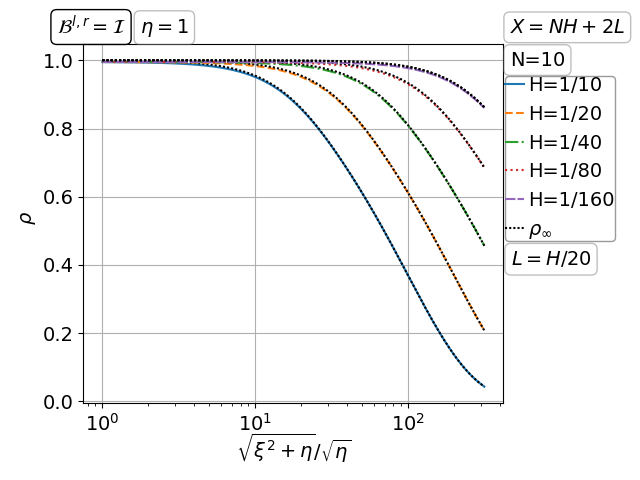}%
  \includegraphics[width=.44\textwidth,height=13em,trim=0 10 0 6,clip]{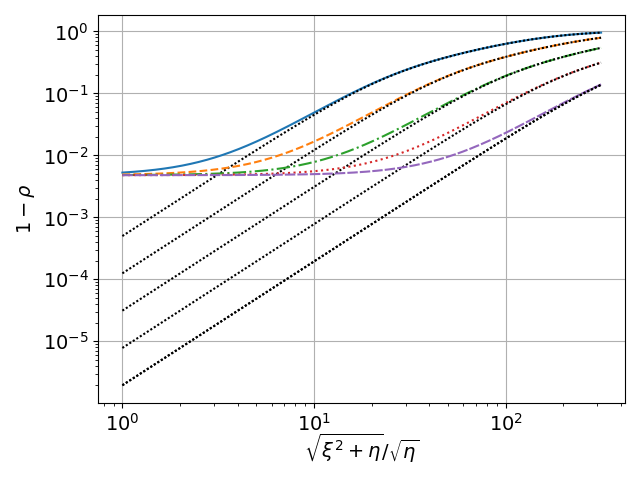}\\
  \includegraphics[width=.5\textwidth,trim=5 6 0 2,clip]{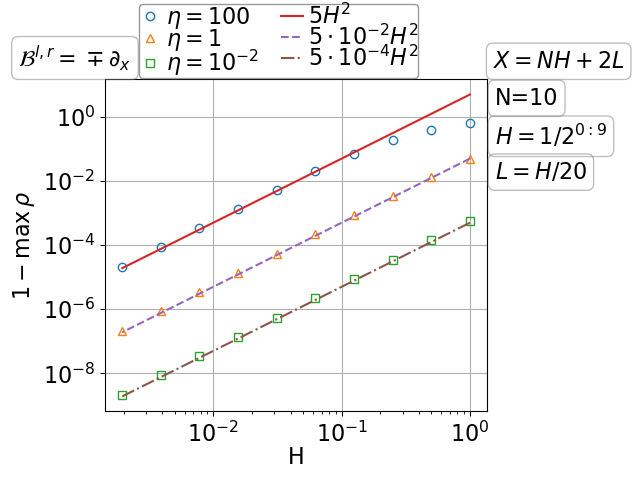}%
  \includegraphics[width=.5\textwidth,trim=5 6 0 2,clip]{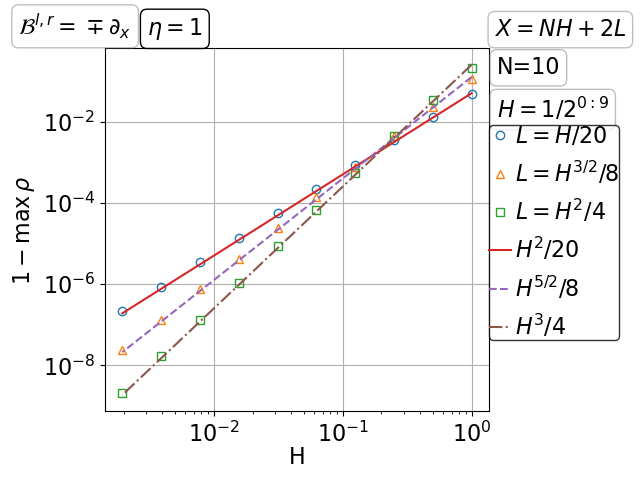}\\
  \includegraphics[width=.5\textwidth,trim=5 6 0 2,clip]{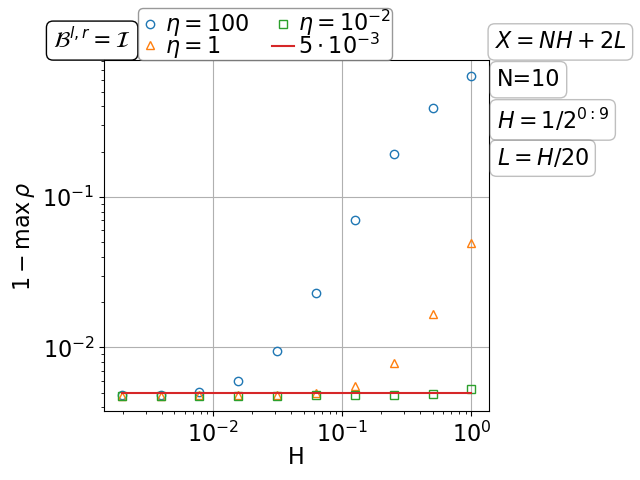}%
  \includegraphics[width=.5\textwidth,trim=5 6 0 2,clip]{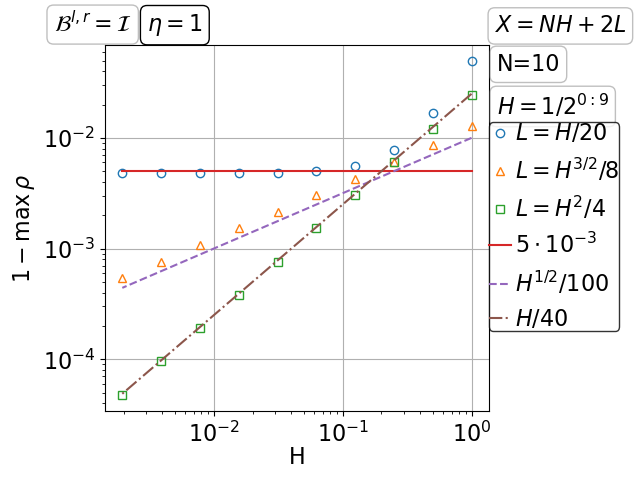}
  \caption{Convergence of the parallel Schwarz method with Dirichlet transmission for
    diffusion with a {fixed number of subdomains} of {shrinking width}.}
  \label{figpdh}
\end{figure}

\end{paragraph}

\begin{paragraph}{Convergence on a fixed domain with a fixed number of subdomains of shrinking
    overlap}

  With both number of subdomains $N$ and the domain width $X$ fixed and the overlap width $L\to 0$,
  we find the convergence factor $\rho(\xi)\to 1$; see the top half of Figure~\ref{figpdl}.  In
  particular, the right column shows a faster convergence of the Schwarz method for the Dirichlet
  problem than for the Neumann problem. The speed of $\max_{\xi}\rho =\rho(0)\to 1$ appears linear
  in $L\to 0$; see the bottom half of Figure~\ref{figpdl}.  The hidden constant factor in
  $O(L)=1-\rho$ is $O(\eta H)$ for the Neumann problem and robust in the
  coefficient $\eta$ and subdomain width $H$ for the Dirichlet problem.

\begin{figure}
  \centering
  \includegraphics[width=.56\textwidth,trim=10 10 0 6,clip]{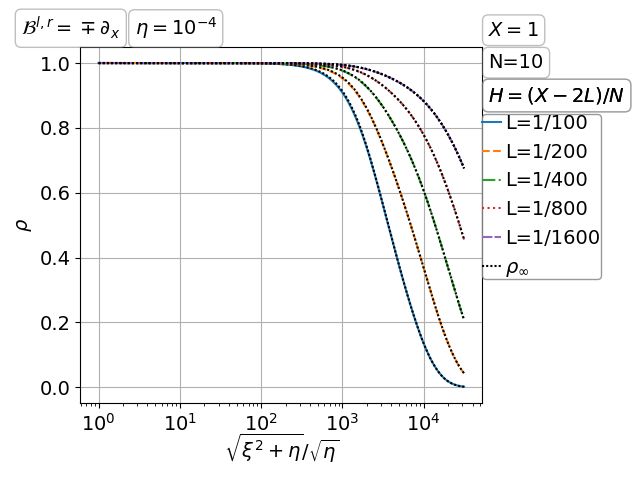}%
  \includegraphics[width=.44\textwidth,height=13em,trim=0 10 0 6,clip]{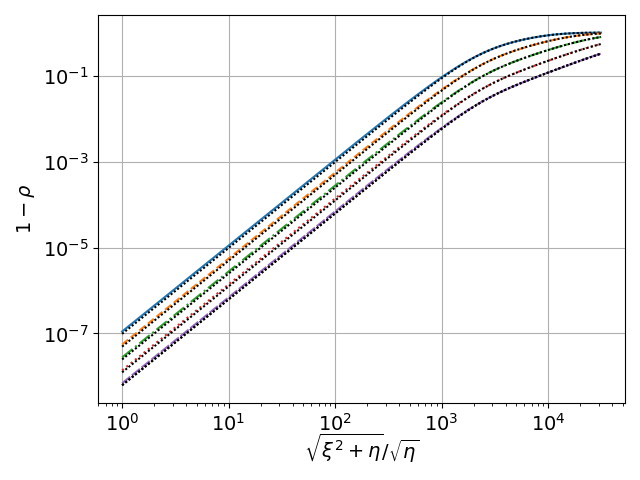}\\
  \includegraphics[width=.56\textwidth,trim=10 10 0 6,clip]{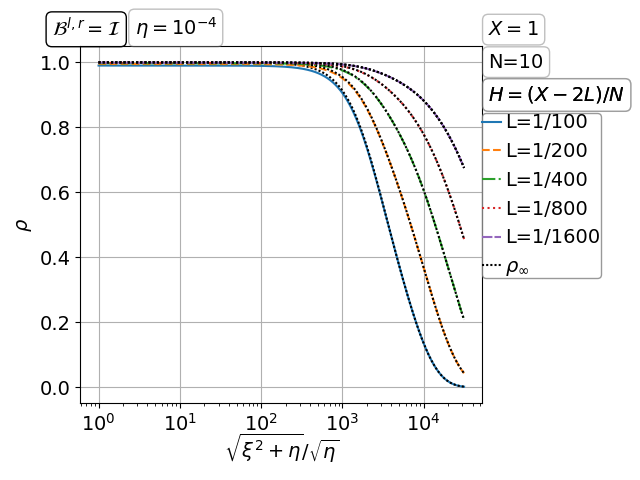}%
  \includegraphics[width=.44\textwidth,height=13em,trim=0 10 0 6,clip]{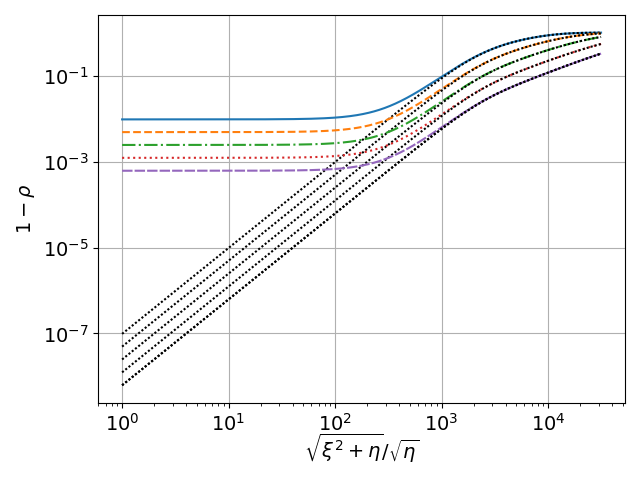}\\
  \includegraphics[width=.5\textwidth,trim=5 6 0 2,clip]{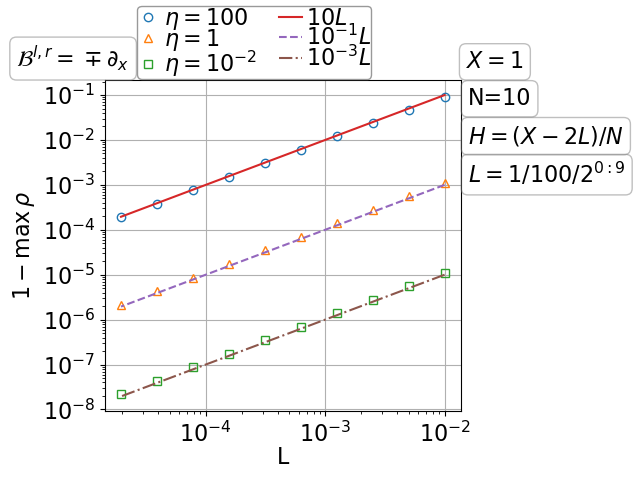}%
  \includegraphics[width=.5\textwidth,trim=5 6 0 2,clip]{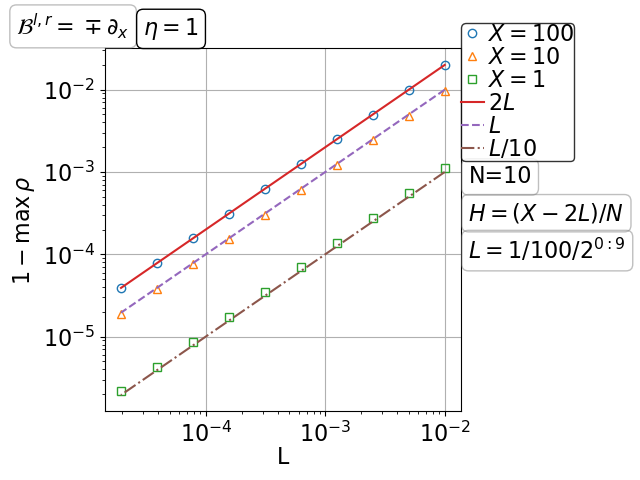}
  \\
  \includegraphics[width=.5\textwidth,trim=5 6 0 2,clip]{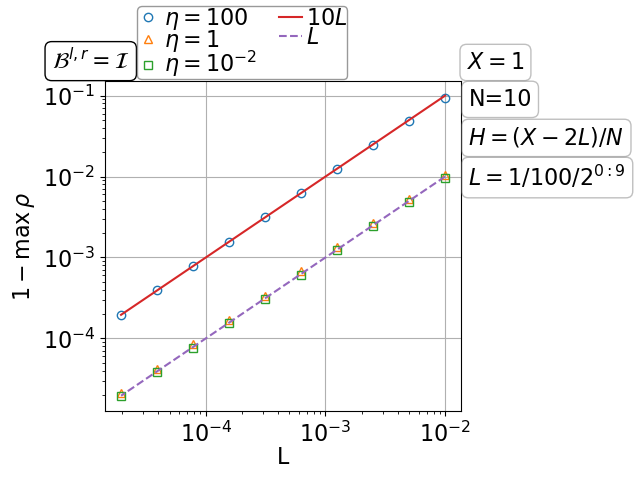}%
  \includegraphics[width=.5\textwidth,trim=5 6 0 2,clip]{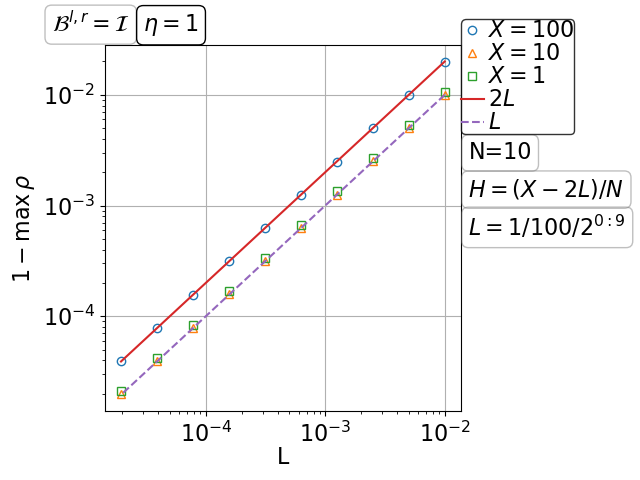}
  \caption{Convergence of the parallel Schwarz method with Dirichlet transmission for diffusion on
    a {fixed domain} with a {fixed number of subdomains} of {shrinking overlap}.}
  \label{figpdl}
\end{figure}

\end{paragraph}

%%%%%%%%%%%%%%%%%%%%%%%%%%%%%%%

\subsubsection{Parallel Schwarz method with Taylor of order zero transmission for the diffusion
  problem}

By using the Taylor of order zero transmission condition, see Table~\ref{tabB}, the subdomain
problem away from the boundary $\{0, X\}\times [0,Y]$ is a domain truncation of the problem on the
infinite pipe $(-\infty, \infty)\times [0,Y]$. It is interesting to check how the original boundary
condtion on $\{0, X\}\times [0,Y]$ influences the convergence. At least, we expect the Schwarz
method to work when the original boundary operator is also Taylor of order zero:
$\mathcal{B}^{l,r}=\mathcal{B}^{l,r}_j$, because then the original problem is a domain trunction of
the infinite pipe problem. In the following paragraphs, we will study the convergence of parallel
Schwarz for the Dirichlet/Neumann/Taylor $\mathcal{B}^{l,r}$ separately. The literature on a general
theory of the optimized Schwarz method with Robin transmission conditions is rather
sparse. \cn{lions1990schwarz} gave the first convergence proof in the non-overlapping case, without
an estimate of the convergence rate, see also \cn{deng1997}. It seems possible only at the discrete
level to have a convergence rate of the non-overlapping optimized Schwarz method. \cn{qin2006} got
the first estimate of the convergence {factor} $1-O(h^{1/2}H^{-1/2})$ with an {optimized} choice of
the Robin parameter; see also \cn{qin2008}, \cn{xu2010}, \cn{lui2009}, \cn{Loisel},
\cn{liu2014robin}, \cn{GH15}, \cn{GH18}. In the overlapping case, the literature becomes even
sparser, and there is only the work of \cn{loisel2010} to our knowledge.

\begin{paragraph}{Convergence with increasing number of fixed size subdomains}

  With the number of subdomains $N\to\infty$ and the subdomain width $H$, the overlap width $L$ fixed,
  two regimes can be observed in the top halves of Figure~\ref{figpt0nn}, \ref{figpt0nd},
  \ref{figpt0nt0}. In one regime, see the first row of each figure, the maximum point of the
  convergence factor $\rho(\xi)$ tends to $\xi=0$ as the number of subdomains $N\to\infty$.  The other
  regime appears when the overlap width $L$ is sufficiently small, see the second row of each
  figure, in which the maximum point of $\rho(\xi)$ is almost fixed at the critical point of
  $\rho_{\infty}(\xi)=\lim_{N\to\infty}\rho(\xi)$. In both regimes, $\max_{\xi}\rho=O(1)<1$ as
  $N\to \infty$.

  \begin{figure}
    \centering%
    \includegraphics[width=.56\textwidth,trim=10 10 0 6,clip]{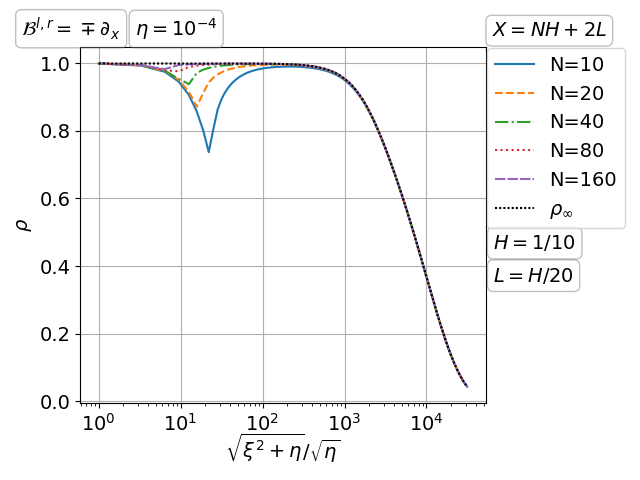}%
    \includegraphics[width=.44\textwidth,height=13em,trim=0 10 0 6,clip]{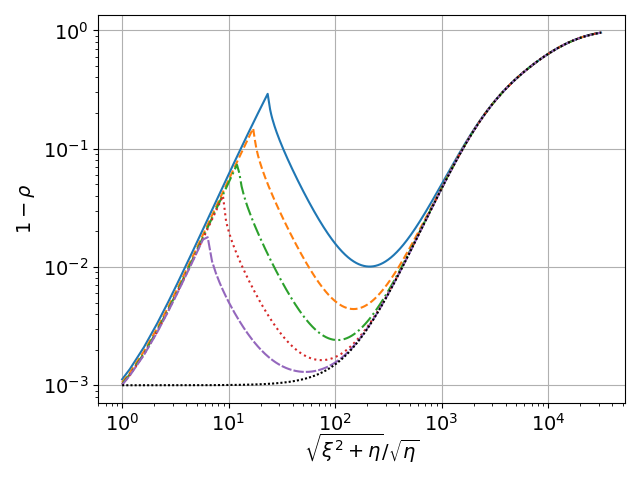}\\
    \includegraphics[width=.56\textwidth,trim=10 10 0 6,clip]{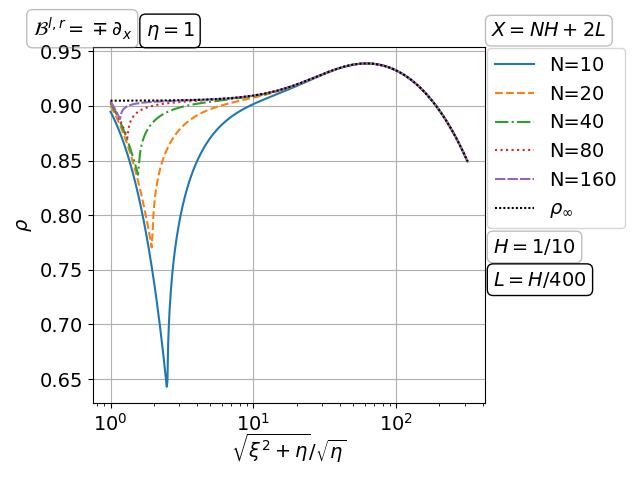}%
    \includegraphics[width=.44\textwidth,height=13em,trim=0 10 0 6,clip]{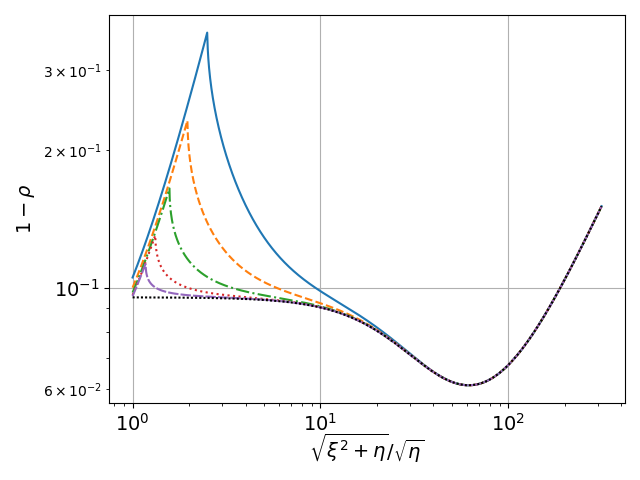}
    \includegraphics[width=.5\textwidth,trim=5 6 0 2,clip]{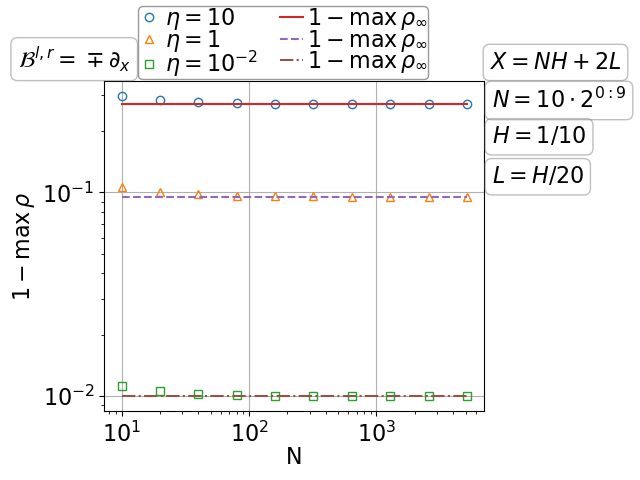}%
    \includegraphics[width=.5\textwidth,trim=5 6 0 2,clip]{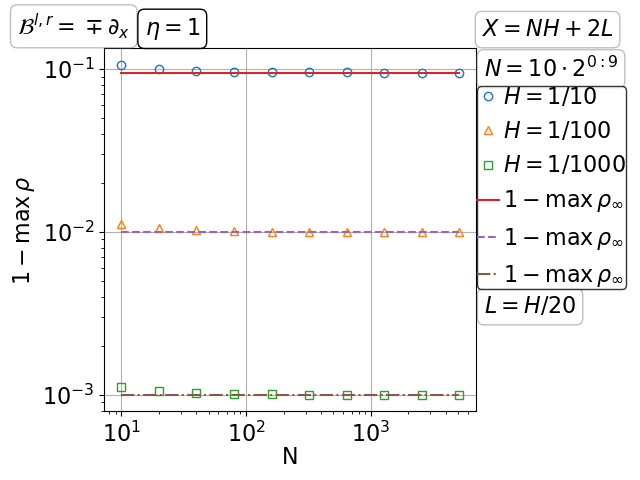}\\
    \includegraphics[width=.5\textwidth,trim=5 6 0 2,clip]{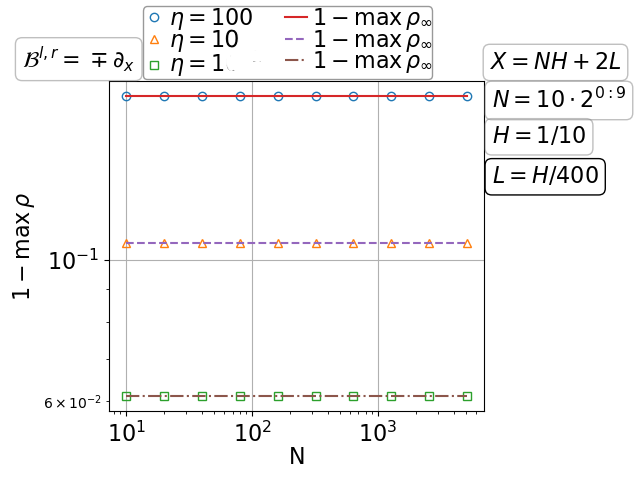}%
    \includegraphics[width=.5\textwidth,trim=5 6 0 2,clip]{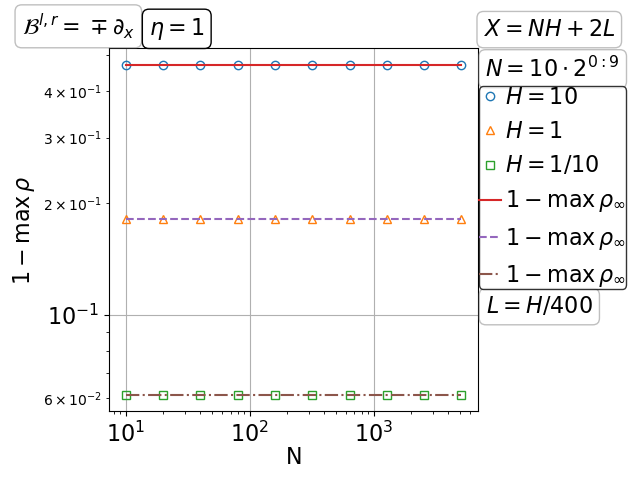}
    \caption{Convergence of the parallel Schwarz method with Taylor of order zero
      transmission for the Neumann problem of diffusion with increasing number of
      fixed size subdomains.}
    \label{figpt0nn}
  \end{figure}
  
  \begin{figure}
    \centering%
    \includegraphics[width=.56\textwidth,trim=10 10 0 6,clip]{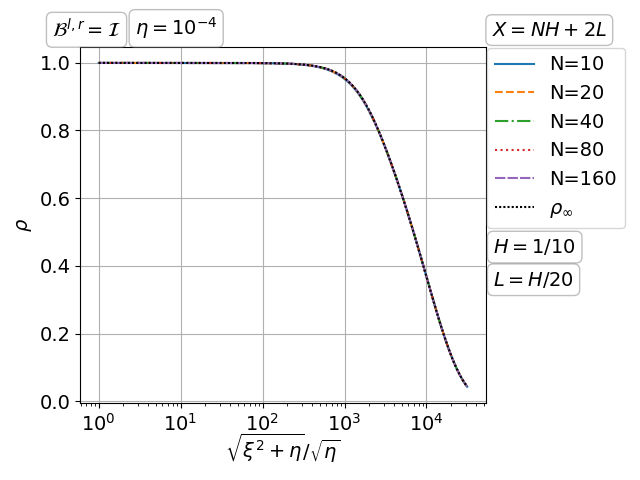}%
    \includegraphics[width=.44\textwidth,height=13em,trim=0 10 0 6,clip]{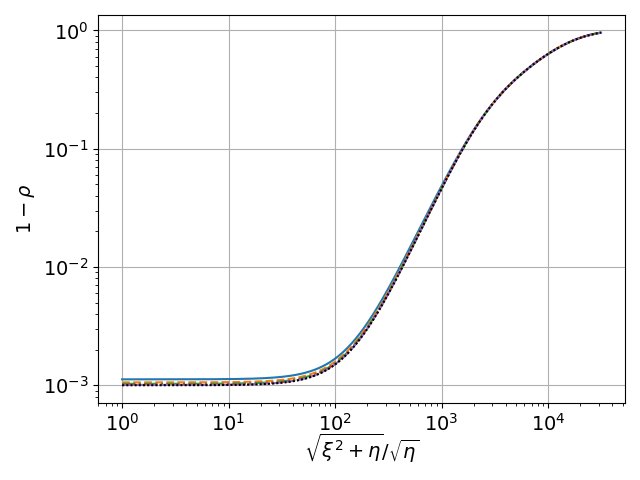}\\
    \includegraphics[width=.56\textwidth,trim=10 10 0 6,clip]{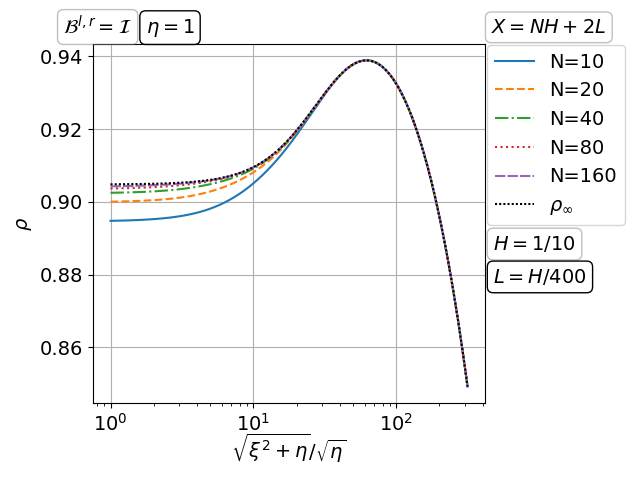}%
    \includegraphics[width=.44\textwidth,height=13em,trim=0 10 0 6,clip]{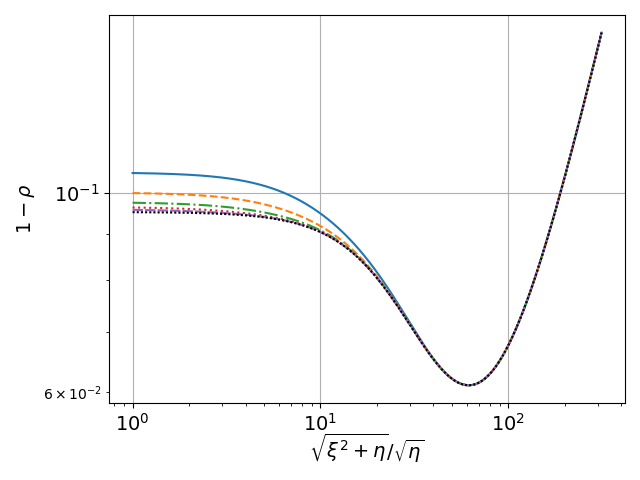}
    \includegraphics[width=.5\textwidth,trim=5 6 0 2,clip]{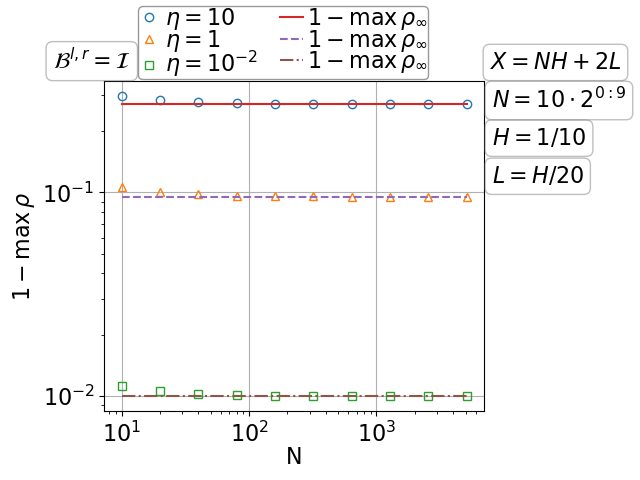}%
    \includegraphics[width=.5\textwidth,trim=5 6 0 2,clip]{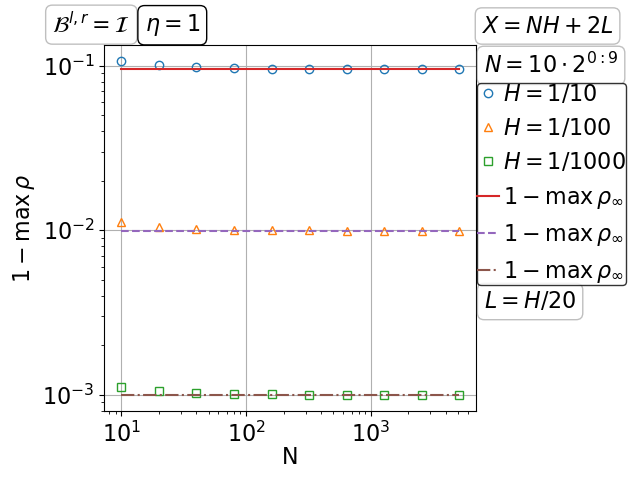}\\
    \includegraphics[width=.5\textwidth,trim=5 6 0 2,clip]{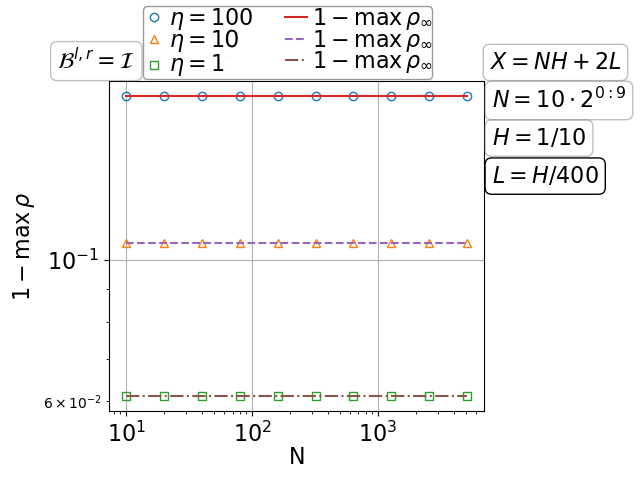}%
    \includegraphics[width=.5\textwidth,trim=5 6 0 2,clip]{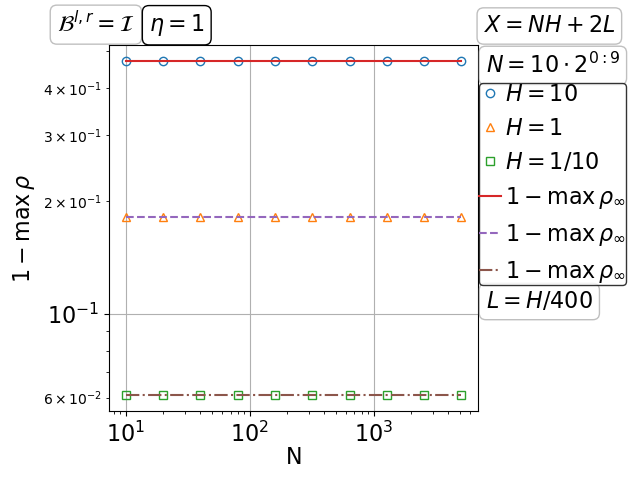}
    \caption{Convergence of the parallel Schwarz method with Taylor of order zero
      transmission for {the Dirichlet problem} of diffusion with increasing number
      of fixed size subdomains.}
    \label{figpt0nd}
  \end{figure}
  
  \begin{figure}
    \centering
    \includegraphics[width=.56\textwidth,trim=10 10 0 6,clip]{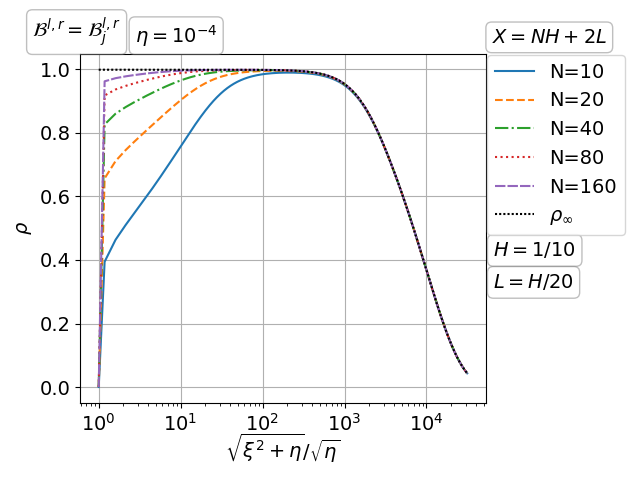}%
    \includegraphics[width=.44\textwidth,height=13em,trim=0 10 0 6,clip]{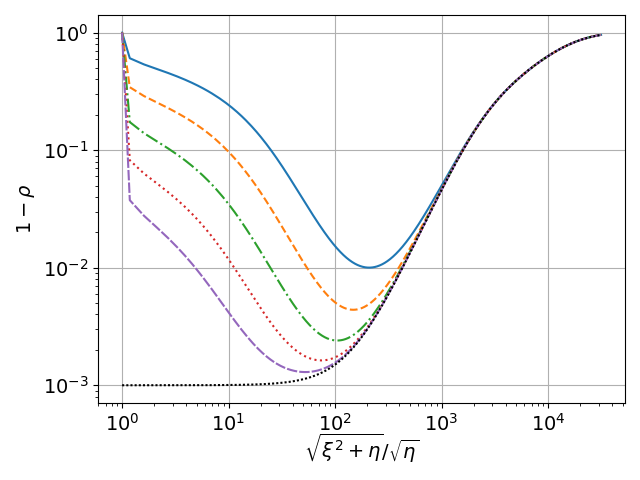}\\
    \includegraphics[width=.56\textwidth,trim=10 10 0 6,clip]{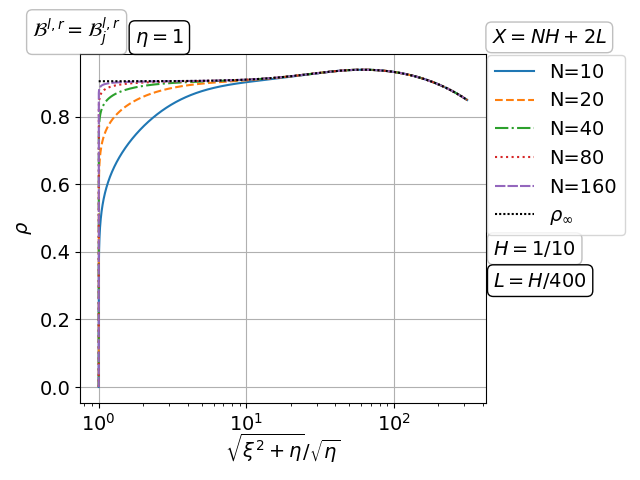}%
    \includegraphics[width=.44\textwidth,height=13em,trim=0 10 0 6,clip]{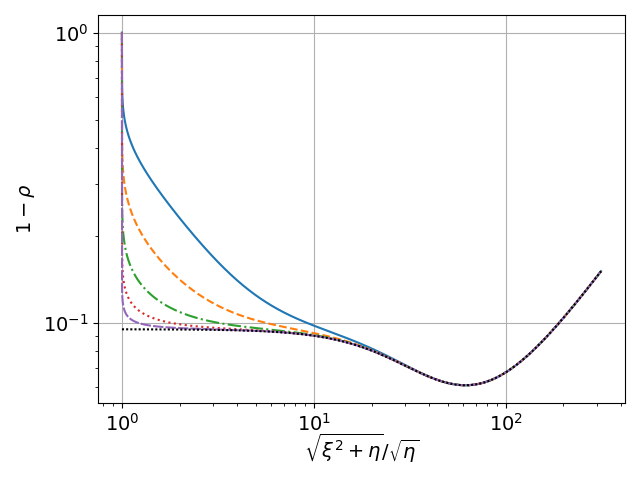}\\
   \includegraphics[width=.5\textwidth,trim=5 6 0 2,clip]{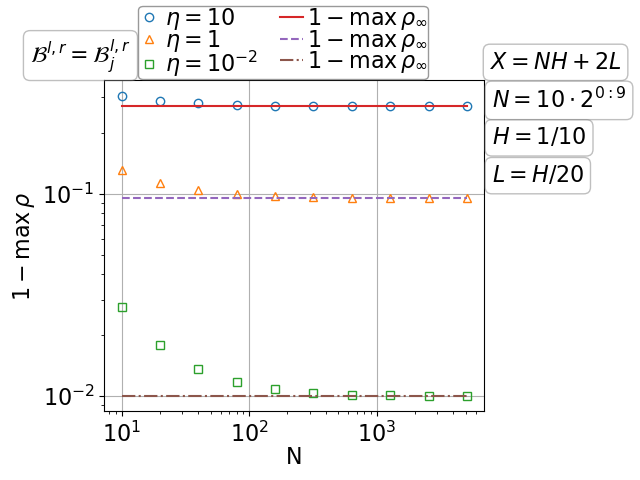}%
   \includegraphics[width=.5\textwidth,trim=5 6 0 2,clip]{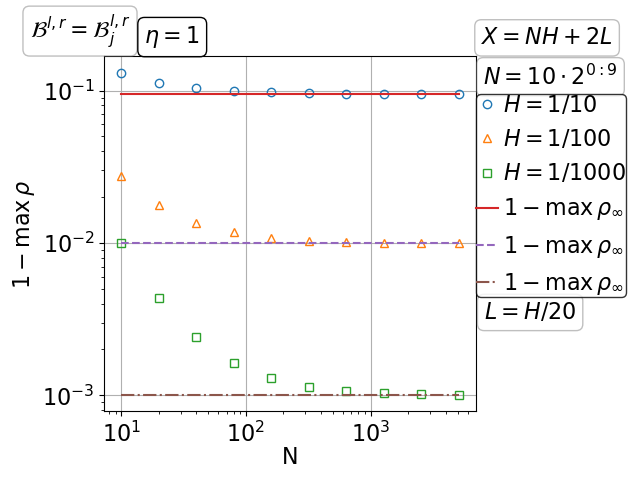}\\
   \includegraphics[width=.5\textwidth,trim=5 6 0 2,clip]{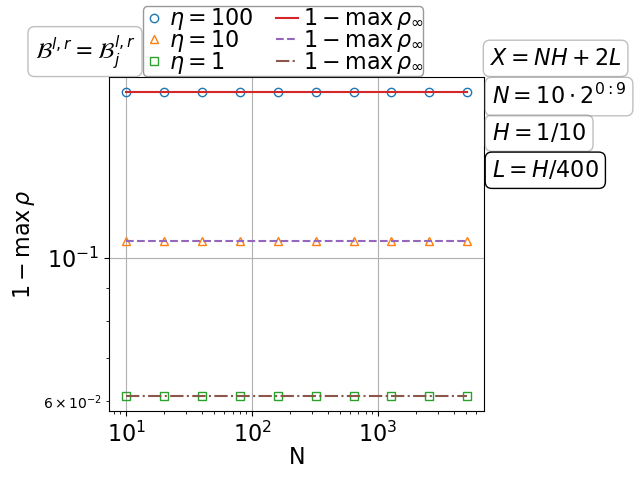}%
   \includegraphics[width=.5\textwidth,trim=5 6 0 2,clip]{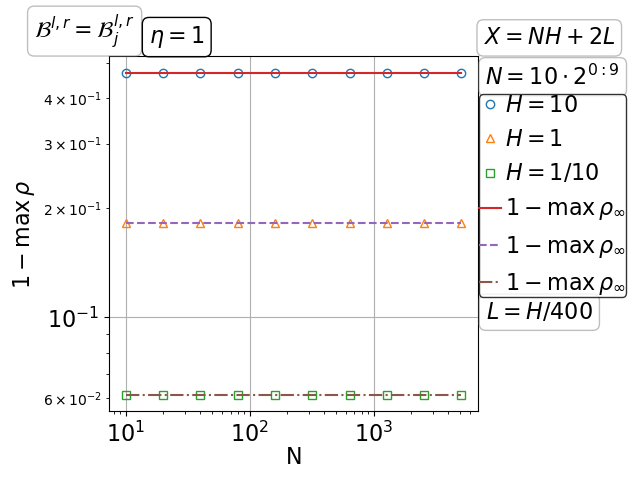}
   \caption{Convergence of the parallel Schwarz method with Taylor of order zero transmission
     for {the infinite pipe} diffusion with increasing number of fixed size
     subdomains.}
    \label{figpt0nt0}
  \end{figure}

\end{paragraph}

\begin{paragraph}{Convergence on a fixed domain with increasing number of subdomains}

  With the number of subdomains $N\to \infty$ and the domain width $X$ fixed, it follows that the
  subdomain width $H\to 0$.  The convergence for the Neumann problem is studied in
  Figure~\ref{figpt01n}.  Recall that $\rho_{\infty}:=\lim_{N\to\infty}\rho$ for fixed $H$, $L$ but
  $X=NH+2L$.  From the top half of Figure~\ref{figpt01n}, we find that the convergence factor
  $\rho(\xi)$ attains its maximum at $\xi=0$ when the overlap width $L=O(H)$ is not too small and
  at the critical point of $\rho_{\infty}(\xi)$ when $L=O(H^2)$ is sufficiently small. In both
  regimes, it holds that $\max_{\xi}\rho=1-O(N^{-1})$.  The difference is in how the hidden factor
  depends on $\eta$ and $L$: in the first regime $\max_{\xi}\rho=1-O(\sqrt{\eta}N^{-1})$
  independent of $L$, while in the latter regime
  $\max_{\xi}\rho=1-O(\eta^{1/4}\sqrt{L}N^{-1})$. Note that we have the same (up to a constant
  factor) dependence on $L$ and $\eta$ as in \R{TaylorRho} from the two-subdomain analysis. The
  convergence for the Dirichlet problem and the infinite pipe problem are studied in
  Figure~\ref{figpt01d} and Figure~\ref{figpt01t0}. Albeit the graphs of $\rho$ look different, the
  maximum of $\rho$ depends on $N$ in the same way as for the Neumann problem.

  \begin{figure}
    \centering%
    \includegraphics[width=.56\textwidth,trim=10 10 0 6,clip]{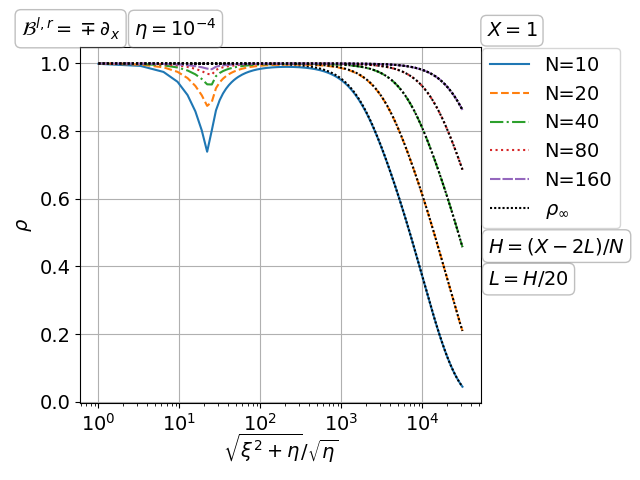}%
    \includegraphics[width=.44\textwidth,height=13em,trim=0 10 0 6,clip]{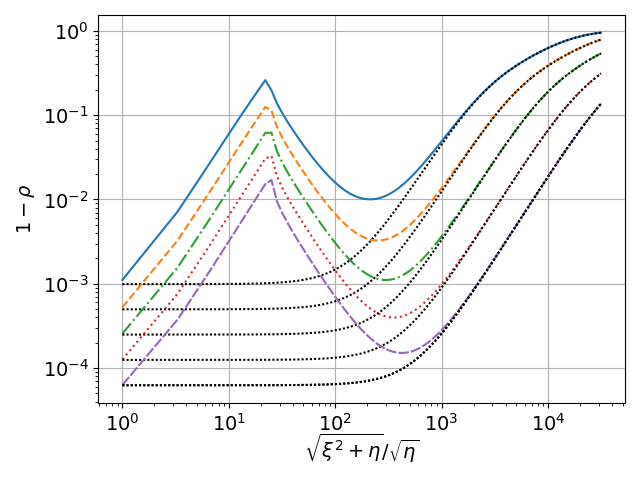}\\
    \includegraphics[width=.56\textwidth,trim=10 10 0 6,clip]{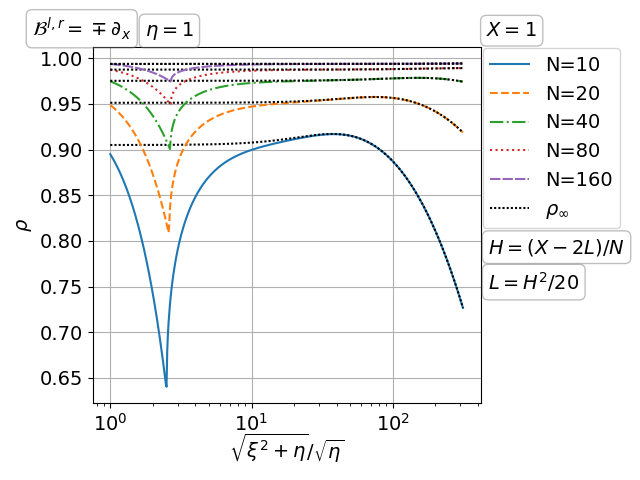}%
    \includegraphics[width=.44\textwidth,height=13em,trim=0 10 0 6,clip]{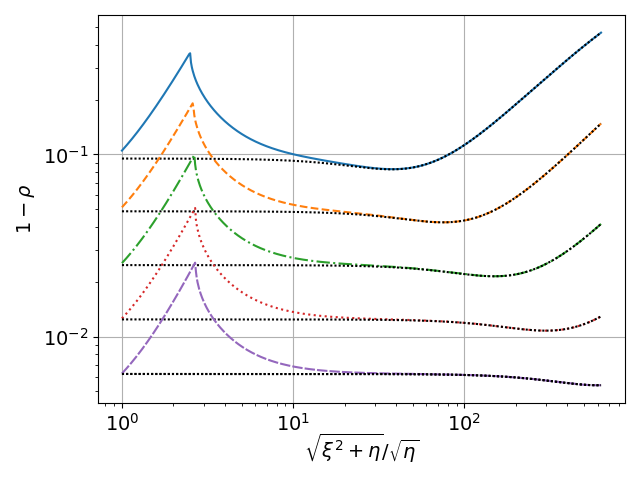}\\
    \includegraphics[width=.5\textwidth,trim=5 6 0 2,clip]{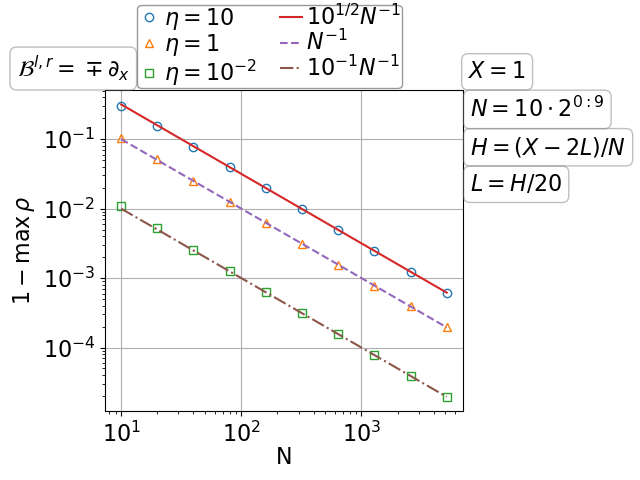}%
    \includegraphics[width=.5\textwidth,trim=5 6 0 2,clip]{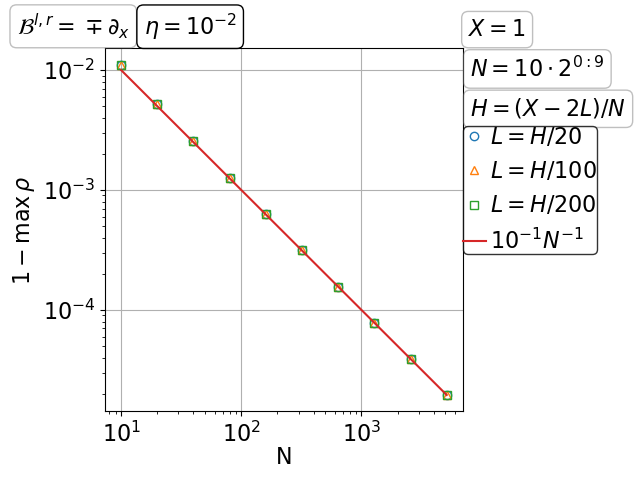}\\
    \includegraphics[width=.5\textwidth,trim=5 6 0 2,clip]{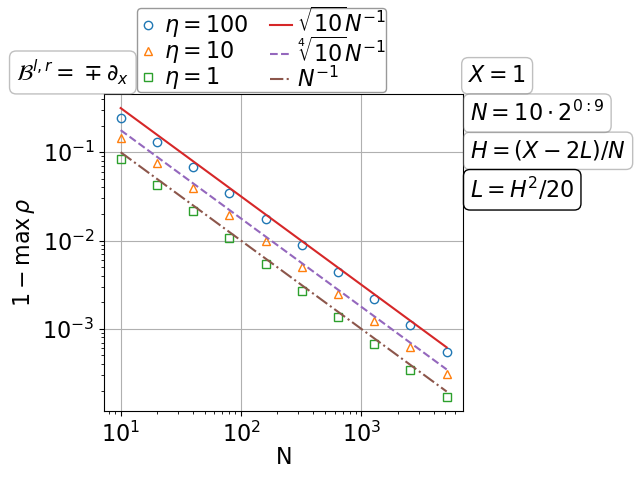}%
    \includegraphics[width=.5\textwidth,trim=5 6 0 2,clip]{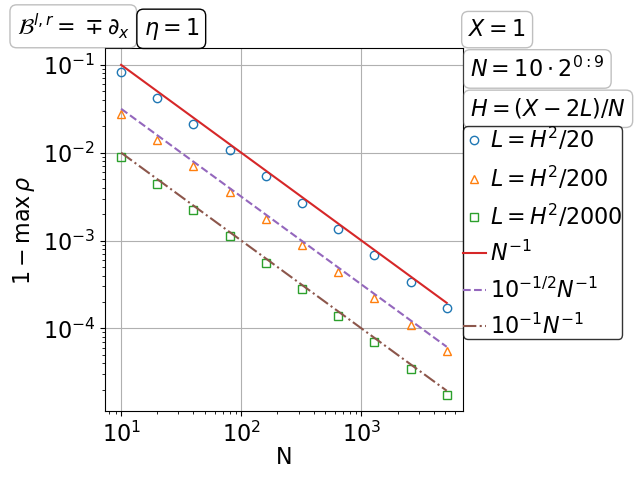}
    \caption{Convergence of the parallel Schwarz method with Taylor of order zero transmission for
      {the Neumann} problem of diffusion on a {fixed domain} with {increasing number of
        subdomains}.}
    \label{figpt01n}
  \end{figure}

  \begin{figure}
    \centering%
    \includegraphics[width=.56\textwidth,trim=10 10 0 6,clip]{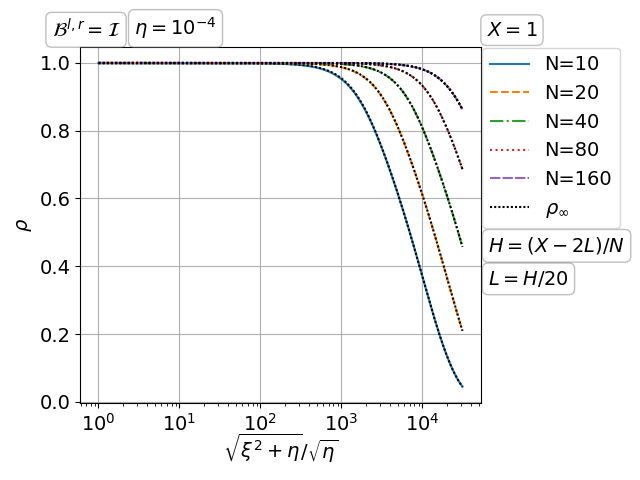}%
    \includegraphics[width=.44\textwidth,height=13em,trim=0 10 0 6,clip]{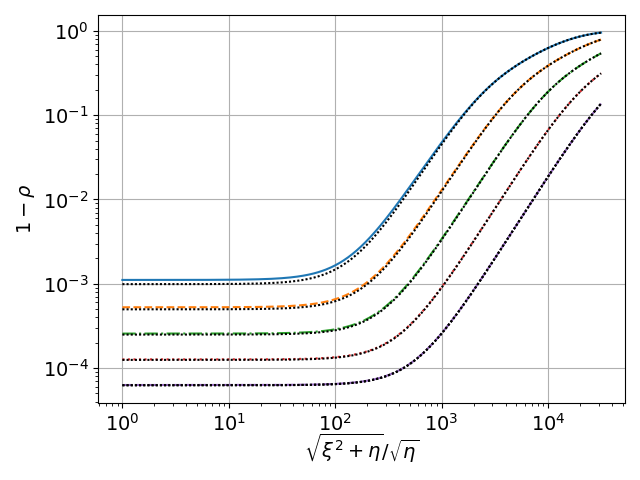}\\
    \includegraphics[width=.56\textwidth,trim=10 10 0 6,clip]{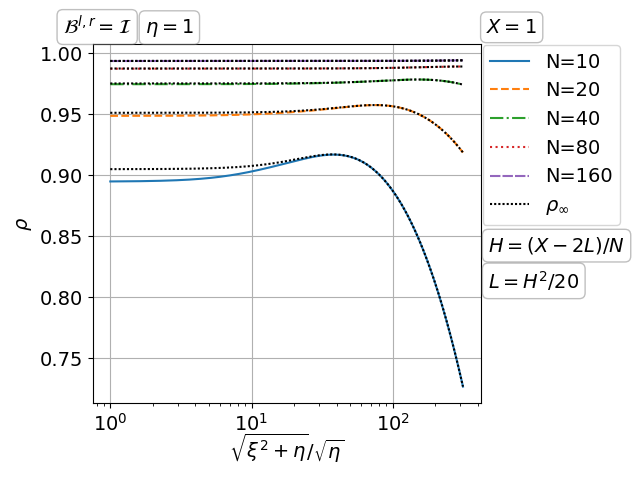}%
    \includegraphics[width=.44\textwidth,height=13em,trim=0 10 0 6,clip]{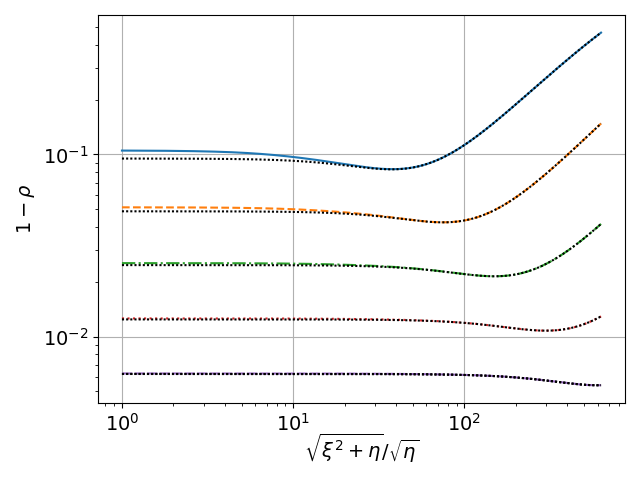}\\
    \includegraphics[width=.5\textwidth,trim=5 6 0 2,clip]{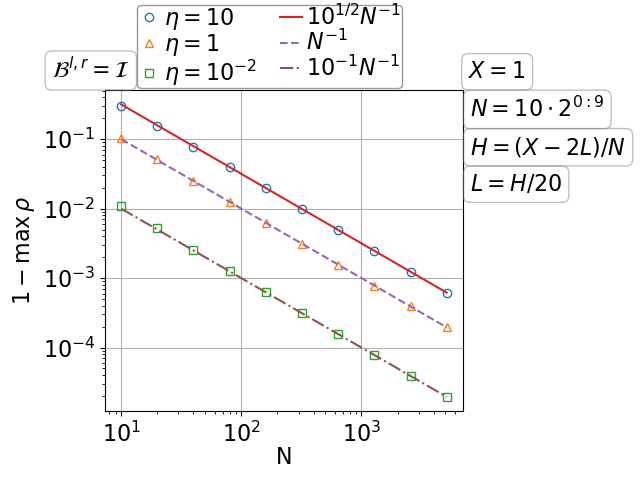}%
    \includegraphics[width=.5\textwidth,trim=5 6 0 2,clip]{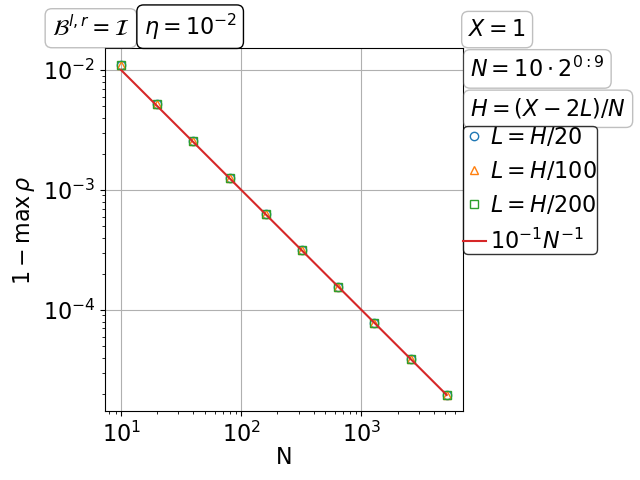}\\
    \includegraphics[width=.5\textwidth,trim=5 6 0 2,clip]{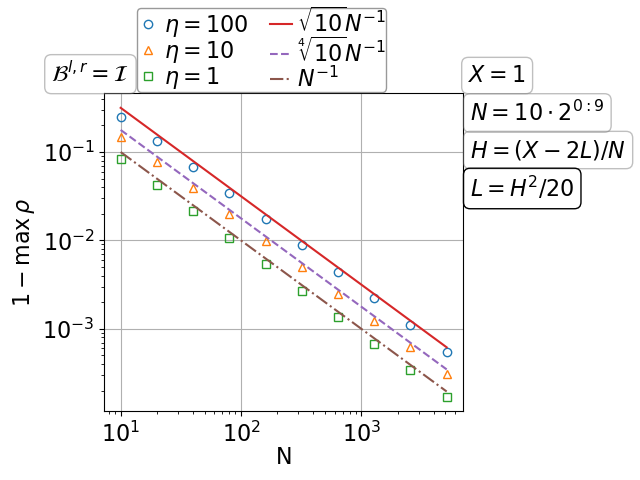}%
    \includegraphics[width=.5\textwidth,trim=5 6 0 2,clip]{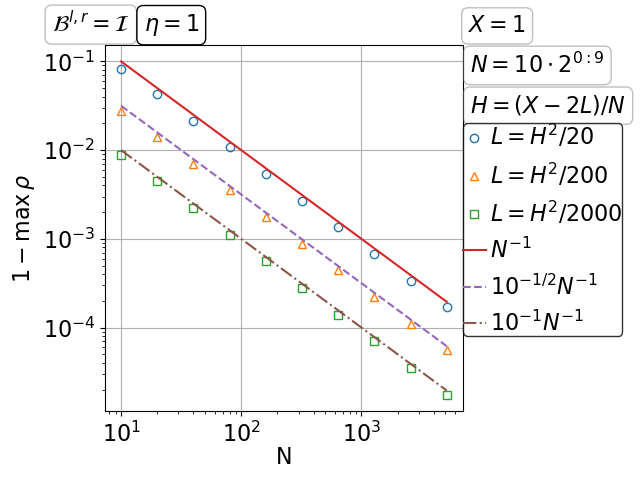}
    \caption{Convergence of the parallel Schwarz method with Taylor of order zero transmission for
      {the Dirichlet} problem of diffusion on a fixed domain with increasing number of subdomains.}
    \label{figpt01d}
  \end{figure}

  \begin{figure}
    \centering%
    \includegraphics[width=.56\textwidth,trim=10 10 0 6,clip]{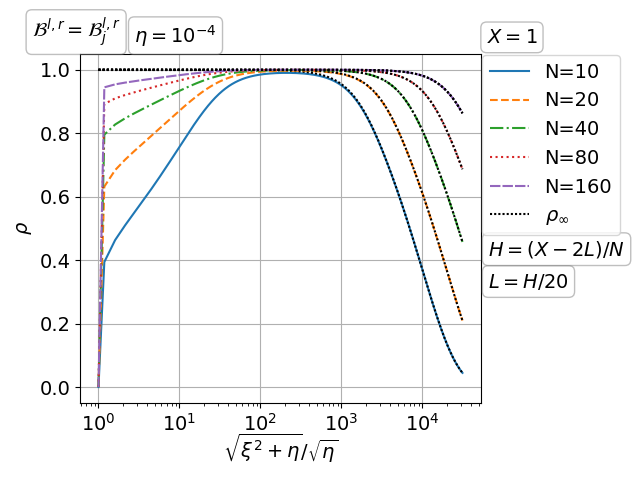}%
    \includegraphics[width=.44\textwidth,height=13em,trim=0 10 0 6,clip]{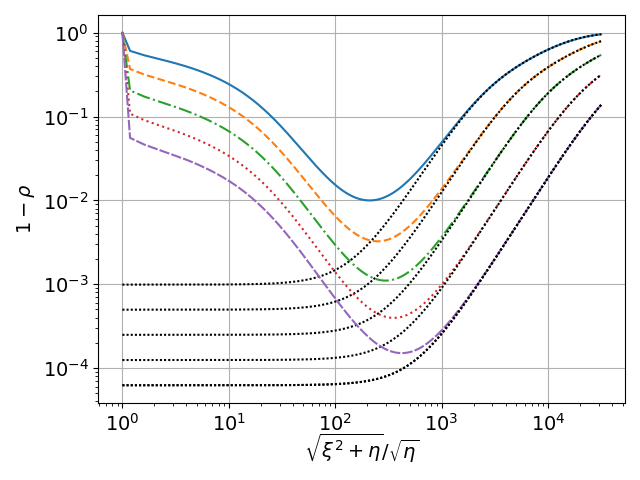}\\
    \includegraphics[width=.56\textwidth,trim=10 10 0 6,clip]{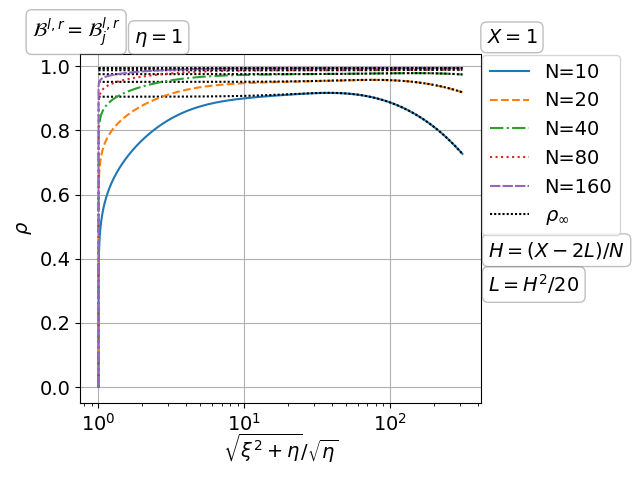}%
    \includegraphics[width=.44\textwidth,height=13em,trim=0 10 0 6,clip]{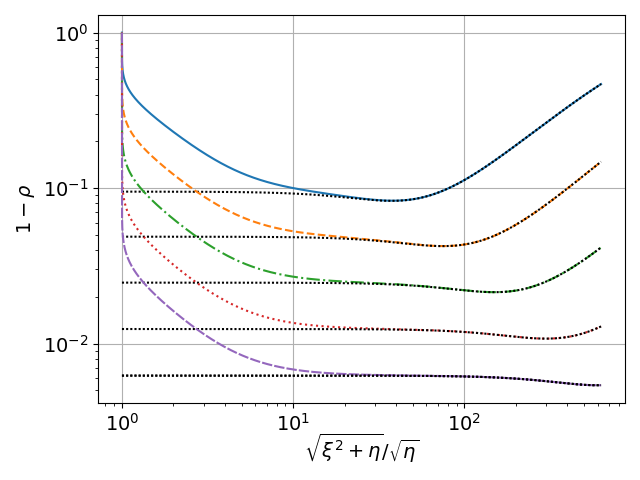}\\
    \includegraphics[width=.5\textwidth,trim=5 6 0 2,clip]{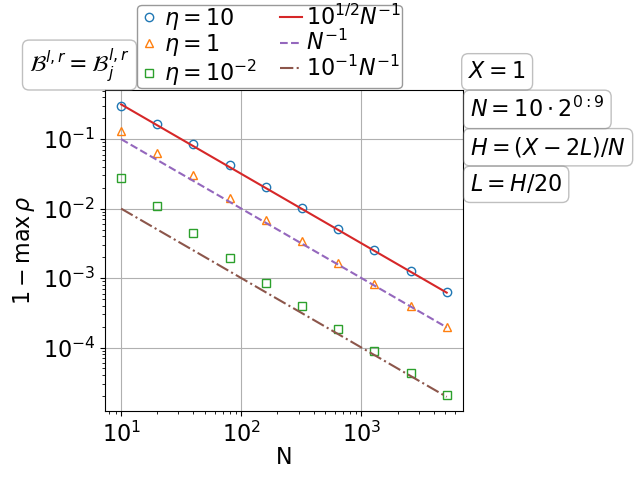}%
    \includegraphics[width=.5\textwidth,trim=5 6 0 2,clip]{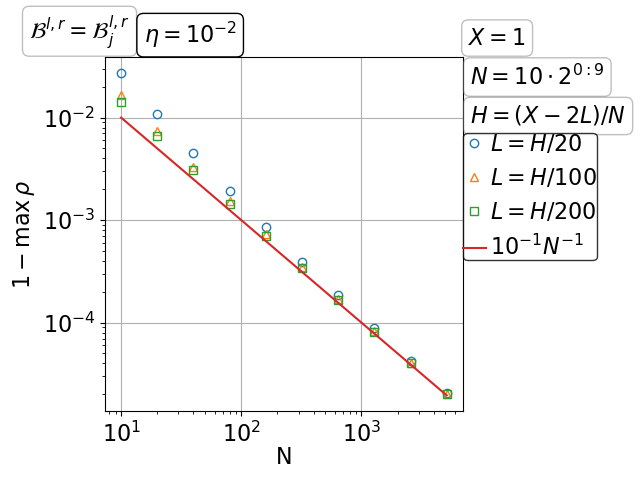}\\
    \includegraphics[width=.5\textwidth,trim=5 6 0 2,clip]{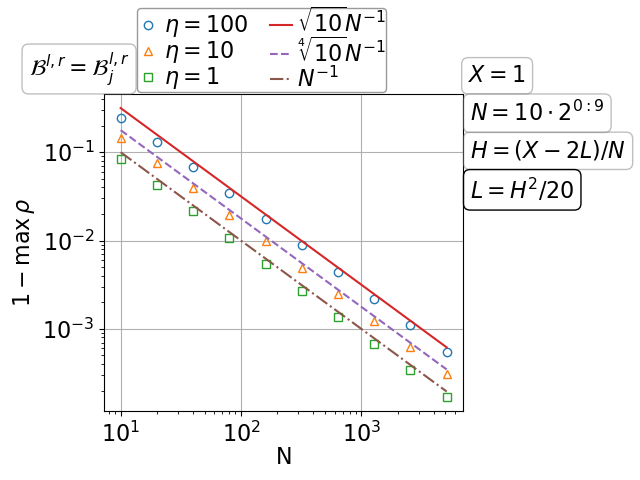}%
    \includegraphics[width=.5\textwidth,trim=5 6 0 2,clip]{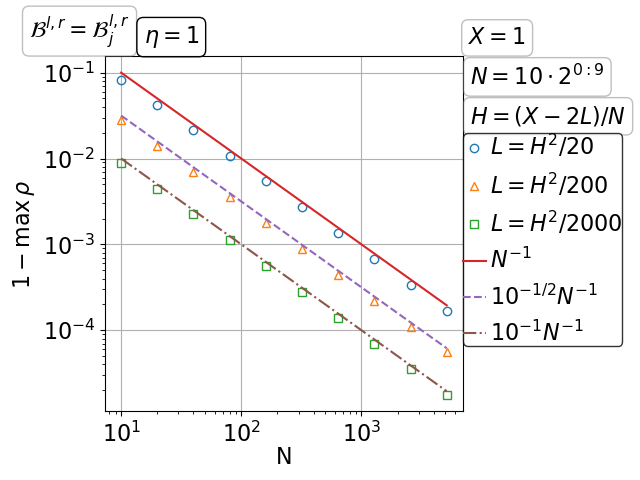}
    \caption{Convergence of the parallel Schwarz method with Taylor of order zero transmission for
      {the infinite pipe} diffusion on a fixed domain with increasing number of subdomains.}
    \label{figpt01t0}
  \end{figure}

\end{paragraph}

%%%%%%%%%%%%%%%%%%%%%%%%%%%%%%%

\subsubsection{Parallel Schwarz method with PML transmission for the diffusion problem}

By using the PML transmission condition, the subdomain problem away from the boundary
$\{0, X\}\times[0,Y]$ can be a very good domain truncation of the problem on the infinite pipe
$(-\infty, \infty)\times [0,Y]$, because PML can make the reflection coefficient
\[
      R=\frac{\hat{\mathcal{S}}-\sqrt{\xi^2+\eta}}{\hat{\mathcal{S}}+\sqrt{\xi^2+\eta}}=\pm\mathrm{e}^{-(2+\gamma)D\sqrt{\xi^2+\eta}}\qquad\text{($\hat{\mathcal{S}}$
        given in \R{eqhats})} 
\]
arbitrarily small by increasing its numerical cost related to the PML
width $D$ and PML strength $\gamma$. If the original boundary condition is also given by PML, \ie,
$\mathcal{B}^{l,r}=\mathcal{B}^{l,r}_j$, then the original problem does approximate the infinite
pipe problem and so {we can expect that the Schwarz method will perform well}. What if
$\mathcal{B}^{l,r}$ is Dirichlet or Neumann? What condition should be used on the external boundary
of the PML? We will address these questions in the following paragraphs.

\begin{paragraph}{Convergence with increasing number of fixed size subdomains}

  The study is carried out on a growing chain of fixed size subdomains. We first consider the
  Neumann problem in Figure~\ref{figppnn}.  The first row is the convergence factor $\rho(\xi)$ from
  using the Neumann condition on the PML external boundaries, while the second row is from using the
  Dirichlet condition. We see that their asymptotics $\rho_{\infty}(\xi)=\lim_{N\to\infty}\rho(\xi)$
  have little difference but the Neumann terminated PML is better than the Dirichlet terminated PML
  for moderate number of subdomains $N$, which is reasonable because the original domain $\Omega$ is
  equipped with the Neumann condition. In the bottom half of Figure~\ref{figppnn}, we see that
  $\max_{\xi}\rho=O(1)<1$ with the constant linearly dependent on the coefficient $\eta$ and the
  PML width $D$. If we compare this figure with Figure~\ref{figpdn} and Figure~\ref{figpdl}, we can
  find many similarities. That is, the PML for diffusion behaves like an overlap ({see also patch substructuring methods \cite{gander2007analysis}, and references therein}). But the PML can be
  of arbitrary width, while the overlap width can not exceed the subdomain width. The same scaling
  is observed for the Dirichlet problem in Figure~\ref{figppnd} and for the truncated infinite pipe
  problem in Figure~\ref{figppnp}.  For a moderate number of subdomains, the Dirichlet terminated PML is
  favorable for the original Dirichlet problem.

  \begin{figure}
    \centering%
    \includegraphics[width=.56\textwidth,trim=10 10 0 6,clip]{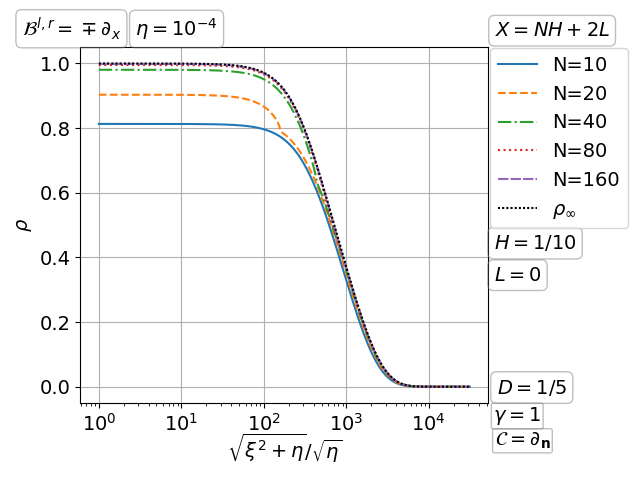}%
    \includegraphics[width=.44\textwidth,height=13em,trim=0 10 0 6,clip]{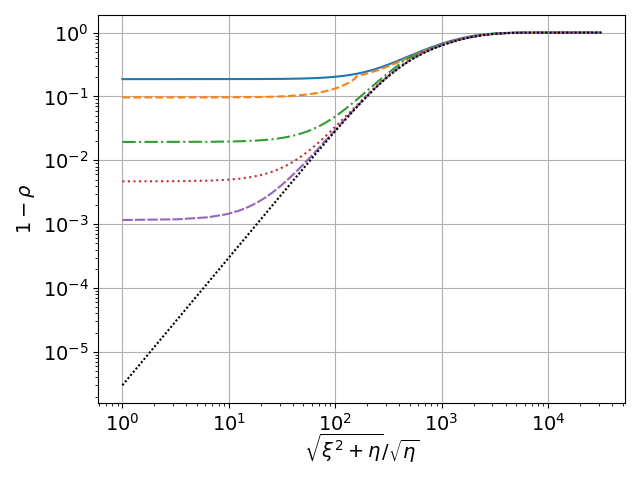}\\
    \includegraphics[width=.56\textwidth,trim=10 10 0 6,clip]{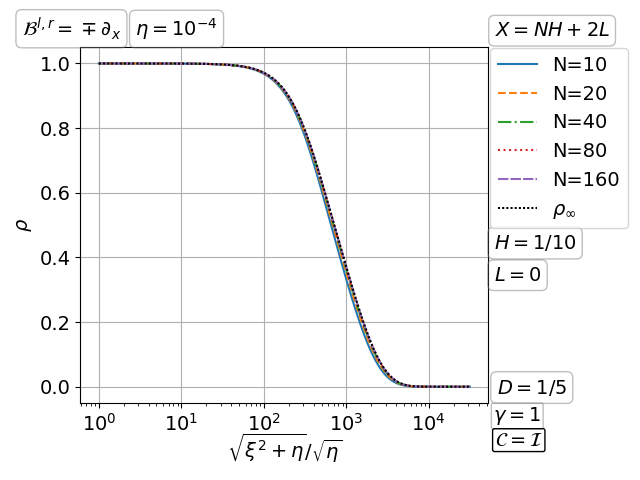}%
    \includegraphics[width=.44\textwidth,height=13em,trim=0 10 0 6,clip]{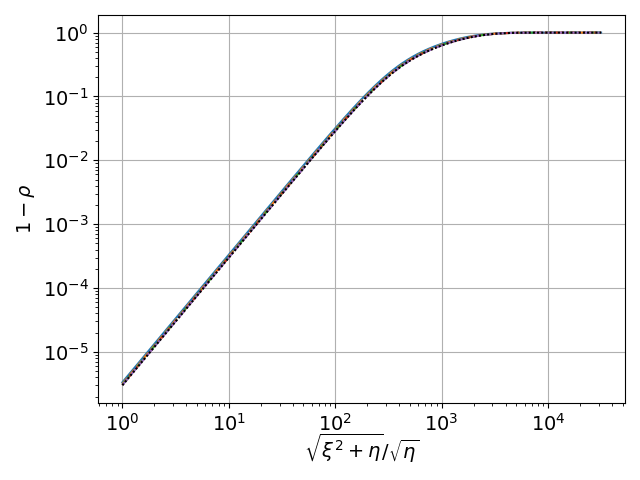}\\
    \includegraphics[width=.5\textwidth,trim=5 6 0 2,clip]{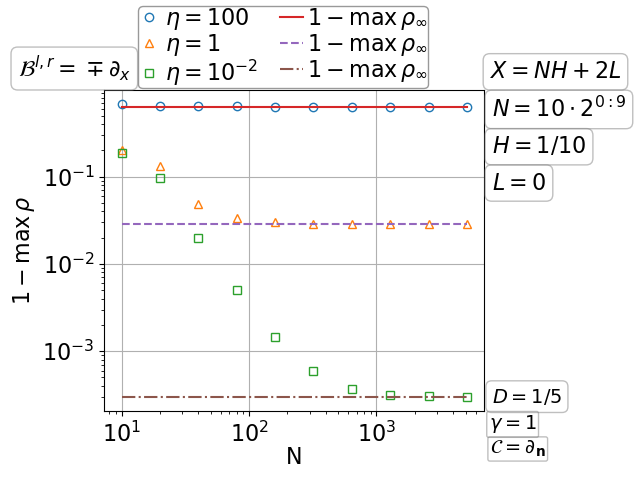}%
    \includegraphics[width=.5\textwidth,trim=5 6 0 2,clip]{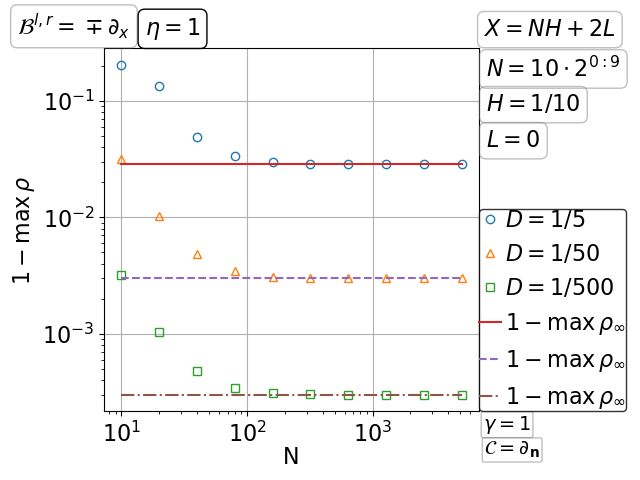}\\
    \includegraphics[width=.5\textwidth,trim=5 6 0 2,clip]{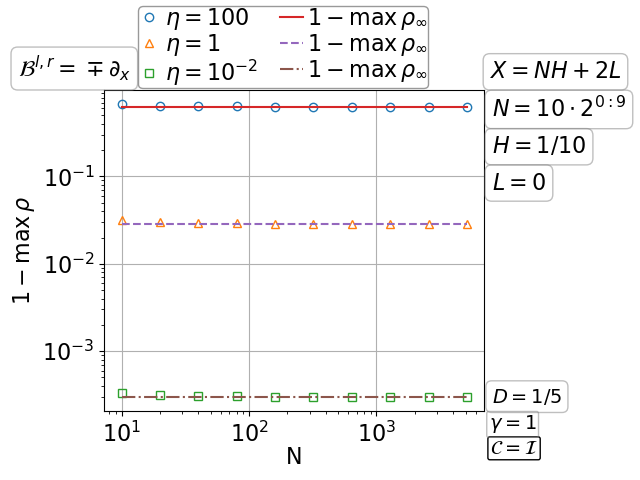}%
    \includegraphics[width=.5\textwidth,trim=5 6 0 2,clip]{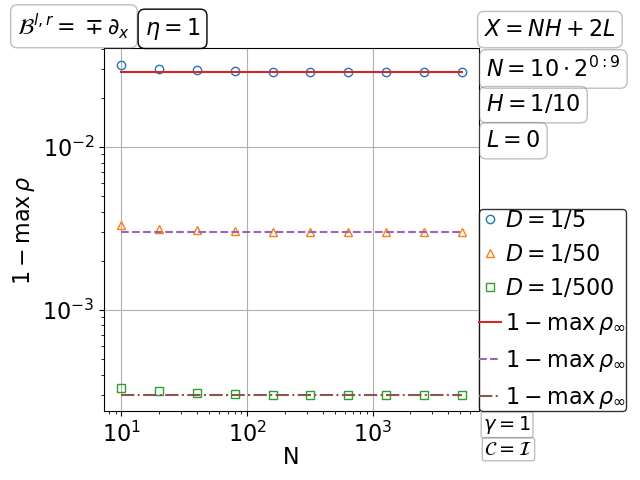}
    \caption{Convergence of the parallel Schwarz method with PML transmission for the Neumann
      problem of diffusion with increasing number of fixed size subdomains.}
    \label{figppnn}
  \end{figure}

  \begin{figure}
    \centering%
    \includegraphics[width=.56\textwidth,trim=10 10 0 6,clip]{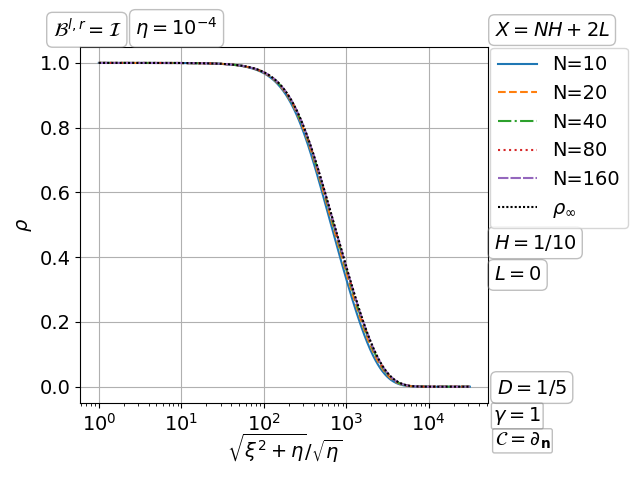}%
    \includegraphics[width=.44\textwidth,height=13em,trim=0 10 0 6,clip]{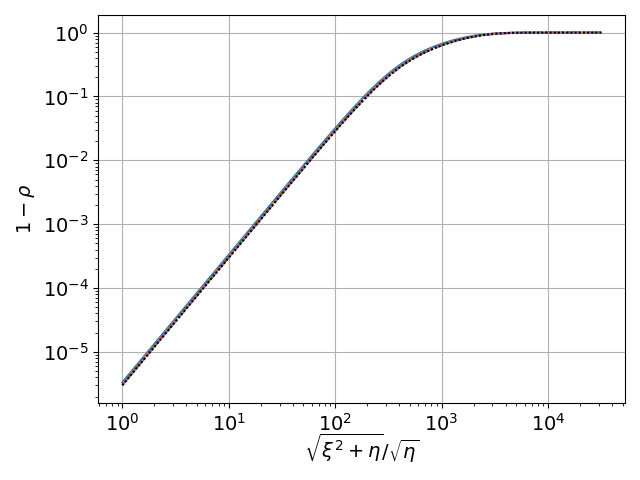}\\
    \includegraphics[width=.56\textwidth,trim=10 10 0 6,clip]{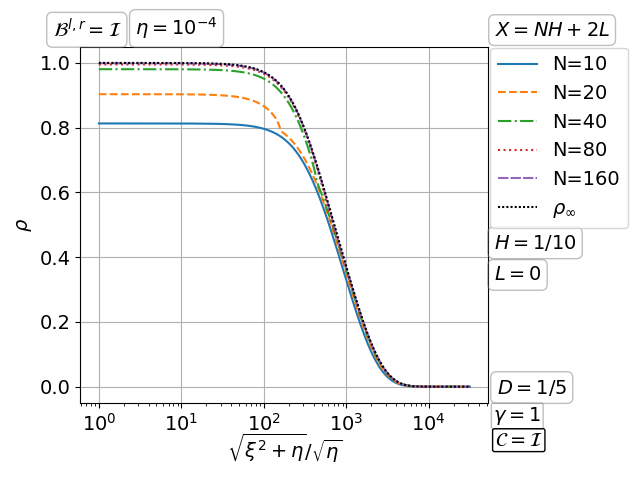}%
    \includegraphics[width=.44\textwidth,height=13em,trim=0 10 0 6,clip]{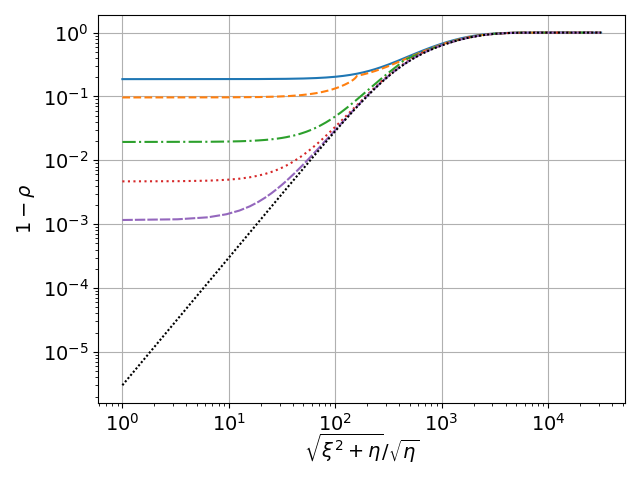}\\
    \includegraphics[width=.5\textwidth,trim=5 6 0 2,clip]{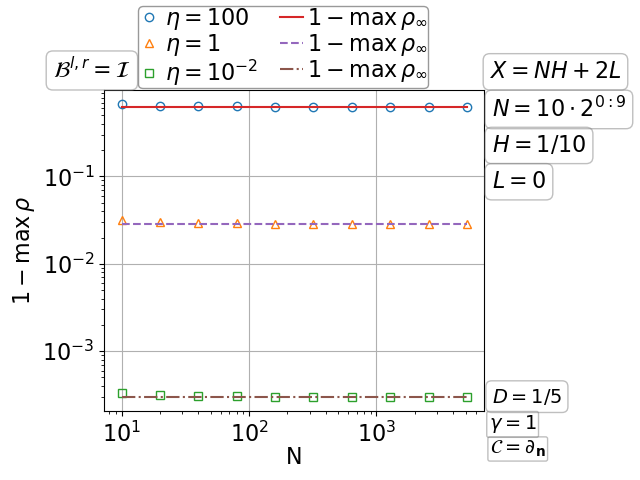}%
    \includegraphics[width=.5\textwidth,trim=5 6 0 2,clip]{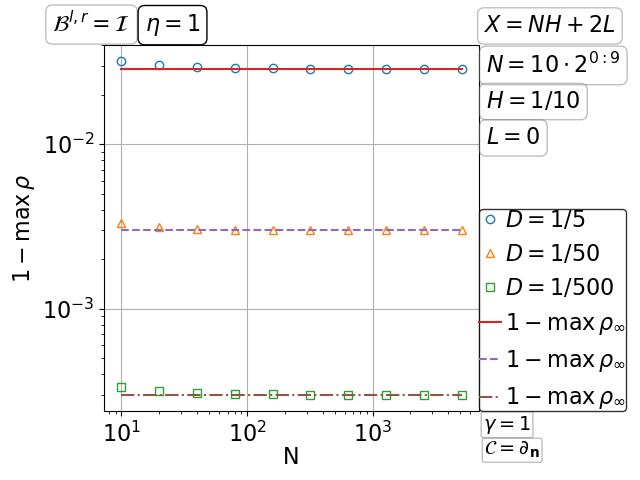}\\
    \includegraphics[width=.5\textwidth,trim=5 6 0 2,clip]{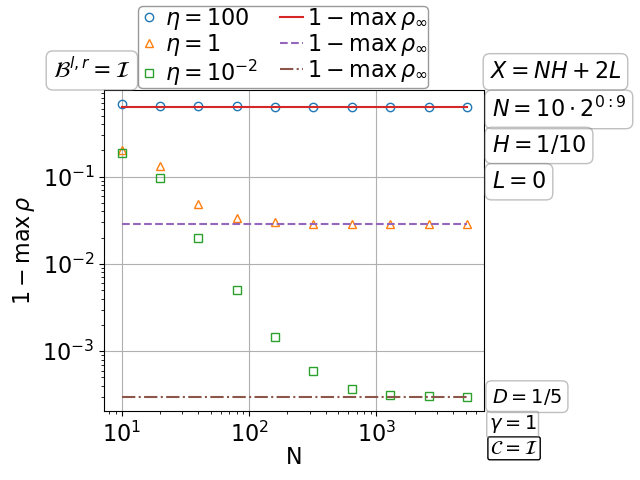}%
    \includegraphics[width=.5\textwidth,trim=5 6 0 2,clip]{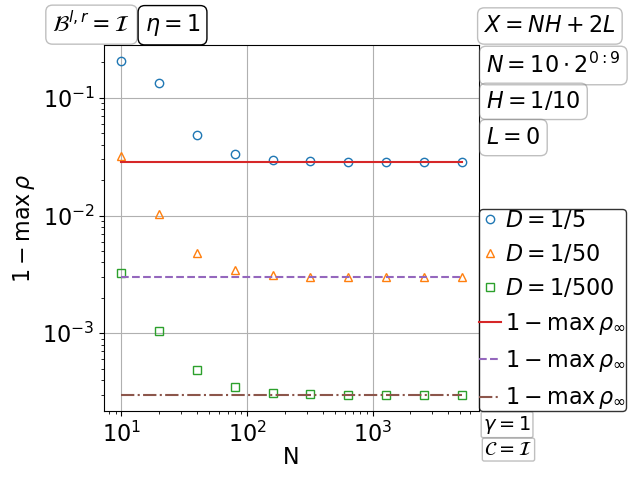}
    \caption{Convergence of the parallel Schwarz method with PML transmission for the
      Dirichlet problem of diffusion with increasing number of fixed size subdomains.}
    \label{figppnd}
  \end{figure}

  \begin{figure}
    \centering%
    \includegraphics[width=.56\textwidth,trim=10 10 0 6,clip]{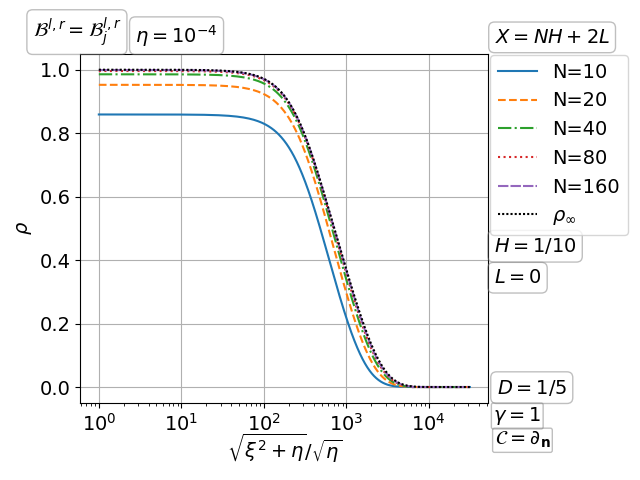}%
    \includegraphics[width=.44\textwidth,height=13em,trim=0 10 0 6,clip]{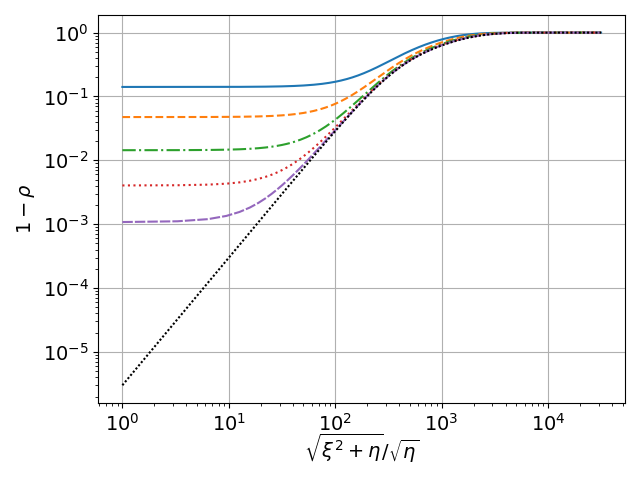}\\
    \includegraphics[width=.56\textwidth,trim=10 10 0 6,clip]{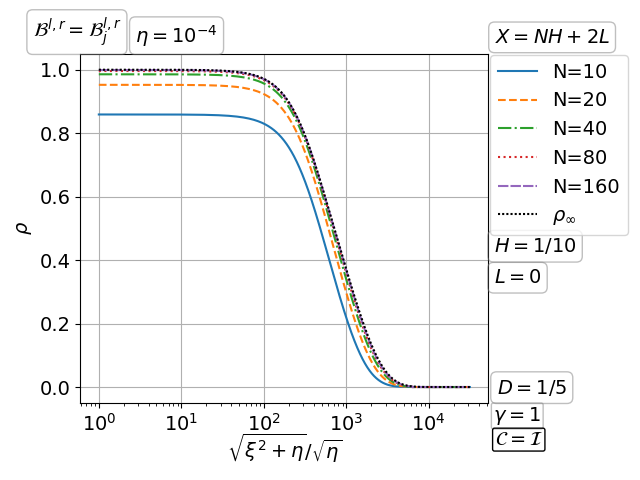}%
    \includegraphics[width=.44\textwidth,height=13em,trim=0 10 0 6,clip]{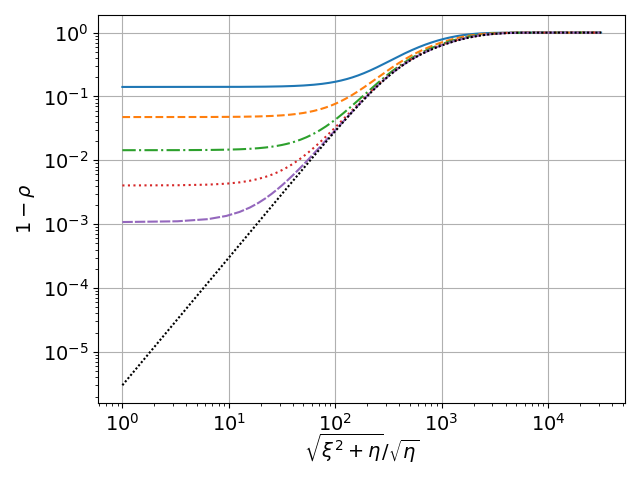}\\
    \includegraphics[width=.5\textwidth,trim=5 6 0 2,clip]{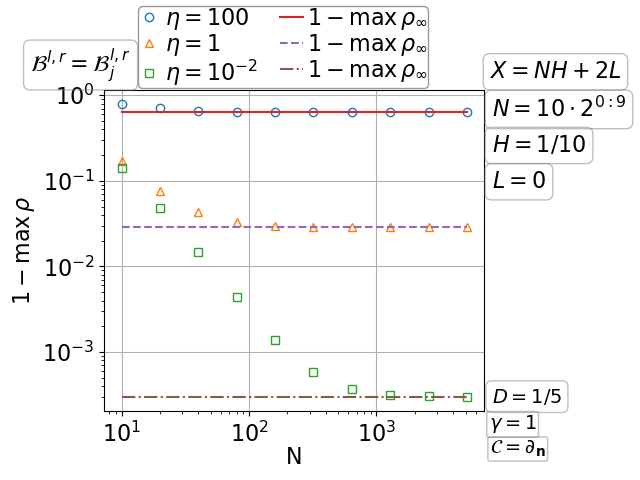}%
    \includegraphics[width=.5\textwidth,trim=5 6 0 2,clip]{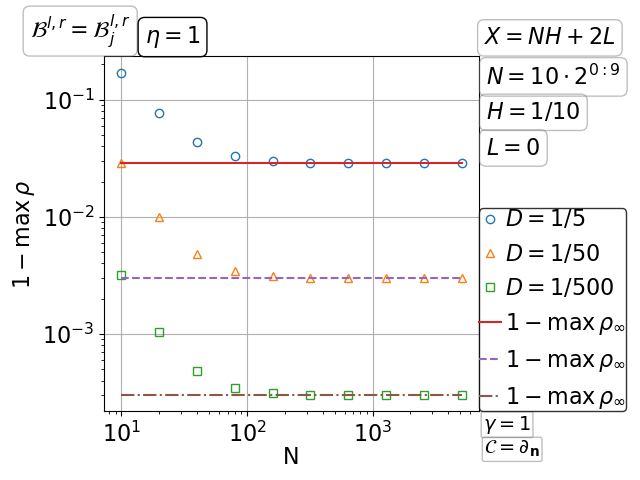}\\
    \includegraphics[width=.5\textwidth,trim=5 6 0 2,clip]{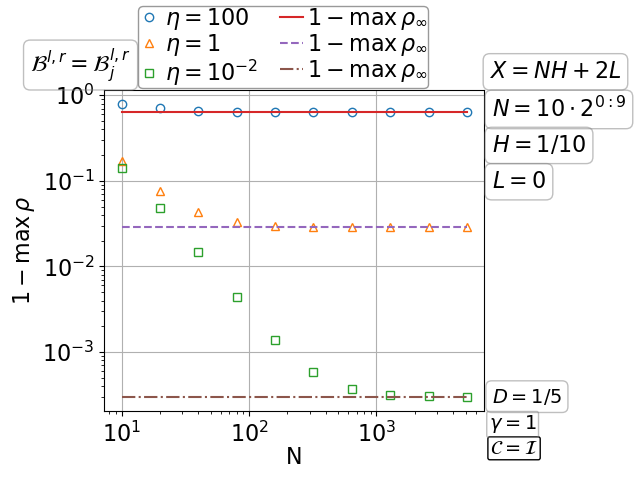}%
    \includegraphics[width=.5\textwidth,trim=5 6 0 2,clip]{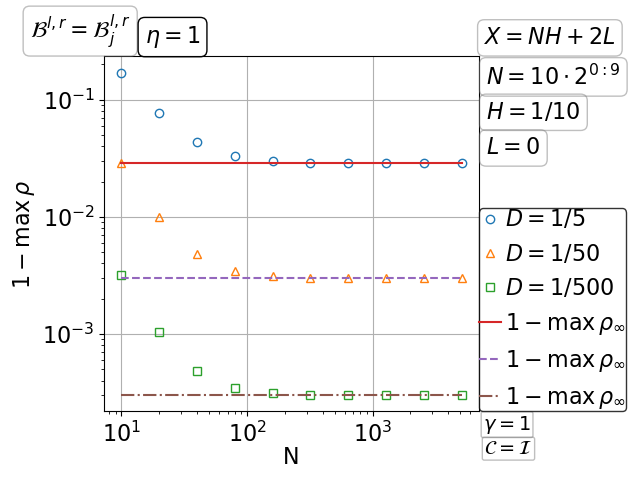}
    \caption{Convergence of the parallel Schwarz method with PML transmission for the infinite
      pipe diffusion with increasing number of fixed size subdomains.}
    \label{figppnp}
  \end{figure}

\end{paragraph}

\begin{paragraph}{Convergence on a fixed domain with increasing number of subdomains}

  In this scaling, the subdomain width $H=(X-2L)/N\to 0$ as the number of subdomains $N\to \infty$
  and the domain width $X$ is fixed. Different from the overlap width $L$, the PML width $D$ is not
  bound to $H$ and so it can be fixed.  From the top half of Figure~\ref{figpp1n}, we can find that
  $\rho_{\infty}=\lim_{N\to\infty}\rho$ (the limit taken for each fixed $H$) is tending to the
  constant one as $H\to0$ and the change of $1-\rho$ with growing $N$ looks the same as the change
  of $1-\rho_{\infty}$.  The convergence deterioration is estimated in the bottom half of
  Figure~\ref{figpp1n} as $1-\max_{\xi}\rho=O(N^{-1})$ with a linear dependence on the PML width
  $D$.  The condition on the external boundary of the PML plays a significant role for small
  $\eta>0$; see the first column of the bottom half of Figure~\ref{figpp1n}.  If the PML is
  terminated with the same condition as for the original domain, \ie,
  $\mathcal{C}=\mathcal{B}^{l,r}$, then $1-\max_{\xi}\rho$ tends to be robust for small $\eta>0$;
  otherwise, the convergence deteriorates with $\eta\to0^{+}$. A similar phenomenon appears for the
  classical Dirichlet transmission, see the earlier Figure~\ref{figpd1}, in particular, the first
  column of the bottom half. All the above observations apply equally to the Dirichlet problem
    (Figure~\ref{figpp1d}) and the infinite pipe problem (Figure~\ref{figpp1p}).

  \begin{figure}
    \centering%
    \includegraphics[width=.56\textwidth,trim=10 10 0 6,clip]{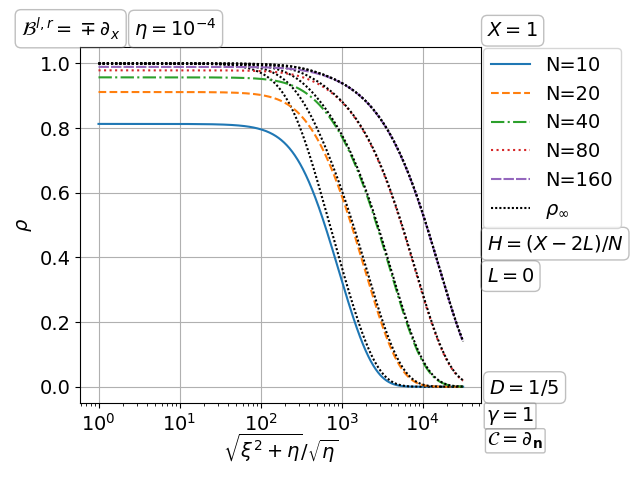}%
    \includegraphics[width=.44\textwidth,height=13em,trim=0 10 0 6,clip]{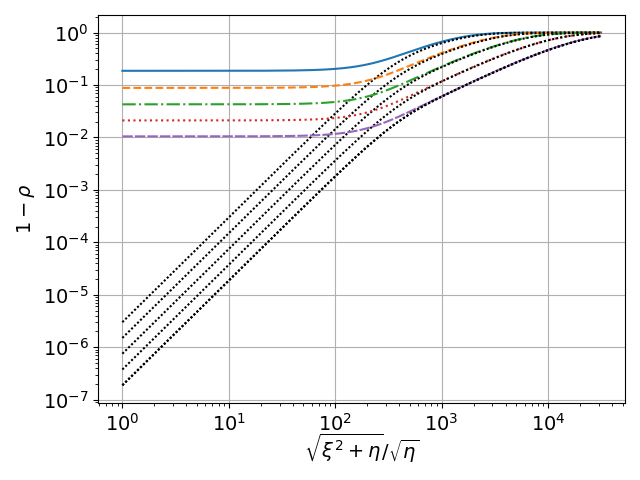}\\
    \includegraphics[width=.56\textwidth,trim=10 10 0 6,clip]{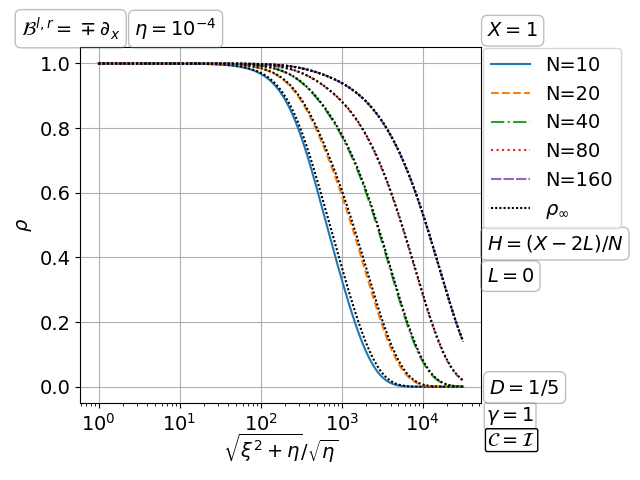}%
    \includegraphics[width=.44\textwidth,height=13em,trim=0 10 0 6,clip]{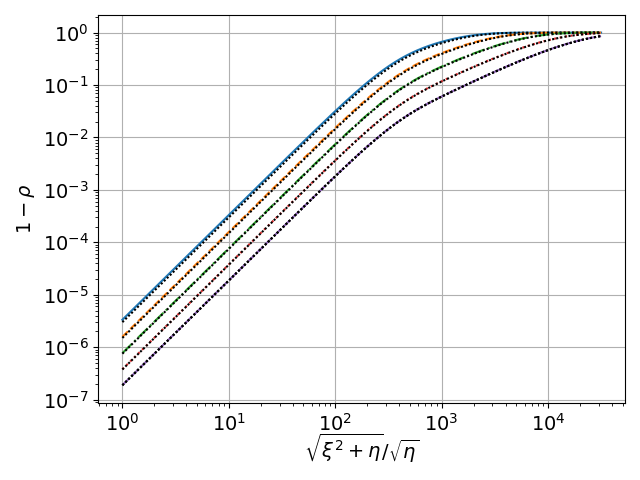}\\
    \includegraphics[width=.5\textwidth,trim=5 6 0 2,clip]{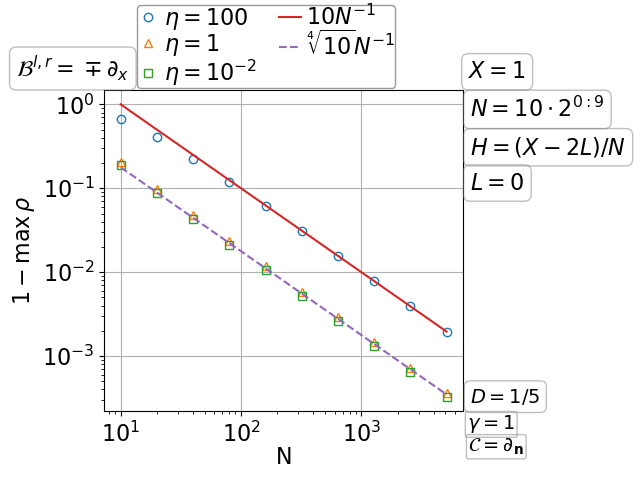}%
    \includegraphics[width=.5\textwidth,trim=5 6 0 2,clip]{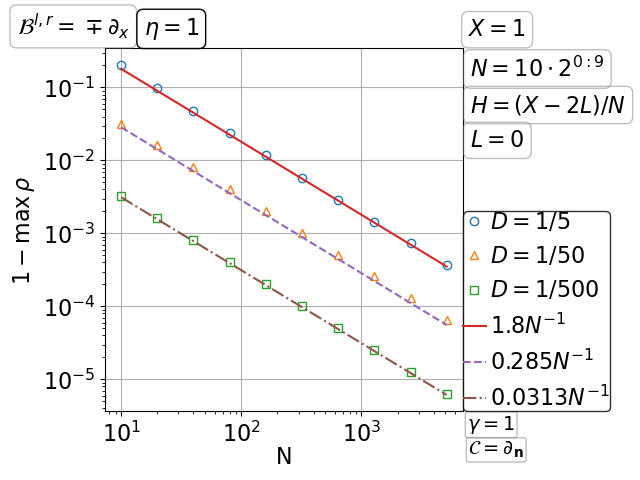}\\
    \includegraphics[width=.5\textwidth,trim=5 6 0 2,clip]{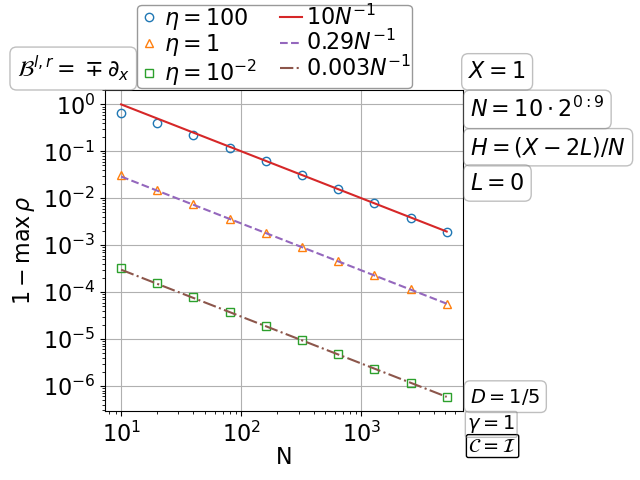}%
    \includegraphics[width=.5\textwidth,trim=5 6 0 2,clip]{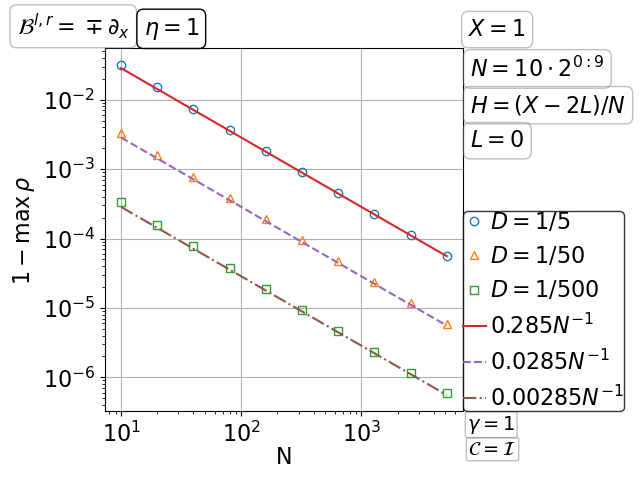}
    \caption{Convergence of the parallel Schwarz method with PML transmission for the Neumann
      problem of diffusion on a fixed domain with increasing number of subdomains.}
    \label{figpp1n}
  \end{figure}

  \begin{figure}
    \centering%
    \includegraphics[width=.56\textwidth,trim=10 10 0 6,clip]{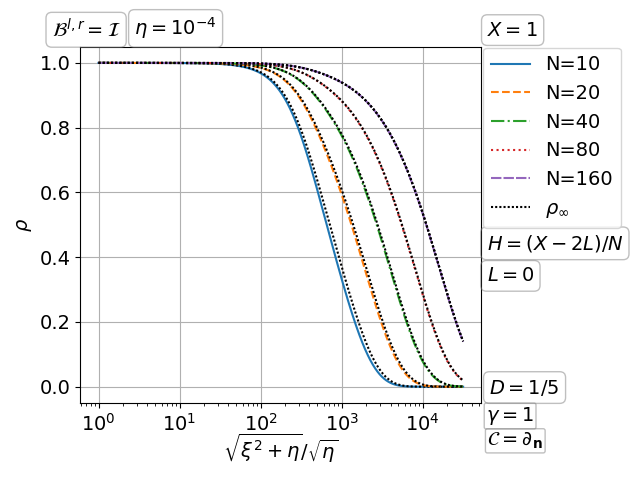}%
    \includegraphics[width=.44\textwidth,height=13em,trim=0 10 0 6,clip]{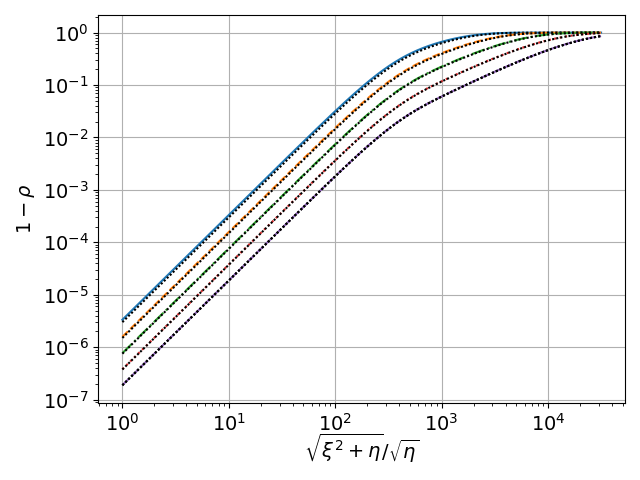}\\
    \includegraphics[width=.56\textwidth,trim=10 10 0 6,clip]{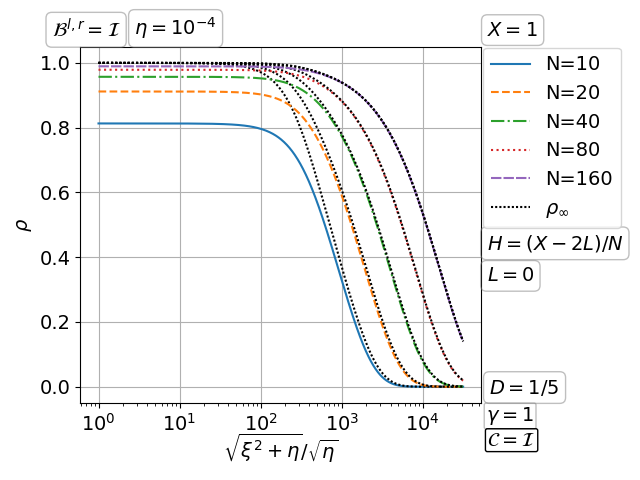}%
    \includegraphics[width=.44\textwidth,height=13em,trim=0 10 0 6,clip]{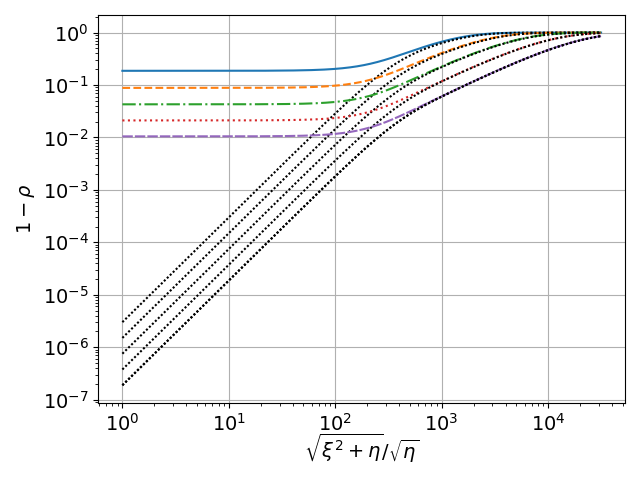}\\
    \includegraphics[width=.5\textwidth,trim=5 6 0 2,clip]{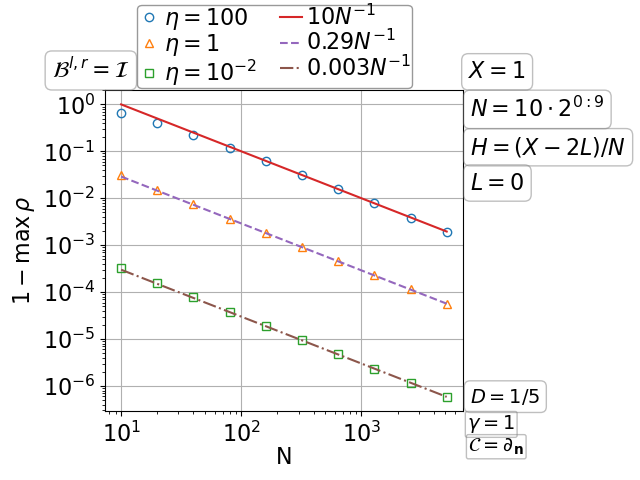}%
    \includegraphics[width=.5\textwidth,trim=5 6 0 2,clip]{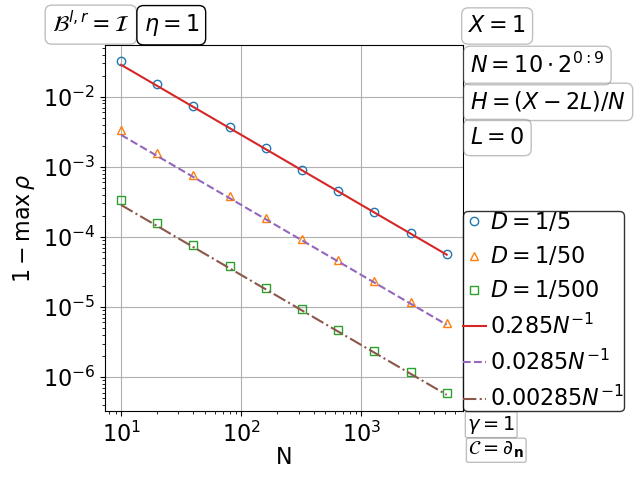}\\
    \includegraphics[width=.5\textwidth,trim=5 6 0 2,clip]{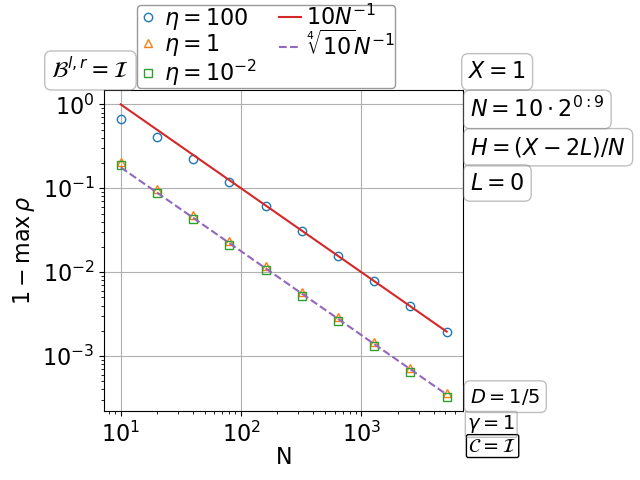}%
    \includegraphics[width=.5\textwidth,trim=5 6 0 2,clip]{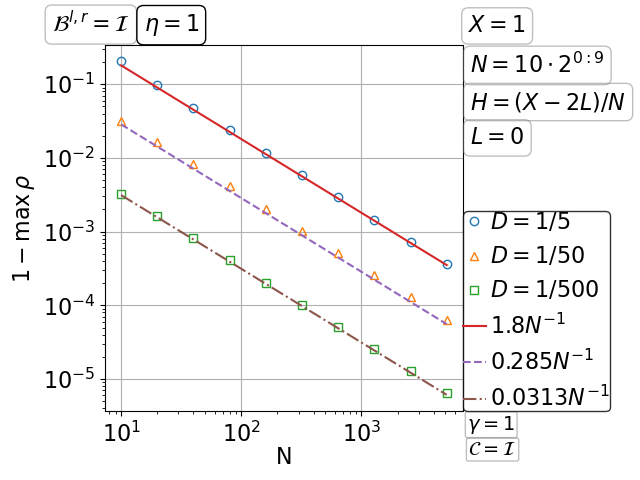}
    \caption{Convergence of the parallel Schwarz method with PML transmission for the
      Dirichlet problem of diffusion on a fixed domain with increasing number of subdomains.}
    \label{figpp1d}
  \end{figure}

  \begin{figure}
    \centering%
    \includegraphics[width=.56\textwidth,trim=10 10 0 6,clip]{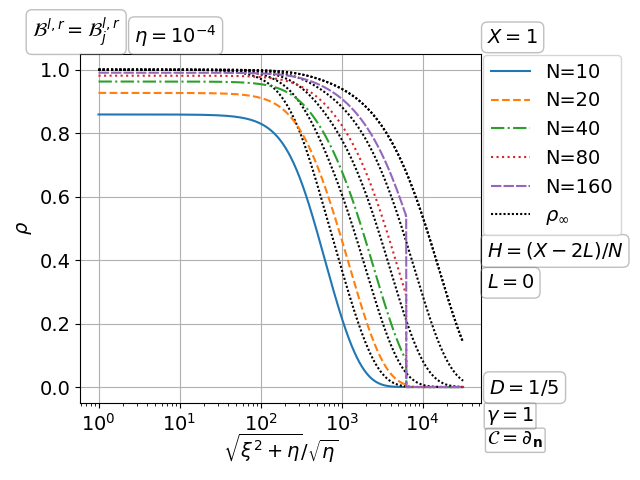}%
    \includegraphics[width=.44\textwidth,height=13em,trim=0 10 0 6,clip]{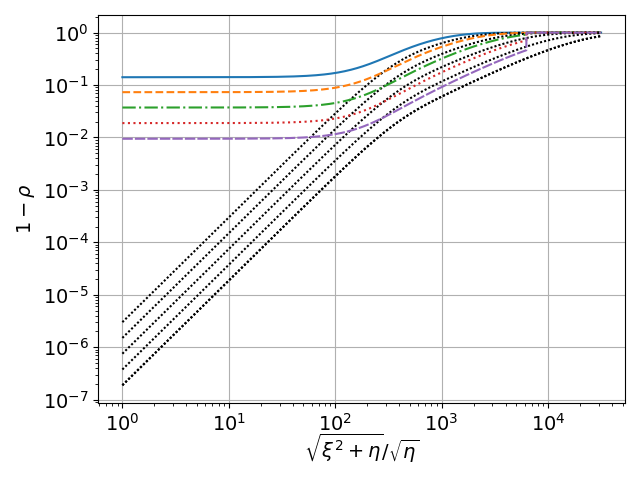}\\
    \includegraphics[width=.56\textwidth,trim=10 10 0 6,clip]{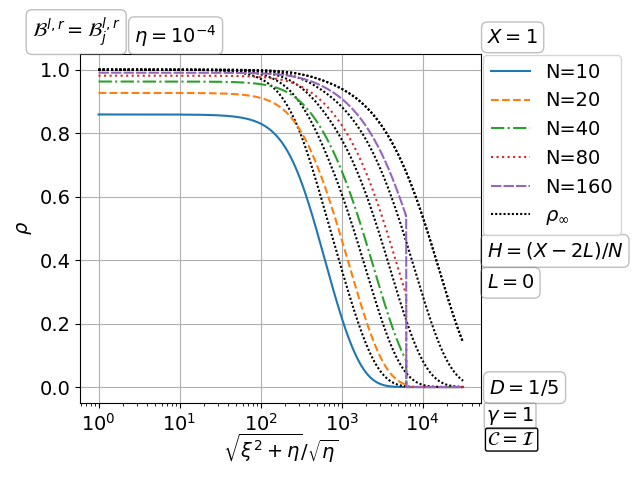}%
    \includegraphics[width=.44\textwidth,height=13em,trim=0 10 0 6,clip]{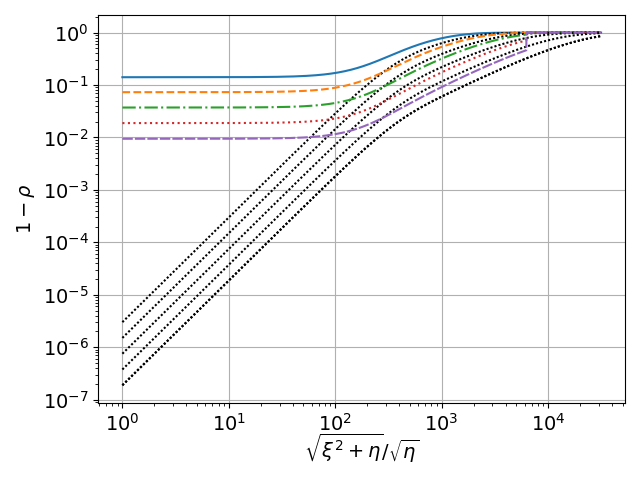}\\
    \includegraphics[width=.5\textwidth,trim=5 6 0 2,clip]{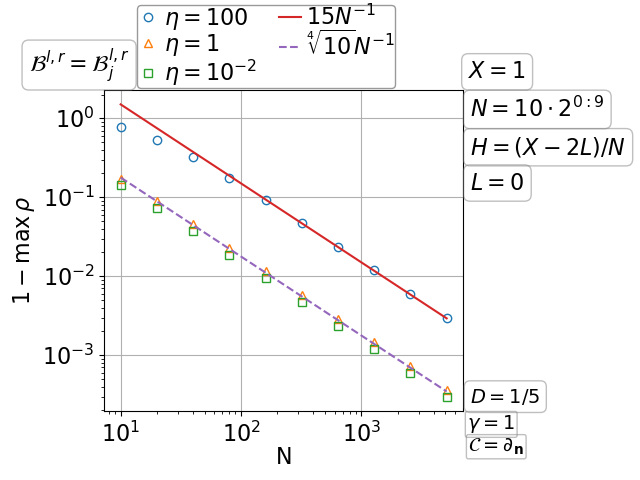}%
    \includegraphics[width=.5\textwidth,trim=5 6 0 2,clip]{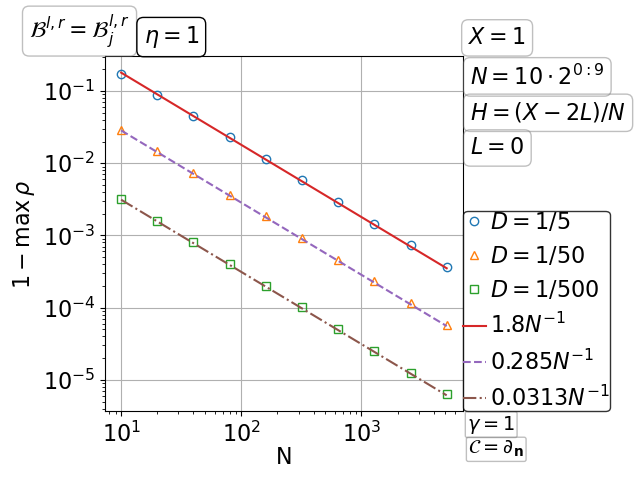}\\
    \includegraphics[width=.5\textwidth,trim=5 6 0 2,clip]{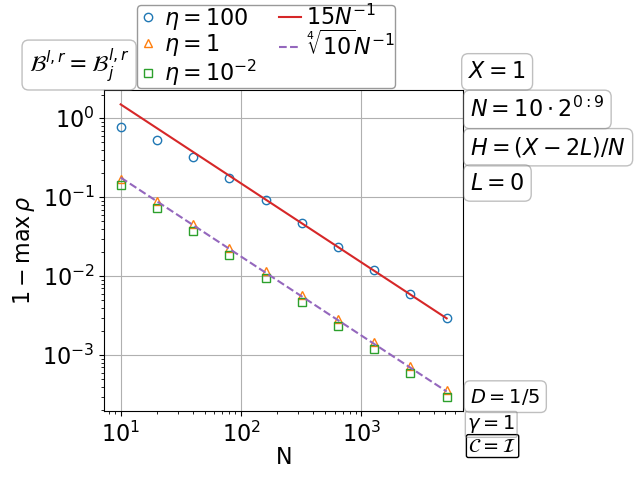}%
    \includegraphics[width=.5\textwidth,trim=5 6 0 2,clip]{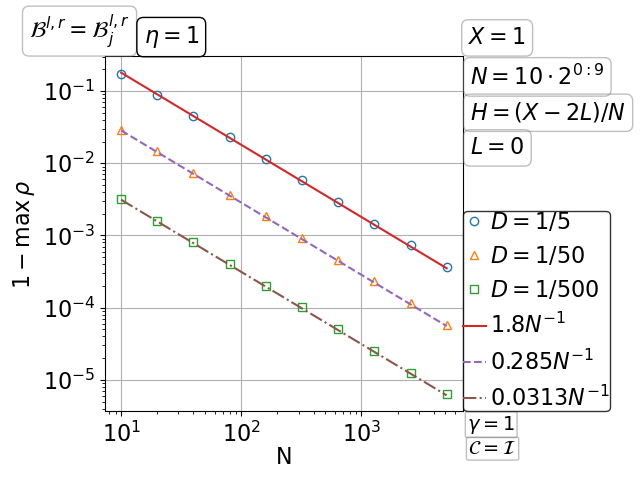}
    \caption{Convergence of the parallel Schwarz method with PML transmission for the infinite
      pipe diffusion on a fixed domain with increasing number of subdomains.}
    \label{figpp1p}
  \end{figure}

\end{paragraph}

% \subsubsection{Red-black Schwarz for diffusion problem}

% half parallelism, half convergence factor?

%%%%%%%%%%%%%%%%%%%%%%%%%%%%%%%%%%%%%%%%%%%%%%%%%%%%%%%%%%%

\subsection{Double sweep Schwarz methods for the diffusion problem}

%%%%%%%%%%%%%%%%%%%%%%%%%%%%%%%%%%%%%%%%%%%%%%%%%%%%%%%%%%%

Double sweep Schwarz methods differ from parallel Schwarz methods in the order of updating subdomain
solutions, which is analogous to the difference between the symmetric Gauss-Seidel and Jacobi
iterations. Moreover, the block symmetric Gauss-Seidel method and the block Jacobi method are
special double sweep and parallel Schwarz methods with minimal overlap and Dirichlet
transmission condition. On {the one hand}, it is typical that the Gauss-Seidel iteration (sweep in only one
order of the unknowns) converges twice as fast as the Jacobi iteration, and the symmetric
Gauss-Seidel method is no more than twice as fast as the Gauss-Seidel method; see
\eg~\cn{hackbusch1994iterative}.  On {the other hand}, the optimal double sweep Schwarz method
converges in one iteration, much faster than the optimal parallel Schwarz method that converges in
$N$ iterations.  Of course, it comes at a price: the double sweep is inherently sequential between
subdomains, while the parallel Schwarz method allows all the subdomain problems to be solved
simultaneously in one iteration. {Our goal in this subsection is to investigate the convergence speed in the general setting
between these two limits.}

%%%%%%%%%%%%%%%%%%%%%%%%%%%%%%%%%%%%

\subsubsection{Double sweep Schwarz method with Dirichlet transmission for the diffusion problem}

The method proposed by \cn{Schwarz:1870:UGA} uses Dirichlet transmission conditions and
overlap. It solves the subdomain problems in alternating order, now also known as double sweep
\cite{NN97}. At the matrix level, the classical Schwarz method can be viewed as an improvement of
the block symmetric Gauss-Seidel method by adding overlaps between blocks. So, yet another name for
the double sweep Schwarz method is the symmetric multiplicative Schwarz method \cite[Section
1.6]{TWbook}. It can be seen from \cn[Theorem 4.1]{bramble1991convergence} that the convergence {factor}
of the method is bounded from above by $1-O(H^2)$ for an overlap $L=O(H)$. In the classical
context with Dirichlet transmission, the symmetric multiplicative Schwarz method is rarely used
because there would be no benefit in convergence \cite{holst1997schwarz} compared to the (single
sweep) multiplicative Schwarz method.

\begin{paragraph}{Convergence with increasing number of fixed size subdomains}

  In the top half of Figure~\ref{figddn}, the convergence factor $\rho(\xi)$ of the double sweep
  Schwarz method is shown for growing $N$ (number of subdomains).  There seems to be a limiting
  curve $\rho_{\infty}:=\lim_{N\to\infty}\rho$ independent of the original boundary condition on
  $\{0, X\}\times [0, Y]$.  From the bottom half of Figure~\ref{figddn}, the scaling of the double
  sweep iteration turns out to be $1-\max_{\xi}\rho=O(1)\in(0,1)$ with a linear dependence on the
  overlap width $L$. The dependence on the coefficient $\eta>0$ is linear for both Neumann and
  Dirichlet problems.  Comparing with the parallel iteration studied in Figure~\ref{figpdn}, the
  double sweep iteration shown in Figure~\ref{figddn} does converge faster. More precisely,
  comparing the botttom left quandrants of the two figures, we find
  % that the original boundary condition on $\{0, X\}\times [0, Y]$ makes a big difference. For the
  % Neumann problem,
  the double sweep iteration is about twice as fast as the parallel iteration. 

  % for the Dirichlet problem, the double sweep iteration can be $100\sim 1000$ times as fast as the
  % parallel iteration.
  
  \begin{figure}
    \centering \includegraphics[width=.56\textwidth,trim=10 10 0
    6,clip]{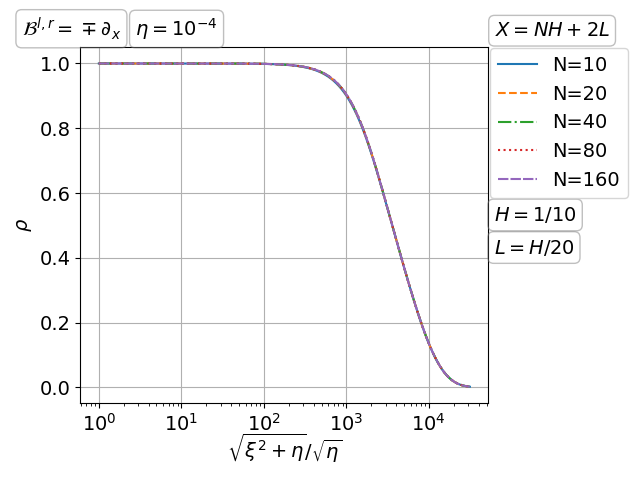}%
    \includegraphics[width=.44\textwidth,height=13em,trim=0 10 0 6,clip]{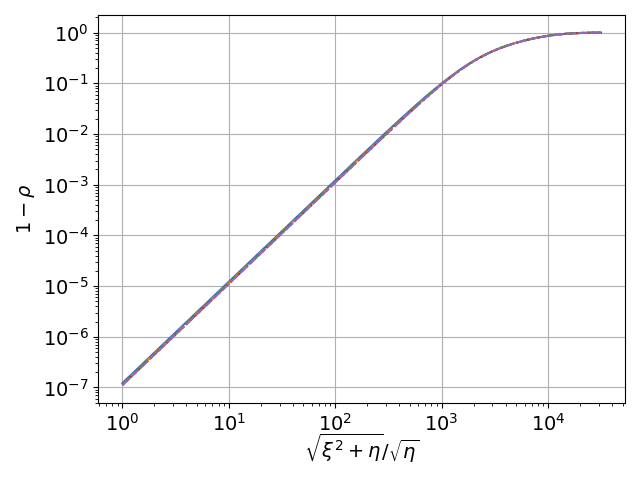}\\
    \includegraphics[width=.56\textwidth,trim=10 10 0
    6,clip]{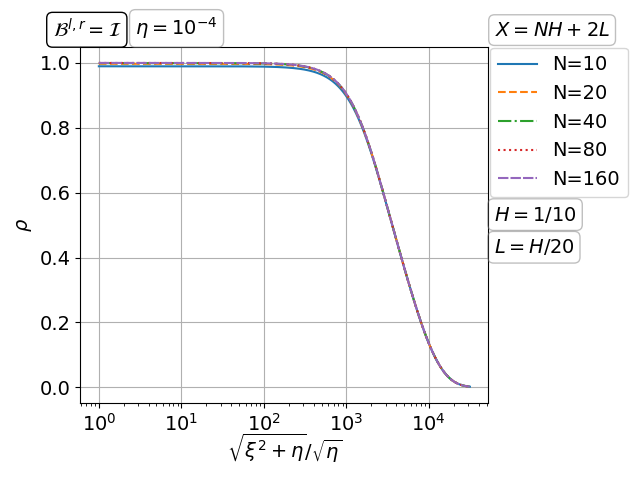}%
    \includegraphics[width=.44\textwidth,height=13em,trim=0 10 0 6,clip]{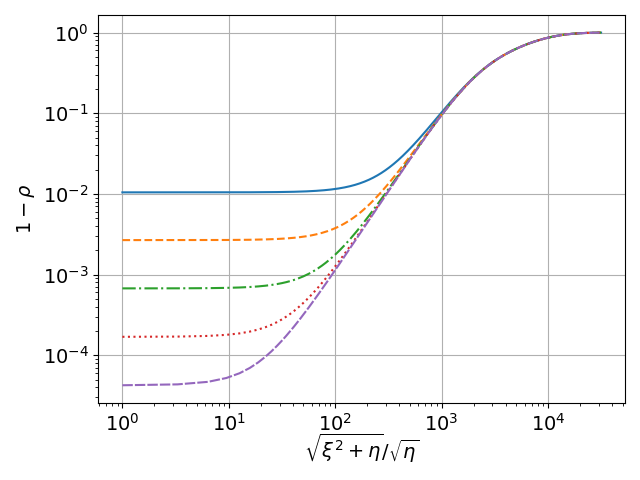}\\
    \includegraphics[width=.5\textwidth,trim=5 6 0
    2,clip]{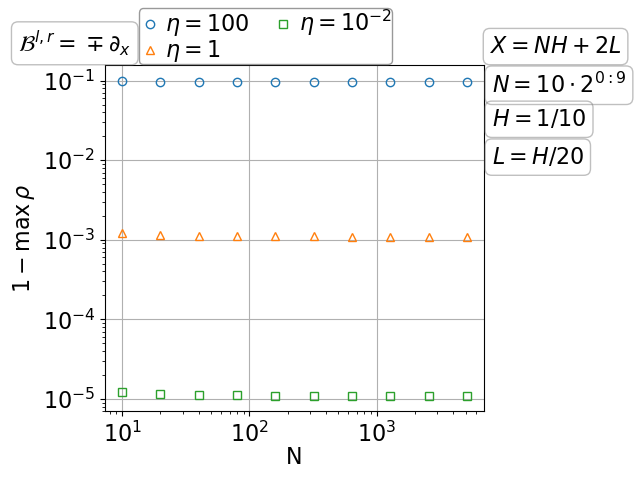}%
    \includegraphics[width=.5\textwidth,trim=5 6 0
    2,clip]{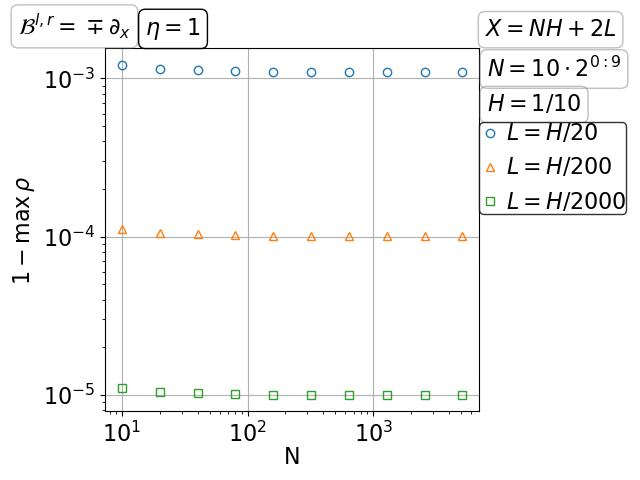}
    \includegraphics[width=.5\textwidth,trim=5 6 0
    2,clip]{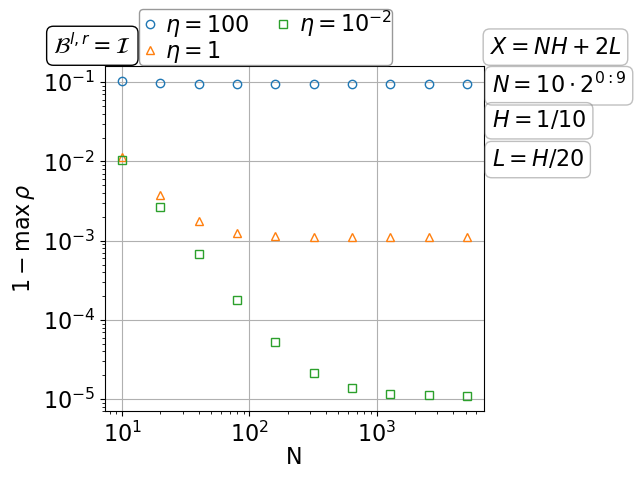}%
    \includegraphics[width=.5\textwidth,trim=5 6 0
    2,clip]{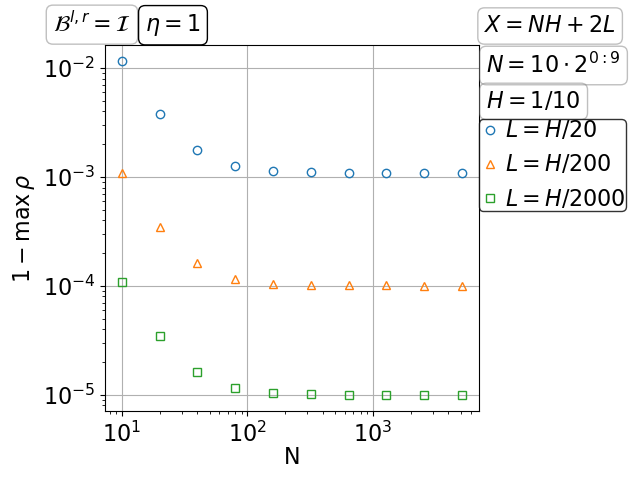}
    \caption{Convergence of the double sweep Schwarz method with Dirichlet transmission for
      diffusion with increasing number of fixed size subdomains.}
    \label{figddn}
  \end{figure}

\end{paragraph}

\begin{paragraph}{Convergence on a fixed domain with increasing number of subdomains}

  We find from the top half of Figure~\ref{figdd1} that the graph of the convergence factor
  $\rho(\xi)$ for the double sweep Schwarz method tends to the constant {$1$} when we refine the
  decomposition on a fixed domain. For small $\eta>0$, the convergence is much faster for the
  Dirichlet problem than for the Neumann problem, a phenomenon observed also with the parallel
  Schwarz method in Figure~\ref{figpd1}. The scaling of $\max_{\xi}\rho(\xi)=\rho(0)$ is illustrated
  in the bottom half of Figure~\ref{figdd1}. Note that a large overlap $L=O(H)$ is
  considered. We find $\max_{\xi}\rho=1-O(N^{-2})$, the same as of the parallel Schwarz
  method shown in Figure~\ref{figpd1}. The dependence on the coefficient $\eta>0$ and the overlap
  width $L$ is also the same as before. A more {careful comparison of} the two figures shows that the double
  sweep iteration is about twice as fast as the parallel iteration.
  
  \begin{figure}
    \centering \includegraphics[width=.56\textwidth,trim=10 10 0
    6,clip]{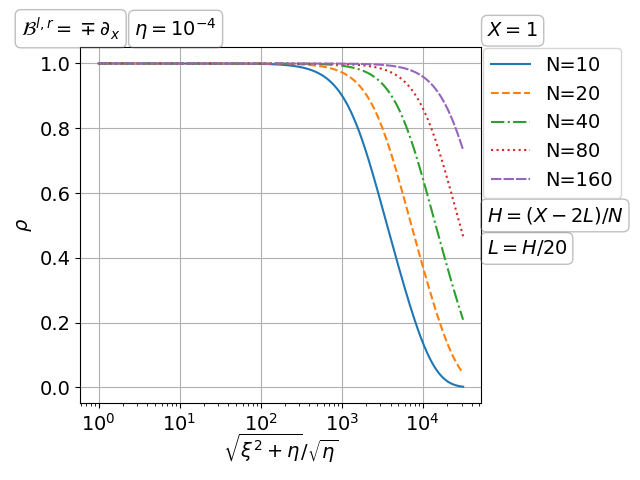}%
    \includegraphics[width=.44\textwidth,height=13em,trim=0 10 0 6,clip]{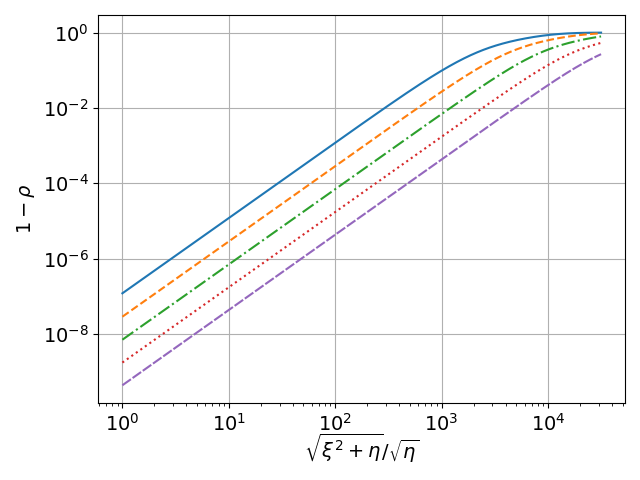}\\
    \includegraphics[width=.56\textwidth,trim=10 10 0
    6,clip]{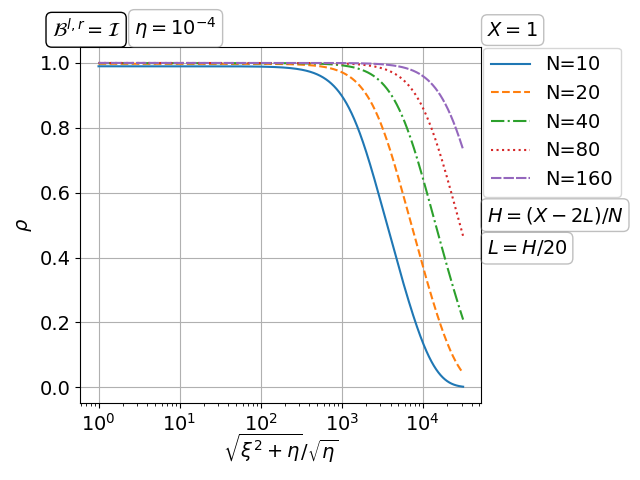}%
    \includegraphics[width=.44\textwidth,height=13em,trim=0 10 0 6,clip]{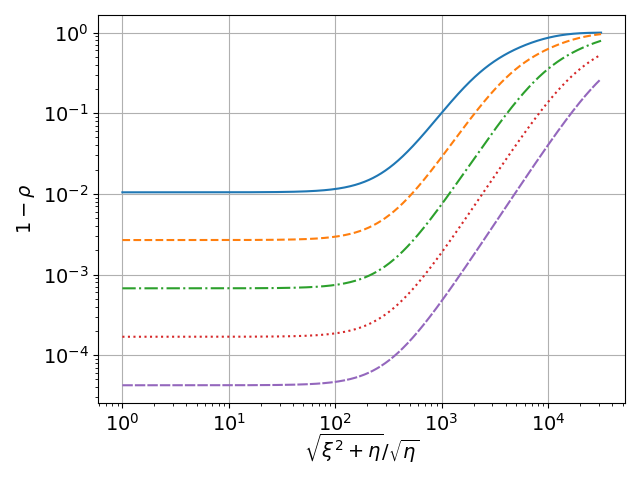}\\
    \includegraphics[width=.5\textwidth,trim=5 6 0
    2,clip]{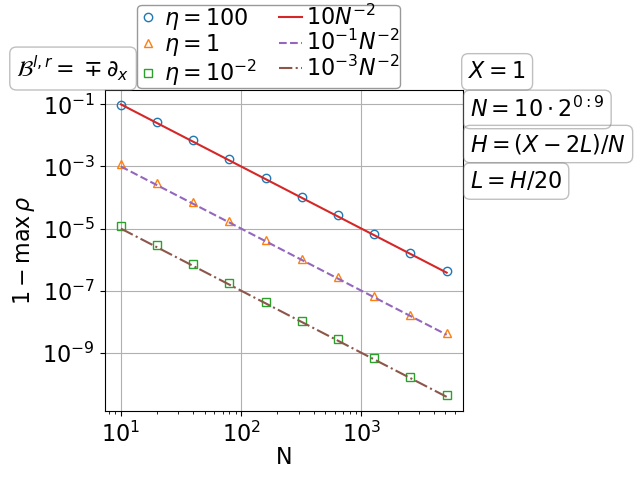}%
    \includegraphics[width=.5\textwidth,trim=5 6 0
    2,clip]{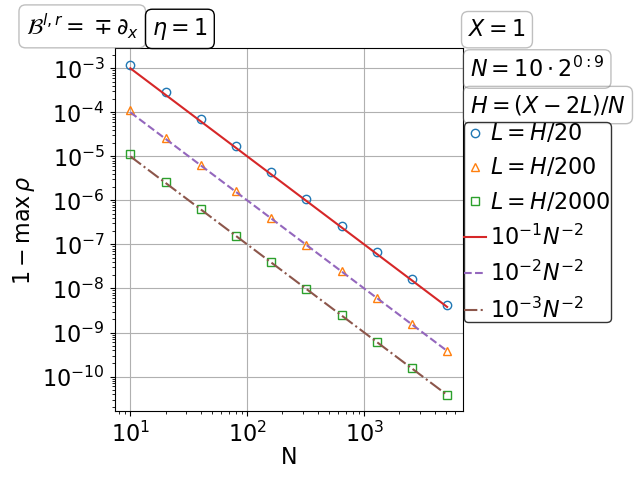}\\
    \includegraphics[width=.5\textwidth,trim=5 6 0
    2,clip]{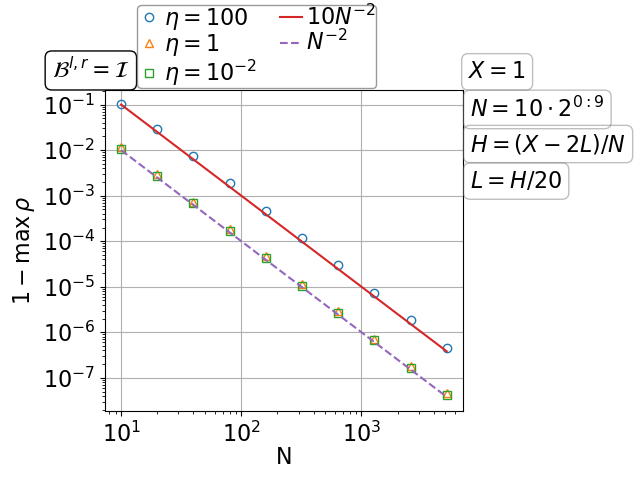}%
    \includegraphics[width=.5\textwidth,trim=5 6 0
    2,clip]{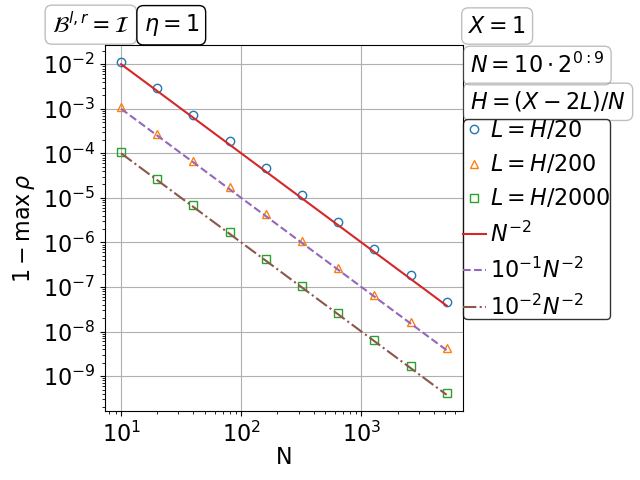}
    \caption{Convergence of the double sweep Schwarz method with Dirichlet transmission for
      diffusion on a fixed domain with increasing number of subdomains.}
    \label{figdd1}
  \end{figure}
  
\end{paragraph}

%%%%%%%%%%%%%%%%%%%%%%%%%%%%%%%%%%%%

\subsubsection{Double sweep Schwarz method with Taylor of order zero transmission for the
  diffusion problem}

The Taylor of order zero condition is designed for domain truncation
of unbounded problems. It is interesting to see how such transmission
conditions work for other problems. As we did for the parallel
Schwarz method, three types of boundary value problems -- the
Dirichlet, Neumann and infinite pipe problems will all be
investigated.  Another important goal is to find out how much is
gained from the parallel iteration to the double sweep iteration.

\begin{paragraph}{Convergence with increasing number of fixed size subdomains}

  For the double sweep Schwarz method, there also seems to be a
  limiting curve $\rho_{\infty}=\lim_{N\to\infty}\rho$ independent of
  the boundary condition defined by $\mathcal{B}^{l,r}$; see
  Figure~\ref{figdt0n}. The maximum of $\rho$ is attained at an
  asymptotic point away from $\xi=0$ as $N\to\infty$. Since the
  scaling is found independent of $\mathcal{B}^{l,r}$, only one
  boundary condition is illustrated; see the bottom row of
  Figure~\ref{figdt0n}. The observation is $\rho = O(1)<1$ with a
  square-root dependence on the overlap width $L$. Compared to the
  subplot in the third row and first column of Figure~\ref{figpt0nd},
  the double sweep iteration is three to four times as fast as the
  parallel iteration.
  
  \begin{figure}
    \centering \includegraphics[width=.56\textwidth,trim=10 10 0
    6,clip]{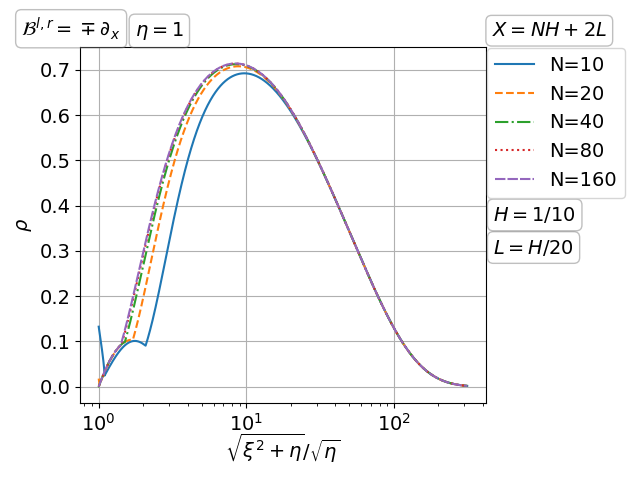}%
    \includegraphics[width=.44\textwidth,height=13em,trim=0 10 0 6,clip]{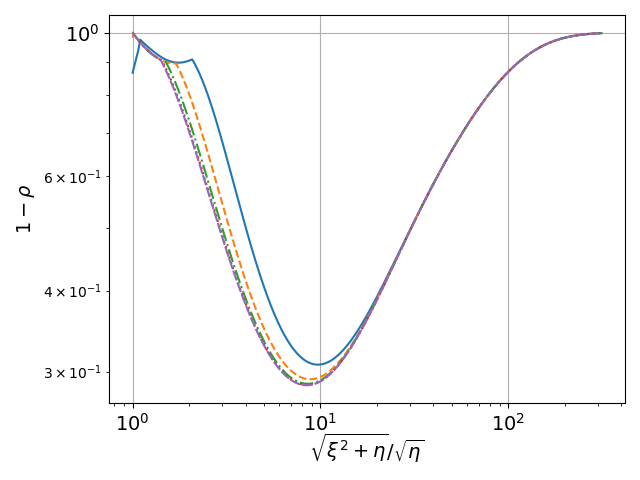}\\
    \includegraphics[width=.56\textwidth,trim=10 10 0
    6,clip]{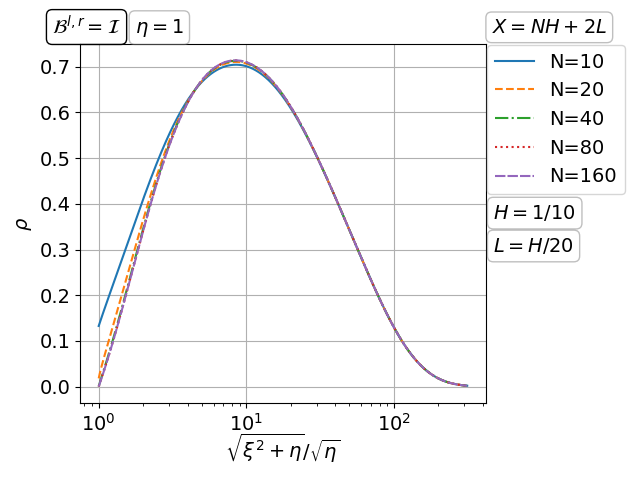}%
    \includegraphics[width=.44\textwidth,height=13em,trim=0 10 0 6,clip]{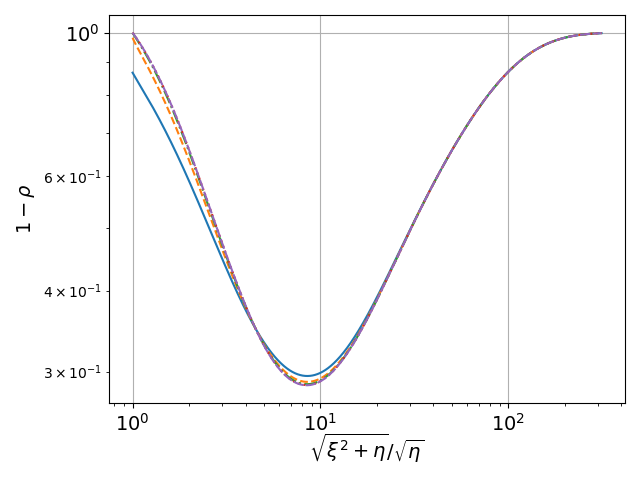}\\
    \includegraphics[width=.56\textwidth,trim=10 10 0
    6,clip]{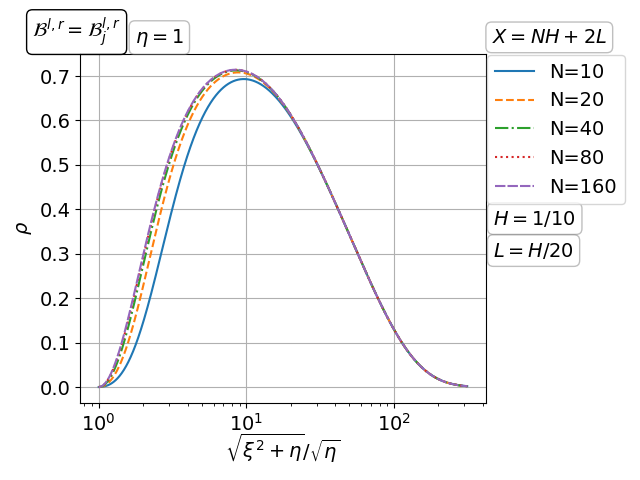}%
    \includegraphics[width=.44\textwidth,height=13em,trim=0 10 0
    6,clip]{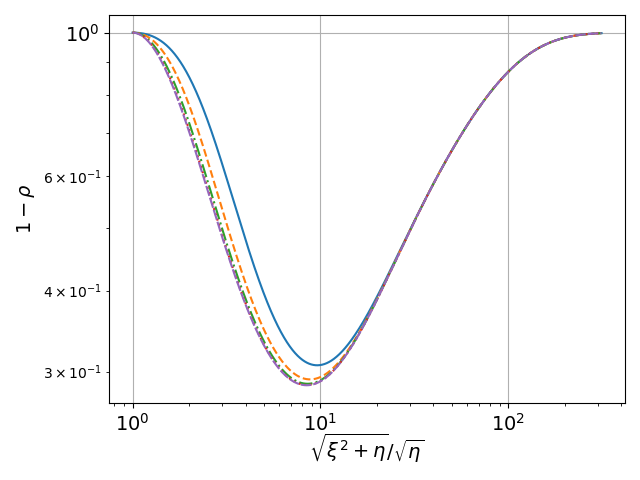}\\
    \includegraphics[width=.5\textwidth,trim=5 6 0
    2,clip]{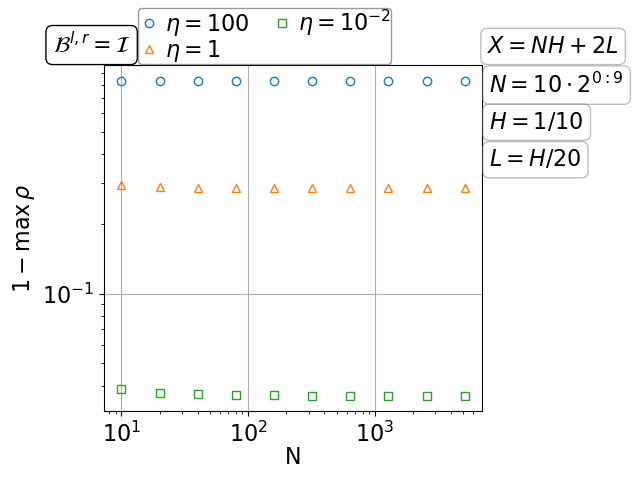}%
    \includegraphics[width=.5\textwidth,trim=5 6 0
    2,clip]{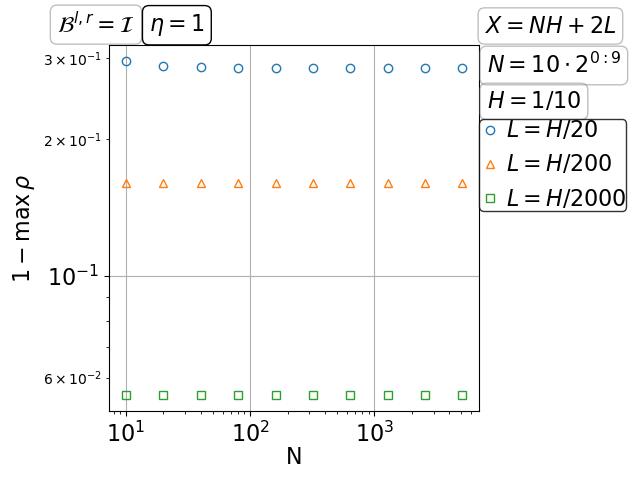}
    \caption{Convergence of the double sweep Schwarz method with Taylor of order zero
      transmission for diffusion with increasing number of fixed size
      subdomains.}
    \label{figdt0n}
  \end{figure}

\end{paragraph}

\begin{paragraph}{Convergence on a fixed domain with increasing number of subdomains}

  In the first three rows of Figure~\ref{figdt01}, we see a deterioration of the convergence as we
  refine the decomposition of a fixed domain: a peak of the convergence factor $\rho$ rises and
  expands to the right to form a plateau. The shape and location of the plateu are almost the same
  for the three types of boundary value problems. Clearly, the double sweep Schwarz method is not a
  smoother. The scaling of $\max_{\xi}\rho$ is illustrated in the bottom row of
  Figure~\ref{figdt01} for the Dirichlet problem. We find
  $\max_{\xi}\rho=1-O(N^{-1})$. The hidden constant factor in $O(N^{-1})$ has a
  square-root dependence on the coefficient $\eta>0$ and a mild dependence on the overlap width
  $L\to 0^+$.  Compared to the parallel iteration explored in Figure~\ref{figpt01d}, the double
  sweep iteration scales the same way but converges $3\sim 4$ times as fast.
  
  \begin{figure}
    \centering \includegraphics[width=.56\textwidth,trim=10 10 0
    6,clip]{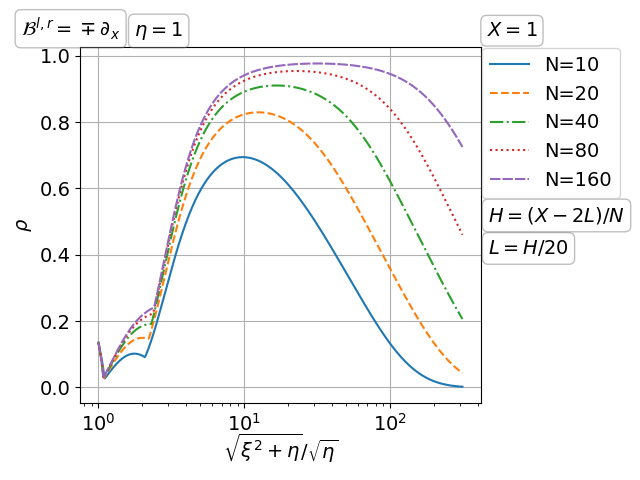}%
    \includegraphics[width=.44\textwidth,height=13em,trim=0 10 0 6,clip]{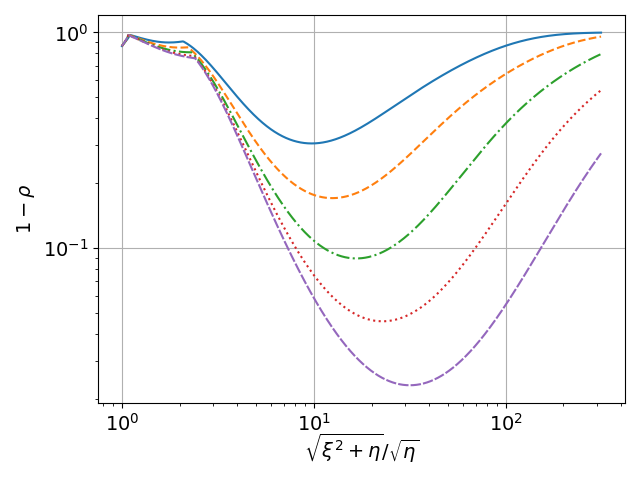}\\
    \includegraphics[width=.56\textwidth,trim=10 10 0
    6,clip]{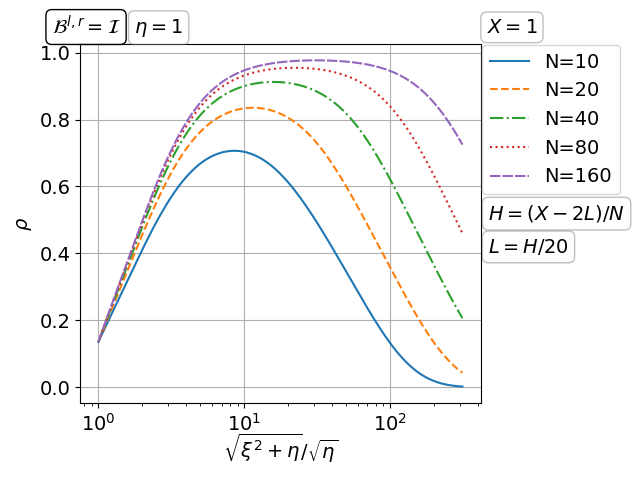}%
    \includegraphics[width=.44\textwidth,height=13em,trim=0 10 0 6,clip]{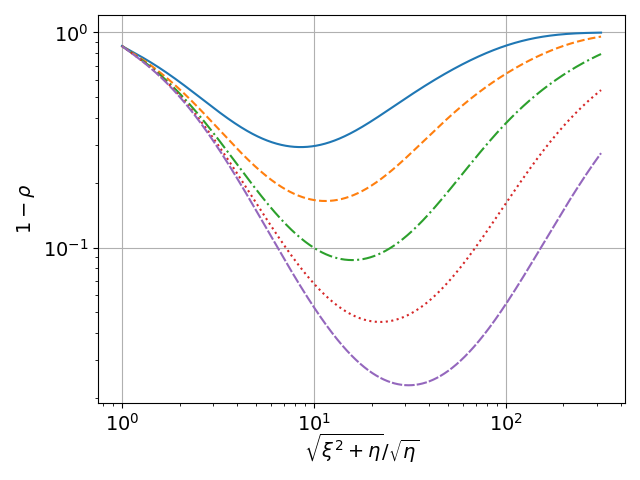}\\
    \includegraphics[width=.56\textwidth,trim=10 10 0
    6,clip]{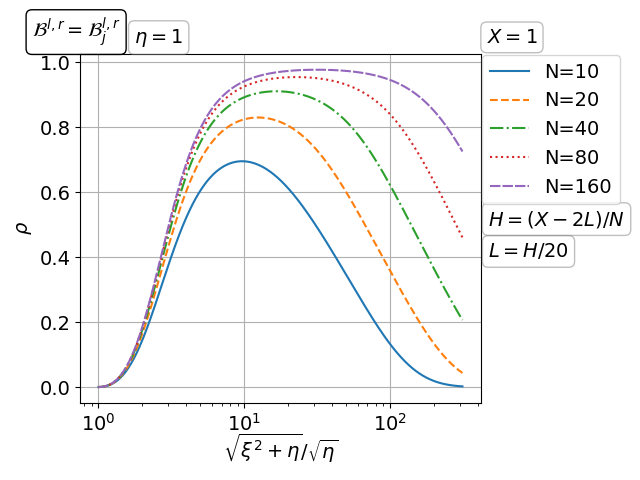}%
    \includegraphics[width=.44\textwidth,height=13em,trim=0 10 0
    6,clip]{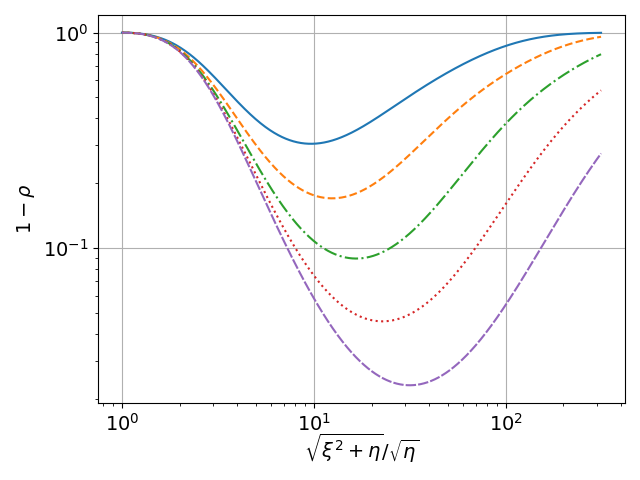}\\
    \includegraphics[width=.5\textwidth,trim=5 6 0
    2,clip]{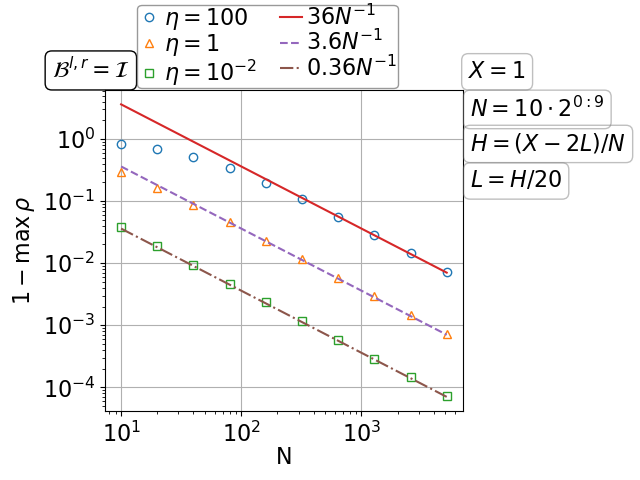}%
    \includegraphics[width=.5\textwidth,trim=5 6 0
    2,clip]{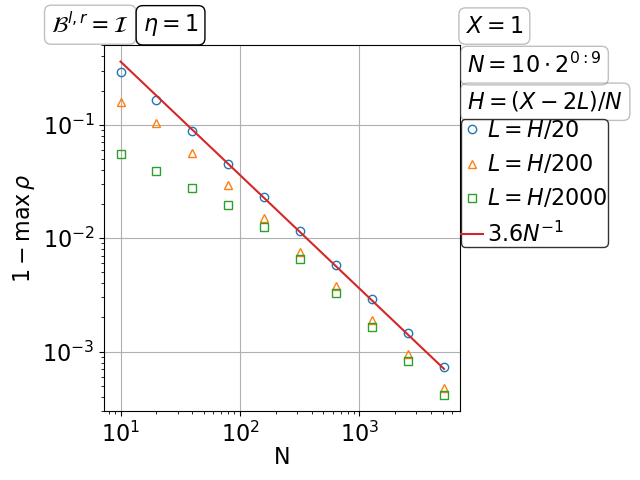}
    \caption{Convergence of the double sweep Schwarz method with Taylor of order zero
      transmission for diffusion on a fixed domain with increasing number of
      subdomains.}
    \label{figdt01}
  \end{figure}
 
\end{paragraph}

%%%%%%%%%%%%%%%%%%%%%%%%%%%%%%%%%%%%%%%

\subsubsection{Double sweep Schwarz method with PML transmission for the diffusion problem}

So far, the transmission conditions that we have explored for the double sweep Schwarz methods are
the Dirichlet and Taylor of order zero conditions, which both rely on the overlap width $L>0$ to
control the high frequency error. Compared to their parallel Schwarz counterparts, these double
sweep Schwarz methods do not improve the scaling orders in $N$ (number of subdomains) but only the
constant factors. Will the comparison carry over to the PML condition? For the parallel Schwarz
methods on a fixed domain, we have seen that a fixed PML condition leads to the convergence rate
$\max_{\xi}\rho=1-O(N^{-1})$, the same as the Taylor of order zero condition does, which
is not strange given that the optimal parallel Schwarz method converges in $N$ iterations. But the
optimal double sweep Schwarz method converges in just one iteration. So, if a double sweep Schwarz
method deteriorates with growing $N$, then we can owe the deterioration to the non-optimal
transmission condition, \eg, the Taylor of order zero condition. Can the PML condition as a more
accurate domain truncation technique push the double sweep Schwarz method toward the optimal one?
How good a PML in terms of the PML width $D$ and strength $\gamma$ do we need to achieve the constant
scaling {independent of} $N$? Let us find out the answers below.

\begin{paragraph}{Convergence with increasing number of fixed size subdomains}

  Three types of boundary value problems -- Neumann, Dirichlet and
  infinite pipe problems are explored in individual figures; see
  Figure~\ref{figdpnn}, Figure~\ref{figdpnd} and
  Figure~\ref{figdpnp}. In each case, we test both the Dirichlet and
  Neumann terminating conditions for the PML. For example, in
  Figure~\ref{figdpnn}, the first row is for the Neumann terminated
  PML and the second row is for the Dirichlet terminated PML. We see
  as $N\to \infty$, the graph of $\rho(\xi)$ tends to some limiting
  profile, independent of the PML terminating condition and the
  original boundary condition. The existence of
  $\lim_{N\to\infty}\rho<1$ implies the scaling
  $\max_{\xi}\rho(\xi)=\rho(0)=O(1)<1$ for all $N$, which is
  confirmed with scaling plots in the bottom halves of the
  figures. For moderate $N$, we find it better to use the same
  condition as for the original problem to terminate the PML. As
  expected, the convergence rate improves for bigger PML width $D$ but
  deteriorates for smaller coefficient $\eta>0$. Compared to the
  parallel Schwarz method shown in Figure~\ref{figppnn},
  Figure~\ref{figppnd} and Figure~\ref{figppnp}, the double sweep
  method is about $10$ times as fast.
  
  \begin{figure}
    \centering \includegraphics[width=.56\textwidth,trim=10 10 0 6,clip]{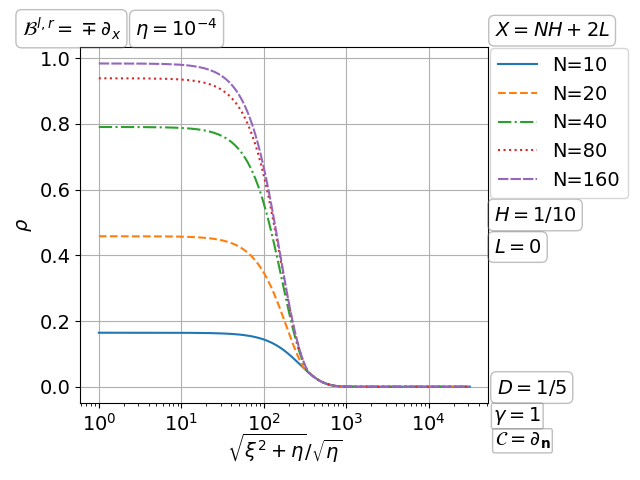}%
    \includegraphics[width=.44\textwidth,height=13em,trim=0 10 0 6,clip]{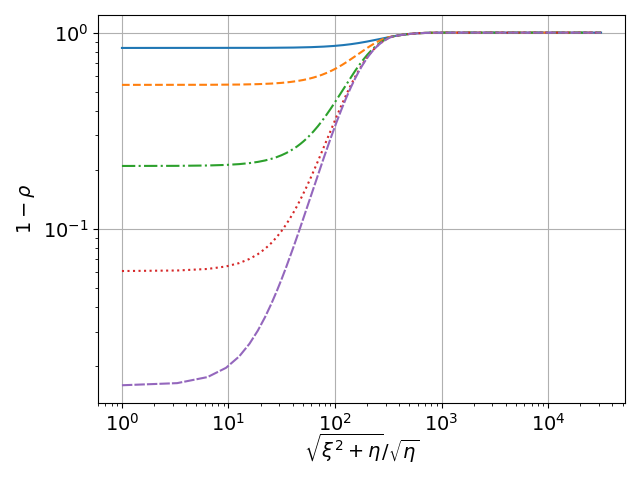}\\
    \includegraphics[width=.56\textwidth,trim=10 10 0 6,clip]{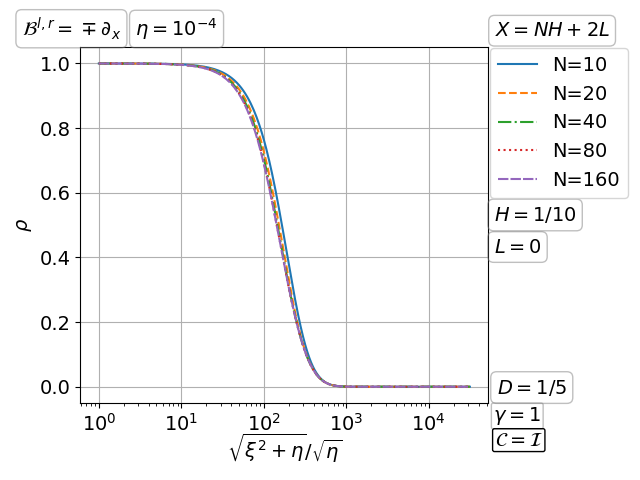}%
    \includegraphics[width=.44\textwidth,height=13em,trim=0 10 0 6,clip]{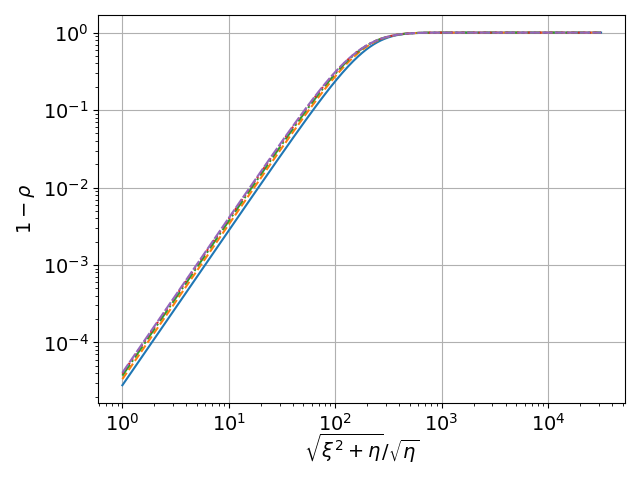}\\
    \includegraphics[width=.5\textwidth,trim=5 6 0 2,clip]{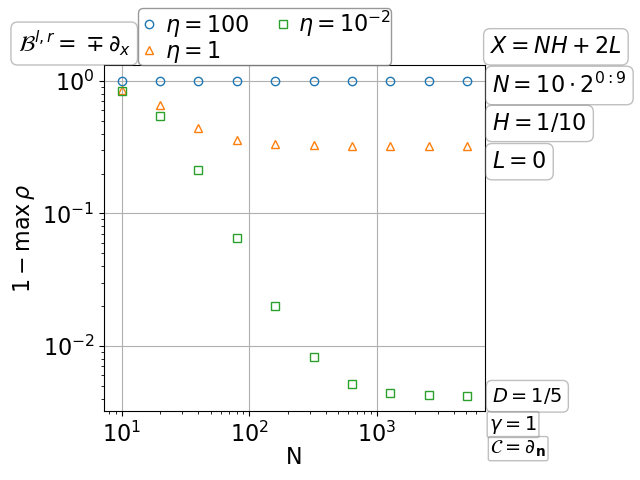}%
    \includegraphics[width=.5\textwidth,trim=5 6 0 2,clip]{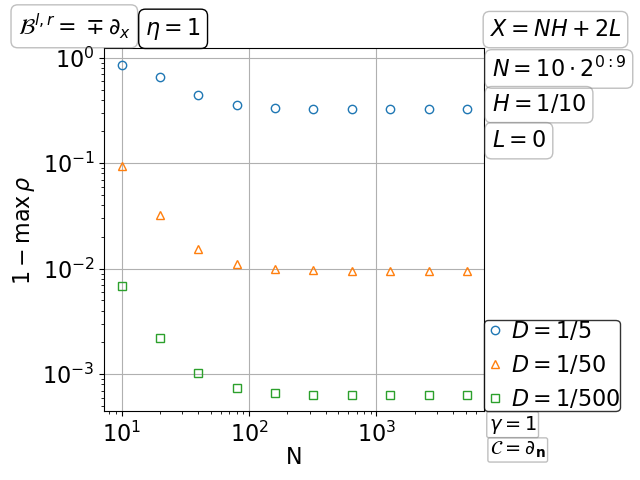}
    \includegraphics[width=.5\textwidth,trim=5 6 0 2,clip]{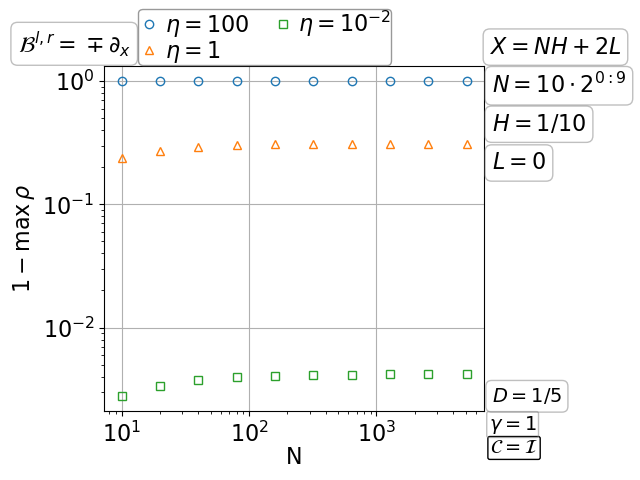}%
    \includegraphics[width=.5\textwidth,trim=5 6 0 2,clip]{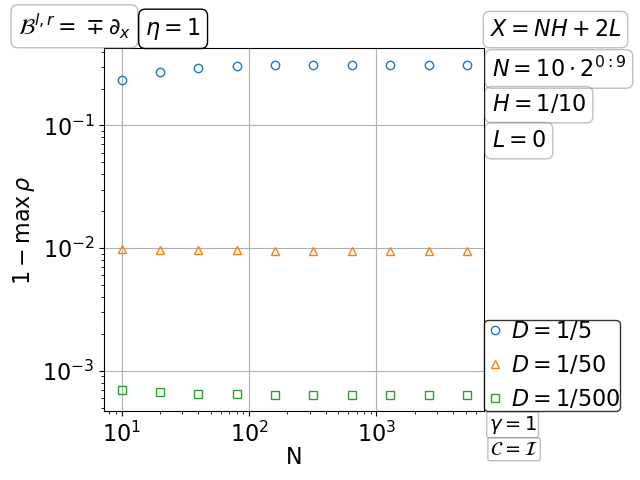}
    \caption{Convergence of the double sweep Schwarz method with PML transmission for the Neumann
      problem of diffusion with increasing number of fixed size subdomains.}
    \label{figdpnn}
  \end{figure}

  \begin{figure}
    \centering%
    \includegraphics[width=.56\textwidth,trim=10 10 0 6,clip]{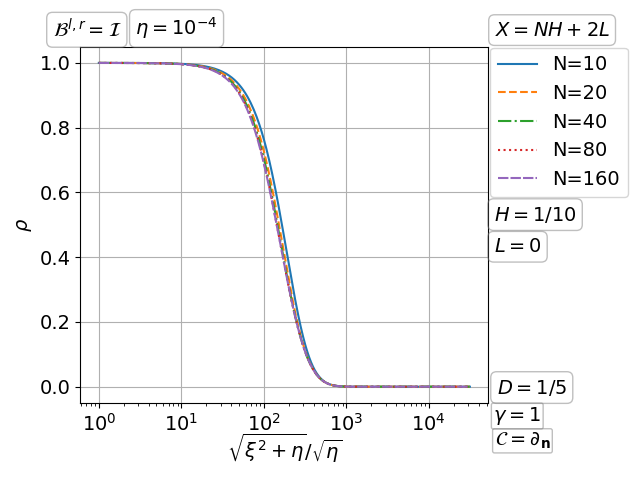}%
    \includegraphics[width=.44\textwidth,height=13em,trim=0 10 0 6,clip]{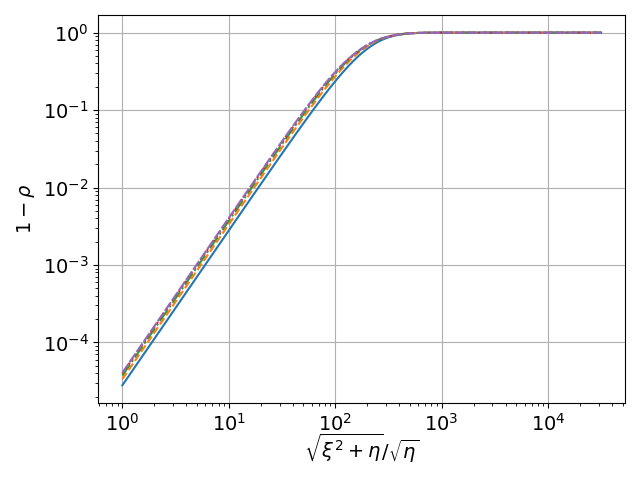}\\
    \includegraphics[width=.56\textwidth,trim=10 10 0 6,clip]{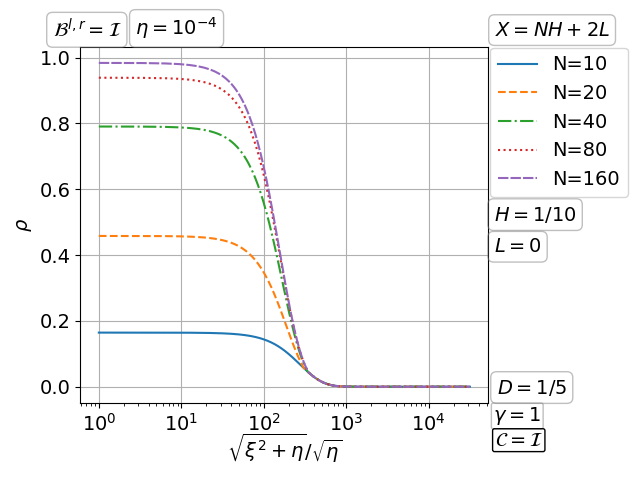}%
    \includegraphics[width=.44\textwidth,height=13em,trim=0 10 0 6,clip]{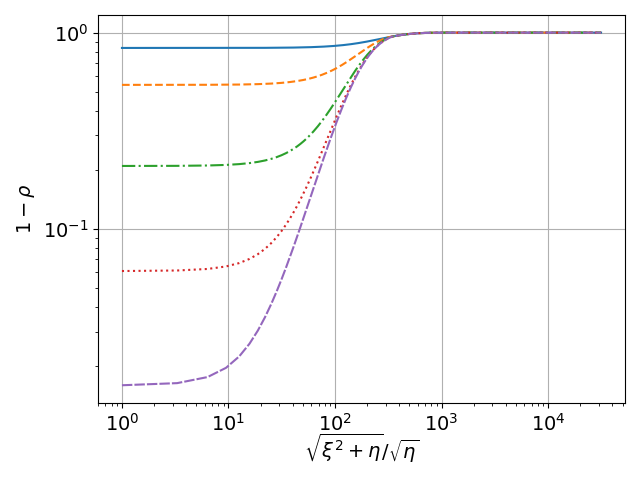}\\
    \includegraphics[width=.5\textwidth,trim=5 6 0 2,clip]{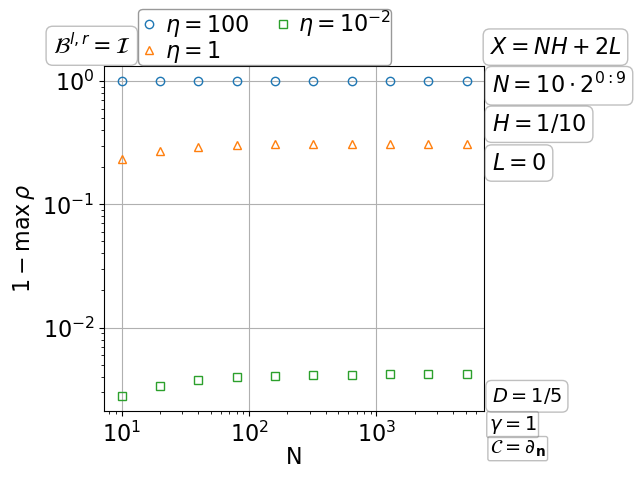}%
    \includegraphics[width=.5\textwidth,trim=5 6 0 2,clip]{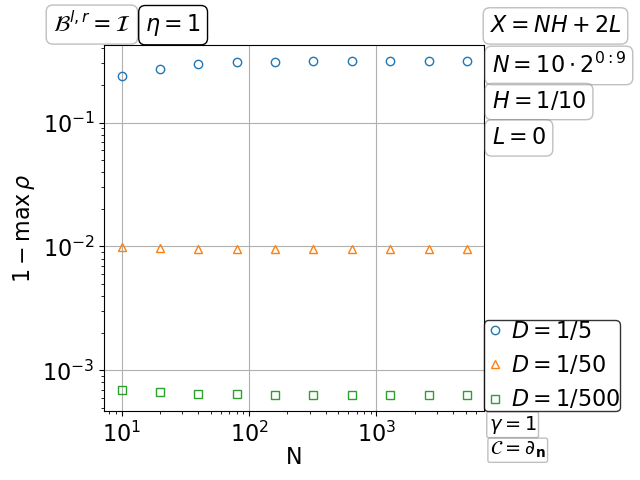}
    \includegraphics[width=.5\textwidth,trim=5 6 0 2,clip]{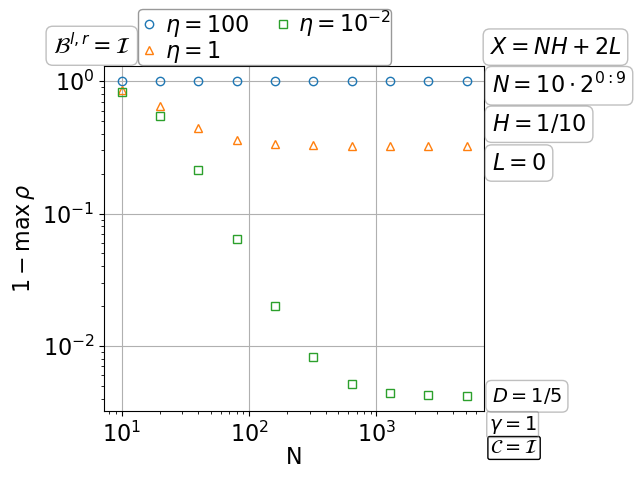}%
    \includegraphics[width=.5\textwidth,trim=5 6 0 2,clip]{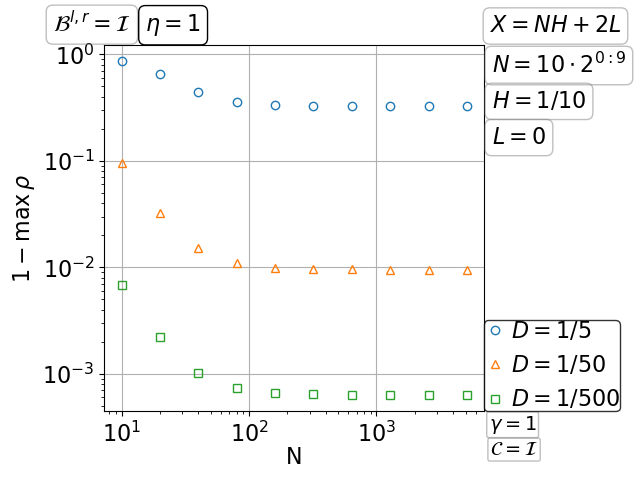}
    \caption{Convergence of the double sweep Schwarz method with PML transmission for the
      Dirichlet problem of diffusion with increasing number of fixed size subdomains.}
    \label{figdpnd}
  \end{figure}

  \begin{figure}
    \centering%
    \includegraphics[width=.56\textwidth,trim=10 10 0 6,clip]{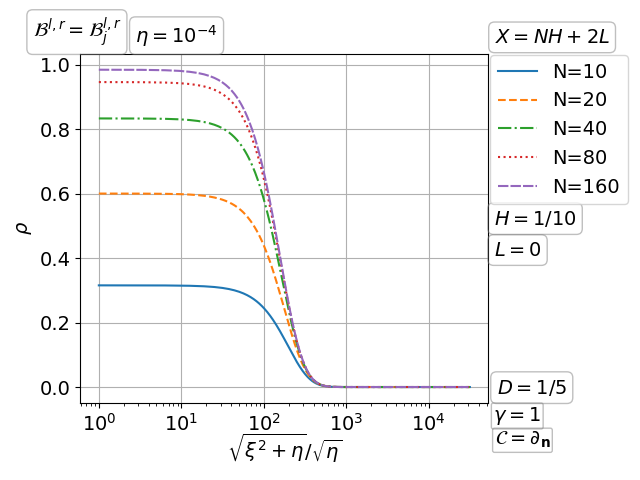}%
    \includegraphics[width=.44\textwidth,height=13em,trim=0 10 0 6,clip]{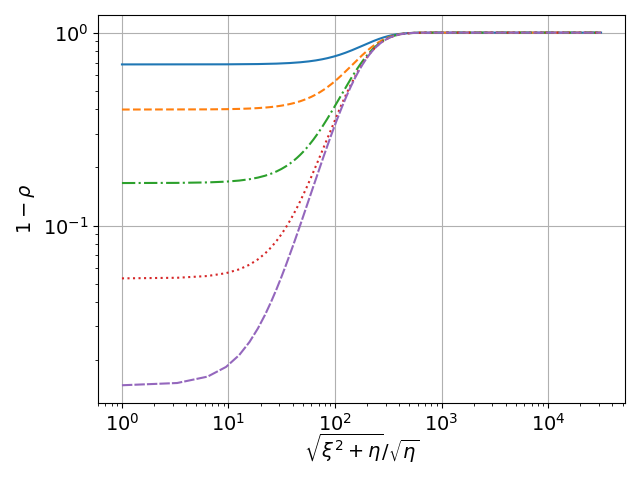}\\
    \includegraphics[width=.56\textwidth,trim=10 10 0 6,clip]{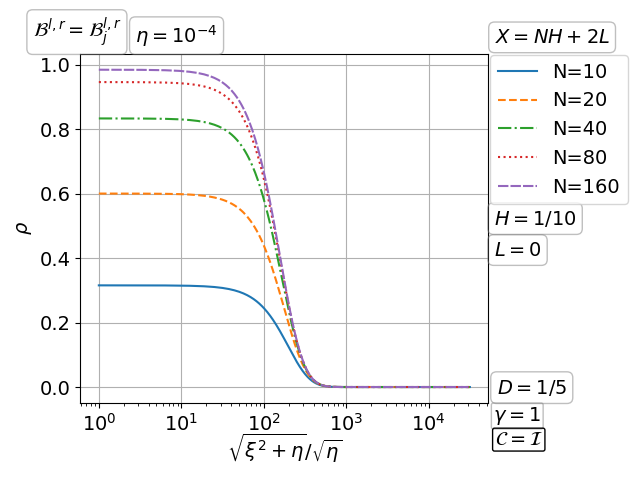}%
    \includegraphics[width=.44\textwidth,height=13em,trim=0 10 0 6,clip]{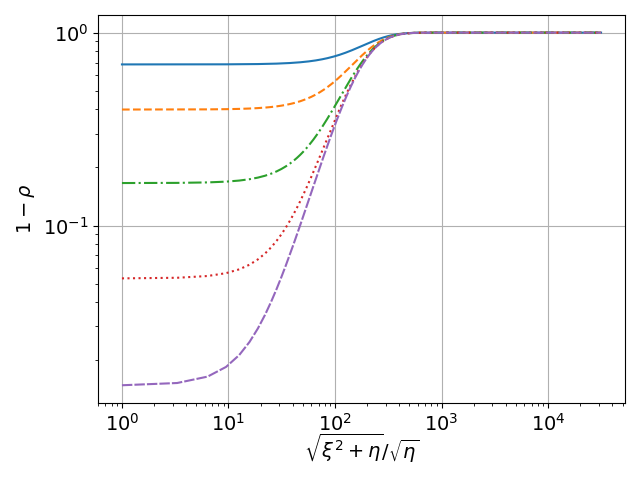}\\
    \includegraphics[width=.5\textwidth,trim=5 6 0 2,clip]{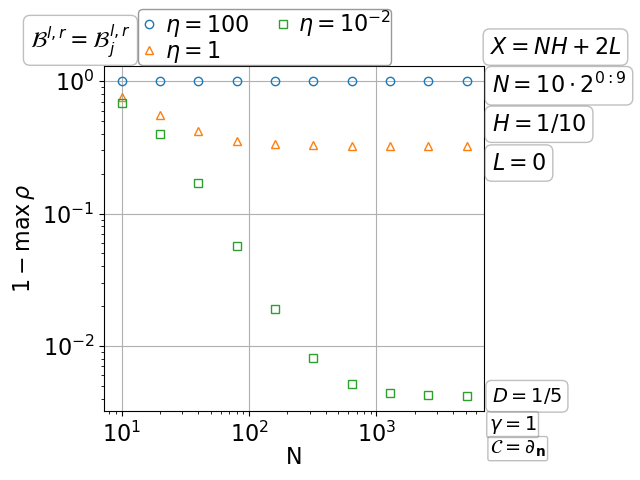}%
    \includegraphics[width=.5\textwidth,trim=5 6 0 2,clip]{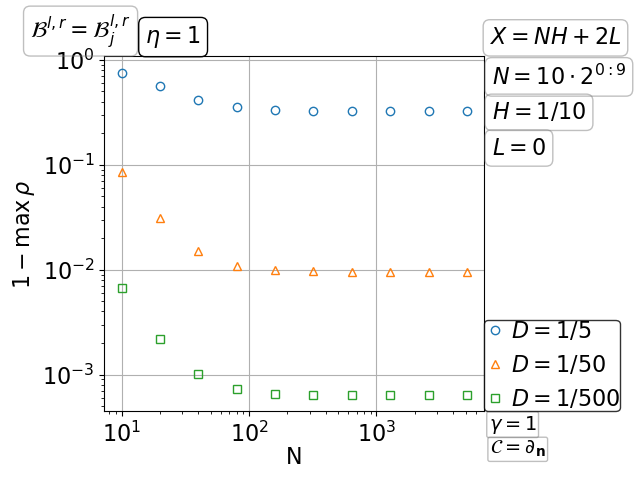}
    \includegraphics[width=.5\textwidth,trim=5 6 0 2,clip]{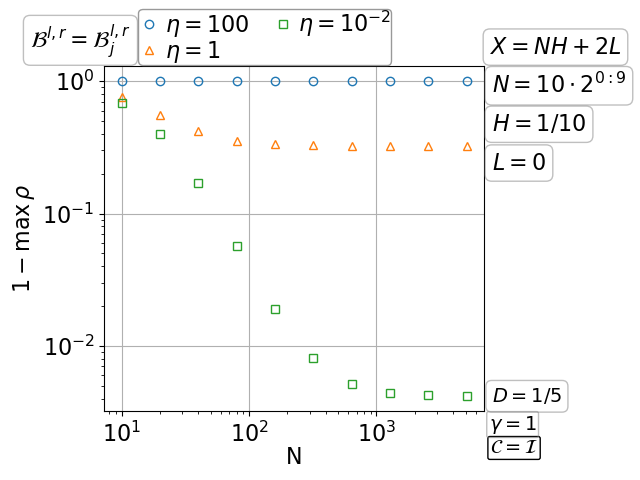}%
    \includegraphics[width=.5\textwidth,trim=5 6 0 2,clip]{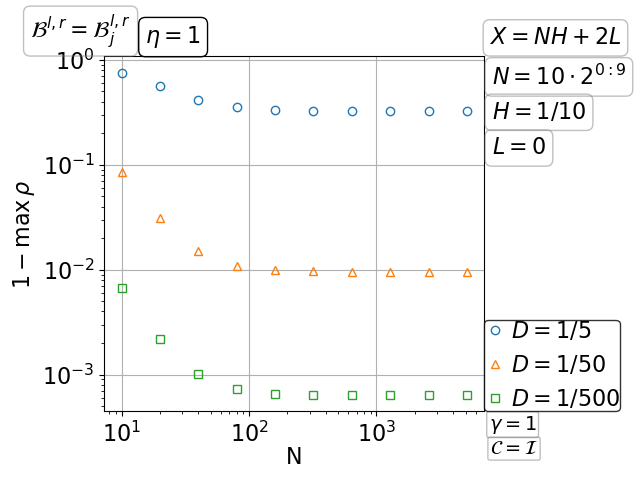}
    \caption{Convergence of the double sweep Schwarz method with PML transmission for the infinite
      pipe diffusion with increasing number of fixed size subdomains.}
    \label{figdpnp}
  \end{figure}
  
\end{paragraph}

\begin{paragraph}{Convergence on a fixed domain with increasing number of subdomains}

  In this scaling, we solve a fixed problem with more and more subdomains. Since we are using the
  PML transmission, the PML width $D$ can be fixed rather than shrinking -- an advantage over the
  overlap. The convergence behavior of the double sweep Schwarz method is illustrated in
  Figure~\ref{figdp1n}, Figure~\ref{figdp1d} and Figure~\ref{figdp1p}. Intriguingly, by fixing $D$,
  the double sweep
  %       % \marginpar{\MG{Is the zero frequency, i.e. Neumann conditions, and $\eta=0$
  %     contained in this?} $1-O(-1/\log\eta)$ as $\eta\to0$}  
  Schwarz method becomes scalable, which does neither happen to the parallel+PML combination nor to
  the double sweep+Taylor combination!  Not only the maximum of $\rho$ but also the whole graph of
  $\rho$ is scalable. That is, a limiting profile of $\rho$ seems to exist as $N\to\infty$.
  Moreover, the terminating condition of PML has a remarkable impact, \eg, when $\eta=10^{-2}$ and
  $D=1/5$, the difference can be by a factor of about 200. The convergence deterioration with small
  coefficient $\eta>0$ is mild when the good PML termination is used.  Also, it seems that
  $\max_{\xi}\rho$ decays exponentially to zero as $D$ increases. Hence, given the PML strength
  $\gamma$, a fixed PML width $D$ is not only sufficient for the scalability but also necessary, for
  which more evidence can be found in \cn{GZDD25}.
  
  \begin{figure}
    \centering%
    \includegraphics[width=.56\textwidth,trim=10 10 0 6,clip]{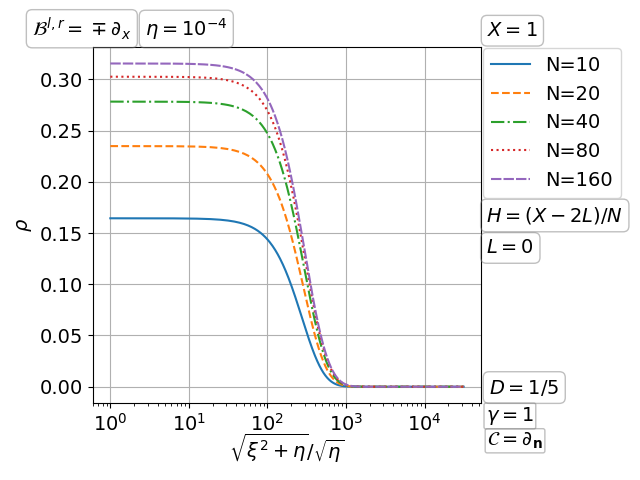}%
    \includegraphics[width=.44\textwidth,height=13em,trim=0 10 0 6,clip]{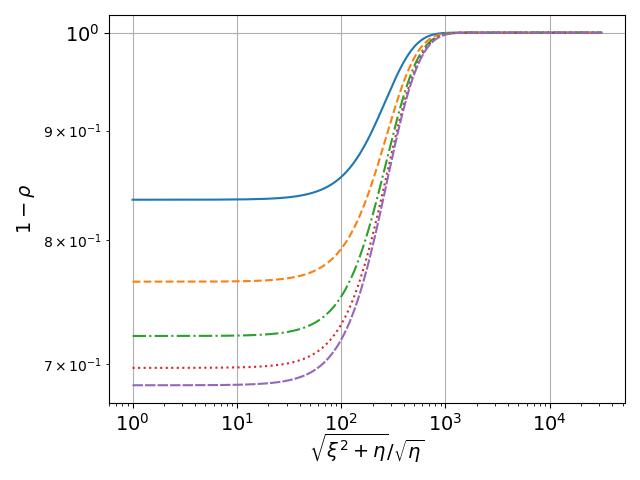}\\
    \includegraphics[width=.56\textwidth,trim=10 10 0 6,clip]{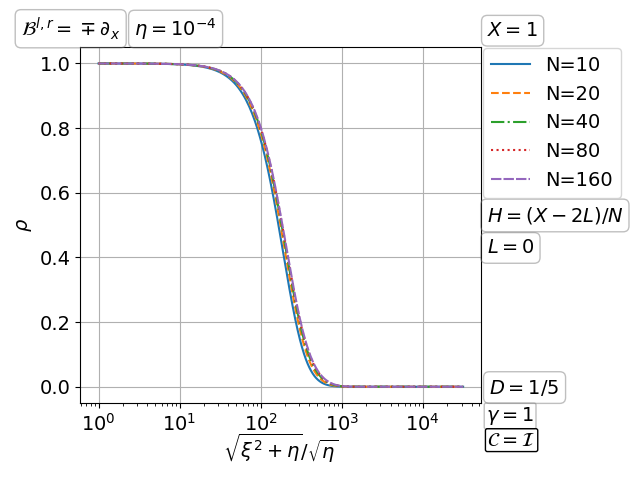}%
    \includegraphics[width=.44\textwidth,height=13em,trim=0 10 0 6,clip]{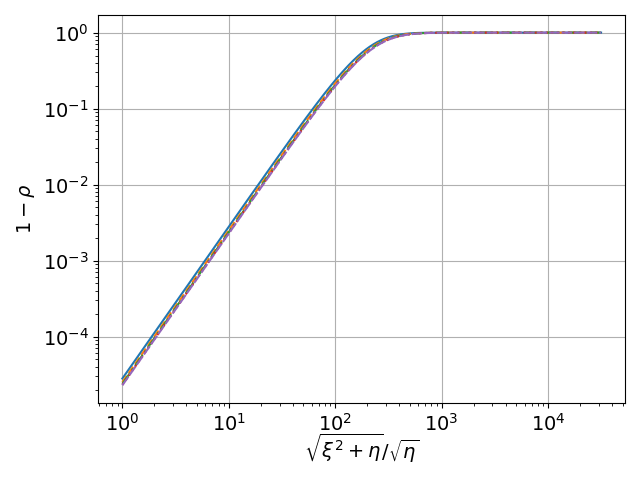}\\
    \includegraphics[width=.5\textwidth,trim=5 6 0 2,clip]{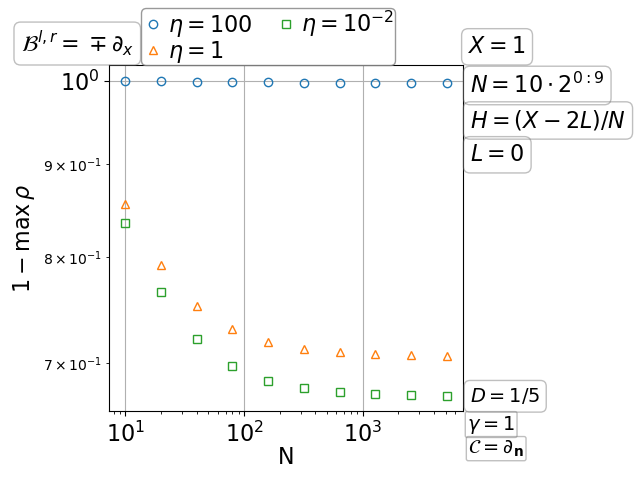}%
    \includegraphics[width=.5\textwidth,trim=5 6 0 2,clip]{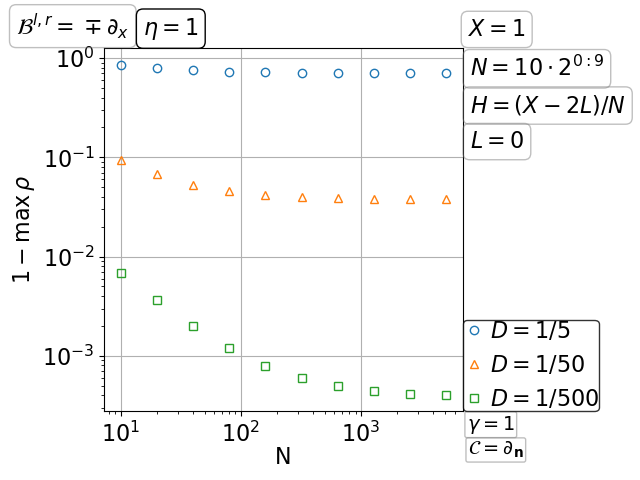}\\
    \includegraphics[width=.5\textwidth,trim=5 6 0 2,clip]{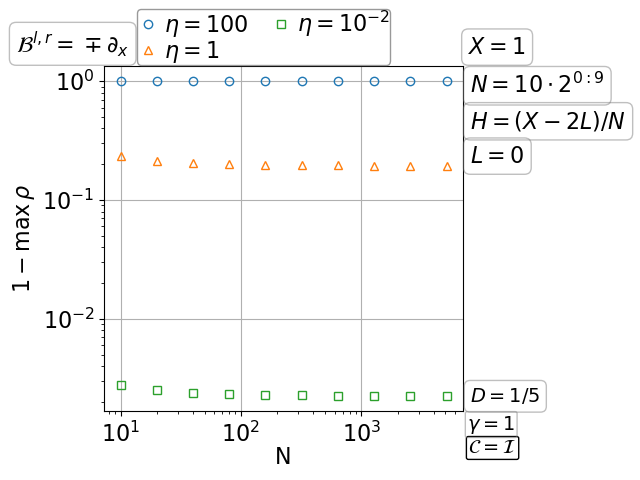}%
    \includegraphics[width=.5\textwidth,trim=5 6 0 2,clip]{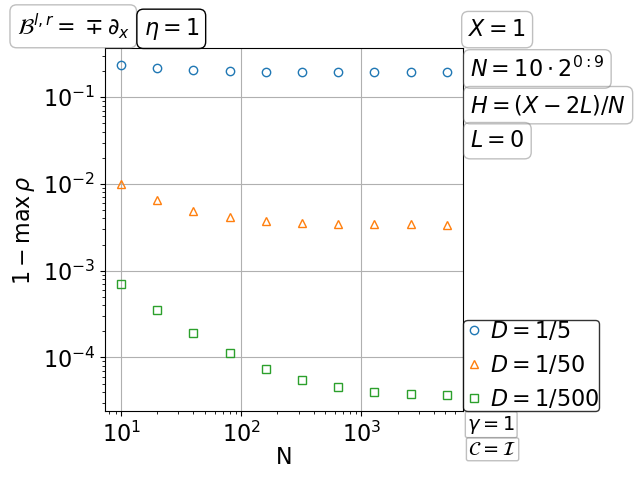}
    \caption{Convergence of the double sweep Schwarz method with PML transmission for the Neumann
      problem of diffusion on a fixed domain with increasing number of subdomains.}
    \label{figdp1n}
  \end{figure}

  \begin{figure}
    \centering%
    \includegraphics[width=.56\textwidth,trim=10 10 0 6,clip]{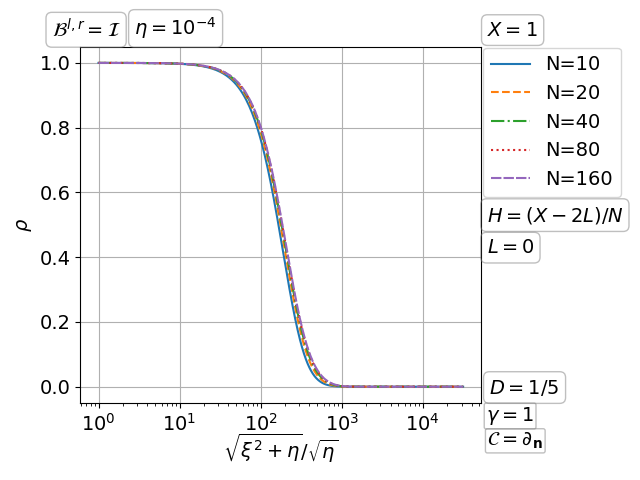}%
    \includegraphics[width=.44\textwidth,height=13em,trim=0 10 0 6,clip]{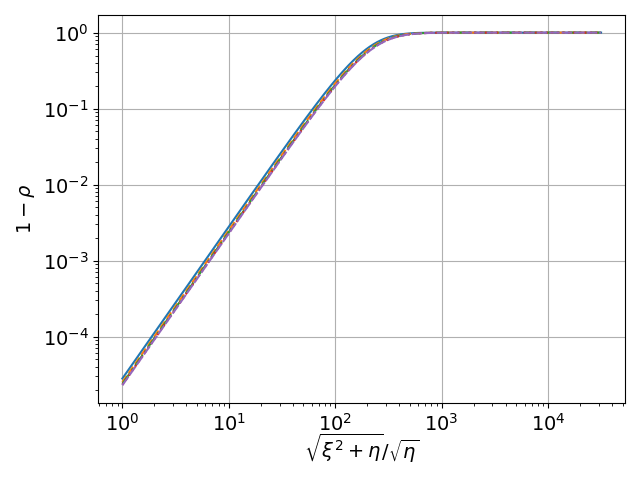}\\
    \includegraphics[width=.56\textwidth,trim=10 10 0 6,clip]{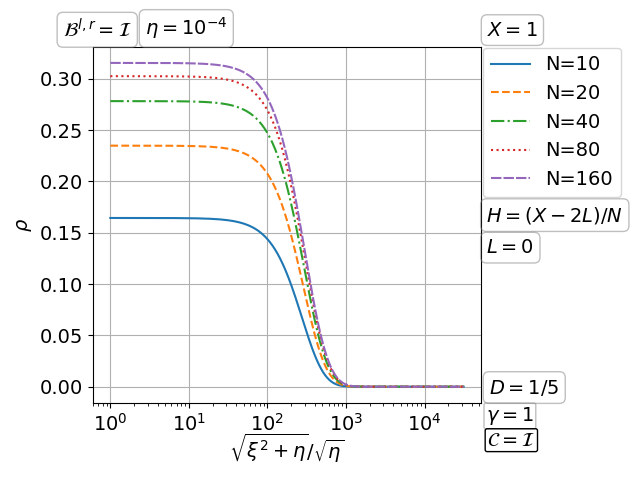}%
    \includegraphics[width=.44\textwidth,height=13em,trim=0 10 0 6,clip]{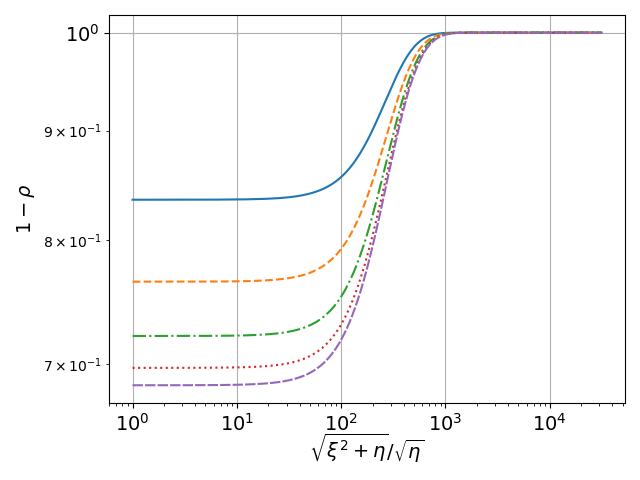}\\
    \includegraphics[width=.5\textwidth,trim=5 6 0 2,clip]{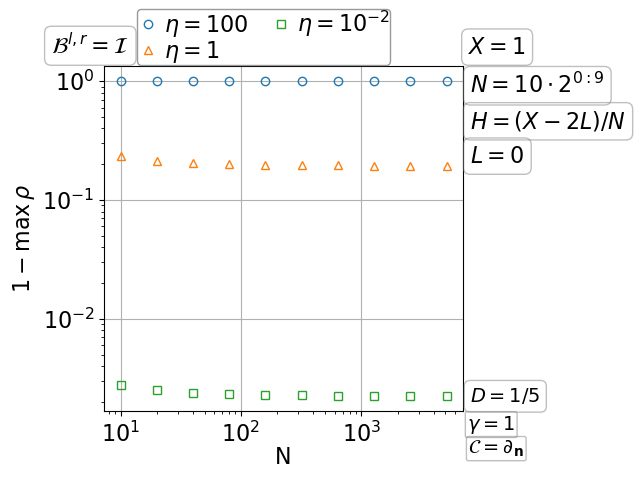}%
    \includegraphics[width=.5\textwidth,trim=5 6 0 2,clip]{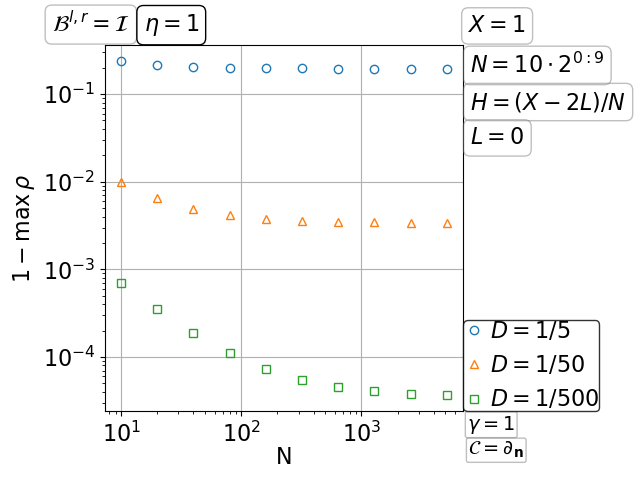}\\
    \includegraphics[width=.5\textwidth,trim=5 6 0 2,clip]{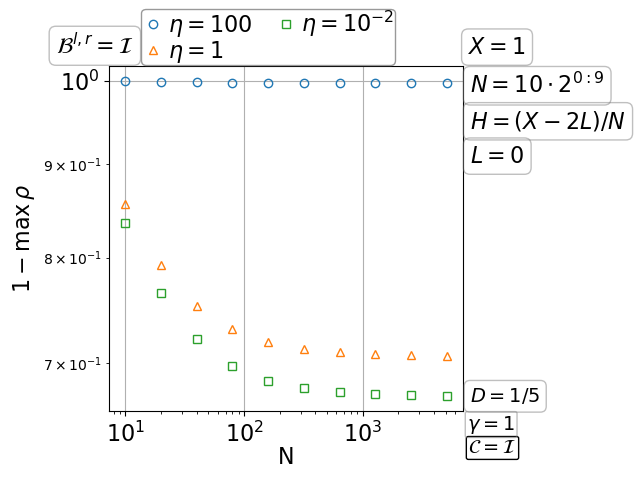}%
    \includegraphics[width=.5\textwidth,trim=5 6 0 2,clip]{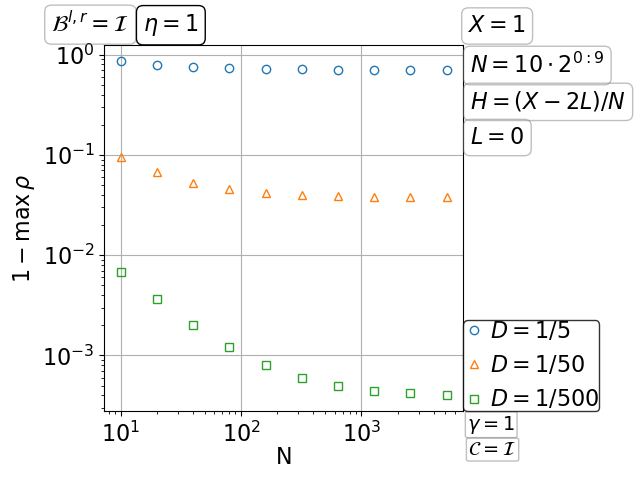}
    \caption{Convergence of the double sweep Schwarz method with PML transmission for the
      Dirichlet problem of diffusion on a fixed domain with increasing number of subdomains.}
    \label{figdp1d}
  \end{figure}

  \begin{figure}
    \centering%
    \includegraphics[width=.56\textwidth,trim=10 10 0 6,clip]{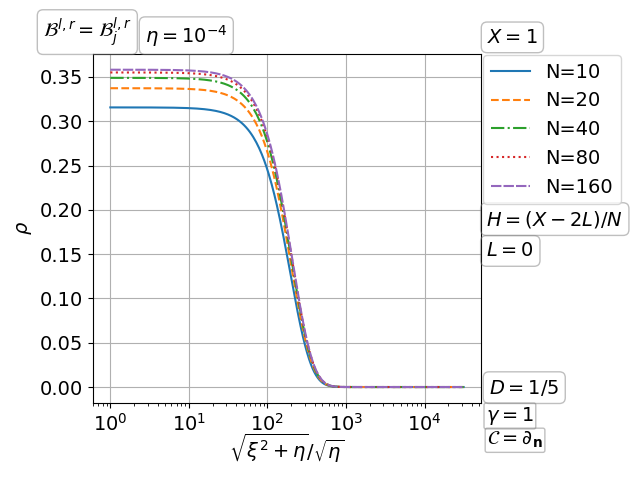}%
    \includegraphics[width=.44\textwidth,height=13em,trim=0 10 0 6,clip]{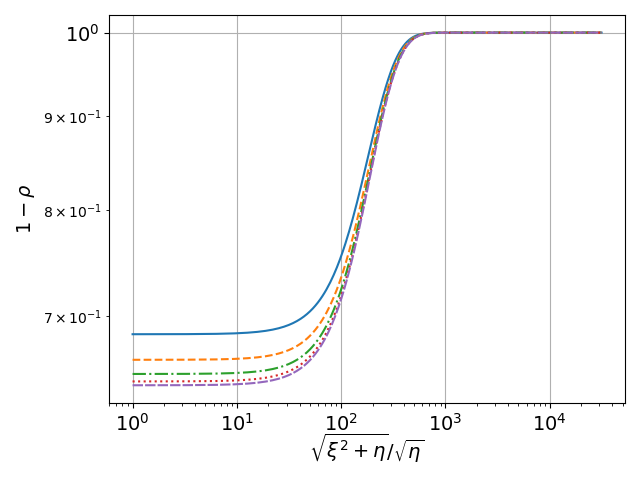}\\
    \includegraphics[width=.56\textwidth,trim=10 10 0 6,clip]{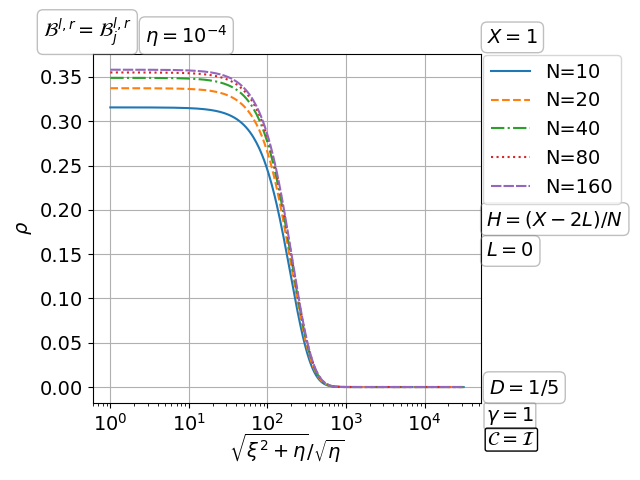}%
    \includegraphics[width=.44\textwidth,height=13em,trim=0 10 0 6,clip]{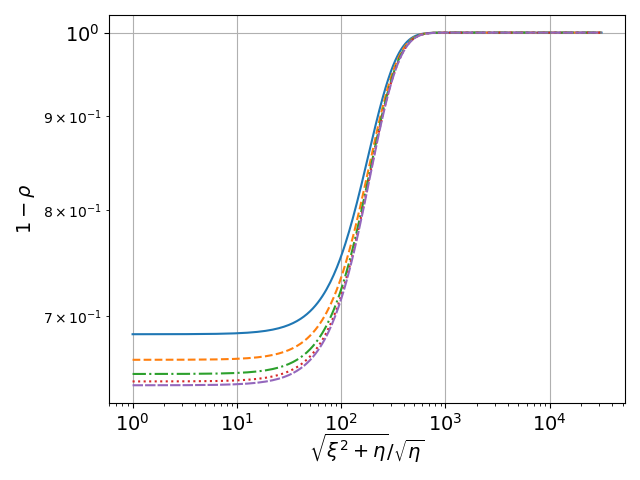}\\
    \includegraphics[width=.5\textwidth,trim=5 6 0 2,clip]{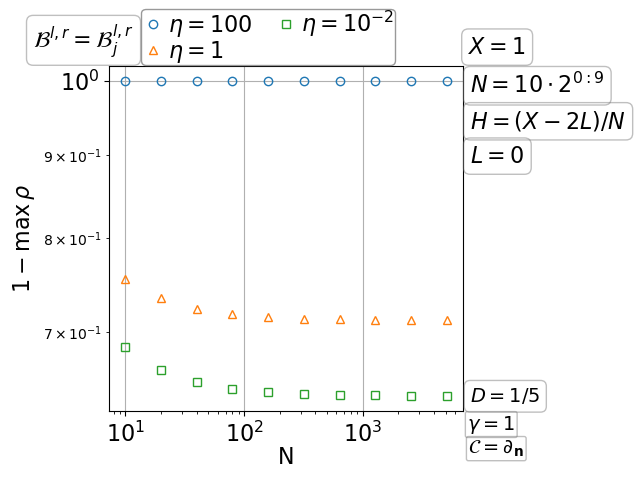}%
    \includegraphics[width=.5\textwidth,trim=5 6 0 2,clip]{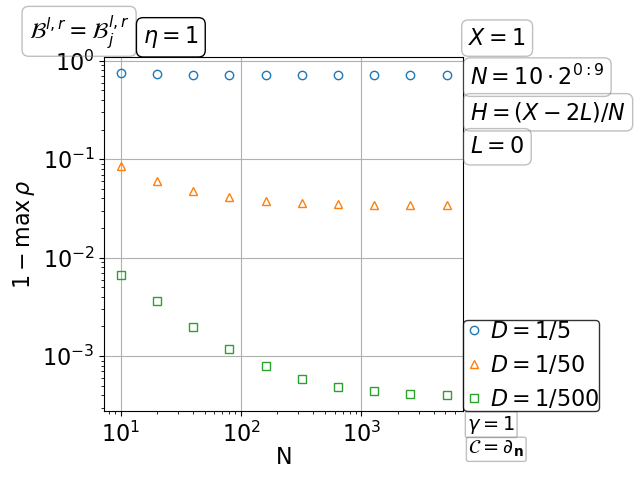}\\
    \includegraphics[width=.5\textwidth,trim=5 6 0 2,clip]{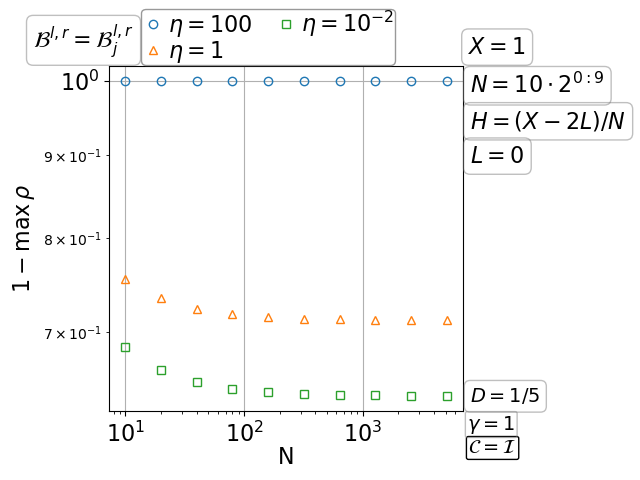}%
    \includegraphics[width=.5\textwidth,trim=5 6 0 2,clip]{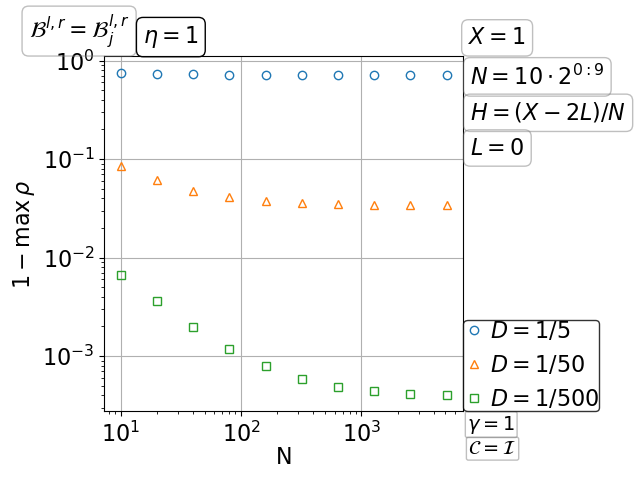}
    \caption{Convergence of the double sweep Schwarz method with PML transmission for the infinite
      pipe diffusion on a fixed domain with increasing number of subdomains.}
    \label{figdp1p}
  \end{figure}

\end{paragraph}

%%%%%%%%%%%%%%%%%%%%%%%%%%%%%%%%%%%%%%%%%%%%%%%%%%%%%%%%%%%%%%%%%%%%

\subsection{Parallel Schwarz method for the free space wave problem}\label{secpsfs}

%%%%%%%%%%%%%%%%%%%%%%%%%%%%%%%%%%%%%%%%%%%%%%%%%%%%%%%%%%%%%%%%%%%%

The free space wave problem corresponds to \R{eqprb} with $\mathcal{L}_x=-\partial_{xx}$,
$\eta=-\omega^2$, $\mathcal{B}^{l,r}=\mathcal{B}_j^{l,r}$, and $\mathcal{L}_y$, $\mathcal{B}^{b,t}$
related to absorbing conditions on bottom and top. If it is the Taylor of order zero condition, then
$\mathcal{L}_y=-\partial_{yy}$, $\mathcal{B}^{b,t}=\mp \partial_y -\I \omega$ and the generalised
Fourier frequency $\xi$ from the Sturm-Liouville problem \R{eqsturm} is on a curve in the complex
plane; see Figure~\ref{figxit0}. For the PML condition on bottom and top, the domain
$(0,X)\times(0,1)$ is extended to $(0,X)\times (-D, 1+D)$,
$\mathcal{L}_y=-\tilde{s}\partial_y(\tilde{s}\partial_y)$, $\mathcal{B}^{b,t}=\mp \partial_y$ and
$\xi=0, \tilde{\xi}\pi, 2\tilde{\xi}\pi, 3\tilde{\xi}\pi, ...$, where $\tilde{s}=1$ on $[0,1]$,
$\tilde{s}=(1-\I\tilde{\sigma}(-x))^{-1}$ on $(-D, 0)$, $\tilde{s}=(1-\I\tilde{\sigma}(x-1))^{-1}$
on $(1,1+D)$, $\int_0^D\tilde{\sigma}(x)\D{x}=\frac{1}{2}D\gamma$ and
$\tilde{\xi}=((2-\I\gamma)D+1)^{-1}$.  The generalised Fourier frequency $\xi$ is on a straight line
in the complex plane. For example, when $\gamma=5\pi/\omega$ and $D=1$, the range of $\xi$ is shown
in Figure~\ref{figxipml}.

\begin{figure}
  \centering
  \includegraphics[scale=.385]{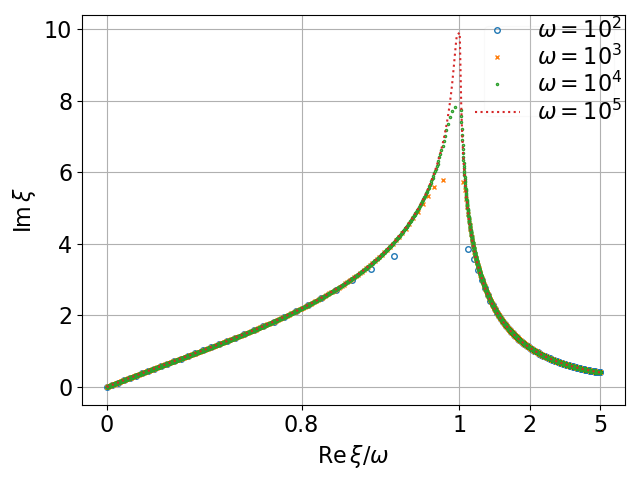}~
  \includegraphics[scale=.385]{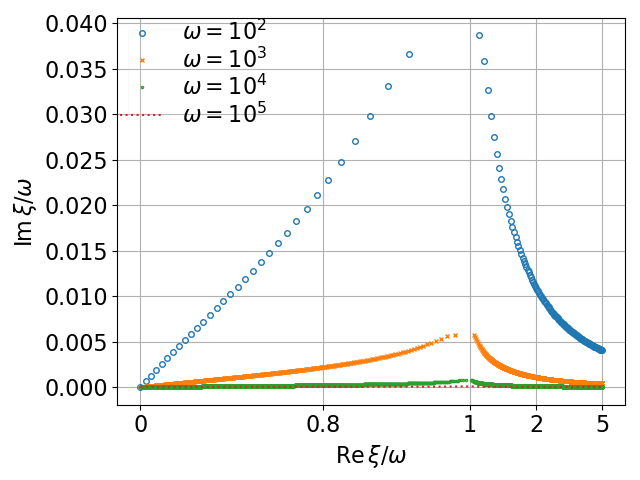}
  \caption{Generalised Fourier frequency $\xi$ from the Sturm-Liouville problem \R{eqsturm} with
    $Y=1$, $\mathcal{L}_y=-\partial_{yy}$, $\mathcal{B}^{b,t}=\mp \partial_y -\I \omega$ (Taylor of
    order zero condition).}
  \label{figxit0}
\end{figure}

\begin{figure}
  \centering
  \includegraphics[scale=.385]{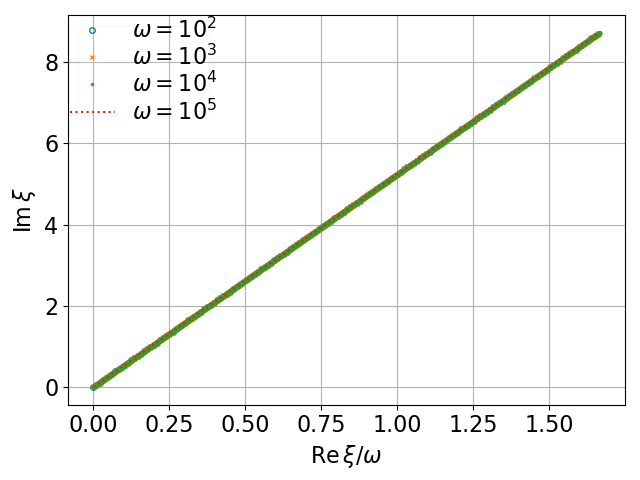}~
  \includegraphics[scale=.385]{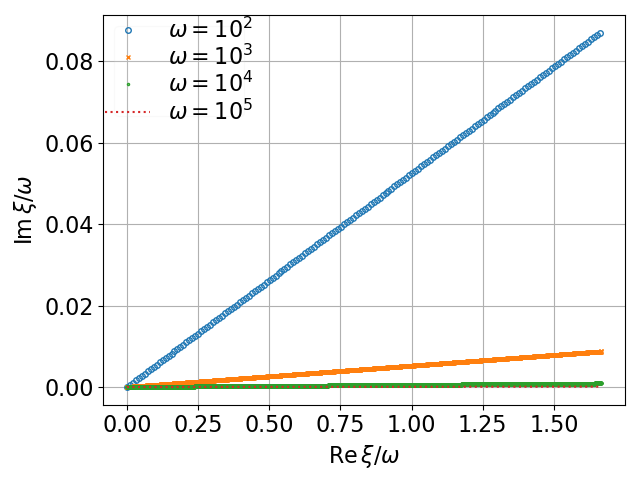}
  \caption{Generalised Fourier frequency $\xi$ from the Sturm-Liouville problem \R{eqsturm} on the
    extended domain $(-D,1+D)$ with the PML condition \R{eqpml}.}
  \label{figxipml}
\end{figure}

We note that the convergence factor $\rho$ can be viewed as a function of the rescaled Fourier
frequency $\xi_{\omega}:=\xi/\omega$, the rescaled subdomain width $H_{\omega}:=\omega H$, the
rescaled overlap width $L_{\omega}:=\omega L$ and the number of subdomains $N$.  But the range of the
generalised Fourier frequency $\xi$ depends on the wavenumber $\omega$ through the original boundary
condition on top and bottom. Increasing $\omega$ is not only to have more sampling points
$\frac{\Re \xi}{\omega}\in \mathbb{R}$ but also to decrease
$\frac{\Im \xi}{\omega}$; see again Figure~\ref{figxit0} and Figure~\ref{figxipml}. So, the
dependence of $\rho$ on $\omega$ is of particular interest besides the dependence on $\xi$, $H$, $L$
and $N$.

There are not many theoretical results on the convergence of Schwarz methods for the Helmholtz
equation. \cn{Despres} in his thesis proposed the first Schwarz method for the Helmholtz
equation. He replaced the classical overlapping Dirichlet transmission with the non-overlapping
Taylor of order zero transmission, and he proved the convergence of the resulting Schwarz iteration
in general geometry, arbitrary decomposition and variable media as long as part of the original
boundary is equipped with the Taylor of order zero condition. That idea has been further developed
and generalised. \cn{CGJ00} showed that for a decomposition without cross points the convergence of
the relaxed Schwarz iteration can be geometric if the Taylor of order zero transmission
$\partial_{\mathbf{n}}+\I\omega$ is enhanced by a square-root operator to
$\partial_{\mathbf{n}}+\I\omega\sqrt{\alpha-\beta\omega^{-2}\partial_{yy}}$. \cn{Claeys:2019:ADD}
analyzed the discrete version and showed that the convergence rate is uniform in the mesh size. For
recent progress along the direction of nonlocal transmission conditions, see \cn{lecouvez2014quasi},
\cn{collino2020exponentially}, \cn{parolin}, \cn{Claeys:2021:RTO}, \cn{claeys2021non}. For local
transmission conditions, the convergence rate of the Schwarz preconditioned Krylov iteration was first
analyzed by \cn{GrahamSpenceZou} and then generalised by \cn{Gonghetero}, \cn{BCNT}. Besides the
above general theories, convergence for domain decomposition in a rectangle has also been
studied. \cn{GongPS} considered the high wavenumber asymptotic. A variational interpretation of the
Schwarz preconditioner was discussed in \cn{Gongvariational}, and an analysis at the discrete level
was given in \cn{GongRAS}. \cn{Chen13a} gave the first convergence estimate of the double sweep
Schwarz method with PML transmission. A recursive double sweep Schwarz method was analyzed by
\cn{du2020pure}. For the waveguide problem, we refer to \cn{KimZhang}, \cn{Kimsweep}, \cn{KimPML}.

%%%%%%%%%%%%%%%%%%%%%%%%%%%%%%%%

\subsubsection{Parallel Schwarz method with Taylor of order zero transmission for the free space
  wave problem}

In this case, we assume that the original boundary condition is also
the Taylor of order zero condition, \ie,
$\mathcal{B}^{b,t}=\mathcal{B}^{l,r}=\mathcal{B}^{l,r}_j$. According
to \cn{Despres}, the parallel Schwarz method with non-overlapping
Taylor of order zero transmission converges. However, the convergence
rate for the evanescent modes is very slow. To mitigate the situation,
here we shall add an overlap, though the convergence is no longer
guaranteed by any theory. On the one hand, divergence was observed if
the overlap exceeds a threshold related to the wavenumber $\omega$ and
the subdomain width $H$\footnote{The restriction of small overlap disappears when $H$ is big engouh, \eg H-12 as shown for two subdomains in Figure and Table. In that case, a fixed generous overlap can lead to convergence which is robust in the wavenumber.}. On the other hand, intuitively, we still
expect convergence if the overlap is sufficiently small, which we
will find to be the case from the following paragraphs, see also
  the two subdomain case in Section \ref{2SubSec}.
Indeed, we carried out a scaling test of shrinking overlap $L\to 0$ on
  a fixed domain with 10 subdomains and fixed wavenumber and found the same
  scaling as in \R{TaylorRhoRRHelmholtz}.

\begin{paragraph}{Convergence with increasing number of fixed size subdomains}

  The graph of the convergence factor $\rho=\rho(\xi)$ is shown in the
  top half of Figure~\ref{figfspt0n}. Note that, for a finite number
  of subdomains $N$, $\rho(0)=0$ because the Taylor expansion in the
  transmission condition is exactly at the point $\xi=0$, but $\rho$
  grows drastically in the neighborhood of $\xi=0$. Actually we see
  that $\rho_{\infty}:=\lim_{N\to\infty}\rho$ is tending to one as
  $\xi\to0$. Note also that there is no singularity at $\Re\xi=\omega$
  because $\Im\xi\ne0$ due to the top and bottom Taylor of order zero
  conditions (see Figure~\ref{figxit0}), and the convergence of the
  evanescent modes corresponding to $\Re\xi>\omega$ is good and
  independent of $N$. Given the limiting graph $\rho_{\infty}$, we see
  that as $N\to\infty$, the maximum point of $\rho$ decreases toward
  zero. Since $\xi$ is discrete and $\rho(0)=0$, it seems that for
  sufficiently large $N$ the maximum of $\rho$ will be attained at the
  first nonzero Fourier frequency $\xi_1$ (close to $\pi$) and the
  maximum value tends to $\rho_{\infty}(\xi_1)$. But this asymptotic
  is hardly observed up to $N=5200$ in the bottom half of
  Figure~\ref{figfspt0n}.
% \marginpar{\MG{this is weird, what is happening here?}}
  Rather, we find $\max_{\xi}\rho =
  1-O(N^{-1})$ in most cases. An exception occurs in the bottom left
  subplot when $\omega=100$, $H=1/20$ and $L=H/40$, for which
  $\max_{\xi}\rho=1-O(N^{-3/2})$ is observed.

  \begin{figure}
    \centering%
    \includegraphics[width=.56\textwidth,trim=10 10 0 6,clip]{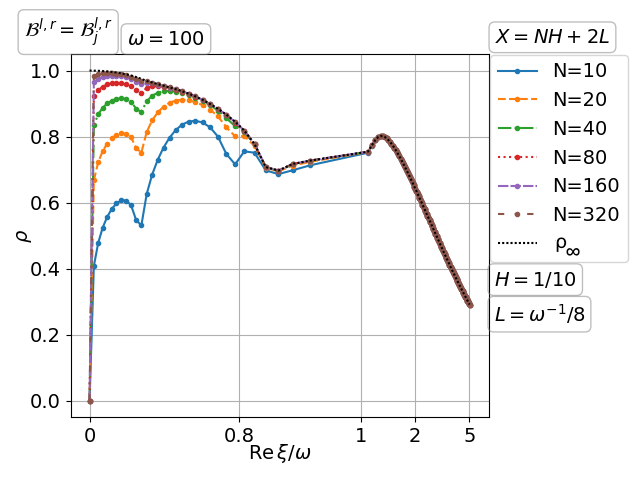}%
    \includegraphics[width=.44\textwidth,height=13em,trim=0 10 0 6,clip]{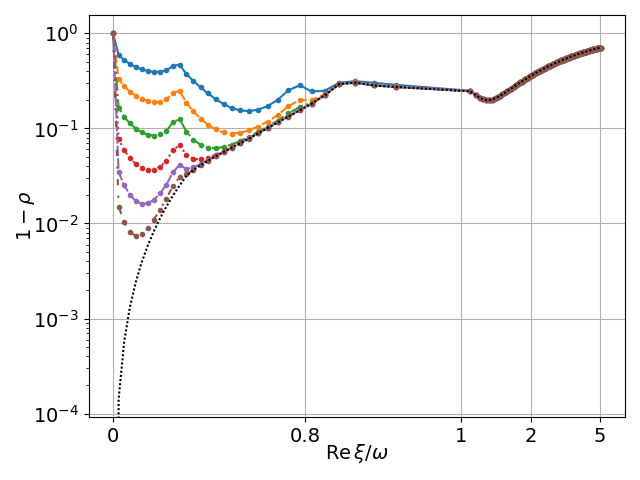}\\
    \includegraphics[width=.56\textwidth,trim=10 10 0 6,clip]{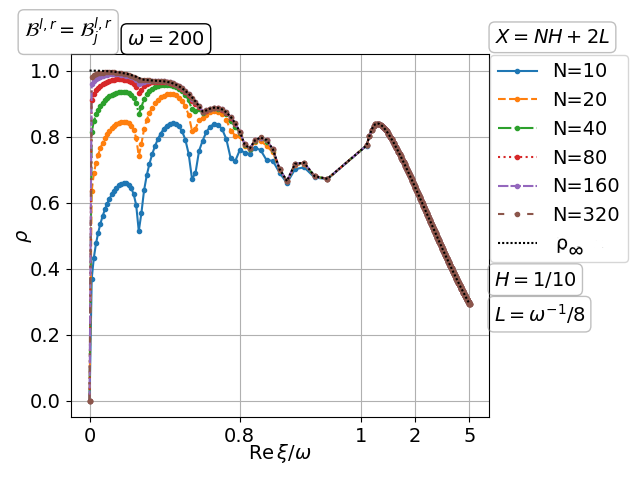}%
    \includegraphics[width=.44\textwidth,height=13em,trim=0 10 0 6,clip]{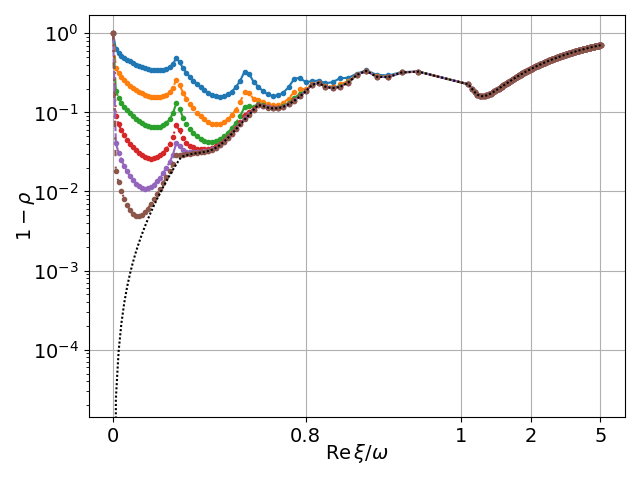}\\
    \includegraphics[width=.5\textwidth,trim=5 6 0 2,clip]{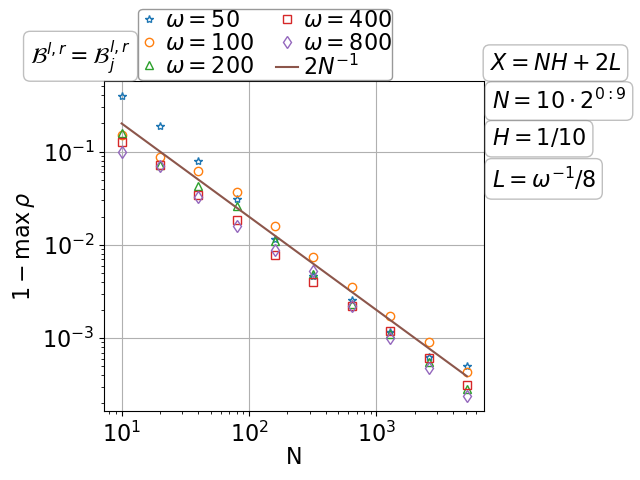}%
    \includegraphics[width=.5\textwidth,trim=5 6 0 2,clip]{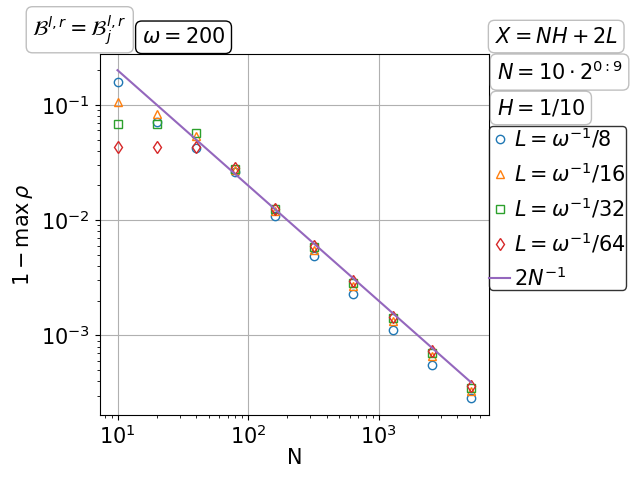}\\
    \includegraphics[width=.5\textwidth,trim=5 6 0 2,clip]{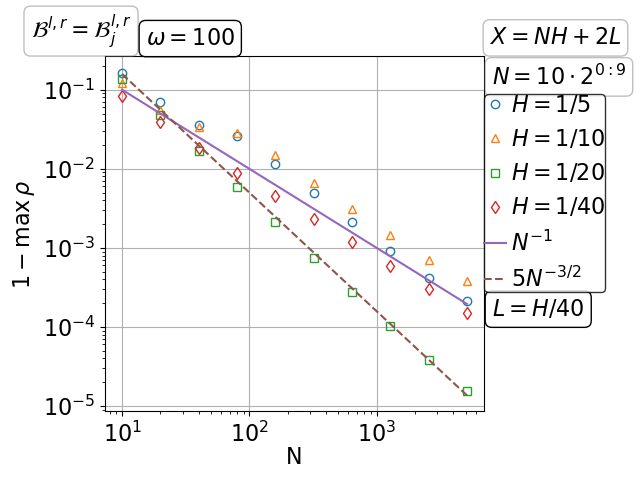}%
    \includegraphics[width=.5\textwidth,trim=5 6 0 2,clip]{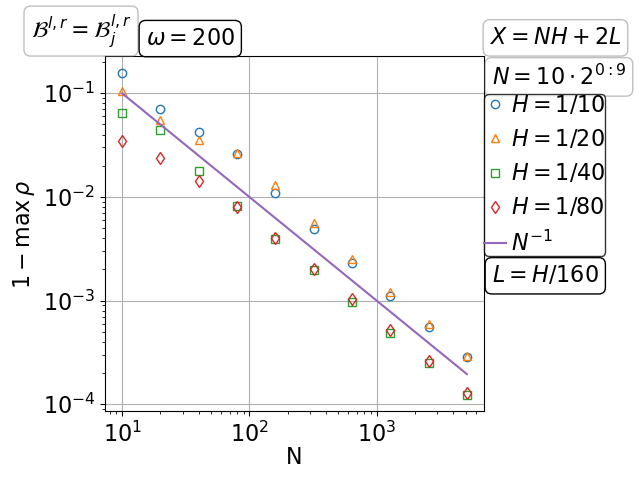}%
    \caption{Convergence of the parallel Schwarz method with Taylor of order zero transmission for
      the free space wave with increasing number of fixed size subdomains.}
    \label{figfspt0n}
  \end{figure}

\end{paragraph}

\begin{paragraph}{Convergence on a fixed domain with increasing number of subdomains}

  Now we increase the number of subdomains $N$ for a fixed domain width $X$ and a fixed wavenumber
  $\omega$. This time all the modes have slower convergence for bigger $N$; see the top half of
  Figure~\ref{figfspt01}. Away from $\xi=0$, the behavior of $\rho$ can be explained by that of the
  limiting curve $\rho_{\infty}$ with the corresponding subdomain width $H=(X-2L)/N\to0$. As $N$
  grows, the maximum point of $\rho$ can change between the two regions $\Re\xi/\omega<1$ and
  $\Re\xi/\omega>1$. This is because smaller $L$ benefits convergence of the propagating modes but
  hinders convergence of the evanescent modes. Nevertheless, a scaling of $\max_{\xi}\rho$ with
  $N\to\infty$ still appears; see the bottom half of Figure~\ref{figfspt01}. It is estimated that
  $\max_{\xi}\rho=1-O(N^{-5/3})$ if $L=O(H)$ is sufficiently small and
  $\max_{\xi}\rho=1-O(N^{-5/2})$ if $L=O(H^2)$ is sufficiently small.

    \begin{figure}
    \centering%
    \includegraphics[width=.56\textwidth,trim=10 10 0 6,clip]{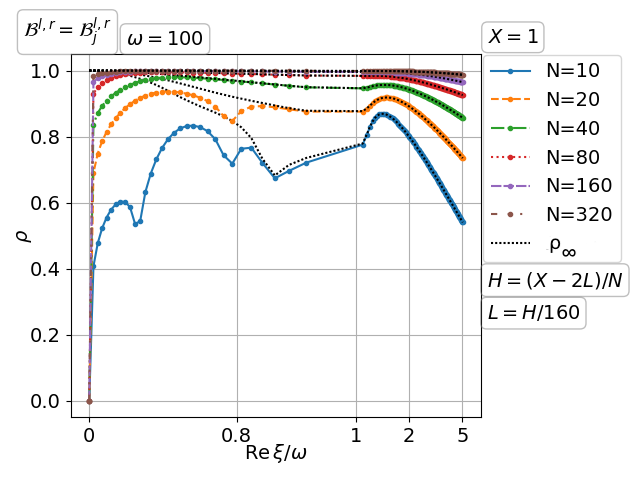}%
    \includegraphics[width=.44\textwidth,height=13em,trim=0 10 0 6,clip]{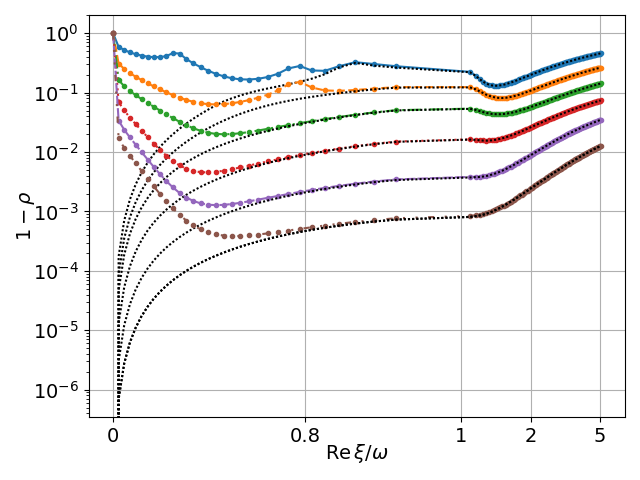}\\
    \includegraphics[width=.56\textwidth,trim=10 10 0 6,clip]{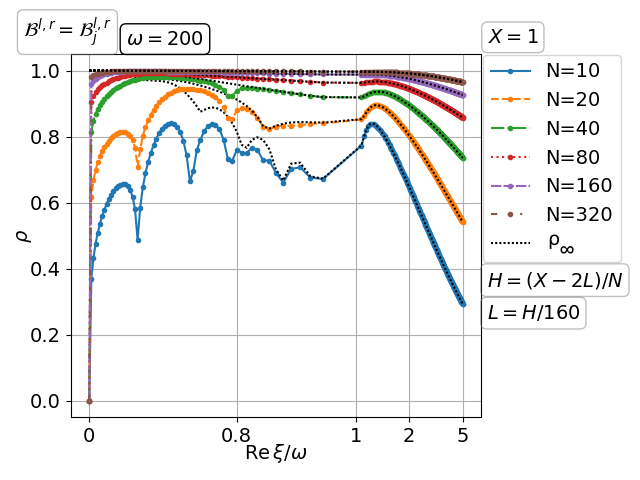}%
    \includegraphics[width=.44\textwidth,height=13em,trim=0 10 0 6,clip]{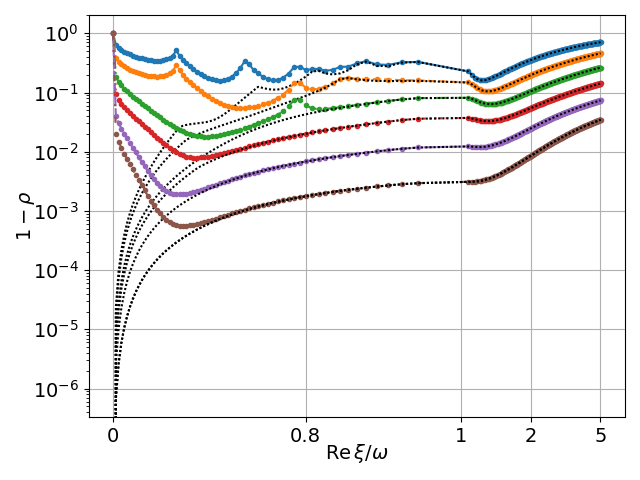}\\
    \includegraphics[width=.5\textwidth,trim=5 6 0 2,clip]{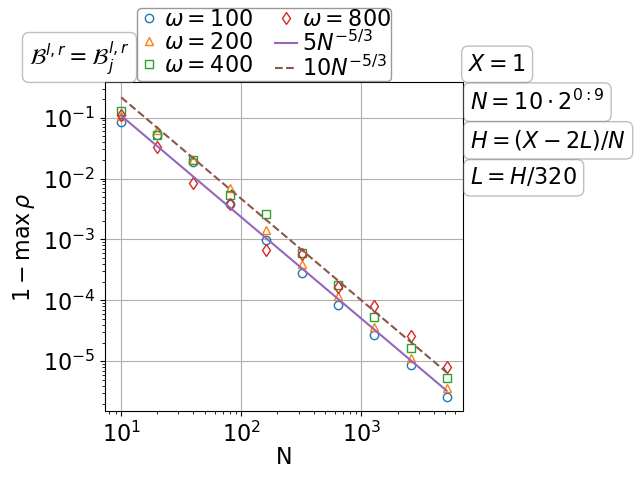}%
    \includegraphics[width=.5\textwidth,trim=5 6 0 2,clip]{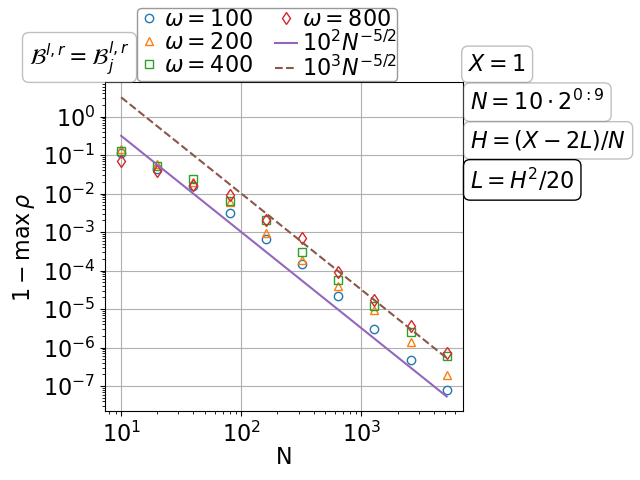}\\
    \includegraphics[width=.5\textwidth,trim=5 6 0 2,clip]{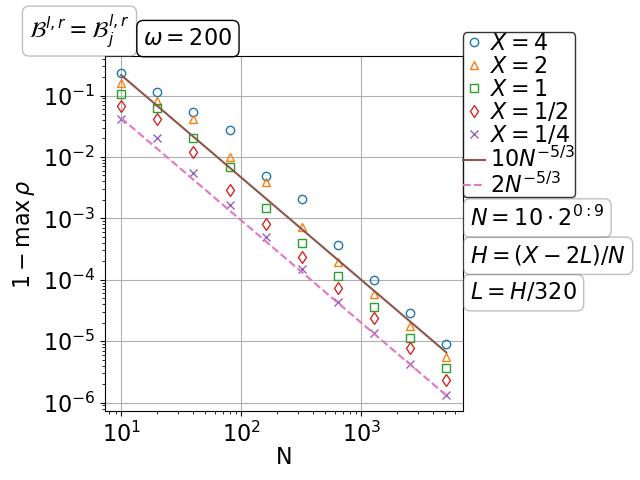}%
    \includegraphics[width=.5\textwidth,trim=5 6 0 2,clip]{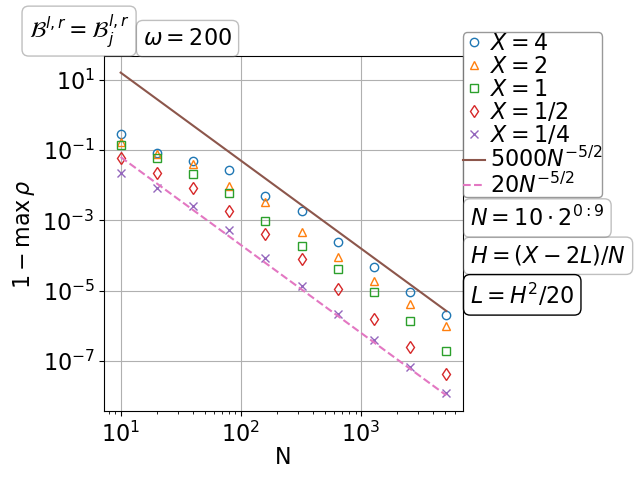}%
    \caption{Convergence of the parallel Schwarz method with Taylor of order zero transmission for
      the free space wave on a fixed domain with increasing number of subdomains.}
    \label{figfspt01}
  \end{figure}
  
\end{paragraph}

\begin{paragraph}{Convergence on a fixed domain with a fixed number of subdomains for increasing
    wavenumber}

  We show the graph of the convergence factor $\rho$ for the domain width $X=1$, the number of
  subdomains $N=40$, the overlap width $L=\omega^{-1}/40$ and growing wavenumber $\omega$ in the top
  half of Figure~\ref{figfspt0w}. We find that the graph of $1-\rho$ has a limiting profile in the
  propagating region $\Re\xi/\omega<1$ as $\omega\to\infty$ and has a unique local minimum in the
  evanescent region $\Re\xi/\omega>1$ which decreases as $\omega\to\infty$. So, the Schwarz method
  is preasymptotically robust about $\omega$ until the local maximum of $\rho$ over
  $\Re\xi/\omega>1$ becomes dominating --a phenomenon observed numerically in \cn{GZ16}. The
  asymptotic scaling is estimated in the bottom half of Figure~\ref{figfspt0w} as
  $\max_{\xi}\rho=\max_{\Re\xi>\omega}\rho=1-O(\omega^{-9/20})$ for $L=O(\omega^{-1})$ and
  $1-O(\omega^{-2/3})$ for $L=O(\omega^{-3/2})$.
  
  \begin{figure}
    \centering%
    \includegraphics[width=.56\textwidth,trim=10 10 0 6,clip]{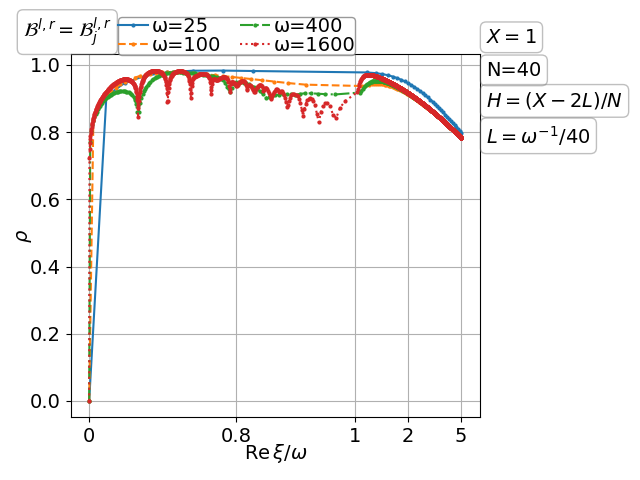}%
    \includegraphics[width=.44\textwidth,height=13em,trim=0 10 0 6,clip]{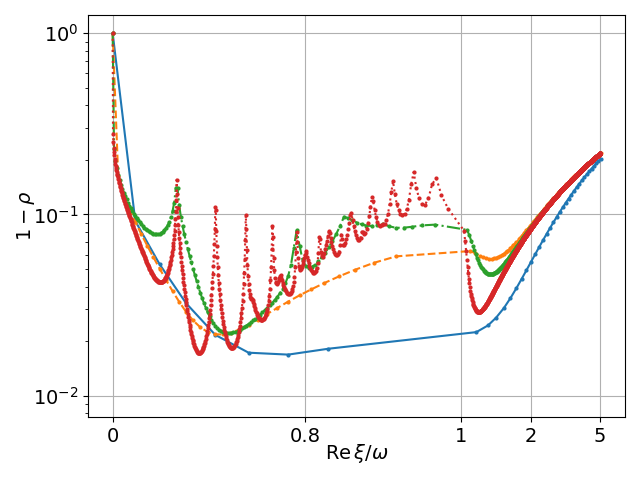}\\
    \includegraphics[width=.5\textwidth,trim=5 6 0 2,clip]{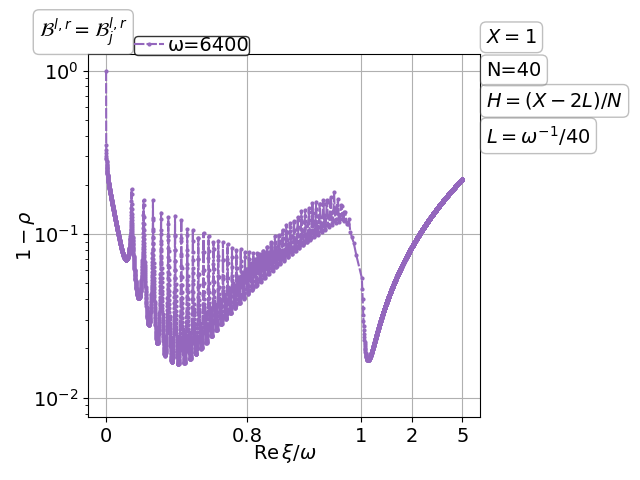}%
    \includegraphics[width=.5\textwidth,trim=5 6 0 2,clip]{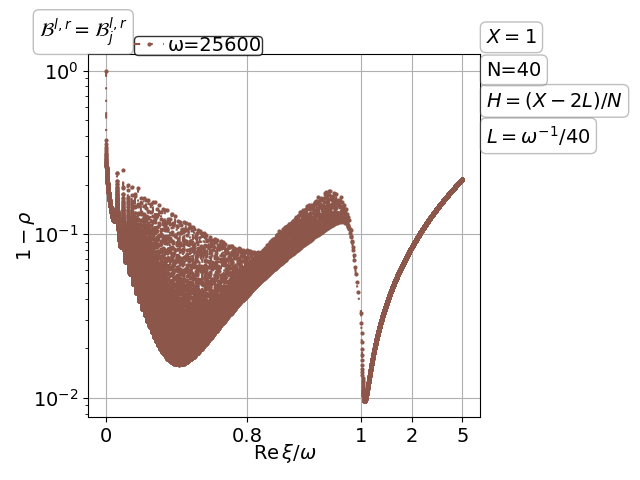}\\
    \includegraphics[width=.5\textwidth,trim=5 6 0 2,clip]{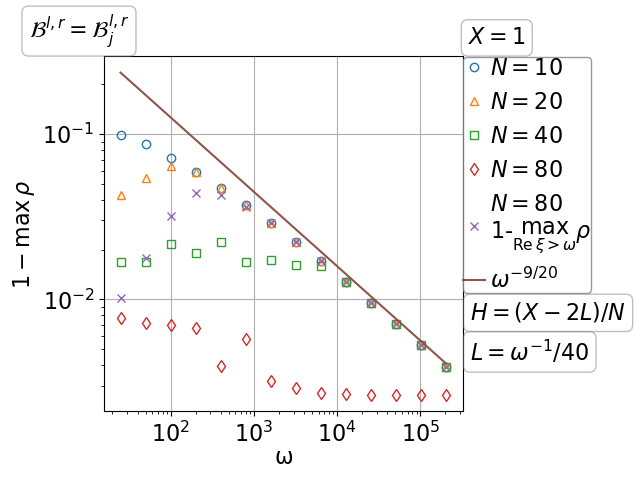}%
    \includegraphics[width=.5\textwidth,trim=5 6 0 2,clip]{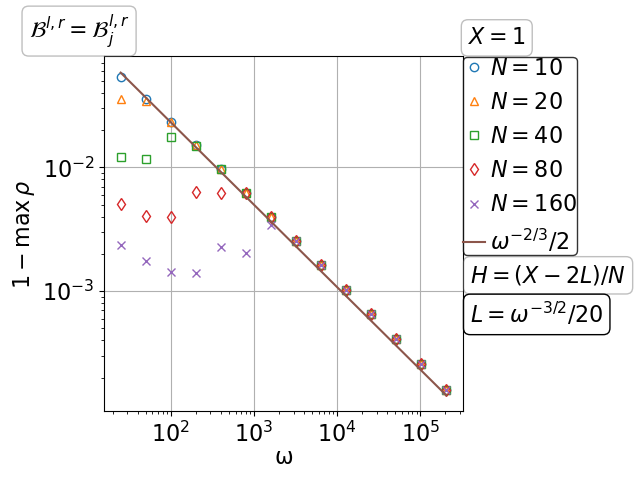}\\
    \includegraphics[width=.5\textwidth,trim=5 6 0 2,clip]{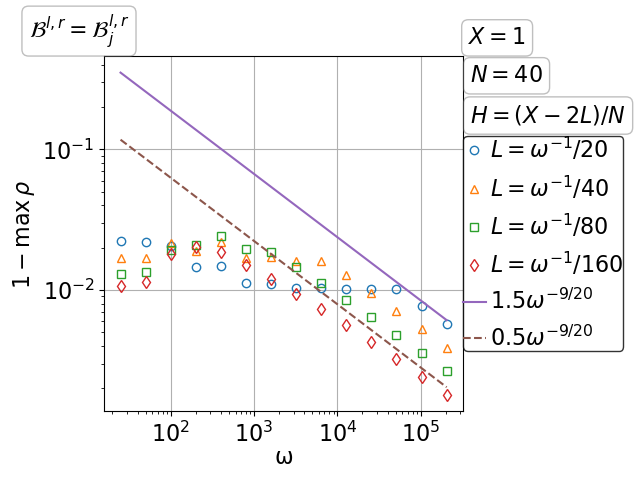}%
    \includegraphics[width=.5\textwidth,trim=5 6 0 2,clip]{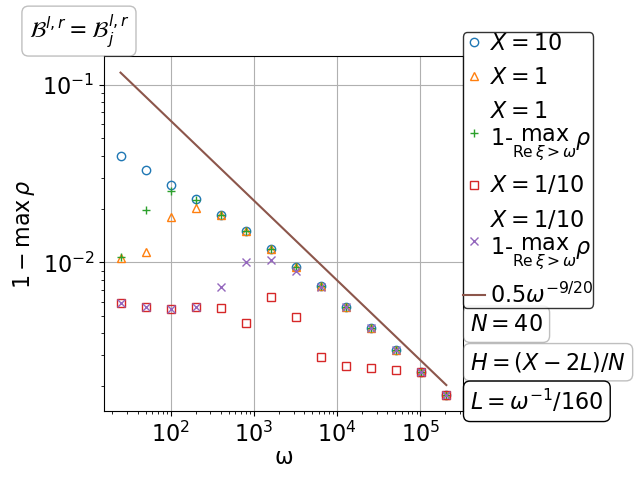}%
    \caption{Convergence of the parallel Schwarz method with Taylor of order zero transmission for
      the free space wave on a fixed domain with a fixed number of subdomains for increasing
      wavenumber.}
    \label{figfspt0w}
  \end{figure}

\end{paragraph}

\begin{paragraph}{Convergence on a fixed domain with number of subdomains increasing with the
    wavenumber}

  The scaling here is more challenging because both the increasing wavenumber and the shrinking
  subdomain width force the overlap width $L$ to be small, and $L=O(\omega^{-1})$ is not small
  enough for convergence. We use $L=O(\omega^{-3/2})$ in the top half of Figure~\ref{figfspt0wh}
  for the graph of the convergence factor $\rho=\rho(\xi)$. Both the local maxima of $\rho$ in the
  two regions $\Re\xi<\omega$ and $\Re\xi>\omega$ increase with $N$, with the first one increasing
  faster and the seccond one dominating in the preasymptotic regime. The scaling of $\max_{\xi}\rho$ is
  estimated in the bottom half of Figure~\ref{figfspt0wh}.  Asymptotically
  $\max_{\xi}\rho=1-O(N^{-2})$ for $L=O(\omega^{-3/2})$ and $\max_{\xi}\rho=1-O(N^{-5/3})$ for
  $L=O(\omega^{-2})$, while the preasymptotic scaling is better.
  
  \begin{figure}
    \centering%
    \includegraphics[width=.56\textwidth,trim=10 10 0 6,clip]{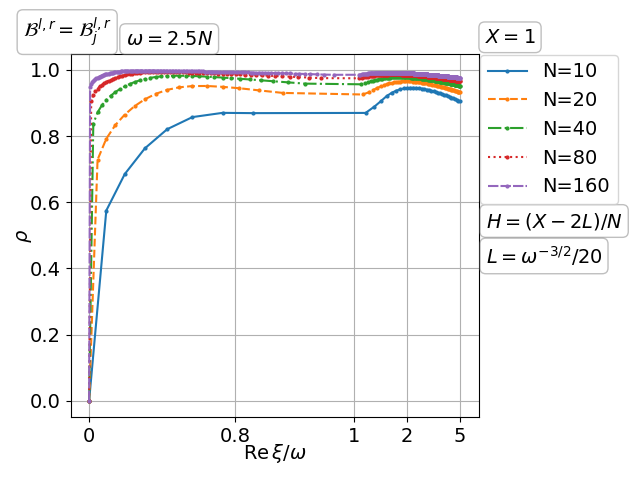}%
    \includegraphics[width=.44\textwidth,height=13em,trim=0 10 0 6,clip]{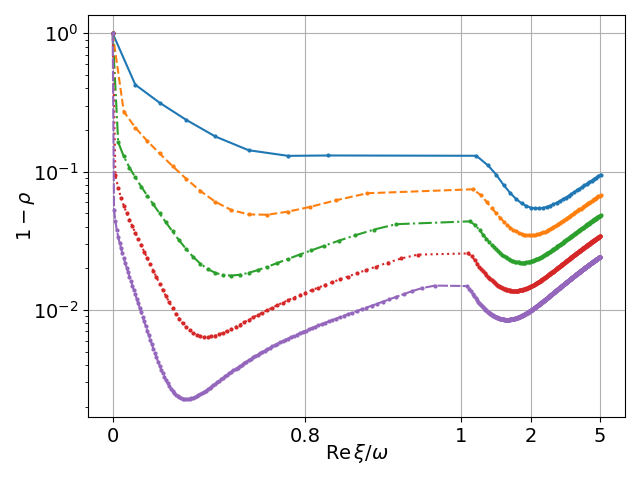}\\
    \includegraphics[width=.56\textwidth,trim=10 10 0 6,clip]{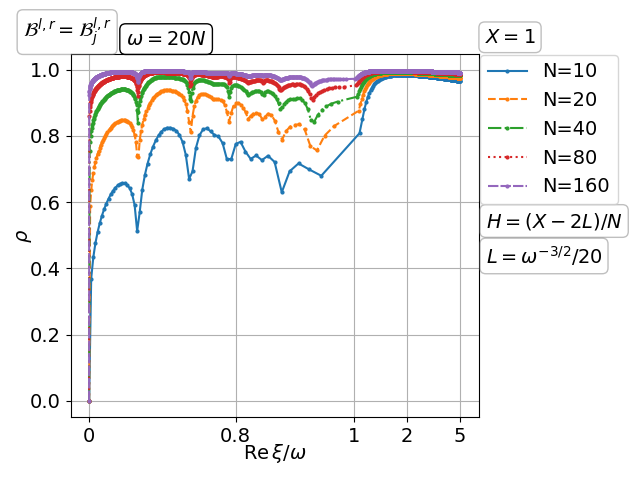}%
    \includegraphics[width=.44\textwidth,height=13em,trim=0 10 0 6,clip]{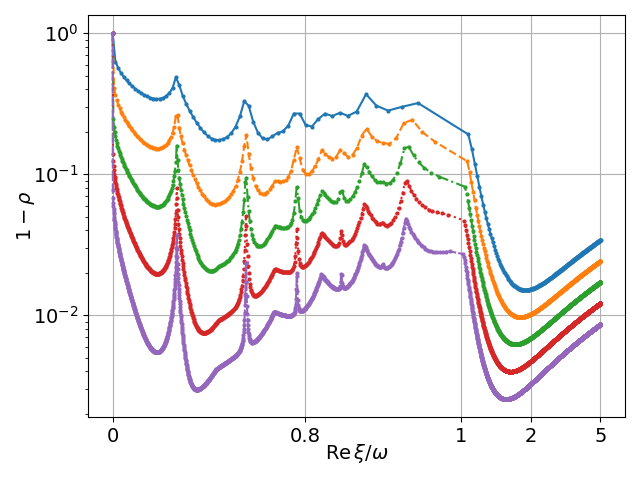}\\
    \includegraphics[width=.5\textwidth,trim=5 6 0 2,clip]{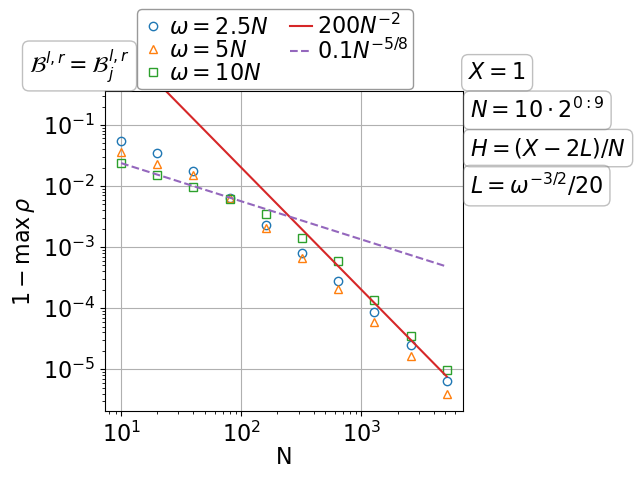}%
    \includegraphics[width=.5\textwidth,trim=5 6 0 2,clip]{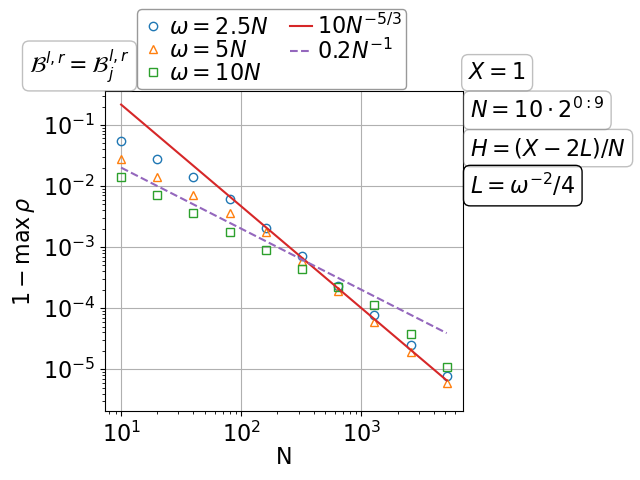}\\
    \includegraphics[width=.5\textwidth,trim=5 6 0 2,clip]{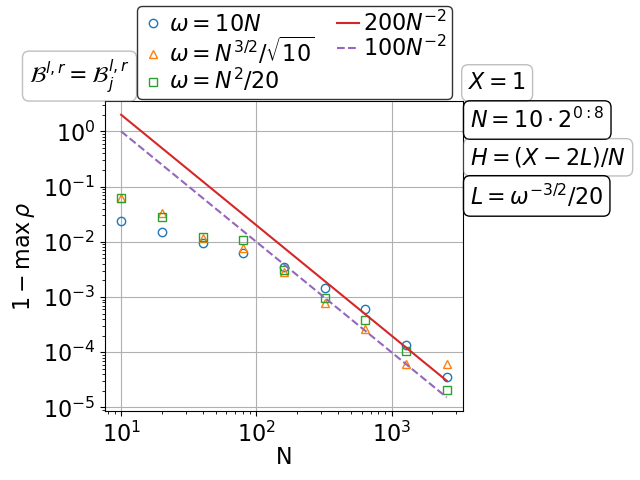}%
    \includegraphics[width=.5\textwidth,trim=5 6 0 2,clip]{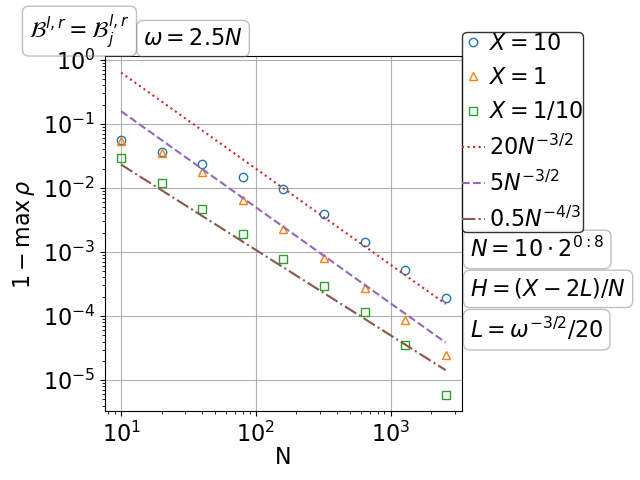}%
    \caption{Convergence of the parallel Schwarz method with Taylor of order zero transmission for
      the free space wave on a fixed domain with number of subdomains increasing with the
      wavenumber.}
    \label{figfspt0wh}
  \end{figure}

\end{paragraph}

%%%%%%%%%%%%%%%%%%%%%%%%%%%%%%%%

\subsubsection{Parallel Schwarz method with PML transmission for the free space wave problem}

In this case, we assume that the original boundary condition is also
the PML condition, \ie,
$\mathcal{B}^{b,t}=\mathcal{B}^{l,r}=\mathcal{B}_j^{l,r}$. For
analysis purposes, the PML on the left and right is regarded as a
boundary condition involving \R{eqsw}, while the PML on top and bottom
is treated as part of the extended domain. The use of PML transmission
conditions for parallel Schwarz methods can be traced back to
\cn{Toselli}, see also \cn{Schadle}, \cn{SZBKS}, \cn{GuddatiDD20}. It
is combined with a coarse space in \cn{astaneh}. Since the PML
condition can be made arbitrarily close to the exact transparent
condition, \ie, the PML Dirichlet-to-Neumann operator \R{eqsw} can
approximate the exact Dirichlet-to-Neumann operator to arbitrary
accuracy, we can expect the Schwarz method to converge as soon as the
PML is sufficiently accurate; see also \cn[Theorem 5.1]{NN97}. This
has been first quantified by \cn{Chen13a}, though for the double sweep
Schwarz method. There is a pending issue about how accurate the PML
needs to be to ensure a robust convergence. We shall address the issue
at the continuous level for the parallel Schwarz method under
different scalings in this subsection.

\begin{paragraph}{Convergence with increasing number of fixed size subdomains}

  On the first row of Figure~\ref{figfsppn}, we check the influence of the terminating condition of
  the PML. One effect is from the top and bottom PMLs. The Neumann terminated PMLs give the zero frequency
  $\xi=0$, while the Dirichlet terminated PMLs do not. This has already made a difference on the
  convergence. The limiting convergence factor $\rho_{\infty}$ for the Neumann terminated PMLs is
  actually greater than one at $\xi=0$, while that for the Dirichlet terminated PMLs is less than one
  over the range of frequencies. Otherwise, quantitively the terminating condition of the PMLs makes little
  difference. The graph of the convergence factor $\rho$ in the top half of Figure~\ref{figfsppn}
  looks somewhat similar to Figure~\ref{figppnp} for the diffusion problem: both are better at 
  higher frequencies $\xi$. Roughly speaking, the PML works like an overlap but with all the modes
  decaying inside it, and the top and bottom PMLs make the decaying faster for higher frequencies. As
  we have encountered for the diffusion problem, the parallel Schwarz method with a fixed PML for the
  wave problem also slows down with increasing number of subdomains $N$. A logarithmic growth of the
  PML width $D$ with $N$ does not change the scaling; see the third row of Figure~\ref{figfsppn},
  where the dependence on the wavenumber $\omega$ is also measured. Then, in the last row of the
  figure, we recorded the dependence on $D$ and the subdomain width $H$. We actually find that
  $1-\max_{\xi}\rho=O(N^{-1})$ improves its hidden factor as $D$ grows and is essentially
  independent of $H$ and $\omega$.
  
  \begin{figure}
    \centering%
    \includegraphics[width=.5\textwidth,trim=5 6 0 2,clip]{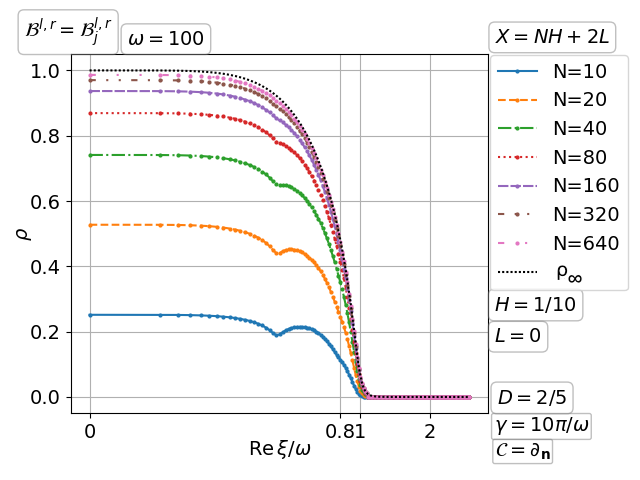}%
    \includegraphics[width=.5\textwidth,trim=5 6 0 2,clip]{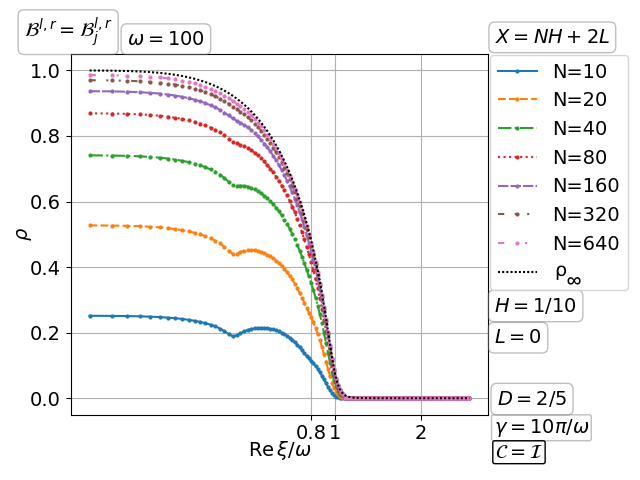}\\
    \includegraphics[width=.56\textwidth,trim=10 10 0 6,clip]{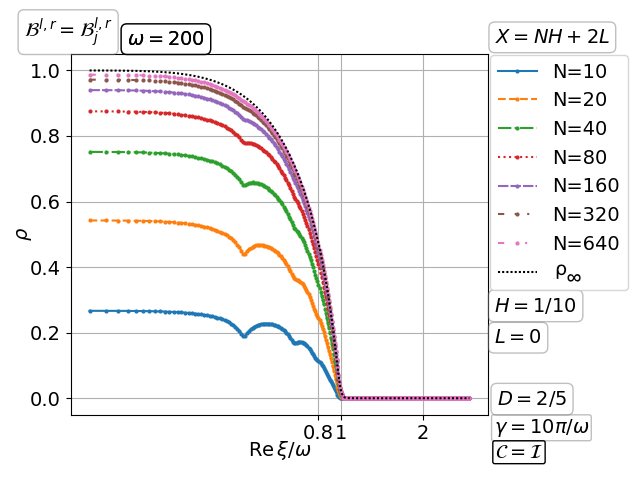}%
    \includegraphics[width=.44\textwidth,height=13em,trim=0 10 0 6,clip]{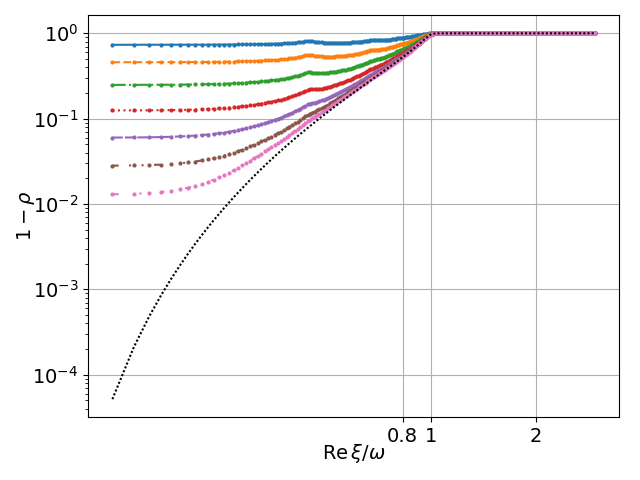}\\
    \includegraphics[width=.5\textwidth,trim=5 6 0 2,clip]{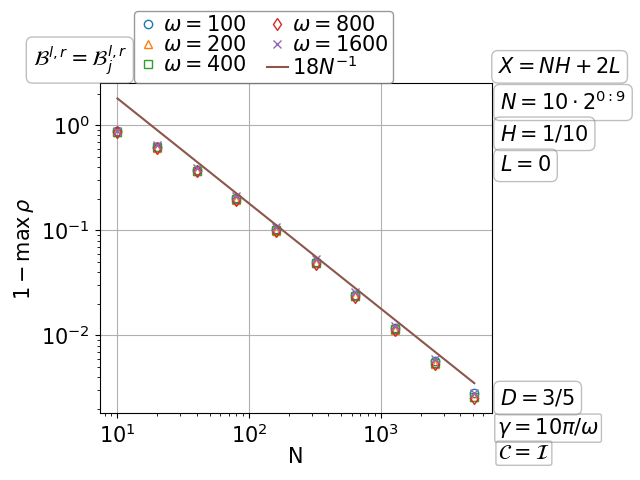}%
    \includegraphics[width=.5\textwidth,trim=5 6 0 2,clip]{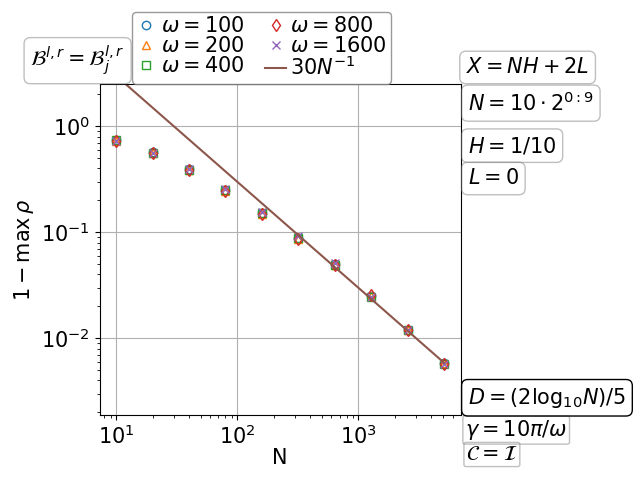}\\
    \includegraphics[width=.5\textwidth,trim=5 6 0 2,clip]{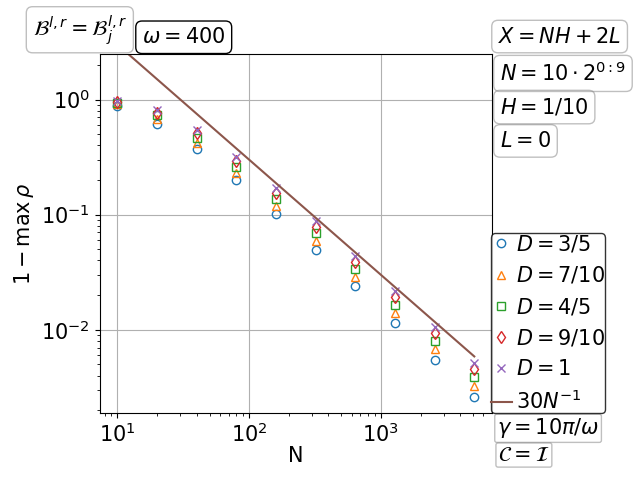}%
    \includegraphics[width=.5\textwidth,trim=5 6 0 2,clip]{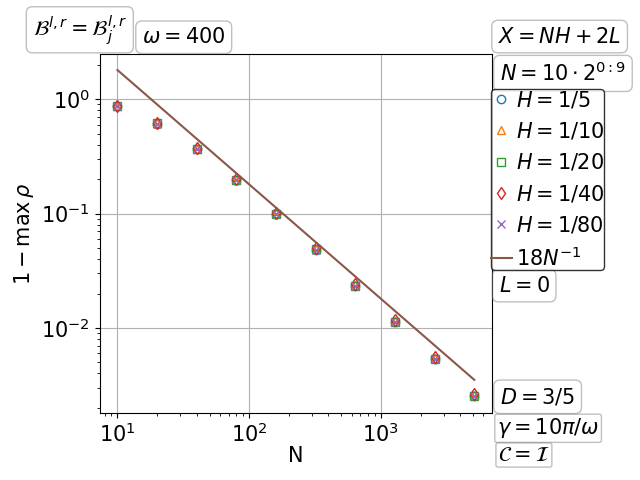}%
    \caption{Convergence of the parallel Schwarz method with PML transmission for
      the free space wave with increasing number of fixed size subdomains.}
    \label{figfsppn}
  \end{figure}

\end{paragraph}

\begin{paragraph}{Convergence on a fixed domain with increasing number of subdomains}

  Now we consider solving a fixed problem with increasing number of subdomains $N$. In the top half
  of Figure~\ref{figfspp1}, the graph of the convergence factor $\rho=\rho(\xi)$ for each $N$ is
  plotted along with the limiting convergence factor $\rho_{\infty}$ with the corresponding
  $H=X/N$. In this case, $\rho$ is much better than $\rho_{\infty}$, similar to that for the
  diffusion problem (see Figure~\ref{figpp1p}). The scaling of $\max_{\xi}\rho$ and its dependence
  on the wavenumber $\omega$ are measured on the third row of Figure~\ref{figfspp1}. On the right,
  the PML width $D$ grows logarithmically with $N$, and we see that when $D$ becomes  sufficiently large, $\max_{\xi}\rho$ drops to numerical
  zero. However, this does not mean {that the convergence does not depend on the number of
    subdomains $N$,} because the optimal parallel Schwarz iteration is nilpotent of index $N$. The
  last row of the figure checks the dependence on $D$ and the domain width $X$. We find
  $\max_{\xi}\rho=1-O(N^{-1})$ which improves its hidden factor as $D$ grows and depends very
  mildly on $X$.

    \begin{figure}
    \centering%
    \includegraphics[width=.56\textwidth,trim=10 10 0 6,clip]{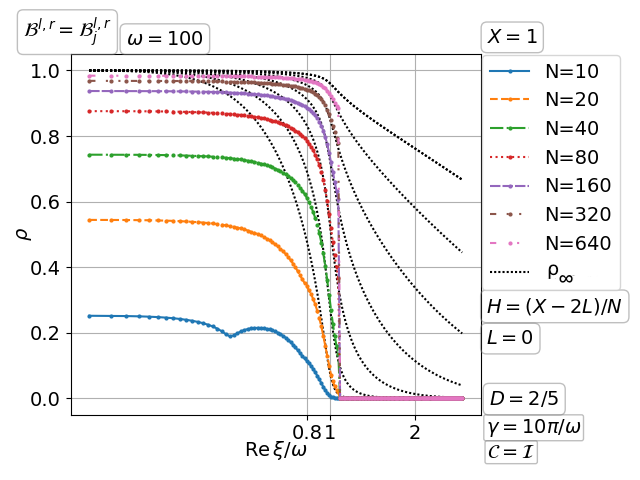}%
    \includegraphics[width=.44\textwidth,height=13em,trim=0 10 0 6,clip]{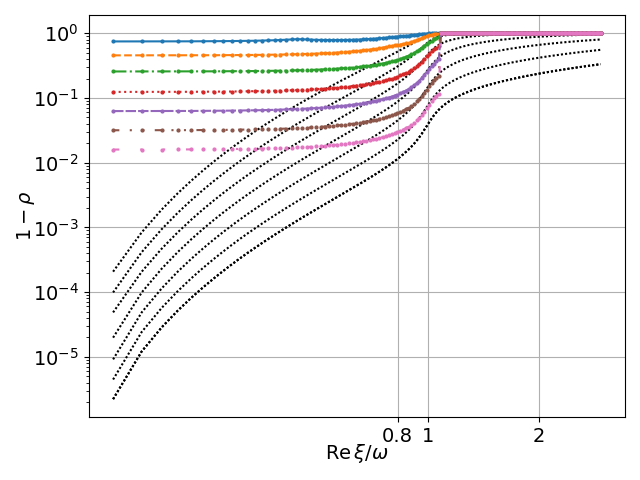}\\
    \includegraphics[width=.56\textwidth,trim=10 10 0 6,clip]{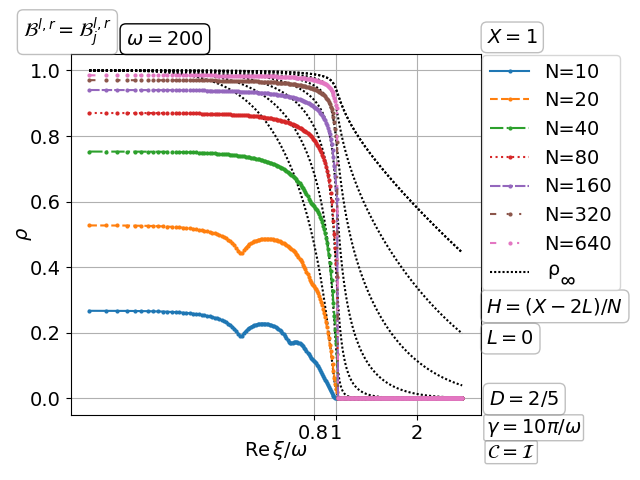}%
    \includegraphics[width=.44\textwidth,height=13em,trim=0 10 0 6,clip]{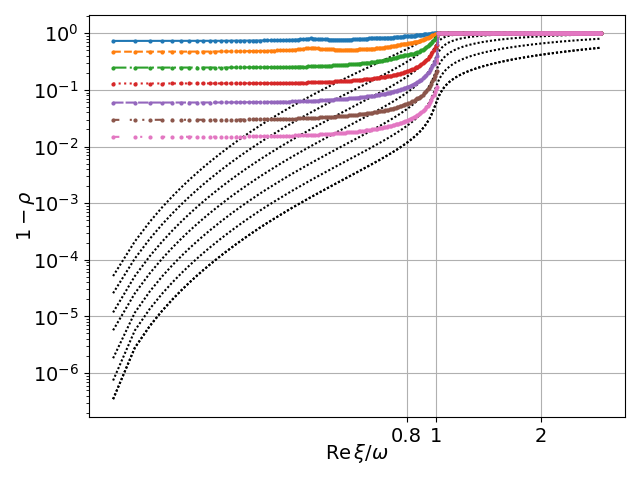}\\
    \includegraphics[width=.5\textwidth,trim=5 6 0 2,clip]{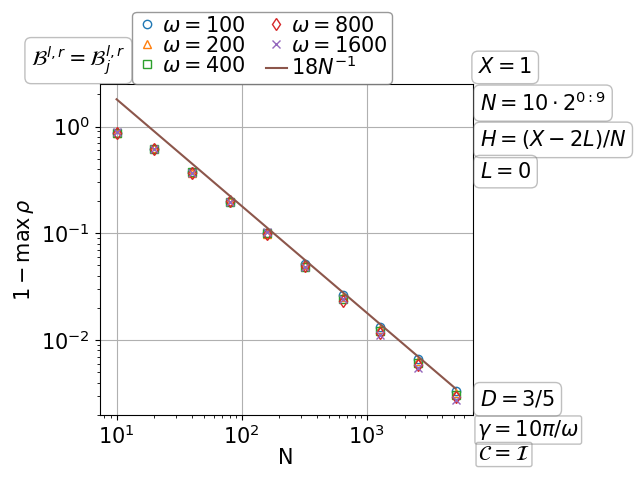}%
    \includegraphics[width=.5\textwidth,trim=5 6 0 2,clip]{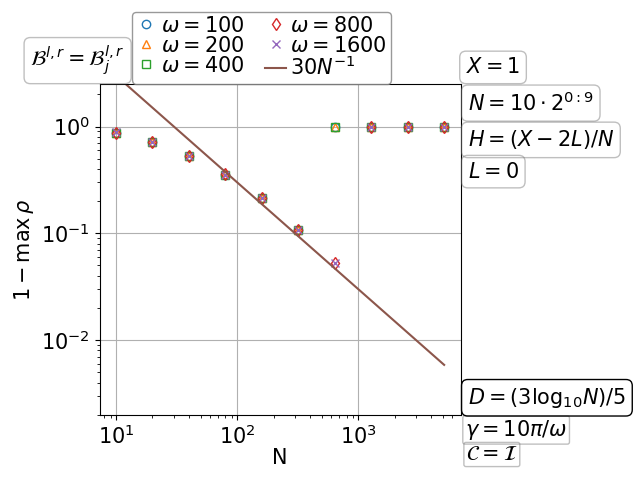}\\
    \includegraphics[width=.5\textwidth,trim=5 6 0 2,clip]{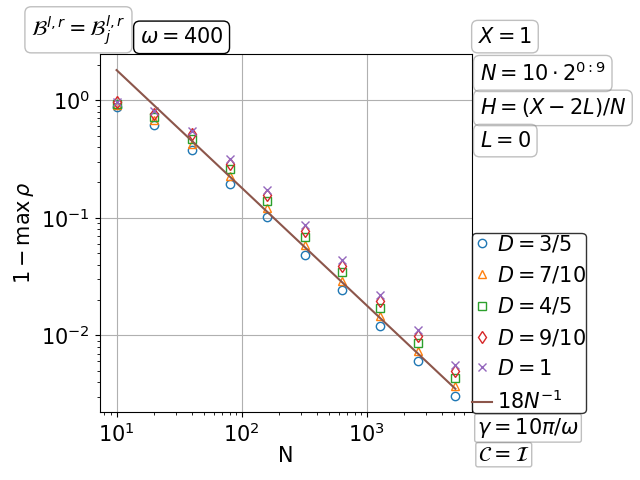}%
    \includegraphics[width=.5\textwidth,trim=5 6 0 2,clip]{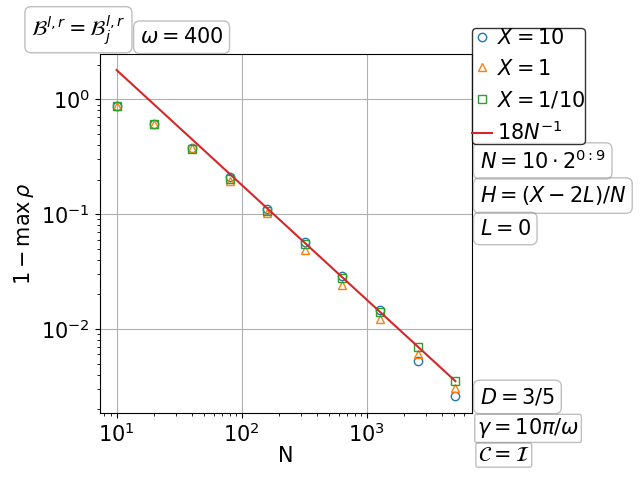}%
    \caption{Convergence of the parallel Schwarz method with PML transmission for
      the free space wave on a fixed domain with increasing number of subdomains.}
    \label{figfspp1}
  \end{figure}
  
\end{paragraph}

\begin{paragraph}{Convergence on a fixed domain with a fixed number of subdomains for increasing
    wavenumber}

  The scaling with the wavenumber $\omega$ is studied while all the other parameters are fixed. A
  limiting profile of the convergence factor $\rho$ as $\omega\to\infty$ is indicated in the top
  half of Figure~\ref{figfsppw}. The scaling of $\max_{\xi}\rho$ with $\omega$ and its dependence on
  the number of subdomains $N$, the PML width $D$ and the domain width $X$ are explored one by one in
  the bottom half of the figure. We find $\max_{\xi}\rho=O(1)<1$ which increases with $N$
  (superlinearly at high wavenumber $\omega$), decays exponentially with $D$ and depends mildly on
  $X$.

  \begin{figure}
    \centering%
    \includegraphics[width=.5\textwidth,trim=5 6 0 2,clip]{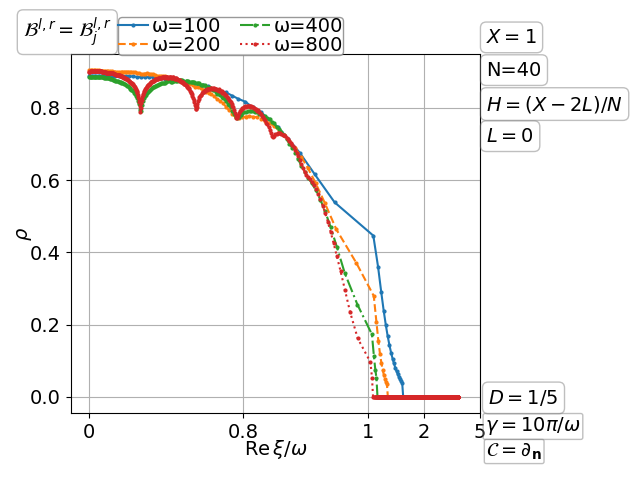}%
    \includegraphics[width=.5\textwidth,trim=5 6 0 2,clip]{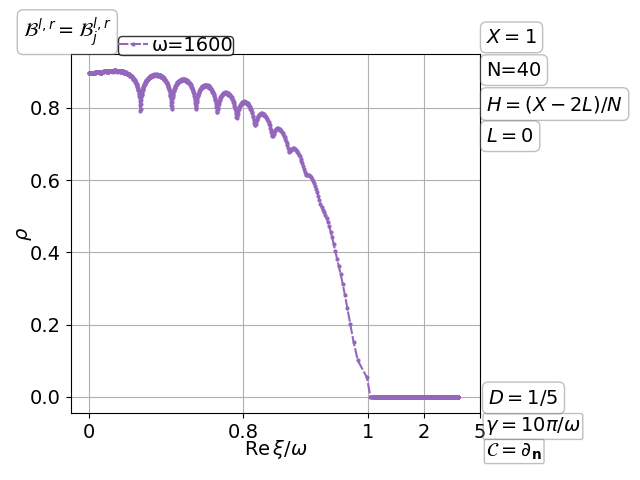}\\
    \includegraphics[width=.5\textwidth,trim=5 6 0 2,clip]{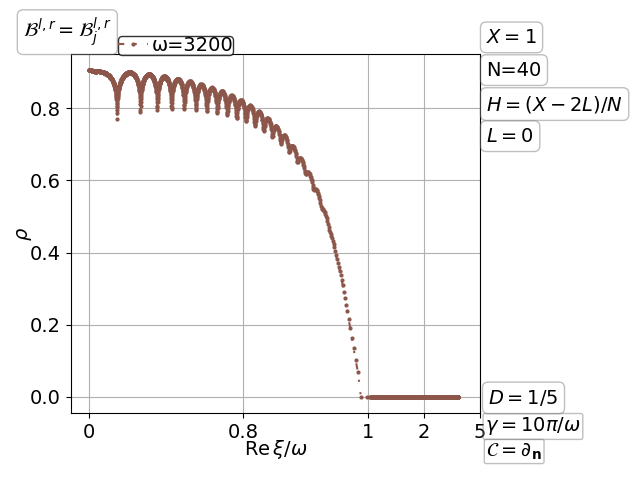}%
    \includegraphics[width=.5\textwidth,trim=5 6 0 2,clip]{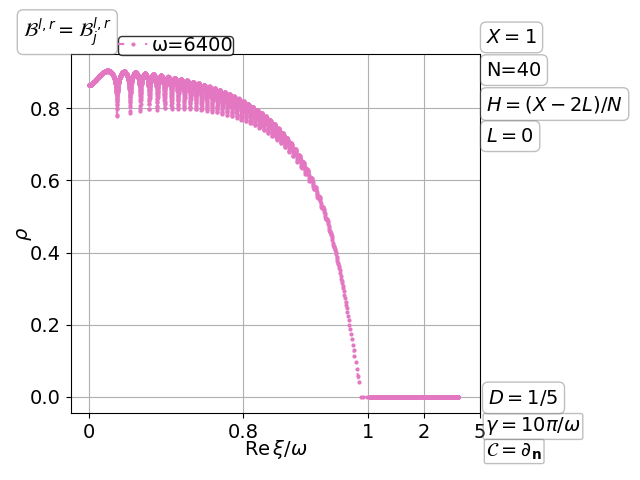}\\
    \includegraphics[width=.5\textwidth,trim=5 6 0 2,clip]{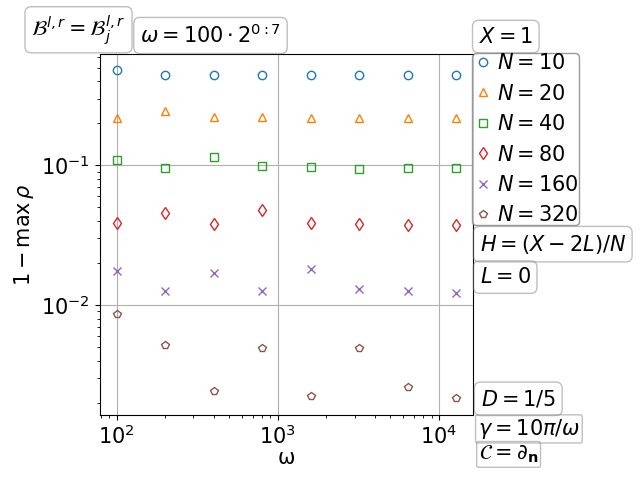}%
    \includegraphics[width=.5\textwidth,trim=5 6 0 2,clip]{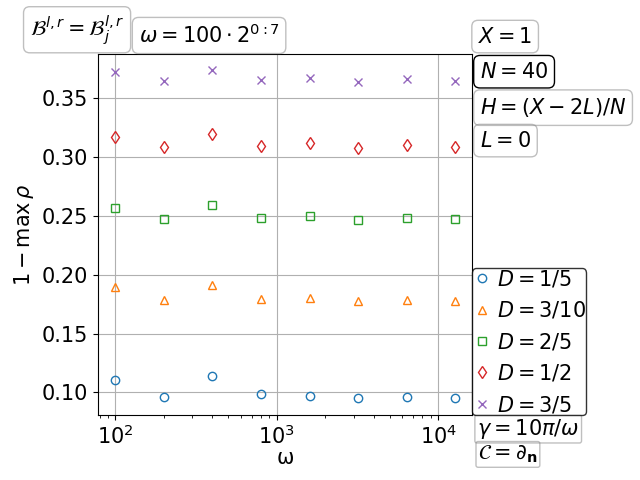}\\
    \includegraphics[width=.5\textwidth,trim=5 6 0 2,clip]{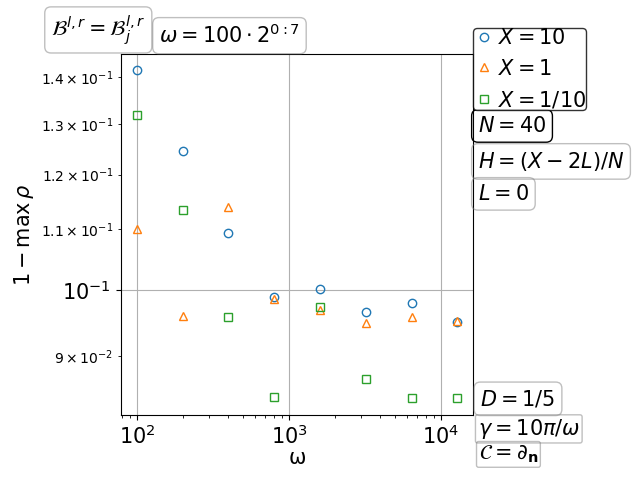}%
    \includegraphics[width=.5\textwidth,trim=5 6 0 2,clip]{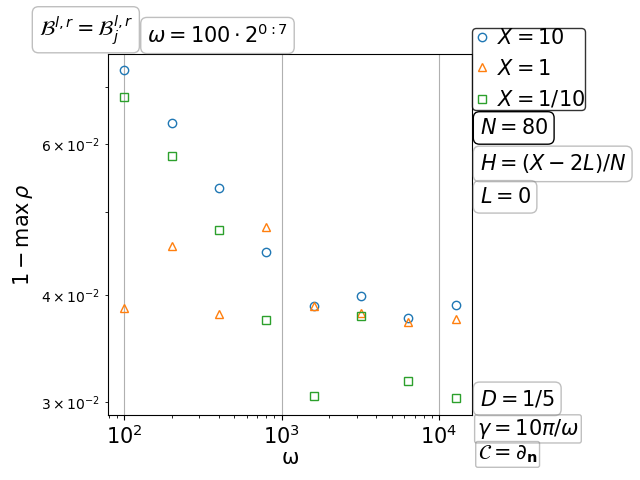}%
    \caption{Convergence of the parallel Schwarz method with PML transmission for
      the free space wave on a fixed domain with a fixed number of subdomains for increasing
      wavenumber.}
    \label{figfsppw}
  \end{figure}

\end{paragraph}

\begin{paragraph}{Convergence on a fixed domain with number of subdomains increasing with the
    wavenumber}

  In the previous scalings, we have found that the subdomain width $H$ and the wavenumber $\omega$
  have little influence on the convergence of the parallel Schwarz method with a fixed PML
  transmission, and the convergence slows down mainly due to an increasing number of subdomains
  $N$. While larger PML width $D$ accelerates the convergence and can even make the convergence
  factor $\rho$ numerical zero, the dependence on $N$ is inherent from the nilpotency index of the
  optimal parallel Schwarz iteration matrix which has the symbol entries
    $a_j=c_j=\exp(-\I\sqrt{\omega^2-\xi^2}H)$ of modulus about $1$ for the propagating modes
    ($\Re\xi<\omega$). Now in Figure~\ref{figfsppwh}, we increase also the wavenumber $\omega$ with
  $N$. It can be seen that the scaling $\max_{\xi}\rho=1-O(N^{-1})$ is the same as for fixed
  $\omega$ in Figure~\ref{figfspp1}.
  
  \begin{figure}
    \centering%
    \includegraphics[width=.56\textwidth,trim=10 10 0 6,clip]{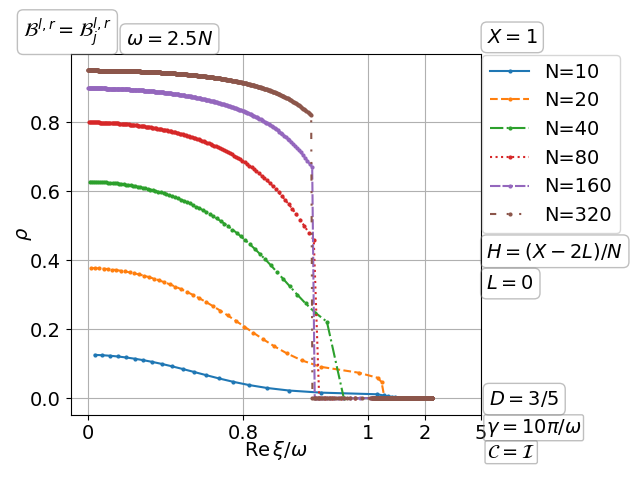}%
    \includegraphics[width=.44\textwidth,height=13em,trim=0 10 0 6,clip]{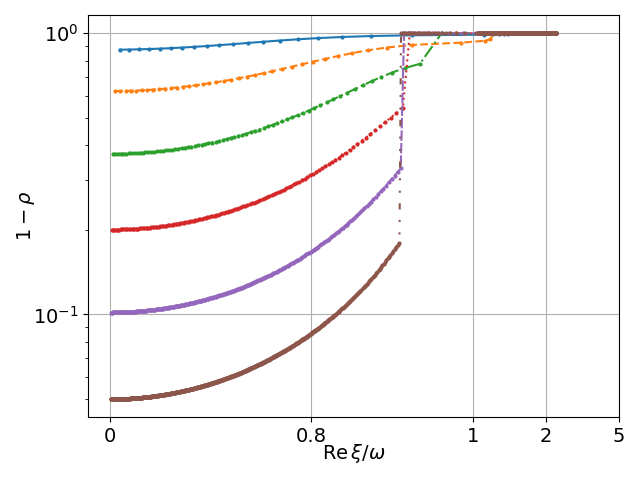}\\
    \includegraphics[width=.56\textwidth,trim=10 10 0 6,clip]{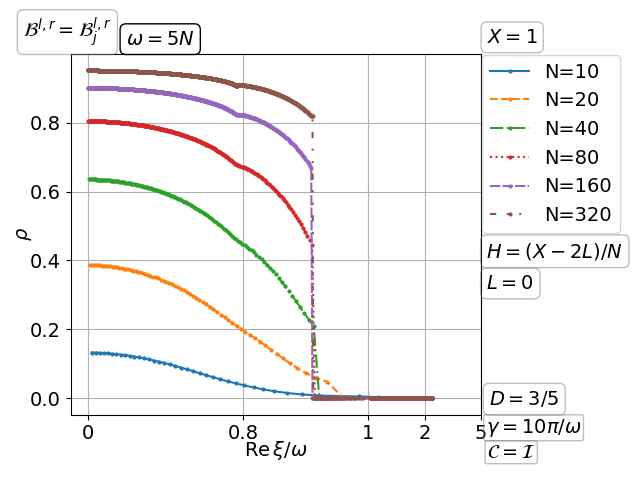}%
    \includegraphics[width=.44\textwidth,height=13em,trim=0 10 0 6,clip]{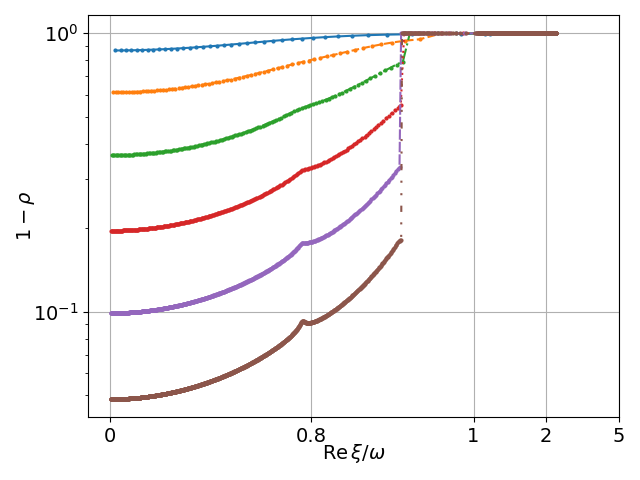}\\
    \includegraphics[width=.5\textwidth,trim=5 6 0 2,clip]{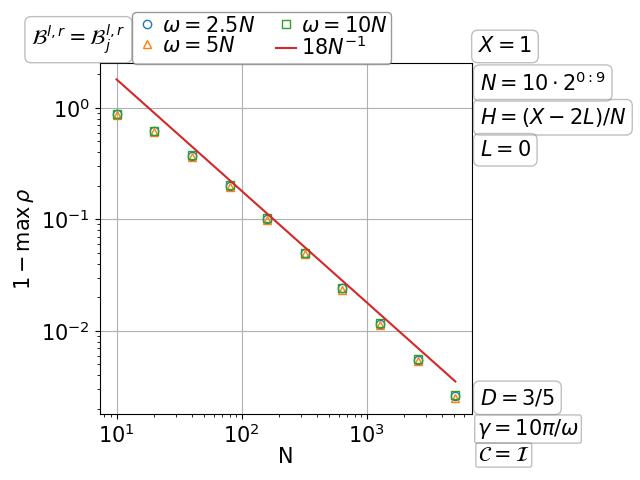}%
    \includegraphics[width=.5\textwidth,trim=5 6 0 2,clip]{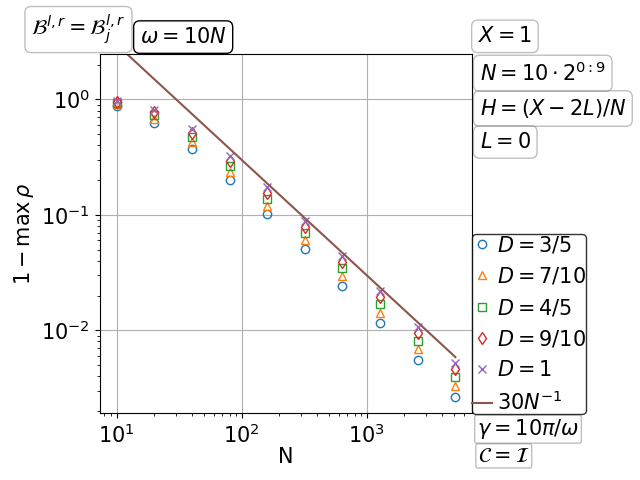}\\
    \includegraphics[width=.5\textwidth,trim=5 6 0 2,clip]{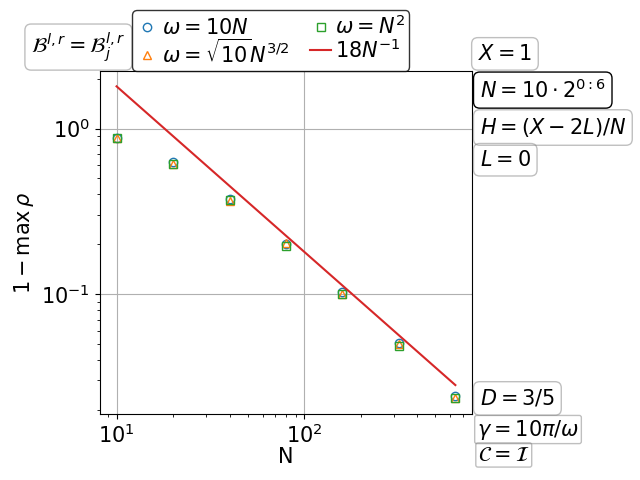}%
    \includegraphics[width=.5\textwidth,trim=5 6 0 2,clip]{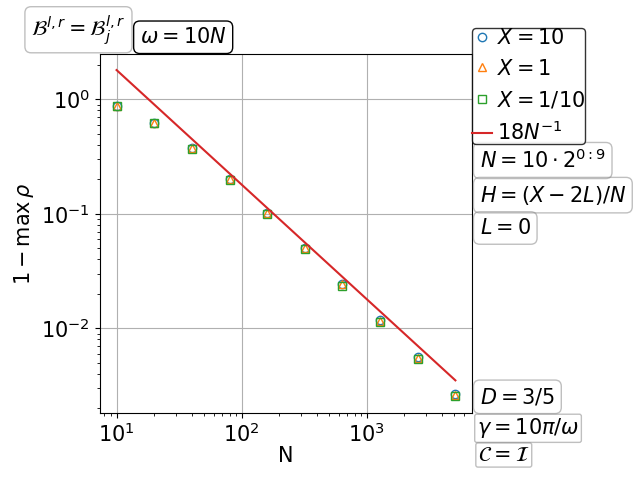}%
    \caption{Convergence of the parallel Schwarz method with PML transmission for the free space
      wave on a fixed domain with number of subdomains increasing with the wavenumber.}
    \label{figfsppwh}
  \end{figure}

\end{paragraph}

%%%%%%%%%%%%%%%%%%%%%%%%%%%%%%%%%%%%%%%%%%%%%%%%%%%%%%%%%%%%%%%%%%%%%%%%%%%

\subsection{Double sweep Schwarz method for the free space wave problem}

%%%%%%%%%%%%%%%%%%%%%%%%%%%%%%%%%%%%%%%%%%%%%%%%%%%%%%%%%%%%%%%%%%%%%%%%%%%

We have already seen for the parallel Schwarz methods that the free space wave problem is
significantly more difficult than the diffusion problem. For example, with fixed size subdomains,
the free space wave problem requires the parallel Schwarz iteration number to grow linearly with
number of subdomains $N$ even with the wavenumber $\omega$ fixed, while the diffusion problem
requires only a constant iteration number\footnote{\label{fnoteN}This is best understood for the optimal parallel Schwarz method of nilpotency index $N$ which has the nonzero symbol entries $a_j=c_j=\exp(-\sqrt{\xi^2+\eta}H)$. In the diffusion problem, $\eta>0$ so the nonzero entries are less than $1$ independent of $N$. But in the wave problem, $\eta=-\omega^2<0$ so the nonzero entries are of modulus about $1$ when $\Re\xi<\omega$.}. The Taylor of order zero transmission and the PML
transmission lead to the same scalings for the diffusion problem, while for the wave problem a
sufficient PML is necessary to guarantee convergence of the parallel Schwarz method and leads to
much faster convergence than the Taylor of order zero transmission conditions do. We also note that
the parallel Schwarz method with a fixed PML on a fixed domain has the same scalings of the
convergence rate for the wave and the diffusion problems.

Another angle of view is to compare the parallel Schwarz methods with their double sweep
counterparts, which has been presented for the diffusion problem. We have seen that the dependence
of the convergence rate on the number of subdomains can be removed by using together a fixed PML
transmission and the double sweep Schwarz iteration.

Now we are going to study the double sweep Schwarz methods for the free space wave problem. Do they
scale with fixed size subdomains, given that the parallel Schwarz methods do for the diffusion
problem but not for the wave problem? With Taylor of order zero transmission without overlap, the
parallel Schwarz method is proved to converge \cite{Despres}. Do we still expect convergence
when we change only the parallel iteration to the double sweep iteration? How do the double sweep
Schwarz methods depend on the wavenumber?  Will a fixed PML make the convergence scalable with the
number of subdomains like it does for the diffusion problem?

%%%%%%%%%%%%%%%%%%%%%%%%%%%%%%%%

\subsubsection{Double sweep Schwarz method with Taylor of order zero transmission for the free
  space wave problem}

For the free space problem, we assume that the original boundary condition is also the Taylor of
order zero condition, \ie, $\mathcal{B}^{b,t}=\mathcal{B}^{l,r}=\mathcal{B}^{l,r}_j$. There is no
theoretical convergence result for the double sweep Schwarz method with Taylor of order zero
transmission. Our experience with the parallel Schwarz method is that a sufficiently small overlap
allows convergence of both evanescent and propagating modes. We shall try different overlap width
$L$ in the following study.

\begin{paragraph}{Convergence with increasing number of fixed size subdomains}

  When the domain grows by adding fixed size subdomains along one direction and the wavenumber
  $\omega$ is fixed, we find a suitable overlap width can be $L=O(\omega^{-1})$; see the top half
  of Figure~\ref{figfsdt0n}. Although not shown in the figure, now not only too large overlap can
  cause divergence of some propgating modes but also too small overlap can cause divergence of some
  evanescent modes! From the figure, we can see that a limiting curve
  $\rho_{\infty}:=\lim_{N\to\infty}\rho$ exists for the double sweep Schwarz method, and $\rho$
  grows with $N$ towards $\rho_{\infty}$. The tendency indicates quite good convergence at large
  $N$, which is verified in the bottom half of the figure. We find $\max_{\xi}\rho=O(1)<1$ which
  deteriorates with increasing wavenumber $\omega$, shrinking overlap width $L$ and shrinking
  subdomain width $H$.

  \begin{figure}
    \centering%
    \includegraphics[width=.56\textwidth,trim=10 10 0 6,clip]{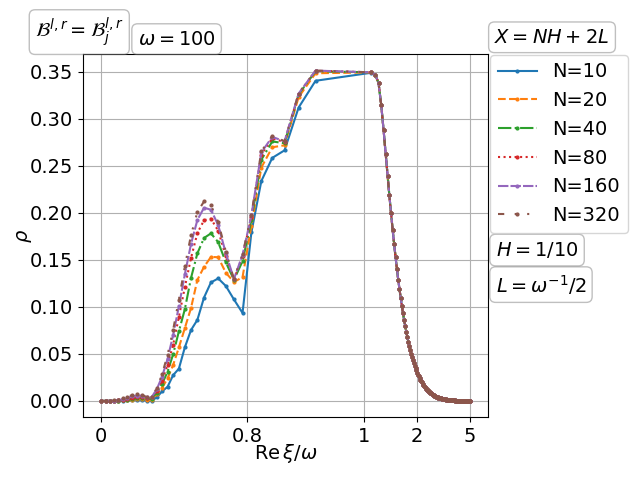}%
    \includegraphics[width=.44\textwidth,height=13em,trim=0 10 0 6,clip]{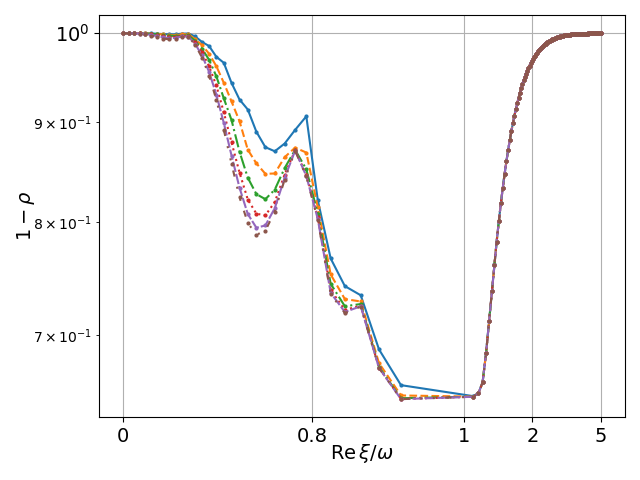}\\
    \includegraphics[width=.56\textwidth,trim=10 10 0 6,clip]{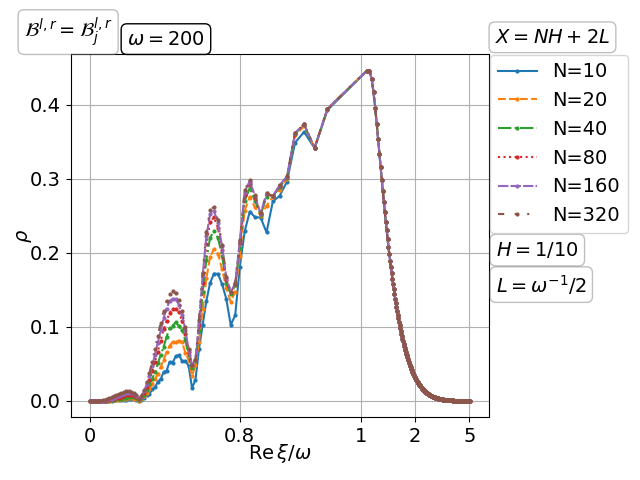}%
    \includegraphics[width=.44\textwidth,height=13em,trim=0 10 0 6,clip]{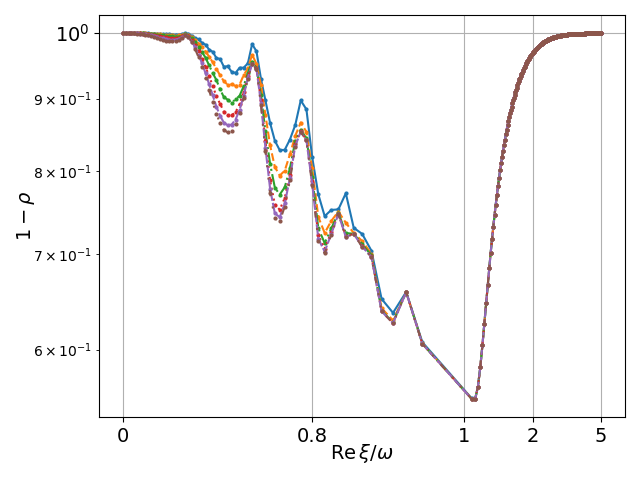}\\
    \includegraphics[width=.5\textwidth,trim=5 6 0 2,clip]{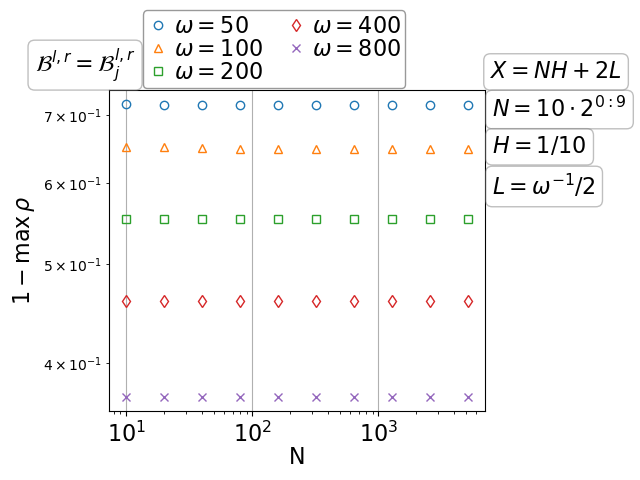}%
    \includegraphics[width=.5\textwidth,trim=5 6 0 2,clip]{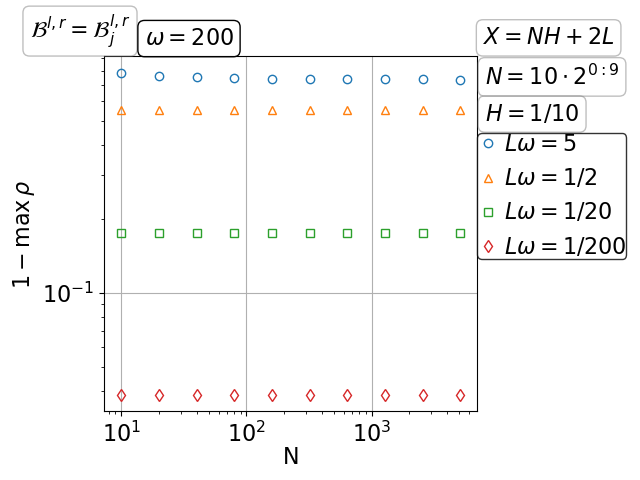}\\
    \includegraphics[width=.5\textwidth,trim=5 6 0 2,clip]{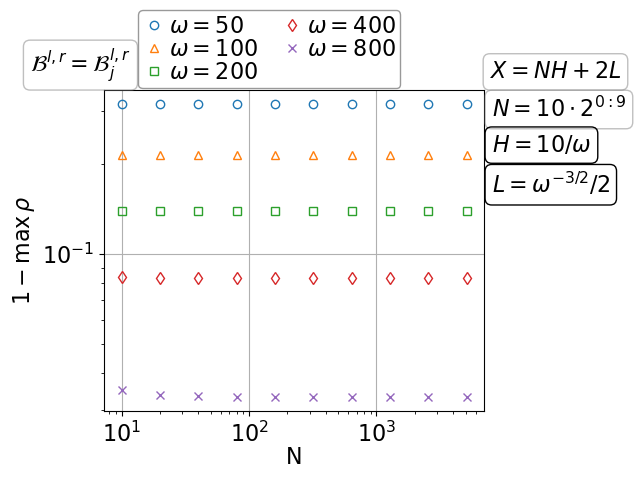}%
    \includegraphics[width=.5\textwidth,trim=5 6 0 2,clip]{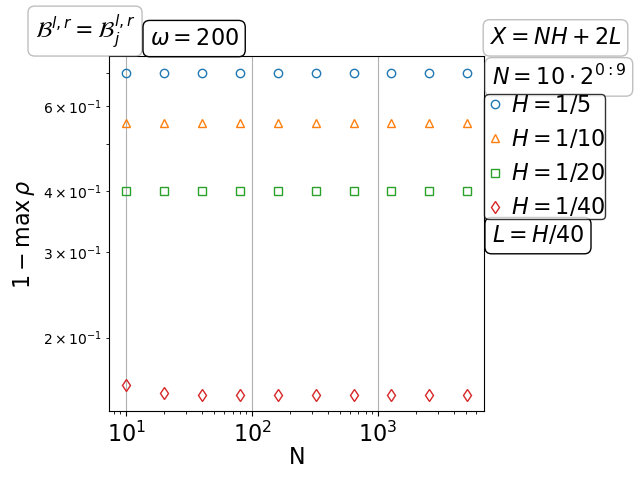}%
    \caption{Convergence of the double sweep Schwarz method with Taylor of order zero transmission
      for the free space wave with increasing number of fixed size subdomains.}
    \label{figfsdt0n}
  \end{figure}

\end{paragraph}

\begin{paragraph}{Divergence on a fixed domain with increasing number of subdomains}

  As we mentioned before, the overlap width $L$ should be small for convergence of propagating modes
  and also should be large for convergence of evanescent modes. The two opposite requirements become
  difficult to satisfy when the subdomain width $H$ is small. Figure~\ref{figfsdt01} illustrates
  this: we see that the double sweep Schwarz method can not converge for
  these examples with whatever overlap. In other words, under the scaling with increasing number of
  subdomains on a fixed domain, the double sweep Schwarz method with Taylor of order zero transmission
  eventually diverges.

    \begin{figure}
    \centering%
    \includegraphics[width=.5\textwidth,trim=5 6 0 2,clip]{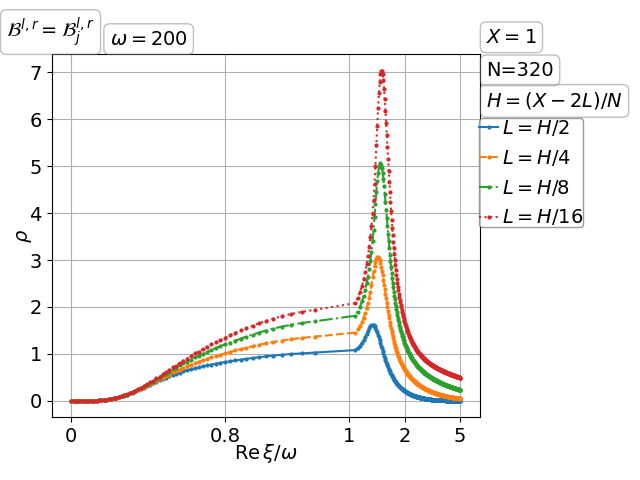}%
    \includegraphics[width=.5\textwidth,trim=5 6 0 2,clip]{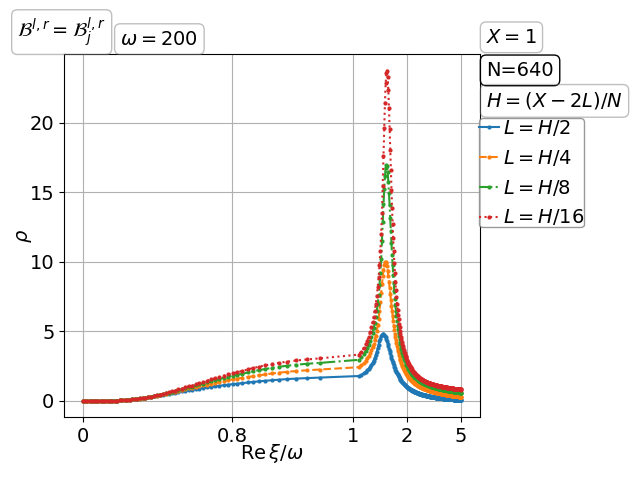}\\
    \includegraphics[width=.5\textwidth,trim=5 6 0 2,clip]{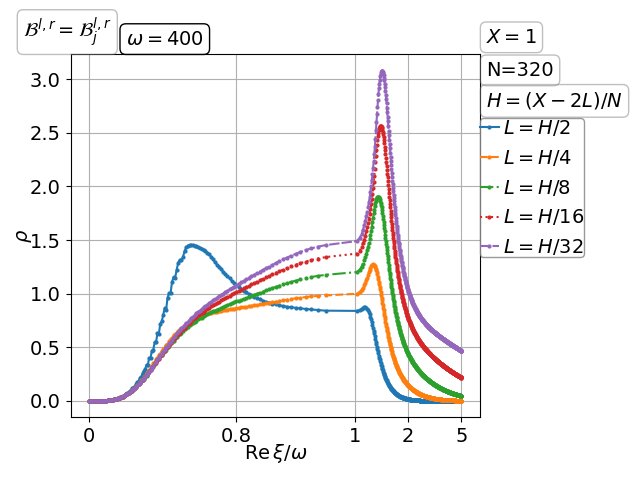}%
    \includegraphics[width=.5\textwidth,trim=5 6 0 2,clip]{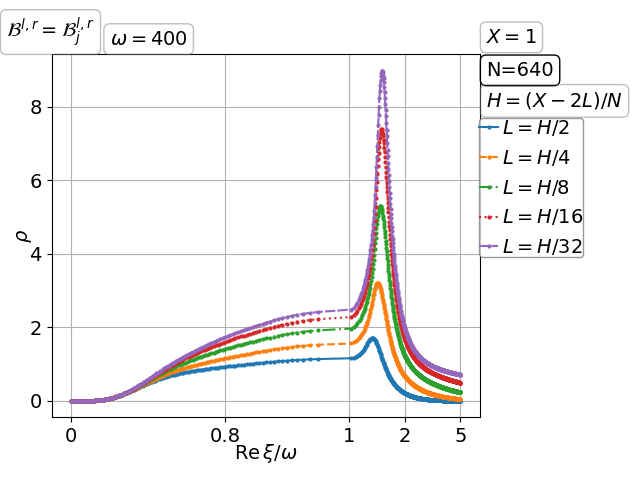}
    \caption{Divergence of the double sweep Schwarz method with Taylor of order zero transmission
      for the free space wave on a fixed domain with increasing number of subdomains.}
    \label{figfsdt01}
  \end{figure}
  
\end{paragraph}

\begin{paragraph}{Convergence on a fixed domain with a fixed number of subdomains for increasing
    wavenumber}

  As we learned from the previous paragraphs, when the subdomain width is sufficiently large, the
  double sweep Schwarz method with Taylor of order zero transmission converges with a suitable overlap
  width. In the convergent regime, we can still study the scaling with the wavenumber $\omega$; see
  Figure~\ref{figfsdt0w}. The top half illustrates the changing graph of $\rho=\rho(\xi)$ with
  $\omega$, which shows a limiting profile away from a neighborhood of $\Re\xi=\omega$, and its
  maximum attained near $\Re\xi=\omega$ increases with $\omega$. The bottom half of the figure shows
  the estimate $\max_{\xi}\rho=1-O(\omega^{-9/20})$ for large $\omega$, with the hidden factor
  independent of number of subdomains $N$ and the domain width $X$.

  \begin{figure}
    \centering%
    \includegraphics[width=.56\textwidth,trim=10 10 0 6,clip]{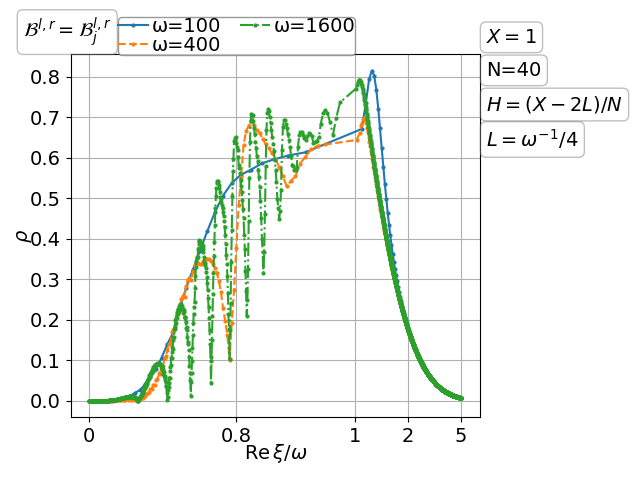}%
    \includegraphics[width=.44\textwidth,height=13em,trim=0 10 0 6,clip]{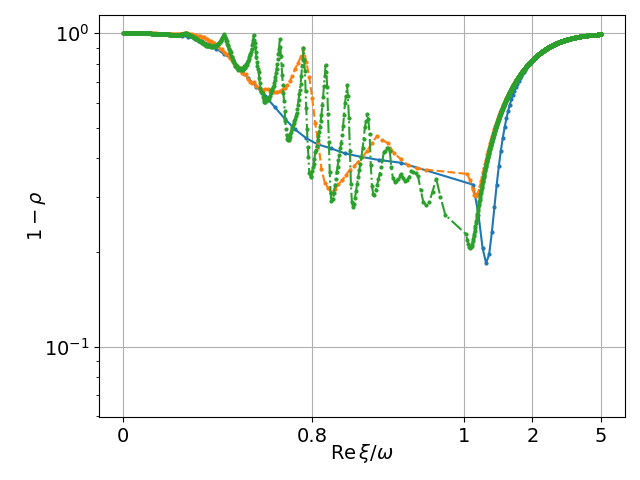}\\
    \includegraphics[width=.5\textwidth,trim=5 6 0 2,clip]{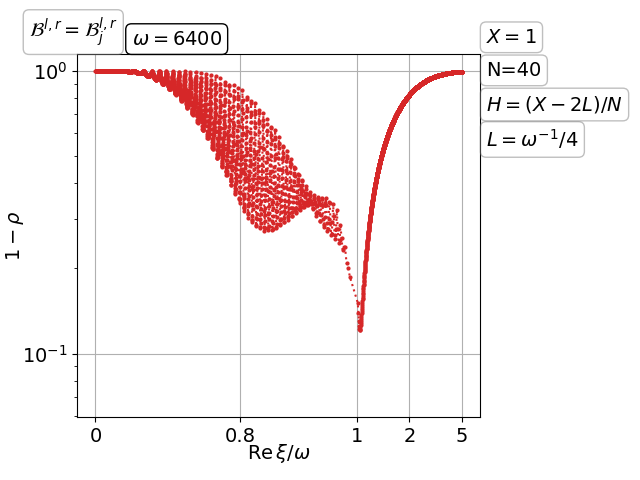}%
    \includegraphics[width=.5\textwidth,trim=5 6 0 2,clip]{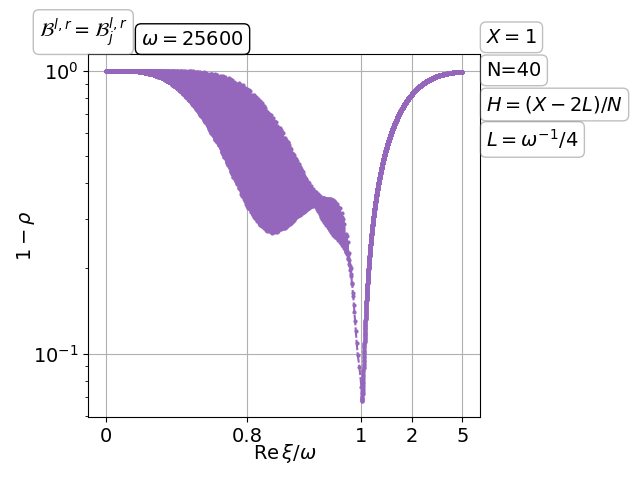}\\
    \includegraphics[width=.5\textwidth,trim=5 6 0 2,clip]{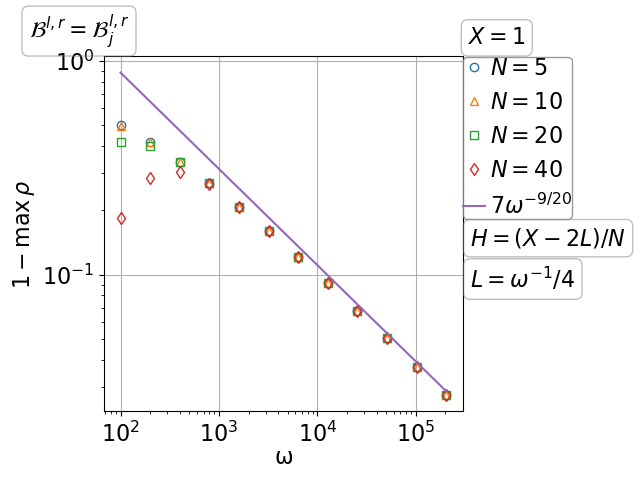}%
    \includegraphics[width=.5\textwidth,trim=5 6 0 2,clip]{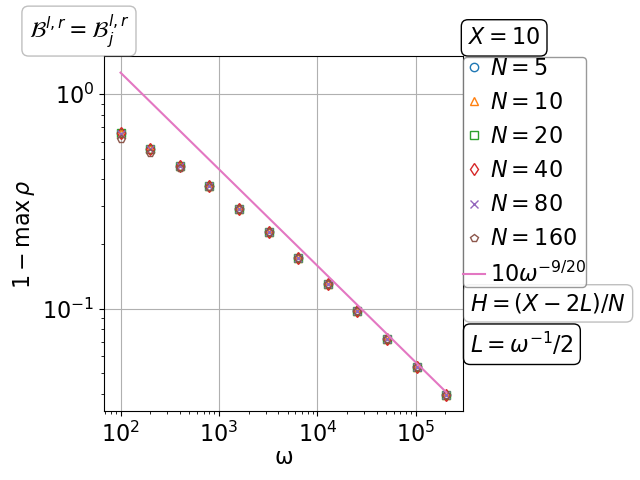}\\
    \includegraphics[width=.5\textwidth,trim=5 6 0 2,clip]{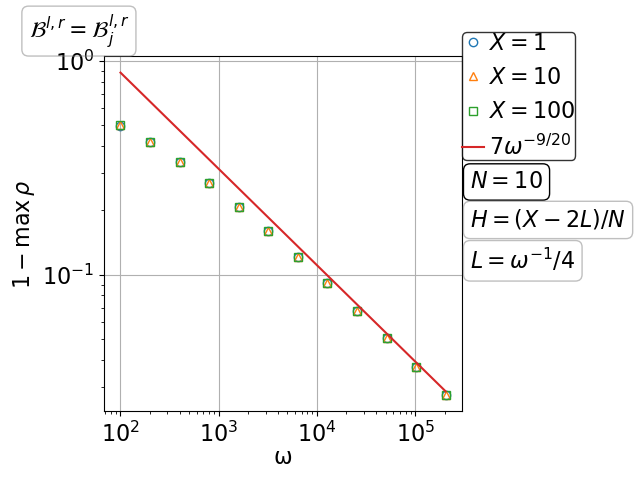}%
    \includegraphics[width=.5\textwidth,trim=5 6 0 2,clip]{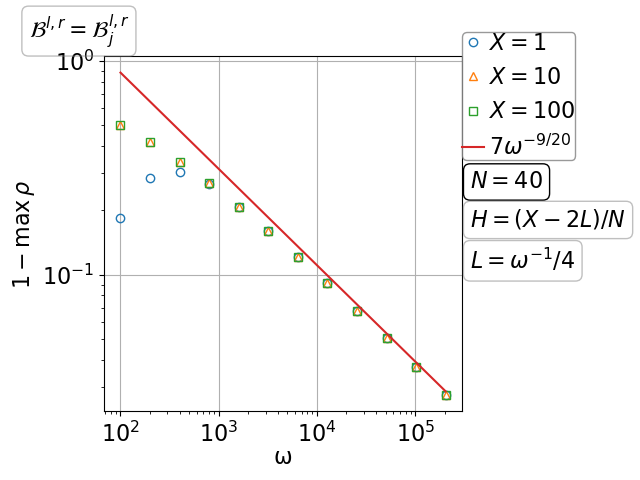}%
    \caption{Convergence of the double sweep Schwarz method with Taylor of order zero transmission
      for the free space wave on a fixed domain with a fixed number of subdomains for increasing
      wavenumber.}
    \label{figfsdt0w}
  \end{figure}

\end{paragraph}

\begin{paragraph}{Divergence on a fixed domain with number of subdomains increasing with the
    wavenumber}

  As we mentioned before, convergence of the double sweep Schwarz method with Taylor of order zero
  transmission requires a sufficienly large subdomain width $H$. So, on a fixed domain under the
  scaling $N\to\infty$, the method eventually must diverge. That the wavenumber increases here with
  $N$ does not change the situation. As can be seen in Figure~\ref{figfsdt0wh}, divergence of some
  propagating modes is persistent at large $N$ for any overlap width $L$ including $L=0$.
  
  \begin{figure}
    \centering%
    \includegraphics[width=.5\textwidth,trim=5 6 0 2,clip]{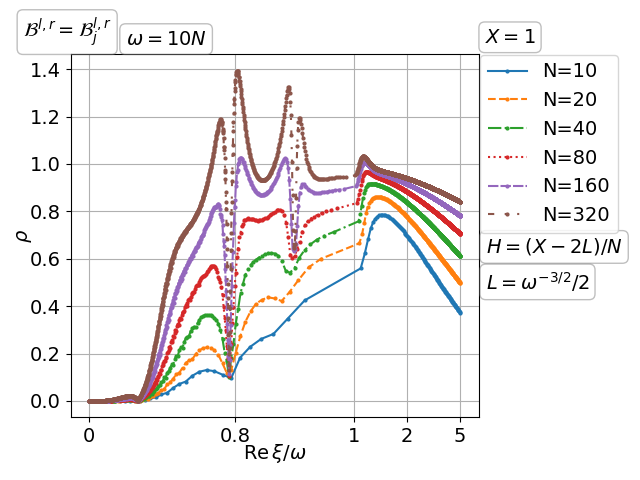}%
    \includegraphics[width=.5\textwidth,trim=5 6 0 2,clip]{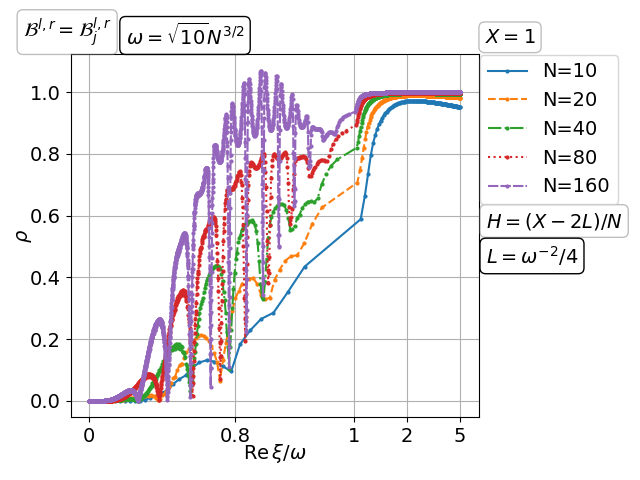}\\
    \includegraphics[width=.5\textwidth,trim=5 6 0 2,clip]{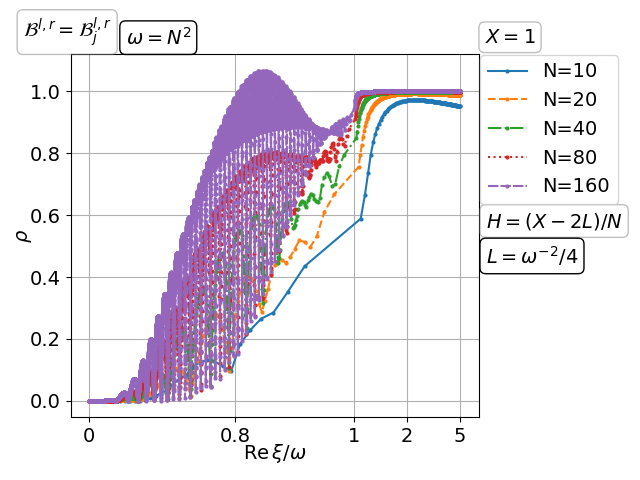}%
    \includegraphics[width=.5\textwidth,trim=5 6 0 2,clip]{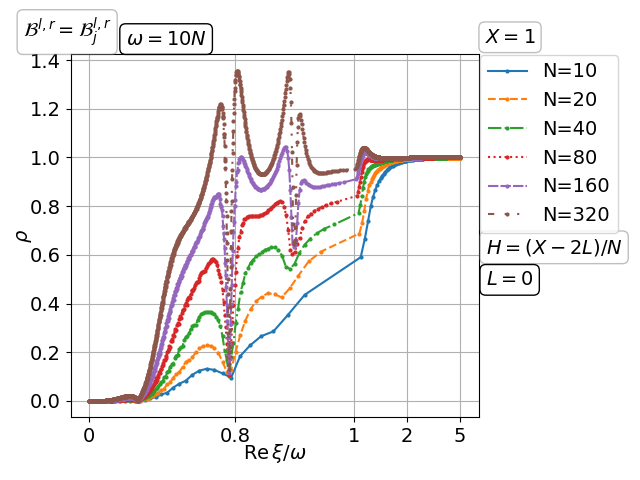}
    \caption{Divergence of the double sweep Schwarz method with Taylor of order zero transmission
      for the free space wave on a fixed domain with number of subdomains increasing with the
      wavenumber.}
    \label{figfsdt0wh}
  \end{figure}

\end{paragraph}

%%%%%%%%%%%%%%%%%%%%%%%%%%%%%%%%

\subsubsection{Double sweep Schwarz method with PML transmission for the free space wave problem}

For the free space problem, the PML is used all around the domain. The top and bottom PML is
treated as part of the extended domain. For analysis purposes, the left and right PML for the subdomains
and the domain goes into the Dirichlet-to-Neumann operator \R{eqsw} in the boundary condition listed
in Table~\ref{tabB}. The double sweep Schwarz method with PML transmission for the free space wave
problem is the most noticeable case in the last decade. Impressive numerical experiments have been
shown in the literature; see \eg~\cn{EY2}, \cn{Poulson}, \cn{Stolk}, \cn{StolkImproved},
\cn{Chen13a}, \cn{Chen13b}, \cn{ZD}, \cn{ZepedaNested}, \cn{xiang2019double}. Recently, it was
applied to an inverse problem \cite{eslaminia}. An interesting attempt is to use double sweeps in
small groups of connected subdomains for accelerating the corresponding parallel Schwarz method
\cite{vion2018}. A question that has not been answered in theory is how many discrete layers of PML
are needed for scalable iterations. \cn{Chen13a} have an estimate at the continuous level indicating
that a PML width $D=O(N\log\omega)$ is sufficient. It was claimed based on numerical experiments
that a logarithmic growth of discrete layers of PML could be sufficient for increasing wavenumber
$\omega$ but fixed oscillations $\omega H$ in the subdomain width $H$; see \eg~\cn{Poulson}. In
this subsection, we shall investigate the question at the continuous level. Note that due to the
small values of $\rho$ the vertical axes in this subsection will be $\rho$ or $\max_{\xi}\rho$
instead of $1-\rho$ or $1-\max_{\xi}\rho$!

\begin{paragraph}{Convergence with increasing number of fixed size subdomains}

  We first try with a fixed PML. As we add more fixed size subdomains
  along one direction to the domain, we see a rapid growth of $\rho$
  in a neighborhood of $\xi=0$ from the first row of
  Figure~\ref{figfsdpn}, which eventually leads to divergence. Then,
  we turn to use a logarithmic growth of the PML width $D$ with the
  number of subdomains $N$ in the second row of the figure. We see
  $\rho$ is decreasing rapidly with $N$ over all the range of
  $\xi$. In the third row and the first column, we try with a smaller
  constant factor for the logarithmic growth of $D$, and we see
  divergence again. The next subplot shows the exponential (or faster)
  decay of $\max_{\xi}\rho$ with increasing $D$. We test the scaling
  of $\max_{\xi}\rho$ with different wavenumber $\omega$ and subdomain
  width $H$ in the last row of the figure. Given $D=(\log_{10}N)/10$,
  the dependence of $\max_{\xi}\rho=O(N^{-2})$ on $\omega$ and $H$ is
  negligible. Actually, larger $\omega$ and smaller $H$ can even give
  smaller $\max_{\xi}\rho$.
  
  \begin{figure}
    \centering%
    \includegraphics[width=.5\textwidth,trim=5 6 0 2,clip]{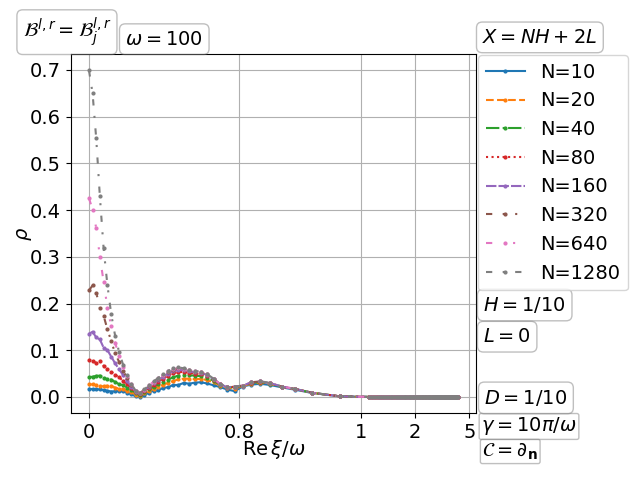}%
    \includegraphics[width=.5\textwidth,trim=5 6 0 2,clip]{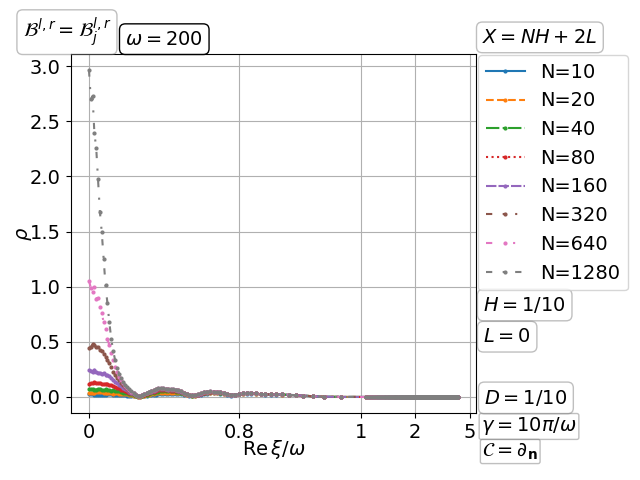}\\
    \includegraphics[width=.5\textwidth,trim=5 6 0 2,clip]{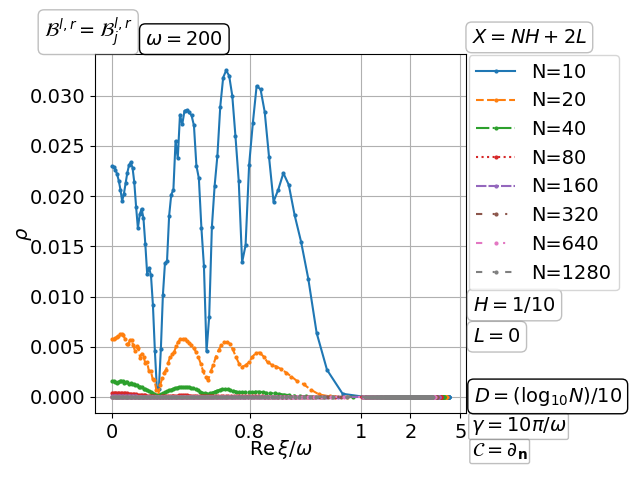}%
    \includegraphics[width=.5\textwidth,trim=5 6 0 2,clip]{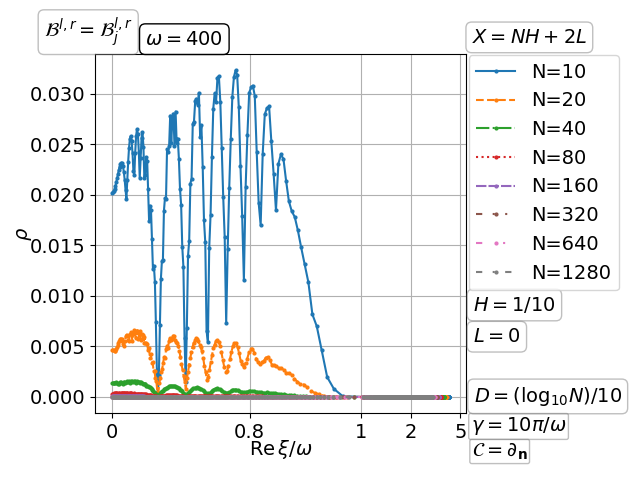}\\
    \includegraphics[width=.5\textwidth,trim=5 6 0 2,clip]{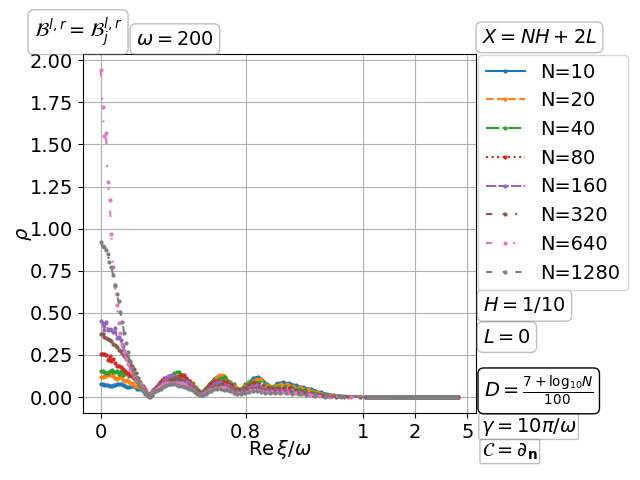}%
    \includegraphics[width=.5\textwidth,trim=5 6 0 2,clip]{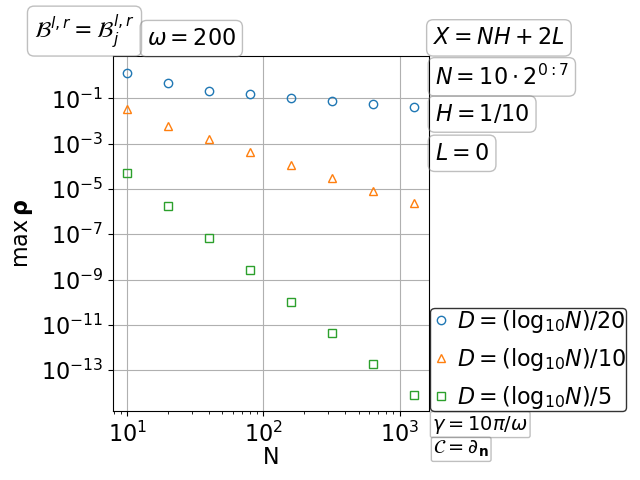}\\
    \includegraphics[width=.5\textwidth,trim=5 6 0 2,clip]{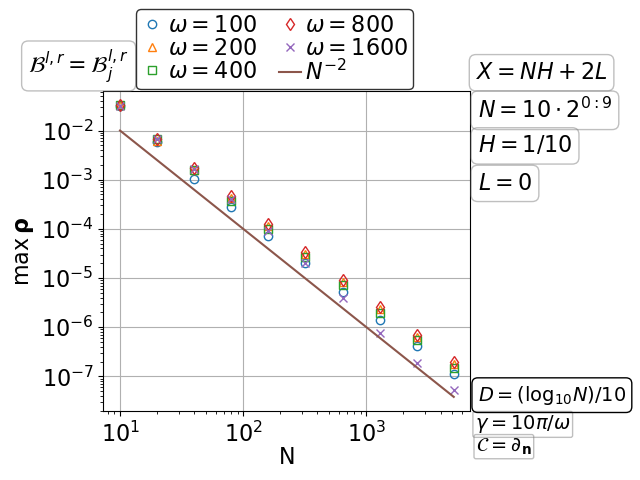}%
    \includegraphics[width=.5\textwidth,trim=5 6 0 2,clip]{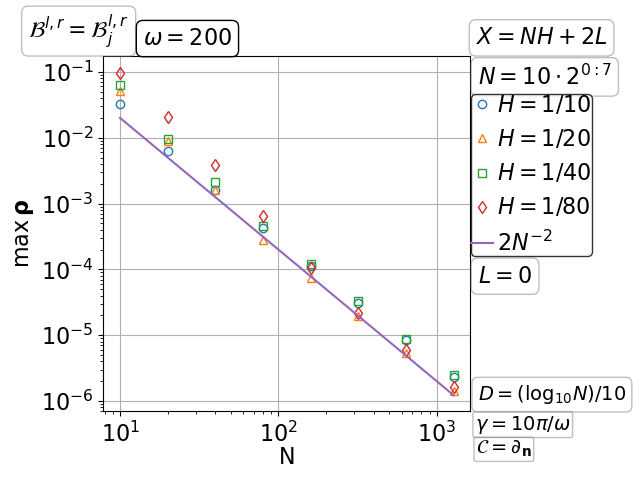}%
    \caption{Convergence and divergence of the double sweep Schwarz method with PML transmission
      for the free space wave with increasing number of fixed size subdomains. (Scaling of
      $\max_{\xi}\rho$ not $1-\max_{\xi}\rho$!)}
    \label{figfsdpn}
  \end{figure}

\end{paragraph}

\begin{paragraph}{Convergence on a fixed domain with increasing number of subdomains}

  Based on the previous paragraph, we choose a sufficiently large PML width $D$ to grow
  logarithmically with the number of subdomains $N$, now on a fixed domain. From the top half
  of Figure~\ref{figfsdp1}, we can see that the convergence factor $\rho$ decreases rapidly with
  $N$. Increasing the wavenumber $\omega$ only introduces more oscillations and changes little the
  envelope profile. We next show the constant PML width $D$ can still not work; see the first
  subplot of the third row. In the following subplots, the tendency of $\max_{\xi}\rho\to0$ as
  $N\to\infty$ is illustrated for different values of $\omega$, $D$ and the domain width $X$. Given
  $D=(\log_{10}N)/10$, the speed of $\max_{\xi}\rho\to0$ is faster than quadratic for all the listed
  values of $\omega$ and $X$. But the speed strongly depends on $D$ because $\max_{\xi}\rho$ decays
  exponentially as $D$ increases.

  \begin{figure}
    \centering%
    \includegraphics[width=.5\textwidth,trim=5 6 0 2,clip]{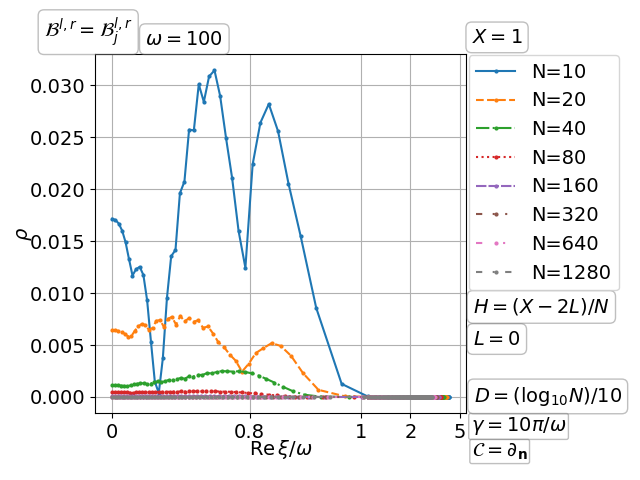}%
    \includegraphics[width=.5\textwidth,trim=5 6 0 2,clip]{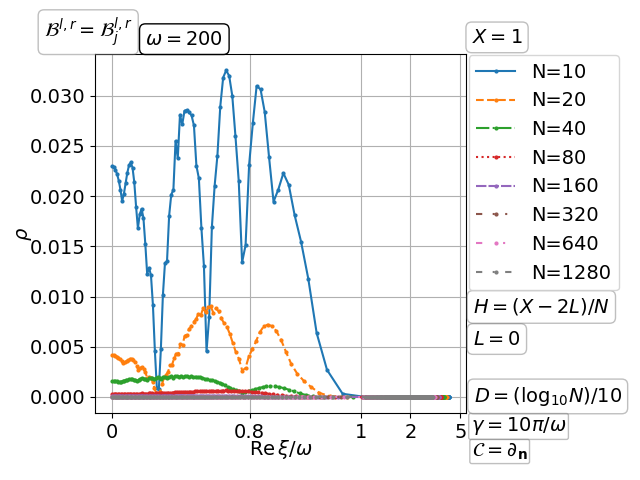}\\
    \includegraphics[width=.5\textwidth,trim=5 6 0 2,clip]{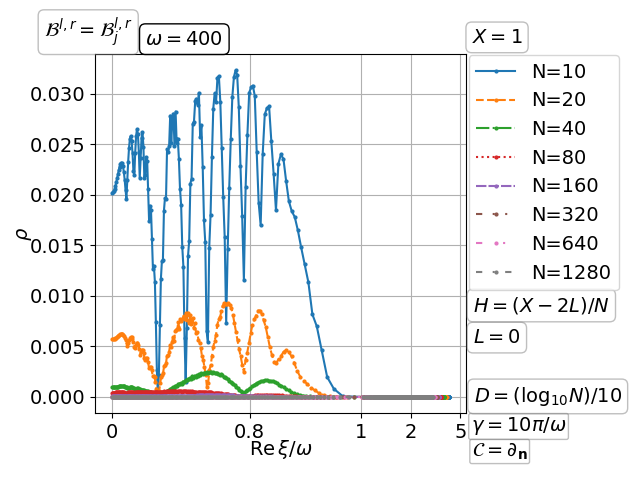}%
    \includegraphics[width=.5\textwidth,trim=5 6 0 2,clip]{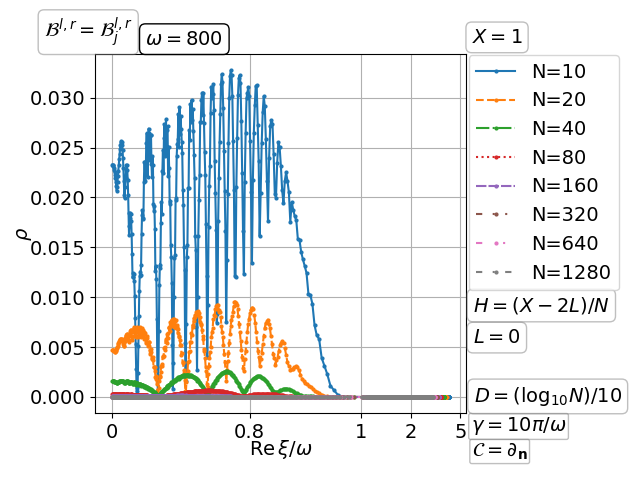}\\
    \includegraphics[width=.5\textwidth,trim=5 6 0 2,clip]{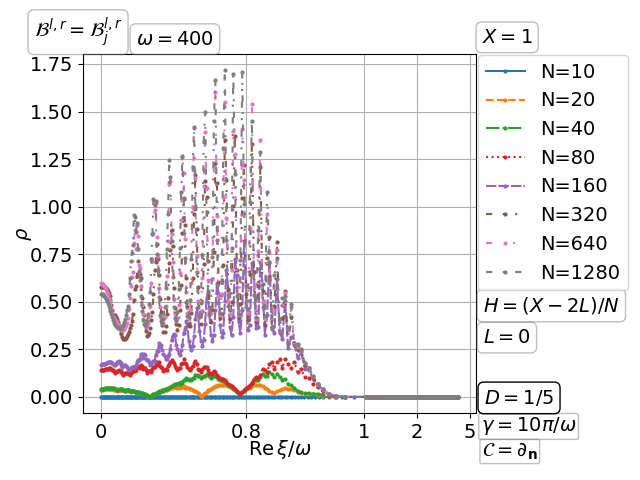}%
    \includegraphics[width=.5\textwidth,trim=5 6 0 2,clip]{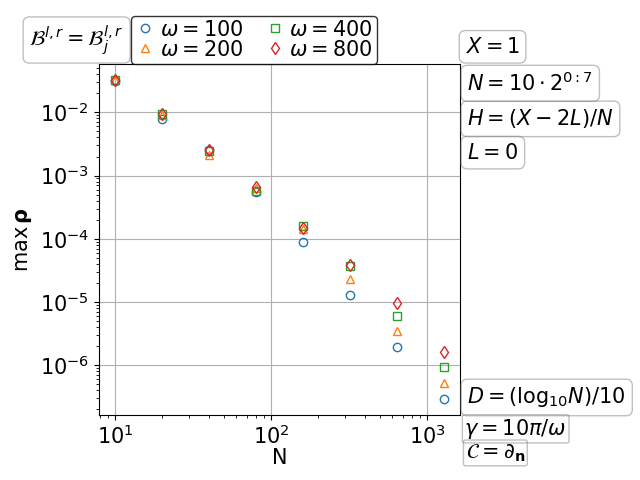}\\
    \includegraphics[width=.5\textwidth,trim=5 6 0 2,clip]{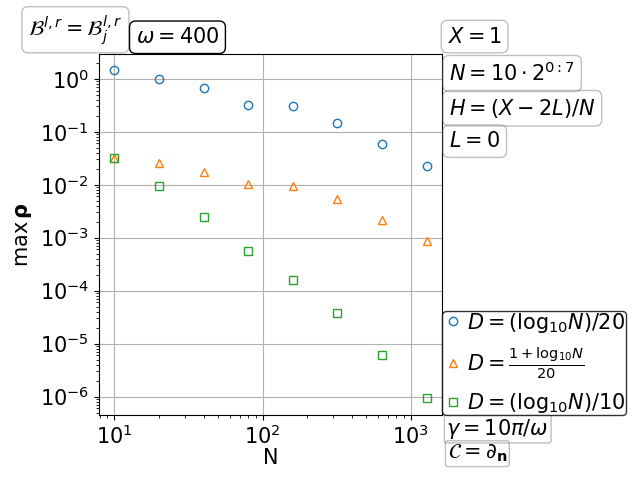}%
    \includegraphics[width=.5\textwidth,trim=5 6 0 2,clip]{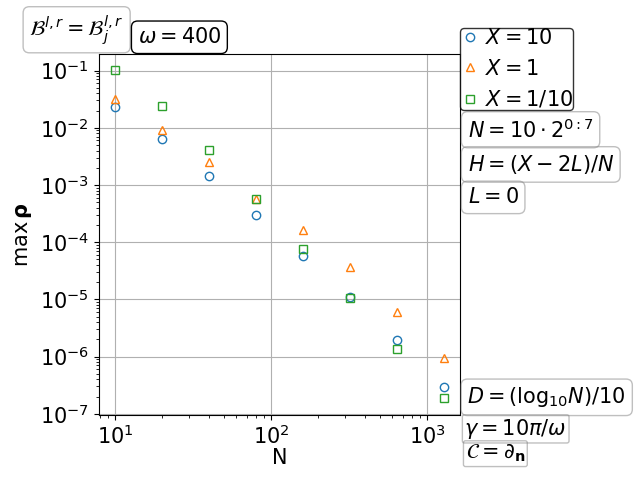}
    \caption{Convergence and divergence of the double sweep Schwarz method with PML transmission
      for the free space wave on a fixed domain with increasing number of subdomains. (Scaling of
      $\max_{\xi}\rho$ not $1-\max_{\xi}\rho$!)}
    \label{figfsdp1}
  \end{figure}
  
\end{paragraph}

\begin{paragraph}{Convergence on a fixed domain with a fixed number of subdomains for increasing
    wavenumber}

  Now we study the scaling with the wavenumber $\omega$ alone. The PML width $D$ is taken as
  constant with respect to $\omega$. In the top half of Figure~\ref{figfsdpw}, we see that the graph
  of the convergence facor $\rho=\rho(\xi)$ tends to a limiting profile as $\omega\to\infty$. So,
  asymptotically $\max_{\xi}\rho=O(1)$ is independent of $\omega$, as verified in the bottom half
  of Figure~\ref{figfsdpw}. At the same time, we see $\max_{\xi}\rho$ grows with the number of
  subdomains $N$, decays exponentially with $D$, and grows with shrinking domain width
  $X$. Moreover, we find also that $D$ needs to be sufficiently big for convergence at large
  $N$ and small $X$; see particularly the last subplot where divergence appears at $N=80$ and
  $X=1/10$.
  
  \begin{figure}
    \centering%
    \includegraphics[width=.5\textwidth,trim=5 6 0 2,clip]{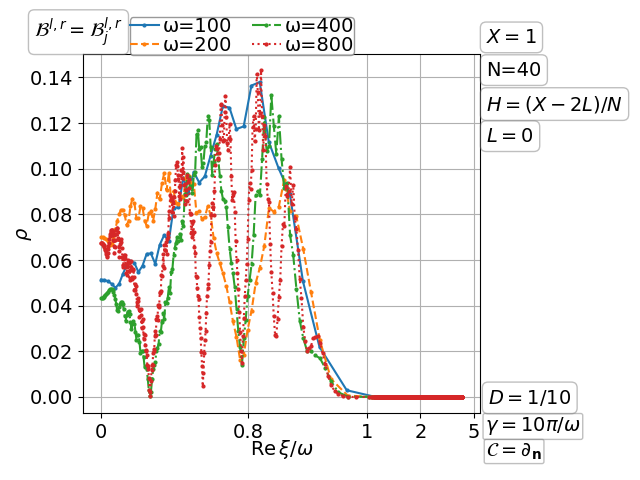}%
    \includegraphics[width=.5\textwidth,trim=5 6 0 2,clip]{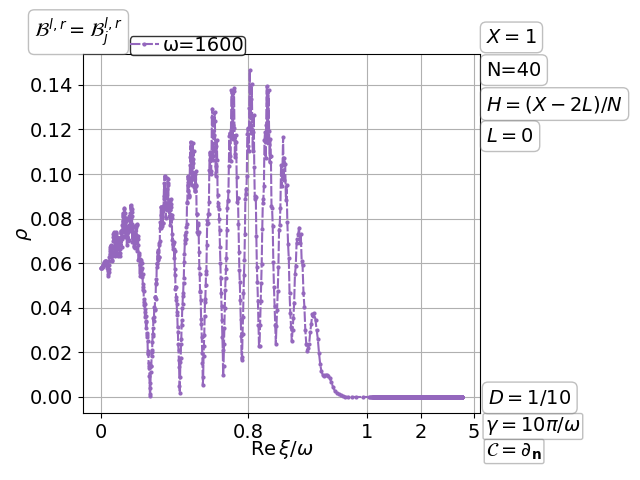}\\
    \includegraphics[width=.5\textwidth,trim=5 6 0 2,clip]{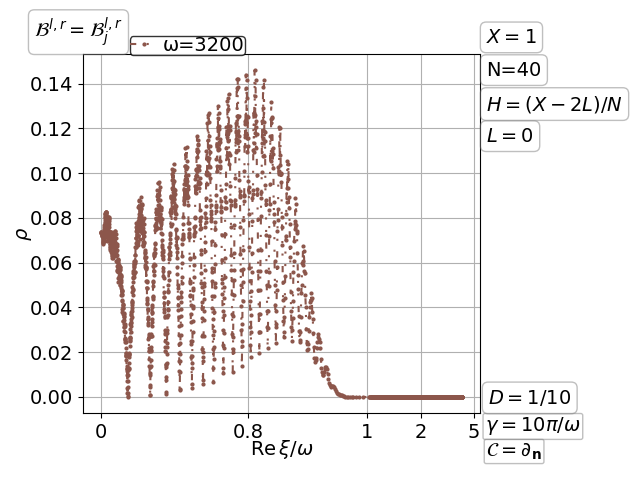}%
    \includegraphics[width=.5\textwidth,trim=5 6 0 2,clip]{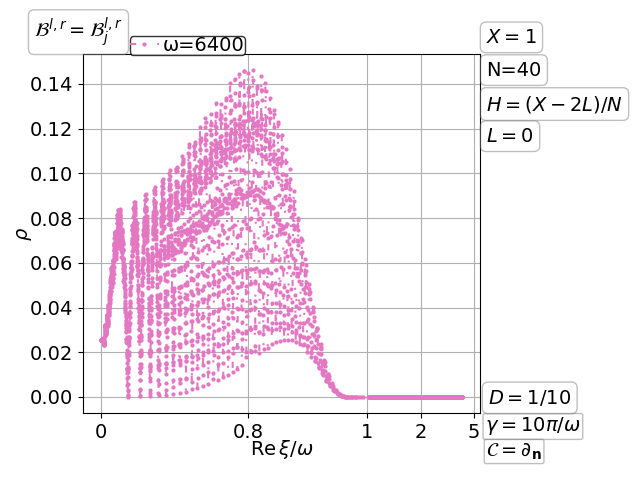}\\
    \includegraphics[width=.5\textwidth,trim=5 6 0 2,clip]{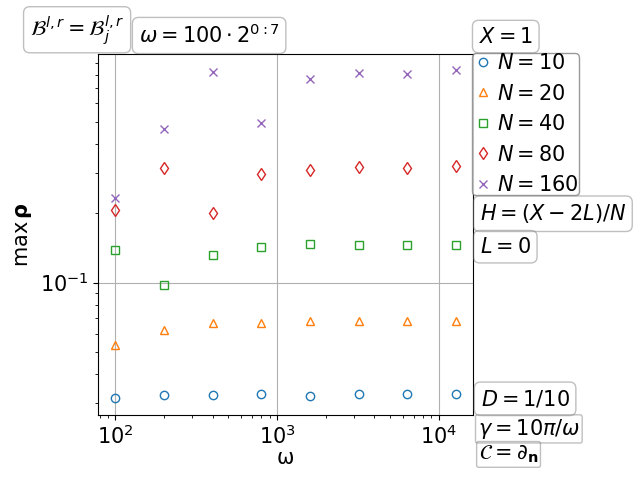}%
    \includegraphics[width=.5\textwidth,trim=5 6 0 2,clip]{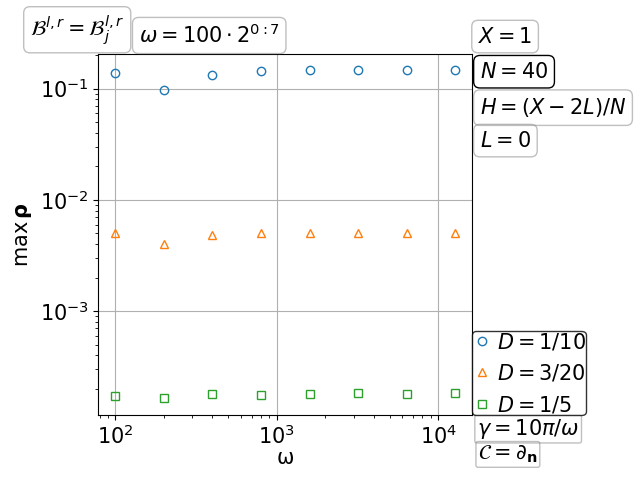}\\
    \includegraphics[width=.5\textwidth,trim=5 6 0 2,clip]{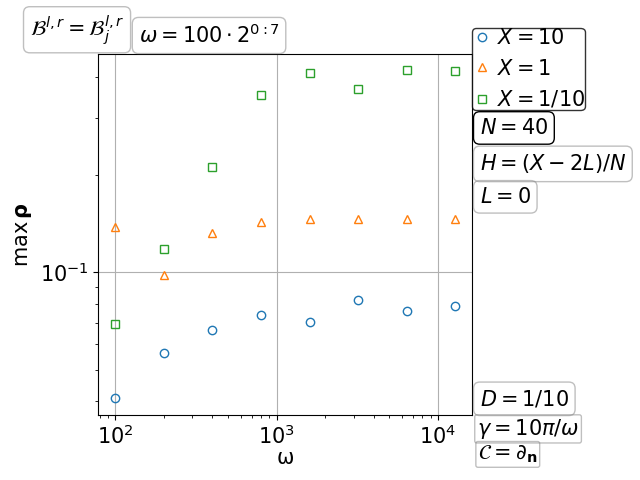}%
    \includegraphics[width=.5\textwidth,trim=5 6 0 2,clip]{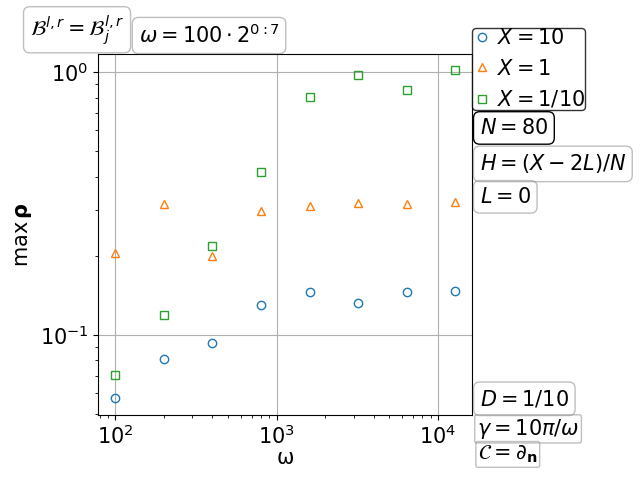}%
    \caption{Convergence and divergence of the double sweep Schwarz method with PML transmission
      for the free space wave on a fixed domain with a fixed number of subdomains for increasing
      wavenumber. (Scaling of $\max_{\xi}\rho$ not $1-\max_{\xi}\rho$!)}
    \label{figfsdpw}
  \end{figure}

\end{paragraph}

\begin{paragraph}{Convergence on a fixed domain with number of subdomains increasing with the
    wavenumber}

  This scaling can be viewed as a combination of the previous two scalings. As learned before, a
  logarithmic growth of the PML width $D$ with the number of subdomains $N$ is necessary for
  convergence, now verified also for this scaling in the first row of Figure~\ref{figfsdpwh}. Then,
  in the second row we see that a larger wavenumber $\omega$ causes more osciallations of the convergence
  factor $\rho=\rho(\xi)$ with little change of the maximum of $\rho$. Given $D=(\log_{10}N)/10$,
  the scaling of $\max_{\xi}\rho$ is about $O(N^{-2})$; see the bottom half of the figure. We can
  find that $\max_{\xi}\rho$ has a mild dependence on $\omega$ and $X$ but an exponential decay with
  $D$.  In particular, in the third row and the last column, we see that different $D$ gives a 
  different power law decay of $\max_{\xi}\rho$ with $N$.
  
  \begin{figure}
    \centering%
    \includegraphics[width=.5\textwidth,trim=5 6 0 2,clip]{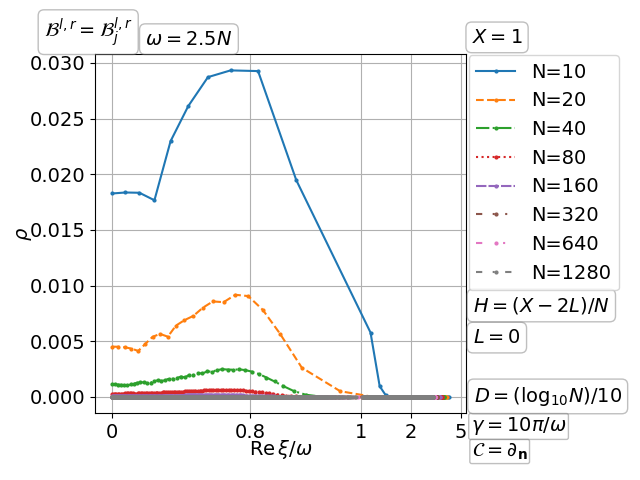}%
    \includegraphics[width=.5\textwidth,trim=5 6 0 2,clip]{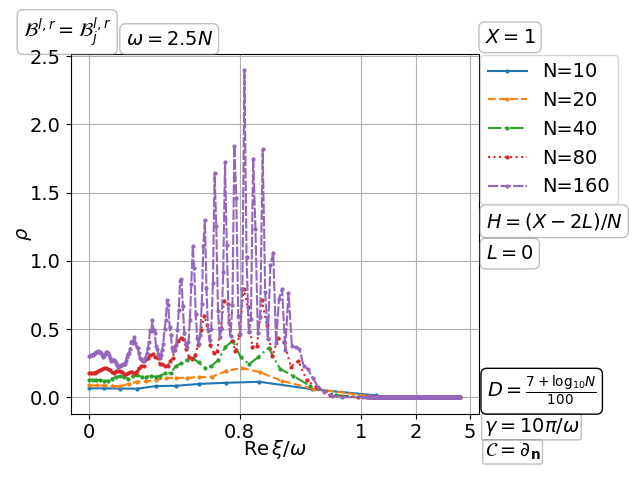}\\
    \includegraphics[width=.5\textwidth,trim=5 6 0 2,clip]{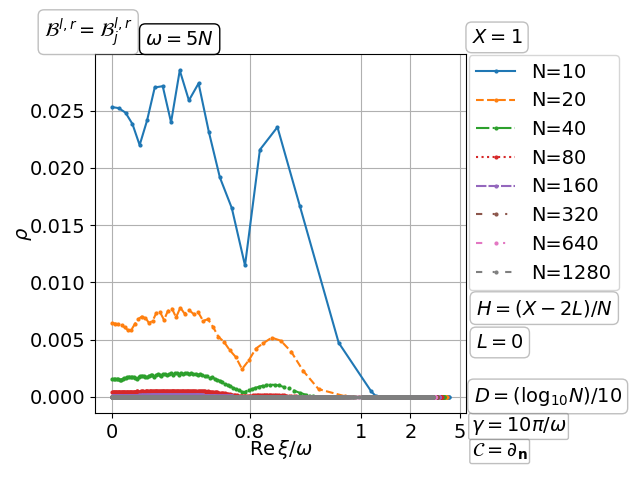}%
    \includegraphics[width=.5\textwidth,trim=5 6 0 2,clip]{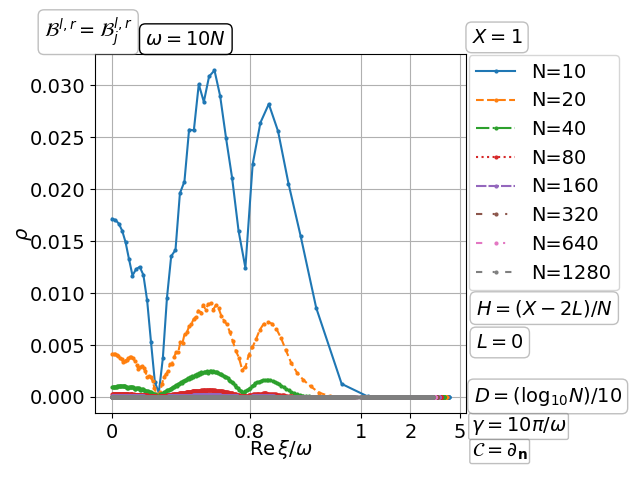}\\
    \includegraphics[width=.5\textwidth,trim=5 6 0 2,clip]{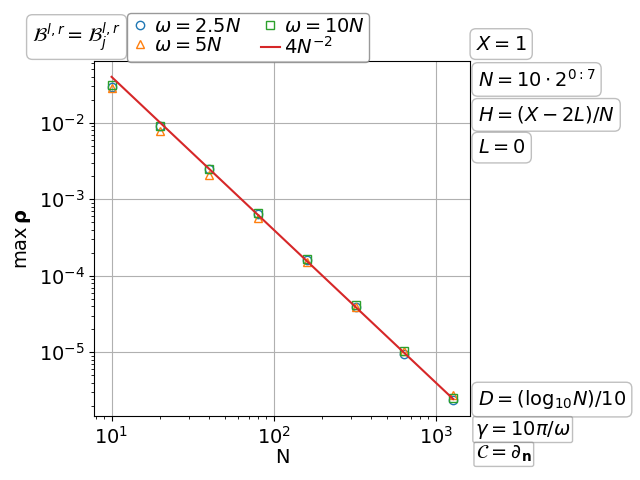}%
    \includegraphics[width=.5\textwidth,trim=5 6 0 2,clip]{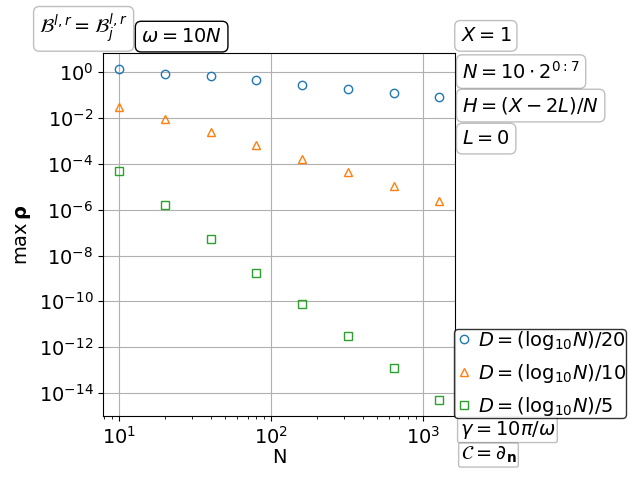}\\
    \includegraphics[width=.5\textwidth,trim=5 6 0 2,clip]{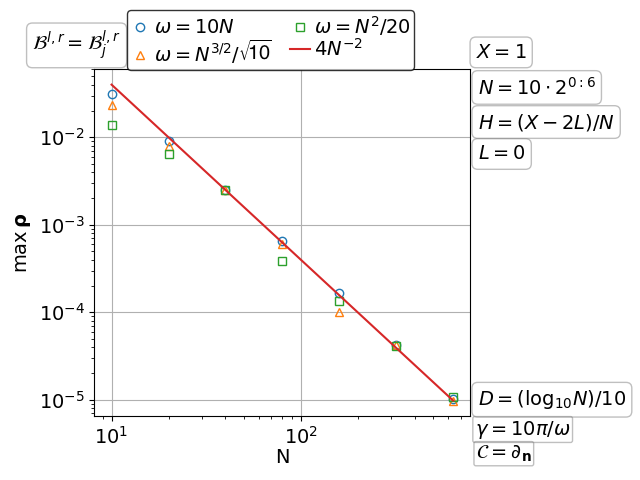}%
    \includegraphics[width=.5\textwidth,trim=5 6 0 2,clip]{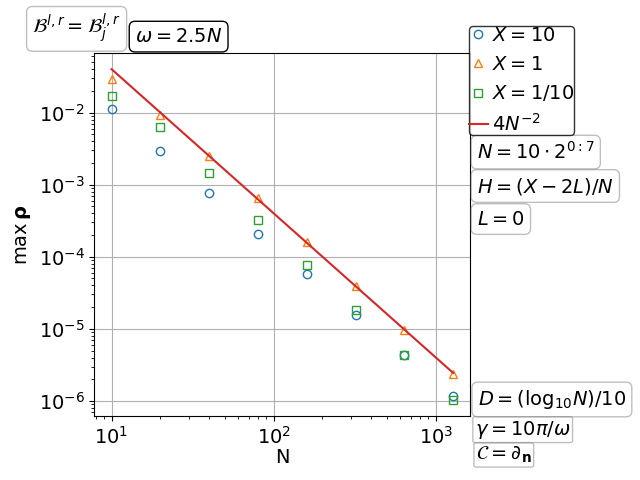}%
    \caption{Convergence and divergence of the double sweep Schwarz method with PML transmission
      for the free space wave on a fixed domain with number of subdomains increasing with the
      wavenumber. (Scaling of $\max_{\xi}\rho$ not $1-\max_{\xi}\rho$!)}
    \label{figfsdpwh}
  \end{figure}

\end{paragraph}

\subsection{Parallel Schwarz methods for the layered medium wave problem}

%%%%%%%%%%%%%%%%%%%%%%%%%%%%%%%%%%%%%%%%%%%%%%%%%%

This case differs from the free space wave problem (see section~\ref{secpsfs}) only in $\eta$. We
assume that the inhomogeneity is bounded, the medium is layer-wise constant in a box, and the
exterior medium is constant. Specifically, $\eta=-\omega^2/v^2$ with the wave velocity given as one
of the following:
\begin{itemize}
\item[vel1] or vel1($v_1$), one inclusion in the full space: $v=v_1$ when
  $(x,y)\in (\frac{1}{3}X,\frac{2}{3}X)\times(0,1)$, and $v=1$ on
  $\mathbb{R}^2-(\frac{1}{3}X,\frac{2}{3}X)\times(0,1)$;
\item[vel2] or vel2($v_1$), two inclusions in the full space: $v=v_1$ when
  $(x,y)\in (\frac{1}{6}X,\frac{1}{3}X)\times(0,1)$ or
  $(x,y)\in (\frac{2}{3}X,\frac{5}{6}X)\times(0,1)$, and $v=1$ otherwise in $\mathbb{R}^2$;
\item[vell] or vel($l$,$v_1$), $l$ inclusions in the full space: $v=v_1$ when
  $(x,y)\in ((j-1)W,jW)\times(0,1)$ for $j=2,4,..,2l$, $v=1$ otherwise in $\mathbb{R}^2$ and
  $W=\frac{X}{2l+1}$.
\end{itemize}
So, the medium interfaces are parallel to the subdomain interfaces. In fact, it is more favorable
for convergence to choose the domain decomposition along the other direction such that the subdomain
interfaces are perpendicular to the medium interfaces. But in practical applications, it is very
possible to have the situation that we choose to study here. For example, if the top and bottom
boundary conditions are Neumann or Dirichlet, a decomposition along the other direction will be hard
for convergence. Moreover, a velocity model may contain horizontal, vertical and curved interfaces
along different directions, which makes reflections traveling back and forth between subdomains
unavoidable.

Numerical experiments of the double sweep Schwarz method with a matrix Schur complement transmission
based on the constant-medium extension have shown the difficulty of convergence in layered media
\cite{gander2019class}. An analysis and approximation of the variable-medium Dirichlet-to-Neumann
operator is carried out by \cn{layer2}, and more recently they proposed a matrix optimization
approach to the infinite element method \cite{hohage2021learned,PreussThesis}. A boundary integral
transmission condition is proposed by \cn{layer1}. The idea of \cn{heikkola} consists in extending
each layer (used as subdomain) to a slab in the physical domain, which can be thought as using a
physical PML, and approximately solving the extended subdomain problem (with zero source in the
extension) by a fast direct solver. The impact of reflected waves in the direction of decomposition
was early recognized by \cn{Schadle} for the PML transmission, see also \cn{ZepedaNested}.

Our goal in this subsection is to explore the convergence factor $\rho$ based on the Fourier
analysis, as we did in the previous part of this review. A main difference in this case is that the
interface-to-interface operators $a_j$, $b_j$, $c_j$, $d_j$ need to be solved from subdomain layered
medium problems. In this case, Taylor of order zero transmission with overlap can diverge for some
medium distribution, and without overlap it hardly gives good convergence, albeit it does converge
without overlap when $\mathcal{B}^r_j=\mathcal{B}_{j+1}^l$ as guaranteed by \cn{Despres}. The
situation can be improved by using relaxation and nonlocal transmission conditions
\cite{collino2020exponentially}. For simplicity, we will focus on the PML transmission which has
been numerically tested for layered media problems \cite{Vion,leng2019additive,lengdiag}.

% \subsubsection{Parallel Schwarz method by the PML transmission for the layered medium wave
% problem}

\begin{paragraph}{Convergence with increasing number of fixed size layers as subdomains}

  We use medium layers as subdomains. So, the subdomain interfaces are aligned with the medium
  interfaces. It is then a question which wave velocity to take in the PML equation along an
  interface. In Figure~\ref{figlmppn} (except the last subplot), the same PML Dirichlet-to-Neumann
  is used for $\mathcal{B}_j^{r}$ and $\mathcal{B}_{j+1}^l$, and in particular, the domain averaged
  velocity $\bar{v}=\frac{1}{|\Omega|}\int_{\Omega}v$ is used in all the PML equations. The choice
  of $\mathcal{B}_j^{r}=\mathcal{B}_{j+1}^{l}$ here seems important for convergence, as emphasized
  in general convergence proofs of parallel optimized Schwarz methods without overlap
  \cite{lions1990schwarz,collino2020exponentially}. For example, if we take the wave velocity in the
  right neighborhood of $\{X_j^r\}\times(0,Y)$ for $\mathcal{B}_j^r$ but the wave velocity in the
  left neighborhood of $\{X_{j+1}^l\}\times(0,Y)$ for $\mathcal{B}_{j+1}^l$, we will get divergence
  no matter how large or strong a PML is used, see the last subplot of Figure~\ref{figlmppn}. In the
  top half of Figure~\ref{figlmppn}, the graphs of $\rho=\rho(\xi)$ tell us roughly that lower
  frequency modes converge slower, higher contrast in the wave velocity leads to slower convergence
  for modest number of subdomains $N$, and increasing PML width $D$ has a limited positive impact on
  convergence of low frequency modes. In the following three subplots, the scaling of
  $1-\max_{\xi}\rho$ with $N$ and its dependence on $\omega$, $D$ and contrast in the wave velocity
  $v_1$ are illustrated. Specifically, $1-\max_{\xi}\rho$ decreases with $N$ almost independently of
  $\omega$ for an initial stage and then stays at a constant which comes later for larger $\omega$.
  
\begin{figure}
  \centering%
  \includegraphics[width=.5\textwidth,trim=5 6 0 2,clip]{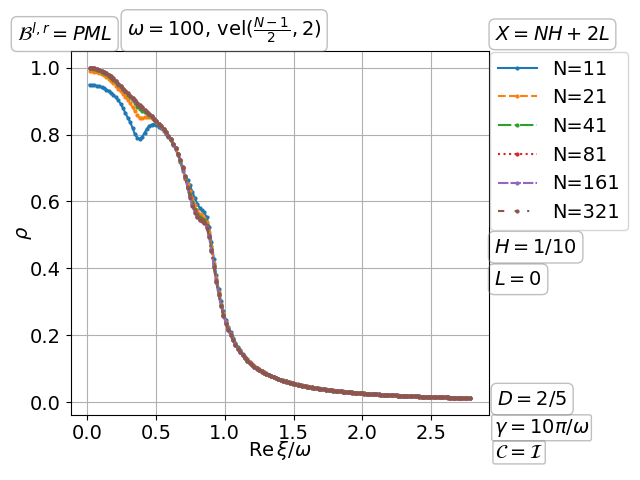}%
  \includegraphics[width=.5\textwidth,trim=5 6 0 2,clip]{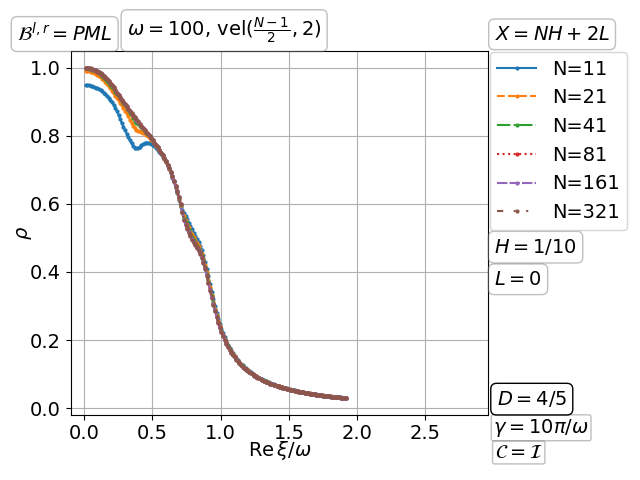}\\
  \includegraphics[width=.5\textwidth,trim=5 6 0 2,clip]{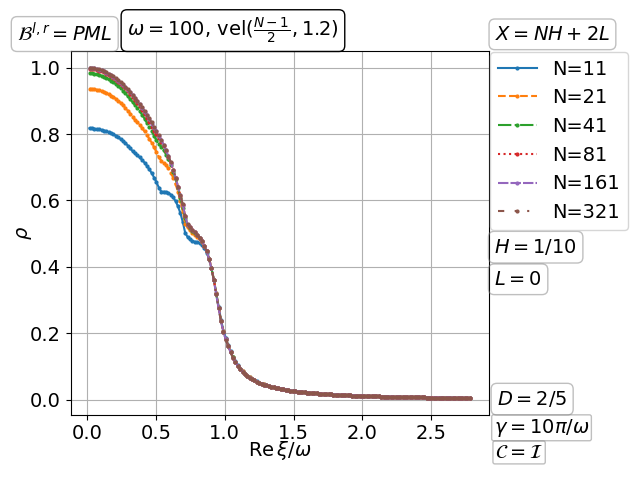}%
  \includegraphics[width=.5\textwidth,trim=5 6 0 2,clip]{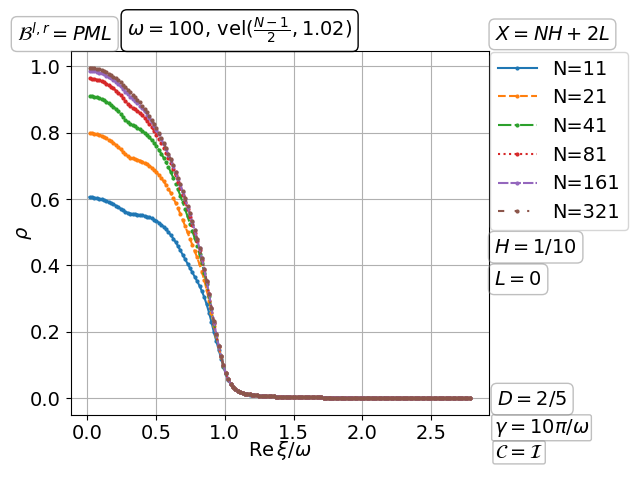}\\
  \includegraphics[width=.5\textwidth,trim=5 6 0 2,clip]{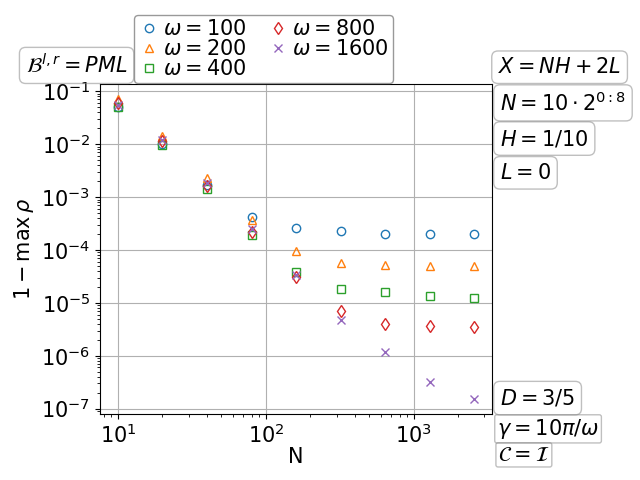}%
  \includegraphics[width=.5\textwidth,trim=5 6 0 2,clip]{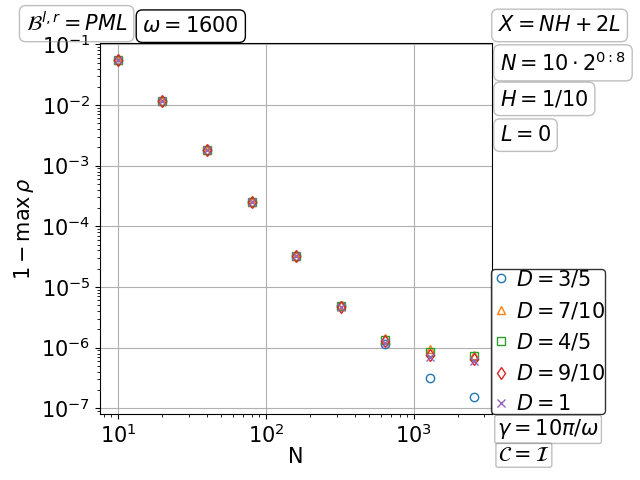}\\
  \includegraphics[width=.5\textwidth,trim=5 6 0 2,clip]{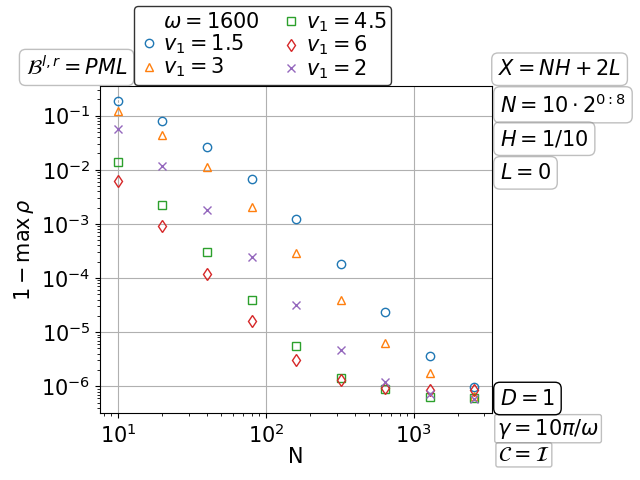}%
  \includegraphics[width=.5\textwidth,trim=5 6 0 2,clip]{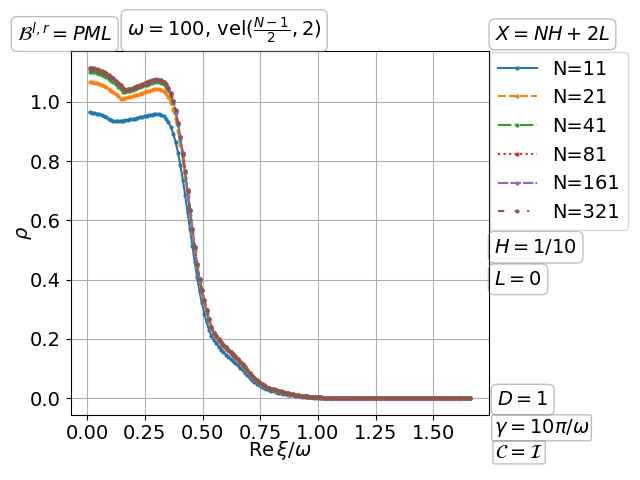}%
  \caption{Convergence ($\mathcal{B}_j^r=\mathcal{B}_{j+1}^{l}$ except bottom right) and divergence
    ($\mathcal{B}_j^r\ne\mathcal{B}_{j+1}^{l}$ bottom right) of the parallel Schwarz method with PML
    transmission for the full space wave problem with fixed size medium layers as subdomains.}
  \label{figlmppn}
\end{figure}

\end{paragraph}

\begin{paragraph}{Convergence on a fixed domain with increasing number of subdomains} In this case,
  the medium property in the domain is also fixed as we increase the number of subdomains. The
  subdomain interfaces are not aligned with the medium interfaces. Here, it turns out not to be a
  good idea (divergence observed) to take the domain average wave velocity as we did in the last
  paragraph. Instead, we simply use the values of the wave velocity at subdomain interfaces for the
  PML equations.  Figure~\ref{figlmpp1} compares the `vel1' and `vel2' velocity models in two
  columns. It can be seen that the convergence for `vel2' is slower than for `vel1'. Their
  convergence rates $\max_{\xi}\rho$ both deteriorate to $1$ linearly with $N^{-1}\to0$ but are
  independent of $\omega$. The convergence also deteriorates with higher contrast in the wave
  velocity $v_1$. Roughly, $1-\max_{\xi}\rho$ looks like $O(v_1^{-1})$ for `vel1' and
  $O(v_1^{-2})$ for `vel2'.
  
\begin{figure}
  \centering%
  \includegraphics[width=.5\textwidth,trim=5 6 0 2,clip]{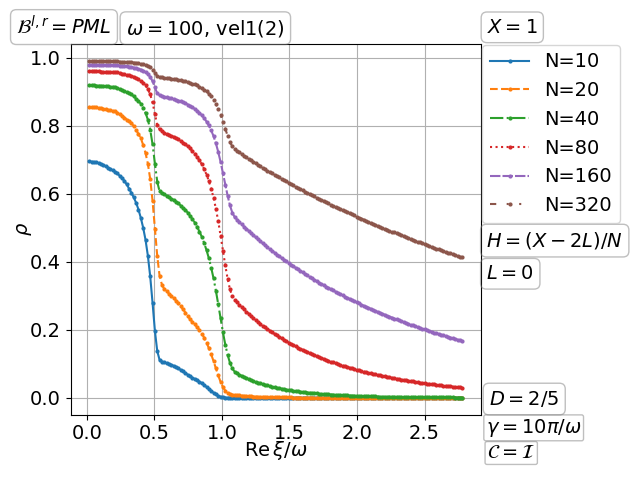}%
  \includegraphics[width=.5\textwidth,trim=5 6 0 2,clip]{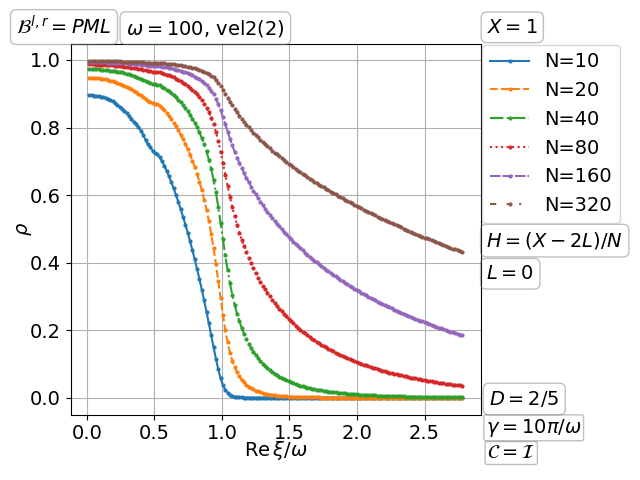}\\
  \includegraphics[width=.5\textwidth,trim=5 6 0 2,clip]{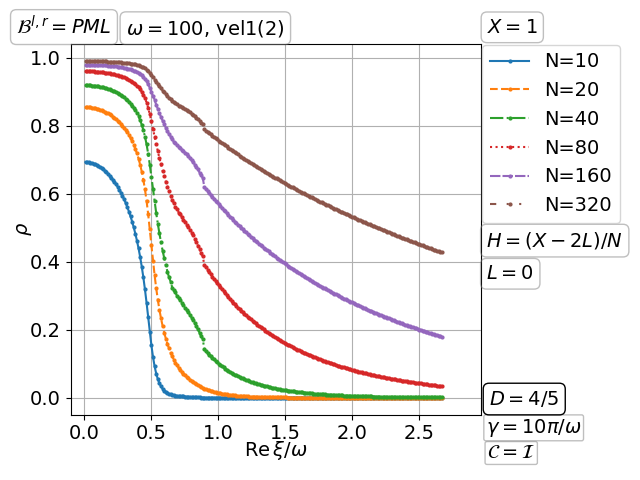}%
  \includegraphics[width=.5\textwidth,trim=5 6 0 2,clip]{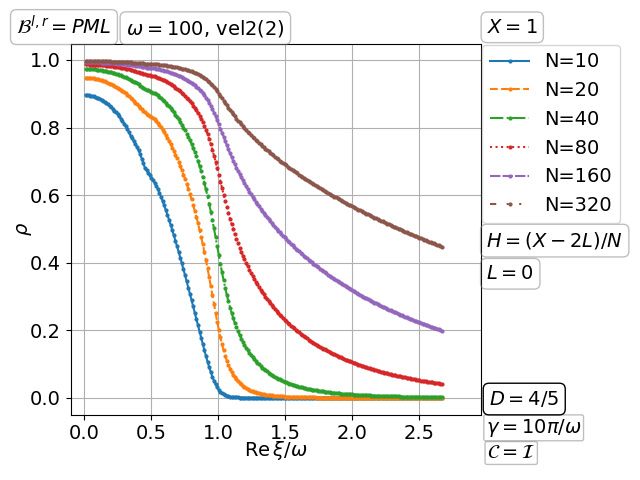}\\
  \includegraphics[width=.5\textwidth,trim=5 6 0 2,clip]{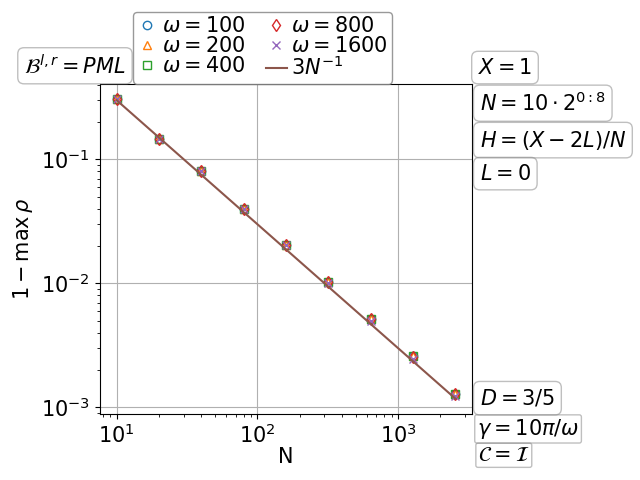}%
  \includegraphics[width=.5\textwidth,trim=5 6 0 2,clip]{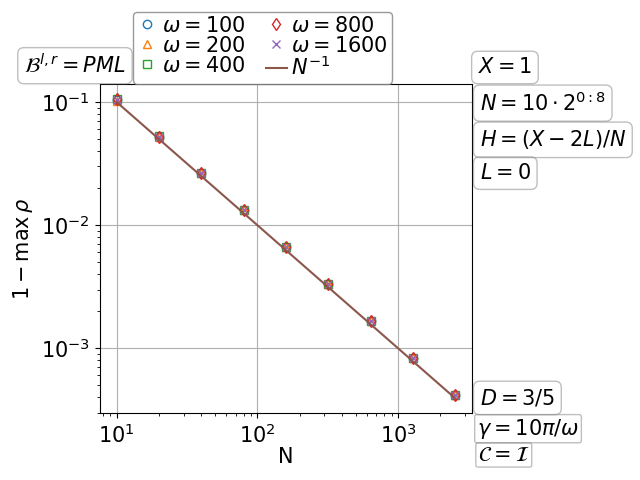}\\
  \includegraphics[width=.5\textwidth,trim=5 6 0 2,clip]{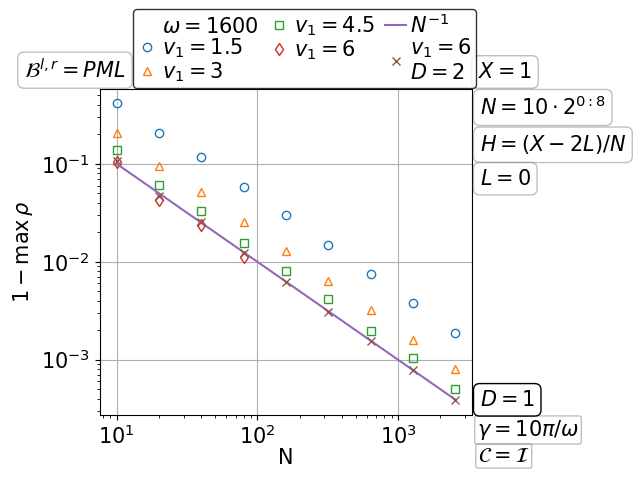}%
  \includegraphics[width=.5\textwidth,trim=5 6 0 2,clip]{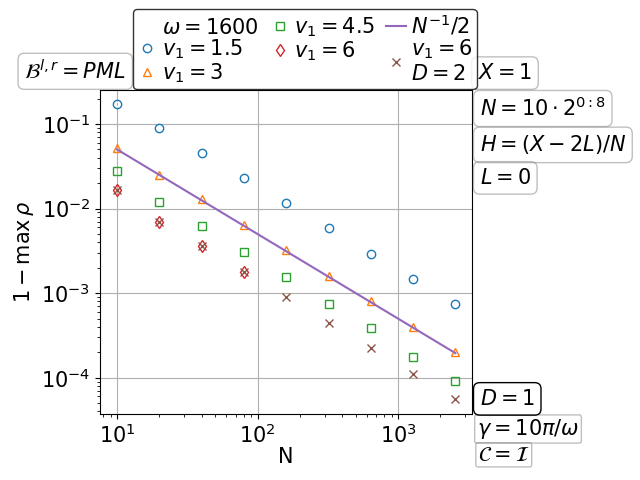}
  \caption{Convergence of the parallel Schwarz method with PML transmission for the full space
    wave problem with one inclusion (left) and two inclusions (right) on a fixed domain with
    increasing number of subdomains.}
  \label{figlmpp1}
\end{figure}

\end{paragraph}

% \subsubsection{Red-black Schwarz for layered media wave problem}

% does ordering of subdomains have something to do with the layout of medium?

\subsection{Double sweep Schwarz methods for the layered medium wave problem}

As seen in the constant medium case, Taylor of order zero transmission can hardly work with the
double sweep iteration. So, we will focus solely on the PML transmission. The readers are referred
to the beginning part of the previous subsection for the velocity models used here.

\begin{paragraph}{Convergence and divergence with increasing number of fixed size layers as
    subdomains}

  In this case, we found it is better to use the neighbor's wave velocity for the PML equation of a
  subdomain. But even for this choice, divergence is observed for a large number of layers (as
  subdomains); see Figure~\ref{figlmdpn}.
  
  \begin{figure}
  \centering%
  \includegraphics[width=.5\textwidth,trim=5 6 0 2,clip]{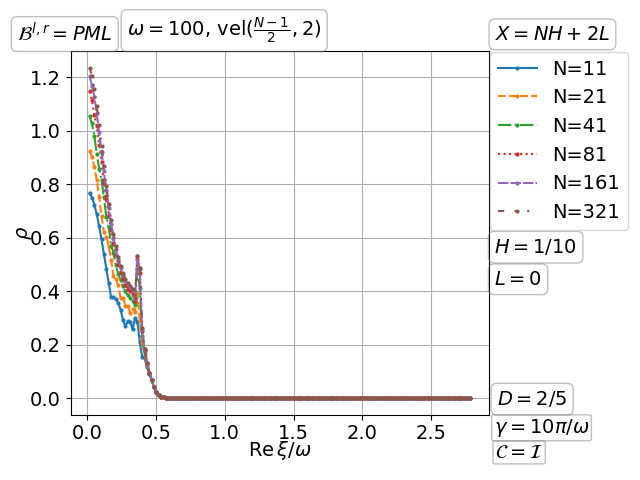}%
  \includegraphics[width=.5\textwidth,trim=5 6 0 2,clip]{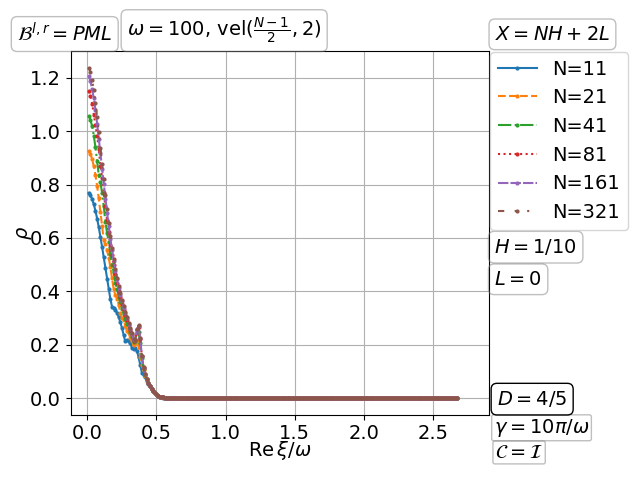}\\
  \includegraphics[width=.5\textwidth,trim=5 6 0 2,clip]{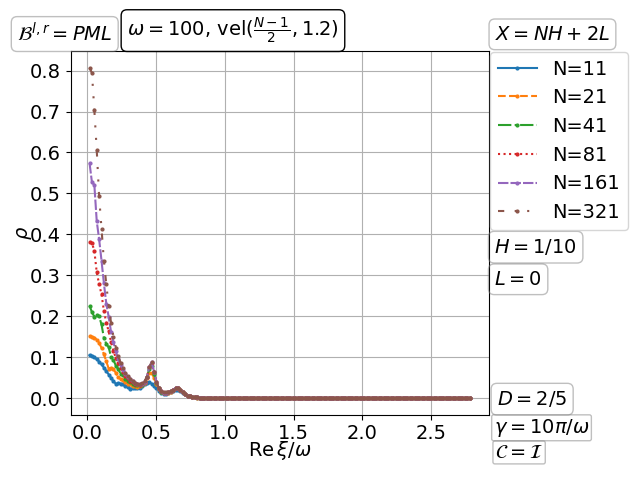}%
  \includegraphics[width=.5\textwidth,trim=5 6 0 2,clip]{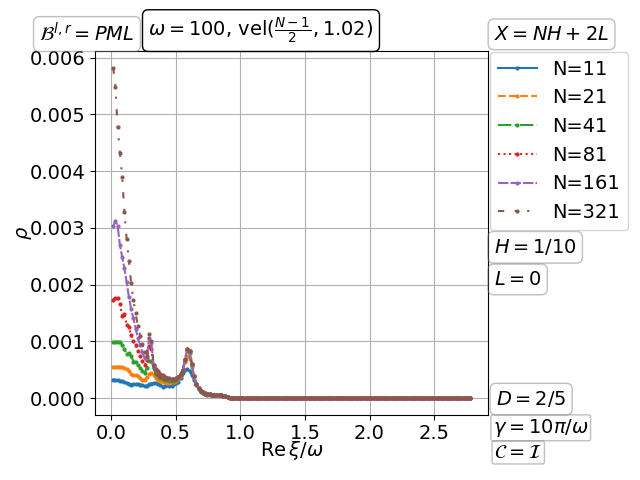}
  \caption{Convergence and divergence of the double sweep Schwarz method with PML transmission for
    the full space wave problem with fixed size medium layers as subdomains.}
  \label{figlmdpn}
\end{figure}

\end{paragraph}

\begin{paragraph}{Convergence on a fixed domain with increasing number of subdomains}

  Compared to the parallel iteration previously shown in Figure~\ref{figlmpp1}, the double sweep
  iteration converges significantly faster (see Figure~\ref{figlmdp1}) and does not seem to
  deteriorate with a larger number of subdomains $N$. For `vel1', the convergence rate
  $\max_{\xi}\rho$ is independent of the wavenumber $\omega$. But for `vel2', $\max_{\xi}\rho$
  oscillates with $\omega$. Higher contrast in the wave velocity $v_1$ leads to slower
  convergence. In the bottom right subplot, we see a deterioration of the convergence factor
  when the inclusions become smaller and more, and in particular divergence with more than
    three inclusions.
  
  \begin{figure}
  \centering%
  \includegraphics[width=.5\textwidth,trim=5 6 0 2,clip]{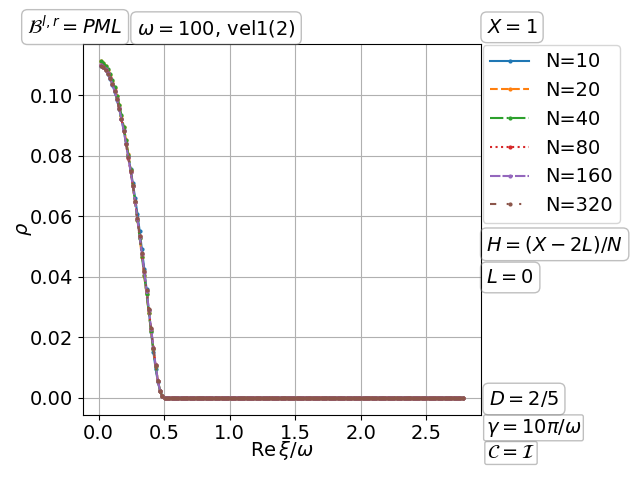}%
  \includegraphics[width=.5\textwidth,trim=5 6 0 2,clip]{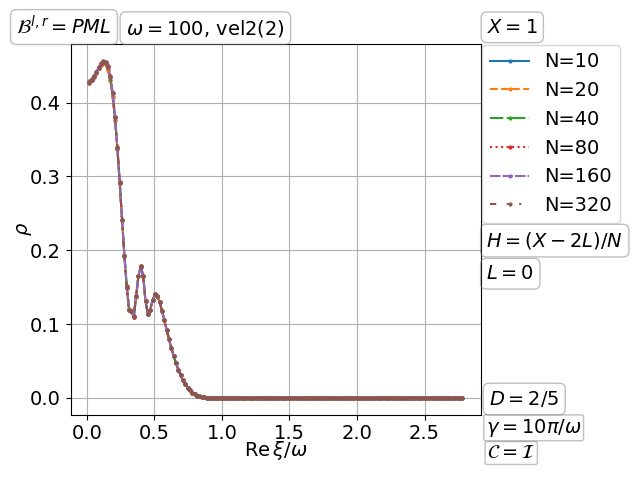}\\
  \includegraphics[width=.5\textwidth,trim=5 6 0 2,clip]{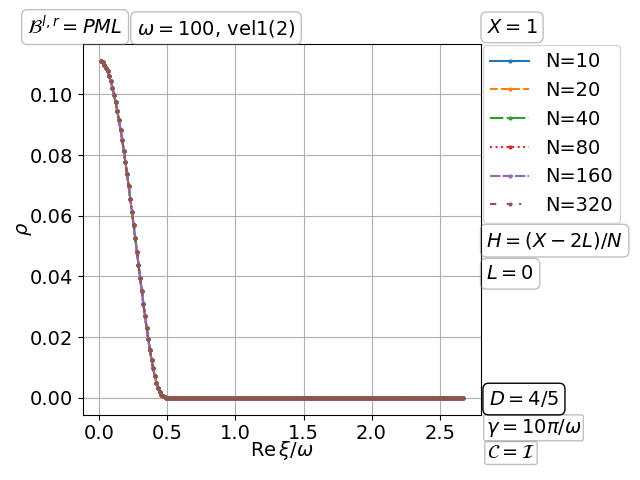}%
  \includegraphics[width=.5\textwidth,trim=5 6 0 2,clip]{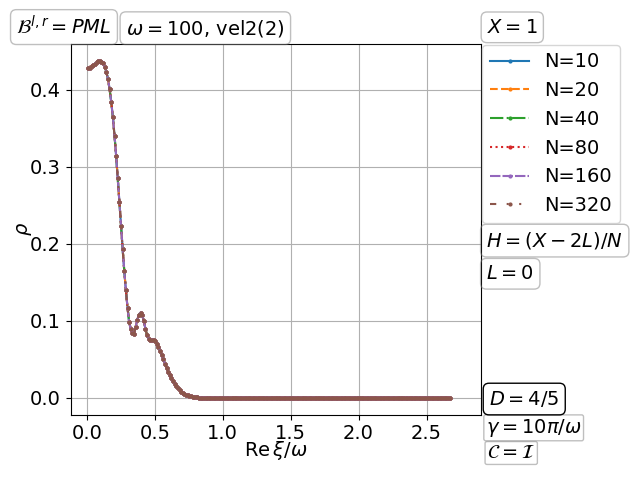}\\
  \includegraphics[width=.5\textwidth,trim=5 6 0 2,clip]{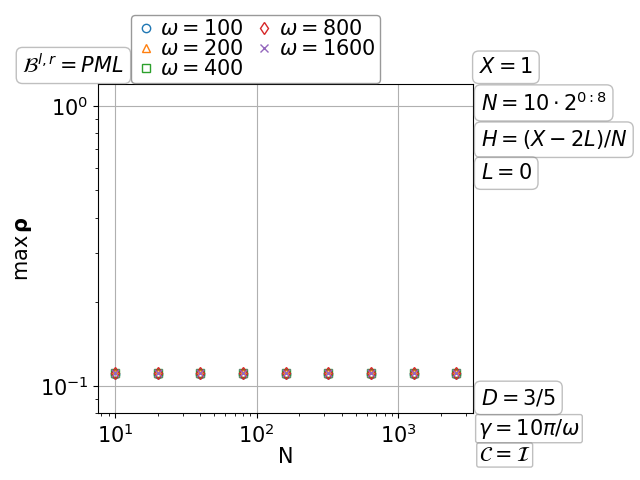}%
  \includegraphics[width=.5\textwidth,trim=5 6 0 2,clip]{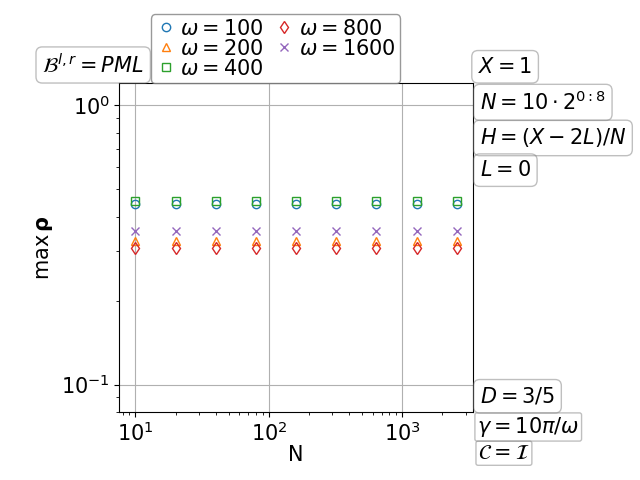}\\
  \includegraphics[width=.5\textwidth,trim=5 6 0 2,clip]{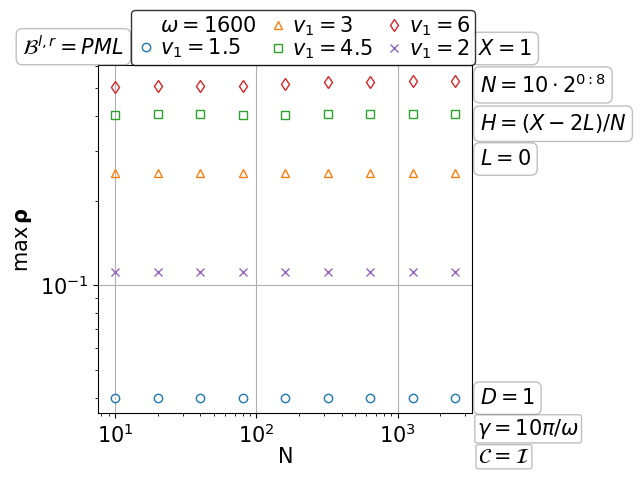}%
  \includegraphics[width=.5\textwidth,trim=5 6 0 2,clip]{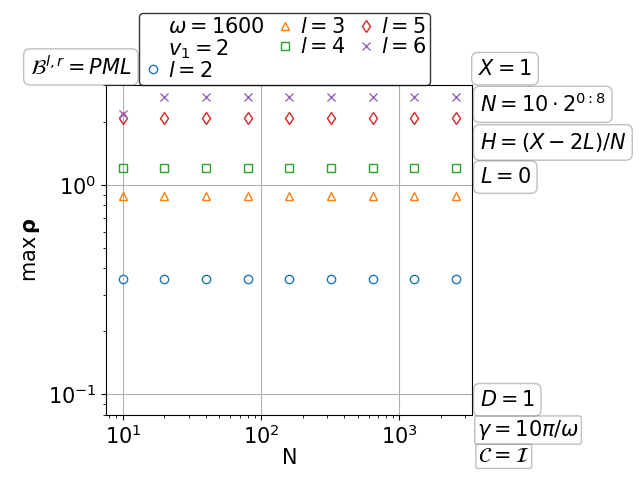}
  \caption{Convergence of the double sweep Schwarz method with PML transmission for the full space
    wave problem with one inclusion (left), two inclusions (right except bottom) and $l$ inclusions
    (bottom right) on a fixed domain with increasing number of subdomains. (Scaling of
    $\max_{\xi}\rho$ not $1-\max_{\xi}\rho$!)}
  \label{figlmdp1}
\end{figure}
  
\end{paragraph}

\vskip0.8em With these results and the results in the previous subsection on layered medium wave
problems, we have seen the limitations of PML transmission conditions. Is there a better way? For
example, are the nonlocal transmission conditions \cite{collino2020exponentially} or the learned
infinite elements \cite{hohage2021learned,PreussThesis} more robust for layered media? To what
extent can PML work for general variable media problems? For example, how to explain the success
seen in the existing literature for numerous geophysical applications?  Can other techniques be
combined with transmission conditions as a remedy?  We emphasize the recent progress of some
techniques for solving wave problems: for example, coarse spaces
\cite{CDKN,bonazzoli2018,bootland2020comparison} or deflation
\cite{dwarka2020scalable,dwarka2021towards,bootland2021inexact}, time domain solvers
\cite{grote2019,grote2020,appelo2020,stolk2021}, absorption or shifted-Laplace
\cite{GGS,GSV,GrahamSpenceZou,hocking2021}, $\mathcal{H}$-matrix
\cite{beams2020,lorca2021,liu2021sparse,bonev2021hierarchical}, DPG \cite{petrides2021} and
high-frequency asymptotic \cite{lu2016,fang2018,jacobs2021}. There are plenty of questions still to
be answered in the future.

%%%%%%%%%%%%%%%%%%%%%%%%%%%%%%%%%%%%%%%%%%%%%%%%%%%%%%%%%%%%%%%%%%%%%%%%%%%%%%%%
%
%
\section{Schwarz methods with cross points}
%
%
%%%%%%%%%%%%%%%%%%%%%%%%%%%%%%%%%%%%%%%%%%%%%%%%%%%%%%%%%%%%%%%%%%%%%%%%%%%%%%%%

In a domain decomposition, a cross point is where more than two subdomains intersect. If a cross
point is on the boundary of a subdomain, a transmission pattern thereof has to be decided. This is
important especially at the discrete level when the cross point is associated with a degree of
freedom. For detailed discussions, see \cn{gander2016cross}, \cn{GK2}, \cn{Loisel}. In particular,
different treatments of cross points for high-order (\eg~PML) transmission conditions have recently
been proposed by \cn{modave2020non}, \cn{daithesis}, \cn{dai2021}, \cn{royer}, \cn{DNT2}. See also
\cn{Claeys:2021:RTO}, \cn{claeys2021non} for nonlocal transmission conditions. A related issue is to
deal with a corner point where the normal direction is not unique
\cite{chniti2006improved,chniti2009improved,DNT1}.

Our emphasis in this section is on optimal Schwarz methods that converge exactly in finite
steps{, \ie\ methods that are nilpotent}. But before that, we mention quickly a couple of workable treatments of cross points in a {\it
  general} decomposition that are practically useful albeit non-optimal.

From the viewpoint of the subdomain boundary, a cross point is where the interfaces of the subdomain
with two different neighbors join. In a non-overlapping decomposition, a simple approach is to
include the cross point for transmission on both the interfaces with the two neighbors. For example,
in the non-overlapping case ($\Omega_{ij}=\tilde{\Omega}_{ij}$) of the decomposition in Figure~1.2
(left), the method of \cn{Despres} for the time-harmonic wave problem $(-\Delta+\omega^2)u=f$ in
$\Omega$ uses $(-\partial_{x}+\I\,\omega)u_{ij}=(-\partial_{x}+\I\,\omega)u_{i-1,j}$ on
$\{X_i^l\}\times[Y_j^b,Y_j^t]$ and
$(-\partial_{y}+\I\,\omega)u_{ij}=(-\partial_{y}+\I\,\omega)u_{i,j-1}$ on
$[X_i^l,X_i^r]\times\{Y_j^b\}$, and any degrees of freedom associated with the cross point at
$(X_i^l,Y_j^b)$ will be taken into account for both of the transmission conditions. In the
substructured form of the Schwarz method that reduces the unkowns to the interfaces, this
cross point treatment means introducing an unknown (data for the transmission condition, also called
dual variable or Lagrange multiplier) at the cross point for each incident interface of the
subdomain. Another approach is to keep the original degrees of freedom (also called primal
variables) and the corresponding equations at cross points unique (unsplit), then separate them out
by a Schur complement from the other primal and the dual variables \cite{Bendali}.

In an overlapping decomposition, the interfaces of a subdomain with its different neighbors can
overlap. To avoid redundant transmission, a partition of unity of the subdomain boundary is used. A
very natural choice is from a non-overlapping decomposition of the domain. For example, in
Figure~1.2 we can take $\partial\Omega_{ij}\cap\tilde{\Omega}_{i-1,j}$ as the interface for
transmission to $\Omega_{ij}$ from $\Omega_{i-1,j}$. But cross points where two interfaces of a
subdomain join still exist. For example, in Figure~1.2, we have the cross point where
$\partial\Omega_{ij}\cap\overline{\tilde{\Omega}}_{i-1,j}$ joins with
$\partial\Omega_{ij}\cap\overline{\tilde{\Omega}}_{i-1,j-1}$. At the discrete level, one can still
assign the cross point uniquely to one of the incident interface, and there is no problem for the
implementation of Schwarz methods based on subdomain iterates or interface substructuring. The
deferred correction form provides another approach, which starts from an initial guess (or the
previous iterate) of the solution on the original domain, takes the residual to subdomains for local
corrections, and make a global correction by gluing the local corrections. To minimise communication
between subdomains, it is practically common to use a non-overlapping decomposition of the domain
for gluing the local corrections, which gives the Restricted Additive Schwarz (RAS) method proposed
by \cn{cai1999restricted} with Dirichlet transmission conditions, and the Optimized Restricted
Additive Schwarz (ORAS) method proposed by \cn{St-Cyr07}, with general Robin-like interface
conditions. Moreover, it was shown in \cn{St-Cyr07} that, under some algebraic assumptions on the
gluing scheme, ORAS is equivalent to the parallel Schwarz method with the corresponding transmission
conditions. In the case of non-equivalence, the transmission data coming from the glued global
iterate can be a mix of multiple subdomain solutions. However, there is no convergence theory
(except using the maximum principle) for RAS and ORAS as of today. \cn{haferssas2017additive}
analyzed a symmetrised version of RAS and ORAS using the same non-overlapping partition-of-unity for
both the restriction/prolongation operators. \cn{GrahamSpenceZou} analyzed a symmetrised optimized
Schwarz method for the Helmholtz equation. See \cn{Gonghetero}, \cn{BCNT} for more analysis of
symmetrised Schwarz preconditioners for indefinite problems.

Now we turn to the main question: {is there an optimal Schwarz method beyond the sequential
  decomposition?} For example, Figure~1.2 (right) is a sequential decomposition and Figure~1.2
(left) is not. A difficulty arises even in Figure~1.2 (right) when the domain is $X$-periodic in
$x$, then the first subdomain is coupled to the last subdomain by identifying the boundary
$\{0\}\times[0,Y]$ with $\{X\}\times[0,Y]$. The difficulty is that an interface
$\{X_i^l\}\times[0,Y]$ can not separate the original problem into two regions because the region on
the left is coupled to the region on the right not only through the interface but also through the
periodic condition. Similarly, an annular domain decomposed into annular sectors has the same
problem. The issue from a loop of subdomains was recognized early by \cn{nier1998remarques}. It
arises also in a decomposition with cross points; as posed by \cn{NRS94}. For example, in Figure~1.2
(left) the interface $\{X_i^l\}\times[Y_j^b,Y_j^t]$ can not separate the domain into two regions
because there is also a loop in the subdomain adjacency relation. Why do we need an interface
separating the domain into two regions? Because that is the case for the domain truncation and
transmission to work as in the sequential decomposition.

The key of domain truncation and transmission can be understood in analogy to Gaussian
elimination. If there is no source term in the complementary domain of a subdomain, one needs only
to approximate the Schur complement which corresponds to a transparent boundary condition along the
interface separating the subdomain and its complementary domain, such that the subdomain solution
coincides with the solution of the original problem. If there is a source in the complementary
domain, one needs only to know the solution corresponding to the source in an arbitrarily small
outer (viewed from the subdomain) neighborhood of the interface, and to take the outer source into
the subdomain through a transmission condition (conveniently using the same boundary operator as in
the transparent boundary condition), which is like a back substitution in Gaussian elimination.

So, how did the recent work of \cn{leng2019additive}, \cn{taus2020sweeps} revive the domain
truncation and transmission in a checkerboard decomposition like Figure~1.2 (left)? To understand
their idea which is truly creative and a big step from the formerly existing optimal Schwarz
methods, it is very worth to read the original work. In this review, we present what we have
learned. First, it is good to have conceptually a domain truncation of the original problem into a
patch including the target subdomain, and use the truncated problem on the patch for transmission
into the subdomain. For example, in Figure~1.2 (left), a patch including $\Omega_{ij}$ can be
$\Xi_{i,j}^l:=\bar{\Omega}_{1j}\cup\cdots\cup\bar{\Omega}_{ij}$. A PML surrounding $\Xi_{i,j}^l$ is
used for domain truncation of the original problem, and we denote the PML augmented patch by
$\hat{\Xi}_{ij}^l$. Note that $\Omega_{ij}$ is also augmented with PML for domain truncation of the
original problem, and the PML augmented subdomain is denoted by $\hat{\Omega}_{ij}$. A transmission
will take place through the interface
$\hat{\Gamma}_{ij}^l:=\hat{\Omega}_{ij}\cap\{(x,y): x=X_i^l\}$. Let
$\hat{\Omega}_{ij}^l:=\hat{\Omega}_{ij}\cap\{(x,y): x\ge X_i^l\}$.
% $\hat{\Omega}_{ij}^r:=\hat{\Omega}_{ij}\cap\{(x,y): x\le X_i^r\}$ and
% $\hat{\Omega}_{ij}^{lr}:=\hat{\Omega}_{ij}\cap\{(x,y): X_i^l\le x\le X_i^r\}$.
We assume the PML on top (and bottom) of $\Xi_{1j}$, $\Xi_{2j}$, ..., $\Omega_{1j}$, $\Omega_{2j}$,
... are conforming (\ie, the same) along their interfaces and in their overlaps, and the PML on the
right of $\Xi_{ij}^l$ and $\Omega_{ij}$ are also the same. In the following lemma, we show
conceptually how the sources on the left of $\Omega_{ij}$ can be transmitted into $\Omega_{ij}$.

\begin{lemma}\label{lem4.1}
  Suppose the source $f$ of the original problem \R{eqprb} vanishes on $(X_i^l,X]\times[0,Y]$ and
  $[0,X]\times([0,Y]-[Y_j^b,Y_j^t])$; see Figure~\ref{lem41}. Let $\hat{v}_{ij}^l$ be the solution
  of the truncated problem on $\hat{\Xi}_{ij}^l$. Let $\hat{u}_{ij}^{l}$ be the solution of
  \begin{equation}
    \begin{aligned}
      (\mathcal{L}_x+\mathcal{L}_y+\eta)\hat{u}_{ij}^{l}&=0 & &\text{ on }\hat{\Omega}_{ij}^l,\\
      \mathcal{C}\hat{u}_{ij}^l&=0 & &\text{ on }[X_i^l, X_i^r+D]\times\{Y_j^b-D, Y_j^t+D\},\\
      \mathcal{C}\hat{u}_{ij}^l&=0 & &\text{ on }\{X_i^r+D\}\times[Y_j^b-D, Y_j^t+D],\\
      \mathcal{B}\hat{u}_{ij}^l&=\mathcal{B}\hat{v}_{ij}^l & &\text{ on
      }\hat{\Gamma}_{ij}^l=\{X_i^l\}\times[Y_j^b-D, Y_j^t+D],
    \end{aligned}
    \label{eqtl}
  \end{equation}
  where $\mathcal{C}$ is the terminating condition of the PML, and $\mathcal{B}$ is any boundary
  operator that makes the problem well-posed. Then, $\hat{u}_{ij}^l$ equals the restriction of
  $\hat{v}_{ij}^l$ onto $\hat{\Omega}_{ij}^l$.
  \label{lemtl}
\end{lemma}

\begin{figure}
  \centering
  \includegraphics[scale=.3]{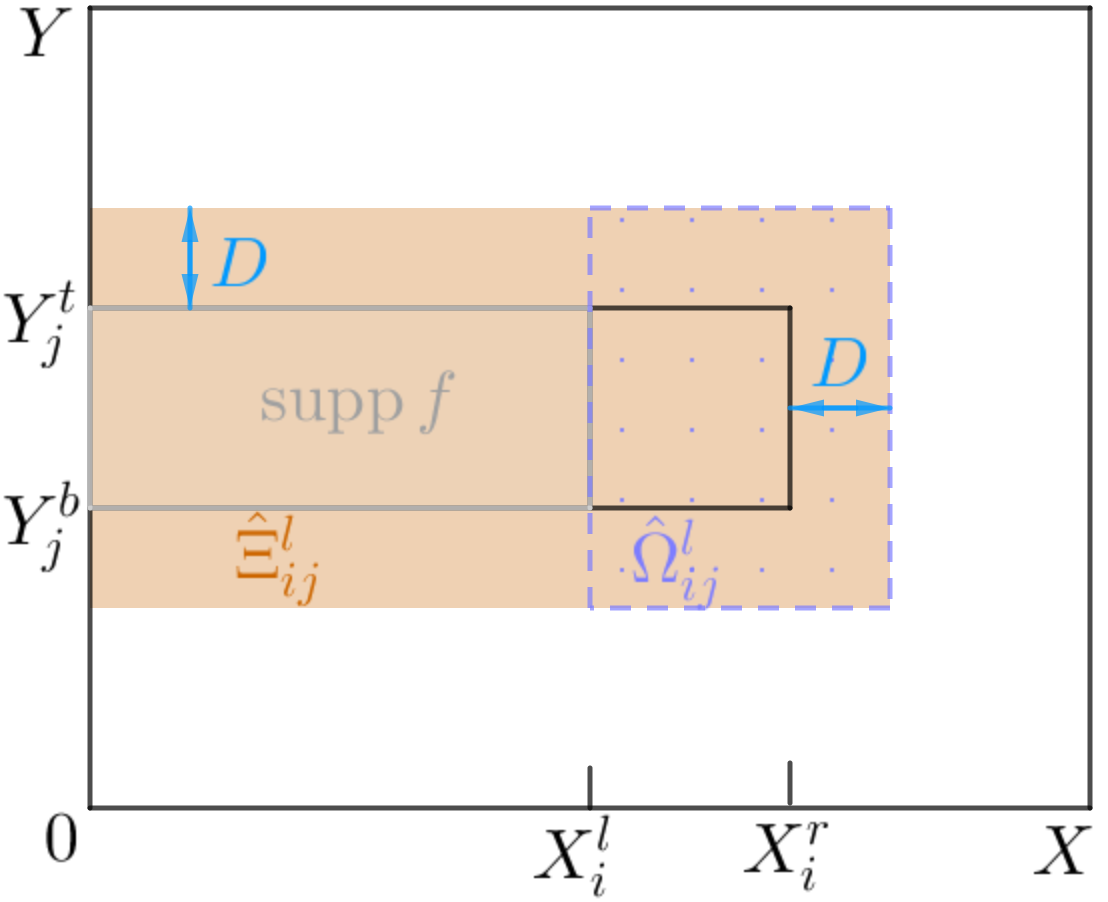}
  \caption{Illustration of Lemma \ref{lem4.1} for the transmission in the patch $\hat{\Xi}_{ij}^l$.}
  \label{lem41}
\end{figure}

\begin{proof}
  Since \R{eqtl} has the unique solution $\hat{u}_{ij}^l$ and the restriction of $\hat{v}_{ij}^l$
  onto $\hat{\Omega}_{ij}^l$ satisfies \R{eqtl}, the two must be equal.
\end{proof}

\begin{remark}
  If the PML of $\hat{\Xi}_{ij}^l$ is exactly transparent, then the restriction of $\hat{v}_{ij}^l$
  onto $\Xi_{ij}^l$ coincides with the solution of the original problem, and so is the restriction of
  $\hat{u}_{ij}^l$ onto $\Omega_{ij}$. A transparent PML does exist at the continuous level, see
  \eg~\cn{yang2021truly}.
\end{remark}

\begin{remark}
  To solve \R{eqtl}, we need only the knowledge of $\hat{v}_{ij}^l$ in a left neighborhood of
  $\hat{\Gamma}_{ij}^l$ to evaluate the transmission data $\mathcal{B}\hat{v}_{ij}^l$. The truncated
  problem on the patch $\hat{\Xi}_{ij}^l$ is never solved directly (like the original problem
  on $\Omega$ is never solved directly). In the algorithm to be detailed later, the transmission
  data is obtained from a solution defined on $\hat{\Omega}_{i-1,j}$. In a deferred correction form,
  the residual instead of the transmission data is evaluated, which gives rise to the source
  transfer method \cite{Chen13a}; see \cn{CGZ} for their relation.
\end{remark}

\begin{remark}
  Since we also want to resolve the local source $f_{ij}$ -- the restriction of $f$ onto ${\Omega}_{ij}$,
  the truncated problem on $\hat{\Omega}_{ij}$ will be used. It is then convenient to take
  $\mathcal{B}$ as $-\partial_x+\mathcal{S}$ with $\mathcal{S}$ the PML Dirichlet-to-Neumann
  operator \R{eqsd} or \R{eqsw} defined by the PML $\hat{\Omega}_{ij}-\hat{\Omega}_{ij}^l$. In this
  case, the problem \R{eqtl} is extended to $\hat{\Omega}_{ij}$ with $\mathcal{S}$ unfolded
  \cite{gander2019class}.
\end{remark}

\begin{remark}
  Note that we can replace in Lemma~\ref{lemtl} the patch and the subdomain with an arbitrary domain
  $\Xi$ and an arbitrary subdomain $\Omega_*\subset\Xi$ of any shape, with the transmission
  condition put on $\partial\Omega_*-\partial\Xi$ and the source $f$ vanishing on $\Omega_*$, as
  long as the problems on $\Xi$ and $\Omega_*$ are well-posed. See also \cn{leng2019additive} for
  the source transfer version.
\end{remark}

Similarly, we can define the patches $\Xi_{ij}^r$, $\Xi_{ij}^t$, $\Xi_{ij}^b$ for the transmission
of the sources outside the right/top/bottom of $\Omega_{ij}$; see Figure~\ref{figpatch}
(left). There are still four corner regions outside $\Omega_{ij}$. For the source to the bottom left
of $\Omega_{ij}$, we have conceptually the patch
$\Xi_{ij}^{bl}:=\cup_{{c\le i;\ r\le j}}\bar{\Omega}_{cr}$ and the PML augmented patch
$\hat{\Xi}_{ij}^{bl}$; see Figure~\ref{figpatch} (right).  Let
$\hat{\Omega}_{ij}^{bl}:=\hat{\Omega}_{ij}\cap\{(x,y): x\ge X_i^l, y\ge Y_j^b\}$. A transmission
will take place through the interface
$\hat{\Gamma}_{ij}^{bl}:=\hat{\Omega}_{ij}\cap\partial\hat{\Omega}_{ij}^{bl}$. Although the source
to be transmitted into $\Omega_{ij}$ is to the bottom left of $\Omega_{ij}$, \ie, on
$\Xi_{i-1,j-1}^{bl}$, a patch containing both $\Omega_{ij}$ and $\Xi_{i-1,j-1}^{bl}$ has to be
larger for the PML to work. That is why the patch $\Xi_{ij}^{bl}$ is used. As we mentioned in the
above remarks following Lemma~\ref{lemtl}, the PML augmented problem on $\hat{\Omega}_{ij}$ can be
solved with either the transmission data on $\hat{\Gamma}_{ij}^{bl}$ or the residual in a top-right
neighborhood of $\hat{\Gamma}_{ij}^{bl}$ which in turn is acquired from a glued version of subdomain
solutions on $\Omega_{i-1,j}$, $\Omega_{i-1,j-1}$ and $\Omega_{i,j-1}$. The resulting solution
$\hat{u}_{ij}^{bl}$ on $\hat{\Omega}_{ij}$ then has the same restriction onto
$\hat{\Omega}_{ij}^{bl}$ as does the conceptual solution $\hat{v}_{ij}^{bl}$ for the source located
to the bottom left of $\Omega_{ij}$. In the same way, we can define the PML augmented patches
$\hat{\Xi}_{ij}^{br}$, $\hat{\Xi}_{ij}^{tr}$, $\hat{\Xi}_{ij}^{tl}$ for the transmission of the
sources located to the bottom-right/top-right/top-left of $\Omega_{ij}$.

\begin{figure}
  \centering
  \includegraphics[scale=.35]{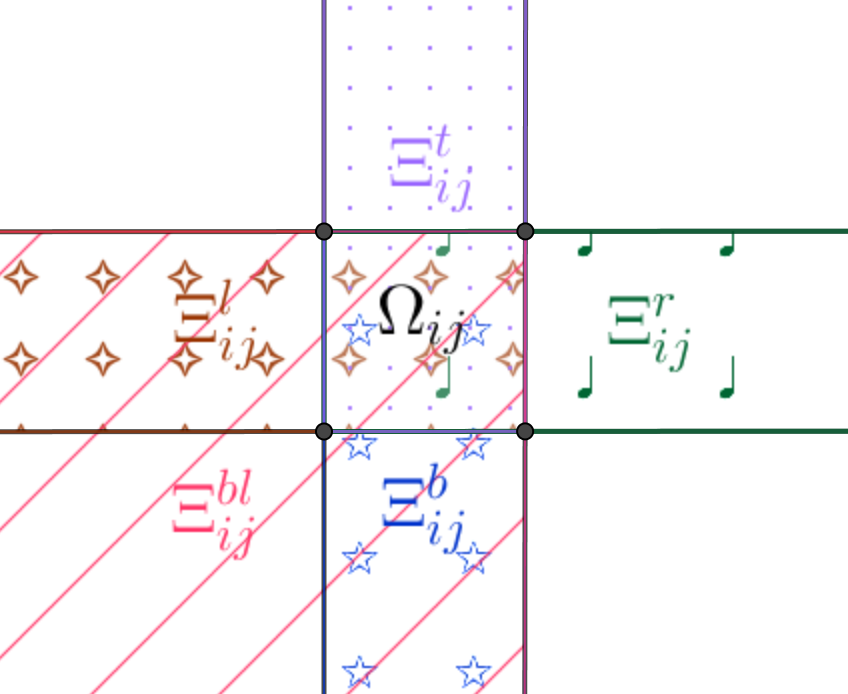}\quad\quad%
  \includegraphics[scale=.3]{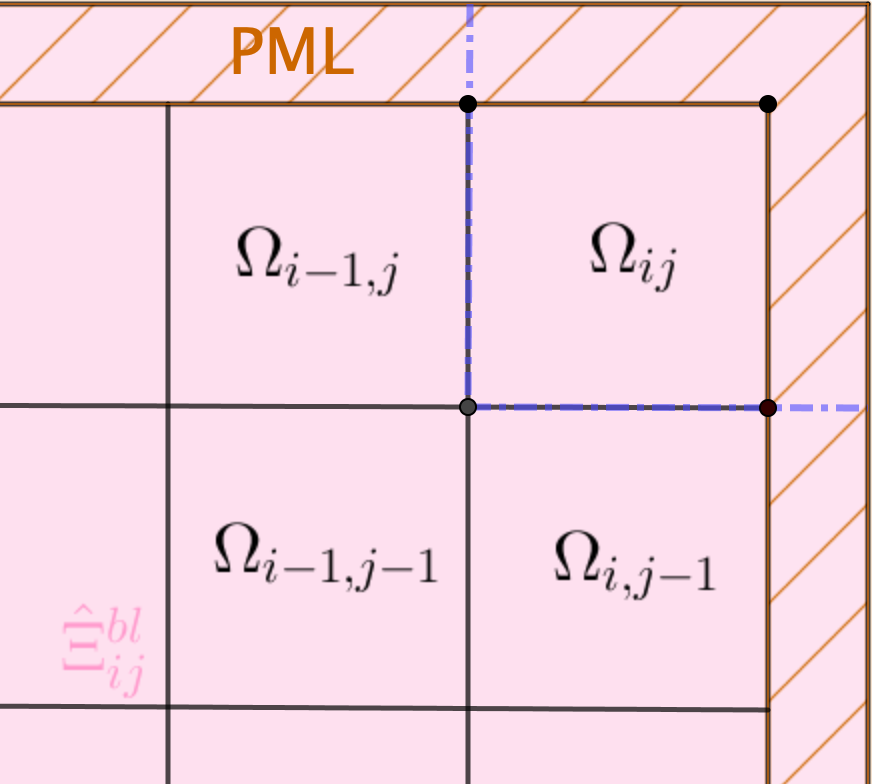}
  \caption{Patches containing the target subdomain $\Omega_{ij}$ (left), and the bottom left patch
    augmented with PML $\hat{\Xi}_{ij}^{bl}$ for the non-overlapping decomposition (right).}
  \label{figpatch}
\end{figure}

Let $\hat{u}_{ij}^c$ be the solution of the truncated problem on $\hat{\Omega}_{ij}$ corresponding
to the original problem \R{eqprb} with the source $f$ vanishing outside $\Omega_{ij}$. Clearly, if
we have the solutions $\hat{u}_{ij}^{l,r,b,t,bl,br,tl,tr,c}$ and the PML is exactly transparent,
then their sum on $\Omega_{ij}$ equals the solution of the original problem with the source $f$
distributed anywhere on $\Omega$, based on the linearity of \R{eqprb}. The remaining question is how
to get the transmission data needed for $\hat{u}_{ij}^{l,r,b,t,bl,br,tl,tr}$ on
$\hat{\Omega}_{ij}$. This is based on the following induction on the subscripts $i$, $j$. Only the
transmission data for $\hat{u}_{ij}^{l, bl}$ are illustrated.

\begin{lemma}
  Let the domain decomposition be non-overlapping. Suppose the PML is exactly transparent and the
  solutions $\hat{u}_{i-1,j}^{l,c}$ are available on $\hat{\Omega}_{i-1,j}$. Let
  $\hat{\Omega}_{i-1,j}^{lr}:=\hat{\Omega}_{i-1,j}\cap\{(x,y): X_{i-1}^l\le x\le X_{i-1}^r\}$. Then
  $\hat{u}_{i-1,j}^l+\hat{u}_{i-1,j}^c$ and $\hat{v}_{i,j}^l$ have the same restriction onto
  $\hat{\Omega}_{i-1,j}^{lr}$ and in particular the same trace $\mathcal{B}\hat{v}_{i,j}^l$ on
  $\hat{\Gamma}_{ij}^l$ for any boundary operator $\mathcal{B}$.
  \label{lemil}
\end{lemma}

\begin{proof}
  By Lemma~\ref{lemtl}, $\hat{u}_{i-1,j}^l$ coincides with $\hat{v}_{i-1,j}^l$ on
  $\hat{\Omega}_{i-1,j}^{l}\supset \hat{\Omega}_{i-1,j}^{lr}$. Let $\hat{V}_{ij}^l$ be the solution
  operator taking a source in $\Xi_{ij}^l$ as input for the truncated problem on
  $\hat{\Xi}_{ij}^l$. Let $\chi_{i-1,j}$ be the indicator function of $\Omega_{i-1,j}$. Let $f$ be
  the source in $\Xi_{i-1,j}^l$ (the domain on the left side of $\Omega_{ij}$) for obtaining
  $\hat{v}_{ij}^l$. Then, $f=\chi_{i-1,j}f+(1-\chi_{i-1,j})f$ and
  $\hat{v}_{ij}^l=\hat{V}_{ij}^l(\chi_{i-1,j}f) + \hat{V}_{ij}^l((1-\chi_{i-1,j})f)$. Since the PML
  is exact, $\hat{u}_{i-1,j}^c=\hat{V}_{ij}^l(\chi_{i-1,j}f)$ and
  $\hat{v}_{i-1,j}^l=\hat{V}_{ij}^l((1-\chi_{i-1,j})f)$ on $\hat{\Omega}_{i-1,j}^{lr}$. Hence,
  $\hat{u}_{i-1,j}^c+\hat{u}_{i-1,j}^l=\hat{v}_{ij}^l$ on $\hat{\Omega}_{i-1,j}^{lr}$.
\end{proof}

\begin{remark}
  Since we rely on the linearity or superposition principle, we need a partition of unity to
  restrict the source to the non-overlapping subdomains. So the source across the interfaces needs
  to be treated carefully. For a square integrable source, the partition is naturally done by the
  variational form. For a singular source \eg\ Dirac line source or point source on the interface,
  one needs to split the source, \eg, regard the source belonging to the subdomain on the left of
  the interface. Similarly, the source at a cross point needs a careful partition of unity.
\end{remark}

\begin{lemma}
  Let the domain decomposition be non-overlapping. Suppose the PML is exactly transparent and the
  solutions $\hat{u}_{i-1,j}^{b,bl}$ on $\hat{\Omega}_{i-1,j}$, $\hat{u}_{i,j-1}^{l,bl}$ on
  $\hat{\Omega}_{i,j-1}$ and $\hat{u}_{i-1,j-1}^{l,b,bl,c}$ on $\hat{\Omega}_{i-1,j-1}$ are
  given. Let
  $\hat{\Omega}_{i-1,j}^{lrb}:=\hat{\Omega}_{i-1,j}\cap\{(x,y): X_{i-1}^l\le x\le X_{i-1}^r, y\ge
  Y_j^b\}$ and
  $\hat{\Omega}_{i,j-1}^{lbt}:=\hat{\Omega}_{i,j-1}\cap\{(x,y): x\ge X_{i}^l, Y_{j-1}^b\le y\le
  Y_{j-1}^t\}$. Then $\hat{u}_{i-1,j}^{b}+\hat{u}_{i-1,j}^{bl}=\hat{v}_{i,j}^{bl}$ on
  $\hat{\Omega}_{i-1,j}^{lrb}$, $\hat{u}_{i,j-1}^{l}+\hat{u}_{i,j-1}^{bl}=\hat{v}_{i,j}^{bl}$ on
  $\hat{\Omega}_{i,j-1}^{lbt}$, and
  $\hat{u}_{i-1,j-1}^{l}+\hat{u}_{i-1,j-1}^{b}+\hat{u}_{i-1,j-1}^{bl}+\hat{u}_{i-1,j-1}^{c}=\hat{v}_{i,j}^{bl}$
  on $\Omega_{i-1,j-1}$.
  \label{lemibl}
\end{lemma}

\begin{proof}
  Similar to the proof of Lemma~\ref{lemil}, so the details are omitted.
\end{proof}

\begin{remark}
  Based on Lemma~\ref{lemibl}, the solutions $\hat{u}_{i-1,j}^{b}+\hat{u}_{i-1,j}^{bl}$,
  $\hat{u}_{i,j-1}^{l}+\hat{u}_{i,j-1}^{bl}$ and
  $\hat{u}_{i-1,j-1}^{l}+\hat{u}_{i-1,j-1}^{b}+\hat{u}_{i-1,j-1}^{bl}+\hat{u}_{i-1,j-1}^{c}$ can be
  glued along the interfaces by a partition of unity to a function on the L-shaped block
  $\overline{\hat{\Omega}_{i-1,j}^{lrb}}\cup \overline{\hat{\Omega}_{i,j-1}^{lbt}}\cup
  \overline{\Omega_{i-1,j-1}}$ which equals $\hat{v}_{ij}^{bl}$. Then, the function can be used for
  evaluation of the transmission data or residual.
\end{remark}

Now we can present the L-(diagonal) sweep preconditioner of \cn{taus2020sweeps}, \cn{lengdiag},
\cn{leng2020trace}. For simplicity, let us consider a non-overlapping decomposition with $3\times 3$
subdomains. In the following we first list the steps corresponding to a sweep from bottom left to
top right.
\begin{itemize}
\item[1$^\circ$] For each $(i,j)$, solve the truncated problem on $\hat{\Omega}_{ij}$ with the local
  source $f$ restricted onto $\Omega_{ij}$, and denote the solution by $\hat{u}_{ij}^c$.
\item[2$^\circ$] Solve for $\hat{u}_{21}^l$ on $\hat{\Omega}_{21}$ and $\hat{u}_{12}^b$ on
  $\hat{\Omega}_{12}$ with the transmission data provided by $\hat{u}_{11}^c$.
\item[3$^\circ$] Solve for $\hat{u}_{31}^l$ on $\hat{\Omega}_{31}$, $\hat{u}_{13}^b$ on
  $\hat{\Omega}_{13}$ and $\hat{u}_{22}^{l,b,bl}$ on $\hat{\Omega}_{22}$ with the transmission data
  provided by $\hat{u}_{21}^l+\hat{u}_{21}^c$, $\hat{u}_{12}^b+\hat{u}_{12}^c$, $\hat{u}_{12}^c$,
  $\hat{u}_{21}^c$ and a glued function from $\{\hat{u}_{11}^c$, $\hat{u}_{12}^b$,
  $\hat{u}_{21}^l\}$.
\item[4$^\circ$] Solve for $\hat{u}_{32}^{l,b,bl}$ and $\hat{u}_{23}^{l,b,bl}$ (details similar to
  $3^\circ$).
\item[5$^\circ$] Solve for $\hat{u}_{33}^{l,b,bl}$.
\end{itemize}
In parallel to the above steps $2^\circ$-$5^\circ$, the sweep from top right to bottom left, the
sweep from bottom right to top left and the sweep from top left to bottom right can be performed
simultaneously. Note that the latter two sweeps produce also the solutions with the superscripts
$l, r, b, t$, which can be shared with the former two sweeps so that any already available solutions
will not be computed again.  After finishing the four parallel sweeps, we just need to add the nine
solutions on each subdomain.
\begin{itemize}
\item[6$^\circ$] Define the subdomain approximate solutions $\tilde{u}_{ij}:=\sum_{*}\hat{u}_{ij}^*$
  on $\Omega_{ij}$ with $*\in\{l, r, b, t, bl, br, tl, tr, c\}$. Define the approximate solution
  $\tilde{u}$ on $\Omega$ by gluing $\tilde{u}_{ij}$ with a partition of unity.
\end{itemize}
Note that the solutions inside each of the steps $2^\circ$-$5^\circ$ can be computed in parallel,
albeit from step to step the execution is sequential. For a pipeline parallel implementation of the
preconditioned Krylov iteration for multiple right hand sides, see \cn{leng2020trace}.

It can be seen that the L-(diagonal) sweep preconditioner is an exact solver if the PML is exactly
transparent. In this sense, we say it is an optimal Schwarz method. In the non-exact setting, it is
not clear to us whether a subdomain iterative method in the spirit of the classical Schwarz method
can be formulated. That is, can we explain the preconditioned Richardson iteration with the
L-(diagonal) sweep preconditioner as some subdomain iteration, like the relation between the ORAS
preconditioner and the parallel Schwarz method \cite{St-Cyr07}?

The above idea of transmission in PML truncated patches was first proposed by \cn{leng2019additive}
for a \emph{parallel} Schwarz preconditioner which can be described as follows. Let the checkerboard
domain decomposition be non-overlapping. Let $\hat{u}_{ij}^{i'j'}$ be a solution on
$\hat{\Omega}_{ij}$ that approximates on $\Omega_{ij}$ the solution of the original problem with the
source $f_{i'j'}$ that is a partition of unity of the original source onto
$\bar{\Omega}_{i'j'}$. The preconditioner constructs on $\Omega_{ij}$ step by step the solutions for
farer and farer sources $f_{i'j'}$. The computation of $\hat{u}_{ij}^{i'j'}$ is totally parallel
between different $(i,j)$'s. In each of the following steps, the computation is for all $(i,j)$. It
is sufficient to describe the first few steps for understanding.
\begin{itemize}
\item[0$^\circ$] Compute $\hat{u}_{ij}^{ij}$.
\item[1$^\circ$] Compute $\hat{u}_{ij}^{i'j'}$ using the transmission data from
  $\hat{u}_{i'j'}^{i'j'}$ for $(i',j')\in\{(i\pm1,j), (i,j\pm 1)\}$.
\item[2$^{\circ}$] Compute $\hat{u}_{ij}^{i\pm2,j}$ using the transmission data from
  $\hat{u}_{i\pm1,j}^{i\pm2,j}$. Also, compute $\hat{u}_{ij}^{i,j\pm2}$ using
  $\hat{u}_{i,j\pm1}^{i,j\pm2}$. Moreover, compute $\hat{u}_{ij}^{i+1,j+1}$ using a glued function
  from $\hat{u}_{i+1,j+1}^{i+1,j+1}$, $\hat{u}_{i+1,j}^{i+1,j+1}$ and
  $\hat{u}_{i,j+1}^{i+1,j+1}$. Similarly, compute $\hat{u}_{ij}^{i-1,j+1}$, $\hat{u}_{ij}^{i-1,j-1}$
  and $\hat{u}_{ij}^{i+1,j-1}$.
\item[3$^{\circ}$] Compute $\hat{u}_{ij}^{i'j'}$ for $(i',j'): |i'-i|+|j'-j|=3$. If $i'<i, j'=j$,
  using the transmission data from $\hat{u}_{i-1,j}^{i'j}$. If $i'>i$, $j'=j$, using
  $\hat{u}_{i+1,j}^{i'j}$. Similarly, using $\hat{u}_{i,j\pm 1}^{ij'}$ if $i'=i$. If $i'>i, j'>j$,
  using a glued function from $\hat{u}_{i+1,j+1}^{i'j'}$, $\hat{u}_{i+1,j}^{i'j'}$ and
  $\hat{u}_{i,j+1}^{i'j'}$. Similarly, using appropriate solutions from the preceding two steps in
  the other cases of $(i', j')$ compared to $(i, j)$.
\end{itemize}
Since each step only looks back for the subdomain solutions obtained in the preceding two steps and
our goal is to get the sum $u_{ij}:=\sum_{(i',j')}\hat{u}_{ij}^{i'j'}$ on $\Omega_{ij}$, in each
step we can add the solutions obtained into $u_{ij}$ and discard any solutions that are no longer
needed. Moreover, in step 3 since the transmission data for $\hat{u}_{ij}^{i+1,j+2}$ and
  $\hat{u}_{ij}^{i+2,j+1}$ are taken from the same neighbors in the same patch
  $\hat{\Xi}_{ij}^{tr}$, we can first add the two transmission data and then compute directly the
  sum $\hat{u}_{ij}^{i+1,j+2}+\hat{u}_{ij}^{i+2,j+1}$ rather than compute $\hat{u}_{ij}^{i+1,j+2}$
  and $\hat{u}_{ij}^{i+2,j+1}$ individually.

% The subdomain approximations in a patch must solve the same equation with the same source
% term. So, a subdomain needs to keep multiple approximations corresponding to sources in different
% locations. This is mainly due to the need of corner transmission. Otherwise, for example, if
% $\hat{u}_{i-1,j}=\hat{u}_{i-1,j}^{b}+\hat{u}_{i-1,j}^{c}$ then the transmission to $\Omega_{ij}$
% in $\Xi_{ij}^l$ and $\Xi_{ij}^{bl}$ will have cross contribution. In particular,
% $\hat{u}_{i-1,j}^{b}$ has a nonzero contribution in $\Xi_{ij}^l$ because it satisfies a
% transmission data in the right PML of $\Omega_{i-1,j}$.

%\begin{remark}
Another approach to an optimal Schwarz method for a checkerboard decomposition is based on a
recursive application of an optimal Schwarz method for a sequential decomposition. For example, with
the decomposition in Figure~1.2 (left), we can combine the subdomains in each column to a block
$\Omega_{i}:=\cup_{j}\Omega_{ij}$. For the sequential decomposition
$\bar{\Omega}=\cup_i\bar{\Omega}_i$, an optimal Schwarz method by exact transparent PML transmission
can be applied. Then, the problem on each block augmented with the PML is solved again by an optimal
Schwarz method using the decomposition of the PML augmented $\Omega_i$ induced from the physical
  decomposition $\Omega_{i}=\cup_{j}\Omega_{ij}$. Such recursive Schwarz methods were proposed by
\cn{LiuYingRecur}; \cn{ZepedaNested}, and \cn{du2020pure} who gave also a convergence analysis. A
recursive optimal parallel Schwarz method converges in $N^2$ steps on a checkerboard of $N\times N$
subdomains, with each step costing one subdomain solve in wall-clock time. A recursive optimal
double sweep Schwarz method converges in one sequential double-double sweep costing $4N^2$ subdomain
solves in wall-clock time. For comparison, the new optimal parallel Schwarz method
\cite{leng2019additive} converges in $2N-1$ steps with each step costing one subdomain solve in
wall-clock time. The optimal L-(diagonal) sweep Schwarz method converges in four sweeps, with each
sweep parallel to the others and costing $2N-1$ subdomain solves in wall-clock time if the
subdomains on the same diagonal are treated in parallel. Of course, if only one worker is available
for computing, the recursive double sweep Schwarz method is competitive as a sequential solver.
% \end{remark}

\section{Conclusions}

While working on this review over the last two years, we discovered
many new results on optimized Schwarz methods based on domain
truncation, and each time realized that there are more further open
questions that it would be very interesting to research.  In
particular, we think the following research questions would be of
great interest for Schwarz methods based on domain truncation and
optimized Schwarz methods:

\begin{itemize}

\item What is the best overlap to be used with the zeroth order Taylor transmission condition for
  the Helmholtz equation? We have seen that for thin subdomains the overlap must be small enough for
  the method to converge, but it then deteriorates when the overlap goes to zero, so there must be
  an optimal overlap size for best performance. How this best choice depends on the Helmholtz
  frequency and decomposition is an open problem. We have however also seen that for large enough
  subdomains and generous overlap, the method can become robust in the wavenumber.

\item What is the asymptotic optimized parameter of the Robin
  condition for the free space wave problem on a bounded domain? Based
  on Fourier analysis, we discovered that with enough absorption in
  the original boundary conditions put on the domain on which the
  problem is posed, the asymptotic dependence of the convergence
  factor on the overlap is comparable to the much simpler screened
  Laplace problem, and the optimized transmission parameters have a
  clear asymptotic behavior, but their dependence on the wave number
  and geometry is not yet known.

\item What are the limiting spectra of double sweep Schwarz methods as
  the number of subdomains goes to infinity? The limiting spectra provide
  a very interesting new and accurate technique to study the
  convergence of Schwarz methods when the number of subdomains becomes
  large, and permit to obtain rather sharp estimates in several of
  the situations we have studied here, but not all of them.

\item What are the optimized parameters for many subdomains in a strip
  or checkerboard decomposition? For the strip decomposition case,
  there are first new results for a complex diffusion problem in
  \cn{KyriakisDD26}, \cn{ThesisAlex}, which indicate that the two
  subdomain asymptotic results also hold in the case of many
  subdomains, but for Helmholtz type problems this is still a largely
  open field.

\item Can we rigorously prove the many detailed convergence results
  and parameter dependence of the convergence factors we obtained from
  the many subdomain Fourier analysis? To this end one needs to be
  able to estimate accurately the spectral radii or norms of the
  substructured iteration matrices from Fourier analysis we studied
  here numerically.

\item When does the parallel or double sweep Schwarz method with PML
  converge for layered media? We have seen that PML can not capture
  the behavior of solutions in layered media, and often even the
  excellent damping properties do not suffice for the Schwarz methods
  based on domain truncation to work. Our concrete iteration matrices
  based on Fourier analysis provide however an excellent tool to get
  more fundamental theoretical insight into this.

\item Is there a better transmission condition for layered media? A
  first fundamental contribution into this direction is the recent PhD
  thesis \cn{PreussThesis} on learned infinite elements, which construct
  transmission conditions based on approximating the symbol of the Dirichlet
  to Neumann operator for the layered outer medium. Strong
  assumptions on the separability are however currently needed for
  this approach to be successful, and further research should be very
  fruitful into this direction.

\item Finally, the research on more general decompositions including
  cross points is today largely open: is it possible to get a precise
  convergence factor in a checkerboard decomposition, and to optimize
  transmission conditions in this case, or estimate the PML depth
  needed for good performance? And what would be good coarse space
  components for these problems and methods?

\end{itemize}

We hope that our present snapshot of the state of the art of Schwarz
methods based on domain truncation and optimized Schwarz methods, and
the above list of challenging open research questions will lead
to further progress in this fascinating and challenging field of
powerful domain decomposition methods, which can not be studied using
classical abstract Schwarz framework techniques.

\section{Appendix A}\label{AppendixA}

The following simple Matlab code allows the user to experiment with
nilpotent Schwarz methods based on block LU decompositions. It is
currently not known in the presence of cross points what kind of
transmission operators this approach generates, but it works for
arbitrary discretized partial differential equations.
{\small
\begin{verbatim}
nx=13;                          % # of mesh points in one direction per subdomain
n=3*nx+2;                       % n must be specific to fit 3 equal subdomains
h=1/(n-1);                      % mesh size
x=0:h:1;
G=numgrid('S',n);               % construct mesh point numbering using Matlab
O1=G(2:(n+1)/3,2:(n+1)/3);      % construct subdomains
o1=O1(:);
O2=G((n+1)/3+1:2*(n+1)/3-1,2:(n+1)/3);
o2=O2(:);
O3=G(2*(n+1)/3:n-1,2:(n+1)/3);
o3=O3(:);
O4=G(2:(n+1)/3,(n+1)/3+1:2*(n+1)/3-1);
o4=O4(:);
O5=G((n+1)/3+1:2*(n+1)/3-1,(n+1)/3+1:2*(n+1)/3-1);
o5=O5(:);
O6=G(2*(n+1)/3:n-1,(n+1)/3+1:2*(n+1)/3-1);
o6=O6(:);
O7=G(2:(n+1)/3,2*(n+1)/3:n-1);
o7=O7(:);
O8=G((n+1)/3+1:2*(n+1)/3-1,2*(n+1)/3:n-1);
o8=O8(:);
O9=G(2*(n+1)/3:n-1,2*(n+1)/3:n-1);
o9=O9(:);

A=delsq(G);                       % construct Laplacian using Matlab 
%op=[o1;o2;o3;o4;o5;o6;o7;o8;o9];  % lexicographic ordering
%op=[o1;o2;o3;o6;o9;o4;o5;o8;o7];  % L sweep ordering
op=[o1;o2;o4;o3;o5;o7;o6;o8;o9];  % diagonal sweep ordering

AA=A(op,op);                      % reorder matrix for block LU factorization

L=[];U=[];
m=length(o1);                     % assume all subdomains the same size
im=1:m;
for i=1:9                         % perform block LU factorization
  U((i-1)*m+im,(i-1)*m+1:(n-2)^2)=AA((i-1)*m+im,(i-1)*m+1:end);
  L((i-1)*m+im,(i-1)*m+im)=speye(m);
  for j=i+1:9
    L((j-1)*m+im,(i-1)*m+im)=AA((j-1)*m+im,(i-1)*m+im)*inv(AA((i-1)*m+im,(i-1)*m+im));
    AA((j-1)*m+im,:)= AA((j-1)*m+im,:)-AA((j-1)*m+im,(i-1)*m+im)*...
        (AA((i-1)*m+im,(i-1)*m+im)\U((i-1)*m+im,:));
  end;
end
AA=A(op,op);                  % since AA was overwritten above

UU=L';LL=U';                  % want solves in L on the diagonal, ok since A=A'
                              % otherwise would need to run LU on A' above
[X,Y]=meshgrid(x,x);
f=exp(-100*((X-0.5).^2+(Y-0.5).^2));
f1=f(2:end-1,2:end-1);
b=100*f1(:)*h^2;              % compute right hand side
% b=ones((n-2)^2,1);
bb=b(op);                     % reorder right hand side

v=zeros((n-2)^2,1);           % visualize block LU solution process
for i=1:9                     % forward substitution
  v((i-1)*m+im,1)=LL((i-1)*m+im,(i-1)*m+im)\bb((i-1)*m+im);
  ur(op)=v;                   % permute back
  Ur=G; Ur(2:end-1,2:end-1)=ur(G(2:end-1,2:end-1)); surf(x,x,Ur);
  axis([0 1 0 1 0 1]);xlabel('x');ylabel('y');
  set(gca,'FontSize',24)
  pause
  bb(i*m+1:end)=bb(i*m+1:end)-LL(i*m+1:end,(i-1)*m+im)*v((i-1)*m+im);
end;
u=v;                          % copy, since last block identity in U on diagonal
for i=9-1:-1:1                % backward substitution
  v(1:i*m)=v(1:i*m)-UU(1:i*m,i*m+im)*u(i*m+im);
  u((i-1)*m+im)=UU((i-1)*m+im,(i-1)*m+im)\v((i-1)*m+im);
  ur(op)=u;  % permute back
  Ur=G; Ur(2:end-1,2:end-1)=ur(G(2:end-1,2:end-1)); surf(x,x,Ur);
  axis([0 1 0 1 0 1]);xlabel('x');ylabel('y');
  set(gca,'FontSize',24)
  pause
end;

\end{verbatim}
}
\section{Appendix B}\label{AppendixB}

The following Maple commands can be used to compute many of the
formulas in the two subdomain analysis for the screened Laplace
problem, and also easily be modified to compute the corresponding
Helmholtz results.  There are many useful tricks in these Maple
commands, as indicated by the comments on the right. Note also that we
absorbed the term in the denominator of the solutions for {\tt E1} and
{\tt E2} obtained by Maple in the constants $A_1$ and $A_2$ in our
expressions used in \eqref{ScreenedLaplaceSols} to simplify the
expressions, without affecting the resulting convergence factor.
{\small
\begin{verbatim}
dsolve({diff(e(x),x,x)-(eta+k^2)*e(x)=0, -D(e)(0)+pl*e(0)=0},e(x)) assuming positive;
assign(%);                            # assign solution of ODE
E1:=subs(_C1=A1,simplify(e(x)));      # use concrete constant A1
e(x):='e(x)';                         # unassign e(x) for further use
E1:=simplify(convert(E1,sinh));       # work with sinh and cosh

dsolve({diff(e(x),x,x)-(eta+k^2)*e(x)=0, D(e)(1)+pr*e(1)=0},e(x)) assuming positive;;
assign(%);                            # assign solution of ODE
E2:=subs(_C1=A2,simplify(e(x)));      # use concrete constant A1
e(x):='e(x)';
E2:=simplify(convert(E2,sinh));       # work with sinh and cosh

E1p:=diff(E1,x);                      # compute derivative of subdomain
E2p:=diff(E2,x);                      # errors for transmission conditions

x:=X1r;                               # use transmisson conditions to 
A1:=A1np;A2:=A2n;                     # determine convergence factor
eq1:=E1p+p1r*E1=E2p+p1r*E2;
x:=X2l;
A1:=A1nm;
eq2:=-E2p+p2l*E2=-E1p+p2l*E1;

r1:=simplify(solve(eq1,A1np)/A2n);    # first and second part of the 
r2:=simplify(solve(eq2,A2n)/A1nm);    # convergence factor

RT0:=r1*r2;                           # Convergence factor
RT0k:=diff(RT0,k):                    # derivative for the maximum, colon
                                      # to suppress long output
				      
k:=Ckb/sqrt(L);                       # asymptotic behavior observed numerically
X1r:=X2l+L;                           # set overlap to L 

# need to determine constant Ckb from RTOk=0 and then evaluate RT0 at k.
# Difficulty is to expand RTOk for L small due to exponential terms in
# sinh and cosh which Maple can not handle automatically.
# Use substitution to simplify and identify structure:

s1:=sinh(sqrt(Ckb^2/L + eta)*(X2l + L - 1)); 
c1:=cosh(sqrt(Ckb^2/L + eta)*(X2l + L - 1));
s2:=sinh(sqrt(Ckb^2/L + eta)*X2l);
c2:=cosh(sqrt(Ckb^2/L + eta)*X2l);
s3:=sinh(sqrt(Ckb^2/L + eta)*(X2l + L));
c3:=cosh(sqrt(Ckb^2/L + eta)*(X2l + L));
s4:=sinh(sqrt(Ckb^2/L + eta)*(X2l - 1));
c4:=cosh(sqrt(Ckb^2/L + eta)*(X2l - 1));

# note that two of the sinh have the opposite sign

RT0sub:=subs({s1=-C1,c1=C1,s2=C2,c2=C2,s3=C3,c3=C3,s4=-C4,c4=C4},RT0);
RT0ksub:=subs({s1=-C1,c1=C1,s2=C2,c2=C2,s3=C3,c3=C3,s4=-C4,c4=C4},RT0k):

# from the exponents, we must have asymptotically for exponentially
# growing K1 and K2

C1:=K1*exp(-(Ckb + eta*(X2l - 1)/(2*Ckb))*sqrt(L));
C2:=K2*exp(eta*X2l*sqrt(L)/(2*Ckb));
C3:=K2*exp((Ckb + eta*X2l/(2*Ckb))*sqrt(L));
C4:=K1*exp(-eta*(X2l - 1)*sqrt(L)/(2*Ckb));

RT0s:=simplify(RT0sub);
RT0ks:=simplify(RT0ksub);

# we see that the exponentially growing terms now all cancel, we have
# treated them manually and can now expand using Maple

se:=series(RT0s,L,2)  assuming positive;
te:=simplify(op(2,se));
sek:=series(RT0ks,L,2)  assuming positive;
tek:=simplify(op(1,sek));
Ckbsols:=solve(tek,Ckb);
Ckb:=simplify(Ckbsols[1]); 

# gives the results from the manuscript:
# 
# Ckb:=sqrt(p1r + p2l), rho=1+te=1-4*sqrt(p1r + p2l)*sqrt(L)
#
# For illustration, we compute the Taylor of order zero coefficients
# from the optimal choice and insert it into the formulas obtained, and
# plot the asymptotic result compared to the exact one

k:='k';                               # unassign k
p1ropt:=simplify(convert(simplify(solve(numer(r1),p1r)),tanh));
p2lopt:=simplify(convert(simplify(solve(numer(r2),p2l)),tanh));
k:=0;                                 # Taylor of order zero condition
p1r:=p1ropt; p2l:=p2lopt;             # as Robin parameters
k:=Ckb/sqrt(L);

X2l:=1/3;pl:=5;pr:=3;eta:=1;
plots[loglogplot]([1-RT0,-te],L=0.0000001..0.1);

\end{verbatim}
}

\bibliographystyle{actaagsm}
\bibliography{references}
\label{lastpage}

\end{document}